\documentclass[12pt,leqno]{amsart}
\usepackage{amssymb,amsthm,amsmath,latexsym}
\usepackage{amsfonts,amscd, amsbsy,mathtools}
\usepackage{wasysym}
\usepackage{bbm}
\usepackage{faktor}
\usepackage{float}
\usepackage{xcolor}
\usepackage{pdfpages}
\usepackage{tcolorbox}
\usepackage{booktabs}
\usepackage{url}
\usepackage{xcolor}
\newtheorem{theorem}{\sc Theorem}[section]
\newtheorem{thm}[theorem]{\sc Theorem}

\newtheorem*{thmE}{\textbf{Theorem E}}

 \DeclareMathOperator{\GL}{GL}
  
 \DeclareMathOperator{\Aut}{Aut} 
 \DeclareMathOperator{\Tor}{Tor}  

\DeclareMathOperator{\diag}{diag}

 \theoremstyle{definition}\newtheorem{definition}{Definition}
\newtheorem{example}{Example}

\newtheorem{proposition}{Proposition}

\theoremstyle{remark}

\title[Automorphism Orbits]{Automorphism Orbits of the Group of Unitriangular Matrices}
\address{Departamento de Matem\'atica, Campus Universit\'{a}rio Darcy Ribeiro,  Universidade de Bras\'ilia,
Brasilia-DF, 70910-900 Brazil}
\author[de Melo]{Emerson de Melo}
\author[Kato]{Júlia Kato}
\email{(de Melo) emerson@mat.unb.br; (Kato) juliakato@mat.unb.br}
\begin{document}

\begin{abstract}
Let $G$ be a group. The orbits of the natural action of $\Aut(G)$ on $G$ 
are called the automorphism orbits of $G$, and their number is 
denoted by $\omega(G)$. Let $\mathbb{F}$ be an infinite field, 
and let $UT_n(\mathbb{F})$ denote the group of unitriangular matrices over 
$\mathbb{F}$. We show that $\omega(UT_n(\mathbb{F}))$ is finite for $n \leq 5$ 
and infinite for $n \geq 6$.
\end{abstract}

\maketitle
\section{Introduction}

Let $G$ be a group, and let $\Aut(G)$ denote its automorphism group.   The group $G$ is partitioned into orbits under the natural action of $\Aut(G)$, where two elements $g, h \in G$ are in the same orbit if there exists an automorphism $\alpha \in \text{Aut}(G)$ such that $g^\alpha = h$. These orbits are referred to as \textit{automorphism orbits} and the number of automorphism orbits is denoted by $\omega(G)$. A group is called a \textit{k-orbit group} if it has exactly $k$ automorphism orbits.

The identity group is the only example of an 1-orbit group. Furthermore, it is straightforward to verify that the finite 2-orbit groups are exactly the elementary abelian groups of prime-power order. Higman, Neumann, and Neumann  constructed a non-abelian torsion-free simple group in which all nontrivial elements are conjugate (\cite{Rob} (6.4.6)) showing the difference of infinite groups in this context. 

The investigation into the relationships between groups and their automorphism orbits began with  Higman \cite{H63} in 1963. Higman initially explored finite 2-groups in which the involutions constitute a single automorphism orbit. 
A related problem had emerged earlier: the classification of groups with a prescribed number of conjugacy classes, which are the orbits under inner automorphisms. This problem was proposed by W. Burnside \cite{B55} in 1955 .
Since then, the study of automorphism orbits has become a prominent area of interest among group theorists. Observe that automorphism orbits are unions of conjugacy classes and hence they give an example of \textit{fusion} in the holomorph group $G \rtimes \Aut(G)$, a well-known concept established in the literature (see for instance \cite[Chapter 7]{gorenstein}).

Laffey and MacHale \cite{LM} in  1986 classified all finite soluble non-$p$-groups with $\omega(G)=3$. Recently, finite 2-groups that are 3-orbit groups were classified by Bors and Glasby \cite{BG20}, while 3-orbit groups that are $p$-groups with odd prime $p$ were classified by Li and Zhu \cite{LZ24}, completing the classification of finite $3$-orbit groups.

Motivated by the work of T. J. Laffey and D. MacHale \cite{LM}, subsequent research has extended this line of investigation to mixed-order groups, that is, groups containing both elements of finite and infinite order. Recall that a group is said to be radicable if for every $g \in G$ and every $n \in \mathbb{N}$, there exists $h \in G$ such that $h^n = g$. Several structural results were obtained  in this direction. For soluble groups of finite rank, Bastos, Dantas, and de Melo \cite{BDdM} showed that if $\omega(G)<\infty$, then $G=K \rtimes H$, where $K$ is a torsion-free radicable nilpotent characteristic subgroup and $H$ is finite. In the virtually nilpotent case, it was shown in \cite{emerson2021} that if $\omega(G)<\infty$, then $G=K \rtimes H$, where $K$ is a torsion-free radicable nilpotent characteristic subgroup and $H$ a torsion subgroup, and moreover $G' = D \times \Tor(G')$ with $D$ torsion-free, nilpotent, radicable, and characteristic.

These results suggest that the torsion-free radicable nilpotent groups groups deserves special attention. The family of unitriangular groups $UT_n(\mathbb{F})$ over a field $\mathbb{F}$, provides a concrete class of nilpotent groups which, for certain fields, are torsion-free and radicable, making it a natural starting point for such an exploration. Furthermore, this perspective motivates the broader problem of classifying linear groups that possess only finitely many automorphism orbits. 

The group $UT_n(\mathbb{F})$ can be viewed as a subgroup of $GL_n(\mathbb{F})$. Its normalizer is known to be the group of upper triangular matrices $T_n(\mathbb{F})$ \cite{AB95}, which decomposes as the semidirect product
$$T_n(\mathbb{F})=UT_n(\mathbb{F})\rtimes D_n(\mathbb{F}),$$
where $D_n(\mathbb{F})$ denotes the diagonal subgroup. Consequently, $T_n(\mathbb{F})$ embeds naturally into $\Aut(UT_n(\mathbb{F}))$.

Platonov conjectured that in $T_n(\mathbb{F})$ the number of the conjugacy classes of $UT_n(\mathbb{F})$ in general is infinite.  Zalesskii \cite{Z}, showed that the number in question is infinite in the triangular group $T_n(\mathbb{F})$ for $n\geq 6$.

The automorphism group of the group of unitriangular matrices over a field was studied by many authors \cite{M93,L83,M10,W55}. Pavlov studied the automorphism group of unitriangular matrices over a finite field of odd prime order. Weir \cite{W55} described the automorphism group of the group of unitriangular matrices over a finite field of odd characteristic. Maginnis \cite{M93} described it for the field of order two. 

Finally, Levchuk \cite{L83} proved that the automorphism group of $UT_n(\mathbb{K})$, where $\mathbb{K}$ is a ring, is generated by certain families of automorphisms. Later, Mahalanobis \cite{M10} provided an alternative proof in 2013 for the case where $\mathbb{K}$ is a field, showing that $\Aut(UT_n(\mathbb{K}))$ is generated by extremal, diagonal, inner, central, and field automorphisms.

Our main results are as follows: 
\begin{thm}
Let $\mathbb{F}$ be an infinite field. The unitriangular group $UT_n(\mathbb{F})$ has finitely many automorphism orbits for $n \leq 5$. In particular, $\omega(UT_3(\mathbb{F}))=3$. Furthermore, when $n=4$, the number of orbits of $UT_4(\mathbb{F})$ under the action of $T_4(\mathbb{F})$ is $16$, and for $n=5$ this number is $61$.
\end{thm}

The case $n=5$ was established using the computer algebra system SageMath, since its verification required a case-by-case analysis involving numerous systems of equations. This computational approach allowed us to systematically handle these cases and rigorously determine the number of automorphism orbits. Further details regarding the case $n=5$ are provided in the appendix below.

As mentioned above, in \cite{Z} it was proved that $UT_n(\mathbb{F})$ has infinitely many automorphism orbits under the action of $T_n(\mathbb{F})$ for $n \geq 6$. We follow the same strategy to obtain the following result.

\begin{thm}
Let $\mathbb{F}$ be an infinite field. Then the unitriangular group $UT_n(\mathbb{F})$ has infinitely many automorphism orbits for all $n \geq 6$.
\end{thm}
\section{Preliminaries}

This section is based on the papers \cite{L83} and \cite{M10}, where the authors determine a generating set for the automorphism group of the unitriangular group over a field.

The upper unitriangular matrix group of dimension $n\times n$ over the field $\mathbb{F}$, denoted $UT_n(\mathbb{F})$, is the group, under multiplication, with $1$'s on the diagonal, $0$'s below the diagonal, and arbitrary entries above the diagonal. 
Explicitly,
$$UT_n(\mathbb{F})=\left\{ \begin{pmatrix}
1 & * & * & \dots & * \\
0 & 1 & * & \dots & * \\
0 & 0 & 1 & \dots & * \\
\dots & \dots & \dots & \dots & * \\
0 & 0 & 0 & \dots & 1 
\end{pmatrix}_{n\times n}\colon \text{all star-marked entries vary arbitrarily over }\mathbb{F}\right\}.  $$

We define the elementary matrix $xe_{i,j}$ to be the $n\times n$ matrix with $x$ in the $(i,j)$ position, 1's in the diagonal and  0 elsewhere, where $x\in \mathbb{F}$.

\begin{definition}
Define 
\begin{equation*}
    \gamma_k=\{M=(m_{i,j})\in UT_n(\mathbb{F})\colon m_{i,j}=0,\text{ }i<j,\text{ } j-i< k\}.
\end{equation*}
\end{definition}

In other words, the $\gamma_1=UT_n(\mathbb{F})$. The subgroup $\gamma_2$ is the commutator of $UT_n(\mathbb{F})$.  It consists of all upper unitriangular matrices with the first superdiagonal entries zero. The first superdiagonal can be specified by all entries $(i,j)$ with
 $j-i =1$. Similarly $\gamma_3$ consists of all matrices with the first two superdiagonals
 zero and so on. It follows that $\gamma_n =1$. We denote the identity matrix by $\mathbbm{1}$.

 \begin{proposition}[\cite{M10}, Proposition 1.1]
     In $UT_n(\mathbb{F})$, the lower central series and the upper central series are identical and it is of the form 
$$UT_n(\mathbb{F})=\gamma_1>\gamma_2>\ldots>\gamma_{n-1}>\gamma_n=\mathbbm{1}.$$
 \end{proposition}

We now describe some examples of automorphisms of $UT_n(\mathbb{F})$. 
They are given in terms of their action on the generators $x e_{i,i+1}$, for $i=1,2,\dots,n-1$. 
The automorphisms are as follows:

\begin{enumerate}
    \item \textbf{Inner automorphisms:} We will denote by  $\mathcal{J}$  the subgroup of inner automorphisms. For an invertible matrix $A=(a_{i,j})\in UT_n(\mathbb{F})$, the inner automorphism induced by $A$ is defined by
    $$X\mapsto A^{-1} X A,$$
    where $X\in UT_n(\mathbb{F})$;

    \item \textbf{Diagonal automorphisms:} We will denote by $D$ the subgroup of diagonal automorphisms. For a diagonal matrix $\diag{[d_1,\dots, d_n]}$, where each $d_i$ is invertible, $i=1,\dots,n$, the diagonal automorphism induced by $\diag{[d_1,\dots, d_n]}$ is given by
    $$e_{i,j}\mapsto d_i^{-1}e_{i,j}d_j;$$

    \item \textbf{Central automorphisms:} The subgroup of central automorphisms will be denoted by $\mathcal{Z}$, it is generated by the automorphisms 
$$xe_{i,i+1}\mapsto xe_{i,i+1}+x^{\lambda}e_{1,n},$$
where $\lambda$ is a linear map of $\mathbb{F}^+$ to itself;

    \item \textbf{Extremal Automorphisms:} We will denote by $\mathcal{U}$  the subgroup of extremal automorphisms, these automorphisms are given by the rule,
\begin{gather*}
	xe_{1,2}\to x(e_{1,2}+\lambda e_{2,n})+\frac{\lambda x^2}{2}e_{1,n},\\
	x e_{n-1,n}\to x(e_{n-1,n}+\mu e_{1,n-1})+\frac{\mu x^2}{2}e_{1,n}, \quad x\in \mathbb{F},
\end{gather*}
for $\lambda,\mu$ running through $\mathbb{F}$. All other generators remain fixed;

    \item \textbf{Field Automorphisms:} We will denote by $\overline{\Aut}(\mathbb{F})$ the subgroup formed by induced field automorphisms. It is generated by the automorphisms 
$$x e_{i,i+1}\mapsto x^{\mu} e_{i,i+1},$$
where $\mu$ is a field automorphism and $i=1,\dots,n-1$;

    \item \textbf{Flip Automorphisms:}  Let $W$ denote the subgroup generated by the flip automorphism which is given by 
$$x e_{i,j}\mapsto(-1)^{i-j-1}x e_{n-j+1,n-i+1}.$$ 
This automorphism is given by flipping the
 matrix by the anti-diagonal and changing the sign of some entries.  This is an automorphism of order 2.
    
\end{enumerate}

Now we state the theorem of characterization of the automorphism group of $UT_n(\mathbb{F})$.

\begin{theorem}[\cite{L83}, Corollary 5]
Let $\mathbb{F}$ be an infinite field. The group of automorphism of
$\Aut((UT_n(\mathbb{F}))$  coincides with the product 
$$(\mathcal{Z}\rtimes \GL_2(\mathbb{F}))\rtimes \overline{\Aut}(\mathbb{F}),$$
for $n=3$, and for $n>3$, coincides with the product 
$$((((\mathcal{ZJ})\rtimes\mathcal{U})\rtimes W)\rtimes D)\rtimes \overline{\Aut}(\mathbb{F}).$$
\end{theorem}

\begin{definition}
For $j>i$ let us define $N_{i,j}$ to be the subset of $UT_n(\mathbb{F})$ all of whose matrices have all rows greater than the $i^{\text{th}}$ row zero and all columns less than the $j^{\text{th}}$ column zero, except from the diagonal entries.  
\end{definition}

It is straightforward to see that $N_{i,j}$ is an abelian normal subgroup of $UT_n(\mathbb{F})$. We now present some examples.

\begin{example}
    Let $G=UT_4(\mathbb{F})$. The subgroup $N_{1,2}$ and $N_{2,3}$ has elements of the form
     \begin{align*}
         &N_{1,2}=\left\{\begin{pmatrix}
		1 & a_{1,2} & a_{1,3} & a_{1,4}  \\
		0 & 1 & 0 & 0  \\
		0 & 0 & 1 & 0  \\
		0 & 0 & 0 & 1 
	\end{pmatrix}\colon a_{i,j}\in\mathbb{F}\right\},\\
    &N_{2,3}=\left\{\begin{pmatrix}
		1 & 0 & a_{1,3} & a_{1,4}  \\
		0 & 1 & a_{2,3} & a_{2,4}  \\
		0 & 0 & 1 & 0  \\
		0 & 0 & 0 & 1 
	\end{pmatrix}\colon a_{i,j}\in\mathbb{F}\right\}.
     \end{align*}

Note that $N_{2,3}=C_{UT_4(\mathbb{F})}(\gamma_2)$.     
\end{example}

\begin{proposition}[\cite{L83}, Lemma 5.]
The centralizer of $N_{i,j}$ is
    $$C_{UT_n(\mathbb{F})}(N_{i,j})=N_{j-1, i+1}.$$
\end{proposition}

Levchuk also proved that the subgroups $N_{i,i+1}$ are maximal abelian normal subgroups of $UT_n(\mathbb{F})$, for $i=1,2,\ldots n-1$.  This property is particularly useful, because the image of maximal abelian normal subgroups under automorphism are maximal abelian normal subgroups.

We conclude this section with the following useful proposition. 
\begin{proposition}\label{remark:ação}
Let $X\in UT_n(\mathbb{F})$. If an entry on the first superdiagonal of $X$ is nonzero, it cannot be transformed to zero via conjugation by a diagonal or unitriangular matrix.
\end{proposition}
\begin{proof}
The group of upper triangular matrices $T_n(\mathbb{F})$ of dimension $n\times n$ over $\mathbb{F}$ is decomposed as the semidirect product
$$T_n(\mathbb{F})=UT_n(\mathbb{F})\rtimes D_n(\mathbb{F}),$$
where $D_n(\mathbb{F})$ is the subgroup of $n\times n$ diagonal matrices over $\mathbb{F}$. Let $T = UD \in T_n(\mathbb{F})$ with $U \in UT_n(\mathbb{F})$ and $D \in D_n(\mathbb{F})$. 
The entries on the first superdiagonal correspond to the quotient 
$UT_n(\mathbb{F}) / \gamma_2(UT_n(\mathbb{F}))$ and therefore are unaffected by the 
unitriangular factor $U$; they depend only on the diagonal part $D$, whose entries are nonzero. 
Consequently, if an entry of a matrix $X$ on the first superdiagonal is nonzero, it cannot be 
transformed into zero by conjugation with either a diagonal or a unitriangular matrix.
\end{proof}

If $A \leq \Aut(G)$ acts on $G$, then the action partitions $G$ into orbits. 
A set $T \subseteq G$ is called a set of orbit representatives for the action of $A$ on $G$ if
\begin{equation*}
G = \bigcup_{t \in T} Orb_A(t),
\end{equation*}
where $Orb_A(t) = \{ t^a : a \in A \}$ denotes the orbit of $t$ under $A$.

\section{The unitriangular group $UT_3(\mathbb{F})$}

In this section, we prove the following.

\begin{theorem}
    The unitriangular group $UT_3(\mathbb{F})$ has 3 automorphism orbits.
\end{theorem}

\begin{proof}
Let $G:=UT_3(\mathbb{F})$. We show that elements in $G'\setminus\{\mathbbm{1}\}$ and $G\setminus G'$ coincide with the two non trivial automorphism orbits. If $x,y\in G'\setminus\{\mathbbm{1}\}$, for 
\begin{equation*}
     x=\left(\begin{matrix}
	1 & 0 & x_{1,3}  \\
	0 & 1 & 0  \\
	0 & 0 & 1 
\end{matrix}\right),      y=\left(\begin{matrix}
	1 & 0 & y_{1,3}  \\
	0 & 1 & 0  \\
	0 & 0 & 1 
\end{matrix}\right),
\end{equation*}
it is clear that there is a diagonal matrix $A$ given by
\begin{equation*}
    A=\left(\begin{matrix}
	x_{1,3} & 0 & 0  \\
	0 & 1 & 0  \\
	0 & 0 & y_{1,3} 
\end{matrix}\right),
\end{equation*}
such that $x^{A}=y$. Since $G'$ is a characteristic subgroup, its non trivial elements form an automorphism orbit. Now we analyze elements in $G\setminus G'$. Set 
\begin{equation*}
     x=\left(\begin{matrix}
	1 & x_{1,2} & x_{1,3}  \\
	0 & 1 & x_{2,3}  \\
	0 & 0 & 1 
\end{matrix}\right)\in G\setminus G'.
\end{equation*}
We now proceed by considering three distinct cases. First assume $x_{1,2}\neq0$ and $x_{2,3}=0$. Conjugation by matrix $A$ yields
\begin{equation*}
     A=\left(\begin{matrix}
	1 & 0 & 0  \\
	0 & x_{1,2}^{-1} & -x_{1,3}x_{1,2}^{-1}  \\
	0 & 0 & 1 
\end{matrix}\right),\quad x^A=\left(\begin{matrix}
	1 & 1 & 0  \\
	0 & 1 & 0  \\
	0 & 0 & 1 
\end{matrix}\right)=B_1.
\end{equation*}
Similarly assume $x_{1,2}=0$ and $x_{2,3}\neq 0$.
Conjugation by matrix $A$ yields
\begin{equation*}
     A=\left(\begin{matrix}
	1 & x_{1,3}x_{2,3}^{-1} & 0  \\
	0 & 1 & 0  \\
	0 & 0 & x_{2,3}^{-1} 
\end{matrix}\right),\quad x^A=\left(\begin{matrix}
	1 & 0 & 0  \\
	0 & 1 & 1  \\
	0 & 0 & 1 
\end{matrix}\right)=B_2.
\end{equation*}
Now assume $x_{1,2}\neq0$ and $x_{2,3}\neq 0$. Conjugation by matrix $A$ yields
\begin{equation*}
     A=\left(\begin{matrix}
	1 & 0 & 0  \\
	0 & x_{1,2}^{-1} & {x_{1,3}}(x_{1,2}^2x_{2,3})^{-1}  \\
	0 & 0 & (x_{2,3}x_{1,2})^{-1} 
\end{matrix}\right),\quad x^A=\left(\begin{matrix}
	1 & 1 & 0  \\
	0 & 1 & 1  \\
	0 & 0 & 1 
\end{matrix}\right)=B_3.
\end{equation*}
It suffices to prove that $\omega(UT_3(\mathbb{F})$ is finite. However, one can verify that the map sending $B_1$ to $B_2$ and $B_2$ to $B_3$ extends to an automorphism of the group. This proves that $G\setminus G'$ forms an automorphism orbit. Thus $\omega(G)=3$.
\end{proof}

\section{The unitriangular group $UT_4(\mathbb{F})$}

In this section, we prove the following. 

\begin{theorem}
    The unitriangular group $UT_4(\mathbb{F})$ has 16 orbits under the action of $T_4(\mathbb{F})$. 
Moreover, the numbers of orbits corresponding to characteristic subgroups are as follows: 
two for $\gamma_3$, five for $\gamma_2$, and seven for $N_{2,3}$.
    
\end{theorem}

First, we outline the strategy used. We establish a partition for $UT_4(\mathbb{F})$ and for each subset of the partition, we will present a finite number of orbit representatives. The approach involves taking an arbitrary element from the subset and constructing an automorphism that maps this element to the representative. 

Consider the following partition of $UT_4(\mathbb{F})$
\[
UT_4(\mathbb{F})= Y_1\cup Y_2\cup Y_3\cup Y_4\cup Y_5\cup Y_6\cup Y_7\cup Y_8,
\]
where

$Y_1=\left\{ \left(
\right)\colon x_{i,j}\in \mathbb{F},x_{3,4}\neq 0\right\}.$

Note that each subset corresponds to a union of cosets of $\gamma_2$. 
Therefore, the conjugacy classes of elements contained in a subset $Y_i$ remains in $Y_i$, for $1 \leq i \leq 8$.
\\

In what follows, let $A \in T_4(\mathbb{F})$ be a matrix of the form  
\[
A=%
.
\]

Since $Y_1$ corresponds to the derived subgroup, it contains further characteristic subgroups. 
For this reason, we postpone the discussion of $Y_1$ until the end. 
In what follows, we provide a detailed explanation for the first case; 
the remaining cases are treated in an analogous way.

\subsection*{Subset $Y_2$:}
The subset $Y_2$ has elements of the form
\begin{equation*}
	x=\left(%
\right),
\end{equation*}
where $x_{1,2},x_{2,3},x_{3,4}$ are non zero. We want to find entries of the matrix $A$ such that
\begin{equation*}
     x^{A}=\left(%
\right).
\end{equation*}
We computed $x^A=(b_{ij})\in UT_4(\mathbb{F})$ and obtained the following entries
{
\begin{align*}
	&b_{1,1} = 1, b_{2,2} = 1, b_{3,3} = 1, b_{4,4} = 1, \\
	&b_{1,2} =\frac{d_2x_{1,2}}{d_1} , b_{1,3} = \frac{a_{2,3}d_2x_{1,2}+d_2d_3x_{1,3}-a_{1,2}d_3x_{2,3}}{d_1d_2}, \\
    &b_{1,4} = \frac{\splitfrac{a_{2,4}d_2d_3x_{1,2}+a_{3,4}d_2d_3x_{1,3}+d_2d_3d_4x_{1,4}-a_{1,2}a_{3,4}d_3x_{2,3}-}{-a_{1,2}d_3d_4x_{2,4}+a_{1,2}a_{2,3}d_4x_{3,4}-a_{1,3}d_2d_4x_{3,4}}}{d_1d_2d_3}, \\
    &b_{2,3} = \frac{d_3x_{2,3}}{d_2}, b_{2,4} = \frac{a_{3,4}d_3x_{2,3}+d_3d_4x_{2,4}-a_{2,3}d_4x_{3,4}}{d_2d_3},  b_{3,4} = \frac{d_4x_{3,4}}{d_3}.
\end{align*}}%
To find the matrix $A$, we equate the corresponding entries $b_{ij}$ with those of the candidate representative and set up a system of equations. Solving this system of equations we found that matrix $A$ can defined by
\begin{align*}
	&d_{1} = 1, d_{2} = \frac{1}{x_{1,2}}, d_{3} = \frac{1}{x_{1,2}x_{2,3}}, d_{4} = \frac{1}{x_{1,2}x_{2,3}x_{3,4}}, \\
	&a_{1,2} =0 , a_{1,3} = 0, a_{1,4} = 0, a_{2,3} = \frac{ - x_{1,3}}{x_{1,2}^{2} x_{2,3}}, \\
    &a_{2,4} = -\frac{x_{1,2} x_{1,4} x_{2,3} - x_{1,2} x_{1,3} x_{2,4} - { x_{1,3}^{2}} x_{3,4}}{x_{1,2}^{3} x_{2,3}^{2} x_{3,4}},  a_{3,4} = -\frac{ x_{1,2} x_{2,4} +   x_{1,3} x_{3,4}}{x_{1,2}^{2} x_{2,3}^{2} x_{3,4}}.
\end{align*}
This means that all elements of $Y_2$ can be mapped to the matrix above via conjugation by an appropriate matrix. Most of the subsets $Y_i$ can be treated in the same manner, following the general procedure outlined above. 
\subsection*{Subset $Y_3$:}\begin{equation*}
	x=\left(\begin{matrix}
		1 &  x_{1,2} & x_{1,3} & x_{1,4}  \\
		0 & 1 & x_{2,3} & x_{2,4}  \\
		0 & 0 & 1 & 0  \\
		0 & 0 & 0 & 1 
	\end{matrix}\right), x^{A}=\left(\begin{matrix}
		1 & 1 & 0 & 0  \\
		0 & 1 & 1 & 0  \\
		0 & 0 & 1 & 0  \\
		0 & 0 & 0 & 1
	\end{matrix}\right)
\end{equation*}
Where matrix $A$ is defined by
\begin{align*}
		&d_{1} = 1, d_{2} = \frac{1}{x_{1,2}}, d_{3} = \frac{1}{x_{1,2}x_{2,3}}, d_{4} = 1, \\
	&a_{1,2} =0 , a_{1,3} = 0, a_{1,4} = 0,  a_{2,3} = \frac{-x_{1,3}}{x_{1,2}^2x_{2,3}}, a_{2,4} = \frac{x_{2,4}x_{1,3}-x_{1,4}x_{2,3}}{x_{1,2}x_{2,3}}, a_{3,4} = \frac{-x_{2,4}}{x_{2,3}}.
\end{align*}

\subsection*{Subset $Y_4$:}A few specific cases require a more detailed analysis. The initial idea for $Y_4$ was to proceed as in the previous case, however we were unable to identify a single representative to which all elements could be conjugated. As a result, we divided the analysis into a few separate cases to account for the distinct behaviors observed. First assume $x_{1,2} x_{2,4} + x_{1,3} x_{3,4}\neq 0$ and $x_{1,3}\neq 0$.\begin{equation*}
	x=\left(
\right).
\end{equation*}
Where matrix $A$ is defined by
\begin{align*}
&d_{1} = 1, d_{2} = \frac{1}{x_{1,2}}, d_{3} = \frac{ x_{3,4}}{x_{1,2} x_{2,4} + x_{1,3} x_{3,4}}, d_{4} = \frac{1}{x_{1,2} x_{2,4} + x_{1,3} x_{3,4}},\\
&a_{1,2} = 0, a_{1,3} = 0, a_{1,4} = 0, a_{2,3} = \frac{ x_{2,4}}{x_{1,2} x_{2,4} + x_{1,3} x_{3,4}}, a_{2,4} = 0, a_{3,4} = -\frac{ x_{1,4} }{x_{1,2} x_{1,3} x_{2,4} + x_{1,3}^{2} x_{3,4}}.
\end{align*}
Now assume $x_{1,2} x_{2,4} + x_{1,3} x_{3,4}\neq 0$ and $x_{1,3}= 0$.\begin{equation*}
	x=\left(%
\right).
\end{equation*}
Where matrix $A$ is defined by
\begin{align*}
&d_{1} = 1, d_{2} = \frac{1}{x_{1,2}}, d_{3} = \frac{ x_{3,4}}{x_{1,2} x_{2,4}}, d_{4} = \frac{1}{x_{1,2} x_{2,4}},\\
&a_{1,2} = 0, a_{1,3} = 0, a_{1,4} = 0, a_{2,3} = \frac{1}{x_{1,2}}, a_{2,4} = \frac{ -x_{1,4}}{x_{1,2}^{2} x_{2,4}}, a_{3,4} = 0.
\end{align*}
Now assume $x_{1,2} x_{2,4} + x_{1,3} x_{3,4}= 0$ and $x_{1,3}\neq 0$.\begin{equation*}
	x=\left(%
\right).
\end{equation*}
Where matrix $A$ is defined by
\begin{align*}
&d_{1} = 1, d_{2} = \frac{1}{x_{1,2}}, d_{3} = 1, d_{4} = \frac{1}{x_{3,4}}, \\
&a_{1,2} = 0, a_{1,3} = 0, a_{1,4} = 0, a_{2,3} = -\frac{ x_{1,3}}{x_{1,2}}, a_{2,4} = 0, a_{3,4} = -\frac{ x_{1,4}}{x_{1,3} x_{3,4}}.
\end{align*}
Now assume $x_{1,2} x_{2,4} + x_{1,3} x_{3,4}= 0$ and $x_{1,3}= 0$.\begin{equation*}
	x=\left(%
\right).
\end{equation*}
Where matrix $A$ is defined by
\begin{align*}
&d_{1} = 1, d_{2} = \frac{1}{x_{1,2}}, d_{3} = 1, d_{4} = \frac{1}{x_{3,4}}, \\
&a_{1,2} = 0, a_{1,3} = 0, a_{1,4} = 0, a_{2,3} = 0, a_{2,4} = -\frac{x_{1,4}}{x_{1,2} x_{3,4}}, a_{3,4} = 0.
\end{align*}
\subsection*{Subset $Y_5$:}\begin{equation*}
	x=\left(%
\right).
\end{equation*}
Where matrix $A$ is defined by
\begin{align*}
	&d_{1} = 1, d_{2} = \frac{1}{x_{1,2}}, d_{3} = 1, d_{4} = \frac{1}{x_{1,2}x_{2,4}},\\
	&a_{1,2} =\frac{x_{1,4}}{x_{1,2}x_{2,4}} , a_{1,3} = 0, a_{1,4} = 0,  a_{2,3} = \frac{-x_{1,3}}{x_{1,2}}, a_{2,4} = 0, a_{3,4} = 0.
\end{align*}
\begin{equation*}
	x=\left(%
\right).
\end{equation*}
Where matrix $A$ is defined by
\begin{align*}
	&d_{1} = 1, d_{2} = \frac{1}{x_{1,2}}, d_{3} = 1, d_{4} = 1,\\
	&a_{1,2} =0 , a_{1,3} = 0, a_{1,4} = 0,  a_{2,3} = \frac{-x_{1,3}}{x_{1,2}}, a_{2,4} = 0, a_{3,4} = \frac{-x_{1,4}}{x_{1,3}}.
\end{align*}

\subsection*{Subset $Y_6$:}First assume $x_{1,4} x_{2,3} - x_{1,3} x_{2,4}\neq 0$.\begin{equation*}
	x=\left(%
\right).
\end{equation*}
Where matrix $A$ is defined by
\begin{align*}
&d_{1} = 1, d_{2} = 1, d_{3} = \frac{1}{x_{2,3}}, d_{4} = -\frac{ x_{2,3}}{x_{1,4} x_{2,3} - x_{1,3} x_{2,4}}, \\
&a_{1,2} = \frac{ x_{1,3} -  x_{2,3}}{x_{2,3}}, a_{1,3} = 0, a_{1,4} = 0, a_{2,3} = 0, a_{2,4} = 0, a_{3,4} = \frac{ x_{1,4} x_{2,3} - {\left( x_{1,3} -  x_{2,3}\right)} x_{2,4}}{x_{1,4} x_{2,3}^{2} - x_{1,3} x_{2,3} x_{2,4}}.
\end{align*}
Now assume $x_{1,4} x_{2,3} - x_{1,3} x_{2,4}= 0$.\begin{equation*}
	x=\left(%
\right).
\end{equation*}
Where matrix $A$ is defined by
\begin{align*}
&d_{1} = 1, d_{2} = 1, d_{3} = \frac{1}{x_{2,3}}, d_{4} = 1, \\
&a_{1,2} = \frac{x_{1,3}}{x_{2,3}}, a_{1,3} = 0, a_{1,4} = 0, a_{2,3} = 0, a_{2,4} = 0, a_{3,4} = -\frac{ x_{2,4}}{x_{2,3}}.
\end{align*}

\subsection*{Subset $Y_7$:}First assume $x_{1,3}\neq 0$.\begin{equation*}
	x=\left(%
\right).
\end{equation*}
Where matrix $A$ is defined by
\begin{align*}
&d_{1} = 1, d_{2} = 1, d_{3} = \frac{1}{x_{2,3}}, d_{4} = \frac{1}{x_{2,3} x_{3,4}}, \\
&a_{1,2} = \frac{ x_{1,3}}{x_{2,3}}, a_{1,3} = 0, a_{1,4} = 0, a_{2,3} = \frac{ -  x_{1,4} x_{2,3} +  x_{1,3} x_{2,4}}{x_{1,3} x_{2,3} x_{3,4}}, \\
&a_{2,4} = 0, a_{3,4} = \frac{ -  x_{1,4}}{x_{1,3} x_{2,3} x_{3,4}}.
\end{align*}
Now assume $x_{1,3}= 0$.
\begin{equation*}
	x=\left(%
\right).
\end{equation*}
Where matrix $A$ is defined by
\begin{align*}
&d_{1} = 1, d_{2} = 1, d_{3} = \frac{1}{x_{2,3}}, d_{4} = \frac{1}{x_{2,3} x_{3,4}}, \\
&a_{1,2} = 0, a_{1,3} = \frac{ x_{1,4}}{x_{2,3} x_{3,4}}, a_{1,4} = 0, a_{2,3} = 0, a_{2,4} = 0,\\
&a_{3,4} = \frac{ -  x_{2,4}}{x_{2,3}^{2} x_{3,4}}.
\end{align*}

\subsection*{Subset $Y_8$:}First assume $x_{1,3}\neq 0$.\begin{equation*}
	x=\left(%
\right).
\end{equation*}
Where matrix $A$ is defined by
\begin{align*}
    &d_{1} = 1, d_{2} = 1, d_{3} = \frac{1}{x_{1,3}}, d_{4} = \frac{1}{x_{1,3} x_{3,4}},\\
    &a_{1,2} = 0, a_{1,3} = 0, a_{1,4} = 0, a_{2,3} = \frac{ x_{2,4}}{x_{1,3} x_{3,4}},\\
    &a_{2,4} = 0, a_{3,4} = \frac{-  x_{1,4}}{x_{1,3}^{2} x_{3,4}}.
\end{align*}

Now assume $x_{1,3}= 0$.\begin{equation*}
	x=\left(%
\right).
\end{equation*}
Where matrix $A$ is defined by
\begin{align*}
    &d_{1} = 1, d_{2} = 1, d_{3} = 1, d_{4} = \frac{1}{x_{3,4}}, \\
    &a_{1,2} = 0, a_{1,3} = \frac{ x_{1,4} -  x_{3,4}}{x_{3,4}}, a_{1,4} = 0, \\
    &a_{2,3} = \frac{x_{2,4}}{x_{3,4}}, a_{2,4} = 0, a_{3,4} = 0.
\end{align*}
\subsection*{Subset $Y_1$:} In this case, the set corresponds to the derived subgroup. Therefore, we will divide the analysis into cases according to the form of the matrix. First consider elements $x\in Y_1$ of the form, where $x_{1,4}\neq 0$
\begin{equation*}
	x=\left(%
\right).
\end{equation*}
Where matrix $A$ is defined by
\begin{align*}
	&d_{1} = x_{1,4}, d_{2} = 1, d_{3} = 1, d_{4} = 1, \\
	&a_{1,2} =0 , a_{1,3} = 0, a_{1,4} = 0, a_{2,3} = 0, a_{2,4} = 0, a_{3,4} = 0.
\end{align*}

Now consider elements $x\in Y_1$ of the form, where $x_{1,3}\neq 0$.
\begin{equation*}
	x=\left(%
\right).
\end{equation*}
Where matrix $A$ is defined by
\begin{align*}
	&d_{1} = 1, d_{2} = 1, d_{3} = \frac{1}{x_{1,3}}, d_{4} = 1, \\
	&a_{1,2} =0 , a_{1,3} = 0, a_{1,4} = 0, a_{2,3} = 0, a_{2,4} = 0, a_{3,4} = \frac{-x_{1,4}}{x_{1,3}}.
\end{align*}

Now consider elements $x\in Y_1$ of the form, where $x_{1,3},x_{2,4}\neq 0$.

\begin{equation*}
	x=\left(%
\right).
\end{equation*}
Where matrix $A$ is defined by
\begin{align*}
	&d_{1} = 1, d_{2} = 1, d_{3} = \frac{1}{x_{1,3}}, d_{4} = \frac{1}{x_{2,4}}, \\
	&a_{1,2} =0 , a_{1,3} = 0, a_{1,4} = 0,  a_{2,3} = 0, a_{2,4} = 0,  a_{3,4} = \frac{-x_{1,4}}{x_{1,3}x_{2,4}}.
\end{align*}
Finally consider elements $x\in Y_1$ of the form, where $x_{2,4}\neq 0$.
\begin{equation*}
	x=\left(%
\right).
\end{equation*}
Where matrix $A$ is defined by
\begin{align*}
	&d_{1} = 1, d_{2} = x_{2,4}, d_{3} = 1, d_{4} = 1, \\
	&a_{1,2} =x_{1,4} , a_{1,3} = 0, a_{1,4} = 0, a_{2,3} = 0, a_{2,4} = 0, a_{3,4} = 0.
\end{align*}

All cases have been computed and we found that every element of $UT_4(\mathbb{F})$ can be conjugated to an element of the set $\mathcal{S}$, where
\begin{equation*}
\begin{aligned}
\mathcal{S}= \Biggl\{&\left(%
\right)\Biggl\}\cup\{\mathbbm{1}\}.
\end{aligned}
\end{equation*}
\\

\textbullet \textbf{Claim.} There are 16 orbits under the action of $T_4(\mathbb{F})$.\\

It has already been mentioned that the orbits under the action of $T_4(\mathbb{F})$ remain within each $Y_i$. Therefore, to verify that the elements in the set $\mathcal{S}$ form a set of orbit representatives, it suffices to check that the conjugacy class under the action of $T_4(\mathbb{F})$ of a specific element does not contain any other element of the set $\mathcal{S}$. We will present a specific case, as the others follow in an analogous manner.

The conjugacy class of the candidate for orbit representative in $Y_5$
$$\left(\begin{matrix}
		1 & 1 & 0 & 0  \\
		0 & 1 & 0 & 1  \\
		0 & 0 & 1 & 0  \\
		0 & 0 & 0 & 1  
	\end{matrix}\right),$$
has the form
$$\left(\begin{matrix}
1 & \frac{d_2}{d_1} & \frac{a_{2,3}}{d_1} & \frac{a_{2,4}d_2-a_{1,2}d_4}{d_1d_2} \\
0 & 1 & 0 & \frac{d_4}{d_2} \\
0 & 0 & 1 & 0 \\
0 & 0 & 0 & 1
\end{matrix}\right),$$
for $d_1,d_2,d_4\in\mathbb{F}\setminus\{0\}$ and $a_{1,2}, a_{2,3}, a_{2,4}\in \mathbb{F}$. Notice that this conjugacy class does not contain
$$\left(\begin{matrix}
		1 & 1 & 0 & 0  \\
		0 & 1 & 0 & 0  \\
		0 & 0 & 1 & 0  \\
		0 & 0 & 0 & 1  
	\end{matrix}\right),$$
once $d_4$ is nonzero. So both elements belong to different orbits under the action of $T_4(\mathbb{F})$.

After performing all computations, we conclude that the set $\mathcal{S}$ constitutes a complete set of orbit representatives for the action of $T_4(\mathbb{F})$. In particular, this establishes that the number of automorphism orbits for $n=4$ is finite. 

\section{The unitriangular group $UT_5(\mathbb{F})$}
 In this section we prove,
\begin{thmE}
    The unitriangular group $UT_5(\mathbb{F})$ has finitely many automorphism orbits. Moreover, the number of orbits under the action of $T_5(\mathbb{F})$ is 61.
\end{thmE}

\begin{proof}
We proceeded in an analogous manner of $UT_4(\mathbb{F})$: for each subset of the partition, we identified a finite set of orbit representatives under the action of $T_5(\mathbb{F})$. 
However, determining these representatives required a case-by-case analysis. We had to consider the cases where each entry is either zero or nonzero in order to solve the system of equations, this leads to a large number of cases. Moreover, there are more equations than in the case for dimension $4$. Therefore, we used software SageMath \cite{Sage} to solve these systems of equations. Our work can be reproduced, we provide the code with all systems of equations on the GitHub repository\footnote{\url{https://github.com/juhmit/unitriangular_n5.git}}.

Based on the computations presented in the Appendix, we obtained a finite list of 61 representatives in $UT_5(\mathbb{F})$ for the action of $T_5(\mathbb{F})$. 
We list them explicitly below.
Consider the following partition of $UT_5(\mathbb{F})$

\begin{align*}
	UT_5(\mathbb{F})=& Y_1\cup Y_2\cup Y_3\cup Y_4\cup Y_5\cup Y_6\cup Y_7\cup Y_8\cup Y_9\cup Y_{10}\cup Y_{11}\cup\\
	&\cup Y_{12}\cup Y_{13}\cup Y_{14}\cup Y_{15}\cup Y_{16},   
\end{align*}
where	

$Y_1=\left\{ \left(
\right)\colon x_{i,j}\in \mathbb{F}, \quad x_{1,2},x_{3,4},x_{4,5}\neq 0\right\}$.

Consider the series of characteristic subgroups,
$$\{\mathbbm{1}\}< \gamma_4< \gamma_3- N_{24}< \gamma_2< G.$$ 

The orbits representatives in $\gamma_4$ are

{\footnotesize
\begin{align*}
    \left(%
\right).
\end{align*}}%

The orbits representatives in 
$\gamma_3\setminus\gamma_4$ are 

{\footnotesize
\begin{align*}
    \left(%
\right).
\end{align*}}%

The orbits representatives in $N_{24}\setminus\gamma_3$ are 

{\footnotesize
\begin{align*}
    \left(%
\right).
\end{align*}}%

The orbits representatives in $\gamma_2\setminus N_{24}$ are 

{\footnotesize
\begin{align*}
    \left(%
\right).
\end{align*}}%

The final case $G \setminus \gamma_2$ is more intricate, as it involves all elements from $Y_i$, for $i=2,\dots,16$. 
To handle this situation, it is necessary to organize the group in a suitable manner. 
In particular, we focus on certain subgroups, mainly the maximal abelian normal subgroups $N_{1,2}, N_{2,3}, N_{3,4}, N_{4,5}$, the characteristic subgroup $\gamma_2$, and various products among them. Denote
\begin{gather*}
    N_{1,2}:=H_1,\\
    N_{2,3}:=H_2,\\
    N_{3,4}:=H_3,\\
    N_{4,5}:=H_4,\\
    \gamma_2:=Z_3.
\end{gather*}

The orbits representatives in $H_1Z_3\setminus Z_3$ are

{\footnotesize
\begin{gather*}
    \left(
\right).
\end{gather*}}%

The orbits representatives in    $H_2Z_3\setminus Z_3$ are

{\footnotesize
\begin{gather*}
    \left(%
\right).
\end{gather*}}%

The orbits representatives in    $H_3Z_3\setminus Z_3$ are

{\footnotesize
\begin{gather*}
    \left(%
\right).
\end{gather*}}%

The orbits representatives in    $H_4Z_3\setminus Z_3$ are

{\footnotesize
\begin{gather*}
    \left(%
\right).
\end{gather*}}%

The orbits representatives in    $H_1H_2Z_3\setminus (H_1Z_3\cup H_2Z_3)$ are

{\footnotesize
\begin{align*}
    &\left(%
\right).
\end{align*}}%

The orbits representatives in    $H_1H_3Z_3\setminus (H_1Z_3\cup H_3Z_3)$ are

{\footnotesize
\begin{align*}
    &\left(%
\right).
\end{align*}}%

The orbits representatives in    $H_1H_4Z_3\setminus (H_1Z_3\cup H_4Z_3)$ are

{\footnotesize
\begin{align*}
    &\left(%
\right).
\end{align*}}%

The orbits representatives in    $H_2H_3Z_3\setminus (H_2Z_3\cup H_2Z_3)$ are

{\footnotesize
\begin{align*}
    &\left(%
\right).
\end{align*}}%

The orbits representatives in    $H_1H_2H_3Z_3\setminus (H_1H_2Z_3\cup H_1H_3Z_3\cup H_2H_3Z_3)$ are

{\footnotesize
\begin{align*}
    &\left(%
\right).
\end{align*}}%

The orbits representatives in    $H_1H_3H_4Z_3\setminus (H_1H_4Z_3\cup H_1H_3Z_3\cup H_3H_4Z_3)$   are

{\footnotesize
\begin{align*}
    &\left(%
\right).
\end{align*}}%

The orbits representatives in    $H_1H_2H_4Z_3\setminus (H_1H_2Z_3\cup H_1H_4Z_3\cup H_2H_4Z_3)$ are

{\footnotesize
\begin{align*}
    &\left(%
\right).
\end{align*}}%

The orbits representatives in    $H_1H_2H_3H_4Z_3\setminus (H_1H_2H_3Z_3\cup H_1H_3H_4Z_3\cup H_1H_2H_4Z_3\cup H_2H_3H_4Z_3)$ are

{\footnotesize
\begin{align*}
    &\left(%
\right).
\end{align*}}%

The number of orbit representatives obtained for each set is presented in the table below.

{\footnotesize
\begin{table}[H]
\centering
\begin{tabular}[t]{lc}\toprule
    Subsets & number of orbit rep.  \\\midrule
    $Z_3$  & 15\\ \hline
    $H_1Z_3\setminus Z_3$ & 6 \\
    $H_2Z_3\setminus Z_3$  & 5\\ 
    $H_3Z_3\setminus Z_3$  & 5\\ 
    $H_4Z_3\setminus Z_3$  & 6\\  \hline
    $H_1H_2Z_3\setminus (H_1Z_3\cup H_2Z_3)$ & 2 \\
    $H_1H_3Z_3\setminus (H_1Z_3\cup H_3Z_3)$  & 4\\ 
    $H_1H_4Z_3\setminus (H_1Z_3\cup H_4Z_3)$  & 3\\ 
    $H_2H_3Z_3\setminus (H_2Z_3\cup H_2Z_3)$  & 2\\ 
    $H_2H_4Z_3\setminus (H_2Z_3\cup H_3Z_3)$  & 4\\ 
    $H_3H_4Z_3\setminus (H_3Z_3\cup H_4Z_3)$  & 2\\ \hline
    $H_1H_2H_3Z_3\setminus (H_1H_2Z_3\cup H_1H_3Z_3\cup H_2H_3Z_3)$ & 1 \\
    $H_1H_3H_4Z_3\setminus (H_1H_4Z_3\cup H_1H_3Z_3\cup H_3H_4Z_3)$  & 2\\ 
    $H_1H_2H_4Z_3\setminus (H_1H_2Z_3\cup H_1H_4Z_3\cup H_2H_4Z_3)$  & 2\\ 
    $H_2H_3H_4Z_3\setminus (H_2H_3Z_3\cup H_2H_4Z_3\cup H_3H_4Z_3)$  &1\\ \hline
    $H_1H_2H_3H_4Z_3\setminus (H_1H_2H_3Z_3\cup H_1H_3H_4Z_3\cup H_1H_2H_4Z_3\cup H_2H_3H_4Z_3)$ & 1\\\midrule
    total & 61\\\bottomrule   
\end{tabular}
\caption{Number of orbits representatives under action of $T_5(\mathbb{F})$.}
\label{tab:my-table2}
\end{table}
}
\end{proof}

\section{The unitriangular group $UT_n(\mathbb{F})$, for $n\geq6$}

In \cite{KP05} it was shown that if $K$ is a field such that the number of orbits of $\Aut(K)$ is finite, then $K$ must be finite itself. We shall make use of this result in order to establish the next theorem.

\begin{thm}
Let $\mathbb{F}$ be an infinite field. Then the unitriangular group $UT_n(\mathbb{F})$ has infinitely many automorphism orbits for all $n \geq 6$.
\end{thm}

\begin{proof}
    First assume $n=6$. For any $x\in\mathbb{F}$  we  denote  by $\boldsymbol{x}$ the following matrix
    $$\boldsymbol{x}={\footnotesize\left(\begin{matrix}
        1 & x & 1 & 0 & 0 & 0 \\
        0 & 1 & 0 & 0 & 1 & 0 \\
        0 & 0 & 1 & 1 & 0 & 0 \\
        0 & 0 & 0 & 1 & 0 & 1 \\
        0 & 0 & 0 & 0 & 1 & 1 \\
        0 & 0 & 0 & 0 & 0 & 1
        \end{matrix}\right)_{6\times6},} $$
    The matrix $\boldsymbol{x}$ is defined as follows: it has $1$’s on the diagonal and in the positions $(5,6)$, $(4,6)$, $(3,4)$, $(2,5)$, and $(1,3)$, the element $x$ in position $(1,2)$, and $0$ in all other positions. Also note that we can write $\boldsymbol{x}$ as product of generators in the following way
    $$\boldsymbol{x}=e_{5,6}[e_{4,5},e_{5,6}]e_{3,4}[[e_{2,3},e_{3,4}],e_{4,5}][e_{1,2},e_{2,3}]xe_{1,2}.$$ 

We consider the action of $\Aut(UT_6(\mathbb{F}))$ on the quotient $UT_6(\mathbb{F})/\gamma_4$. 
By Corollary 4 and 5 of \cite{L83}, this action is given by the group 
\[
((\mathcal{J} \rtimes W)\rtimes D)\rtimes \overline{\Aut}(\mathbb{F}),
\]
since in this case it is not necessary to take central and extremal automorphisms into account.

Let $\Phi$ be an automorphism of $UT_6(\mathbb{F})$. We will prove that if $\boldsymbol{x}^{\Phi}=\boldsymbol{y}$, then $x^{\varphi}=y$ for some $\varphi\in Aut(\mathbb{F})$. Modulo $\gamma_4$, we may write $\Phi=kwd\overline{\varphi}$ where $\overline{\varphi}$ denotes the automorphism of $UT_6(\mathbb{F})$ induced by a field automorphism $\varphi\in \Aut(\mathbb{F})$, $k\in \mathcal{J}$, $w\in W$, $d\in D$. Let us compute $x^{\Phi}=\boldsymbol{x}^{\overline{\varphi} dwuk}$ step by step, applying each automorphism successively from left to right.  By the decomposition of $\boldsymbol{x}$ as product of generators we have that 
$${\footnotesize(\boldsymbol{x})^{\overline{\varphi}}=\left(\begin{matrix}
        1 & x^{\varphi} & 1 & 0 & 0 & 0 \\
        0 & 1 & 0 & 0 & 1 & 0 \\
        0 & 0 & 1 & 1 & 0 & 0 \\
        0 & 0 & 0 & 1 & 0 & 1 \\
        0 & 0 & 0 & 0 & 1 & 1 \\
        0 & 0 & 0 & 0 & 0 & 1
        \end{matrix}\right)}$$%

Let $d$ be the diagonal matrix $\diag[d_1,d_2,d_3,d_4,d_5,d_6]$.  Then

$${\footnotesize \boldsymbol{x}^{\overline{\varphi} d}=\left(\begin{matrix}
        1 & \frac{d_2}{d_1}x^{\varphi} & \frac{d_3}{d_1} & 0 & 0 & 0 \\
        0 & 1 & 0 & 0 & \frac{d_5}{d_2} & 0 \\
        0 & 0 & 1 & \frac{d_4}{d_3}& 0 & 0 \\
        0 & 0 & 0 & 1 & 0 & \frac{d_6}{d_4} \\
        0 & 0 & 0 & 0 & 1 & \frac{d_6}{d_5} \\
        0 & 0 & 0 & 0 & 0 & 1
        \end{matrix}\right), }$$%
             
Now, we assume that the flip automorphism $w$ is not trivial.

$${\footnotesize
\boldsymbol{x}^{\overline{\varphi}dw}=\left(\begin{matrix}
1 & \frac{d_6}{d_5} & \frac{-d_6}{d_4} & \frac{-d_6}{d_3} & \frac{d_6}{d_2} & \frac{d_6+d_6x^{\varphi}}{d_1} \\
0 & 1 & 0 & 0 & \frac{d_5}{d_2} & \frac{d_5x^{\varphi}}{d_1} \\
0 & 0 & 1 & \frac{d_4}{d_3} & 0 & \frac{-d_4}{d_1} \\
0 & 0 & 0 & 1 & 0 & \frac{-d_3}{d_1} \\
0 & 0 & 0 & 0 & 1 & \frac{d_2x^{\varphi}}{d_1} \\
0 & 0 & 0 & 0 & 0 & 1
\end{matrix}\right).}$$%
        
Since we are considering the quotient by $\gamma_4$ we have 

$${\footnotesize\boldsymbol{x}^{\overline{\varphi}dw}==\left(\begin{matrix}
1 & \frac{d_6}{d_5} & \frac{-d_6}{d_4} & \frac{-d_6}{d_3} & * & * \\
0 & 1 & 0 & 0 & \frac{d_5}{d_2} & * \\
0 & 0 & 1 & \frac{d_4}{d_3} & 0 & \frac{-d_4}{d_1} \\
0 & 0 & 0 & 1 & 0 & \frac{-d_3}{d_1} \\
0 & 0 & 0 & 0 & 1 & \frac{d_2x^{\varphi}}{d_1} \\
0 & 0 & 0 & 0 & 0 & 1
\end{matrix}\right),}$$%
where the star-marked entries depend on $w$.
        
Assume that $k\in\mathcal{J}$ is the conjugation by unitriangular matrix $A=(a_{i,j})$, then 
    $$\boldsymbol{x}^{\overline{\varphi}dwk}=\boldsymbol{y}\implies A\boldsymbol{x}^{\overline{\varphi}dw}A^{-1}=\boldsymbol{y}\implies A\boldsymbol{x}^{\overline{\varphi}dw}=\boldsymbol{y}A.$$
    Denote $(x'_{ij})=\boldsymbol{x}^{\overline{\varphi}dw}$. We compare the elements of matrices $A(x'_{ij})$ and $\boldsymbol{y}A$ at positions:
    
    \begin{align*}
    &(1,2): x'_{1,2}=y, \\
    &(1,3):  x'_{1,3} = ya_{2,3} +1 , \\
    &(2,4): a_{2,3}x'_{3,4} = 0 , \\
    &(2,5):  x'_{2,5} = 1, \\
    &(3,4): x'_{3,4} = 1, \\
    &(3,5): 0 = a_{4,5},\\
    &(4,6): x'_{4,6} +a_{4,5}x'_{5,6} = 1,\\
    &(5,6):  x'_{5,6} = 1.
    \end{align*}
    We can conclude that 
    \begin{align*}
    & x'_{1,2}=y, \quad  x'_{1,3} = 1 , \quad a_{2,3} = 0 , \quad  x'_{2,5} = 1, \\
    & x'_{3,4} =1,\quad a_{4,5}=0,\quad x'_{4,6}  = 1,\quad  x'_{5,6} = 1.
    \end{align*}
Therefore
\begin{align*}
    & \frac{d_6}{d_5}=y, \quad  -\frac{d_6}{d_4} = 1, \quad  \frac{d_5}{d_2} = 1, \\
    & \frac{d_4}{d_3} = 1,\quad -\frac{d_3}{d_1}  = 1,\quad  \frac{d_2}{d_1}x^{\varphi} = 1.
    \end{align*}
    
From the relations above we obtain that
$$\frac{d_6}{d_5}\frac{d_2}{d_1}x^{\varphi}=y, \quad \frac{d_3}{d_1}\frac{d_6}{d_4}=1, \quad {d_5}={d_2}, \quad {d_4}={d_3}.$$
So $x^{\varphi}=y$.

Now assume $w\in \mathcal{W}$ is trivial. Simirlarly, we denote $(x'_{ij})=\boldsymbol{x}^{\overline{\varphi}d}$ and again we compare the elements of matrices $A(x'_{ij})$ and $\boldsymbol{y}A$ in entries
    \begin{align*}
    &(1,2): \frac{d_2}{d_1}x^{\varphi}=y, \\
    &(1,3):  \frac{d_3}{d_1} = 1 , \\
    &(2,5):  \frac{d_5}{d_2} = 1, \\
    &(3,4): \frac{d_4}{d_3} = 1, \\
    &(4,6): \frac{d_6}{d_4}  = 1,\\
    &(5,6):  \frac{d_6}{d_5} = 1.
    \end{align*}
From the relations above we obtain that ${d_1}={d_3}={d_4}={d_6}={d_5}={d_2}$ so the relations also implies that $x^{\varphi}=y$.

This implies that if two elements 
    $${\footnotesize (\boldsymbol{x^{\varphi}})=\left(\begin{matrix}
        1 & x^{\varphi} & 1 & 0 & 0 & 0 \\
        0 & 1 & 0 & 0 & 1 & 0 \\
        0 & 0 & 1 & 1 & 0 & 0 \\
        0 & 0 & 0 & 1 & 0 & 1 \\
        0 & 0 & 0 & 0 & 1 & 1 \\
        0 & 0 & 0 & 0 & 0 & 1
        \end{matrix}\right)\quad \textrm{and} \quad \boldsymbol{y}=\left(\begin{matrix}
        1 & y & 1 & 0 & 0 & 0 \\
        0 & 1 & 0 & 0 & 1 & 0 \\
        0 & 0 & 1 & 1 & 0 & 0 \\
        0 & 0 & 0 & 1 & 0 & 1 \\
        0 & 0 & 0 & 0 & 1 & 1 \\
        0 & 0 & 0 & 0 & 0 & 1
        \end{matrix}\right) } $$%
lie in the same automorphism orbit, then $x^{\varphi} = y$ for some field automorphism $\varphi$. Consequently, $UT_6(\mathbb{F})$ admits infinitely many orbit representatives whenever $\mathbb{F}$ is infinite. This proves the Theorem for $n=6$.

Now we extend to $n\geq 6$. One can view $UT_n(\mathbb{F})$ as $\left(\begin{matrix}
	UT_6(\mathbb{F}) & \star  \\
	0 & UT_{n-6}(\mathbb{F}) 
\end{matrix}\right)$ so that  $\left(\begin{matrix}
X & \star  \\
0 & I 
\end{matrix}\right)$ and $\left(\begin{matrix}
Y & \star  \\
0 & I 
\end{matrix}\right)$ are in the same automorphism orbit in $UT_n(\mathbb{F})$ if and only if $X$ and $Y$ are in the same automorphism orbit, for $X,Y\in UT_6(\mathbb{F})$. So we have even fewer automorphisms to consider, as we can disregard flip and extremal automorphisms. Nevertheless, the same argument applies to prove that the number of automorphism orbits is infinite, as desired.
\end{proof}

\noindent {\bf Acknowledgments.} The authors would like to express their sincere gratitude to Professor Alexandre Zalesski for proposing the research problem and for his valuable insights and stimulating discussions throughout the development of this work.

\end{document}


\maketitle

	\maketitle
	\tableofcontents

Appendix A presents a set $\mathcal{S}$ of elements from $UT_5(\mathbb{F})$. Every element in $UT_5(\mathbb{F})$ can be mapped via conjugation by an element of $T_5(\mathbb{F})$ to an element in $\mathcal{S}$.

Appendix B through Q contain detailed computations referenced throughout the paper. These include matrix conjugations, solutions to systems of equations, and verifications of orbit representatives under group actions. Due to the large number of cases and the need for accuracy, most computations were performed using the computer algebra system SageMath \cite{sage}.
	
	For each subset of the partition, we identified a finite set of orbit representatives under the action of $T_5(\mathbb{F})$. This was done by selecting an element $x\in Y_i$, $1\leq i\leq16$, and finding a matrix $A\in T_n(\mathbb{F})$ that conjugates it to a chosen representative.
	
	$$A=\left(
\right).$$
	
	However, determining these representatives required a case-by-case analysis. We considered all possibilities where each matrix entry was either zero or nonzero to solve the resulting systems of equations, which led to a large number of cases.

	\newgeometry{margin=0.5in}
	
	\appendix
	\section{Elements of Subset $\mathcal{S}$}
	
	$$\left(%
\right).$$
	
	In what follows, whenever $x_{ij}$ is written to represent an entry of the matrix $x$ it indicates that $x_{ij}$ is a non zero  element of $\mathbb{F}$. 
	\fontsize{10}{10}{
		\section{Subcases of $Y_1$}

\begin{equation*} 
	x=\left(%
\right)
\end{equation*} \newline
Where matrix $A$ has entries
\begin{align*}
	&d_{1} = 1, d_{2} = 1, d_{3} = \frac{1}{x_{13}}, d_{4} = \frac{1}{x_{24}}, d_{5} = \frac{1}{x_{13} x_{35}}, \\ 
	&a_{12} = 1, a_{13} = 1, a_{14} = 1, a_{15} = 1, \\ 
	&a_{23} = 1, a_{24} = 1, a_{25} = 1, \\ 
	&a_{34} = -\frac{x_{14} - x_{24}}{x_{13} x_{24}}, a_{35} = -\frac{x_{15} x_{24} - x_{14} x_{25} + {\left(x_{13} x_{14} - x_{13} x_{24}\right)} x_{35}}{x_{13}^{2} x_{24} x_{35}}, \\ 
	&a_{45} = \frac{x_{13} x_{35} - x_{25}}{x_{13} x_{24} x_{35}}
\end{align*}

		\section{Subcases of $Y_2$}

\begin{equation*}x=\left(
\right)
\end{equation*}
Where matrix $A$ has entries
\begin{align*}
	&d_{1} = 1, d_{2} = \frac{1}{x_{12}}, d_{3} = \frac{x_{35}}{x_{12} x_{25} + x_{13} x_{35}}, d_{4} = 1, d_{5} = \frac{1}{x_{12} x_{25} + x_{13} x_{35}}, \\ 
	&a_{12} = 1, a_{13} = 1, a_{14} = 1, a_{15} = 1, \\ 
	&a_{23} = \frac{x_{25}}{x_{12} x_{25} + x_{13} x_{35}}, a_{24} = 1, a_{25} = 1, \\ 
	&a_{34} = -\frac{x_{12}}{x_{13}}, a_{35} = -\frac{{\left(x_{12} - 1\right)} x_{13} x_{35} + x_{15} + {\left(x_{12}^{2} - x_{12}\right)} x_{25}}{x_{12} x_{13} x_{25} + x_{13}^{2} x_{35}}, a_{45} = 1
\end{align*}

Now assume $x13 = \frac{-x_{12}x_{25}}{x_{35}}$
\begin{equation*}x=\left(
\right)
\end{equation*}
Where matrix $A$ has entries
\begin{align*}
	&d_{1} = 1, d_{2} = \frac{1}{x_{12}}, d_{3} = \frac{ x_{35}}{x_{12} x_{25} + x_{13} x_{35}}, d_{4} = 1, d_{5} = \frac{1}{x_{12} x_{25} + x_{13} x_{35}}, \\ 
	&a_{12} = 1, a_{13} = 1, a_{14} = 1, a_{15} = 1, \\ 
	&a_{23} = \frac{ x_{25}}{x_{12} x_{25} + x_{13} x_{35}}, a_{24} = 1, a_{25} = 1, \\ 
	&a_{34} = -\frac{ x_{12} +  x_{14}}{x_{13}}, a_{35} = 1, \\ 
	&a_{45} = -\frac{ x_{15} + {\left( x_{12}^{2} +  x_{12} x_{13} -  x_{12}\right)} x_{25} + {\left( x_{13}^{2} + {\left( x_{12} - 1\right)} x_{13}\right)} x_{35}}{x_{12} x_{14} x_{25} + x_{13} x_{14} x_{35}}
\end{align*}
Now assume $x13 = \frac{-x_{12}x_{25}}{x_{35}}$
\begin{equation*}x=\left(
\right)
\end{equation*}
Where matrix $A$ has entries
\begin{align*}
	&d_{1} = 1, d_{2} = \frac{1}{x_{12}}, d_{3} = 1, d_{4} = \frac{1}{x_{12} x_{24}}, d_{5} = \frac{1}{x_{35}}, \\ 
	&a_{12} = 1, a_{13} = 1, a_{14} = 1, a_{15} = 1, \\ 
	&a_{23} = -\frac{ x_{13}}{x_{12}}, a_{24} = 1, a_{25} = 1, \\ 
	&a_{34} = -\frac{ x_{14} + {\left( x_{12}^{2} -  x_{12}\right)} x_{24}}{x_{12} x_{13} x_{24}}, a_{35} = -\frac{ x_{12} x_{15} x_{24} - {\left( x_{13} x_{14} - {\left( x_{12}^{2} -  x_{12}\right)} x_{24}\right)} x_{35}}{x_{12} x_{13} x_{24} x_{35}}, \\ 
	&a_{45} = -\frac{ x_{13}}{x_{12} x_{24}}
\end{align*}

\begin{equation*}x=\left(
\right)
\end{equation*} \newline
Where matrix $A$ has entries
\begin{align*}
	&d_{1} = 1, d_{2} = \frac{1}{x_{12}}, d_{3} = 1, d_{4} = \frac{1}{x_{12} x_{24}}, d_{5} = \frac{1}{x_{35}}, \\ 
	&a_{12} = 1, a_{13} = 1, a_{14} = 1, a_{15} = 1, \\ 
	&a_{23} = -\frac{ x_{13}}{x_{12}}, a_{24} = 1, a_{25} = 1, \\ 
	&a_{34} = -\frac{ x_{14} + {\left( x_{12}^{2} -  x_{12}\right)} x_{24}}{x_{12} x_{13} x_{24}}, a_{35} = \frac{ x_{12} x_{14} x_{25} + {\left( x_{13} x_{14} - {\left( x_{12}^{2} -  x_{12}\right)} x_{24}\right)} x_{35}}{x_{12} x_{13} x_{24} x_{35}}, \\ 
	&a_{45} = -\frac{ x_{12} x_{25} +  x_{13} x_{35}}{x_{12} x_{24} x_{35}}
\end{align*}

\begin{equation*}x=\left(\begin{matrix}
		1 & x_{12} & x_{13} & x_{14} & x_{15} \\
		0 & 1 & 0 & x_{24} & x_{25} \\
		0 & 0 & 1 & 0 & x_{35}  \\
		0 & 0 & 0 & 1 & 0  \\
		0 & 0 & 0 & 0 & 1 \\
	\end{matrix}\right),x^A=\left(\begin{matrix}
		1 & 1 & 0 & 0 & 0 \\
		0 & 1 & 0 & 1 & 0 \\
		0 & 0 & 1 & 0 & 1 \\
		0 & 0 & 0 & 1 & 0 \\
		0 & 0 & 0 & 0 & 1
	\end{matrix}\right)
\end{equation*} \newline
Where matrix $A$ has entries
\begin{align*}
	&d_{1} = 1, d_{2} = \frac{1}{x_{12}}, d_{3} = 1, d_{4} = \frac{1}{x_{12} x_{24}}, d_{5} = \frac{1}{x_{35}}, \\ 
	&a_{12} = 1, a_{13} = 1, a_{14} = 1, a_{15} = 1, \\ 
	&a_{23} = -\frac{x_{13}}{x_{12}}, a_{24} = 1, a_{25} = 1, \\ 
	&a_{34} = -\frac{x_{14} + {\left(x_{12}^{2} - x_{12}\right)} x_{24}}{x_{12} x_{13} x_{24}}, a_{35} = -\frac{x_{12} x_{15} x_{24} - x_{12} x_{14} x_{25} - {\left(x_{13} x_{14} - {\left(x_{12}^{2} - x_{12}\right)} x_{24}\right)} x_{35}}{x_{12} x_{13} x_{24} x_{35}}, \\
	&a_{45} = -\frac{x_{12} x_{25} + x_{13} x_{35}}{x_{12} x_{24} x_{35}}
\end{align*}

		\section{Subcases of $Y_3$}

\begin{equation*}x=\left(
\right)
\end{equation*} \newline
Where matrix $A$ has entries
\begin{align*}
	&d_{1} = 1, d_{2} = 1, d_{3} = \frac{1}{x_{23}}, d_{4} = \frac{x_{23}}{x_{14} x_{23} - x_{13} x_{24}}, d_{5} = 1, \\ 
	&a_{12} = \frac{x_{13}}{x_{23}}, a_{13} = 1, a_{14} = 1, a_{15} = 1, \\ 
	&a_{23} = 1, a_{24} = 1, a_{25} = 1, \\ 
	&a_{34} = -\frac{x_{24}}{x_{14} x_{23} - x_{13} x_{24}}, a_{35} = \frac{x_{15} x_{24} - x_{14} x_{25}}{x_{14} x_{23} - x_{13} x_{24}}, \\ 
	&a_{45} = -\frac{x_{15} x_{23} - x_{13} x_{25}}{x_{14} x_{23} - x_{13} x_{24}}
\end{align*}

Now assume $x_{14} = \frac{x_{13}x_{24}}{x_{23}}$ and $x_{15} \neq \frac{x_{13}x_{25}}{x_{23}}$
\begin{equation*}x=\left(\begin{matrix}
		1 & 0 & x_{13} & x_{14} & x_{15}  \\
		0 & 1 & x_{23} & x_{24} & x_{25}  \\
		0 & 0 & 1 & 0 & 0  \\
		0 & 0 & 0 & 1 & 0  \\
		0 & 0 & 0 & 0 & 1  \\
	\end{matrix}\right),x^A=\left(\begin{matrix}
		1 & 0 & 0 & 0 & 1 \\
		0 & 1 & 1 & 0 & 0 \\
		0 & 0 & 1 & 0 & 0 \\
		0 & 0 & 0 & 1 & 0 \\
		0 & 0 & 0 & 0 & 1
	\end{matrix}\right)
\end{equation*} \newline
Where matrix $A$ has entries
\begin{align*}
	&d_{1} = 1, d_{2} = 1, d_{3} = \frac{1}{x_{23}}, d_{4} = 1, d_{5} = \frac{x_{23}}{x_{15} x_{23} - x_{13} x_{25}}, \\ 
	&a_{12} = \frac{x_{13}}{x_{23}}, a_{13} = 1, a_{14} = 1, a_{15} = 1, \\ 
	&a_{23} = 1, a_{24} = 1, a_{25} = 1, \\ 
	&a_{34} = -\frac{x_{24}}{x_{23}}, a_{35} = 1, \\ 
	&a_{45} = -\frac{x_{15} x_{23}^{2} - {\left(x_{13} - 1\right)} x_{23} x_{25}}{x_{15} x_{23} x_{24} - x_{13} x_{24} x_{25}}
\end{align*}

Now assume $x_{14} = \frac{x_{13}x_{24}}{x_{23}}$ and $x_{15} = \frac{x_{13}x_{25}}{x_{23}}$
\begin{equation*}x=\left(
\right)
\end{equation*} \newline
Where matrix $A$ has entries
\begin{align*}
	&d_{1} = 1, d_{2} = 1, d_{3} = \frac{1}{x_{23}}, d_{4} = -\frac{x_{23}}{x_{13} x_{24}}, d_{5} = \frac{1}{x_{23} x_{35}}, \\ 
	&a_{12} = \frac{x_{13}}{x_{23}}, a_{13} = 1, a_{14} = 1, a_{15} = 1, \\ 
	&a_{23} = 1, a_{24} = 1, a_{25} = 1, \\ 
	&a_{34} = \frac{1}{x_{13}}, a_{35} = \frac{x_{23} x_{35} - x_{15}}{x_{13} x_{23} x_{35}}, \\ 
	&a_{45} = \frac{x_{15} + {\left(x_{13} - x_{23}\right)} x_{35}}{x_{13} x_{24} x_{35}}
\end{align*}

First assume $x_{14} \neq \frac{x_{13}x_{24}}{x_{23}}$.
\begin{equation*}x=\left(\begin{matrix}
		1 & 0 & x_{13} & x_{14} & 0  \\
		0 & 1 & x_{23} & x_{24} & 0  \\
		0 & 0 & 1 & 0 & x_{35}  \\
		0 & 0 & 0 & 1 & 0  \\
		0 & 0 & 0 & 0 & 1  \\
	\end{matrix}\right),x^A=\left(\begin{matrix}
		1 & 0 & 0 & 1 & 0 \\
		0 & 1 & 1 & 0 & 0 \\
		0 & 0 & 1 & 0 & 1 \\
		0 & 0 & 0 & 1 & 0 \\
		0 & 0 & 0 & 0 & 1
	\end{matrix}\right)
\end{equation*} \newline
Where matrix $A$ has entries
\begin{align*}
	&d_{1} = 1, d_{2} = 1, d_{3} = \frac{1}{x_{23}}, d_{4} = \frac{x_{23}}{x_{14} x_{23} - x_{13} x_{24}}, d_{5} = \frac{1}{x_{23} x_{35}}, \\ 
	&a_{12} = \frac{x_{13}}{x_{23}}, a_{13} = 1, a_{14} = 1, a_{15} = 1, \\ 
	&a_{23} = 1, a_{24} = 1, a_{25} = 1, \\ 
	&a_{34} = -\frac{x_{24}}{x_{14} x_{23} - x_{13} x_{24}}, a_{35} = \frac{x_{14} - x_{24}}{x_{14} x_{23} - x_{13} x_{24}}, \\ 
	&a_{45} = -\frac{x_{13} - x_{23}}{x_{14} x_{23} - x_{13} x_{24}}
\end{align*}

Now assume $x_{14} = \frac{x_{13}x_{24}}{x_{23}}$.
\begin{equation*}x=\left(\begin{matrix}
		1 & 0 & x_{13} & x_{14} & 0  \\
		0 & 1 & x_{23} & x_{24} & 0  \\
		0 & 0 & 1 & 0 & x_{35}  \\
		0 & 0 & 0 & 1 & 0  \\
		0 & 0 & 0 & 0 & 1  \\
	\end{matrix}\right),x^A=\left(\begin{matrix}
		1 & 0 & 0 & 0 & 0 \\
		0 & 1 & 1 & 0 & 0 \\
		0 & 0 & 1 & 0 & 1 \\
		0 & 0 & 0 & 1 & 0 \\
		0 & 0 & 0 & 0 & 1
	\end{matrix}\right)
\end{equation*} \newline
Where matrix $A$ has entries
\begin{align*}
	&d_{1} = 1, d_{2} = 1, d_{3} = \frac{1}{x_{23}}, d_{4} = 1, d_{5} = \frac{1}{x_{23} x_{35}}, \\ 
	&a_{12} = \frac{x_{13}}{x_{23}}, a_{13} = 1, a_{14} = 1, a_{15} = 1, \\ 
	&a_{23} = \frac{x_{23}}{x_{13}}, a_{24} = 1, a_{25} = 1, \\ 
	&a_{34} = -\frac{x_{24}}{x_{23}}, a_{35} = 1, \\ 
	&a_{45} = -\frac{{\left(x_{13} - 1\right)} x_{23}}{x_{13} x_{24}}
\end{align*}
First assume $x_{14} \neq \frac{x_{13}x_{24}}{x_{23}}$.
\begin{equation*}x=\left(\begin{matrix}
		1 & 0 & x_{13} & x_{14} & x_{15}  \\
		0 & 1 & x_{23} & x_{24} & 0  \\
		0 & 0 & 1 & 0 & x_{35}  \\
		0 & 0 & 0 & 1 & 0  \\
		0 & 0 & 0 & 0 & 1  \\
	\end{matrix}\right),x^A=\left(\begin{matrix}
		1 & 0 & 0 & 1 & 0 \\
		0 & 1 & 1 & 0 & 0 \\
		0 & 0 & 1 & 0 & 1 \\
		0 & 0 & 0 & 1 & 0 \\
		0 & 0 & 0 & 0 & 1
	\end{matrix}\right)
\end{equation*} \newline
Where matrix $A$ has entries
\begin{align*}
	&d_{1} = 1, d_{2} = 1, d_{3} = \frac{1}{x_{23}}, d_{4} = \frac{x_{23}}{x_{14} x_{23} - x_{13} x_{24}}, d_{5} = \frac{1}{x_{23} x_{35}}, \\ 
	&a_{12} = \frac{x_{13}}{x_{23}}, a_{13} = 1, a_{14} = 1, a_{15} = 1, \\ 
	&a_{23} = 1, a_{24} = 1, a_{25} = 1, \\ 
	&a_{34} = -\frac{x_{24}}{x_{14} x_{23} - x_{13} x_{24}}, a_{35} = \frac{x_{15} x_{24} + {\left(x_{14} x_{23} - x_{23} x_{24}\right)} x_{35}}{{\left(x_{14} x_{23}^{2} - x_{13} x_{23} x_{24}\right)} x_{35}}, \\ 
	&a_{45} = -\frac{x_{15} + {\left(x_{13} - x_{23}\right)} x_{35}}{{\left(x_{14} x_{23} - x_{13} x_{24}\right)} x_{35}}
\end{align*}

Now assume $x_{14} = \frac{x_{13}x_{24}}{x_{23}}$.
\begin{equation*}x=\left(
\right)
\end{equation*} \newline
Where matrix $A$ has entries
\begin{align*}
	&d_{1} = 1, d_{2} = 1, d_{3} = \frac{1}{x_{23}}, d_{4} = -\frac{x_{23}}{x_{13} x_{24}}, d_{5} = \frac{1}{x_{23} x_{35}}, \\ 
	&a_{12} = \frac{x_{13}}{x_{23}}, a_{13} = 1, a_{14} = 1, a_{15} = 1, \\ 
	&a_{23} = 1, a_{24} = 1, a_{25} = 1, \\ 
	&a_{34} = \frac{1}{x_{13}}, a_{35} = \frac{x_{23} x_{35} - x_{15}}{x_{13} x_{23} x_{35}}, \\ 
	&a_{45} = \frac{x_{15} x_{23} - x_{13} x_{25} + {\left(x_{13} x_{23} - x_{23}^{2}\right)} x_{35}}{x_{13} x_{23} x_{24} x_{35}}
\end{align*}

First assume $x_{14} \neq \frac{x_{13}x_{24}}{x_{23}}$.
\begin{equation*}x=\left(\begin{matrix}
		1 & 0 & x_{13} & x_{14} & 0  \\
		0 & 1 & x_{23} & x_{24} & x_{25}  \\
		0 & 0 & 1 & 0 & x_{35}  \\
		0 & 0 & 0 & 1 & 0  \\
		0 & 0 & 0 & 0 & 1  \\
	\end{matrix}\right),x^A=\left(\begin{matrix}
		1 & 0 & 0 & 1 & 0 \\
		0 & 1 & 1 & 0 & 0 \\
		0 & 0 & 1 & 0 & 1 \\
		0 & 0 & 0 & 1 & 0 \\
		0 & 0 & 0 & 0 & 1
	\end{matrix}\right)
\end{equation*} \newline
Where matrix $A$ has entries
\begin{align*}
	&d_{1} = 1, d_{2} = 1, d_{3} = \frac{1}{x_{23}}, d_{4} = \frac{x_{23}}{x_{14} x_{23} - x_{13} x_{24}}, d_{5} = \frac{1}{x_{23} x_{35}}, \\ 
	&a_{12} = \frac{x_{13}}{x_{23}}, a_{13} = 1, a_{14} = 1, a_{15} = 1, \\ 
	&a_{23} = 1, a_{24} = 1, a_{25} = 1, \\ 
	&a_{34} = -\frac{x_{24}}{x_{14} x_{23} - x_{13} x_{24}}, a_{35} = -\frac{x_{14} x_{25} - {\left(x_{14} x_{23} - x_{23} x_{24}\right)} x_{35}}{{\left(x_{14} x_{23}^{2} - x_{13} x_{23} x_{24}\right)} x_{35}}, \\ 
	&a_{45} = \frac{x_{13} x_{25} - {\left(x_{13} x_{23} - x_{23}^{2}\right)} x_{35}}{{\left(x_{14} x_{23}^{2} - x_{13} x_{23} x_{24}\right)} x_{35}}
\end{align*}

Now assume $x_{14} = \frac{x_{13}x_{24}}{x_{23}}$.
\begin{equation*}x=\left(\begin{matrix}
		1 & 0 & x_{13} & x_{14} & 0  \\
		0 & 1 & x_{23} & x_{24} & x_{25}  \\
		0 & 0 & 1 & 0 & x_{35}  \\
		0 & 0 & 0 & 1 & 0  \\
		0 & 0 & 0 & 0 & 1  \\
	\end{matrix}\right),x^A=\left(\begin{matrix}
		1 & 0 & 0 & 0 & 0 \\
		0 & 1 & 1 & 0 & 0 \\
		0 & 0 & 1 & 0 & 1 \\
		0 & 0 & 0 & 1 & 0 \\
		0 & 0 & 0 & 0 & 1
	\end{matrix}\right)
\end{equation*} \newline
Where matrix $A$ has entries
\begin{align*}
	&d_{1} = 1, d_{2} = 1, d_{3} = \frac{1}{x_{23}}, d_{4} = 1, d_{5} = \frac{1}{x_{23} x_{35}}, \\ 
	&a_{12} = \frac{x_{13}}{x_{23}}, a_{13} = 1, a_{14} = 1, a_{15} = 1, \\ 
	&a_{23} = \frac{x_{23}^{2} x_{35} + x_{13} x_{25}}{x_{13} x_{23} x_{35}}, a_{24} = 1, a_{25} = 1, \\ 
	&a_{34} = -\frac{x_{24}}{x_{23}}, a_{35} = 1, \\ 
	&a_{45} = -\frac{{\left(x_{13} - 1\right)} x_{23}}{x_{13} x_{24}}
\end{align*}

First assume $x_{14} \neq \frac{x_{13}x_{24}}{x_{23}}$.
\begin{equation*}x=\left(\begin{matrix}
		1 & 0 & x_{13} & x_{14} & x_{15}  \\
		0 & 1 & x_{23} & x_{24} & x_{25}  \\
		0 & 0 & 1 & 0 & x_{35}  \\
		0 & 0 & 0 & 1 & 0  \\
		0 & 0 & 0 & 0 & 1  \\
	\end{matrix}\right),x^A=\left(\begin{matrix}
		1 & 0 & 0 & 1 & 0 \\
		0 & 1 & 1 & 0 & 0 \\
		0 & 0 & 1 & 0 & 1 \\
		0 & 0 & 0 & 1 & 0 \\
		0 & 0 & 0 & 0 & 1
	\end{matrix}\right)
\end{equation*} \newline
Where matrix $A$ has entries
\begin{align*}
	&d_{1} = 1, d_{2} = 1, d_{3} = \frac{1}{x_{23}}, d_{4} = \frac{x_{23}}{x_{14} x_{23} - x_{13} x_{24}}, d_{5} = \frac{1}{x_{23} x_{35}}, \\ 
	&a_{12} = \frac{x_{13}}{x_{23}}, a_{13} = 1, a_{14} = 1, a_{15} = 1, \\ 
	&a_{23} = 1, a_{24} = 1, a_{25} = 1, \\ 
	&a_{34} = -\frac{x_{24}}{x_{14} x_{23} - x_{13} x_{24}}, a_{35} = \frac{x_{15} x_{24} - x_{14} x_{25} + {\left(x_{14} x_{23} - x_{23} x_{24}\right)} x_{35}}{{\left(x_{14} x_{23}^{2} - x_{13} x_{23} x_{24}\right)} x_{35}}, \\ 
	&a_{45} = -\frac{x_{15} x_{23} - x_{13} x_{25} + {\left(x_{13} x_{23} - x_{23}^{2}\right)} x_{35}}{{\left(x_{14} x_{23}^{2} - x_{13} x_{23} x_{24}\right)} x_{35}}
\end{align*}

Now assume $x_{14} = \frac{x_{13}x_{24}}{x_{23}}$.
\begin{equation*}x=\left(
\right)
\end{equation*} \newline
Where matrix $A$ has entries
\begin{align*}
	&d_{1} = 1, d_{2} = 1, d_{3} = \frac{1}{x_{23}}, d_{4} = 1, d_{5} = \frac{1}{x_{23} x_{35}}, \\ 
	&a_{12} = \frac{x_{13}}{x_{23}}, a_{13} = 1, a_{14} = 1, a_{15} = 1, \\ 
	&a_{23} = \frac{x_{23}^{2} x_{35} - x_{15} x_{23} + x_{13} x_{25}}{x_{13} x_{23} x_{35}}, a_{24} = 1, a_{25} = 1, \\ 
	&a_{34} = -\frac{x_{24}}{x_{23}}, a_{35} = 1, \\ 
	&a_{45} = -\frac{{\left(x_{13} - 1\right)} x_{23} x_{35} + x_{15}}{x_{13} x_{24} x_{35}}
\end{align*}
		\section{Subcases of $Y_4$}

\begin{equation*}x=\left(%
\right)
\end{equation*} \newline
Where matrix $A$ has entries
\begin{align*}
	&d_{1} = 1, d_{2} = \frac{1}{x_{12}}, d_{3} = \frac{1}{x_{12} x_{23}}, d_{4} = 1, d_{5} = 1, \\ 
	&a_{12} = 1, a_{13} = 1, a_{14} = 1, a_{15} = 1, \\ 
	&a_{23} = \frac{x_{12} x_{23} - x_{13}}{x_{12}^{2} x_{23}}, a_{24} = \frac{x_{13} x_{24}}{x_{12} x_{23}}, a_{25} = 1, \\ 
	&a_{34} = -\frac{x_{24}}{x_{23}}, a_{35} = -\frac{x_{12} + x_{15}}{x_{13}}, \\ 
	&a_{45} = \frac{{\left(x_{12} + x_{15}\right)} x_{23}}{x_{13} x_{24}}
\end{align*}


First assume $x_{14} \neq \frac{x_{13}x_{24}}{x_{23}}$
\begin{equation*}x=\left(\begin{matrix}
		1 & x_{12} & x_{13} & x_{14} & 0 \\
		0 & 1 & x_{23} & x_{24} & 0 \\
		0 & 0 & 1 & 0 & 0  \\
		0 & 0 & 0 & 1 & 0  \\
		0 & 0 & 0 & 0 & 1 \\
	\end{matrix}\right),x^A=\left(\begin{matrix}
		1 & 1 & 0 & 0 & 0 \\
		0 & 1 & 1 & 0 & 0 \\
		0 & 0 & 1 & 0 & 0 \\
		0 & 0 & 0 & 1 & 0 \\
		0 & 0 & 0 & 0 & 1
	\end{matrix}\right)
\end{equation*} \newline
Where matrix $A$ has entries
\begin{align*}
	&d_{1} = 1, d_{2} = \frac{1}{x_{12}}, d_{3} = \frac{1}{x_{12} x_{23}}, d_{4} = 1, d_{5} = 1, \\ 
	&a_{12} = 1, a_{13} = 1, a_{14} = 1, a_{15} = 1, \\ 
	&a_{23} = \frac{x_{12} x_{23} - x_{13}}{x_{12}^{2} x_{23}}, a_{24} = -\frac{x_{14} x_{23} - x_{13} x_{24}}{x_{12} x_{23}}, a_{25} = 1, \\ 
	&a_{34} = -\frac{x_{24}}{x_{23}}, a_{35} = \frac{x_{12} x_{24}}{x_{14} x_{23} - x_{13} x_{24}}, \\ 
	&a_{45} = -\frac{x_{12} x_{23}}{x_{14} x_{23} - x_{13} x_{24}}
\end{align*}

Now assume $x_{14} = \frac{x_{13}x_{24}}{x_{23}}$
\begin{equation*}x=\left(\begin{matrix}
		1 & x_{12} & x_{13} & x_{14} & 0 \\
		0 & 1 & x_{23} & x_{24} & 0 \\
		0 & 0 & 1 & 0 & 0  \\
		0 & 0 & 0 & 1 & 0  \\
		0 & 0 & 0 & 0 & 1 \\
	\end{matrix}\right),x^A=\left(\begin{matrix}
		1 & 1 & 0 & 0 & 0 \\
		0 & 1 & 1 & 0 & 0 \\
		0 & 0 & 1 & 0 & 0 \\
		0 & 0 & 0 & 1 & 0 \\
		0 & 0 & 0 & 0 & 1
	\end{matrix}\right)
\end{equation*} \newline
Where matrix $A$ has entries
\begin{align*}
	&d_{1} = 1, d_{2} = \frac{1}{x_{12}}, d_{3} = \frac{1}{x_{12} x_{23}}, d_{4} = 1, d_{5} = 1, \\ 
	&a_{12} = 1, a_{13} = 1, a_{14} = 1, a_{15} = 1, \\ 
	&a_{23} = \frac{x_{12} x_{23} - x_{13}}{x_{12}^{2} x_{23}}, a_{24} = 0, a_{25} = 0, \\ 
	&a_{34} = -\frac{x_{24}}{x_{23}}, a_{35} = 1, \\ 
	&a_{45} = -\frac{x_{23}}{x_{24}}
\end{align*}

First assume $x_{14} \neq \frac{x_{13}x_{24}}{x_{23}}$
\begin{equation*}x=\left(\begin{matrix}
		1 & x_{12} & x_{13} & x_{14} & x_{15} \\
		0 & 1 & x_{23} & x_{24} & 0 \\
		0 & 0 & 1 & 0 & 0  \\
		0 & 0 & 0 & 1 & 0  \\
		0 & 0 & 0 & 0 & 1 \\
	\end{matrix}\right),x^A=\left(\begin{matrix}
		1 & 1 & 0 & 0 & 0 \\
		0 & 1 & 1 & 0 & 0 \\
		0 & 0 & 1 & 0 & 0 \\
		0 & 0 & 0 & 1 & 0 \\
		0 & 0 & 0 & 0 & 1
	\end{matrix}\right)
\end{equation*} \newline
Where matrix $A$ has entries
\begin{align*}
	&d_{1} = 1, d_{2} = \frac{1}{x_{12}}, d_{3} = \frac{1}{x_{12} x_{23}}, d_{4} = 1, d_{5} = 1, \\ 
	&a_{12} = 1, a_{13} = 1, a_{14} = 1, a_{15} = 1, \\ 
	&a_{23} = \frac{x_{12} x_{23} - x_{13}}{x_{12}^{2} x_{23}}, a_{24} = -\frac{x_{14} x_{23} - x_{13} x_{24}}{x_{12} x_{23}}, a_{25} = 1, \\ 
	&a_{34} = -\frac{x_{24}}{x_{23}}, a_{35} = \frac{{\left(x_{12} + x_{15}\right)} x_{24}}{x_{14} x_{23} - x_{13} x_{24}}, \\ 
	&a_{45} = -\frac{{\left(x_{12} + x_{15}\right)} x_{23}}{x_{14} x_{23} - x_{13} x_{24}}
\end{align*}

Now assume $x_{14} = \frac{x_{13}x_{24}}{x_{23}}$
\begin{equation*}x=\left(
\right)
\end{equation*} \newline
Where matrix $A$ has entries
\begin{align*}
	&d_{1} = 1, d_{2} = \frac{1}{x_{12}}, d_{3} = \frac{1}{x_{12} x_{23}}, d_{4} = 1, d_{5} = 1, \\ 
	&a_{12} = 1, a_{13} = 1, a_{14} = 1, a_{15} = 1, \\ 
	&a_{23} = \frac{x_{12} x_{23} - x_{13}}{x_{12}^{2} x_{23}}, a_{24} = \frac{x_{13} x_{24}}{x_{12} x_{23}}, a_{25} = 1, \\ 
	&a_{34} = -\frac{x_{24}}{x_{23}}, a_{35} = -\frac{x_{12} + x_{15}}{x_{13}}, \\ 
	&a_{45} = -\frac{x_{13} x_{25} - {\left(x_{12} + x_{15}\right)} x_{23}}{x_{13} x_{24}}
\end{align*}

First assume $x_{14}\neq \frac{x_{13}x_{24}}{x_{23}}$.
\begin{equation*}x=\left(\begin{matrix}
		1 & x_{12} & x_{13} & x_{14} & 0 \\
		0 & 1 & x_{23} & x_{24} & x_{25} \\
		0 & 0 & 1 & 0 & 0  \\
		0 & 0 & 0 & 1 & 0  \\
		0 & 0 & 0 & 0 & 1 \\
	\end{matrix}\right),x^A=\left(\begin{matrix}
		1 & 1 & 0 & 0 & 0 \\
		0 & 1 & 1 & 0 & 0 \\
		0 & 0 & 1 & 0 & 0 \\
		0 & 0 & 0 & 1 & 0 \\
		0 & 0 & 0 & 0 & 1
	\end{matrix}\right)
\end{equation*} \newline
Where matrix $A$ has entries
\begin{align*}
	&d_{1} = 1, d_{2} = \frac{1}{x_{12}}, d_{3} = \frac{1}{x_{12} x_{23}}, d_{4} = 1, d_{5} = 1, \\ 
	&a_{12} = 1, a_{13} = 1, a_{14} = 1, a_{15} = 1, \\ 
	&a_{23} = \frac{x_{12} x_{23} - x_{13}}{x_{12}^{2} x_{23}}, a_{24} = -\frac{x_{14} x_{23} - x_{13} x_{24}}{x_{12} x_{23}}, a_{25} = 1, \\ 
	&a_{34} = -\frac{x_{24}}{x_{23}}, a_{35} = \frac{x_{12} x_{24} - x_{14} x_{25}}{x_{14} x_{23} - x_{13} x_{24}}, \\ 
	&a_{45} = -\frac{x_{12} x_{23} - x_{13} x_{25}}{x_{14} x_{23} - x_{13} x_{24}}
\end{align*}
Now assume $x_{14}= \frac{x_{13}x_{24}}{x_{23}}$.
\begin{equation*}x=\left(\begin{matrix}
		1 & x_{12} & x_{13} & x_{14} & 0 \\
		0 & 1 & x_{23} & x_{24} & x_{25} \\
		0 & 0 & 1 & 0 & 0  \\
		0 & 0 & 0 & 1 & 0  \\
		0 & 0 & 0 & 0 & 1 \\
	\end{matrix}\right),x^A=\left(\begin{matrix}
		1 & 1 & 0 & 0 & 0 \\
		0 & 1 & 1 & 0 & 0 \\
		0 & 0 & 1 & 0 & 0 \\
		0 & 0 & 0 & 1 & 0 \\
		0 & 0 & 0 & 0 & 1
	\end{matrix}\right)
\end{equation*} \newline
Where matrix $A$ has entries
\begin{align*}
	&d_{1} = 1, d_{2} = \frac{1}{x_{12}}, d_{3} = \frac{1}{x_{12} x_{23}}, d_{4} = 1, d_{5} = 1, \\ 
	&a_{12} = 1, a_{13} = 1, a_{14} = 1, a_{15} = 1, \\ 
	&a_{23} = \frac{x_{12} x_{23} - x_{13}}{x_{12}^{2} x_{23}}, a_{24} = 0, a_{25} = \frac{x_{13} x_{25}}{x_{12} x_{23}}, \\ 
	&a_{34} = -\frac{x_{24}}{x_{23}}, a_{35} = 1, \\ 
	&a_{45} = -\frac{x_{23} + x_{25}}{x_{24}}
\end{align*}
First assume $x_{14}\neq \frac{x_{13}x_{24}}{x_{23}}$.
\begin{equation*}x=\left(\begin{matrix}
		1 & x_{12} & x_{13} & x_{14} & x_{15} \\
		0 & 1 & x_{23} & x_{24} & x_{25} \\
		0 & 0 & 1 & 0 & 0  \\
		0 & 0 & 0 & 1 & 0  \\
		0 & 0 & 0 & 0 & 1 \\
	\end{matrix}\right),x^A=\left(\begin{matrix}
		1 & 1 & 0 & 0 & 0 \\
		0 & 1 & 1 & 0 & 0 \\
		0 & 0 & 1 & 0 & 0 \\
		0 & 0 & 0 & 1 & 0 \\
		0 & 0 & 0 & 0 & 1
	\end{matrix}\right)
\end{equation*} \newline
Where matrix $A$ has entries
\begin{align*}
	&d_{1} = 1, d_{2} = \frac{1}{x_{12}}, d_{3} = \frac{1}{x_{12} x_{23}}, d_{4} = 1, d_{5} = 1, \\ 
	&a_{12} = 1, a_{13} = 1, a_{14} = 1, a_{15} = 1, \\ 
	&a_{23} = \frac{x_{12} x_{23} - x_{13}}{x_{12}^{2} x_{23}}, a_{24} = -\frac{x_{14} x_{23} - x_{13} x_{24}}{x_{12} x_{23}}, a_{25} = 1, \\ 
	&a_{34} = -\frac{x_{24}}{x_{23}}, a_{35} = -\frac{x_{14} x_{25} - {\left(x_{12} + x_{15}\right)} x_{24}}{x_{14} x_{23} - x_{13} x_{24}}, \\ 
	&a_{45} = \frac{x_{13} x_{25} - {\left(x_{12} + x_{15}\right)} x_{23}}{x_{14} x_{23} - x_{13} x_{24}}
\end{align*}

Now assume $x_{14}= \frac{x_{13}x_{24}}{x_{23}}$.
\begin{equation*}x=\left(
\right)
\end{equation*} \newline
Where matrix $A$ has entries
\begin{align*}
	&d_{1} = 1, d_{2} = \frac{1}{x_{12}}, d_{3} = \frac{1}{x_{12} x_{23}}, d_{4} = 1, d_{5} = \frac{1}{x_{12} x_{23} x_{35}}, \\ 
	&a_{12} = 1, a_{13} = 1, a_{14} = 1, a_{15} = 1, \\ 
	&a_{23} = \frac{x_{12} x_{23} - x_{13}}{x_{12}^{2} x_{23}}, a_{24} = 0, a_{25} = -\frac{x_{12} x_{15} x_{23} - {\left(x_{12}^{2} x_{23}^{2} - x_{12} x_{13} x_{23} + x_{13}^{2}\right)} x_{35}}{x_{12}^{3} x_{23}^{2} x_{35}}, \\ 
	&a_{34} = 0, a_{35} = \frac{x_{12} x_{23} - x_{13}}{x_{12}^{2} x_{23}^{2}}, \\ 
	&a_{45} = 1
\end{align*}

\begin{equation*}x=\left(\begin{matrix}
		1 & x_{12} & x_{13} & x_{14} & 0 \\
		0 & 1 & x_{23} & 0 & 0 \\
		0 & 0 & 1 & 0 & x_{35}  \\
		0 & 0 & 0 & 1 & 0  \\
		0 & 0 & 0 & 0 & 1 \\
	\end{matrix}\right),x^A=\left(\begin{matrix}
		1 & 1 & 0 & 0 & 0 \\
		0 & 1 & 1 & 0 & 0 \\
		0 & 0 & 1 & 0 & 1 \\
		0 & 0 & 0 & 1 & 0 \\
		0 & 0 & 0 & 0 & 1
	\end{matrix}\right)
\end{equation*} \newline
Where matrix $A$ has entries
\begin{align*}
	&d_{1} = 1, d_{2} = \frac{1}{x_{12}}, d_{3} = \frac{1}{x_{12} x_{23}}, d_{4} = 1, d_{5} = \frac{1}{x_{12} x_{23} x_{35}}, \\ 
	&a_{12} = 1, a_{13} = 1, a_{14} = 1, a_{15} = 1, \\ 
	&a_{23} = \frac{x_{12} x_{23} - x_{13}}{x_{12}^{2} x_{23}}, a_{24} = -\frac{x_{14}}{x_{12}}, a_{25} = 1, \\ 
	&a_{34} = 0, a_{35} = \frac{x_{12} x_{23} - x_{13}}{x_{12}^{2} x_{23}^{2}}, \\ 
	&a_{45} = -\frac{x_{12} x_{13} x_{23} - x_{13}^{2} + {\left(x_{12}^{3} - x_{12}^{2}\right)} x_{23}^{2}}{x_{12}^{2} x_{14} x_{23}^{2}}
\end{align*}

\begin{equation*}x=\left(\begin{matrix}
		1 & x_{12} & x_{13} & x_{14} & x_{15} \\
		0 & 1 & x_{23} & 0 & 0 \\
		0 & 0 & 1 & 0 & x_{35}  \\
		0 & 0 & 0 & 1 & 0  \\
		0 & 0 & 0 & 0 & 1 \\
	\end{matrix}\right),x^A=\left(\begin{matrix}
		1 & 1 & 0 & 0 & 0 \\
		0 & 1 & 1 & 0 & 0 \\
		0 & 0 & 1 & 0 & 1 \\
		0 & 0 & 0 & 1 & 0 \\
		0 & 0 & 0 & 0 & 1
	\end{matrix}\right)
\end{equation*} \newline
Where matrix $A$ has entries
\begin{align*}
	&d_{1} = 1, d_{2} = \frac{1}{x_{12}}, d_{3} = \frac{1}{x_{12} x_{23}}, d_{4} = 1, d_{5} = \frac{1}{x_{12} x_{23} x_{35}}, \\ 
	&a_{12} = 1, a_{13} = 1, a_{14} = 1, a_{15} = 1, \\ 
	&a_{23} = \frac{x_{12} x_{23} - x_{13}}{x_{12}^{2} x_{23}}, a_{24} = -\frac{x_{14}}{x_{12}}, a_{25} = 1, \\ 
	&a_{34} = 0, a_{35} = \frac{x_{12} x_{23} - x_{13}}{x_{12}^{2} x_{23}^{2}}, \\ 
	&a_{45} = -\frac{x_{12} x_{15} x_{23} + {\left(x_{12} x_{13} x_{23} - x_{13}^{2} + {\left(x_{12}^{3} - x_{12}^{2}\right)} x_{23}^{2}\right)} x_{35}}{x_{12}^{2} x_{14} x_{23}^{2} x_{35}}
\end{align*}

\begin{equation*}x=\left(
\right)
\end{equation*} \newline
Where matrix $A$ has entries
\begin{align*}
	&d_{1} = 1, d_{2} = \frac{1}{x_{12}}, d_{3} = \frac{1}{x_{12} x_{23}}, d_{4} = 1, d_{5} = \frac{1}{x_{12} x_{23} x_{35}}, \\ 
	&a_{12} = 1, a_{13} = 1, a_{14} = 1, a_{15} = 1, \\ 
	&a_{23} = \frac{x_{12} x_{23} - x_{13}}{x_{12}^{2} x_{23}}, a_{24} = 0, a_{25} = \frac{x_{12} x_{13} x_{25} + {\left(x_{12}^{2} x_{23}^{2} - x_{12} x_{13} x_{23} + x_{13}^{2}\right)} x_{35}}{x_{12}^{3} x_{23}^{2} x_{35}}, \\ 
	&a_{34} = 0, a_{35} = -\frac{x_{12} x_{25} - {\left(x_{12} x_{23} - x_{13}\right)} x_{35}}{x_{12}^{2} x_{23}^{2} x_{35}}, \\ 
	&a_{45} = 1
\end{align*}

\begin{equation*}x=\left(\begin{matrix}
		1 & x_{12} & x_{13} & 0 & x_{15} \\
		0 & 1 & x_{23} & 0 & x_{25} \\
		0 & 0 & 1 & 0 & x_{35}  \\
		0 & 0 & 0 & 1 & 0  \\
		0 & 0 & 0 & 0 & 1 \\
	\end{matrix}\right),x^A=\left(\begin{matrix}
		1 & 1 & 0 & 0 & 0 \\
		0 & 1 & 1 & 0 & 0 \\
		0 & 0 & 1 & 0 & 1 \\
		0 & 0 & 0 & 1 & 0 \\
		0 & 0 & 0 & 0 & 1
	\end{matrix}\right)
\end{equation*} \newline
Where matrix $A$ has entries
\begin{align*}
	&d_{1} = 1, d_{2} = \frac{1}{x_{12}}, d_{3} = \frac{1}{x_{12} x_{23}}, d_{4} = 1, d_{5} = \frac{1}{x_{12} x_{23} x_{35}}, \\ 
	&a_{12} = 1, a_{13} = 1, a_{14} = 1, a_{15} = 1, \\ 
	&a_{23} = \frac{x_{12} x_{23} - x_{13}}{x_{12}^{2} x_{23}}, a_{24} = 0, a_{25} = -\frac{x_{12} x_{15} x_{23} - x_{12} x_{13} x_{25} - {\left(x_{12}^{2} x_{23}^{2} - x_{12} x_{13} x_{23} + x_{13}^{2}\right)} x_{35}}{x_{12}^{3} x_{23}^{2} x_{35}}, \\ 
	&a_{34} = 0, a_{35} = -\frac{x_{12} x_{25} - {\left(x_{12} x_{23} - x_{13}\right)} x_{35}}{x_{12}^{2} x_{23}^{2} x_{35}}, \\ 
	&a_{45} = 1
\end{align*}

\begin{equation*}x=\left(\begin{matrix}
		1 & x_{12} & x_{13} & x_{14} & 0 \\
		0 & 1 & x_{23} & 0 & x_{25} \\
		0 & 0 & 1 & 0 & x_{35}  \\
		0 & 0 & 0 & 1 & 0  \\
		0 & 0 & 0 & 0 & 1 \\
	\end{matrix}\right),x^A=\left(\begin{matrix}
		1 & 1 & 0 & 0 & 0 \\
		0 & 1 & 1 & 0 & 0 \\
		0 & 0 & 1 & 0 & 1 \\
		0 & 0 & 0 & 1 & 0 \\
		0 & 0 & 0 & 0 & 1
	\end{matrix}\right)
\end{equation*} \newline
Where matrix $A$ has entries
\begin{align*}
	&d_{1} = 1, d_{2} = \frac{1}{x_{12}}, d_{3} = \frac{1}{x_{12} x_{23}}, d_{4} = 1, d_{5} = \frac{1}{x_{12} x_{23} x_{35}}, \\ 
	&a_{12} = 1, a_{13} = 1, a_{14} = 1, a_{15} = 1, \\ 
	&a_{23} = \frac{x_{12} x_{23} - x_{13}}{x_{12}^{2} x_{23}}, a_{24} = -\frac{x_{14}}{x_{12}}, a_{25} = 1, \\ 
	&a_{34} = 0, a_{35} = -\frac{x_{12} x_{25} - {\left(x_{12} x_{23} - x_{13}\right)} x_{35}}{x_{12}^{2} x_{23}^{2} x_{35}}, \\ 
	&a_{45} = \frac{x_{12} x_{13} x_{25} - {\left(x_{12} x_{13} x_{23} - x_{13}^{2} + {\left(x_{12}^{3} - x_{12}^{2}\right)} x_{23}^{2}\right)} x_{35}}{x_{12}^{2} x_{14} x_{23}^{2} x_{35}}
\end{align*}

\begin{equation*}x=\left(\begin{matrix}
		1 & x_{12} & x_{13} & x_{14} & x_{15} \\
		0 & 1 & x_{23} & 0 & x_{25} \\
		0 & 0 & 1 & 0 & x_{35}  \\
		0 & 0 & 0 & 1 & 0  \\
		0 & 0 & 0 & 0 & 1 \\
	\end{matrix}\right),x^A=\left(\begin{matrix}
		1 & 1 & 0 & 0 & 0 \\
		0 & 1 & 1 & 0 & 0 \\
		0 & 0 & 1 & 0 & 1 \\
		0 & 0 & 0 & 1 & 0 \\
		0 & 0 & 0 & 0 & 1
	\end{matrix}\right)
\end{equation*} \newline
Where matrix $A$ has entries
\begin{align*}
	&d_{1} = 1, d_{2} = \frac{1}{x_{12}}, d_{3} = \frac{1}{x_{12} x_{23}}, d_{4} = 1, d_{5} = \frac{1}{x_{12} x_{23} x_{35}}, \\ 
	&a_{12} = 1, a_{13} = 1, a_{14} = 1, a_{15} = 1, \\ 
	&a_{23} = \frac{x_{12} x_{23} - x_{13}}{x_{12}^{2} x_{23}}, a_{24} = -\frac{x_{14}}{x_{12}}, a_{25} = 1, \\ 
	&a_{34} = 0, a_{35} = -\frac{x_{12} x_{25} - {\left(x_{12} x_{23} - x_{13}\right)} x_{35}}{x_{12}^{2} x_{23}^{2} x_{35}}, \\ 
	&a_{45} = -\frac{x_{12} x_{15} x_{23} - x_{12} x_{13} x_{25} + {\left(x_{12} x_{13} x_{23} - x_{13}^{2} + {\left(x_{12}^{3} - x_{12}^{2}\right)} x_{23}^{2}\right)} x_{35}}{x_{12}^{2} x_{14} x_{23}^{2} x_{35}}
\end{align*}

\begin{equation*}x=\left(
\right)
\end{equation*} \newline
Where matrix $A$ has entries
\begin{align*}
	&d_{1} = 1, d_{2} = \frac{1}{x_{12}}, d_{3} = \frac{1}{x_{12} x_{23}}, d_{4} = 1, d_{5} = \frac{1}{x_{12} x_{23} x_{35}}, \\ 
	&a_{12} = 1, a_{13} = 1, a_{14} = 1, a_{15} = 1, \\ 
	&a_{23} = \frac{1}{x_{12}}, a_{24} = -\frac{x_{14}}{x_{12}}, a_{25} = 1, \\ 
	&a_{34} = -\frac{x_{24}}{x_{23}}, a_{35} = \frac{x_{15} x_{24} + {\left(x_{14} x_{23} + {\left(x_{12}^{2} - x_{12}\right)} x_{23} x_{24}\right)} x_{35}}{x_{12} x_{14} x_{23}^{2} x_{35}}, \\ 
	&a_{45} = -\frac{{\left(x_{12}^{2} - x_{12}\right)} x_{23} x_{35} + x_{15}}{x_{12} x_{14} x_{23} x_{35}}
\end{align*}

\begin{equation*}x=\left(\begin{matrix}
		1 & x_{12} & x_{13} & 0 & 0 \\
		0 & 1 & x_{23} & x_{24} & 0 \\
		0 & 0 & 1 & 0 & x_{35}  \\
		0 & 0 & 0 & 1 & 0  \\
		0 & 0 & 0 & 0 & 1 \\
	\end{matrix}\right),x^A=\left(\begin{matrix}
		1 & 1 & 0 & 0 & 0 \\
		0 & 1 & 1 & 0 & 0 \\
		0 & 0 & 1 & 0 & 0 \\
		0 & 0 & 0 & 1 & 0 \\
		0 & 0 & 0 & 0 & 1
	\end{matrix}\right)
\end{equation*} \newline
Where matrix $A$ has entries
\begin{align*}
	&d_{1} = 1, d_{2} = \frac{1}{x_{12}}, d_{3} = \frac{1}{x_{12} x_{23}}, d_{4} = 1, d_{5} = \frac{1}{x_{12} x_{23} x_{35}}, \\ 
	&a_{12} = 1, a_{13} = 1, a_{14} = 1, a_{15} = 1, \\ 
	&a_{23} = \frac{x_{12} x_{23} - x_{13}}{x_{12}^{2} x_{23}}, a_{24} = \frac{x_{13} x_{24}}{x_{12} x_{23}}, a_{25} = 1, \\ 
	&a_{34} = -\frac{x_{24}}{x_{23}}, a_{35} = -\frac{x_{12} - 1}{x_{13}}, \\ 
	&a_{45} = \frac{x_{12} x_{13} x_{23} - x_{13}^{2} + {\left(x_{12}^{3} - x_{12}^{2}\right)} x_{23}^{2}}{x_{12}^{2} x_{13} x_{23} x_{24}}
\end{align*}

\begin{equation*}x=\left(\begin{matrix}
		1 & x_{12} & x_{13} & 0 & x_{15} \\
		0 & 1 & x_{23} & x_{24} & 0 \\
		0 & 0 & 1 & 0 & x_{35}  \\
		0 & 0 & 0 & 1 & 0  \\
		0 & 0 & 0 & 0 & 1 \\
	\end{matrix}\right),x^A=\left(\begin{matrix}
		1 & 1 & 0 & 0 & 0 \\
		0 & 1 & 1 & 0 & 0 \\
		0 & 0 & 1 & 0 & 1 \\
		0 & 0 & 0 & 1 & 0 \\
		0 & 0 & 0 & 0 & 1
	\end{matrix}\right)
\end{equation*} \newline
Where matrix $A$ has entries
\begin{align*}
	&d_{1} = 1, d_{2} = \frac{1}{x_{12}}, d_{3} = \frac{1}{x_{12} x_{23}}, d_{4} = 1, d_{5} = \frac{1}{x_{12} x_{23} x_{35}}, \\ 
	&a_{12} = 1, a_{13} = 1, a_{14} = 1, a_{15} = 1, \\ 
	&a_{23} = \frac{x_{12} x_{23} - x_{13}}{x_{12}^{2} x_{23}}, a_{24} = \frac{x_{13} x_{24}}{x_{12} x_{23}}, a_{25} = 1, \\ 
	&a_{34} = -\frac{x_{24}}{x_{23}}, a_{35} = -\frac{{\left(x_{12}^{2} - x_{12}\right)} x_{23} x_{35} + x_{15}}{x_{12} x_{13} x_{23} x_{35}}, \\ 
	&a_{45} = \frac{x_{12} x_{15} x_{23} + {\left(x_{12} x_{13} x_{23} - x_{13}^{2} + {\left(x_{12}^{3} - x_{12}^{2}\right)} x_{23}^{2}\right)} x_{35}}{x_{12}^{2} x_{13} x_{23} x_{24} x_{35}}
\end{align*}

First assume $x_{14}\neq \frac{x_{13}x_{24}}{x_{23}}$
\begin{equation*}x=\left(\begin{matrix}
		1 & x_{12} & x_{13} & x_{14} & 0 \\
		0 & 1 & x_{23} & x_{24} & 0 \\
		0 & 0 & 1 & 0 & x_{35}  \\
		0 & 0 & 0 & 1 & 0  \\
		0 & 0 & 0 & 0 & 1 \\
	\end{matrix}\right),x^A=\left(\begin{matrix}
		1 & 1 & 0 & 0 & 0 \\
		0 & 1 & 1 & 0 & 0 \\
		0 & 0 & 1 & 0 & 1 \\
		0 & 0 & 0 & 1 & 0 \\
		0 & 0 & 0 & 0 & 1
	\end{matrix}\right)
\end{equation*} \newline
Where matrix $A$ has entries
\begin{align*}
	&d_{1} = 1, d_{2} = \frac{1}{x_{12}}, d_{3} = \frac{1}{x_{12} x_{23}}, d_{4} = 1, d_{5} = \frac{1}{x_{12} x_{23} x_{35}}, \\ 
	&a_{12} = 1, a_{13} = 1, a_{14} = 1, a_{15} = 1, \\ 
	&a_{23} = \frac{x_{12} x_{23} - x_{13}}{x_{12}^{2} x_{23}}, a_{24} = -\frac{x_{14} x_{23} - x_{13} x_{24}}{x_{12} x_{23}}, a_{25} = 1, \\ 
	&a_{34} = -\frac{x_{24}}{x_{23}}, a_{35} = \frac{x_{12} x_{14} x_{23} - x_{13} x_{14} + {\left(x_{12}^{3} - x_{12}^{2}\right)} x_{23} x_{24}}{x_{12}^{2} x_{14} x_{23}^{2} - x_{12}^{2} x_{13} x_{23} x_{24}}, \\ 
	&a_{45} = -\frac{x_{12} x_{13} x_{23} - x_{13}^{2} + {\left(x_{12}^{3} - x_{12}^{2}\right)} x_{23}^{2}}{x_{12}^{2} x_{14} x_{23}^{2} - x_{12}^{2} x_{13} x_{23} x_{24}}
\end{align*}

Now assume $x_{14}= \frac{x_{13}x_{24}}{x_{23}}$
\begin{equation*}x=\left(\begin{matrix}
		1 & x_{12} & x_{13} & x_{14} & 0 \\
		0 & 1 & x_{23} & x_{24} & 0 \\
		0 & 0 & 1 & 0 & x_{35}  \\
		0 & 0 & 0 & 1 & 0  \\
		0 & 0 & 0 & 0 & 1 \\
	\end{matrix}\right),x^A=\left(\begin{matrix}
		1 & 1 & 0 & 0 & 0 \\
		0 & 1 & 1 & 0 & 0 \\
		0 & 0 & 1 & 0 & 1 \\
		0 & 0 & 0 & 1 & 0 \\
		0 & 0 & 0 & 0 & 1
	\end{matrix}\right)
\end{equation*} \newline
Where matrix $A$ has entries
\begin{align*}
	&d_{1} = 1, d_{2} = \frac{1}{x_{12}}, d_{3} = \frac{1}{x_{12} x_{23}}, d_{4} = 1, d_{5} = \frac{1}{x_{12} x_{23} x_{35}}, \\ 
	&a_{12} = 1, a_{13} = 1, a_{14} = 1, a_{15} = 1, \\ 
	&a_{23} = \frac{x_{12} x_{23} - x_{13}}{x_{12}^{2} x_{23}}, a_{24} = 0, a_{25} = \frac{x_{12}^{2} x_{23}^{2} - x_{12} x_{13} x_{23} + x_{13}^{2}}{x_{12}^{3} x_{23}^{2}}, \\ 
	&a_{34} = -\frac{x_{24}}{x_{23}}, a_{35} = 1, \\ 
	&a_{45} = -\frac{x_{12}^{2} x_{23}^{2} - x_{12} x_{23} + x_{13}}{x_{12}^{2} x_{23} x_{24}}
\end{align*}


First assume $x_{14}\neq \frac{x_{13}x_{24}}{x_{23}}$
\begin{equation*}x=\left(\begin{matrix}
		1 & x_{12} & x_{13} & x_{14} & x_{15} \\
		0 & 1 & x_{23} & x_{24} & 0 \\
		0 & 0 & 1 & 0 & x_{35}  \\
		0 & 0 & 0 & 1 & 0  \\
		0 & 0 & 0 & 0 & 1 \\
	\end{matrix}\right),x^A=\left(\begin{matrix}
		1 & 1 & 0 & 0 & 0 \\
		0 & 1 & 1 & 0 & 0 \\
		0 & 0 & 1 & 0 & 1 \\
		0 & 0 & 0 & 1 & 0 \\
		0 & 0 & 0 & 0 & 1
	\end{matrix}\right)
\end{equation*} \newline
Where matrix $A$ has entries
\begin{align*}
	&d_{1} = 1, d_{2} = \frac{1}{x_{12}}, d_{3} = \frac{1}{x_{12} x_{23}}, d_{4} = 1, d_{5} = \frac{1}{x_{12} x_{23} x_{35}}, \\ 
	&a_{12} = 1, a_{13} = 1, a_{14} = 1, a_{15} = 1, \\ 
	&a_{23} = \frac{x_{12} x_{23} - x_{13}}{x_{12}^{2} x_{23}}, a_{24} = -\frac{x_{14} x_{23} - x_{13} x_{24}}{x_{12} x_{23}}, a_{25} = 1, \\ 
	&a_{34} = -\frac{x_{24}}{x_{23}}, a_{35} = \frac{x_{12} x_{15} x_{24} + {\left(x_{12} x_{14} x_{23} - x_{13} x_{14} + {\left(x_{12}^{3} - x_{12}^{2}\right)} x_{23} x_{24}\right)} x_{35}}{{\left(x_{12}^{2} x_{14} x_{23}^{2} - x_{12}^{2} x_{13} x_{23} x_{24}\right)} x_{35}}, \\ 
	&a_{45} = -\frac{x_{12} x_{15} x_{23} + {\left(x_{12} x_{13} x_{23} - x_{13}^{2} + {\left(x_{12}^{3} - x_{12}^{2}\right)} x_{23}^{2}\right)} x_{35}}{{\left(x_{12}^{2} x_{14} x_{23}^{2} - x_{12}^{2} x_{13} x_{23} x_{24}\right)} x_{35}}
\end{align*}

Now assume $x_{14}= \frac{x_{13}x_{24}}{x_{23}}$
\begin{equation*}x=\left(\begin{matrix}
		1 & x_{12} & x_{13} & x_{14} & x_{15} \\
		0 & 1 & x_{23} & x_{24} & 0 \\
		0 & 0 & 1 & 0 & x_{35}  \\
		0 & 0 & 0 & 1 & 0  \\
		0 & 0 & 0 & 0 & 1 \\
	\end{matrix}\right),x^A=\left(\begin{matrix}
		1 & 1 & 0 & 0 & 0 \\
		0 & 1 & 1 & 0 & 0 \\
		0 & 0 & 1 & 0 & 1 \\
		0 & 0 & 0 & 1 & 0 \\
		0 & 0 & 0 & 0 & 1
	\end{matrix}\right)
\end{equation*} \newline
Where matrix $A$ has entries
\begin{align*}
	&d_{1} = 1, d_{2} = \frac{1}{x_{12}}, d_{3} = \frac{1}{x_{12} x_{23}}, d_{4} = 1, d_{5} = \frac{1}{x_{12} x_{23} x_{35}}, \\ 
	&a_{12} = 1, a_{13} = 1, a_{14} = 1, a_{15} = 1, \\ 
	&a_{23} = \frac{x_{12} x_{23} - x_{13}}{x_{12}^{2} x_{23}}, a_{24} = 0, a_{25} = -\frac{x_{12} x_{15} x_{23} - {\left(x_{12}^{2} x_{23}^{2} - x_{12} x_{13} x_{23} + x_{13}^{2}\right)} x_{35}}{x_{12}^{3} x_{23}^{2} x_{35}}, \\ 
	&a_{34} = -\frac{x_{24}}{x_{23}}, a_{35} = 1, \\ 
	&a_{45} = -\frac{x_{12}^{2} x_{23}^{2} - x_{12} x_{23} + x_{13}}{x_{12}^{2} x_{23} x_{24}}
\end{align*}

\begin{equation*}x=\left(
\right)
\end{equation*} \newline
Where matrix $A$ has entries
\begin{align*}
	&d_{1} = 1, d_{2} = \frac{1}{x_{12}}, d_{3} = \frac{1}{x_{12} x_{23}}, d_{4} = 1, d_{5} = \frac{1}{x_{12} x_{23} x_{35}}, \\ 
	&a_{12} = 1, a_{13} = 1, a_{14} = 1, a_{15} = 1, \\ 
	&a_{23} = \frac{1}{x_{12}}, a_{24} = -\frac{x_{14}}{x_{12}}, a_{25} = 1, \\ 
	&a_{34} = -\frac{x_{24}}{x_{23}}, a_{35} = -\frac{x_{14} x_{25} - {\left(x_{14} x_{23} + {\left(x_{12}^{2} - x_{12}\right)} x_{23} x_{24}\right)} x_{35}}{x_{12} x_{14} x_{23}^{2} x_{35}}, \\ 
	&a_{45} = -\frac{x_{12} - 1}{x_{14}}
\end{align*}

\begin{equation*}x=\left(\begin{matrix}
		1 & x_{12} & 0 & x_{14} & x_{15} \\
		0 & 1 & x_{23} & x_{24} & x_{25} \\
		0 & 0 & 1 & 0 & x_{35}  \\
		0 & 0 & 0 & 1 & 0  \\
		0 & 0 & 0 & 0 & 1 \\
	\end{matrix}\right),x^A=\left(\begin{matrix}
		1 & 1 & 0 & 0 & 0 \\
		0 & 1 & 1 & 0 & 0 \\
		0 & 0 & 1 & 0 & 1 \\
		0 & 0 & 0 & 1 & 0 \\
		0 & 0 & 0 & 0 & 1
	\end{matrix}\right)
\end{equation*} \newline
Where matrix $A$ has entries
\begin{align*}
	&d_{1} = 1, d_{2} = \frac{1}{x_{12}}, d_{3} = \frac{1}{x_{12} x_{23}}, d_{4} = 1, d_{5} = \frac{1}{x_{12} x_{23} x_{35}}, \\ 
	&a_{12} = 1, a_{13} = 1, a_{14} = 1, a_{15} = 1, \\ 
	&a_{23} = \frac{1}{x_{12}}, a_{24} = -\frac{x_{14}}{x_{12}}, a_{25} = 1, \\ 
	&a_{34} = -\frac{x_{24}}{x_{23}}, a_{35} = \frac{x_{15} x_{24} - x_{14} x_{25} + {\left(x_{14} x_{23} + {\left(x_{12}^{2} - x_{12}\right)} x_{23} x_{24}\right)} x_{35}}{x_{12} x_{14} x_{23}^{2} x_{35}}, \\ 
	&a_{45} = -\frac{{\left(x_{12}^{2} - x_{12}\right)} x_{23} x_{35} + x_{15}}{x_{12} x_{14} x_{23} x_{35}}
\end{align*}

\begin{equation*}x=\left(\begin{matrix}
		1 & x_{12} & x_{13} & 0 & 0 \\
		0 & 1 & x_{23} & x_{24} & x_{25} \\
		0 & 0 & 1 & 0 & x_{35}  \\
		0 & 0 & 0 & 1 & 0  \\
		0 & 0 & 0 & 0 & 1 \\
	\end{matrix}\right),x^A=\left(\begin{matrix}
		1 & 1 & 0 & 0 & 0 \\
		0 & 1 & 1 & 0 & 0 \\
		0 & 0 & 1 & 0 & 1 \\
		0 & 0 & 0 & 1 & 0 \\
		0 & 0 & 0 & 0 & 1
	\end{matrix}\right)
\end{equation*} \newline
Where matrix $A$ has entries
\begin{align*}
	&d_{1} = 1, d_{2} = \frac{1}{x_{12}}, d_{3} = \frac{1}{x_{12} x_{23}}, d_{4} = 1, d_{5} = \frac{1}{x_{12} x_{23} x_{35}}, \\ 
	&a_{12} = 1, a_{13} = 1, a_{14} = 1, a_{15} = 1, \\ 
	&a_{23} = \frac{x_{12} x_{23} - x_{13}}{x_{12}^{2} x_{23}}, a_{24} = \frac{x_{13} x_{24}}{x_{12} x_{23}}, a_{25} = 1, \\ 
	&a_{34} = -\frac{x_{24}}{x_{23}}, a_{35} = -\frac{x_{12} - 1}{x_{13}}, \\ 
	&a_{45} = -\frac{x_{12} x_{13} x_{25} - {\left(x_{12} x_{13} x_{23} - x_{13}^{2} + {\left(x_{12}^{3} - x_{12}^{2}\right)} x_{23}^{2}\right)} x_{35}}{x_{12}^{2} x_{13} x_{23} x_{24} x_{35}}
\end{align*}

\begin{equation*}x=\left(\begin{matrix}
		1 & x_{12} & x_{13} & 0 & x_{15} \\
		0 & 1 & x_{23} & x_{24} & x_{25} \\
		0 & 0 & 1 & 0 & x_{35}  \\
		0 & 0 & 0 & 1 & 0  \\
		0 & 0 & 0 & 0 & 1 \\
	\end{matrix}\right),x^A=\left(\begin{matrix}
		1 & 1 & 0 & 0 & 0 \\
		0 & 1 & 1 & 0 & 0 \\
		0 & 0 & 1 & 0 & 1 \\
		0 & 0 & 0 & 1 & 0 \\
		0 & 0 & 0 & 0 & 1
	\end{matrix}\right)
\end{equation*} \newline
Where matrix $A$ has entries
\begin{align*}
	&d_{1} = 1, d_{2} = \frac{1}{x_{12}}, d_{3} = \frac{1}{x_{12} x_{23}}, d_{4} = 1, d_{5} = \frac{1}{x_{12} x_{23} x_{35}}, \\ 
	&a_{12} = 1, a_{13} = 1, a_{14} = 1, a_{15} = 1, \\ 
	&a_{23} = \frac{x_{12} x_{23} - x_{13}}{x_{12}^{2} x_{23}}, a_{24} = \frac{x_{13} x_{24}}{x_{12} x_{23}}, a_{25} = 1, \\ 
	&a_{34} = -\frac{x_{24}}{x_{23}}, a_{35} = -\frac{{\left(x_{12}^{2} - x_{12}\right)} x_{23} x_{35} + x_{15}}{x_{12} x_{13} x_{23} x_{35}}, \\ 
	&a_{45} = \frac{x_{12} x_{15} x_{23} - x_{12} x_{13} x_{25} + {\left(x_{12} x_{13} x_{23} - x_{13}^{2} + {\left(x_{12}^{3} - x_{12}^{2}\right)} x_{23}^{2}\right)} x_{35}}{x_{12}^{2} x_{13} x_{23} x_{24} x_{35}}
\end{align*}


First assume $x_{14}\neq \frac{x_{13}x_{24}}{x_{23}}$
\begin{equation*}x=\left(\begin{matrix}
		1 & x_{12} & x_{13} & x_{14} & 0 \\
		0 & 1 & x_{23} & x_{24} & x_{25} \\
		0 & 0 & 1 & 0 & x_{35}  \\
		0 & 0 & 0 & 1 & 0  \\
		0 & 0 & 0 & 0 & 1 \\
	\end{matrix}\right),x^A=\left(\begin{matrix}
		1 & 1 & 0 & 0 & 0 \\
		0 & 1 & 1 & 0 & 0 \\
		0 & 0 & 1 & 0 & 1 \\
		0 & 0 & 0 & 1 & 0 \\
		0 & 0 & 0 & 0 & 1
	\end{matrix}\right)
\end{equation*} \newline
Where matrix $A$ has entries
\begin{align*}
	&d_{1} = 1, d_{2} = \frac{1}{x_{12}}, d_{3} = \frac{1}{x_{12} x_{23}}, d_{4} = 1, d_{5} = \frac{1}{x_{12} x_{23} x_{35}}, \\ 
	&a_{12} = 1, a_{13} = 1, a_{14} = 1, a_{15} = 1, \\ 
	&a_{23} = \frac{x_{12} x_{23} - x_{13}}{x_{12}^{2} x_{23}}, a_{24} = -\frac{x_{14} x_{23} - x_{13} x_{24}}{x_{12} x_{23}}, a_{25} = 1, \\ 
	&a_{34} = -\frac{x_{24}}{x_{23}}, a_{35} = -\frac{x_{12} x_{14} x_{25} - {\left(x_{12} x_{14} x_{23} - x_{13} x_{14} + {\left(x_{12}^{3} - x_{12}^{2}\right)} x_{23} x_{24}\right)} x_{35}}{{\left(x_{12}^{2} x_{14} x_{23}^{2} - x_{12}^{2} x_{13} x_{23} x_{24}\right)} x_{35}}, \\ 
	&a_{45} = \frac{x_{12} x_{13} x_{25} - {\left(x_{12} x_{13} x_{23} - x_{13}^{2} + {\left(x_{12}^{3} - x_{12}^{2}\right)} x_{23}^{2}\right)} x_{35}}{{\left(x_{12}^{2} x_{14} x_{23}^{2} - x_{12}^{2} x_{13} x_{23} x_{24}\right)} x_{35}}
\end{align*}

Now assume $x_{14}= \frac{x_{13}x_{24}}{x_{23}}$
\begin{equation*}x=\left(\begin{matrix}
		1 & x_{12} & x_{13} & x_{14} & 0 \\
		0 & 1 & x_{23} & x_{24} & x_{25} \\
		0 & 0 & 1 & 0 & x_{35}  \\
		0 & 0 & 0 & 1 & 0  \\
		0 & 0 & 0 & 0 & 1 \\
	\end{matrix}\right),x^A=\left(\begin{matrix}
		1 & 1 & 0 & 0 & 0 \\
		0 & 1 & 1 & 0 & 0 \\
		0 & 0 & 1 & 0 & 1 \\
		0 & 0 & 0 & 1 & 0 \\
		0 & 0 & 0 & 0 & 1
	\end{matrix}\right)
\end{equation*} \newline
Where matrix $A$ has entries
\begin{align*}
	&d_{1} = 1, d_{2} = \frac{1}{x_{12}}, d_{3} = \frac{1}{x_{12} x_{23}}, d_{4} = 1, d_{5} = \frac{1}{x_{12} x_{23} x_{35}}, \\ 
	&a_{12} = 1, a_{13} = 1, a_{14} = 1, a_{15} = 1, \\ 
	&a_{23} = \frac{x_{12} x_{23} - x_{13}}{x_{12}^{2} x_{23}}, a_{24} = 0, a_{25} = \frac{x_{12} x_{13} x_{25} + {\left(x_{12}^{2} x_{23}^{2} - x_{12} x_{13} x_{23} + x_{13}^{2}\right)} x_{35}}{x_{12}^{3} x_{23}^{2} x_{35}}, \\ 
	&a_{34} = -\frac{x_{24}}{x_{23}}, a_{35} = 1, \\ 
	&a_{45} = -\frac{x_{12} x_{25} + {\left(x_{12}^{2} x_{23}^{2} - x_{12} x_{23} + x_{13}\right)} x_{35}}{x_{12}^{2} x_{23} x_{24} x_{35}}
\end{align*}


First assume $x_{14}\neq \frac{x_{13}x_{24}}{x_{23}}$
\begin{equation*}x=\left(\begin{matrix}
		1 & x_{12} & x_{13} & x_{14} & x_{15} \\
		0 & 1 & x_{23} & x_{24} & x_{25} \\
		0 & 0 & 1 & 0 & x_{35}  \\
		0 & 0 & 0 & 1 & 0  \\
		0 & 0 & 0 & 0 & 1 \\
	\end{matrix}\right),x^A=\left(\begin{matrix}
		1 & 1 & 0 & 0 & 0 \\
		0 & 1 & 1 & 0 & 0 \\
		0 & 0 & 1 & 0 & 1 \\
		0 & 0 & 0 & 1 & 0 \\
		0 & 0 & 0 & 0 & 1
	\end{matrix}\right)
\end{equation*} \newline
Where matrix $A$ has entries
\begin{align*}
	&d_{1} = 1, d_{2} = \frac{1}{x_{12}}, d_{3} = \frac{1}{x_{12} x_{23}}, d_{4} = 1, d_{5} = \frac{1}{x_{12} x_{23} x_{35}}, \\ 
	&a_{12} = 1, a_{13} = 1, a_{14} = 1, a_{15} = 1, \\ 
	&a_{23} = \frac{x_{12} x_{23} - x_{13}}{x_{12}^{2} x_{23}}, a_{24} = -\frac{x_{14} x_{23} - x_{13} x_{24}}{x_{12} x_{23}}, a_{25} = 1, \\ 
	&a_{34} = -\frac{x_{24}}{x_{23}}, a_{35} = \frac{x_{12} x_{15} x_{24} - x_{12} x_{14} x_{25} + {\left(x_{12} x_{14} x_{23} - x_{13} x_{14} + {\left(x_{12}^{3} - x_{12}^{2}\right)} x_{23} x_{24}\right)} x_{35}}{{\left(x_{12}^{2} x_{14} x_{23}^{2} - x_{12}^{2} x_{13} x_{23} x_{24}\right)} x_{35}}, \\ 
	&a_{45} = -\frac{x_{12} x_{15} x_{23} - x_{12} x_{13} x_{25} + {\left(x_{12} x_{13} x_{23} - x_{13}^{2} + {\left(x_{12}^{3} - x_{12}^{2}\right)} x_{23}^{2}\right)} x_{35}}{{\left(x_{12}^{2} x_{14} x_{23}^{2} - x_{12}^{2} x_{13} x_{23} x_{24}\right)} x_{35}}
\end{align*}

Now assume $x_{14}= \frac{x_{13}x_{24}}{x_{23}}$
\begin{equation*}x=\left(\begin{matrix}
		1 & x_{12} & x_{13} & x_{14} & x_{15} \\
		0 & 1 & x_{23} & x_{24} & x_{25} \\
		0 & 0 & 1 & 0 & x_{35}  \\
		0 & 0 & 0 & 1 & 0  \\
		0 & 0 & 0 & 0 & 1 \\
	\end{matrix}\right),x^A=\left(\begin{matrix}
		1 & 1 & 0 & 0 & 0 \\
		0 & 1 & 1 & 0 & 0 \\
		0 & 0 & 1 & 0 & 1 \\
		0 & 0 & 0 & 1 & 0 \\
		0 & 0 & 0 & 0 & 1
	\end{matrix}\right)
\end{equation*} \newline
Where matrix $A$ has entries
\begin{align*}
	&d_{1} = 1, d_{2} = \frac{1}{x_{12}}, d_{3} = \frac{1}{x_{12} x_{23}}, d_{4} = 1, d_{5} = \frac{1}{x_{12} x_{23} x_{35}}, \\ 
	&a_{12} = 1, a_{13} = 1, a_{14} = 1, a_{15} = 1, \\ 
	&a_{23} = \frac{x_{12} x_{23} - x_{13}}{x_{12}^{2} x_{23}}, a_{24} = 0, a_{25} = -\frac{x_{12} x_{15} x_{23} - x_{12} x_{13} x_{25} - {\left(x_{12}^{2} x_{23}^{2} - x_{12} x_{13} x_{23} + x_{13}^{2}\right)} x_{35}}{x_{12}^{3} x_{23}^{2} x_{35}}, \\ 
	&a_{34} = -\frac{x_{24}}{x_{23}}, a_{35} = 1, \\ 
	&a_{45} = -\frac{x_{12} x_{25} + {\left(x_{12}^{2} x_{23}^{2} - x_{12} x_{23} + x_{13}\right)} x_{35}}{x_{12}^{2} x_{23} x_{24} x_{35}}
\end{align*}

		\section{Subcases of  $Y_5$}

\begin{equation*}x=\left(
\right)
\end{equation*} \newline
Where matrix $A$ has entries
\begin{align*}
	&d_{1} = 1, d_{2} = \frac{x_{24}}{x_{14}}, d_{3} = \frac{x_{24}}{x_{14} x_{23}}, d_{4} = \frac{1}{x_{14}}, d_{5} = \frac{1}{x_{14} x_{45}}, \\ 
	&a_{12} = 0, a_{13} = 1, a_{14} = 1, a_{15} = 1, \\ 
	&a_{23} = 1, a_{24} = 1, a_{25} = 1, \\ 
	&a_{34} = -\frac{x_{24}}{x_{14} x_{23}}, a_{35} = \frac{x_{15} x_{24} + {\left({2} \, x_{14}^{2} - {2} \, x_{14} x_{24}\right)} x_{45}}{x_{14}^{2} x_{23} x_{45}}, \\ 
	&a_{45} = \frac{{2} \, x_{14} x_{45} - x_{15}}{x_{14}^{2} x_{45}}
\end{align*}

\begin{equation*}x=\left(\begin{matrix}
		1 & 0 & x_{13} & 0 & 0 \\
		0 & 1 & x_{23} & x_{24} & 0 \\
		0 & 0 & 1 & 0 & 0  \\
		0 & 0 & 0 & 1 & x_{45}  \\
		0 & 0 & 0 & 0 & 1 \\
	\end{matrix}\right),x^A=\left(\begin{matrix}
		1 & 0 & 0 & 1 & 0 \\
		0 & 1 & 1 & 0 & 0 \\
		0 & 0 & 1 & 0 & 1 \\
		0 & 0 & 0 & 1 & 1 \\
		0 & 0 & 0 & 0 & 1
	\end{matrix}\right)
\end{equation*} \newline
Where matrix $A$ has entries
\begin{align*}
	&d_{1} = 1, d_{2} = -\frac{x_{23}}{x_{13}}, d_{3} = -\frac{1}{x_{13}}, d_{4} = -\frac{x_{23}}{x_{13} x_{24}}, d_{5} = -\frac{x_{23}}{x_{13} x_{24} x_{45}}, \\ 
	&a_{12} = -1, a_{13} = 1, a_{14} = 1, a_{15} = 1, \\ 
	&a_{23} = 1, a_{24} = 1, a_{25} = 1, \\ 
	&a_{34} = \frac{1}{x_{13}}, a_{35} = \frac{2}{x_{13}}, \\ 
	&a_{45} = \frac{{2} \, x_{13} - {2} \, x_{23}}{x_{13} x_{24}}
\end{align*}

\begin{equation*}x=\left(\begin{matrix}
		1 & 0 & x_{13} & 0 & x_{15} \\
		0 & 1 & x_{23} & x_{24} & 0 \\
		0 & 0 & 1 & 0 & 0  \\
		0 & 0 & 0 & 1 & x_{45}  \\
		0 & 0 & 0 & 0 & 1 \\
	\end{matrix}\right),x^A=\left(\begin{matrix}
		1 & 0 & 0 & 1 & 0 \\
		0 & 1 & 1 & 0 & 0 \\
		0 & 0 & 1 & 0 & 1 \\
		0 & 0 & 0 & 1 & 1 \\
		0 & 0 & 0 & 0 & 1
	\end{matrix}\right)
\end{equation*} \newline
Where matrix $A$ has entries
\begin{align*}
	&d_{1} = 1, d_{2} = -\frac{x_{23}}{x_{13}}, d_{3} = -\frac{1}{x_{13}}, d_{4} = -\frac{x_{23}}{x_{13} x_{24}}, d_{5} = -\frac{x_{23}}{x_{13} x_{24} x_{45}}, \\ 
	&a_{12} = -1, a_{13} = 1, a_{14} = 1, a_{15} = 1, \\ 
	&a_{23} = 1, a_{24} = 1, a_{25} = 1, \\ 
	&a_{34} = \frac{1}{x_{13}}, a_{35} = \frac{{2} \, x_{13} x_{24} x_{45} + x_{15} x_{23}}{x_{13}^{2} x_{24} x_{45}}, \\ 
	&a_{45} = -\frac{x_{15} x_{23}^{2} - {\left({2} \, x_{13}^{2} - {2} \, x_{13} x_{23}\right)} x_{24} x_{45}}{x_{13}^{2} x_{24}^{2} x_{45}}
\end{align*}

First assume $x_{14}\neq \frac{x_{13}x_{24}}{x_{23}}$
\begin{equation*}x=\left(\begin{matrix}
		1 & 0 & x_{13} & x_{14} & 0 \\
		0 & 1 & x_{23} & x_{24} & 0 \\
		0 & 0 & 1 & 0 & 0  \\
		0 & 0 & 0 & 1 & x_{45}  \\
		0 & 0 & 0 & 0 & 1 \\
	\end{matrix}\right),x^A=\left(\begin{matrix}
		1 & 0 & 0 & 1 & 0 \\
		0 & 1 & 1 & 0 & 0 \\
		0 & 0 & 1 & 0 & 1 \\
		0 & 0 & 0 & 1 & 1 \\
		0 & 0 & 0 & 0 & 1
	\end{matrix}\right)
\end{equation*} \newline
Where matrix $A$ has entries
\begin{align*}
	&d_{1} = 1, d_{2} = \frac{x_{23} x_{24}}{x_{14} x_{23} - x_{13} x_{24}}, d_{3} = \frac{x_{24}}{x_{14} x_{23} - x_{13} x_{24}}, d_{4} = \frac{x_{23}}{x_{14} x_{23} - x_{13} x_{24}}, d_{5} = \frac{x_{23}}{{\left(x_{14} x_{23} - x_{13} x_{24}\right)} x_{45}}, \\ 
	&a_{12} = \frac{x_{13} x_{24}}{x_{14} x_{23} - x_{13} x_{24}}, a_{13} = 1, a_{14} = 1, a_{15} = 1, \\ 
	&a_{23} = 1, a_{24} = 1, a_{25} = 1, \\ 
	&a_{34} = -\frac{x_{24}}{x_{14} x_{23} - x_{13} x_{24}}, a_{35} = \frac{{2} \, x_{14} - {2} \, x_{24}}{x_{14} x_{23} - x_{13} x_{24}}, \\ 
	&a_{45} = -\frac{{2} \, x_{13} - {2} \, x_{23}}{x_{14} x_{23} - x_{13} x_{24}}
\end{align*}

Now assume $x_{14}= \frac{x_{13}x_{24}}{x_{23}}$
\begin{equation*}x=\left(\begin{matrix}
		1 & 0 & x_{13} & x_{14} & 0 \\
		0 & 1 & x_{23} & x_{24} & 0 \\
		0 & 0 & 1 & 0 & 0  \\
		0 & 0 & 0 & 1 & x_{45}  \\
		0 & 0 & 0 & 0 & 1 \\
	\end{matrix}\right),x^A=\left(\begin{matrix}
		1 & 0 & 0 & 0 & 0 \\
		0 & 1 & 1 & 0 & 0 \\
		0 & 0 & 1 & 0 & 1 \\
		0 & 0 & 0 & 1 & 1 \\
		0 & 0 & 0 & 0 & 1
	\end{matrix}\right)
\end{equation*} \newline
Where matrix $A$ has entries
\begin{align*}
	&d_{1} = 1, d_{2} = 1, d_{3} = \frac{1}{x_{23}}, d_{4} = \frac{1}{x_{24}}, d_{5} = \frac{1}{x_{24} x_{45}}, \\ 
	&a_{12} = \frac{x_{13}}{x_{23}}, a_{13} = 1, a_{14} = 1, a_{15} = 1, \\ 
	&a_{23} = 1, a_{24} = -\frac{x_{13} - {2} \, x_{23}}{x_{13}}, a_{25} = 1, \\ 
	&a_{34} = -\frac{1}{x_{23}}, a_{35} = 1, \\ 
	&a_{45} = -\frac{{\left(x_{13} - {2} \,\right)} x_{23}}{x_{13} x_{24}}
\end{align*}


First assume $x_{14}\neq \frac{x_{13}x_{24}}{x_{23}}$
\begin{equation*}x=\left(\begin{matrix}
		1 & 0 & x_{13} & x_{14} & x_{15} \\
		0 & 1 & x_{23} & x_{24} & 0 \\
		0 & 0 & 1 & 0 & 0  \\
		0 & 0 & 0 & 1 & x_{45}  \\
		0 & 0 & 0 & 0 & 1 \\
	\end{matrix}\right),x^A=\left(\begin{matrix}
		1 & 0 & 0 & 1 & 0 \\
		0 & 1 & 1 & 0 & 0 \\
		0 & 0 & 1 & 0 & 1 \\
		0 & 0 & 0 & 1 & 1 \\
		0 & 0 & 0 & 0 & 1
	\end{matrix}\right)
\end{equation*} \newline
Where matrix $A$ has entries
\begin{align*}
	&d_{1} = 1, d_{2} = \frac{x_{23} x_{24}}{x_{14} x_{23} - x_{13} x_{24}}, d_{3} = \frac{x_{24}}{x_{14} x_{23} - x_{13} x_{24}}, d_{4} = \frac{x_{23}}{x_{14} x_{23} - x_{13} x_{24}}, d_{5} = \frac{x_{23}}{{\left(x_{14} x_{23} - x_{13} x_{24}\right)} x_{45}}, \\ 
	&a_{12} = \frac{x_{13} x_{24}}{x_{14} x_{23} - x_{13} x_{24}}, a_{13} = 1, a_{14} = 1, a_{15} = 1, \\ 
	&a_{23} = 1, a_{24} = 1, a_{25} = 1, \\ 
	&a_{34} = -\frac{x_{24}}{x_{14} x_{23} - x_{13} x_{24}}, a_{35} = \frac{x_{15} x_{23} x_{24} + {\left({2} \, x_{14}^{2} x_{23} + {2} \, x_{13} x_{24}^{2} - {\left({2} \, x_{13} x_{14} + {2} \, x_{14} x_{23}\right)} x_{24}\right)} x_{45}}{{\left(x_{14}^{2} x_{23}^{2} - 2 \, x_{13} x_{14} x_{23} x_{24} + x_{13}^{2} x_{24}^{2}\right)} x_{45}}, \\ 
	&a_{45} = -\frac{x_{15} x_{23}^{2} + {\left({2} \, x_{13} x_{14} x_{23} - {2} \, x_{14} x_{23}^{2} - {\left({2} \, x_{13}^{2} - {2} \, x_{13} x_{23}\right)} x_{24}\right)} x_{45}}{{\left(x_{14}^{2} x_{23}^{2} - 2 \, x_{13} x_{14} x_{23} x_{24} + x_{13}^{2} x_{24}^{2}\right)} x_{45}}
\end{align*}

Now assume $x_{14}= \frac{x_{13}x_{24}}{x_{23}}$
\begin{equation*}x=\left(\begin{matrix}
		1 & 0 & x_{13} & x_{14} & x_{15} \\
		0 & 1 & x_{23} & x_{24} & 0 \\
		0 & 0 & 1 & 0 & 0  \\
		0 & 0 & 0 & 1 & x_{45}  \\
		0 & 0 & 0 & 0 & 1 \\
	\end{matrix}\right),x^A=\left(\begin{matrix}
		1 & 0 & 0 & 0 & 0 \\
		0 & 1 & 1 & 0 & 0 \\
		0 & 0 & 1 & 0 & 1 \\
		0 & 0 & 0 & 1 & 1 \\
		0 & 0 & 0 & 0 & 1
	\end{matrix}\right)
\end{equation*} \newline
Where matrix $A$ has entries
\begin{align*}
	&d_{1} = 1, d_{2} = 1, d_{3} = \frac{1}{x_{23}}, d_{4} = \frac{1}{x_{24}}, d_{5} = \frac{1}{x_{24} x_{45}}, \\ 
	&a_{12} = \frac{x_{13}}{x_{23}}, a_{13} = 1, a_{14} = 1, a_{15} = 1, \\ 
	&a_{23} = 1, a_{24} = -\frac{x_{15} x_{23} + {\left(x_{13} - {2} \, x_{23}\right)} x_{24} x_{45}}{x_{13} x_{24} x_{45}}, a_{25} = 1, \\ 
	&a_{34} = -\frac{1}{x_{23}}, a_{35} = 1, \\ 
	&a_{45} = -\frac{{\left(x_{13} - {2} \,\right)} x_{23} x_{24} x_{45} + x_{15} x_{23}}{x_{13} x_{24}^{2} x_{45}}
\end{align*}

\begin{equation*}x=\left(
\right)
\end{equation*} \newline
Where matrix $A$ has entries
\begin{align*}
	&d_{1} = 1, d_{2} = \frac{x_{24}}{x_{14}}, d_{3} = \frac{x_{24}}{x_{14} x_{23}}, d_{4} = \frac{1}{x_{14}}, d_{5} = \frac{1}{x_{14} x_{45}}, \\ 
	&a_{12} = 0, a_{13} = 1, a_{14} = 1, a_{15} = 1, \\ 
	&a_{23} = 1, a_{24} = 1, a_{25} = 1, \\ 
	&a_{34} = -\frac{x_{24}}{x_{14} x_{23}}, a_{35} = \frac{x_{15} x_{24} - x_{14} x_{25} + {\left({2} \, x_{14}^{2} - {2} \, x_{14} x_{24}\right)} x_{45}}{x_{14}^{2} x_{23} x_{45}}, \\ 
	&a_{45} = \frac{{2} \, x_{14} x_{45} - x_{15}}{x_{14}^{2} x_{45}}
\end{align*}

\begin{equation*}x=\left(\begin{matrix}
		1 & 0 & x_{13} & 0 & 0 \\
		0 & 1 & x_{23} & x_{24} & x_{25} \\
		0 & 0 & 1 & 0 & 0  \\
		0 & 0 & 0 & 1 & x_{45}  \\
		0 & 0 & 0 & 0 & 1 \\
	\end{matrix}\right),x^A=\left(\begin{matrix}
		1 & 0 & 0 & 1 & 0 \\
		0 & 1 & 1 & 0 & 0 \\
		0 & 0 & 1 & 0 & 1 \\
		0 & 0 & 0 & 1 & 1 \\
		0 & 0 & 0 & 0 & 1
	\end{matrix}\right)
\end{equation*} \newline
Where matrix $A$ has entries
\begin{align*}
	&d_{1} = 1, d_{2} = -\frac{x_{23}}{x_{13}}, d_{3} = -\frac{1}{x_{13}}, d_{4} = -\frac{x_{23}}{x_{13} x_{24}}, d_{5} = -\frac{x_{23}}{x_{13} x_{24} x_{45}}, \\ 
	&a_{12} = -1, a_{13} = 1, a_{14} = 1, a_{15} = 1, \\ 
	&a_{23} = 1, a_{24} = 1, a_{25} = 1, \\ 
	&a_{34} = \frac{1}{x_{13}}, a_{35} = \frac{2}{x_{13}}, \\ 
	&a_{45} = \frac{x_{23} x_{25} + {\left({2} \, x_{13} - {2} \, x_{23}\right)} x_{24} x_{45}}{x_{13} x_{24}^{2} x_{45}}
\end{align*}

\begin{equation*}x=\left(\begin{matrix}
		1 & 0 & x_{13} & 0 & x_{15} \\
		0 & 1 & x_{23} & x_{24} & x_{25} \\
		0 & 0 & 1 & 0 & 0  \\
		0 & 0 & 0 & 1 & x_{45}  \\
		0 & 0 & 0 & 0 & 1 \\
	\end{matrix}\right),x^A=\left(\begin{matrix}
		1 & 0 & 0 & 1 & 0 \\
		0 & 1 & 1 & 0 & 0 \\
		0 & 0 & 1 & 0 & 1 \\
		0 & 0 & 0 & 1 & 1 \\
		0 & 0 & 0 & 0 & 1
	\end{matrix}\right)
\end{equation*} \newline
Where matrix $A$ has entries
\begin{align*}
	&d_{1} = 1, d_{2} = -\frac{x_{23}}{x_{13}}, d_{3} = -\frac{1}{x_{13}}, d_{4} = -\frac{x_{23}}{x_{13} x_{24}}, d_{5} = -\frac{x_{23}}{x_{13} x_{24} x_{45}}, \\ 
	&a_{12} = -1, a_{13} = 1, a_{14} = 1, a_{15} = 1, \\ 
	&a_{23} = 1, a_{24} = 1, a_{25} = 1, \\ 
	&a_{34} = \frac{1}{x_{13}}, a_{35} = \frac{{2} \, x_{13} x_{24} x_{45} + x_{15} x_{23}}{x_{13}^{2} x_{24} x_{45}}, \\ 
	&a_{45} = -\frac{x_{15} x_{23}^{2} - x_{13} x_{23} x_{25} - {\left({2} \, x_{13}^{2} - {2} \, x_{13} x_{23}\right)} x_{24} x_{45}}{x_{13}^{2} x_{24}^{2} x_{45}}
\end{align*}

First assume $x_{14}\neq \frac{x_{13}x_{24}}{x_{23}}$
\begin{equation*}x=\left(\begin{matrix}
		1 & 0 & x_{13} & x_{14} & 0 \\
		0 & 1 & x_{23} & x_{24} & x_{25} \\
		0 & 0 & 1 & 0 & 0  \\
		0 & 0 & 0 & 1 & x_{45}  \\
		0 & 0 & 0 & 0 & 1 \\
	\end{matrix}\right),x^A=\left(\begin{matrix}
		1 & 0 & 0 & 1 & 0 \\
		0 & 1 & 1 & 0 & 0 \\
		0 & 0 & 1 & 0 & 1 \\
		0 & 0 & 0 & 1 & 1 \\
		0 & 0 & 0 & 0 & 1
	\end{matrix}\right)
\end{equation*} \newline
Where matrix $A$ has entries
\begin{align*}
	&d_{1} = 1, d_{2} = \frac{x_{23} x_{24}}{x_{14} x_{23} - x_{13} x_{24}}, d_{3} = \frac{x_{24}}{x_{14} x_{23} - x_{13} x_{24}}, d_{4} = \frac{x_{23}}{x_{14} x_{23} - x_{13} x_{24}}, d_{5} = \frac{x_{23}}{{\left(x_{14} x_{23} - x_{13} x_{24}\right)} x_{45}}, \\ 
	&a_{12} = \frac{x_{13} x_{24}}{x_{14} x_{23} - x_{13} x_{24}}, a_{13} = 1, a_{14} = 1, a_{15} = 1, \\ 
	&a_{23} = 1, a_{24} = 1, a_{25} = 1, \\ 
	&a_{34} = -\frac{x_{24}}{x_{14} x_{23} - x_{13} x_{24}}, a_{35} = -\frac{x_{14} x_{23} x_{25} - {\left({2} \, x_{14}^{2} x_{23} + {2} \, x_{13} x_{24}^{2} - {\left({2} \, x_{13} x_{14} + {2} \, x_{14} x_{23}\right)} x_{24}\right)} x_{45}}{{\left(x_{14}^{2} x_{23}^{2} - 2 \, x_{13} x_{14} x_{23} x_{24} + x_{13}^{2} x_{24}^{2}\right)} x_{45}}, \\ 
	&a_{45} = \frac{x_{13} x_{23} x_{25} - {\left({2} \, x_{13} x_{14} x_{23} - {2} \, x_{14} x_{23}^{2} - {\left({2} \, x_{13}^{2} - {2} \, x_{13} x_{23}\right)} x_{24}\right)} x_{45}}{{\left(x_{14}^{2} x_{23}^{2} - 2 \, x_{13} x_{14} x_{23} x_{24} + x_{13}^{2} x_{24}^{2}\right)} x_{45}}
\end{align*}
Now assume $x_{14}= \frac{x_{13}x_{24}}{x_{23}}$
\begin{equation*}x=\left(\begin{matrix}
		1 & 0 & x_{13} & x_{14} & 0 \\
		0 & 1 & x_{23} & x_{24} & x_{25} \\
		0 & 0 & 1 & 0 & 0  \\
		0 & 0 & 0 & 1 & x_{45}  \\
		0 & 0 & 0 & 0 & 1 \\
	\end{matrix}\right),x^A=\left(\begin{matrix}
		1 & 0 & 0 & 0 & 0 \\
		0 & 1 & 1 & 0 & 0 \\
		0 & 0 & 1 & 0 & 1 \\
		0 & 0 & 0 & 1 & 1 \\
		0 & 0 & 0 & 0 & 1
	\end{matrix}\right)
\end{equation*} \newline
Where matrix $A$ has entries
\begin{align*}
	&d_{1} = 1, d_{2} = 1, d_{3} = \frac{1}{x_{23}}, d_{4} = \frac{1}{x_{24}}, d_{5} = \frac{1}{x_{24} x_{45}}, \\ 
	&a_{12} = \frac{x_{13}}{x_{23}}, a_{13} = 1, a_{14} = 1, a_{15} = 1, \\ 
	&a_{23} = 1, a_{24} = \frac{x_{13} x_{25} - {\left(x_{13} - {2} \, x_{23}\right)} x_{24} x_{45}}{x_{13} x_{24} x_{45}}, a_{25} = 1, \\ 
	&a_{34} = -\frac{1}{x_{23}}, a_{35} = 1, \\ 
	&a_{45} = -\frac{{\left(x_{13} - {2} \,\right)} x_{23}}{x_{13} x_{24}}
\end{align*}

First assume $x_{14}\neq \frac{x_{13}x_{24}}{x_{23}}$
\begin{equation*}x=\left(\begin{matrix}
		1 & 0 & x_{13} & x_{14} & x_{15} \\
		0 & 1 & x_{23} & x_{24} & x_{25} \\
		0 & 0 & 1 & 0 & 0  \\
		0 & 0 & 0 & 1 & x_{45}  \\
		0 & 0 & 0 & 0 & 1 \\
	\end{matrix}\right),x^A=\left(\begin{matrix}
		1 & 0 & 0 & 1 & 0 \\
		0 & 1 & 1 & 0 & 0 \\
		0 & 0 & 1 & 0 & 1 \\
		0 & 0 & 0 & 1 & 1 \\
		0 & 0 & 0 & 0 & 1 
	\end{matrix}\right)
\end{equation*} \newline
Where matrix $A$ has entries
\begin{align*}
	&d_{1} = 1, d_{2} = \frac{x_{23} x_{24}}{x_{14} x_{23} - x_{13} x_{24}}, d_{3} = \frac{x_{24}}{x_{14} x_{23} - x_{13} x_{24}}, d_{4} = \frac{x_{23}}{x_{14} x_{23} - x_{13} x_{24}}, d_{5} = \frac{x_{23}}{{\left(x_{14} x_{23} - x_{13} x_{24}\right)} x_{45}}, \\ 
	&a_{12} = \frac{x_{13} x_{24}}{x_{14} x_{23} - x_{13} x_{24}}, a_{13} = 1, a_{14} = 1, a_{15} = 1, \\ 
	&a_{23} = 1, a_{24} = 1, a_{25} = 1, \\ 
	&a_{34} = -\frac{x_{24}}{x_{14} x_{23} - x_{13} x_{24}}, a_{35} = \frac{x_{15} x_{23} x_{24} - x_{14} x_{23} x_{25} + {\left({2} \, x_{14}^{2} x_{23} + {2} \, x_{13} x_{24}^{2} - {\left({2} \, x_{13} x_{14} + {2} \, x_{14} x_{23}\right)} x_{24}\right)} x_{45}}{{\left(x_{14}^{2} x_{23}^{2} - 2 \, x_{13} x_{14} x_{23} x_{24} + x_{13}^{2} x_{24}^{2}\right)} x_{45}}, \\ 
	&a_{45} = -\frac{x_{15} x_{23}^{2} - x_{13} x_{23} x_{25} + {\left({2} \, x_{13} x_{14} x_{23} - {2} \, x_{14} x_{23}^{2} - {\left({2} \, x_{13}^{2} - {2} \, x_{13} x_{23}\right)} x_{24}\right)} x_{45}}{{\left(x_{14}^{2} x_{23}^{2} - 2 \, x_{13} x_{14} x_{23} x_{24} + x_{13}^{2} x_{24}^{2}\right)} x_{45}}
\end{align*}

Now assume $x_{14}= \frac{x_{13}x_{24}}{x_{23}}$
\begin{equation*}x=\left(\begin{matrix}
		1 & 0 & x_{13} & x_{14} & x_{15} \\
		0 & 1 & x_{23} & x_{24} & x_{25} \\
		0 & 0 & 1 & 0 & 0  \\
		0 & 0 & 0 & 1 & x_{45}  \\
		0 & 0 & 0 & 0 & 1 \\
	\end{matrix}\right),x^A=\left(\begin{matrix}
		1 & 0 & 0 & 0 & 0 \\
		0 & 1 & 1 & 0 & 0 \\
		0 & 0 & 1 & 0 & 1 \\
		0 & 0 & 0 & 1 & 1 \\
		0 & 0 & 0 & 0 & 1
	\end{matrix}\right)
\end{equation*} \newline
Where matrix $A$ has entries
\begin{align*}
	&d_{1} = 1, d_{2} = 1, d_{3} = \frac{1}{x_{23}}, d_{4} = \frac{1}{x_{24}}, d_{5} = \frac{1}{x_{24} x_{45}}, \\ 
	&a_{12} = \frac{x_{13}}{x_{23}}, a_{13} = 1, a_{14} = 1, a_{15} = 1, \\ 
	&a_{23} = 1, a_{24} = -\frac{x_{15} x_{23} - x_{13} x_{25} + {\left(x_{13} - {2} \, x_{23}\right)} x_{24} x_{45}}{x_{13} x_{24} x_{45}}, a_{25} = 1, \\ 
	&a_{34} = -\frac{1}{x_{23}}, a_{35} = 1, \\ 
	&a_{45} = -\frac{{\left(x_{13} - {2} \,\right)} x_{23} x_{24} x_{45} + x_{15} x_{23}}{x_{13} x_{24}^{2} x_{45}}
\end{align*}

\begin{equation*}x=\left(
\right)
\end{equation*} \newline
Where matrix $A$ has entries
\begin{align*}
	&d_{1} = 1, d_{2} = \frac{x_{23} x_{35}}{x_{14} x_{45}}, d_{3} = \frac{x_{35}}{x_{14} x_{45}}, d_{4} = \frac{1}{x_{14}}, d_{5} = \frac{1}{x_{14} x_{45}}, \\ 
	&a_{12} = \frac{x_{13} x_{35}}{x_{14} x_{45}}, a_{13} = 1, a_{14} = 1, a_{15} = 1, \\ 
	&a_{23} = 1, a_{24} = 1, a_{25} = 1, \\ 
	&a_{34} = 0, a_{35} = \frac{{2} \, x_{14} x_{45} - x_{25}}{x_{14} x_{23} x_{45}}, \\ 
	&a_{45} = \frac{x_{13} x_{25} - {\left({2} \, x_{13} x_{14} - {2} \, x_{14} x_{23}\right)} x_{45}}{x_{14}^{2} x_{23} x_{45}}
\end{align*}

\begin{equation*}x=\left(\begin{matrix}
		1 & 0 & x_{13} & x_{14} & x_{15} \\
		0 & 1 & x_{23} & 0 & x_{25} \\
		0 & 0 & 1 & 0 & x_{35}  \\
		0 & 0 & 0 & 1 & x_{45}  \\
		0 & 0 & 0 & 0 & 1 \\
	\end{matrix}\right),x^A=\left(\begin{matrix}
		1 & 0 & 0 & 1 & 0 \\
		0 & 1 & 1 & 0 & 0 \\
		0 & 0 & 1 & 0 & 1 \\
		0 & 0 & 0 & 1 & 1 \\
		0 & 0 & 0 & 0 & 1
	\end{matrix}\right)
\end{equation*} \newline
Where matrix $A$ has entries
\begin{align*}
	&d_{1} = 1, d_{2} = \frac{x_{23} x_{35}}{x_{14} x_{45}}, d_{3} = \frac{x_{35}}{x_{14} x_{45}}, d_{4} = \frac{1}{x_{14}}, d_{5} = \frac{1}{x_{14} x_{45}}, \\ 
	&a_{12} = \frac{x_{13} x_{35}}{x_{14} x_{45}}, a_{13} = 1, a_{14} = 1, a_{15} = 1, \\ 
	&a_{23} = 1, a_{24} = 1, a_{25} = 1, \\ 
	&a_{34} = 0, a_{35} = \frac{{2} \, x_{14} x_{45} - x_{25}}{x_{14} x_{23} x_{45}}, \\ 
	&a_{45} = -\frac{x_{15} x_{23} - x_{13} x_{25} + {\left({2} \, x_{13} x_{14} - {2} \, x_{14} x_{23}\right)} x_{45}}{x_{14}^{2} x_{23} x_{45}}
\end{align*}


First assume $x_{24}\neq \frac{-x_{35}x_{23}}{x_{45}}$
\begin{equation*}x=\left(
\right)
\end{equation*} \newline
Where matrix $A$ has entries
\begin{align*}
	&d_{1} = 1, d_{2} = 1, d_{3} = \frac{1}{x_{23}}, d_{4} = \frac{x_{45}}{x_{23} x_{35} + x_{24} x_{45}}, d_{5} = \frac{1}{x_{23} x_{35} + x_{24} x_{45}}, \\ 
	&a_{12} = 0, a_{13} = 1, a_{14} = -\frac{x_{23} x_{35} + x_{24} x_{45} - x_{15}}{x_{23} x_{35} + x_{24} x_{45}}, a_{15} = 1, \\ 
	&a_{23} = 1, a_{24} = 1, a_{25} = 1, \\ 
	&a_{34} = -\frac{x_{24} x_{45}}{x_{23}^{2} x_{35} + x_{23} x_{24} x_{45}}, a_{35} = 1, \\ 
	&a_{45} = -\frac{x_{23} - {2} \,}{x_{24}}
\end{align*}

Now assume $x_{24}= \frac{-x_{35}x_{23}}{x_{45}}$
\begin{equation*}x=\left(\begin{matrix}
		1 & 0 & 0 & 0 & x_{15} \\
		0 & 1 & x_{23} & x_{24} & 0 \\
		0 & 0 & 1 & 0 & x_{35}  \\
		0 & 0 & 0 & 1 & x_{45}  \\
		0 & 0 & 0 & 0 & 1 \\
	\end{matrix}\right),x^A=\left(\begin{matrix}
		1 & 0 & 0 & 0 & 0 \\
		0 & 1 & 1 & 0 & 0 \\
		0 & 0 & 1 & 0 & 0 \\
		0 & 0 & 0 & 1 & 1 \\
		0 & 0 & 0 & 0 & 1
	\end{matrix}\right)
\end{equation*} \newline
Where matrix $A$ has entries
\begin{align*}
	&d_{1} = 1, d_{2} = 1, d_{3} = \frac{1}{x_{23}}, d_{4} = 1, d_{5} = \frac{1}{x_{45}}, \\ 
	&a_{12} = 0, a_{13} = 1, a_{14} = \frac{x_{15}}{x_{45}}, a_{15} = 1, \\ 
	&a_{23} = 1, a_{24} = 1, a_{25} = 1, \\ 
	&a_{34} = \frac{x_{35}}{x_{45}}, a_{35} = 1, \\ 
	&a_{45} = \frac{{\left(x_{23} - 1\right)} x_{45}}{x_{23} x_{35}}
\end{align*}


First assume $x_{35}\neq \frac{-x_{45}x_{24}}{x_{23}}$
\begin{equation*}x=\left(\begin{matrix}
		1 & 0 & 0 & x_{14} & 0 \\
		0 & 1 & x_{23} & x_{24} & 0 \\
		0 & 0 & 1 & 0 & x_{35}  \\
		0 & 0 & 0 & 1 & x_{45}  \\
		0 & 0 & 0 & 0 & 1 \\
	\end{matrix}\right),x^A=\left(\begin{matrix}
		1 & 0 & 0 & 1 & 0 \\
		0 & 1 & 1 & 0 & 0 \\
		0 & 0 & 1 & 0 & 1 \\
		0 & 0 & 0 & 1 & 1 \\
		0 & 0 & 0 & 0 & 1
	\end{matrix}\right)
\end{equation*} \newline
Where matrix $A$ has entries
\begin{align*}
	&d_{1} = 1, d_{2} = \frac{x_{23} x_{35} + x_{24} x_{45}}{x_{14} x_{45}}, d_{3} = \frac{x_{23} x_{35} + x_{24} x_{45}}{x_{14} x_{23} x_{45}}, d_{4} = \frac{1}{x_{14}}, d_{5} = \frac{1}{x_{14} x_{45}}, \\ 
	&a_{12} = 0, a_{13} = 1, a_{14} = 1, a_{15} = 1, \\ 
	&a_{23} = 1, a_{24} = 1, a_{25} = 1, \\ 
	&a_{34} = -\frac{x_{24}}{x_{14} x_{23}}, a_{35} = \frac{{2} \, x_{14} - {2} \, x_{24}}{x_{14} x_{23}}, \\ 
	&a_{45} = \frac{2}{x_{14}}
\end{align*}

Now assume $x_{35}= \frac{-x_{45}x_{24}}{x_{23}}$
\begin{equation*}x=\left(\begin{matrix}
		1 & 0 & 0 & x_{14} & 0 \\
		0 & 1 & x_{23} & x_{24} & 0 \\
		0 & 0 & 1 & 0 & x_{35}  \\
		0 & 0 & 0 & 1 & x_{45}  \\
		0 & 0 & 0 & 0 & 1 \\
	\end{matrix}\right),x^A=\left(\begin{matrix}
		1 & 0 & 0 & 1 & 0 \\
		0 & 1 & 1 & 0 & 0 \\
		0 & 0 & 1 & 0 & 0 \\
		0 & 0 & 0 & 1 & 1 \\
		0 & 0 & 0 & 0 & 1
	\end{matrix}\right)
\end{equation*} \newline
Where matrix $A$ has entries
\begin{align*}
	&d_{1} = 1, d_{2} = 1, d_{3} = \frac{1}{x_{23}}, d_{4} = \frac{1}{x_{14}}, d_{5} = \frac{1}{x_{14} x_{45}}, \\ 
	&a_{12} = 0, a_{13} = 1, a_{14} = 1, a_{15} = 1, \\ 
	&a_{23} = 1, a_{24} = 1, a_{25} = 1, \\ 
	&a_{34} = -\frac{x_{24}}{x_{14} x_{23}}, a_{35} = \frac{x_{14} - x_{24}}{x_{14} x_{23}}, \\ 
	&a_{45} = \frac{1}{x_{14}}
\end{align*}


First assume $x_{35}\neq \frac{-x_{45}x_{24}}{x_{23}}$
\begin{equation*}x=\left(\begin{matrix}
		1 & 0 & 0 & x_{14} & x_{15} \\
		0 & 1 & x_{23} & x_{24} & 0 \\
		0 & 0 & 1 & 0 & x_{35}  \\
		0 & 0 & 0 & 1 & x_{45}  \\
		0 & 0 & 0 & 0 & 1 \\
	\end{matrix}\right),x^A=\left(\begin{matrix}
		1 & 0 & 0 & 1 & 0 \\
		0 & 1 & 1 & 0 & 0 \\
		0 & 0 & 1 & 0 & 1 \\
		0 & 0 & 0 & 1 & 1 \\
		0 & 0 & 0 & 0 & 1
	\end{matrix}\right)
\end{equation*} \newline
Where matrix $A$ has entries
\begin{align*}
	&d_{1} = 1, d_{2} = \frac{x_{23} x_{35} + x_{24} x_{45}}{x_{14} x_{45}}, d_{3} = \frac{x_{23} x_{35} + x_{24} x_{45}}{x_{14} x_{23} x_{45}}, d_{4} = \frac{1}{x_{14}}, d_{5} = \frac{1}{x_{14} x_{45}}, \\ 
	&a_{12} = 0, a_{13} = 1, a_{14} = 1, a_{15} = 1, \\ 
	&a_{23} = 1, a_{24} = 1, a_{25} = 1, \\ 
	&a_{34} = -\frac{x_{24}}{x_{14} x_{23}}, a_{35} = \frac{x_{15} x_{24} + {\left({2} \, x_{14}^{2} - {2} \, x_{14} x_{24}\right)} x_{45}}{x_{14}^{2} x_{23} x_{45}}, \\ 
	&a_{45} = \frac{{2} \, x_{14} x_{45} - x_{15}}{x_{14}^{2} x_{45}}
\end{align*}

Now assume $x_{35}= \frac{-x_{45}x_{24}}{x_{23}}$
\begin{equation*}x=\left(\begin{matrix}
		1 & 0 & 0 & x_{14} & x_{15} \\
		0 & 1 & x_{23} & x_{24} & 0 \\
		0 & 0 & 1 & 0 & x_{35}  \\
		0 & 0 & 0 & 1 & x_{45}  \\
		0 & 0 & 0 & 0 & 1 \\
	\end{matrix}\right),x^A=\left(\begin{matrix}
		1 & 0 & 0 & 1 & 0 \\
		0 & 1 & 1 & 0 & 0 \\
		0 & 0 & 1 & 0 & 0 \\
		0 & 0 & 0 & 1 & 1 \\
		0 & 0 & 0 & 0 & 1
	\end{matrix}\right)
\end{equation*} \newline
Where matrix $A$ has entries
\begin{align*}
	&d_{1} = 1, d_{2} = 1, d_{3} = \frac{1}{x_{23}}, d_{4} = \frac{1}{x_{14}}, d_{5} = \frac{1}{x_{14} x_{45}}, \\ 
	&a_{12} = 0, a_{13} = 1, a_{14} = 1, a_{15} = 1, \\ 
	&a_{23} = 1, a_{24} = 1, a_{25} = 1, \\ 
	&a_{34} = -\frac{x_{24}}{x_{14} x_{23}}, a_{35} = \frac{x_{15} x_{24} + {\left(x_{14}^{2} - x_{14} x_{24}\right)} x_{45}}{x_{14}^{2} x_{23} x_{45}}, \\ 
	&a_{45} = \frac{x_{14} x_{45} - x_{15}}{x_{14}^{2} x_{45}}
\end{align*}


First assume $x_{35}\neq \frac{-x_{45}x_{24}}{x_{23}}$
\begin{equation*}x=\left(\begin{matrix}
		1 & 0 & x_{13} & 0 & 0 \\
		0 & 1 & x_{23} & x_{24} & 0 \\
		0 & 0 & 1 & 0 & x_{35}  \\
		0 & 0 & 0 & 1 & x_{45}  \\
		0 & 0 & 0 & 0 & 1 \\
	\end{matrix}\right),x^A=\left(\begin{matrix}
		1 & 0 & 0 & 1 & 0 \\
		0 & 1 & 1 & 0 & 0 \\
		0 & 0 & 1 & 0 & 1 \\
		0 & 0 & 0 & 1 & 1 \\
		0 & 0 & 0 & 0 & 1
	\end{matrix}\right)
\end{equation*} \newline
Where matrix $A$ has entries
\begin{align*}
	&d_{1} = 1, d_{2} = -\frac{x_{23}^{2} x_{35} + x_{23} x_{24} x_{45}}{x_{13} x_{24} x_{45}}, d_{3} = -\frac{x_{23} x_{35} + x_{24} x_{45}}{x_{13} x_{24} x_{45}}, d_{4} = -\frac{x_{23}}{x_{13} x_{24}}, d_{5} = -\frac{x_{23}}{x_{13} x_{24} x_{45}}, \\ 
	&a_{12} = -\frac{x_{23} x_{35} + x_{24} x_{45}}{x_{24} x_{45}}, a_{13} = 1, a_{14} = 1, a_{15} = 1, \\ 
	&a_{23} = 1, a_{24} = 1, a_{25} = 1, \\ 
	&a_{34} = \frac{1}{x_{13}}, a_{35} = \frac{2}{x_{13}}, \\ 
	&a_{45} = \frac{{2} \, x_{13} - {2} \, x_{23}}{x_{13} x_{24}}
\end{align*}

Now assume $x_{35}= \frac{-x_{45}x_{24}}{x_{23}}$
\begin{equation*}x=\left(\begin{matrix}
		1 & 0 & x_{13} & 0 & 0 \\
		0 & 1 & x_{23} & x_{24} & 0 \\
		0 & 0 & 1 & 0 & x_{35}  \\
		0 & 0 & 0 & 1 & x_{45}  \\
		0 & 0 & 0 & 0 & 1 \\
	\end{matrix}\right),x^A=\left(\begin{matrix}
		1 & 0 & 0 & 1 & 0 \\
		0 & 1 & 1 & 0 & 0 \\
		0 & 0 & 1 & 0 & 0 \\
		0 & 0 & 0 & 1 & 1 \\
		0 & 0 & 0 & 0 & 1
	\end{matrix}\right)
\end{equation*} \newline
Where matrix $A$ has entries
\begin{align*}
	&d_{1} = 1, d_{2} = 1, d_{3} = \frac{1}{x_{23}}, d_{4} = -\frac{x_{23}}{x_{13} x_{24}}, d_{5} = -\frac{x_{23}}{x_{13} x_{24} x_{45}}, \\ 
	&a_{12} = \frac{x_{13}}{x_{23}}, a_{13} = 1, a_{14} = 1, a_{15} = 1, \\ 
	&a_{23} = 1, a_{24} = 1, a_{25} = 1, \\ 
	&a_{34} = \frac{1}{x_{13}}, a_{35} = \frac{1}{x_{13}}, \\ 
	&a_{45} = \frac{x_{13} - x_{23}}{x_{13} x_{24}}
\end{align*}

First assume $x_{35}\neq \frac{-x_{45}x_{24}}{x_{23}}$
\begin{equation*}x=\left(\begin{matrix}
		1 & 0 & x_{13} & 0 & x_{15} \\
		0 & 1 & x_{23} & x_{24} & 0 \\
		0 & 0 & 1 & 0 & x_{35}  \\
		0 & 0 & 0 & 1 & x_{45}  \\
		0 & 0 & 0 & 0 & 1 \\
	\end{matrix}\right),x^A=\left(\begin{matrix}
		1 & 0 & 0 & 1 & 0 \\
		0 & 1 & 1 & 0 & 0 \\
		0 & 0 & 1 & 0 & 1 \\
		0 & 0 & 0 & 1 & 1 \\
		0 & 0 & 0 & 0 & 1
	\end{matrix}\right)
\end{equation*} \newline
Where matrix $A$ has entries
\begin{align*}
	&d_{1} = 1, d_{2} = -\frac{x_{23}^{2} x_{35} + x_{23} x_{24} x_{45}}{x_{13} x_{24} x_{45}}, d_{3} = -\frac{x_{23} x_{35} + x_{24} x_{45}}{x_{13} x_{24} x_{45}}, d_{4} = -\frac{x_{23}}{x_{13} x_{24}}, d_{5} = -\frac{x_{23}}{x_{13} x_{24} x_{45}}, \\ 
	&a_{12} = -\frac{x_{23} x_{35} + x_{24} x_{45}}{x_{24} x_{45}}, a_{13} = 1, a_{14} = 1, a_{15} = 1, \\ 
	&a_{23} = 1, a_{24} = 1, a_{25} = 1, \\ 
	&a_{34} = \frac{1}{x_{13}}, a_{35} = \frac{{2} \, x_{13} x_{24} x_{45} + x_{15} x_{23}}{x_{13}^{2} x_{24} x_{45}}, \\ 
	&a_{45} = -\frac{x_{15} x_{23}^{2} - {\left({2} \, x_{13}^{2} - {2} \, x_{13} x_{23}\right)} x_{24} x_{45}}{x_{13}^{2} x_{24}^{2} x_{45}}
\end{align*}

Now assume $x_{35}= \frac{-x_{45}x_{24}}{x_{23}}$
\begin{equation*}x=\left(\begin{matrix}
		1 & 0 & x_{13} & 0 & x_{15} \\
		0 & 1 & x_{23} & x_{24} & 0 \\
		0 & 0 & 1 & 0 & x_{35}  \\
		0 & 0 & 0 & 1 & x_{45}  \\
		0 & 0 & 0 & 0 & 1 \\
	\end{matrix}\right),x^A=\left(\begin{matrix}
		1 & 0 & 0 & 1 & 0 \\
		0 & 1 & 1 & 0 & 0 \\
		0 & 0 & 1 & 0 & 0 \\
		0 & 0 & 0 & 1 & 1 \\
		0 & 0 & 0 & 0 & 1
	\end{matrix}\right)
\end{equation*} \newline
Where matrix $A$ has entries
\begin{align*}
	&d_{1} = 1, d_{2} = 1, d_{3} = \frac{1}{x_{23}}, d_{4} = -\frac{x_{23}}{x_{13} x_{24}}, d_{5} = -\frac{x_{23}}{x_{13} x_{24} x_{45}}, \\ 
	&a_{12} = \frac{x_{13}}{x_{23}}, a_{13} = 1, a_{14} = 1, a_{15} = 1, \\ 
	&a_{23} = 1, a_{24} = 1, a_{25} = 1, \\ 
	&a_{34} = \frac{1}{x_{13}}, a_{35} = \frac{x_{13} x_{24} x_{45} + x_{15} x_{23}}{x_{13}^{2} x_{24} x_{45}}, \\ 
	&a_{45} = -\frac{x_{15} x_{23}^{2} - {\left(x_{13}^{2} - x_{13} x_{23}\right)} x_{24} x_{45}}{x_{13}^{2} x_{24}^{2} x_{45}}
\end{align*}


First assume $x_{14} \neq \frac{x_{13}x_{24}}{x_{23}}$ and $x_{35}\neq \frac{-x_{45}x_{24}}{x_{23}}$
\begin{equation*}x=\left(\begin{matrix}
		1 & 0 & x_{13} & x_{14} & 0 \\
		0 & 1 & x_{23} & x_{24} & 0 \\
		0 & 0 & 1 & 0 & x_{35}  \\
		0 & 0 & 0 & 1 & x_{45}  \\
		0 & 0 & 0 & 0 & 1 \\
	\end{matrix}\right),x^A=\left(\begin{matrix}
		1 & 0 & 0 & 1 & 0 \\
		0 & 1 & 1 & 0 & 0 \\
		0 & 0 & 1 & 0 & 1 \\
		0 & 0 & 0 & 1 & 1 \\
		0 & 0 & 0 & 0 & 1
	\end{matrix}\right)
\end{equation*} \newline
Where matrix $A$ has entries
\begin{align*}
	&d_{1} = 1, d_{2} = \frac{x_{23}^{2} x_{35} + x_{23} x_{24} x_{45}}{{\left(x_{14} x_{23} - x_{13} x_{24}\right)} x_{45}}, d_{3} = \frac{x_{23} x_{35} + x_{24} x_{45}}{{\left(x_{14} x_{23} - x_{13} x_{24}\right)} x_{45}}, d_{4} = \frac{x_{23}}{x_{14} x_{23} - x_{13} x_{24}}, d_{5} = \frac{x_{23}}{{\left(x_{14} x_{23} - x_{13} x_{24}\right)} x_{45}}, \\ 
	&a_{12} = \frac{x_{13} x_{23} x_{35} + x_{13} x_{24} x_{45}}{{\left(x_{14} x_{23} - x_{13} x_{24}\right)} x_{45}}, a_{13} = 1, a_{14} = 1, a_{15} = 1, \\ 
	&a_{23} = 1, a_{24} = 1, a_{25} = 1, \\ 
	&a_{34} = -\frac{x_{24}}{x_{14} x_{23} - x_{13} x_{24}}, a_{35} = \frac{{2} \, x_{14} - {2} \, x_{24}}{x_{14} x_{23} - x_{13} x_{24}}, \\ 
	&a_{45} = -\frac{{2} \, x_{13} - {2} \, x_{23}}{x_{14} x_{23} - x_{13} x_{24}}
\end{align*}

Now assume $x_{14} = \frac{x_{13}x_{24}}{x_{23}}$ and $x_{24} \neq \frac{-x_{23}x_{35}}{x_{45}}$
\begin{equation*}x=\left(\begin{matrix}
		1 & 0 & x_{13} & x_{14} & 0 \\
		0 & 1 & x_{23} & x_{24} & 0 \\
		0 & 0 & 1 & 0 & x_{35}  \\
		0 & 0 & 0 & 1 & x_{45}  \\
		0 & 0 & 0 & 0 & 1 \\
	\end{matrix}\right),x^A=\left(\begin{matrix}
		1 & 0 & 0 & 0 & 0 \\
		0 & 1 & 1 & 0 & 0 \\
		0 & 0 & 1 & 0 & 1 \\
		0 & 0 & 0 & 1 & 1 \\
		0 & 0 & 0 & 0 & 1
	\end{matrix}\right)
\end{equation*} \newline
Where matrix $A$ has entries
\begin{align*}
	&d_{1} = 1, d_{2} = 1, d_{3} = \frac{1}{x_{23}}, d_{4} = \frac{x_{45}}{x_{23} x_{35} + x_{24} x_{45}}, d_{5} = \frac{1}{x_{23} x_{35} + x_{24} x_{45}}, \\ 
	&a_{12} = \frac{x_{13}}{x_{23}}, a_{13} = 1, a_{14} = 1, a_{15} = 1, \\ 
	&a_{23} = 1, a_{24} = -\frac{x_{13} - {2} \, x_{23}}{x_{13}}, a_{25} = 1, \\ 
	&a_{34} = -\frac{x_{24} x_{45}}{x_{23}^{2} x_{35} + x_{23} x_{24} x_{45}}, a_{35} = 1, \\ 
	&a_{45} = -\frac{{\left(x_{13} - {2} \,\right)} x_{23}}{x_{13} x_{24}}
\end{align*}

Now assume $x_{14} = \frac{x_{13}x_{24}}{x_{23}}$ and $x_{24} = \frac{-x_{23}x_{35}}{x_{45}}$
\begin{equation*}x=\left(\begin{matrix}
		1 & 0 & x_{13} & x_{14} & 0 \\
		0 & 1 & x_{23} & x_{24} & 0 \\
		0 & 0 & 1 & 0 & x_{35}  \\
		0 & 0 & 0 & 1 & x_{45}  \\
		0 & 0 & 0 & 0 & 1 \\
	\end{matrix}\right),x^A=\left(\begin{matrix}
		1 & 0 & 0 & 0 & 0 \\
		0 & 1 & 1 & 0 & 0 \\
		0 & 0 & 1 & 0 & 0 \\
		0 & 0 & 0 & 1 & 1 \\
		0 & 0 & 0 & 0 & 1
	\end{matrix}\right)
\end{equation*} \newline
Where matrix $A$ has entries
\begin{align*}
	&d_{1} = 1, d_{2} = 1, d_{3} = \frac{1}{x_{23}}, d_{4} = 1, d_{5} = \frac{1}{x_{45}}, \\ 
	&a_{12} = \frac{x_{13}}{x_{23}}, a_{13} = 1, a_{14} = 1, a_{15} = 1, \\ 
	&a_{23} = 1, a_{24} = \frac{x_{23}}{x_{13}}, a_{25} = 1, \\ 
	&a_{34} = \frac{x_{35}}{x_{45}}, a_{35} = 1, \\ 
	&a_{45} = \frac{{\left(x_{13} - 1\right)} x_{45}}{x_{13} x_{35}}
\end{align*}

Now assume $x_{14} \neq \frac{x_{13}x_{24}}{x_{23}}$, $x_{24}= \frac{-x_{23}x_{35}}{x_{45}}$ and $x_{13}\neq \frac{-x_{14}x_{45}}{x_{35}}$ 
\begin{equation*}x=\left(\begin{matrix}
		1 & 0 & x_{13} & x_{14} & 0 \\
		0 & 1 & x_{23} & x_{24} & 0 \\
		0 & 0 & 1 & 0 & x_{35}  \\
		0 & 0 & 0 & 1 & x_{45}  \\
		0 & 0 & 0 & 0 & 1 \\
	\end{matrix}\right),x^A=\left(\begin{matrix}
		1 & 0 & 0 & 1 & 0 \\
		0 & 1 & 1 & 0 & 0 \\
		0 & 0 & 1 & 0 & 0 \\
		0 & 0 & 0 & 1 & 1 \\
		0 & 0 & 0 & 0 & 1
	\end{matrix}\right)
\end{equation*} \newline
Where matrix $A$ has entries
\begin{align*}
	&d_{1} = 1, d_{2} = 1, d_{3} = \frac{1}{x_{23}}, d_{4} = \frac{x_{45}}{x_{13} x_{35} + x_{14} x_{45}}, d_{5} = \frac{1}{x_{13} x_{35} + x_{14} x_{45}}, \\ 
	&a_{12} = \frac{x_{13}}{x_{23}}, a_{13} = 1, a_{14} = 1, a_{15} = 1, \\ 
	&a_{23} = 1, a_{24} = 1, a_{25} = 1, \\ 
	&a_{34} = \frac{x_{35}}{x_{13} x_{35} + x_{14} x_{45}}, a_{35} = \frac{x_{23} x_{35} + x_{14} x_{45}}{x_{13} x_{23} x_{35} + x_{14} x_{23} x_{45}}, \\ 
	&a_{45} = -\frac{{\left(x_{13} - x_{23}\right)} x_{45}}{x_{13} x_{23} x_{35} + x_{14} x_{23} x_{45}}
\end{align*}

Now assume $x_{14} \neq \frac{x_{13}x_{24}}{x_{23}}$, $x_{24}= \frac{-x_{23}x_{35}}{x_{45}}$ and $x_{13}= \frac{-x_{14}x_{45}}{x_{35}}$ 
\begin{equation*}x=\left(\begin{matrix}
		1 & 0 & x_{13} & x_{14} & 0 \\
		0 & 1 & x_{23} & x_{24} & 0 \\
		0 & 0 & 1 & 0 & x_{35}  \\
		0 & 0 & 0 & 1 & x_{45}  \\
		0 & 0 & 0 & 0 & 1 \\
	\end{matrix}\right),x^A=\left(\begin{matrix}
		1 & 0 & 0 & 0 & 0 \\
		0 & 1 & 1 & 0 & 0 \\
		0 & 0 & 1 & 0 & 0 \\
		0 & 0 & 0 & 1 & 1 \\
		0 & 0 & 0 & 0 & 1
	\end{matrix}\right)
\end{equation*} \newline
Where matrix $A$ has entries
\begin{align*}
	&d_{1} = 1, d_{2} = 1, d_{3} = \frac{1}{x_{23}}, d_{4} = 1, d_{5} = \frac{1}{x_{45}}, \\ 
	&a_{12} = -\frac{x_{14} x_{45}}{x_{23} x_{35}}, a_{13} = 1, a_{14} = 1, a_{15} = 1, \\ 
	&a_{23} = 1, a_{24} = -\frac{x_{23} x_{35}}{x_{14} x_{45}}, a_{25} = 1, \\ 
	&a_{34} = \frac{x_{35}}{x_{45}}, a_{35} = 1, \\ 
	&a_{45} = \frac{x_{14} x_{45} + x_{35}}{x_{14} x_{35}}
\end{align*}

First assume $x_{14} \neq \frac{x_{13}x_{24}}{x_{23}}$ and $x_{35}\neq \frac{-x_{45}x_{24}}{x_{23}}$
\begin{equation*}x=\left(\begin{matrix}
		1 & 0 & x_{13} & x_{14} & x_{15} \\
		0 & 1 & x_{23} & x_{24} & 0 \\
		0 & 0 & 1 & 0 & x_{35}  \\
		0 & 0 & 0 & 1 & x_{45}  \\
		0 & 0 & 0 & 0 & 1 \\
	\end{matrix}\right),x^A=\left(\begin{matrix}
		1 & 0 & 0 & 1 & 0 \\
		0 & 1 & 1 & 0 & 0 \\
		0 & 0 & 1 & 0 & 1 \\
		0 & 0 & 0 & 1 & 1 \\
		0 & 0 & 0 & 0 & 1
	\end{matrix}\right)
\end{equation*} \newline
Where matrix $A$ has entries
\begin{align*}
	&d_{1} = 1, d_{2} = \frac{x_{23}^{2} x_{35} + x_{23} x_{24} x_{45}}{{\left(x_{14} x_{23} - x_{13} x_{24}\right)} x_{45}}, d_{3} = \frac{x_{23} x_{35} + x_{24} x_{45}}{{\left(x_{14} x_{23} - x_{13} x_{24}\right)} x_{45}}, d_{4} = \frac{x_{23}}{x_{14} x_{23} - x_{13} x_{24}}, d_{5} = \frac{x_{23}}{{\left(x_{14} x_{23} - x_{13} x_{24}\right)} x_{45}}, \\ 
	&a_{12} = \frac{x_{13} x_{23} x_{35} + x_{13} x_{24} x_{45}}{{\left(x_{14} x_{23} - x_{13} x_{24}\right)} x_{45}}, a_{13} = 1, a_{14} = 1, a_{15} = 1, \\ 
	&a_{23} = 1, a_{24} = 1, a_{25} = 1, \\ 
	&a_{34} = -\frac{x_{24}}{x_{14} x_{23} - x_{13} x_{24}}, a_{35} = \frac{x_{15} x_{23} x_{24} + {\left({2} \, x_{14}^{2} x_{23} + {2} \, x_{13} x_{24}^{2} - {\left({2} \, x_{13} x_{14} + {2} \, x_{14} x_{23}\right)} x_{24}\right)} x_{45}}{{\left(x_{14}^{2} x_{23}^{2} - 2 \, x_{13} x_{14} x_{23} x_{24} + x_{13}^{2} x_{24}^{2}\right)} x_{45}}, \\ 
	&a_{45} = -\frac{x_{15} x_{23}^{2} + {\left({2} \, x_{13} x_{14} x_{23} - {2} \, x_{14} x_{23}^{2} - {\left({2} \, x_{13}^{2} - {2} \, x_{13} x_{23}\right)} x_{24}\right)} x_{45}}{{\left(x_{14}^{2} x_{23}^{2} - 2 \, x_{13} x_{14} x_{23} x_{24} + x_{13}^{2} x_{24}^{2}\right)} x_{45}}
\end{align*}

Now assume $x_{14} = \frac{x_{13}x_{24}}{x_{23}}$ and $x_{35} \neq \frac{-x_{45}x_{24}}{x_{23}}$
\begin{equation*}x=\left(\begin{matrix}
		1 & 0 & x_{13} & x_{14} & x_{15} \\
		0 & 1 & x_{23} & x_{24} & 0 \\
		0 & 0 & 1 & 0 & x_{35}  \\
		0 & 0 & 0 & 1 & x_{45}  \\
		0 & 0 & 0 & 0 & 1 \\
	\end{matrix}\right),x^A=\left(\begin{matrix}
		1 & 0 & 0 & 0 & 0 \\
		0 & 1 & 1 & 0 & 0 \\
		0 & 0 & 1 & 0 & 1 \\
		0 & 0 & 0 & 1 & 1 \\
		0 & 0 & 0 & 0 & 1
	\end{matrix}\right)
\end{equation*} \newline
Where matrix $A$ has entries
\begin{align*}
	&d_{1} = 1, d_{2} = 1, d_{3} = \frac{1}{x_{23}}, d_{4} = \frac{x_{45}}{x_{23} x_{35} + x_{24} x_{45}}, d_{5} = \frac{1}{x_{23} x_{35} + x_{24} x_{45}}, \\ 
	&a_{12} = \frac{x_{13}}{x_{23}}, a_{13} = 1, a_{14} = 1, a_{15} = 1, \\ 
	&a_{23} = 1, a_{24} = -\frac{x_{15} x_{23} + {\left(x_{13} - {2} \, x_{23}\right)} x_{24} x_{45} + {\left(x_{13} x_{23} - {2} \, x_{23}^{2}\right)} x_{35}}{x_{13} x_{23} x_{35} + x_{13} x_{24} x_{45}}, a_{25} = 1, \\ 
	&a_{34} = -\frac{x_{24} x_{45}}{x_{23}^{2} x_{35} + x_{23} x_{24} x_{45}}, a_{35} = 1, \\ 
	&a_{45} = -\frac{{\left(x_{13} - {2} \,\right)} x_{23}^{2} x_{35} + {\left(x_{13} - {2} \,\right)} x_{23} x_{24} x_{45} + x_{15} x_{23}}{x_{13} x_{23} x_{24} x_{35} + x_{13} x_{24}^{2} x_{45}}
\end{align*}

Now assume $x_{14} \neq \frac{x_{13}x_{24}}{x_{23}}$ and $x_{35} = \frac{-x_{45}x_{24}}{x_{23}}$
\begin{equation*}x=\left(\begin{matrix}
		1 & 0 & x_{13} & x_{14} & x_{15} \\
		0 & 1 & x_{23} & x_{24} & 0 \\
		0 & 0 & 1 & 0 & x_{35}  \\
		0 & 0 & 0 & 1 & x_{45}  \\
		0 & 0 & 0 & 0 & 1 \\
	\end{matrix}\right),x^A=\left(\begin{matrix}
		1 & 0 & 0 & 1 & 0 \\
		0 & 1 & 1 & 0 & 0 \\
		0 & 0 & 1 & 0 & 0 \\
		0 & 0 & 0 & 1 & 1 \\
		0 & 0 & 0 & 0 & 1
	\end{matrix}\right)
\end{equation*} \newline
Where matrix $A$ has entries
\begin{align*}
	&d_{1} = 1, d_{2} = 1, d_{3} = \frac{1}{x_{23}}, d_{4} = \frac{x_{23}}{x_{14} x_{23} - x_{13} x_{24}}, d_{5} = \frac{x_{23}}{{\left(x_{14} x_{23} - x_{13} x_{24}\right)} x_{45}}, \\ 
	&a_{12} = \frac{x_{13}}{x_{23}}, a_{13} = 1, a_{14} = 1, a_{15} = 1, \\ 
	&a_{23} = 1, a_{24} = 1, a_{25} = 1, \\ 
	&a_{34} = -\frac{x_{24}}{x_{14} x_{23} - x_{13} x_{24}}, a_{35} = \frac{x_{15} x_{23} x_{24} + {\left(x_{14}^{2} x_{23} + x_{13} x_{24}^{2} - {\left(x_{13} x_{14} + x_{14} x_{23}\right)} x_{24}\right)} x_{45}}{{\left(x_{14}^{2} x_{23}^{2} - 2 \, x_{13} x_{14} x_{23} x_{24} + x_{13}^{2} x_{24}^{2}\right)} x_{45}}, \\ 
	&a_{45} = -\frac{x_{15} x_{23}^{2} + {\left(x_{13} x_{14} x_{23} - x_{14} x_{23}^{2} - {\left(x_{13}^{2} - x_{13} x_{23}\right)} x_{24}\right)} x_{45}}{{\left(x_{14}^{2} x_{23}^{2} - 2 \, x_{13} x_{14} x_{23} x_{24} + x_{13}^{2} x_{24}^{2}\right)} x_{45}}
\end{align*}

Now assume $x_{14} = \frac{x_{13}x_{24}}{x_{23}}$ and $x_{24}\neq \frac{-x_{23}x_{35}}{x_{45}}$
\begin{equation*}x=\left(\begin{matrix}
		1 & 0 & x_{13} & x_{14} & x_{15} \\
		0 & 1 & x_{23} & x_{24} & 0 \\
		0 & 0 & 1 & 0 & x_{35}  \\
		0 & 0 & 0 & 1 & x_{45}  \\
		0 & 0 & 0 & 0 & 1 \\
	\end{matrix}\right),x^A=\left(\begin{matrix}
		1 & 0 & 0 & 0 & 0 \\
		0 & 1 & 1 & 0 & 0 \\
		0 & 0 & 1 & 0 & 1 \\
		0 & 0 & 0 & 1 & 1 \\
		0 & 0 & 0 & 0 & 1
	\end{matrix}\right)
\end{equation*} \newline
Where matrix $A$ has entries
\begin{align*}
	&d_{1} = 1, d_{2} = 1, d_{3} = \frac{1}{x_{23}}, d_{4} = \frac{x_{45}}{x_{23} x_{35} + x_{24} x_{45}}, d_{5} = \frac{1}{x_{23} x_{35} + x_{24} x_{45}}, \\ 
	&a_{12} = \frac{x_{13}}{x_{23}}, a_{13} = 1, a_{14} = 1, a_{15} = 1, \\ 
	&a_{23} = 1, a_{24} = -\frac{x_{15} x_{23} + {\left(x_{13} - {2} \, x_{23}\right)} x_{24} x_{45} + {\left(x_{13} x_{23} - {2} \, x_{23}^{2}\right)} x_{35}}{x_{13} x_{23} x_{35} + x_{13} x_{24} x_{45}}, a_{25} = 1, \\ 
	&a_{34} = -\frac{x_{24} x_{45}}{x_{23}^{2} x_{35} + x_{23} x_{24} x_{45}}, a_{35} = 1, \\ 
	&a_{45} = -\frac{{\left(x_{13} - {2} \,\right)} x_{23}^{2} x_{35} + {\left(x_{13} - {2} \,\right)} x_{23} x_{24} x_{45} + x_{15} x_{23}}{x_{13} x_{23} x_{24} x_{35} + x_{13} x_{24}^{2} x_{45}}
\end{align*}

Now assume $x_{14} = \frac{x_{13}x_{24}}{x_{23}}$ and $x_{24}= \frac{-x_{23}x_{35}}{x_{45}}$
\begin{equation*}x=\left(\begin{matrix}
		1 & 0 & x_{13} & x_{14} & x_{15} \\
		0 & 1 & x_{23} & x_{24} & 0 \\
		0 & 0 & 1 & 0 & x_{35}  \\
		0 & 0 & 0 & 1 & x_{45}  \\
		0 & 0 & 0 & 0 & 1 \\
	\end{matrix}\right),x^A=\left(\begin{matrix}
		1 & 0 & 0 & 0 & 0 \\
		0 & 1 & 1 & 0 & 0 \\
		0 & 0 & 1 & 0 & 0 \\
		0 & 0 & 0 & 1 & 1 \\
		0 & 0 & 0 & 0 & 1
	\end{matrix}\right)
\end{equation*} \newline
Where matrix $A$ has entries
\begin{align*}
	&d_{1} = 1, d_{2} = 1, d_{3} = \frac{1}{x_{23}}, d_{4} = 1, d_{5} = \frac{1}{x_{45}}, \\ 
	&a_{12} = \frac{x_{13}}{x_{23}}, a_{13} = 1, a_{14} = 1, a_{15} = 1, \\ 
	&a_{23} = 1, a_{24} = -\frac{x_{15} x_{23} - x_{23} x_{45}}{x_{13} x_{45}}, a_{25} = 1, \\ 
	&a_{34} = \frac{x_{35}}{x_{45}}, a_{35} = 1, \\ 
	&a_{45} = \frac{x_{15} + {\left(x_{13} - 1\right)} x_{45}}{x_{13} x_{35}}
\end{align*}

First assume $x_{24} \neq \frac{-x_{23}x_{35}}{x_{45}}$
\begin{equation*}x=\left(\begin{matrix}
		1 & 0 & 0 & 0 & 0 \\
		0 & 1 & x_{23} & x_{24} & x_{25} \\
		0 & 0 & 1 & 0 & x_{35}  \\
		0 & 0 & 0 & 1 & x_{45}  \\
		0 & 0 & 0 & 0 & 1 \\
	\end{matrix}\right),x^A=\left(\begin{matrix}
		1 & 0 & 0 & 0 & 0 \\
		0 & 1 & 1 & 0 & 0 \\
		0 & 0 & 1 & 0 & 1 \\
		0 & 0 & 0 & 1 & 1 \\
		0 & 0 & 0 & 0 & 1
	\end{matrix}\right)
\end{equation*} \newline
Where matrix $A$ has entries
\begin{align*}
	&d_{1} = 1, d_{2} = 1, d_{3} = \frac{1}{x_{23}}, d_{4} = \frac{x_{45}}{x_{23} x_{35} + x_{24} x_{45}}, d_{5} = \frac{1}{x_{23} x_{35} + x_{24} x_{45}}, \\ 
	&a_{12} = 0, a_{13} = 1, a_{14} = -1, a_{15} = 1, \\ 
	&a_{23} = 1, a_{24} = 1, a_{25} = 1, \\ 
	&a_{34} = -\frac{x_{24} x_{45}}{x_{23}^{2} x_{35} + x_{23} x_{24} x_{45}}, a_{35} = 1, \\ 
	&a_{45} = -\frac{{\left(x_{23} - {2} \,\right)} x_{24} x_{45} + x_{25} + {\left(x_{23}^{2} - {2} \, x_{23}\right)} x_{35}}{x_{23} x_{24} x_{35} + x_{24}^{2} x_{45}}
\end{align*}

Now assume $x_{24} = \frac{-x_{23}x_{35}}{x_{45}}$
\begin{equation*}x=\left(\begin{matrix}
		1 & 0 & 0 & 0 & 0 \\
		0 & 1 & x_{23} & x_{24} & x_{25} \\
		0 & 0 & 1 & 0 & x_{35}  \\
		0 & 0 & 0 & 1 & x_{45}  \\
		0 & 0 & 0 & 0 & 1 \\
	\end{matrix}\right),x^A=\left(\begin{matrix}
		1 & 0 & 0 & 0 & 0 \\
		0 & 1 & 1 & 0 & 0 \\
		0 & 0 & 1 & 0 & 0 \\
		0 & 0 & 0 & 1 & 1 \\
		0 & 0 & 0 & 0 & 1
	\end{matrix}\right)
\end{equation*} \newline
Where matrix $A$ has entries
\begin{align*}
	&d_{1} = 1, d_{2} = 1, d_{3} = \frac{1}{x_{23}}, d_{4} = 1, d_{5} = \frac{1}{x_{45}}, \\ 
	&a_{12} = 0, a_{13} = 1, a_{14} = 0, a_{15} = 1, \\ 
	&a_{23} = 1, a_{24} = 1, a_{25} = 1, \\ 
	&a_{34} = \frac{x_{35}}{x_{45}}, a_{35} = 1, \\ 
	&a_{45} = \frac{x_{25} + {\left(x_{23} - 1\right)} x_{45}}{x_{23} x_{35}}
\end{align*}


First assume $x_{24} \neq \frac{-x_{23}x_{35}}{x_{45}}$
\begin{equation*}x=\left(\begin{matrix}
		1 & 0 & 0 & 0 & x_{15} \\
		0 & 1 & x_{23} & x_{24} & x_{25} \\
		0 & 0 & 1 & 0 & x_{35}  \\
		0 & 0 & 0 & 1 & x_{45}  \\
		0 & 0 & 0 & 0 & 1 \\
	\end{matrix}\right),x^A=\left(\begin{matrix}
		1 & 0 & 0 & 0 & 0 \\
		0 & 1 & 1 & 0 & 0 \\
		0 & 0 & 1 & 0 & 1 \\
		0 & 0 & 0 & 1 & 1 \\
		0 & 0 & 0 & 0 & 1
	\end{matrix}\right)
\end{equation*} \newline
Where matrix $A$ has entries
\begin{align*}
	&d_{1} = 1, d_{2} = 1, d_{3} = \frac{1}{x_{23}}, d_{4} = \frac{x_{45}}{x_{23} x_{35} + x_{24} x_{45}}, d_{5} = \frac{1}{x_{23} x_{35} + x_{24} x_{45}}, \\ 
	&a_{12} = 0, a_{13} = 1, a_{14} = -\frac{x_{23} x_{35} + x_{24} x_{45} - x_{15}}{x_{23} x_{35} + x_{24} x_{45}}, a_{15} = 1, \\ 
	&a_{23} = 1, a_{24} = 1, a_{25} = 1, \\ 
	&a_{34} = -\frac{x_{24} x_{45}}{x_{23}^{2} x_{35} + x_{23} x_{24} x_{45}}, a_{35} = 1, \\ 
	&a_{45} = -\frac{{\left(x_{23} - {2} \,\right)} x_{24} x_{45} + x_{25} + {\left(x_{23}^{2} - {2} \, x_{23}\right)} x_{35}}{x_{23} x_{24} x_{35} + x_{24}^{2} x_{45}}
\end{align*}

Now assume $x_{24} = \frac{-x_{23}x_{35}}{x_{45}}$
\begin{equation*}x=\left(\begin{matrix}
		1 & 0 & 0 & 0 & x_{15} \\
		0 & 1 & x_{23} & x_{24} & x_{25} \\
		0 & 0 & 1 & 0 & x_{35}  \\
		0 & 0 & 0 & 1 & x_{45}  \\
		0 & 0 & 0 & 0 & 1 \\
	\end{matrix}\right),x^A=\left(\begin{matrix}
		1 & 0 & 0 & 0 & 0 \\
		0 & 1 & 1 & 0 & 0 \\
		0 & 0 & 1 & 0 & 0 \\
		0 & 0 & 0 & 1 & 1 \\
		0 & 0 & 0 & 0 & 1
	\end{matrix}\right)
\end{equation*} \newline
Where matrix $A$ has entries
\begin{align*}
	&d_{1} = 1, d_{2} = 1, d_{3} = \frac{1}{x_{23}}, d_{4} = 1, d_{5} = \frac{1}{x_{45}}, \\ 
	&a_{12} = 0, a_{13} = 1, a_{14} = \frac{x_{15}}{x_{45}}, a_{15} = 1, \\ 
	&a_{23} = 1, a_{24} = 1, a_{25} = 1, \\ 
	&a_{34} = \frac{x_{35}}{x_{45}}, a_{35} = 1, \\ 
	&a_{45} = \frac{x_{25} + {\left(x_{23} - 1\right)} x_{45}}{x_{23} x_{35}}
\end{align*}

First assume $x_{45} \neq \frac{-x_{23}x_{35}}{x_{24}}$
\begin{equation*}x=\left(\begin{matrix}
		1 & 0 & 0 & x_{14} & 0 \\
		0 & 1 & x_{23} & x_{24} & x_{25} \\
		0 & 0 & 1 & 0 & x_{35}  \\
		0 & 0 & 0 & 1 & x_{45}  \\
		0 & 0 & 0 & 0 & 1 \\
	\end{matrix}\right),x^A=\left(\begin{matrix}
		1 & 0 & 0 & 1 & 0 \\
		0 & 1 & 1 & 0 & 0 \\
		0 & 0 & 1 & 0 & 1 \\
		0 & 0 & 0 & 1 & 1 \\
		0 & 0 & 0 & 0 & 1
	\end{matrix}\right)
\end{equation*} \newline
Where matrix $A$ has entries
\begin{align*}
	&d_{1} = 1, d_{2} = \frac{x_{23} x_{35} + x_{24} x_{45}}{x_{14} x_{45}}, d_{3} = \frac{x_{23} x_{35} + x_{24} x_{45}}{x_{14} x_{23} x_{45}}, d_{4} = \frac{1}{x_{14}}, d_{5} = \frac{1}{x_{14} x_{45}}, \\ 
	&a_{12} = 0, a_{13} = 1, a_{14} = 1, a_{15} = 1, \\ 
	&a_{23} = 1, a_{24} = 1, a_{25} = 1, \\ 
	&a_{34} = -\frac{x_{24}}{x_{14} x_{23}}, a_{35} = -\frac{x_{25} - {\left({2} \, x_{14} - {2} \, x_{24}\right)} x_{45}}{x_{14} x_{23} x_{45}}, \\ 
	&a_{45} = \frac{2}{x_{14}}
\end{align*}

Now assume $x_{45} = \frac{-x_{23}x_{35}}{x_{24}}$
\begin{equation*}x=\left(\begin{matrix}
		1 & 0 & 0 & x_{14} & 0 \\
		0 & 1 & x_{23} & x_{24} & x_{25} \\
		0 & 0 & 1 & 0 & x_{35}  \\
		0 & 0 & 0 & 1 & x_{45}  \\
		0 & 0 & 0 & 0 & 1 \\
	\end{matrix}\right),x^A=\left(\begin{matrix}
		1 & 0 & 0 & 1 & 0 \\
		0 & 1 & 1 & 0 & 0 \\
		0 & 0 & 1 & 0 & 1 \\
		0 & 0 & 0 & 1 & 1 \\
		0 & 0 & 0 & 0 & 1
	\end{matrix}\right)
\end{equation*} \newline
Where matrix $A$ has entries
\begin{align*}
	&d_{1} = 1, d_{2} = 1, d_{3} = \frac{1}{x_{23}}, d_{4} = \frac{1}{x_{14}}, d_{5} = -\frac{x_{24}}{x_{14} x_{23} x_{35}}, \\ 
	&a_{12} = 0, a_{13} = 1, a_{14} = 1, a_{15} = 1, \\ 
	&a_{23} = 1, a_{24} = 1, a_{25} = 1, \\ 
	&a_{34} = -\frac{x_{24}}{x_{14} x_{23}}, a_{35} = \frac{x_{24} x_{25} + {\left(x_{14} x_{23} - x_{23} x_{24}\right)} x_{35}}{x_{14} x_{23}^{2} x_{35}}, \\ 
	&a_{45} = \frac{1}{x_{14}}
\end{align*}


First assume $x_{45} \neq \frac{-x_{23}x_{35}}{x_{24}}$
\begin{equation*}x=\left(\begin{matrix}
		1 & 0 & 0 & x_{14} & x_{15} \\
		0 & 1 & x_{23} & x_{24} & x_{25} \\
		0 & 0 & 1 & 0 & x_{35}  \\
		0 & 0 & 0 & 1 & x_{45}  \\
		0 & 0 & 0 & 0 & 1 \\
	\end{matrix}\right),x^A=\left(\begin{matrix}
		1 & 0 & 0 & 1 & 0 \\
		0 & 1 & 1 & 0 & 0 \\
		0 & 0 & 1 & 0 & 1 \\
		0 & 0 & 0 & 1 & 1 \\
		0 & 0 & 0 & 0 & 1
	\end{matrix}\right)
\end{equation*} \newline
Where matrix $A$ has entries
\begin{align*}
	&d_{1} = 1, d_{2} = \frac{x_{23} x_{35} + x_{24} x_{45}}{x_{14} x_{45}}, d_{3} = \frac{x_{23} x_{35} + x_{24} x_{45}}{x_{14} x_{23} x_{45}}, d_{4} = \frac{1}{x_{14}}, d_{5} = \frac{1}{x_{14} x_{45}}, \\ 
	&a_{12} = 0, a_{13} = 1, a_{14} = 1, a_{15} = 1, \\ 
	&a_{23} = 1, a_{24} = 1, a_{25} = 1, \\ 
	&a_{34} = -\frac{x_{24}}{x_{14} x_{23}}, a_{35} = \frac{x_{15} x_{24} - x_{14} x_{25} + {\left({2} \, x_{14}^{2} - {2} \, x_{14} x_{24}\right)} x_{45}}{x_{14}^{2} x_{23} x_{45}}, \\ 
	&a_{45} = \frac{{2} \, x_{14} x_{45} - x_{15}}{x_{14}^{2} x_{45}}
\end{align*}

Now assume $x_{45} = \frac{-x_{23}x_{35}}{x_{24}}$
\begin{equation*}x=\left(\begin{matrix}
		1 & 0 & 0 & x_{14} & x_{15} \\
		0 & 1 & x_{23} & x_{24} & x_{25} \\
		0 & 0 & 1 & 0 & x_{35}  \\
		0 & 0 & 0 & 1 & x_{45}  \\
		0 & 0 & 0 & 0 & 1 \\
	\end{matrix}\right),x^A=\left(\begin{matrix}
		1 & 0 & 0 & 1 & 0 \\
		0 & 1 & 1 & 0 & 0 \\
		0 & 0 & 1 & 0 & 1 \\
		0 & 0 & 0 & 1 & 1 \\
		0 & 0 & 0 & 0 & 1
	\end{matrix}\right)
\end{equation*} \newline
Where matrix $A$ has entries
\begin{align*}
	&d_{1} = 1, d_{2} = 1, d_{3} = \frac{1}{x_{23}}, d_{4} = \frac{1}{x_{14}}, d_{5} = -\frac{x_{24}}{x_{14} x_{23} x_{35}}, \\ 
	&a_{12} = 0, a_{13} = 1, a_{14} = 1, a_{15} = 1, \\ 
	&a_{23} = 1, a_{24} = 1, a_{25} = 1, \\ 
	&a_{34} = -\frac{x_{24}}{x_{14} x_{23}}, a_{35} = -\frac{x_{15} x_{24}^{2} - x_{14} x_{24} x_{25} - {\left(x_{14}^{2} x_{23} - x_{14} x_{23} x_{24}\right)} x_{35}}{x_{14}^{2} x_{23}^{2} x_{35}}, \\ 
	&a_{45} = \frac{x_{14} x_{23} x_{35} + x_{15} x_{24}}{x_{14}^{2} x_{23} x_{35}}
\end{align*}

First assume $x_{45} \neq \frac{-x_{23}x_{35}}{x_{24}}$
\begin{equation*}x=\left(\begin{matrix}
		1 & 0 & x_{13} & 0 & 0 \\
		0 & 1 & x_{23} & x_{24} & x_{25} \\
		0 & 0 & 1 & 0 & x_{35}  \\
		0 & 0 & 0 & 1 & x_{45}  \\
		0 & 0 & 0 & 0 & 1 \\
	\end{matrix}\right),x^A=\left(\begin{matrix}
		1 & 0 & 0 & 1 & 0 \\
		0 & 1 & 1 & 0 & 0 \\
		0 & 0 & 1 & 0 & 1 \\
		0 & 0 & 0 & 1 & 1 \\
		0 & 0 & 0 & 0 & 1
	\end{matrix}\right)
\end{equation*} \newline
Where matrix $A$ has entries
\begin{align*}
	&d_{1} = 1, d_{2} = -\frac{x_{23}^{2} x_{35} + x_{23} x_{24} x_{45}}{x_{13} x_{24} x_{45}}, d_{3} = -\frac{x_{23} x_{35} + x_{24} x_{45}}{x_{13} x_{24} x_{45}}, d_{4} = -\frac{x_{23}}{x_{13} x_{24}}, d_{5} = -\frac{x_{23}}{x_{13} x_{24} x_{45}}, \\ 
	&a_{12} = -\frac{x_{23} x_{35} + x_{24} x_{45}}{x_{24} x_{45}}, a_{13} = 1, a_{14} = 1, a_{15} = 1, \\ 
	&a_{23} = 1, a_{24} = 1, a_{25} = 1, \\ 
	&a_{34} = \frac{1}{x_{13}}, a_{35} = \frac{2}{x_{13}}, \\ 
	&a_{45} = \frac{x_{23} x_{25} + {\left({2} \, x_{13} - {2} \, x_{23}\right)} x_{24} x_{45}}{x_{13} x_{24}^{2} x_{45}}
\end{align*}

Now assume $x_{45} = \frac{-x_{23}x_{35}}{x_{24}}$
\begin{equation*}x=\left(\begin{matrix}
		1 & 0 & x_{13} & 0 & 0 \\
		0 & 1 & x_{23} & x_{24} & x_{25} \\
		0 & 0 & 1 & 0 & x_{35}  \\
		0 & 0 & 0 & 1 & x_{45}  \\
		0 & 0 & 0 & 0 & 1 \\
	\end{matrix}\right),x^A=\left(\begin{matrix}
		1 & 0 & 0 & 1 & 0 \\
		0 & 1 & 1 & 0 & 0 \\
		0 & 0 & 1 & 0 & 0 \\
		0 & 0 & 0 & 1 & 1 \\
		0 & 0 & 0 & 0 & 1
	\end{matrix}\right)
\end{equation*} \newline
Where matrix $A$ has entries
\begin{align*}
	&d_{1} = 1, d_{2} = 1, d_{3} = \frac{1}{x_{23}}, d_{4} = -\frac{x_{23}}{x_{13} x_{24}}, d_{5} = \frac{1}{x_{13} x_{35}}, \\ 
	&a_{12} = \frac{x_{13}}{x_{23}}, a_{13} = 1, a_{14} = 1, a_{15} = 1, \\ 
	&a_{23} = 1, a_{24} = 1, a_{25} = 1, \\ 
	&a_{34} = \frac{1}{x_{13}}, a_{35} = \frac{1}{x_{13}}, \\ 
	&a_{45} = -\frac{x_{25} - {\left(x_{13} - x_{23}\right)} x_{35}}{x_{13} x_{24} x_{35}}
\end{align*}


First assume $x_{45} \neq \frac{-x_{23}x_{35}}{x_{24}}$
\begin{equation*}x=\left(\begin{matrix}
		1 & 0 & x_{13} & 0 & x_{15} \\
		0 & 1 & x_{23} & x_{24} & x_{25} \\
		0 & 0 & 1 & 0 & x_{35}  \\
		0 & 0 & 0 & 1 & x_{45}  \\
		0 & 0 & 0 & 0 & 1 \\
	\end{matrix}\right),x^A=\left(\begin{matrix}
		1 & 0 & 0 & 1 & 0 \\
		0 & 1 & 1 & 0 & 0 \\
		0 & 0 & 1 & 0 & 1 \\
		0 & 0 & 0 & 1 & 1 \\
		0 & 0 & 0 & 0 & 1
	\end{matrix}\right)
\end{equation*} \newline
Where matrix $A$ has entries
\begin{align*}
	&d_{1} = 1, d_{2} = -\frac{x_{23}^{2} x_{35} + x_{23} x_{24} x_{45}}{x_{13} x_{24} x_{45}}, d_{3} = -\frac{x_{23} x_{35} + x_{24} x_{45}}{x_{13} x_{24} x_{45}}, d_{4} = -\frac{x_{23}}{x_{13} x_{24}}, d_{5} = -\frac{x_{23}}{x_{13} x_{24} x_{45}}, \\ 
	&a_{12} = -\frac{x_{23} x_{35} + x_{24} x_{45}}{x_{24} x_{45}}, a_{13} = 1, a_{14} = 1, a_{15} = 1, \\ 
	&a_{23} = 1, a_{24} = 1, a_{25} = 1, \\ 
	&a_{34} = \frac{1}{x_{13}}, a_{35} = \frac{{2} \, x_{13} x_{24} x_{45} + x_{15} x_{23}}{x_{13}^{2} x_{24} x_{45}}, \\ 
	&a_{45} = -\frac{x_{15} x_{23}^{2} - x_{13} x_{23} x_{25} - {\left({2} \, x_{13}^{2} - {2} \, x_{13} x_{23}\right)} x_{24} x_{45}}{x_{13}^{2} x_{24}^{2} x_{45}}
\end{align*}

Now assume $x_{45} = \frac{-x_{23}x_{35}}{x_{24}}$
\begin{equation*}x=\left(\begin{matrix}
		1 & 0 & x_{13} & 0 & x_{15} \\
		0 & 1 & x_{23} & x_{24} & x_{25} \\
		0 & 0 & 1 & 0 & x_{35}  \\
		0 & 0 & 0 & 1 & x_{45}  \\
		0 & 0 & 0 & 0 & 1 \\
	\end{matrix}\right),x^A=\left(\begin{matrix}
		1 & 0 & 0 & 1 & 0 \\
		0 & 1 & 1 & 0 & 0 \\
		0 & 0 & 1 & 0 & 0 \\
		0 & 0 & 0 & 1 & 1 \\
		0 & 0 & 0 & 0 & 1
	\end{matrix}\right)
\end{equation*} \newline
Where matrix $A$ has entries
\begin{align*}
	&d_{1} = 1, d_{2} = 1, d_{3} = \frac{1}{x_{23}}, d_{4} = -\frac{x_{23}}{x_{13} x_{24}}, d_{5} = \frac{1}{x_{13} x_{35}}, \\ 
	&a_{12} = \frac{x_{13}}{x_{23}}, a_{13} = 1, a_{14} = 1, a_{15} = 1, \\ 
	&a_{23} = 1, a_{24} = 1, a_{25} = 1, \\ 
	&a_{34} = \frac{1}{x_{13}}, a_{35} = \frac{x_{13} x_{35} - x_{15}}{x_{13}^{2} x_{35}}, \\ 
	&a_{45} = \frac{x_{15} x_{23} - x_{13} x_{25} + {\left(x_{13}^{2} - x_{13} x_{23}\right)} x_{35}}{x_{13}^{2} x_{24} x_{35}}
\end{align*}

First assume $x_{14}\neq \frac{x_{13}x_{24}}{x_{23}}$ and $x_{24}\neq \frac{-x_{23}x_{35}}{x_{45}}$
\begin{equation*}x=\left(\begin{matrix}
		1 & 0 & x_{13} & x_{14} & 0 \\
		0 & 1 & x_{23} & x_{24} & x_{25} \\
		0 & 0 & 1 & 0 & x_{35}  \\
		0 & 0 & 0 & 1 & x_{45}  \\
		0 & 0 & 0 & 0 & 1 \\
	\end{matrix}\right),x^A=\left(\begin{matrix}
		1 & 0 & 0 & 1 & 0 \\
		0 & 1 & 1 & 0 & 0 \\
		0 & 0 & 1 & 0 & 1 \\
		0 & 0 & 0 & 1 & 1 \\
		0 & 0 & 0 & 0 & 1
	\end{matrix}\right)
\end{equation*} \newline
Where matrix $A$ has entries
\begin{align*}
	&d_{1} = 1, d_{2} = \frac{x_{23}^{2} x_{35} + x_{23} x_{24} x_{45}}{{\left(x_{14} x_{23} - x_{13} x_{24}\right)} x_{45}}, d_{3} = \frac{x_{23} x_{35} + x_{24} x_{45}}{{\left(x_{14} x_{23} - x_{13} x_{24}\right)} x_{45}}, d_{4} = \frac{x_{23}}{x_{14} x_{23} - x_{13} x_{24}}, d_{5} = \frac{x_{23}}{{\left(x_{14} x_{23} - x_{13} x_{24}\right)} x_{45}}, \\ 
	&a_{12} = \frac{x_{13} x_{23} x_{35} + x_{13} x_{24} x_{45}}{{\left(x_{14} x_{23} - x_{13} x_{24}\right)} x_{45}}, a_{13} = 1, a_{14} = 1, a_{15} = 1, \\ 
	&a_{23} = 1, a_{24} = 1, a_{25} = 1, \\ 
	&a_{34} = -\frac{x_{24}}{x_{14} x_{23} - x_{13} x_{24}}, a_{35} = -\frac{x_{14} x_{23} x_{25} - {\left({2} \, x_{14}^{2} x_{23} + {2} \, x_{13} x_{24}^{2} - {\left({2} \, x_{13} x_{14} + {2} \, x_{14} x_{23}\right)} x_{24}\right)} x_{45}}{{\left(x_{14}^{2} x_{23}^{2} - 2 \, x_{13} x_{14} x_{23} x_{24} + x_{13}^{2} x_{24}^{2}\right)} x_{45}}, \\ 
	&a_{45} = \frac{x_{13} x_{23} x_{25} - {\left({2} \, x_{13} x_{14} x_{23} - {2} \, x_{14} x_{23}^{2} - {\left({2} \, x_{13}^{2} - {2} \, x_{13} x_{23}\right)} x_{24}\right)} x_{45}}{{\left(x_{14}^{2} x_{23}^{2} - 2 \, x_{13} x_{14} x_{23} x_{24} + x_{13}^{2} x_{24}^{2}\right)} x_{45}}
\end{align*}

Now assume $x_{14}= \frac{x_{13}x_{24}}{x_{23}}$ and $x_{24}\neq\frac{-x_{23}x_{35}}{x_{45}}$
\begin{equation*}x=\left(\begin{matrix}
		1 & 0 & x_{13} & x_{14} & 0 \\
		0 & 1 & x_{23} & x_{24} & x_{25} \\
		0 & 0 & 1 & 0 & x_{35}  \\
		0 & 0 & 0 & 1 & x_{45}  \\
		0 & 0 & 0 & 0 & 1 \\
	\end{matrix}\right),x^A=\left(\begin{matrix}
		1 & 0 & 0 & 0 & 0 \\
		0 & 1 & 1 & 0 & 0 \\
		0 & 0 & 1 & 0 & 1 \\
		0 & 0 & 0 & 1 & 1 \\
		0 & 0 & 0 & 0 & 1
	\end{matrix}\right)
\end{equation*} \newline
Where matrix $A$ has entries
\begin{align*}
	&d_{1} = 1, d_{2} = 1, d_{3} = \frac{1}{x_{23}}, d_{4} = \frac{x_{45}}{x_{23} x_{35} + x_{24} x_{45}}, d_{5} = \frac{1}{x_{23} x_{35} + x_{24} x_{45}}, \\ 
	&a_{12} = \frac{x_{13}}{x_{23}}, a_{13} = 1, a_{14} = 1, a_{15} = 1, \\ 
	&a_{23} = 1, a_{24} = \frac{x_{13} x_{25} - {\left(x_{13} - {2} \, x_{23}\right)} x_{24} x_{45} - {\left(x_{13} x_{23} - {2} \, x_{23}^{2}\right)} x_{35}}{x_{13} x_{23} x_{35} + x_{13} x_{24} x_{45}}, a_{25} = 1, \\ 
	&a_{34} = -\frac{x_{24} x_{45}}{x_{23}^{2} x_{35} + x_{23} x_{24} x_{45}}, a_{35} = 1, \\ 
	&a_{45} = -\frac{{\left(x_{13} - {2} \,\right)} x_{23}}{x_{13} x_{24}}
\end{align*}

Now assume $x_{14}= \frac{x_{13}x_{24}}{x_{23}}$ and $x_{24}=\frac{-x_{23}x_{35}}{x_{45}}$
\begin{equation*}x=\left(\begin{matrix}
		1 & 0 & x_{13} & x_{14} & 0 \\
		0 & 1 & x_{23} & x_{24} & x_{25} \\
		0 & 0 & 1 & 0 & x_{35}  \\
		0 & 0 & 0 & 1 & x_{45}  \\
		0 & 0 & 0 & 0 & 1 \\
	\end{matrix}\right),x^A=\left(\begin{matrix}
		1 & 0 & 0 & 0 & 0 \\
		0 & 1 & 1 & 0 & 0 \\
		0 & 0 & 1 & 0 & 0 \\
		0 & 0 & 0 & 1 & 1 \\
		0 & 0 & 0 & 0 & 1
	\end{matrix}\right)
\end{equation*} \newline
Where matrix $A$ has entries
\begin{align*}
	&d_{1} = 1, d_{2} = 1, d_{3} = \frac{1}{x_{23}}, d_{4} = 1, d_{5} = \frac{1}{x_{45}}, \\ 
	&a_{12} = \frac{x_{13}}{x_{23}}, a_{13} = 1, a_{14} = 1, a_{15} = 1, \\ 
	&a_{23} = 1, a_{24} = \frac{x_{13} x_{25} + x_{23} x_{45}}{x_{13} x_{45}}, a_{25} = 1, \\ 
	&a_{34} = \frac{x_{35}}{x_{45}}, a_{35} = 1, \\ 
	&a_{45} = \frac{{\left(x_{13} - 1\right)} x_{45}}{x_{13} x_{35}}
\end{align*}

Now assume $x_{14}\neq \frac{x_{13}x_{24}}{x_{23}}$, $x_{24}=\frac{-x_{23}x_{35}}{x_{45}}$ and $x_{14}\neq \frac{x_{13}x_{35}}{x_{45}}$
\begin{equation*}x=\left(\begin{matrix}
		1 & 0 & x_{13} & x_{14} & 0 \\
		0 & 1 & x_{23} & x_{24} & x_{25} \\
		0 & 0 & 1 & 0 & x_{35}  \\
		0 & 0 & 0 & 1 & x_{45}  \\
		0 & 0 & 0 & 0 & 1 \\
	\end{matrix}\right),x^A=\left(\begin{matrix}
		1 & 0 & 0 & 1 & 0 \\
		0 & 1 & 1 & 0 & 0 \\
		0 & 0 & 1 & 0 & 0 \\
		0 & 0 & 0 & 1 & 1 \\
		0 & 0 & 0 & 0 & 1
	\end{matrix}\right)
\end{equation*} \newline
Where matrix $A$ has entries
\begin{align*}
	&d_{1} = 1, d_{2} = 1, d_{3} = \frac{1}{x_{23}}, d_{4} = \frac{x_{45}}{x_{13} x_{35} + x_{14} x_{45}}, d_{5} = \frac{1}{x_{13} x_{35} + x_{14} x_{45}}, \\ 
	&a_{12} = \frac{x_{13}}{x_{23}}, a_{13} = 1, a_{14} = 1, a_{15} = 1, \\ 
	&a_{23} = 1, a_{24} = 1, a_{25} = 1, \\ 
	&a_{34} = \frac{x_{35}}{x_{13} x_{35} + x_{14} x_{45}}, a_{35} = \frac{x_{13} x_{23} x_{35}^{2} + x_{14}^{2} x_{45}^{2} - {\left(x_{14} x_{25} - {\left(x_{13} x_{14} + x_{14} x_{23}\right)} x_{35}\right)} x_{45}}{x_{13}^{2} x_{23} x_{35}^{2} + 2 \, x_{13} x_{14} x_{23} x_{35} x_{45} + x_{14}^{2} x_{23} x_{45}^{2}}, \\ 
	&a_{45} = -\frac{{\left(x_{13} x_{14} - x_{14} x_{23}\right)} x_{45}^{2} - {\left(x_{13} x_{25} - {\left(x_{13}^{2} - x_{13} x_{23}\right)} x_{35}\right)} x_{45}}{x_{13}^{2} x_{23} x_{35}^{2} + 2 \, x_{13} x_{14} x_{23} x_{35} x_{45} + x_{14}^{2} x_{23} x_{45}^{2}}
\end{align*}

Now assume $x_{14}\neq \frac{x_{13}x_{24}}{x_{23}}$, $x_{24}=\frac{-x_{23}x_{35}}{x_{45}}$ and $x_{14}= \frac{x_{13}x_{35}}{x_{45}}$
\begin{equation*}x=\left(\begin{matrix}
		1 & 0 & x_{13} & x_{14} & 0 \\
		0 & 1 & x_{23} & x_{24} & x_{25} \\
		0 & 0 & 1 & 0 & x_{35}  \\
		0 & 0 & 0 & 1 & x_{45}  \\
		0 & 0 & 0 & 0 & 1 \\
	\end{matrix}\right),x^A=\left(\begin{matrix}
		1 & 0 & 0 & 1 & 0 \\
		0 & 1 & 1 & 0 & 0 \\
		0 & 0 & 1 & 0 & 0 \\
		0 & 0 & 0 & 1 & 1 \\
		0 & 0 & 0 & 0 & 1
	\end{matrix}\right)
\end{equation*} \newline
Where matrix $A$ has entries
\begin{align*}
	&d_{1} = 1, d_{2} = 1, d_{3} = \frac{1}{x_{23}}, d_{4} = \frac{x_{45}}{2 \, x_{13} x_{35}}, d_{5} = \frac{1}{2 \, x_{13} x_{35}}, \\ 
	&a_{12} = \frac{x_{13}}{x_{23}}, a_{13} = 1, a_{14} = 1, a_{15} = 1, \\ 
	&a_{23} = 1, a_{24} = 1, a_{25} = 1, \\ 
	&a_{34} = \frac{1}{2 \, x_{13}}, a_{35} = -\frac{x_{25} - 2 \, {\left(x_{13} + x_{23}\right)} x_{35}}{4 \, x_{13} x_{23} x_{35}}, \\ 
	&a_{45} = \frac{{\left(x_{25} - 2 \, {\left(x_{13} - x_{23}\right)} x_{35}\right)} x_{45}}{4 \, x_{13} x_{23} x_{35}^{2}}
\end{align*}


First assume $x_{14}\neq \frac{x_{13}x_{24}}{x_{23}}$ and $x_{24}\neq \frac{-x_{23}x_{35}}{x_{45}}$
\begin{equation*}x=\left(\begin{matrix}
		1 & 0 & x_{13} & x_{14} & x_{15} \\
		0 & 1 & x_{23} & x_{24} & x_{25} \\
		0 & 0 & 1 & 0 & x_{35}  \\
		0 & 0 & 0 & 1 & x_{45}  \\
		0 & 0 & 0 & 0 & 1 \\
	\end{matrix}\right),x^A=\left(\begin{matrix}
		1 & 0 & 0 & 1 & 0 \\
		0 & 1 & 1 & 0 & 0 \\
		0 & 0 & 1 & 0 & 1 \\
		0 & 0 & 0 & 1 & 1 \\
		0 & 0 & 0 & 0 & 1
	\end{matrix}\right)
\end{equation*} \newline
Where matrix $A$ has entries
\begin{align*}
	&d_{1} = 1, d_{2} = \frac{x_{23}^{2} x_{35} + x_{23} x_{24} x_{45}}{{\left(x_{14} x_{23} - x_{13} x_{24}\right)} x_{45}}, d_{3} = \frac{x_{23} x_{35} + x_{24} x_{45}}{{\left(x_{14} x_{23} - x_{13} x_{24}\right)} x_{45}}, d_{4} = \frac{x_{23}}{x_{14} x_{23} - x_{13} x_{24}}, d_{5} = \frac{x_{23}}{{\left(x_{14} x_{23} - x_{13} x_{24}\right)} x_{45}}, \\ 
	&a_{12} = \frac{x_{13} x_{23} x_{35} + x_{13} x_{24} x_{45}}{{\left(x_{14} x_{23} - x_{13} x_{24}\right)} x_{45}}, a_{13} = 1, a_{14} = 1, a_{15} = 1, \\ 
	&a_{23} = 1, a_{24} = 1, a_{25} = 1, \\ 
	&a_{34} = -\frac{x_{24}}{x_{14} x_{23} - x_{13} x_{24}}, a_{35} = \frac{x_{15} x_{23} x_{24} - x_{14} x_{23} x_{25} + {\left({2} \, x_{14}^{2} x_{23} + {2} \, x_{13} x_{24}^{2} - {\left({2} \, x_{13} x_{14} + {2} \, x_{14} x_{23}\right)} x_{24}\right)} x_{45}}{{\left(x_{14}^{2} x_{23}^{2} - 2 \, x_{13} x_{14} x_{23} x_{24} + x_{13}^{2} x_{24}^{2}\right)} x_{45}}, \\ 
	&a_{45} = -\frac{x_{15} x_{23}^{2} - x_{13} x_{23} x_{25} + {\left({2} \, x_{13} x_{14} x_{23} - {2} \, x_{14} x_{23}^{2} - {\left({2} \, x_{13}^{2} - {2} \, x_{13} x_{23}\right)} x_{24}\right)} x_{45}}{{\left(x_{14}^{2} x_{23}^{2} - 2 \, x_{13} x_{14} x_{23} x_{24} + x_{13}^{2} x_{24}^{2}\right)} x_{45}}
\end{align*}

Now assume $x_{14}= \frac{x_{13}x_{24}}{x_{23}}$ and $x_{24}\neq\frac{-x_{23}x_{35}}{x_{45}}$
\begin{equation*}x=\left(\begin{matrix}
		1 & 0 & x_{13} & x_{14} & x_{15} \\
		0 & 1 & x_{23} & x_{24} & x_{25} \\
		0 & 0 & 1 & 0 & x_{35}  \\
		0 & 0 & 0 & 1 & x_{45}  \\
		0 & 0 & 0 & 0 & 1 \\
	\end{matrix}\right),x^A=\left(\begin{matrix}
		1 & 0 & 0 & 0 & 0 \\
		0 & 1 & 1 & 0 & 0 \\
		0 & 0 & 1 & 0 & 1 \\
		0 & 0 & 0 & 1 & 1 \\
		0 & 0 & 0 & 0 & 1
	\end{matrix}\right)
\end{equation*} \newline
Where matrix $A$ has entries
\begin{align*}
	&d_{1} = 1, d_{2} = 1, d_{3} = \frac{1}{x_{23}}, d_{4} = \frac{x_{45}}{x_{23} x_{35} + x_{24} x_{45}}, d_{5} = \frac{1}{x_{23} x_{35} + x_{24} x_{45}}, \\ 
	&a_{12} = \frac{x_{13}}{x_{23}}, a_{13} = 1, a_{14} = 1, a_{15} = 1, \\ 
	&a_{23} = 1, a_{24} = -\frac{x_{15} x_{23} - x_{13} x_{25} + {\left(x_{13} - {2} \, x_{23}\right)} x_{24} x_{45} + {\left(x_{13} x_{23} - {2} \, x_{23}^{2}\right)} x_{35}}{x_{13} x_{23} x_{35} + x_{13} x_{24} x_{45}}, a_{25} = 1, \\ 
	&a_{34} = -\frac{x_{24} x_{45}}{x_{23}^{2} x_{35} + x_{23} x_{24} x_{45}}, a_{35} = 1, \\ 
	&a_{45} = -\frac{{\left(x_{13} - {2} \,\right)} x_{23}^{2} x_{35} + {\left(x_{13} - {2} \,\right)} x_{23} x_{24} x_{45} + x_{15} x_{23}}{x_{13} x_{23} x_{24} x_{35} + x_{13} x_{24}^{2} x_{45}}
\end{align*}

Now assume $x_{14}= \frac{x_{13}x_{24}}{x_{23}}$ and $x_{24}=\frac{-x_{23}x_{35}}{x_{45}}$
\begin{equation*}x=\left(\begin{matrix}
		1 & 0 & x_{13} & x_{14} & x_{15} \\
		0 & 1 & x_{23} & x_{24} & x_{25} \\
		0 & 0 & 1 & 0 & x_{35}  \\
		0 & 0 & 0 & 1 & x_{45}  \\
		0 & 0 & 0 & 0 & 1 \\
	\end{matrix}\right),x^A=\left(\begin{matrix}
		1 & 0 & 0 & 0 & 0 \\
		0 & 1 & 1 & 0 & 0 \\
		0 & 0 & 1 & 0 & 0 \\
		0 & 0 & 0 & 1 & 1 \\
		0 & 0 & 0 & 0 & 1
	\end{matrix}\right)
\end{equation*} \newline
Where matrix $A$ has entries
\begin{align*}
	&d_{1} = 1, d_{2} = 1, d_{3} = \frac{1}{x_{23}}, d_{4} = 1, d_{5} = \frac{1}{x_{45}}, \\ 
	&a_{12} = \frac{x_{13}}{x_{23}}, a_{13} = 1, a_{14} = 1, a_{15} = 1, \\ 
	&a_{23} = 1, a_{24} = -\frac{x_{15} x_{23} - x_{13} x_{25} - x_{23} x_{45}}{x_{13} x_{45}}, a_{25} = 1, \\ 
	&a_{34} = \frac{x_{35}}{x_{45}}, a_{35} = 1, \\ 
	&a_{45} = \frac{x_{15} + {\left(x_{13} - 1\right)} x_{45}}{x_{13} x_{35}}
\end{align*}

Now assume $x_{14}\neq \frac{x_{13}x_{24}}{x_{23}}$, $x_{24}=\frac{-x_{23}x_{35}}{x_{45}}$ and $x_{14}\neq \frac{x_{13}x_{35}}{x_{45}}$
\begin{equation*}x=\left(\begin{matrix}
		1 & 0 & x_{13} & x_{14} & x_{15} \\
		0 & 1 & x_{23} & x_{24} & x_{25} \\
		0 & 0 & 1 & 0 & x_{35}  \\
		0 & 0 & 0 & 1 & x_{45}  \\
		0 & 0 & 0 & 0 & 1 \\
	\end{matrix}\right),x^A=\left(\begin{matrix}
		1 & 0 & 0 & 1 & 0 \\
		0 & 1 & 1 & 0 & 0 \\
		0 & 0 & 1 & 0 & 0 \\
		0 & 0 & 0 & 1 & 1 \\
		0 & 0 & 0 & 0 & 1
	\end{matrix}\right)
\end{equation*} \newline
Where matrix $A$ has entries
\begin{align*}
	&d_{1} = 1, d_{2} = 1, d_{3} = \frac{1}{x_{23}}, d_{4} = \frac{x_{45}}{x_{13} x_{35} + x_{14} x_{45}}, d_{5} = \frac{1}{x_{13} x_{35} + x_{14} x_{45}}, \\ 
	&a_{12} = \frac{x_{13}}{x_{23}}, a_{13} = 1, a_{14} = 1, a_{15} = 1, \\ 
	&a_{23} = 1, a_{24} = 1, a_{25} = 1, \\ 
	&a_{34} = \frac{x_{35}}{x_{13} x_{35} + x_{14} x_{45}}, a_{35} = \frac{x_{13} x_{23} x_{35}^{2} + x_{14}^{2} x_{45}^{2} - x_{15} x_{23} x_{35} - {\left(x_{14} x_{25} - {\left(x_{13} x_{14} + x_{14} x_{23}\right)} x_{35}\right)} x_{45}}{x_{13}^{2} x_{23} x_{35}^{2} + 2 \, x_{13} x_{14} x_{23} x_{35} x_{45} + x_{14}^{2} x_{23} x_{45}^{2}}, \\ 
	&a_{45} = -\frac{{\left(x_{13} x_{14} - x_{14} x_{23}\right)} x_{45}^{2} + {\left(x_{15} x_{23} - x_{13} x_{25} + {\left(x_{13}^{2} - x_{13} x_{23}\right)} x_{35}\right)} x_{45}}{x_{13}^{2} x_{23} x_{35}^{2} + 2 \, x_{13} x_{14} x_{23} x_{35} x_{45} + x_{14}^{2} x_{23} x_{45}^{2}}
\end{align*}

Now assume $x_{14}\neq \frac{x_{13}x_{24}}{x_{23}}$, $x_{24}=\frac{-x_{23}x_{35}}{x_{45}}$ and $x_{14}= \frac{x_{13}x_{35}}{x_{45}}$
\begin{equation*}x=\left(\begin{matrix}
		1 & 0 & x_{13} & x_{14} & x_{15} \\
		0 & 1 & x_{23} & x_{24} & x_{25} \\
		0 & 0 & 1 & 0 & x_{35}  \\
		0 & 0 & 0 & 1 & x_{45}  \\
		0 & 0 & 0 & 0 & 1 \\
	\end{matrix}\right),x^A=\left(\begin{matrix}
		1 & 0 & 0 & 0 & 0 \\
		0 & 1 & 1 & 0 & 0 \\
		0 & 0 & 1 & 0 & 0 \\
		0 & 0 & 0 & 1 & 1 \\
		0 & 0 & 0 & 0 & 1
	\end{matrix}\right)
\end{equation*} \newline
Where matrix $A$ has entries
\begin{align*}
	&d_{1} = 1, d_{2} = 1, d_{3} = \frac{1}{x_{23}}, d_{4} = 1, d_{5} = \frac{1}{x_{45}}, \\ 
	&a_{12} = -\frac{x_{14} x_{45}}{x_{23} x_{35}}, a_{13} = 1, a_{14} = 1, a_{15} = 1, \\ 
	&a_{23} = 1, a_{24} = \frac{x_{15} x_{23} x_{35} + {\left(x_{14} x_{25} - x_{23} x_{35}\right)} x_{45}}{x_{14} x_{45}^{2}}, a_{25} = 1, \\ 
	&a_{34} = \frac{x_{35}}{x_{45}}, a_{35} = 1, \\ 
	&a_{45} = \frac{x_{14} x_{45}^{2} - x_{15} x_{35} + x_{35} x_{45}}{x_{14} x_{35} x_{45}}
\end{align*}

		\section{Subcases of $Y_6$}

\begin{equation*}x=\left(
\right)
\end{equation*} \newline
Where matrix $A$ has entries
\begin{align*}
	&d_{1} = 1, d_{2} = 1, d_{3} = \frac{1}{x_{23}}, d_{4} = \frac{1}{x_{23} x_{34}}, d_{5} = \frac{x_{23}}{x_{15} x_{23} - x_{13} x_{25}},\\ 
	&a_{12} = \frac{x_{13}}{x_{23}}, a_{13} = 1, a_{14} = 1, a_{15} = 1,\\ 
	&a_{23} = \frac{x_{23} x_{34} - x_{14}}{x_{13} x_{34}}, a_{24} = 1, a_{25} = 1,\\ 
	&a_{34} = \frac{x_{23} x_{34} - x_{14}}{x_{13} x_{23} x_{34}}, a_{35} = -\frac{x_{25}}{x_{15} x_{23} - x_{13} x_{25}},\\ 
	&a_{45} = 0
\end{align*}

Now assume $x_{15}=\frac{x_{13}x_{25}}{x_{23}}$
\begin{equation*}x=\left(
\right)
\end{equation*} \newline
Where matrix $A$ has entries
\begin{align*}
	&d_{1} = 1, d_{2} = 1, d_{3} = \frac{1}{x_{23}}, d_{4} = \frac{1}{x_{23} x_{34}}, d_{5} = -\frac{x_{23}}{x_{13} x_{25}},\\ 
	&a_{12} = \frac{x_{13}}{x_{23}}, a_{13} = 1, a_{14} = 1, a_{15} = 1,\\ 
	&a_{23} = \frac{x_{23}^{2} x_{34} - x_{14} x_{23} + x_{13} x_{24}}{x_{13} x_{23} x_{34}}, a_{24} = 1, a_{25} = 1,\\ 
	&a_{34} = \frac{x_{23} x_{34} - x_{14}}{x_{13} x_{23} x_{34}}, a_{35} = \frac{1}{x_{13}},\\ 
	&a_{45} = 0
\end{align*}


First assume $x_{15}\neq \frac{x_{13}x_{25}}{x_{23}}$
\begin{equation*}x=\left(\begin{matrix}
		1 & 0 & x_{13} & x_{14} & x_{15} \\
		0 & 1 & x_{23} & x_{24} & x_{25} \\
		0 & 0 & 1 & x_{34} & 0  \\
		0 & 0 & 0 & 1 & 0  \\
		0 & 0 & 0 & 0 & 1 \\
	\end{matrix}\right),x^A=\left(\begin{matrix}
		1 & 0 & 0 & 0 & 1 \\
		0 & 1 & 1 & 0 & 0 \\
		0 & 0 & 1 & 1 & 0 \\
		0 & 0 & 0 & 1 & 0 \\
		0 & 0 & 0 & 0 & 1
	\end{matrix}\right)
\end{equation*} \newline
Where matrix $A$ has entries
\begin{align*}
	&d_{1} = 1, d_{2} = 1, d_{3} = \frac{1}{x_{23}}, d_{4} = \frac{1}{x_{23} x_{34}}, d_{5} = \frac{x_{23}}{x_{15} x_{23} - x_{13} x_{25}},\\ 
	&a_{12} = \frac{x_{13}}{x_{23}}, a_{13} = 1, a_{14} = 1, a_{15} = 1,\\ 
	&a_{23} = \frac{x_{23}^{2} x_{34} - x_{14} x_{23} + x_{13} x_{24}}{x_{13} x_{23} x_{34}}, a_{24} = 1, a_{25} = 1,\\ 
	&a_{34} = \frac{x_{23} x_{34} - x_{14}}{x_{13} x_{23} x_{34}}, a_{35} = -\frac{x_{25}}{x_{15} x_{23} - x_{13} x_{25}},\\ 
	&a_{45} = 0
\end{align*}

First assume $x_{15}= \frac{x_{13}x_{25}}{x_{23}}$
\begin{equation*}x=\left(
\right)
\end{equation*} \newline
Where matrix $A$ has entries
\begin{align*}
	&d_{1} = 1, d_{2} = 1, d_{3} = \frac{1}{x_{23}}, d_{4} = \frac{1}{x_{23} x_{34}}, d_{5} = \frac{x_{34}}{x_{15} x_{34} - x_{14} x_{35}},\\ 
	&a_{12} = \frac{x_{13}}{x_{23}}, a_{13} = 1, a_{14} = 1, a_{15} = 1,\\ 
	&a_{23} = \frac{x_{23} x_{34} - x_{14}}{x_{13} x_{34}}, a_{24} = 1, a_{25} = 1,\\ 
	&a_{34} = \frac{x_{23} x_{34} - x_{14}}{x_{13} x_{23} x_{34}}, a_{35} = 0,\\ 
	&a_{45} = -\frac{x_{35}}{x_{15} x_{34} - x_{14} x_{35}}
\end{align*}

Now assume $x_{15}= \frac{x_{14}x_{35}}{x_{34}}$
\begin{equation*}x=\left(
\right)
\end{equation*} \newline
Where matrix $A$ has entries
\begin{align*}
	&d_{1} = 1, d_{2} = 1, d_{3} = \frac{1}{x_{23}}, d_{4} = \frac{1}{x_{23} x_{34}}, d_{5} = \frac{x_{34}}{x_{15} x_{34} - x_{14} x_{35}},\\ 
	&a_{12} = 0, a_{13} = \frac{x_{14}}{x_{23} x_{34}}, a_{14} = 1, a_{15} = 1,\\ 
	&a_{23} = 1, a_{24} = 1, a_{25} = 1,\\ 
	&a_{34} = \frac{1}{x_{23}}, a_{35} = -\frac{x_{25} x_{34}}{x_{15} x_{23} x_{34} - x_{14} x_{23} x_{35}},\\ 
	&a_{45} = -\frac{x_{35}}{x_{15} x_{34} - x_{14} x_{35}}
\end{align*}

Now assume $x_{15}= \frac{x_{14}x_{35}}{x_{34}}$
\begin{equation*}x=\left(\begin{matrix}
		1 & 0 & 0 & x_{14} & x_{15} \\
		0 & 1 & x_{23} & 0 & x_{25} \\
		0 & 0 & 1 & x_{34} & x_{35}  \\
		0 & 0 & 0 & 1 & 0  \\
		0 & 0 & 0 & 0 & 1 \\
	\end{matrix}\right),x^A=\left(\begin{matrix}
		1 & 0 & 0 & 0 & 0 \\
		0 & 1 & 1 & 0 & 0 \\
		0 & 0 & 1 & 1 & 0 \\
		0 & 0 & 0 & 1 & 0 \\
		0 & 0 & 0 & 0 & 1
	\end{matrix}\right)
\end{equation*} \newline
Where matrix $A$ has entries
\begin{align*}
	&d_{1} = 1, d_{2} = 1, d_{3} = \frac{1}{x_{23}}, d_{4} = \frac{1}{x_{23} x_{34}}, d_{5} = 1,\\ 
	&a_{12} = 0, a_{13} = \frac{x_{14}}{x_{23} x_{34}}, a_{14} = 1, a_{15} = 1,\\ 
	&a_{23} = 1, a_{24} = 1, a_{25} = 1,\\ 
	&a_{34} = \frac{1}{x_{23}}, a_{35} = -\frac{x_{25}}{x_{23}},\\ 
	&a_{45} = -\frac{x_{35}}{x_{34}}
\end{align*}

\begin{equation*}x=\left(\begin{matrix}
		1 & 0 & x_{13} & 0 & 0 \\
		0 & 1 & x_{23} & 0 & x_{25} \\
		0 & 0 & 1 & x_{34} & x_{35}  \\
		0 & 0 & 0 & 1 & 0  \\
		0 & 0 & 0 & 0 & 1 \\
	\end{matrix}\right),x^A=\left(\begin{matrix}
		1 & 0 & 0 & 0 & 1 \\
		0 & 1 & 1 & 0 & 0 \\
		0 & 0 & 1 & 1 & 0 \\
		0 & 0 & 0 & 1 & 0 \\
		0 & 0 & 0 & 0 & 1
	\end{matrix}\right)
\end{equation*} \newline
Where matrix $A$ has entries
\begin{align*}
	&d_{1} = 1, d_{2} = 1, d_{3} = \frac{1}{x_{23}}, d_{4} = \frac{1}{x_{23} x_{34}}, d_{5} = -\frac{x_{23}}{x_{13} x_{25}},\\ 
	&a_{12} = \frac{x_{13}}{x_{23}}, a_{13} = 1, a_{14} = 1, a_{15} = 1,\\ 
	&a_{23} = \frac{x_{23}}{x_{13}}, a_{24} = 1, a_{25} = 1,\\ 
	&a_{34} = \frac{1}{x_{13}}, a_{35} = \frac{1}{x_{13}},\\ 
	&a_{45} = \frac{x_{23} x_{35}}{x_{13} x_{25} x_{34}}
\end{align*}

First assume $x_{15}\neq \frac{x_{13}x_{25}}{x_{23}}$
\begin{equation*}x=\left(\begin{matrix}
		1 & 0 & x_{13} & 0 & x_{15} \\
		0 & 1 & x_{23} & 0 & x_{25} \\
		0 & 0 & 1 & x_{34} & x_{35}  \\
		0 & 0 & 0 & 1 & 0  \\
		0 & 0 & 0 & 0 & 1 \\
	\end{matrix}\right),x^A=\left(\begin{matrix}
		1 & 0 & 0 & 0 & 1 \\
		0 & 1 & 1 & 0 & 0 \\
		0 & 0 & 1 & 1 & 0 \\
		0 & 0 & 0 & 1 & 0 \\
		0 & 0 & 0 & 0 & 1
	\end{matrix}\right)
\end{equation*} \newline
Where matrix $A$ has entries
\begin{align*}
	&d_{1} = 1, d_{2} = 1, d_{3} = \frac{1}{x_{23}}, d_{4} = \frac{1}{x_{23} x_{34}}, d_{5} = \frac{x_{23}}{x_{15} x_{23} - x_{13} x_{25}},\\ 
	&a_{12} = \frac{x_{13}}{x_{23}}, a_{13} = 1, a_{14} = 1, a_{15} = 1,\\ 
	&a_{23} = \frac{x_{23}}{x_{13}}, a_{24} = 1, a_{25} = 1,\\ 
	&a_{34} = \frac{1}{x_{13}}, a_{35} = -\frac{x_{25}}{x_{15} x_{23} - x_{13} x_{25}},\\ 
	&a_{45} = -\frac{x_{23} x_{35}}{{\left(x_{15} x_{23} - x_{13} x_{25}\right)} x_{34}}
\end{align*}

Now assume $x_{15}= \frac{x_{13}x_{25}}{x_{23}}$
\begin{equation*}x=\left(\begin{matrix}
		1 & 0 & x_{13} & 0 & x_{15} \\
		0 & 1 & x_{23} & 0 & x_{25} \\
		0 & 0 & 1 & x_{34} & x_{35}  \\
		0 & 0 & 0 & 1 & 0  \\
		0 & 0 & 0 & 0 & 1 \\
	\end{matrix}\right),x^A=\left(\begin{matrix}
		1 & 0 & 0 & 0 & 0 \\
		0 & 1 & 1 & 0 & 0 \\
		0 & 0 & 1 & 1 & 0 \\
		0 & 0 & 0 & 1 & 0 \\
		0 & 0 & 0 & 0 & 1
	\end{matrix}\right)
\end{equation*} \newline
Where matrix $A$ has entries
\begin{align*}
	&d_{1} = 1, d_{2} = 1, d_{3} = \frac{1}{x_{23}}, d_{4} = \frac{1}{x_{23} x_{34}}, d_{5} = 1,\\ 
	&a_{12} = \frac{x_{13}}{x_{23}}, a_{13} = 1, a_{14} = 1, a_{15} = 1,\\ 
	&a_{23} = \frac{x_{23}}{x_{13}}, a_{24} = 1, a_{25} = 1,\\ 
	&a_{34} = \frac{1}{x_{13}}, a_{35} = -\frac{x_{25}}{x_{23}},\\ 
	&a_{45} = -\frac{x_{35}}{x_{34}}
\end{align*}

First assume $x_{35} \neq \frac{-x_{13}x_{25}x_{34}}{x_{14}x_{23}}$
\begin{equation*}x=\left(\begin{matrix}
		1 & 0 & x_{13} & x_{14} & 0 \\
		0 & 1 & x_{23} & 0 & x_{25} \\
		0 & 0 & 1 & x_{34} & x_{35}  \\
		0 & 0 & 0 & 1 & 0  \\
		0 & 0 & 0 & 0 & 1 \\
	\end{matrix}\right),x^A=\left(\begin{matrix}
		1 & 0 & 0 & 0 & 1 \\
		0 & 1 & 1 & 0 & 0 \\
		0 & 0 & 1 & 1 & 0 \\
		0 & 0 & 0 & 1 & 0 \\
		0 & 0 & 0 & 0 & 1
	\end{matrix}\right)
\end{equation*} \newline
Where matrix $A$ has entries
\begin{align*}
	&d_{1} = 1, d_{2} = 1, d_{3} = \frac{1}{x_{23}}, d_{4} = \frac{1}{x_{23} x_{34}}, d_{5} = -\frac{x_{23} x_{34}}{x_{13} x_{25} x_{34} + x_{14} x_{23} x_{35}},\\ 
	&a_{12} = \frac{x_{13}}{x_{23}}, a_{13} = 1, a_{14} = 1, a_{15} = 1,\\ 
	&a_{23} = \frac{x_{23} x_{34} - x_{14}}{x_{13} x_{34}}, a_{24} = 1, a_{25} = 1,\\ 
	&a_{34} = \frac{x_{23} x_{34} - x_{14}}{x_{13} x_{23} x_{34}}, a_{35} = \frac{x_{25} x_{34}}{x_{13} x_{25} x_{34} + x_{14} x_{23} x_{35}},\\ 
	&a_{45} = \frac{x_{23} x_{35}}{x_{13} x_{25} x_{34} + x_{14} x_{23} x_{35}}
\end{align*}

First assume $x_{35} = \frac{-x_{13}x_{25}x_{34}}{x_{14}x_{23}}$
\begin{equation*}x=\left(\begin{matrix}
		1 & 0 & x_{13} & x_{14} & 0 \\
		0 & 1 & x_{23} & 0 & x_{25} \\
		0 & 0 & 1 & x_{34} & x_{35}  \\
		0 & 0 & 0 & 1 & 0  \\
		0 & 0 & 0 & 0 & 1 \\
	\end{matrix}\right),x^A=\left(\begin{matrix}
		1 & 0 & 0 & 0 & 0 \\
		0 & 1 & 1 & 0 & 0 \\
		0 & 0 & 1 & 1 & 0 \\
		0 & 0 & 0 & 1 & 0 \\
		0 & 0 & 0 & 0 & 1
	\end{matrix}\right)
\end{equation*} \newline
Where matrix $A$ has entries
\begin{align*}
	&d_{1} = 1, d_{2} = 1, d_{3} = \frac{1}{x_{23}}, d_{4} = \frac{1}{x_{23} x_{34}}, d_{5} = 1,\\ 
	&a_{12} = \frac{x_{13}}{x_{23}}, a_{13} = 1, a_{14} = 1, a_{15} = 1,\\ 
	&a_{23} = \frac{x_{23} x_{34} - x_{14}}{x_{13} x_{34}}, a_{24} = 1, a_{25} = 1,\\ 
	&a_{34} = \frac{x_{23} x_{34} - x_{14}}{x_{13} x_{23} x_{34}}, a_{35} = -\frac{x_{25}}{x_{23}},\\ 
	&a_{45} = \frac{x_{13} x_{25}}{x_{14} x_{23}}
\end{align*}


First assume $x_{35} \neq \frac{(x_{15}x_{23}-x_{13}x_{25})x_{34}}{x_{14}x_{23}}$

\begin{equation*}x=\left(\begin{matrix}
		1 & 0 & x_{13} & x_{14} & x_{15} \\
		0 & 1 & x_{23} & 0 & x_{25} \\
		0 & 0 & 1 & x_{34} & x_{35}  \\
		0 & 0 & 0 & 1 & 0  \\
		0 & 0 & 0 & 0 & 1 \\
	\end{matrix}\right),x^A=\left(\begin{matrix}
		1 & 0 & 0 & 0 & 1 \\
		0 & 1 & 1 & 0 & 0 \\
		0 & 0 & 1 & 1 & 0 \\
		0 & 0 & 0 & 1 & 0 \\
		0 & 0 & 0 & 0 & 1
	\end{matrix}\right)
\end{equation*} \newline
Where matrix $A$ has entries
\begin{align*}
	&d_{1} = 1, d_{2} = 1, d_{3} = \frac{1}{x_{23}}, d_{4} = \frac{1}{x_{23} x_{34}}, d_{5} = -\frac{x_{23} x_{34}}{x_{14} x_{23} x_{35} - {\left(x_{15} x_{23} - x_{13} x_{25}\right)} x_{34}},\\ 
	&a_{12} = \frac{x_{13}}{x_{23}}, a_{13} = 1, a_{14} = 1, a_{15} = 1,\\ 
	&a_{23} = \frac{x_{23} x_{34} - x_{14}}{x_{13} x_{34}}, a_{24} = 1, a_{25} = 1,\\ 
	&a_{34} = \frac{x_{23} x_{34} - x_{14}}{x_{13} x_{23} x_{34}}, a_{35} = \frac{x_{25} x_{34}}{x_{14} x_{23} x_{35} - {\left(x_{15} x_{23} - x_{13} x_{25}\right)} x_{34}},\\ 
	&a_{45} = \frac{x_{23} x_{35}}{x_{14} x_{23} x_{35} - {\left(x_{15} x_{23} - x_{13} x_{25}\right)} x_{34}}
\end{align*}

Now assume $x_{35} = \frac{(x_{15}x_{23}-x_{13}x_{25})x_{34}}{x_{14}x_{23}}$

\begin{equation*}x=\left(
\right)
\end{equation*} \newline
Where matrix $A$ has entries
\begin{align*}
	&d_{1} = 1, d_{2} = 1, d_{3} = \frac{1}{x_{23}}, d_{4} = \frac{1}{x_{23} x_{34}}, d_{5} = \frac{x_{34}}{x_{15} x_{34} - x_{14} x_{35}},\\ 
	&a_{12} = 0, a_{13} = \frac{x_{14}}{x_{23} x_{34}}, a_{14} = 1, a_{15} = 1,\\ 
	&a_{23} = 1, a_{24} = 1, a_{25} = 1,\\ 
	&a_{34} = \frac{x_{23} x_{34} - x_{24}}{x_{23}^{2} x_{34}}, a_{35} = \frac{x_{24} x_{35}}{x_{15} x_{23} x_{34} - x_{14} x_{23} x_{35}},\\ 
	&a_{45} = -\frac{x_{35}}{x_{15} x_{34} - x_{14} x_{35}}
\end{align*}

Now assume $x_{15}= \frac{x_{14}x_{35}}{x_{34}}$
\begin{equation*}x=\left(\begin{matrix}
		1 & 0 & 0 & x_{14} & x_{15} \\
		0 & 1 & x_{23} & x_{24} & 0 \\
		0 & 0 & 1 & x_{34} & x_{35}  \\
		0 & 0 & 0 & 1 & 0  \\
		0 & 0 & 0 & 0 & 1 \\
	\end{matrix}\right),x^A=\left(\begin{matrix}
		1 & 0 & 0 & 0 & 0 \\
		0 & 1 & 1 & 0 & 0 \\
		0 & 0 & 1 & 1 & 0 \\
		0 & 0 & 0 & 1 & 0 \\
		0 & 0 & 0 & 0 & 1
	\end{matrix}\right)
\end{equation*} \newline
Where matrix $A$ has entries
\begin{align*}
	&d_{1} = 1, d_{2} = 1, d_{3} = \frac{1}{x_{23}}, d_{4} = \frac{1}{x_{23} x_{34}}, d_{5} = 1,\\ 
	&a_{12} = 0, a_{13} = \frac{x_{14}}{x_{23} x_{34}}, a_{14} = 1, a_{15} = 1,\\ 
	&a_{23} = 1, a_{24} = 1, a_{25} = 1,\\ 
	&a_{34} = \frac{x_{23} x_{34} - x_{24}}{x_{23}^{2} x_{34}}, a_{35} = \frac{x_{24} x_{35}}{x_{23} x_{34}},\\ 
	&a_{45} = -\frac{x_{35}}{x_{34}}
\end{align*}

\begin{equation*}x=\left(\begin{matrix}
		1 & 0 & x_{13} & 0 & 0 \\
		0 & 1 & x_{23} & x_{24} & 0 \\
		0 & 0 & 1 & x_{34} & x_{35}  \\
		0 & 0 & 0 & 1 & 0  \\
		0 & 0 & 0 & 0 & 1 \\
	\end{matrix}\right),x^A=\left(\begin{matrix}
		1 & 0 & 0 & 0 & 1 \\
		0 & 1 & 1 & 0 & 0 \\
		0 & 0 & 1 & 1 & 0 \\
		0 & 0 & 0 & 1 & 0 \\
		0 & 0 & 0 & 0 & 1
	\end{matrix}\right)
\end{equation*} \newline
Where matrix $A$ has entries
\begin{align*}
	&d_{1} = 1, d_{2} = 1, d_{3} = \frac{1}{x_{23}}, d_{4} = \frac{1}{x_{23} x_{34}}, d_{5} = \frac{x_{23} x_{34}}{x_{13} x_{24} x_{35}},\\ 
	&a_{12} = \frac{x_{13}}{x_{23}}, a_{13} = 1, a_{14} = 1, a_{15} = 1,\\ 
	&a_{23} = \frac{x_{23}^{2} x_{34} + x_{13} x_{24}}{x_{13} x_{23} x_{34}}, a_{24} = 1, a_{25} = 1,\\ 
	&a_{34} = \frac{1}{x_{13}}, a_{35} = \frac{1}{x_{13}},\\ 
	&a_{45} = -\frac{x_{23}}{x_{13} x_{24}}
\end{align*}


First assume $x_{35} \neq \frac{-x_{15}x_{23}x_{34}}{x_{13}x_{23}}$
\begin{equation*}x=\left(\begin{matrix}
		1 & 0 & x_{13} & 0 & x_{15} \\
		0 & 1 & x_{23} & x_{24} & 0 \\
		0 & 0 & 1 & x_{34} & x_{35}  \\
		0 & 0 & 0 & 1 & 0  \\
		0 & 0 & 0 & 0 & 1 \\
	\end{matrix}\right),x^A=\left(\begin{matrix}
		1 & 0 & 0 & 0 & 1 \\
		0 & 1 & 1 & 0 & 0 \\
		0 & 0 & 1 & 1 & 0 \\
		0 & 0 & 0 & 1 & 0 \\
		0 & 0 & 0 & 0 & 1
	\end{matrix}\right)
\end{equation*} \newline
Where matrix $A$ has entries
\begin{align*}
	&d_{1} = 1, d_{2} = 1, d_{3} = \frac{1}{x_{23}}, d_{4} = \frac{1}{x_{23} x_{34}}, d_{5} = \frac{x_{23} x_{34}}{x_{15} x_{23} x_{34} + x_{13} x_{24} x_{35}},\\ 
	&a_{12} = \frac{x_{13}}{x_{23}}, a_{13} = 1, a_{14} = 1, a_{15} = 1,\\ 
	&a_{23} = \frac{x_{23}^{2} x_{34} + x_{13} x_{24}}{x_{13} x_{23} x_{34}}, a_{24} = 1, a_{25} = 1,\\ 
	&a_{34} = \frac{1}{x_{13}}, a_{35} = \frac{x_{24} x_{35}}{x_{15} x_{23} x_{34} + x_{13} x_{24} x_{35}},\\ 
	&a_{45} = -\frac{x_{23} x_{35}}{x_{15} x_{23} x_{34} + x_{13} x_{24} x_{35}}
\end{align*}

Now assume $x_{35} = \frac{-x_{15}x_{23}x_{34}}{x_{13}x_{23}}$
\begin{equation*}x=\left(\begin{matrix}
		1 & 0 & x_{13} & 0 & x_{15} \\
		0 & 1 & x_{23} & x_{24} & 0 \\
		0 & 0 & 1 & x_{34} & x_{35}  \\
		0 & 0 & 0 & 1 & 0  \\
		0 & 0 & 0 & 0 & 1 \\
	\end{matrix}\right),x^A=\left(\begin{matrix}
		1 & 0 & 0 & 0 & 0 \\
		0 & 1 & 1 & 0 & 0 \\
		0 & 0 & 1 & 1 & 0 \\
		0 & 0 & 0 & 1 & 0 \\
		0 & 0 & 0 & 0 & 1
	\end{matrix}\right)
\end{equation*} \newline
Where matrix $A$ has entries
\begin{align*}
	&d_{1} = 1, d_{2} = 1, d_{3} = \frac{1}{x_{23}}, d_{4} = \frac{1}{x_{23} x_{34}}, d_{5} = 1,\\ 
	&a_{12} = \frac{x_{13}}{x_{23}}, a_{13} = 1, a_{14} = 1, a_{15} = 1,\\ 
	&a_{23} = \frac{x_{23}^{2} x_{34} + x_{13} x_{24}}{x_{13} x_{23} x_{34}}, a_{24} = 1, a_{25} = 1,\\ 
	&a_{34} = \frac{1}{x_{13}}, a_{35} = -\frac{x_{15}}{x_{13}},\\ 
	&a_{45} = \frac{x_{15} x_{23}}{x_{13} x_{24}}
\end{align*}

First assume $x_{14}\neq \frac{x_{13}x_{24}}{x_{23}}$
\begin{equation*}x=\left(\begin{matrix}
		1 & 0 & x_{13} & x_{14} & 0 \\
		0 & 1 & x_{23} & x_{24} & 0 \\
		0 & 0 & 1 & x_{34} & x_{35}  \\
		0 & 0 & 0 & 1 & 0  \\
		0 & 0 & 0 & 0 & 1 \\
	\end{matrix}\right),x^A=\left(\begin{matrix}
		1 & 0 & 0 & 0 & 1 \\
		0 & 1 & 1 & 0 & 0 \\
		0 & 0 & 1 & 1 & 0 \\
		0 & 0 & 0 & 1 & 0 \\
		0 & 0 & 0 & 0 & 1
	\end{matrix}\right)
\end{equation*} \newline
Where matrix $A$ has entries
\begin{align*}
	&d_{1} = 1, d_{2} = 1, d_{3} = \frac{1}{x_{23}}, d_{4} = \frac{1}{x_{23} x_{34}}, d_{5} = -\frac{x_{23} x_{34}}{{\left(x_{14} x_{23} - x_{13} x_{24}\right)} x_{35}},\\ 
	&a_{12} = \frac{x_{13}}{x_{23}}, a_{13} = 1, a_{14} = 1, a_{15} = 1,\\ 
	&a_{23} = \frac{x_{23}^{2} x_{34} - x_{14} x_{23} + x_{13} x_{24}}{x_{13} x_{23} x_{34}}, a_{24} = 1, a_{25} = 1,\\ 
	&a_{34} = \frac{x_{23} x_{34} - x_{14}}{x_{13} x_{23} x_{34}}, a_{35} = -\frac{x_{24}}{x_{14} x_{23} - x_{13} x_{24}},\\ 
	&a_{45} = \frac{x_{23}}{x_{14} x_{23} - x_{13} x_{24}}
\end{align*}

First assume $x_{14}= \frac{x_{13}x_{24}}{x_{23}}$
\begin{equation*}x=\left(\begin{matrix}
		1 & 0 & x_{13} & x_{14} & 0 \\
		0 & 1 & x_{23} & x_{24} & 0 \\
		0 & 0 & 1 & x_{34} & x_{35}  \\
		0 & 0 & 0 & 1 & 0  \\
		0 & 0 & 0 & 0 & 1 \\
	\end{matrix}\right),x^A=\left(\begin{matrix}
		1 & 0 & 0 & 0 & 0 \\
		0 & 1 & 1 & 0 & 0 \\
		0 & 0 & 1 & 1 & 0 \\
		0 & 0 & 0 & 1 & 0 \\
		0 & 0 & 0 & 0 & 1
	\end{matrix}\right)
\end{equation*} \newline
Where matrix $A$ has entries
\begin{align*}
	&d_{1} = 1, d_{2} = 1, d_{3} = \frac{1}{x_{23}}, d_{4} = \frac{1}{x_{23} x_{34}}, d_{5} = 1,\\ 
	&a_{12} = \frac{x_{13}}{x_{23}}, a_{13} = 1, a_{14} = 1, a_{15} = 1,\\ 
	&a_{23} = \frac{x_{23}}{x_{13}}, a_{24} = 1, a_{25} = 1,\\ 
	&a_{34} = \frac{x_{23}^{2} x_{34} - x_{13} x_{24}}{x_{13} x_{23}^{2} x_{34}}, a_{35} = \frac{x_{24} x_{35}}{x_{23} x_{34}},\\ 
	&a_{45} = -\frac{x_{35}}{x_{34}}
\end{align*}


First assume $x_{15} \neq \frac{(x_{14}x_{23}-x_{13}x_{24})x_{35}}{x_{34}x_{23}}$

\begin{equation*}x=\left(\begin{matrix}
		1 & 0 & x_{13} & x_{14} & x_{15} \\
		0 & 1 & x_{23} & x_{24} & 0 \\
		0 & 0 & 1 & x_{34} & x_{35}  \\
		0 & 0 & 0 & 1 & 0  \\
		0 & 0 & 0 & 0 & 1 \\
	\end{matrix}\right),x^A=\left(\begin{matrix}
		1 & 0 & 0 & 0 & 1 \\
		0 & 1 & 1 & 0 & 0 \\
		0 & 0 & 1 & 1 & 0 \\
		0 & 0 & 0 & 1 & 0 \\
		0 & 0 & 0 & 0 & 1
	\end{matrix}\right)
\end{equation*} \newline
Where matrix $A$ has entries
\begin{align*}
	&d_{1} = 1, d_{2} = 1, d_{3} = \frac{1}{x_{23}}, d_{4} = \frac{1}{x_{23} x_{34}}, d_{5} = \frac{x_{23} x_{34}}{x_{15} x_{23} x_{34} - {\left(x_{14} x_{23} - x_{13} x_{24}\right)} x_{35}},\\ 
	&a_{12} = \frac{x_{13}}{x_{23}}, a_{13} = 1, a_{14} = 1, a_{15} = 1,\\ 
	&a_{23} = \frac{x_{23}^{2} x_{34} - x_{14} x_{23} + x_{13} x_{24}}{x_{13} x_{23} x_{34}}, a_{24} = 1, a_{25} = 1,\\ 
	&a_{34} = \frac{x_{23} x_{34} - x_{14}}{x_{13} x_{23} x_{34}}, a_{35} = \frac{x_{24} x_{35}}{x_{15} x_{23} x_{34} - {\left(x_{14} x_{23} - x_{13} x_{24}\right)} x_{35}},\\ 
	&a_{45} = -\frac{x_{23} x_{35}}{x_{15} x_{23} x_{34} - {\left(x_{14} x_{23} - x_{13} x_{24}\right)} x_{35}}
\end{align*}

First assume $x_{15} = \frac{(x_{14}x_{23}-x_{13}x_{24})x_{35}}{x_{34}x_{23}}$

\begin{equation*}x=\left(
\right)
\end{equation*} \newline
Where matrix $A$ has entries
\begin{align*}
	&d_{1} = 1, d_{2} = 1, d_{3} = \frac{1}{x_{23}}, d_{4} = \frac{1}{x_{23} x_{34}}, d_{5} = -\frac{x_{34}}{x_{14} x_{35}},\\ 
	&a_{12} = 0, a_{13} = \frac{x_{14}}{x_{23} x_{34}}, a_{14} = 1, a_{15} = 1,\\ 
	&a_{23} = 1, a_{24} = 1, a_{25} = 1,\\ 
	&a_{34} = \frac{x_{23} x_{34} - x_{24}}{x_{23}^{2} x_{34}}, a_{35} = \frac{x_{25} x_{34} - x_{24} x_{35}}{x_{14} x_{23} x_{35}},\\ 
	&a_{45} = \frac{1}{x_{14}}
\end{align*}


First assume $x_{15}\neq \frac{x_{14}x_{35}}{x_{34}}$
\begin{equation*}x=\left(\begin{matrix}
		1 & 0 & 0 & x_{14} & x_{15} \\
		0 & 1 & x_{23} & x_{24} & x_{25} \\
		0 & 0 & 1 & x_{34} & x_{35}  \\
		0 & 0 & 0 & 1 & 0  \\
		0 & 0 & 0 & 0 & 1 \\
	\end{matrix}\right),x^A=\left(\begin{matrix}
		1 & 0 & 0 & 0 & 1 \\
		0 & 1 & 1 & 0 & 0 \\
		0 & 0 & 1 & 1 & 0 \\
		0 & 0 & 0 & 1 & 0 \\
		0 & 0 & 0 & 0 & 1
	\end{matrix}\right)
\end{equation*} \newline
Where matrix $A$ has entries
\begin{align*}
	&d_{1} = 1, d_{2} = 1, d_{3} = \frac{1}{x_{23}}, d_{4} = \frac{1}{x_{23} x_{34}}, d_{5} = \frac{x_{34}}{x_{15} x_{34} - x_{14} x_{35}},\\ 
	&a_{12} = 0, a_{13} = \frac{x_{14}}{x_{23} x_{34}}, a_{14} = 1, a_{15} = 1,\\ 
	&a_{23} = 1, a_{24} = 1, a_{25} = 1,\\ 
	&a_{34} = \frac{x_{23} x_{34} - x_{24}}{x_{23}^{2} x_{34}}, a_{35} = -\frac{x_{25} x_{34} - x_{24} x_{35}}{x_{15} x_{23} x_{34} - x_{14} x_{23} x_{35}},\\ 
	&a_{45} = -\frac{x_{35}}{x_{15} x_{34} - x_{14} x_{35}}
\end{align*}

Now assume $x_{15}= \frac{x_{14}x_{35}}{x_{34}}$
\begin{equation*}x=\left(\begin{matrix}
		1 & 0 & 0 & x_{14} & x_{15} \\
		0 & 1 & x_{23} & x_{24} & x_{25} \\
		0 & 0 & 1 & x_{34} & x_{35}  \\
		0 & 0 & 0 & 1 & 0  \\
		0 & 0 & 0 & 0 & 1 \\
	\end{matrix}\right),x^A=\left(\begin{matrix}
		1 & 0 & 0 & 0 & 0 \\
		0 & 1 & 1 & 0 & 0 \\
		0 & 0 & 1 & 1 & 0 \\
		0 & 0 & 0 & 1 & 0 \\
		0 & 0 & 0 & 0 & 1
	\end{matrix}\right)
\end{equation*} \newline
Where matrix $A$ has entries
\begin{align*}
	&d_{1} = 1, d_{2} = 1, d_{3} = \frac{1}{x_{23}}, d_{4} = \frac{1}{x_{23} x_{34}}, d_{5} = 1,\\ 
	&a_{12} = 0, a_{13} = \frac{x_{14}}{x_{23} x_{34}}, a_{14} = 1, a_{15} = 1,\\ 
	&a_{23} = 1, a_{24} = 1, a_{25} = 1,\\ 
	&a_{34} = \frac{x_{23} x_{34} - x_{24}}{x_{23}^{2} x_{34}}, a_{35} = -\frac{x_{25} x_{34} - x_{24} x_{35}}{x_{23} x_{34}},\\ 
	&a_{45} = -\frac{x_{35}}{x_{34}}
\end{align*}


First assume $x_{25}\neq \frac{x_{24}x_{35}}{x_{34}}$
\begin{equation*}x=\left(\begin{matrix}
		1 & 0 & x_{13} & 0 & 0 \\
		0 & 1 & x_{23} & x_{24} & x_{25} \\
		0 & 0 & 1 & x_{34} & x_{35}  \\
		0 & 0 & 0 & 1 & 0  \\
		0 & 0 & 0 & 0 & 1 \\
	\end{matrix}\right),x^A=\left(\begin{matrix}
		1 & 0 & 0 & 0 & 1 \\
		0 & 1 & 1 & 0 & 0 \\
		0 & 0 & 1 & 1 & 0 \\
		0 & 0 & 0 & 1 & 0 \\
		0 & 0 & 0 & 0 & 1
	\end{matrix}\right)
\end{equation*} \newline
Where matrix $A$ has entries
\begin{align*}
	&d_{1} = 1, d_{2} = 1, d_{3} = \frac{1}{x_{23}}, d_{4} = \frac{1}{x_{23} x_{34}}, d_{5} = -\frac{x_{23} x_{34}}{x_{13} x_{25} x_{34} - x_{13} x_{24} x_{35}},\\ 
	&a_{12} = \frac{x_{13}}{x_{23}}, a_{13} = 1, a_{14} = 1, a_{15} = 1,\\ 
	&a_{23} = \frac{x_{23}^{2} x_{34} + x_{13} x_{24}}{x_{13} x_{23} x_{34}}, a_{24} = 1, a_{25} = 1,\\ 
	&a_{34} = \frac{1}{x_{13}}, a_{35} = \frac{1}{x_{13}},\\ 
	&a_{45} = \frac{x_{23} x_{35}}{x_{13} x_{25} x_{34} - x_{13} x_{24} x_{35}}
\end{align*}

Now assume $x_{25}= \frac{x_{24}x_{35}}{x_{34}}$
\begin{equation*}x=\left(\begin{matrix}
		1 & 0 & x_{13} & 0 & 0 \\
		0 & 1 & x_{23} & x_{24} & x_{25} \\
		0 & 0 & 1 & x_{34} & x_{35}  \\
		0 & 0 & 0 & 1 & 0  \\
		0 & 0 & 0 & 0 & 1 \\
	\end{matrix}\right),x^A=\left(\begin{matrix}
		1 & 0 & 0 & 0 & 0 \\
		0 & 1 & 1 & 0 & 0 \\
		0 & 0 & 1 & 1 & 0 \\
		0 & 0 & 0 & 1 & 0 \\
		0 & 0 & 0 & 0 & 1
	\end{matrix}\right)
\end{equation*} \newline
Where matrix $A$ has entries
\begin{align*}
	&d_{1} = 1, d_{2} = 1, d_{3} = \frac{1}{x_{23}}, d_{4} = \frac{1}{x_{23} x_{34}}, d_{5} = 1,\\ 
	&a_{12} = \frac{x_{13}}{x_{23}}, a_{13} = 1, a_{14} = 1, a_{15} = 1,\\ 
	&a_{23} = \frac{x_{23}^{2} x_{34} + x_{13} x_{24}}{x_{13} x_{23} x_{34}}, a_{24} = 1, a_{25} = 1,\\ 
	&a_{34} = \frac{1}{x_{13}}, a_{35} = 0,\\ 
	&a_{45} = -\frac{x_{35}}{x_{34}}
\end{align*}


First assume $x_{15} \neq \frac{(x_{34}x_{25}-x_{24}x_{35})x_{13}}{x_{34}x_{23}}$
\begin{equation*}x=\left(\begin{matrix}
		1 & 0 & x_{13} & 0 & x_{15} \\
		0 & 1 & x_{23} & x_{24} & x_{25} \\
		0 & 0 & 1 & x_{34} & x_{35}  \\
		0 & 0 & 0 & 1 & 0  \\
		0 & 0 & 0 & 0 & 1 \\
	\end{matrix}\right),x^A=\left(\begin{matrix}
		1 & 0 & 0 & 0 & 1 \\
		0 & 1 & 1 & 0 & 0 \\
		0 & 0 & 1 & 1 & 0 \\
		0 & 0 & 0 & 1 & 0 \\
		0 & 0 & 0 & 0 & 1
	\end{matrix}\right)
\end{equation*} \newline
Where matrix $A$ has entries
\begin{align*}
	&d_{1} = 1, d_{2} = 1, d_{3} = \frac{1}{x_{23}}, d_{4} = \frac{1}{x_{23} x_{34}}, d_{5} = \frac{x_{23} x_{34}}{x_{13} x_{24} x_{35} + {\left(x_{15} x_{23} - x_{13} x_{25}\right)} x_{34}},\\ 
	&a_{12} = \frac{x_{13}}{x_{23}}, a_{13} = 1, a_{14} = 1, a_{15} = 1,\\ 
	&a_{23} = \frac{x_{23}^{2} x_{34} + x_{13} x_{24}}{x_{13} x_{23} x_{34}}, a_{24} = 1, a_{25} = 1,\\ 
	&a_{34} = \frac{1}{x_{13}}, a_{35} = -\frac{x_{25} x_{34} - x_{24} x_{35}}{x_{13} x_{24} x_{35} + {\left(x_{15} x_{23} - x_{13} x_{25}\right)} x_{34}},\\ 
	&a_{45} = -\frac{x_{23} x_{35}}{x_{13} x_{24} x_{35} + {\left(x_{15} x_{23} - x_{13} x_{25}\right)} x_{34}}
\end{align*}

Now assume $x_{15} = \frac{(x_{34}x_{25}-x_{24}x_{35})x_{13}}{x_{34}x_{23}}$
\begin{equation*}x=\left(\begin{matrix}
		1 & 0 & x_{13} & 0 & x_{15} \\
		0 & 1 & x_{23} & x_{24} & x_{25} \\
		0 & 0 & 1 & x_{34} & x_{35}  \\
		0 & 0 & 0 & 1 & 0  \\
		0 & 0 & 0 & 0 & 1 \\
	\end{matrix}\right),x^A=\left(\begin{matrix}
		1 & 0 & 0 & 0 & 0 \\
		0 & 1 & 1 & 0 & 0 \\
		0 & 0 & 1 & 1 & 0 \\
		0 & 0 & 0 & 1 & 0 \\
		0 & 0 & 0 & 0 & 1
	\end{matrix}\right)
\end{equation*} \newline
Where matrix $A$ has entries
\begin{align*}
	&d_{1} = 1, d_{2} = 1, d_{3} = \frac{1}{x_{23}}, d_{4} = \frac{1}{x_{23} x_{34}}, d_{5} = 1,\\ 
	&a_{12} = \frac{x_{13}}{x_{23}}, a_{13} = 1, a_{14} = 1, a_{15} = 1,\\ 
	&a_{23} = \frac{x_{23}^{2} x_{34} + x_{13} x_{24}}{x_{13} x_{23} x_{34}}, a_{24} = 1, a_{25} = 1,\\ 
	&a_{34} = \frac{1}{x_{13}}, a_{35} = -\frac{x_{25} x_{34} - x_{24} x_{35}}{x_{23} x_{34}},\\ 
	&a_{45} = -\frac{x_{35}}{x_{34}}
\end{align*}


First assume $x_{14} \neq \frac{(-x_{34}x_{25}+x_{24}x_{35})x_{13}}{x_{35}x_{23}}$
\begin{equation*}x=\left(\begin{matrix}
		1 & 0 & x_{13} & x_{14} & 0 \\
		0 & 1 & x_{23} & x_{24} & x_{25} \\
		0 & 0 & 1 & x_{34} & x_{35}  \\
		0 & 0 & 0 & 1 & 0  \\
		0 & 0 & 0 & 0 & 1 \\
	\end{matrix}\right),x^A=\left(\begin{matrix}
		1 & 0 & 0 & 0 & 1 \\
		0 & 1 & 1 & 0 & 0 \\
		0 & 0 & 1 & 1 & 0 \\
		0 & 0 & 0 & 1 & 0 \\
		0 & 0 & 0 & 0 & 1
	\end{matrix}\right)
\end{equation*} \newline
Where matrix $A$ has entries
\begin{align*}
	&d_{1} = 1, d_{2} = 1, d_{3} = \frac{1}{x_{23}}, d_{4} = \frac{1}{x_{23} x_{34}}, d_{5} = -\frac{x_{23} x_{34}}{x_{13} x_{25} x_{34} + {\left(x_{14} x_{23} - x_{13} x_{24}\right)} x_{35}},\\ 
	&a_{12} = \frac{x_{13}}{x_{23}}, a_{13} = 1, a_{14} = 1, a_{15} = 1,\\ 
	&a_{23} = \frac{x_{23}^{2} x_{34} - x_{14} x_{23} + x_{13} x_{24}}{x_{13} x_{23} x_{34}}, a_{24} = 1, a_{25} = 1,\\ 
	&a_{34} = \frac{x_{23} x_{34} - x_{14}}{x_{13} x_{23} x_{34}}, a_{35} = \frac{x_{25} x_{34} - x_{24} x_{35}}{x_{13} x_{25} x_{34} + {\left(x_{14} x_{23} - x_{13} x_{24}\right)} x_{35}},\\ 
	&a_{45} = \frac{x_{23} x_{35}}{x_{13} x_{25} x_{34} + {\left(x_{14} x_{23} - x_{13} x_{24}\right)} x_{35}}
\end{align*}

Now assume $x_{14} = \frac{(-x_{34}x_{25}+x_{24}x_{35})x_{13}}{x_{35}x_{23}}$
\begin{equation*}x=\left(\begin{matrix}
		1 & 0 & x_{13} & x_{14} & 0 \\
		0 & 1 & x_{23} & x_{24} & x_{25} \\
		0 & 0 & 1 & x_{34} & x_{35}  \\
		0 & 0 & 0 & 1 & 0  \\
		0 & 0 & 0 & 0 & 1 \\
	\end{matrix}\right),x^A=\left(\begin{matrix}
		1 & 0 & 0 & 0 & 0 \\
		0 & 1 & 1 & 0 & 0 \\
		0 & 0 & 1 & 1 & 0 \\
		0 & 0 & 0 & 1 & 0 \\
		0 & 0 & 0 & 0 & 1
	\end{matrix}\right)
\end{equation*} \newline
Where matrix $A$ has entries
\begin{align*}
	&d_{1} = 1, d_{2} = 1, d_{3} = \frac{1}{x_{23}}, d_{4} = \frac{1}{x_{23} x_{34}}, d_{5} = 1,\\ 
	&a_{12} = \frac{x_{13}}{x_{23}}, a_{13} = 1, a_{14} = 1, a_{15} = 1,\\ 
	&a_{23} = \frac{x_{23}^{2} x_{35} + x_{13} x_{25}}{x_{13} x_{23} x_{35}}, a_{24} = 1, a_{25} = 1,\\ 
	&a_{34} = \frac{x_{13} x_{25} x_{34} + {\left(x_{23}^{2} x_{34} - x_{13} x_{24}\right)} x_{35}}{x_{13} x_{23}^{2} x_{34} x_{35}}, a_{35} = -\frac{x_{25} x_{34} - x_{24} x_{35}}{x_{23} x_{34}},\\ 
	&a_{45} = -\frac{x_{35}}{x_{34}}
\end{align*}


First assume $x_{15}\neq \frac{x_{13}x_{25}}{x_{23}}$ and $x_{34} \neq \frac{(x_{14}x_{23}-x_{13}x_{24})x_{35}}{x_{15}x_{23}-x_{13}x_{25}}$ 
\begin{equation*}x=\left(\begin{matrix}
		1 & 0 & x_{13} & x_{14} & x_{15} \\
		0 & 1 & x_{23} & x_{24} & x_{25} \\
		0 & 0 & 1 & x_{34} & x_{35}  \\
		0 & 0 & 0 & 1 & 0  \\
		0 & 0 & 0 & 0 & 1 \\
	\end{matrix}\right),x^A=\left(\begin{matrix}
		1 & 0 & 0 & 0 & 1 \\
		0 & 1 & 1 & 0 & 0 \\
		0 & 0 & 1 & 1 & 0 \\
		0 & 0 & 0 & 1 & 0 \\
		0 & 0 & 0 & 0 & 1
	\end{matrix}\right)
\end{equation*} \newline
Where matrix $A$ has entries
\begin{align*}
	&d_{1} = 1, d_{2} = 1, d_{3} = \frac{1}{x_{23}}, d_{4} = \frac{1}{x_{23} x_{34}}, d_{5} = \frac{x_{23} x_{34}}{{\left(x_{15} x_{23} - x_{13} x_{25}\right)} x_{34} - {\left(x_{14} x_{23} - x_{13} x_{24}\right)} x_{35}},\\ 
	&a_{12} = \frac{x_{13}}{x_{23}}, a_{13} = 1, a_{14} = 1, a_{15} = 1,\\ 
	&a_{23} = \frac{x_{23}^{2} x_{34} - x_{14} x_{23} + x_{13} x_{24}}{x_{13} x_{23} x_{34}}, a_{24} = 1, a_{25} = 1,\\ 
	&a_{34} = \frac{x_{23} x_{34} - x_{14}}{x_{13} x_{23} x_{34}}, a_{35} = -\frac{x_{25} x_{34} - x_{24} x_{35}}{{\left(x_{15} x_{23} - x_{13} x_{25}\right)} x_{34} - {\left(x_{14} x_{23} - x_{13} x_{24}\right)} x_{35}},\\ 
	&a_{45} = -\frac{x_{23} x_{35}}{{\left(x_{15} x_{23} - x_{13} x_{25}\right)} x_{34} - {\left(x_{14} x_{23} - x_{13} x_{24}\right)} x_{35}}
\end{align*}

Now assume $x_{15}= \frac{x_{13}x_{25}}{x_{23}}$ and $x_{14}\neq \frac{x_{13}x_{24}}{x_{23}}$
\begin{equation*}x=\left(\begin{matrix}
		1 & 0 & x_{13} & x_{14} & x_{15} \\
		0 & 1 & x_{23} & x_{24} & x_{25} \\
		0 & 0 & 1 & x_{34} & x_{35}  \\
		0 & 0 & 0 & 1 & 0  \\
		0 & 0 & 0 & 0 & 1 \\
	\end{matrix}\right),x^A=\left(\begin{matrix}
		1 & 0 & 0 & 0 & 1 \\
		0 & 1 & 1 & 0 & 0 \\
		0 & 0 & 1 & 1 & 0 \\
		0 & 0 & 0 & 1 & 0 \\
		0 & 0 & 0 & 0 & 1
	\end{matrix}\right)
\end{equation*} \newline
Where matrix $A$ has entries
\begin{align*}
	&d_{1} = 1, d_{2} = 1, d_{3} = \frac{1}{x_{23}}, d_{4} = \frac{1}{x_{23} x_{34}}, d_{5} = -\frac{x_{23} x_{34}}{{\left(x_{14} x_{23} - x_{13} x_{24}\right)} x_{35}},\\ 
	&a_{12} = \frac{x_{13}}{x_{23}}, a_{13} = 1, a_{14} = 1, a_{15} = 1,\\ 
	&a_{23} = \frac{x_{23}^{2} x_{34} - x_{14} x_{23} + x_{13} x_{24}}{x_{13} x_{23} x_{34}}, a_{24} = 1, a_{25} = 1,\\ 
	&a_{34} = \frac{x_{23} x_{34} - x_{14}}{x_{13} x_{23} x_{34}}, a_{35} = \frac{x_{25} x_{34} - x_{24} x_{35}}{{\left(x_{14} x_{23} - x_{13} x_{24}\right)} x_{35}},\\ 
	&a_{45} = \frac{x_{23}}{x_{14} x_{23} - x_{13} x_{24}}
\end{align*}

Now assume $x_{15}=\frac{x_{13}x_{25}}{x_{23}}$ and $x_{14}= \frac{x_{13}x_{24}}{x_{23}}$
\begin{equation*}x=\left(\begin{matrix}
		1 & 0 & x_{13} & x_{14} & x_{15} \\
		0 & 1 & x_{23} & x_{24} & x_{25} \\
		0 & 0 & 1 & x_{34} & x_{35}  \\
		0 & 0 & 0 & 1 & 0  \\
		0 & 0 & 0 & 0 & 1 \\
	\end{matrix}\right),x^A=\left(\begin{matrix}
		1 & 0 & 0 & 0 & 0 \\
		0 & 1 & 1 & 0 & 0 \\
		0 & 0 & 1 & 1 & 0 \\
		0 & 0 & 0 & 1 & 0 \\
		0 & 0 & 0 & 0 & 1
	\end{matrix}\right)
\end{equation*} \newline
Where matrix $A$ has entries
\begin{align*}
	&d_{1} = 1, d_{2} = 1, d_{3} = \frac{1}{x_{23}}, d_{4} = \frac{1}{x_{23} x_{34}}, d_{5} = 1,\\ 
	&a_{12} = \frac{x_{13}}{x_{23}}, a_{13} = 1, a_{14} = 1, a_{15} = 1,\\ 
	&a_{23} = \frac{x_{23}}{x_{13}}, a_{24} = 1, a_{25} = 1,\\ 
	&a_{34} = \frac{x_{23}^{2} x_{34} - x_{13} x_{24}}{x_{13} x_{23}^{2} x_{34}}, a_{35} = -\frac{x_{25} x_{34} - x_{24} x_{35}}{x_{23} x_{34}},\\ 
	&a_{45} = -\frac{x_{35}}{x_{34}}
\end{align*}

Now assume $x_{15}\neq \frac{x_{13}x_{25}}{x_{23}}$ and $x_{34} = \frac{(x_{14}x_{23}-x_{13}x_{24})x_{35}}{x_{15}x_{23}-x_{13}x_{25}}$  and $x_{14}\neq \frac{x_{13}x_{24}}{x_{23}}$
\begin{equation*}x=\left(\begin{matrix}
		1 & 0 & x_{13} & x_{14} & x_{15} \\
		0 & 1 & x_{23} & x_{24} & x_{25} \\
		0 & 0 & 1 & x_{34} & x_{35}  \\
		0 & 0 & 0 & 1 & 0  \\
		0 & 0 & 0 & 0 & 1 \\
	\end{matrix}\right),x^A=\left(\begin{matrix}
		1 & 0 & 0 & 0 & 0 \\
		0 & 1 & 1 & 0 & 0 \\
		0 & 0 & 1 & 1 & 0 \\
		0 & 0 & 0 & 1 & 0 \\
		0 & 0 & 0 & 0 & 1
	\end{matrix}\right)
\end{equation*} \newline
Where matrix $A$ has entries
\begin{align*}
	&d_{1} = 1, d_{2} = 1, d_{3} = \frac{1}{x_{23}}, d_{4} = \frac{x_{15} x_{23} - x_{13} x_{25}}{{\left(x_{14} x_{23}^{2} - x_{13} x_{23} x_{24}\right)} x_{35}}, d_{5} = 1,\\ 
	&a_{12} = \frac{x_{13}}{x_{23}}, a_{13} = 1, a_{14} = 1, a_{15} = 1,\\ 
	&a_{23} = \frac{x_{23}^{2} x_{35} - x_{15} x_{23} + x_{13} x_{25}}{x_{13} x_{23} x_{35}}, a_{24} = 1, a_{25} = 1,\\ 
	&a_{34} = -\frac{x_{14} x_{15} x_{23} - x_{13} x_{14} x_{25} - {\left(x_{14} x_{23}^{2} - x_{13} x_{23} x_{24}\right)} x_{35}}{{\left(x_{13} x_{14} x_{23}^{2} - x_{13}^{2} x_{23} x_{24}\right)} x_{35}}, a_{35} = \frac{x_{15} x_{24} - x_{14} x_{25}}{x_{14} x_{23} - x_{13} x_{24}},\\ 
	&a_{45} = -\frac{x_{15} x_{23} - x_{13} x_{25}}{x_{14} x_{23} - x_{13} x_{24}}
\end{align*}

These are the only possible cases, given that $x_{ij}\neq 0$.
		\section{Subcases of $Y_7$}

\begin{equation*}x=\left(
\right)
\end{equation*} \newline
Where matrix $A$ has entries
\begin{align*}
	&d_{1} = 1, d_{2} = \frac{1}{x_{12}}, d_{3} = \frac{1}{x_{12} x_{23}}, d_{4} = \frac{1}{x_{12} x_{23} x_{34}}, d_{5} = 1,\\ 
	&a_{12} = 1, a_{13} = 1, a_{14} = 1, a_{15} = 1,\\ 
	&a_{23} = \frac{x_{12} x_{23} - x_{13}}{x_{12}^{2} x_{23}}, a_{24} = \frac{x_{12}^{2} x_{23}^{2} - x_{12} x_{13} x_{23} + x_{13}^{2}}{x_{12}^{3} x_{23}^{2}}, a_{25} = -\frac{x_{15}}{x_{12}},\\ 
	&a_{34} = \frac{x_{12} x_{23} - x_{13}}{x_{12}^{2} x_{23}^{2}}, a_{35} = 0,\\ 
	&a_{45} = 0
\end{align*}

\begin{equation*}x=\left(\begin{matrix}
		1 & x_{12} & x_{13} & x_{14} & 0 \\
		0 & 1 & x_{23} & 0 & 0 \\
		0 & 0 & 1 & x_{34} & 0  \\
		0 & 0 & 0 & 1 & 0  \\
		0 & 0 & 0 & 0 & 1 \\
	\end{matrix}\right),x^A=\left(\begin{matrix}
		1 & 1 & 0 & 0 & 0 \\
		0 & 1 & 1 & 0 & 0 \\
		0 & 0 & 1 & 1 & 0 \\
		0 & 0 & 0 & 1 & 0 \\
		0 & 0 & 0 & 0 & 1
	\end{matrix}\right)
\end{equation*} \newline
Where matrix $A$ has entries
\begin{align*}
	&d_{1} = 1, d_{2} = \frac{1}{x_{12}}, d_{3} = \frac{1}{x_{12} x_{23}}, d_{4} = \frac{1}{x_{12} x_{23} x_{34}}, d_{5} = 1,\\ 
	&a_{12} = 1, a_{13} = 1, a_{14} = 1, a_{15} = 1,\\ 
	&a_{23} = \frac{x_{12} x_{23} - x_{13}}{x_{12}^{2} x_{23}}, a_{24} = -\frac{x_{12} x_{14} x_{23} - {\left(x_{12}^{2} x_{23}^{2} - x_{12} x_{13} x_{23} + x_{13}^{2}\right)} x_{34}}{x_{12}^{3} x_{23}^{2} x_{34}}, a_{25} = 0,\\ 
	&a_{34} = \frac{x_{12} x_{23} - x_{13}}{x_{12}^{2} x_{23}^{2}}, a_{35} = 0,\\ 
	&a_{45} = 0
\end{align*}

\begin{equation*}x=\left(\begin{matrix}
		1 & x_{12} & x_{13} & x_{14} & x_{15} \\
		0 & 1 & x_{23} & 0 & 0 \\
		0 & 0 & 1 & x_{34} & 0  \\
		0 & 0 & 0 & 1 & 0  \\
		0 & 0 & 0 & 0 & 1 \\
	\end{matrix}\right),x^A=\left(\begin{matrix}
		1 & 1 & 0 & 0 & 0 \\
		0 & 1 & 1 & 0 & 0 \\
		0 & 0 & 1 & 1 & 0 \\
		0 & 0 & 0 & 1 & 0 \\
		0 & 0 & 0 & 0 & 1
	\end{matrix}\right)
\end{equation*} \newline
Where matrix $A$ has entries
\begin{align*}
	&d_{1} = 1, d_{2} = \frac{1}{x_{12}}, d_{3} = \frac{1}{x_{12} x_{23}}, d_{4} = \frac{1}{x_{12} x_{23} x_{34}}, d_{5} = 1,\\ 
	&a_{12} = 1, a_{13} = 1, a_{14} = 1, a_{15} = 1,\\ 
	&a_{23} = \frac{x_{12} x_{23} - x_{13}}{x_{12}^{2} x_{23}}, a_{24} = -\frac{x_{12} x_{14} x_{23} - {\left(x_{12}^{2} x_{23}^{2} - x_{12} x_{13} x_{23} + x_{13}^{2}\right)} x_{34}}{x_{12}^{3} x_{23}^{2} x_{34}}, a_{25} = -\frac{x_{15}}{x_{12}},\\ 
	&a_{34} = \frac{x_{12} x_{23} - x_{13}}{x_{12}^{2} x_{23}^{2}}, a_{35} = 0,\\ 
	&a_{45} = 0
\end{align*}

\begin{equation*}x=\left(
\right)
\end{equation*} \newline
Where matrix $A$ has entries
\begin{align*}
	&d_{1} = 1, d_{2} = \frac{1}{x_{12}}, d_{3} = \frac{1}{x_{12} x_{23}}, d_{4} = \frac{1}{x_{12} x_{23} x_{34}}, d_{5} = 1,\\ 
	&a_{12} = 1, a_{13} = 1, a_{14} = 1, a_{15} = 1,\\ 
	&a_{23} = \frac{x_{12} x_{23} - x_{13}}{x_{12}^{2} x_{23}}, a_{24} = \frac{x_{12}^{2} x_{23}^{2} - x_{12} x_{13} x_{23} + x_{13}^{2}}{x_{12}^{3} x_{23}^{2}}, a_{25} = \frac{x_{13} x_{25}}{x_{12} x_{23}},\\ 
	&a_{34} = \frac{x_{12} x_{23} - x_{13}}{x_{12}^{2} x_{23}^{2}}, a_{35} = -\frac{x_{25}}{x_{23}},\\ 
	&a_{45} = 0
\end{align*}

\begin{equation*}x=\left(\begin{matrix}
		1 & x_{12} & x_{13} & 0 & x_{15} \\
		0 & 1 & x_{23} & 0 & x_{25} \\
		0 & 0 & 1 & x_{34} & 0  \\
		0 & 0 & 0 & 1 & 0  \\
		0 & 0 & 0 & 0 & 1 \\
	\end{matrix}\right),x^A=\left(\begin{matrix}
		1 & 1 & 0 & 0 & 0 \\
		0 & 1 & 1 & 0 & 0 \\
		0 & 0 & 1 & 1 & 0 \\
		0 & 0 & 0 & 1 & 0 \\
		0 & 0 & 0 & 0 & 1
	\end{matrix}\right)
\end{equation*} \newline
Where matrix $A$ has entries
\begin{align*}
	&d_{1} = 1, d_{2} = \frac{1}{x_{12}}, d_{3} = \frac{1}{x_{12} x_{23}}, d_{4} = \frac{1}{x_{12} x_{23} x_{34}}, d_{5} = 1,\\ 
	&a_{12} = 1, a_{13} = 1, a_{14} = 1, a_{15} = 1,\\ 
	&a_{23} = \frac{x_{12} x_{23} - x_{13}}{x_{12}^{2} x_{23}}, a_{24} = \frac{x_{12}^{2} x_{23}^{2} - x_{12} x_{13} x_{23} + x_{13}^{2}}{x_{12}^{3} x_{23}^{2}}, a_{25} = -\frac{x_{15} x_{23} - x_{13} x_{25}}{x_{12} x_{23}},\\ 
	&a_{34} = \frac{x_{12} x_{23} - x_{13}}{x_{12}^{2} x_{23}^{2}}, a_{35} = -\frac{x_{25}}{x_{23}},\\ 
	&a_{45} = 0
\end{align*}

\begin{equation*}x=\left(\begin{matrix}
		1 & x_{12} & x_{13} & x_{14} & 0 \\
		0 & 1 & x_{23} & 0 & x_{25} \\
		0 & 0 & 1 & x_{34} & 0  \\
		0 & 0 & 0 & 1 & 0  \\
		0 & 0 & 0 & 0 & 1 \\
	\end{matrix}\right),x^A=\left(\begin{matrix}
		1 & 1 & 0 & 0 & 0 \\
		0 & 1 & 1 & 0 & 0 \\
		0 & 0 & 1 & 1 & 0 \\
		0 & 0 & 0 & 1 & 0 \\
		0 & 0 & 0 & 0 & 1 
	\end{matrix}\right)
\end{equation*} \newline
Where matrix $A$ has entries
\begin{align*}
	&d_{1} = 1, d_{2} = \frac{1}{x_{12}}, d_{3} = \frac{1}{x_{12} x_{23}}, d_{4} = \frac{1}{x_{12} x_{23} x_{34}}, d_{5} = 1,\\ 
	&a_{12} = 1, a_{13} = 1, a_{14} = 1, a_{15} = 1,\\ 
	&a_{23} = \frac{x_{12} x_{23} - x_{13}}{x_{12}^{2} x_{23}}, a_{24} = -\frac{x_{12} x_{14} x_{23} - {\left(x_{12}^{2} x_{23}^{2} - x_{12} x_{13} x_{23} + x_{13}^{2}\right)} x_{34}}{x_{12}^{3} x_{23}^{2} x_{34}}, a_{25} = \frac{x_{13} x_{25}}{x_{12} x_{23}},\\ 
	&a_{34} = \frac{x_{12} x_{23} - x_{13}}{x_{12}^{2} x_{23}^{2}}, a_{35} = -\frac{x_{25}}{x_{23}},\\ 
	&a_{45} = 0
\end{align*}

\begin{equation*}x=\left(\begin{matrix}
		1 & x_{12} & x_{13} & x_{14} & x_{15} \\
		0 & 1 & x_{23} & 0 & x_{25} \\
		0 & 0 & 1 & x_{34} & 0  \\
		0 & 0 & 0 & 1 & 0  \\
		0 & 0 & 0 & 0 & 1 \\
	\end{matrix}\right),x^A=\left(\begin{matrix}
		1 & 1 & 0 & 0 & 0 \\
		0 & 1 & 1 & 0 & 0 \\
		0 & 0 & 1 & 1 & 0 \\
		0 & 0 & 0 & 1 & 0 \\
		0 & 0 & 0 & 0 & 1
	\end{matrix}\right)
\end{equation*} \newline
Where matrix $A$ has entries
\begin{align*}
	&d_{1} = 1, d_{2} = \frac{1}{x_{12}}, d_{3} = \frac{1}{x_{12} x_{23}}, d_{4} = \frac{1}{x_{12} x_{23} x_{34}}, d_{5} = 1,\\ 
	&a_{12} = 1, a_{13} = 1, a_{14} = 1, a_{15} = 1,\\ 
	&a_{23} = \frac{x_{12} x_{23} - x_{13}}{x_{12}^{2} x_{23}}, a_{24} = -\frac{x_{12} x_{14} x_{23} - {\left(x_{12}^{2} x_{23}^{2} - x_{12} x_{13} x_{23} + x_{13}^{2}\right)} x_{34}}{x_{12}^{3} x_{23}^{2} x_{34}}, a_{25} = -\frac{x_{15} x_{23} - x_{13} x_{25}}{x_{12} x_{23}},\\ 
	&a_{34} = \frac{x_{12} x_{23} - x_{13}}{x_{12}^{2} x_{23}^{2}}, a_{35} = -\frac{x_{25}}{x_{23}},\\ 
	&a_{45} = 0
\end{align*}

\begin{equation*}x=\left(
\right)
\end{equation*} \newline
Where matrix $A$ has entries
\begin{align*}
	&d_{1} = 1, d_{2} = \frac{1}{x_{12}}, d_{3} = \frac{1}{x_{12} x_{23}}, d_{4} = \frac{1}{x_{12} x_{23} x_{34}}, d_{5} = 1,\\ 
	&a_{12} = 1, a_{13} = 1, a_{14} = 1, a_{15} = 1,\\ 
	&a_{23} = \frac{x_{12} x_{23} - x_{13}}{x_{12}^{2} x_{23}}, a_{24} = \frac{x_{12} x_{13} x_{24} + {\left(x_{12}^{2} x_{23}^{2} - x_{12} x_{13} x_{23} + x_{13}^{2}\right)} x_{34}}{x_{12}^{3} x_{23}^{2} x_{34}}, a_{25} = 0,\\ 
	&a_{34} = -\frac{x_{12} x_{24} - {\left(x_{12} x_{23} - x_{13}\right)} x_{34}}{x_{12}^{2} x_{23}^{2} x_{34}}, a_{35} = 0,\\ 
	&a_{45} = 0
\end{align*}

\begin{equation*}x=\left(\begin{matrix}
		1 & x_{12} & x_{13} & 0 & x_{15} \\
		0 & 1 & x_{23} & x_{24} & 0 \\
		0 & 0 & 1 & x_{34} & 0  \\
		0 & 0 & 0 & 1 & 0  \\
		0 & 0 & 0 & 0 & 1 \\
	\end{matrix}\right),x^A=\left(\begin{matrix}
		1 & 1 & 0 & 0 & 0 \\
		0 & 1 & 1 & 0 & 0 \\
		0 & 0 & 1 & 1 & 0 \\
		0 & 0 & 0 & 1 & 0 \\
		0 & 0 & 0 & 0 & 1
	\end{matrix}\right)
\end{equation*} \newline
Where matrix $A$ has entries
\begin{align*}
	&d_{1} = 1, d_{2} = \frac{1}{x_{12}}, d_{3} = \frac{1}{x_{12} x_{23}}, d_{4} = \frac{1}{x_{12} x_{23} x_{34}}, d_{5} = 1,\\ 
	&a_{12} = 1, a_{13} = 1, a_{14} = 1, a_{15} = 1,\\ 
	&a_{23} = \frac{x_{12} x_{23} - x_{13}}{x_{12}^{2} x_{23}}, a_{24} = \frac{x_{12} x_{13} x_{24} + {\left(x_{12}^{2} x_{23}^{2} - x_{12} x_{13} x_{23} + x_{13}^{2}\right)} x_{34}}{x_{12}^{3} x_{23}^{2} x_{34}}, a_{25} = -\frac{x_{15}}{x_{12}},\\ 
	&a_{34} = -\frac{x_{12} x_{24} - {\left(x_{12} x_{23} - x_{13}\right)} x_{34}}{x_{12}^{2} x_{23}^{2} x_{34}}, a_{35} = 0,\\ 
	&a_{45} = 0
\end{align*}

\begin{equation*}x=\left(\begin{matrix}
		1 & x_{12} & x_{13} & x_{14} & 0 \\
		0 & 1 & x_{23} & x_{24} & 0 \\
		0 & 0 & 1 & x_{34} & 0  \\
		0 & 0 & 0 & 1 & 0  \\
		0 & 0 & 0 & 0 & 1 \\
	\end{matrix}\right),x^A=\left(\begin{matrix}
		1 & 1 & 0 & 0 & 0 \\
		0 & 1 & 1 & 0 & 0 \\
		0 & 0 & 1 & 1 & 0 \\
		0 & 0 & 0 & 1 & 0 \\
		0 & 0 & 0 & 0 & 1
	\end{matrix}\right)
\end{equation*} \newline
Where matrix $A$ has entries
\begin{align*}
	&d_{1} = 1, d_{2} = \frac{1}{x_{12}}, d_{3} = \frac{1}{x_{12} x_{23}}, d_{4} = \frac{1}{x_{12} x_{23} x_{34}}, d_{5} = 1,\\ 
	&a_{12} = 1, a_{13} = 1, a_{14} = 1, a_{15} = 1,\\ 
	&a_{23} = \frac{x_{12} x_{23} - x_{13}}{x_{12}^{2} x_{23}}, a_{24} = -\frac{x_{12} x_{14} x_{23} - x_{12} x_{13} x_{24} - {\left(x_{12}^{2} x_{23}^{2} - x_{12} x_{13} x_{23} + x_{13}^{2}\right)} x_{34}}{x_{12}^{3} x_{23}^{2} x_{34}}, a_{25} = 0,\\ 
	&a_{34} = -\frac{x_{12} x_{24} - {\left(x_{12} x_{23} - x_{13}\right)} x_{34}}{x_{12}^{2} x_{23}^{2} x_{34}}, a_{35} = 0,\\ 
	&a_{45} = 0
\end{align*}

\begin{equation*}x=\left(\begin{matrix}
		1 & x_{12} & x_{13} & x_{14} & x_{15} \\
		0 & 1 & x_{23} & x_{24} & 0 \\
		0 & 0 & 1 & x_{34} & 0  \\
		0 & 0 & 0 & 1 & 0  \\
		0 & 0 & 0 & 0 & 1 \\
	\end{matrix}\right),x^A=\left(\begin{matrix}
		1 & 1 & 0 & 0 & 0 \\
		0 & 1 & 1 & 0 & 0 \\
		0 & 0 & 1 & 1 & 0 \\
		0 & 0 & 0 & 1 & 0 \\
		0 & 0 & 0 & 0 & 1
	\end{matrix}\right)
\end{equation*} \newline
Where matrix $A$ has entries
\begin{align*}
	&d_{1} = 1, d_{2} = \frac{1}{x_{12}}, d_{3} = \frac{1}{x_{12} x_{23}}, d_{4} = \frac{1}{x_{12} x_{23} x_{34}}, d_{5} = 1,\\ 
	&a_{12} = 1, a_{13} = 1, a_{14} = 1, a_{15} = 1,\\ 
	&a_{23} = \frac{x_{12} x_{23} - x_{13}}{x_{12}^{2} x_{23}}, a_{24} = -\frac{x_{12} x_{14} x_{23} - x_{12} x_{13} x_{24} - {\left(x_{12}^{2} x_{23}^{2} - x_{12} x_{13} x_{23} + x_{13}^{2}\right)} x_{34}}{x_{12}^{3} x_{23}^{2} x_{34}}, a_{25} = -\frac{x_{15}}{x_{12}},\\ 
	&a_{34} = -\frac{x_{12} x_{24} - {\left(x_{12} x_{23} - x_{13}\right)} x_{34}}{x_{12}^{2} x_{23}^{2} x_{34}}, a_{35} = 0,\\ 
	&a_{45} = 0
\end{align*}

\begin{equation*}x=\left(
\right)
\end{equation*} \newline
Where matrix $A$ has entries
\begin{align*}
	&d_{1} = 1, d_{2} = \frac{1}{x_{12}}, d_{3} = \frac{1}{x_{12} x_{23}}, d_{4} = \frac{1}{x_{12} x_{23} x_{34}}, d_{5} = 1,\\ 
	&a_{12} = 1, a_{13} = 1, a_{14} = 1, a_{15} = 1,\\ 
	&a_{23} = \frac{x_{12} x_{23} - x_{13}}{x_{12}^{2} x_{23}}, a_{24} = \frac{x_{12} x_{13} x_{24} + {\left(x_{12}^{2} x_{23}^{2} - x_{12} x_{13} x_{23} + x_{13}^{2}\right)} x_{34}}{x_{12}^{3} x_{23}^{2} x_{34}}, a_{25} = \frac{x_{13} x_{25}}{x_{12} x_{23}},\\ 
	&a_{34} = -\frac{x_{12} x_{24} - {\left(x_{12} x_{23} - x_{13}\right)} x_{34}}{x_{12}^{2} x_{23}^{2} x_{34}}, a_{35} = -\frac{x_{25}}{x_{23}},\\ 
	&a_{45} = 0
\end{align*}

\begin{equation*}x=\left(\begin{matrix}
		1 & x_{12} & x_{13} & 0 & x_{15} \\
		0 & 1 & x_{23} & x_{24} & x_{25} \\
		0 & 0 & 1 & x_{34} & 0  \\
		0 & 0 & 0 & 1 & 0  \\
		0 & 0 & 0 & 0 & 1 \\
	\end{matrix}\right),x^A=\left(\begin{matrix}
		1 & 1 & 0 & 0 & 0 \\
		0 & 1 & 1 & 0 & 0 \\
		0 & 0 & 1 & 1 & 0 \\
		0 & 0 & 0 & 1 & 0 \\
		0 & 0 & 0 & 0 & 1
	\end{matrix}\right)
\end{equation*} \newline
Where matrix $A$ has entries
\begin{align*}
	&d_{1} = 1, d_{2} = \frac{1}{x_{12}}, d_{3} = \frac{1}{x_{12} x_{23}}, d_{4} = \frac{1}{x_{12} x_{23} x_{34}}, d_{5} = 1,\\ 
	&a_{12} = 1, a_{13} = 1, a_{14} = 1, a_{15} = 1,\\ 
	&a_{23} = \frac{x_{12} x_{23} - x_{13}}{x_{12}^{2} x_{23}}, a_{24} = \frac{x_{12} x_{13} x_{24} + {\left(x_{12}^{2} x_{23}^{2} - x_{12} x_{13} x_{23} + x_{13}^{2}\right)} x_{34}}{x_{12}^{3} x_{23}^{2} x_{34}}, a_{25} = -\frac{x_{15} x_{23} - x_{13} x_{25}}{x_{12} x_{23}},\\ 
	&a_{34} = -\frac{x_{12} x_{24} - {\left(x_{12} x_{23} - x_{13}\right)} x_{34}}{x_{12}^{2} x_{23}^{2} x_{34}}, a_{35} = -\frac{x_{25}}{x_{23}},\\ 
	&a_{45} = 0
\end{align*}

\begin{equation*}x=\left(\begin{matrix}
		1 & x_{12} & x_{13} & x_{14} & 0 \\
		0 & 1 & x_{23} & x_{24} & x_{25} \\
		0 & 0 & 1 & x_{34} & 0  \\
		0 & 0 & 0 & 1 & 0  \\
		0 & 0 & 0 & 0 & 1 \\
	\end{matrix}\right),x^A=\left(\begin{matrix}
		1 & 1 & 0 & 0 & 0 \\
		0 & 1 & 1 & 0 & 0 \\
		0 & 0 & 1 & 1 & 0 \\
		0 & 0 & 0 & 1 & 0 \\
		0 & 0 & 0 & 0 & 1
	\end{matrix}\right)
\end{equation*} \newline
Where matrix $A$ has entries
\begin{align*}
	&d_{1} = 1, d_{2} = \frac{1}{x_{12}}, d_{3} = \frac{1}{x_{12} x_{23}}, d_{4} = \frac{1}{x_{12} x_{23} x_{34}}, d_{5} = 1,\\ 
	&a_{12} = 1, a_{13} = 1, a_{14} = 1, a_{15} = 1,\\ 
	&a_{23} = \frac{x_{12} x_{23} - x_{13}}{x_{12}^{2} x_{23}}, a_{24} = -\frac{x_{12} x_{14} x_{23} - x_{12} x_{13} x_{24} - {\left(x_{12}^{2} x_{23}^{2} - x_{12} x_{13} x_{23} + x_{13}^{2}\right)} x_{34}}{x_{12}^{3} x_{23}^{2} x_{34}}, a_{25} = \frac{x_{13} x_{25}}{x_{12} x_{23}},\\ 
	&a_{34} = -\frac{x_{12} x_{24} - {\left(x_{12} x_{23} - x_{13}\right)} x_{34}}{x_{12}^{2} x_{23}^{2} x_{34}}, a_{35} = -\frac{x_{25}}{x_{23}},\\ 
	&a_{45} = 0
\end{align*}

\begin{equation*}x=\left(\begin{matrix}
		1 & x_{12} & x_{13} & x_{14} & x_{15} \\
		0 & 1 & x_{23} & x_{24} & x_{25} \\
		0 & 0 & 1 & x_{34} & 0  \\
		0 & 0 & 0 & 1 & 0  \\
		0 & 0 & 0 & 0 & 1 \\
	\end{matrix}\right),x^A=\left(\begin{matrix}
		1 & 1 & 0 & 0 & 0 \\
		0 & 1 & 1 & 0 & 0 \\
		0 & 0 & 1 & 1 & 0 \\
		0 & 0 & 0 & 1 & 0 \\
		0 & 0 & 0 & 0 & 1
	\end{matrix}\right)
\end{equation*} \newline
Where matrix $A$ has entries
\begin{align*}
	&d_{1} = 1, d_{2} = \frac{1}{x_{12}}, d_{3} = \frac{1}{x_{12} x_{23}}, d_{4} = \frac{1}{x_{12} x_{23} x_{34}}, d_{5} = 1,\\ 
	&a_{12} = 1, a_{13} = 1, a_{14} = 1, a_{15} = 1,\\ 
	&a_{23} = \frac{x_{12} x_{23} - x_{13}}{x_{12}^{2} x_{23}}, a_{24} = -\frac{x_{12} x_{14} x_{23} - x_{12} x_{13} x_{24} - {\left(x_{12}^{2} x_{23}^{2} - x_{12} x_{13} x_{23} + x_{13}^{2}\right)} x_{34}}{x_{12}^{3} x_{23}^{2} x_{34}}, a_{25} = -\frac{x_{15} x_{23} - x_{13} x_{25}}{x_{12} x_{23}},\\ 
	&a_{34} = -\frac{x_{12} x_{24} - {\left(x_{12} x_{23} - x_{13}\right)} x_{34}}{x_{12}^{2} x_{23}^{2} x_{34}}, a_{35} = -\frac{x_{25}}{x_{23}},\\ 
	&a_{45} = 0
\end{align*}

\begin{equation*}x=\left(
\right)
\end{equation*} \newline
Where matrix $A$ has entries
\begin{align*}
	&d_{1} = 1, d_{2} = \frac{1}{x_{12}}, d_{3} = \frac{1}{x_{12} x_{23}}, d_{4} = \frac{1}{x_{12} x_{23} x_{34}}, d_{5} = 1,\\ 
	&a_{12} = 1, a_{13} = 1, a_{14} = 1, a_{15} = 1,\\ 
	&a_{23} = \frac{x_{12} x_{23} - x_{13}}{x_{12}^{2} x_{23}}, a_{24} = \frac{x_{12}^{2} x_{23}^{2} - x_{12} x_{13} x_{23} + x_{13}^{2}}{x_{12}^{3} x_{23}^{2}}, a_{25} = 0,\\ 
	&a_{34} = \frac{x_{12} x_{23} - x_{13}}{x_{12}^{2} x_{23}^{2}}, a_{35} = 0,\\ 
	&a_{45} = -\frac{x_{35}}{x_{34}}
\end{align*}

\begin{equation*}x=\left(\begin{matrix}
		1 & x_{12} & x_{13} & 0 & x_{15} \\
		0 & 1 & x_{23} & 0 & 0 \\
		0 & 0 & 1 & x_{34} & x_{35}  \\
		0 & 0 & 0 & 1 & 0  \\
		0 & 0 & 0 & 0 & 1 \\
	\end{matrix}\right),x^A=\left(\begin{matrix}
		1 & 1 & 0 & 0 & 0 \\
		0 & 1 & 1 & 0 & 0 \\
		0 & 0 & 1 & 1 & 0 \\
		0 & 0 & 0 & 1 & 0 \\
		0 & 0 & 0 & 0 & 1
	\end{matrix}\right)
\end{equation*} \newline
Where matrix $A$ has entries
\begin{align*}
	&d_{1} = 1, d_{2} = \frac{1}{x_{12}}, d_{3} = \frac{1}{x_{12} x_{23}}, d_{4} = \frac{1}{x_{12} x_{23} x_{34}}, d_{5} = 1,\\ 
	&a_{12} = 1, a_{13} = 1, a_{14} = 1, a_{15} = 1,\\ 
	&a_{23} = \frac{x_{12} x_{23} - x_{13}}{x_{12}^{2} x_{23}}, a_{24} = \frac{x_{12}^{2} x_{23}^{2} - x_{12} x_{13} x_{23} + x_{13}^{2}}{x_{12}^{3} x_{23}^{2}}, a_{25} = -\frac{x_{15}}{x_{12}},\\ 
	&a_{34} = \frac{x_{12} x_{23} - x_{13}}{x_{12}^{2} x_{23}^{2}}, a_{35} = 0,\\ 
	&a_{45} = -\frac{x_{35}}{x_{34}}
\end{align*}

\begin{equation*}x=\left(\begin{matrix}
		1 & x_{12} & x_{13} & x_{14} & 0 \\
		0 & 1 & x_{23} & 0 & 0 \\
		0 & 0 & 1 & x_{34} & x_{35}  \\
		0 & 0 & 0 & 1 & 0  \\
		0 & 0 & 0 & 0 & 1 \\
	\end{matrix}\right),x^A=\left(\begin{matrix}
		1 & 1 & 0 & 0 & 0 \\
		0 & 1 & 1 & 0 & 0 \\
		0 & 0 & 1 & 1 & 0 \\
		0 & 0 & 0 & 1 & 0 \\
		0 & 0 & 0 & 0 & 1
	\end{matrix}\right)
\end{equation*} \newline
Where matrix $A$ has entries
\begin{align*}
	&d_{1} = 1, d_{2} = \frac{1}{x_{12}}, d_{3} = \frac{1}{x_{12} x_{23}}, d_{4} = \frac{1}{x_{12} x_{23} x_{34}}, d_{5} = 1,\\ 
	&a_{12} = 1, a_{13} = 1, a_{14} = 1, a_{15} = 1,\\ 
	&a_{23} = \frac{x_{12} x_{23} - x_{13}}{x_{12}^{2} x_{23}}, a_{24} = -\frac{x_{12} x_{14} x_{23} - {\left(x_{12}^{2} x_{23}^{2} - x_{12} x_{13} x_{23} + x_{13}^{2}\right)} x_{34}}{x_{12}^{3} x_{23}^{2} x_{34}}, a_{25} = \frac{x_{14} x_{35}}{x_{12} x_{34}},\\ 
	&a_{34} = \frac{x_{12} x_{23} - x_{13}}{x_{12}^{2} x_{23}^{2}}, a_{35} = 0,\\ 
	&a_{45} = -\frac{x_{35}}{x_{34}}
\end{align*}

\begin{equation*}x=\left(\begin{matrix}
		1 & x_{12} & x_{13} & x_{14} & x_{15} \\
		0 & 1 & x_{23} & 0 & 0 \\
		0 & 0 & 1 & x_{34} & x_{35}  \\
		0 & 0 & 0 & 1 & 0  \\
		0 & 0 & 0 & 0 & 1 \\
	\end{matrix}\right),x^A=\left(\begin{matrix}
		1 & 1 & 0 & 0 & 0 \\
		0 & 1 & 1 & 0 & 0 \\
		0 & 0 & 1 & 1 & 0 \\
		0 & 0 & 0 & 1 & 0 \\
		0 & 0 & 0 & 0 & 1
	\end{matrix}\right)
\end{equation*} \newline
Where matrix $A$ has entries
\begin{align*}
	&d_{1} = 1, d_{2} = \frac{1}{x_{12}}, d_{3} = \frac{1}{x_{12} x_{23}}, d_{4} = \frac{1}{x_{12} x_{23} x_{34}}, d_{5} = 1,\\ 
	&a_{12} = 1, a_{13} = 1, a_{14} = 1, a_{15} = 1,\\ 
	&a_{23} = \frac{x_{12} x_{23} - x_{13}}{x_{12}^{2} x_{23}}, a_{24} = -\frac{x_{12} x_{14} x_{23} - {\left(x_{12}^{2} x_{23}^{2} - x_{12} x_{13} x_{23} + x_{13}^{2}\right)} x_{34}}{x_{12}^{3} x_{23}^{2} x_{34}}, a_{25} = -\frac{x_{15} x_{34} - x_{14} x_{35}}{x_{12} x_{34}},\\ 
	&a_{34} = \frac{x_{12} x_{23} - x_{13}}{x_{12}^{2} x_{23}^{2}}, a_{35} = 0,\\ 
	&a_{45} = -\frac{x_{35}}{x_{34}}
\end{align*}

\begin{equation*}x=\left(
\right)
\end{equation*} \newline
Where matrix $A$ has entries
\begin{align*}
	&d_{1} = 1, d_{2} = \frac{1}{x_{12}}, d_{3} = \frac{1}{x_{12} x_{23}}, d_{4} = \frac{1}{x_{12} x_{23} x_{34}}, d_{5} = 1,\\ 
	&a_{12} = 1, a_{13} = 1, a_{14} = 1, a_{15} = 1,\\ 
	&a_{23} = \frac{x_{12} x_{23} - x_{13}}{x_{12}^{2} x_{23}}, a_{24} = \frac{x_{12}^{2} x_{23}^{2} - x_{12} x_{13} x_{23} + x_{13}^{2}}{x_{12}^{3} x_{23}^{2}}, a_{25} = \frac{x_{13} x_{25}}{x_{12} x_{23}},\\ 
	&a_{34} = \frac{x_{12} x_{23} - x_{13}}{x_{12}^{2} x_{23}^{2}}, a_{35} = -\frac{x_{25}}{x_{23}},\\ 
	&a_{45} = -\frac{x_{35}}{x_{34}}
\end{align*}

\begin{equation*}x=\left(\begin{matrix}
		1 & x_{12} & x_{13} & 0 & x_{15} \\
		0 & 1 & x_{23} & 0 & x_{25} \\
		0 & 0 & 1 & x_{34} & x_{35}  \\
		0 & 0 & 0 & 1 & 0  \\
		0 & 0 & 0 & 0 & 1 \\
	\end{matrix}\right),x^A=\left(\begin{matrix}
		1 & 1 & 0 & 0 & 0 \\
		0 & 1 & 1 & 0 & 0 \\
		0 & 0 & 1 & 1 & 0 \\
		0 & 0 & 0 & 1 & 0 \\
		0 & 0 & 0 & 0 & 1
	\end{matrix}\right)
\end{equation*} \newline
Where matrix $A$ has entries
\begin{align*}
	&d_{1} = 1, d_{2} = \frac{1}{x_{12}}, d_{3} = \frac{1}{x_{12} x_{23}}, d_{4} = \frac{1}{x_{12} x_{23} x_{34}}, d_{5} = 1,\\ 
	&a_{12} = 1, a_{13} = 1, a_{14} = 1, a_{15} = 1,\\ 
	&a_{23} = \frac{x_{12} x_{23} - x_{13}}{x_{12}^{2} x_{23}}, a_{24} = \frac{x_{12}^{2} x_{23}^{2} - x_{12} x_{13} x_{23} + x_{13}^{2}}{x_{12}^{3} x_{23}^{2}}, a_{25} = -\frac{x_{15} x_{23} - x_{13} x_{25}}{x_{12} x_{23}},\\ 
	&a_{34} = \frac{x_{12} x_{23} - x_{13}}{x_{12}^{2} x_{23}^{2}}, a_{35} = -\frac{x_{25}}{x_{23}},\\ 
	&a_{45} = -\frac{x_{35}}{x_{34}}
\end{align*}

\begin{equation*}x=\left(\begin{matrix}
		1 & x_{12} & x_{13} & x_{14} & 0 \\
		0 & 1 & x_{23} & 0 & x_{25} \\
		0 & 0 & 1 & x_{34} & x_{35}  \\
		0 & 0 & 0 & 1 & 0  \\
		0 & 0 & 0 & 0 & 1 \\
	\end{matrix}\right),x^A=\left(\begin{matrix}
		1 & 1 & 0 & 0 & 0 \\
		0 & 1 & 1 & 0 & 0 \\
		0 & 0 & 1 & 1 & 0 \\
		0 & 0 & 0 & 1 & 0 \\
		0 & 0 & 0 & 0 & 1
	\end{matrix}\right)
\end{equation*} \newline
Where matrix $A$ has entries
\begin{align*}
	&d_{1} = 1, d_{2} = \frac{1}{x_{12}}, d_{3} = \frac{1}{x_{12} x_{23}}, d_{4} = \frac{1}{x_{12} x_{23} x_{34}}, d_{5} = 1,\\ 
	&a_{12} = 1, a_{13} = 1, a_{14} = 1, a_{15} = 1,\\ 
	&a_{23} = \frac{x_{12} x_{23} - x_{13}}{x_{12}^{2} x_{23}}, a_{24} = -\frac{x_{12} x_{14} x_{23} - {\left(x_{12}^{2} x_{23}^{2} - x_{12} x_{13} x_{23} + x_{13}^{2}\right)} x_{34}}{x_{12}^{3} x_{23}^{2} x_{34}}, a_{25} = \frac{x_{13} x_{25} x_{34} + x_{14} x_{23} x_{35}}{x_{12} x_{23} x_{34}},\\ 
	&a_{34} = \frac{x_{12} x_{23} - x_{13}}{x_{12}^{2} x_{23}^{2}}, a_{35} = -\frac{x_{25}}{x_{23}},\\ 
	&a_{45} = -\frac{x_{35}}{x_{34}}
\end{align*}

\begin{equation*}x=\left(\begin{matrix}
		1 & x_{12} & x_{13} & x_{14} & x_{15} \\
		0 & 1 & x_{23} & 0 & x_{25} \\
		0 & 0 & 1 & x_{34} & x_{35}  \\
		0 & 0 & 0 & 1 & 0  \\
		0 & 0 & 0 & 0 & 1 \\
	\end{matrix}\right),x^A=\left(\begin{matrix}
		1 & 1 & 0 & 0 & 0 \\
		0 & 1 & 1 & 0 & 0 \\
		0 & 0 & 1 & 1 & 0 \\
		0 & 0 & 0 & 1 & 0 \\
		0 & 0 & 0 & 0 & 1
	\end{matrix}\right)
\end{equation*} \newline
Where matrix $A$ has entries
\begin{align*}
	&d_{1} = 1, d_{2} = \frac{1}{x_{12}}, d_{3} = \frac{1}{x_{12} x_{23}}, d_{4} = \frac{1}{x_{12} x_{23} x_{34}}, d_{5} = 1,\\ 
	&a_{12} = 1, a_{13} = 1, a_{14} = 1, a_{15} = 1,\\ 
	&a_{23} = \frac{x_{12} x_{23} - x_{13}}{x_{12}^{2} x_{23}}, a_{24} = -\frac{x_{12} x_{14} x_{23} - {\left(x_{12}^{2} x_{23}^{2} - x_{12} x_{13} x_{23} + x_{13}^{2}\right)} x_{34}}{x_{12}^{3} x_{23}^{2} x_{34}}, a_{25} = \frac{x_{14} x_{23} x_{35} - {\left(x_{15} x_{23} - x_{13} x_{25}\right)} x_{34}}{x_{12} x_{23} x_{34}},\\ 
	&a_{34} = \frac{x_{12} x_{23} - x_{13}}{x_{12}^{2} x_{23}^{2}}, a_{35} = -\frac{x_{25}}{x_{23}},\\ 
	&a_{45} = -\frac{x_{35}}{x_{34}}
\end{align*}

\begin{equation*}x=\left(
\right)
\end{equation*} \newline
Where matrix $A$ has entries
\begin{align*}
	&d_{1} = 1, d_{2} = \frac{1}{x_{12}}, d_{3} = \frac{1}{x_{12} x_{23}}, d_{4} = \frac{1}{x_{12} x_{23} x_{34}}, d_{5} = 1,\\ 
	&a_{12} = 1, a_{13} = 1, a_{14} = 1, a_{15} = 1,\\ 
	&a_{23} = \frac{1}{x_{12}}, a_{24} = \frac{x_{12} x_{23} x_{34} - x_{14}}{x_{12}^{2} x_{23} x_{34}}, a_{25} = \frac{x_{14} x_{35}}{x_{12} x_{34}},\\ 
	&a_{34} = \frac{x_{23} x_{34} - x_{24}}{x_{12} x_{23}^{2} x_{34}}, a_{35} = \frac{x_{24} x_{35}}{x_{23} x_{34}},\\ 
	&a_{45} = -\frac{x_{35}}{x_{34}}
\end{align*}

\begin{equation*}x=\left(\begin{matrix}
		1 & x_{12} & 0 & x_{14} & x_{15} \\
		0 & 1 & x_{23} & x_{24} & 0 \\
		0 & 0 & 1 & x_{34} & x_{35}  \\
		0 & 0 & 0 & 1 & 0  \\
		0 & 0 & 0 & 0 & 1 \\
	\end{matrix}\right),x^A=\left(\begin{matrix}
		1 & 1 & 0 & 0 & 0 \\
		0 & 1 & 1 & 0 & 0 \\
		0 & 0 & 1 & 1 & 0 \\
		0 & 0 & 0 & 1 & 0 \\
		0 & 0 & 0 & 0 & 1
	\end{matrix}\right)
\end{equation*} \newline
Where matrix $A$ has entries
\begin{align*}
	&d_{1} = 1, d_{2} = \frac{1}{x_{12}}, d_{3} = \frac{1}{x_{12} x_{23}}, d_{4} = \frac{1}{x_{12} x_{23} x_{34}}, d_{5} = 1,\\ 
	&a_{12} = 1, a_{13} = 1, a_{14} = 1, a_{15} = 1,\\ 
	&a_{23} = \frac{1}{x_{12}}, a_{24} = \frac{x_{12} x_{23} x_{34} - x_{14}}{x_{12}^{2} x_{23} x_{34}}, a_{25} = -\frac{x_{15} x_{34} - x_{14} x_{35}}{x_{12} x_{34}},\\ 
	&a_{34} = \frac{x_{23} x_{34} - x_{24}}{x_{12} x_{23}^{2} x_{34}}, a_{35} = \frac{x_{24} x_{35}}{x_{23} x_{34}},\\ 
	&a_{45} = -\frac{x_{35}}{x_{34}}
\end{align*}

\begin{equation*}x=\left(\begin{matrix}
		1 & x_{12} & x_{13} & 0 & 0 \\
		0 & 1 & x_{23} & x_{24} & 0 \\
		0 & 0 & 1 & x_{34} & x_{35}  \\
		0 & 0 & 0 & 1 & 0  \\
		0 & 0 & 0 & 0 & 1 \\
	\end{matrix}\right),x^A=\left(\begin{matrix}
		1 & 1 & 0 & 0 & 0 \\
		0 & 1 & 1 & 0 & 0 \\
		0 & 0 & 1 & 1 & 0 \\
		0 & 0 & 0 & 1 & 0 \\
		0 & 0 & 0 & 0 & 1
	\end{matrix}\right)
\end{equation*} \newline
Where matrix $A$ has entries
\begin{align*}
	&d_{1} = 1, d_{2} = \frac{1}{x_{12}}, d_{3} = \frac{1}{x_{12} x_{23}}, d_{4} = \frac{1}{x_{12} x_{23} x_{34}}, d_{5} = 1,\\ 
	&a_{12} = 1, a_{13} = 1, a_{14} = 1, a_{15} = 1,\\ 
	&a_{23} = \frac{x_{12} x_{23} - x_{13}}{x_{12}^{2} x_{23}}, a_{24} = \frac{x_{12} x_{13} x_{24} + {\left(x_{12}^{2} x_{23}^{2} - x_{12} x_{13} x_{23} + x_{13}^{2}\right)} x_{34}}{x_{12}^{3} x_{23}^{2} x_{34}}, a_{25} = -\frac{x_{13} x_{24} x_{35}}{x_{12} x_{23} x_{34}},\\ 
	&a_{34} = -\frac{x_{12} x_{24} - {\left(x_{12} x_{23} - x_{13}\right)} x_{34}}{x_{12}^{2} x_{23}^{2} x_{34}}, a_{35} = \frac{x_{24} x_{35}}{x_{23} x_{34}},\\ 
	&a_{45} = -\frac{x_{35}}{x_{34}}
\end{align*}

\begin{equation*}x=\left(\begin{matrix}
		1 & x_{12} & x_{13} & 0 & x_{15} \\
		0 & 1 & x_{23} & x_{24} & 0 \\
		0 & 0 & 1 & x_{34} & x_{35}  \\
		0 & 0 & 0 & 1 & 0  \\
		0 & 0 & 0 & 0 & 1 \\
	\end{matrix}\right),x^A=\left(\begin{matrix}
		1 & 1 & 0 & 0 & 0 \\
		0 & 1 & 1 & 0 & 0 \\
		0 & 0 & 1 & 1 & 0 \\
		0 & 0 & 0 & 1 & 0 \\
		0 & 0 & 0 & 0 & 1
	\end{matrix}\right)
\end{equation*} \newline
Where matrix $A$ has entries
\begin{align*}
	&d_{1} = 1, d_{2} = \frac{1}{x_{12}}, d_{3} = \frac{1}{x_{12} x_{23}}, d_{4} = \frac{1}{x_{12} x_{23} x_{34}}, d_{5} = 1,\\ 
	&a_{12} = 1, a_{13} = 1, a_{14} = 1, a_{15} = 1,\\ 
	&a_{23} = \frac{x_{12} x_{23} - x_{13}}{x_{12}^{2} x_{23}}, a_{24} = \frac{x_{12} x_{13} x_{24} + {\left(x_{12}^{2} x_{23}^{2} - x_{12} x_{13} x_{23} + x_{13}^{2}\right)} x_{34}}{x_{12}^{3} x_{23}^{2} x_{34}}, a_{25} = -\frac{x_{15} x_{23} x_{34} + x_{13} x_{24} x_{35}}{x_{12} x_{23} x_{34}},\\ 
	&a_{34} = -\frac{x_{12} x_{24} - {\left(x_{12} x_{23} - x_{13}\right)} x_{34}}{x_{12}^{2} x_{23}^{2} x_{34}}, a_{35} = \frac{x_{24} x_{35}}{x_{23} x_{34}},\\ 
	&a_{45} = -\frac{x_{35}}{x_{34}}
\end{align*}

\begin{equation*}x=\left(\begin{matrix}
		1 & x_{12} & x_{13} & x_{14} & 0 \\
		0 & 1 & x_{23} & x_{24} & 0 \\
		0 & 0 & 1 & x_{34} & x_{35}  \\
		0 & 0 & 0 & 1 & 0  \\
		0 & 0 & 0 & 0 & 1 \\
	\end{matrix}\right),x^A=\left(\begin{matrix}
		1 & 1 & 0 & 0 & 0 \\
		0 & 1 & 1 & 0 & 0 \\
		0 & 0 & 1 & 1 & 0 \\
		0 & 0 & 0 & 1 & 0 \\
		0 & 0 & 0 & 0 & 1
	\end{matrix}\right)
\end{equation*} \newline
Where matrix $A$ has entries
\begin{align*}
	&d_{1} = 1, d_{2} = \frac{1}{x_{12}}, d_{3} = \frac{1}{x_{12} x_{23}}, d_{4} = \frac{1}{x_{12} x_{23} x_{34}}, d_{5} = 1,\\ 
	&a_{12} = 1, a_{13} = 1, a_{14} = 1, a_{15} = 1,\\ 
	&a_{23} = \frac{x_{12} x_{23} - x_{13}}{x_{12}^{2} x_{23}}, a_{24} = -\frac{x_{12} x_{14} x_{23} - x_{12} x_{13} x_{24} - {\left(x_{12}^{2} x_{23}^{2} - x_{12} x_{13} x_{23} + x_{13}^{2}\right)} x_{34}}{x_{12}^{3} x_{23}^{2} x_{34}}, a_{25} = \frac{{\left(x_{14} x_{23} - x_{13} x_{24}\right)} x_{35}}{x_{12} x_{23} x_{34}},\\ 
	&a_{34} = -\frac{x_{12} x_{24} - {\left(x_{12} x_{23} - x_{13}\right)} x_{34}}{x_{12}^{2} x_{23}^{2} x_{34}}, a_{35} = \frac{x_{24} x_{35}}{x_{23} x_{34}},\\ 
	&a_{45} = -\frac{x_{35}}{x_{34}}
\end{align*}

\begin{equation*}x=\left(\begin{matrix}
		1 & x_{12} & x_{13} & x_{14} & x_{15} \\
		0 & 1 & x_{23} & x_{24} & 0 \\
		0 & 0 & 1 & x_{34} & x_{35}  \\
		0 & 0 & 0 & 1 & 0  \\
		0 & 0 & 0 & 0 & 1 \\
	\end{matrix}\right),x^A=\left(\begin{matrix}
		1 & 1 & 0 & 0 & 0 \\
		0 & 1 & 1 & 0 & 0 \\
		0 & 0 & 1 & 1 & 0 \\
		0 & 0 & 0 & 1 & 0 \\
		0 & 0 & 0 & 0 & 1
	\end{matrix}\right)
\end{equation*} \newline
Where matrix $A$ has entries
\begin{align*}
	&d_{1} = 1, d_{2} = \frac{1}{x_{12}}, d_{3} = \frac{1}{x_{12} x_{23}}, d_{4} = \frac{1}{x_{12} x_{23} x_{34}}, d_{5} = 1,\\ 
	&a_{12} = 1, a_{13} = 1, a_{14} = 1, a_{15} = 1,\\ 
	&a_{23} = \frac{x_{12} x_{23} - x_{13}}{x_{12}^{2} x_{23}}, a_{24} = -\frac{x_{12} x_{14} x_{23} - x_{12} x_{13} x_{24} - {\left(x_{12}^{2} x_{23}^{2} - x_{12} x_{13} x_{23} + x_{13}^{2}\right)} x_{34}}{x_{12}^{3} x_{23}^{2} x_{34}}, a_{25} = -\frac{x_{15} x_{23} x_{34} - {\left(x_{14} x_{23} - x_{13} x_{24}\right)} x_{35}}{x_{12} x_{23} x_{34}},\\ 
	&a_{34} = -\frac{x_{12} x_{24} - {\left(x_{12} x_{23} - x_{13}\right)} x_{34}}{x_{12}^{2} x_{23}^{2} x_{34}}, a_{35} = \frac{x_{24} x_{35}}{x_{23} x_{34}},\\ 
	&a_{45} = -\frac{x_{35}}{x_{34}}
\end{align*}

\begin{equation*}x=\left(
\right)
\end{equation*} \newline
Where matrix $A$ has entries
\begin{align*}
	&d_{1} = 1, d_{2} = \frac{1}{x_{12}}, d_{3} = \frac{1}{x_{12} x_{23}}, d_{4} = \frac{1}{x_{12} x_{23} x_{34}}, d_{5} = 1,\\ 
	&a_{12} = 1, a_{13} = 1, a_{14} = 1, a_{15} = 1,\\ 
	&a_{23} = \frac{1}{x_{12}}, a_{24} = \frac{x_{12} x_{23} x_{34} - x_{14}}{x_{12}^{2} x_{23} x_{34}}, a_{25} = \frac{x_{14} x_{35}}{x_{12} x_{34}},\\ 
	&a_{34} = \frac{x_{23} x_{34} - x_{24}}{x_{12} x_{23}^{2} x_{34}}, a_{35} = -\frac{x_{25} x_{34} - x_{24} x_{35}}{x_{23} x_{34}},\\ 
	&a_{45} = -\frac{x_{35}}{x_{34}}
\end{align*}

\begin{equation*}x=\left(\begin{matrix}
		1 & x_{12} & 0 & x_{14} & x_{15} \\
		0 & 1 & x_{23} & x_{24} & x_{25} \\
		0 & 0 & 1 & x_{34} & x_{35}  \\
		0 & 0 & 0 & 1 & 0  \\
		0 & 0 & 0 & 0 & 1 \\
	\end{matrix}\right),x^A=\left(\begin{matrix}
		1 & 1 & 0 & 0 & 0 \\
		0 & 1 & 1 & 0 & 0 \\
		0 & 0 & 1 & 1 & 0 \\
		0 & 0 & 0 & 1 & 0 \\
		0 & 0 & 0 & 0 & 1
	\end{matrix}\right)
\end{equation*} \newline
Where matrix $A$ has entries
\begin{align*}
	&d_{1} = 1, d_{2} = \frac{1}{x_{12}}, d_{3} = \frac{1}{x_{12} x_{23}}, d_{4} = \frac{1}{x_{12} x_{23} x_{34}}, d_{5} = 1,\\ 
	&a_{12} = 1, a_{13} = 1, a_{14} = 1, a_{15} = 1,\\ 
	&a_{23} = \frac{1}{x_{12}}, a_{24} = \frac{x_{12} x_{23} x_{34} - x_{14}}{x_{12}^{2} x_{23} x_{34}}, a_{25} = -\frac{x_{15} x_{34} - x_{14} x_{35}}{x_{12} x_{34}},\\ 
	&a_{34} = \frac{x_{23} x_{34} - x_{24}}{x_{12} x_{23}^{2} x_{34}}, a_{35} = -\frac{x_{25} x_{34} - x_{24} x_{35}}{x_{23} x_{34}},\\ 
	&a_{45} = -\frac{x_{35}}{x_{34}}
\end{align*}

\begin{equation*}x=\left(\begin{matrix}
		1 & x_{12} & x_{13} & 0 & 0 \\
		0 & 1 & x_{23} & x_{24} & x_{25} \\
		0 & 0 & 1 & x_{34} & x_{35}  \\
		0 & 0 & 0 & 1 & 0  \\
		0 & 0 & 0 & 0 & 1 \\
	\end{matrix}\right),x^A=\left(\begin{matrix}
		1 & 1 & 0 & 0 & 0 \\
		0 & 1 & 1 & 0 & 0 \\
		0 & 0 & 1 & 1 & 0 \\
		0 & 0 & 0 & 1 & 0 \\
		0 & 0 & 0 & 0 & 1
	\end{matrix}\right)
\end{equation*} \newline
Where matrix $A$ has entries
\begin{align*}
	&d_{1} = 1, d_{2} = \frac{1}{x_{12}}, d_{3} = \frac{1}{x_{12} x_{23}}, d_{4} = \frac{1}{x_{12} x_{23} x_{34}}, d_{5} = 1,\\ 
	&a_{12} = 1, a_{13} = 1, a_{14} = 1, a_{15} = 1,\\ 
	&a_{23} = \frac{x_{12} x_{23} - x_{13}}{x_{12}^{2} x_{23}}, a_{24} = \frac{x_{12} x_{13} x_{24} + {\left(x_{12}^{2} x_{23}^{2} - x_{12} x_{13} x_{23} + x_{13}^{2}\right)} x_{34}}{x_{12}^{3} x_{23}^{2} x_{34}}, a_{25} = \frac{x_{13} x_{25} x_{34} - x_{13} x_{24} x_{35}}{x_{12} x_{23} x_{34}},\\ 
	&a_{34} = -\frac{x_{12} x_{24} - {\left(x_{12} x_{23} - x_{13}\right)} x_{34}}{x_{12}^{2} x_{23}^{2} x_{34}}, a_{35} = -\frac{x_{25} x_{34} - x_{24} x_{35}}{x_{23} x_{34}},\\ 
	&a_{45} = -\frac{x_{35}}{x_{34}}
\end{align*}

\begin{equation*}x=\left(\begin{matrix}
		1 & x_{12} & x_{13} & 0 & x_{15} \\
		0 & 1 & x_{23} & x_{24} & x_{25} \\
		0 & 0 & 1 & x_{34} & x_{35}  \\
		0 & 0 & 0 & 1 & 0  \\
		0 & 0 & 0 & 0 & 1 \\
	\end{matrix}\right),x^A=\left(\begin{matrix}
		1 & 1 & 0 & 0 & 0 \\
		0 & 1 & 1 & 0 & 0 \\
		0 & 0 & 1 & 1 & 0 \\
		0 & 0 & 0 & 1 & 0 \\
		0 & 0 & 0 & 0 & 1
	\end{matrix}\right)
\end{equation*} \newline
Where matrix $A$ has entries
\begin{align*}
	&d_{1} : 1, d_{2} : \frac{1}{x_{12}}, d_{3} : \frac{1}{x_{12} x_{23}}, d_{4} : \frac{1}{x_{12} x_{23} x_{34}}, d_{5} : 1, \\
	&a_{12} : 1, a_{13} : 1, a_{14} : 1, a_{15} : 1, \\
	&a_{23} : \frac{x_{12} x_{23} - x_{13}}{x_{12}^{2} x_{23}}, a_{24} : \frac{x_{12} x_{13} x_{24} + {\left(x_{12}^{2} x_{23}^{2} - x_{12} x_{13} x_{23} + x_{13}^{2}\right)} x_{34}}{x_{12}^{3} x_{23}^{2} x_{34}}, a_{25} : -\frac{x_{13} x_{24} x_{35} + {\left(x_{15} x_{23} - x_{13} x_{25}\right)} x_{34}}{x_{12} x_{23} x_{34}},\\
	&a_{34} : -\frac{x_{12} x_{24} - {\left(x_{12} x_{23} - x_{13}\right)} x_{34}}{x_{12}^{2} x_{23}^{2} x_{34}}, a_{35} : -\frac{x_{25} x_{34} - x_{24} x_{35}}{x_{23} x_{34}}, \\
	&a_{45} : -\frac{x_{35}}{x_{34}}
\end{align*}

\begin{equation*}x=\left(\begin{matrix}
		1 & x_{12} & x_{13} & x_{14} & 0 \\
		0 & 1 & x_{23} & x_{24} & x_{25} \\
		0 & 0 & 1 & x_{34} & x_{35}  \\
		0 & 0 & 0 & 1 & 0  \\
		0 & 0 & 0 & 0 & 1 \\
	\end{matrix}\right),x^A=\left(\begin{matrix}
		1 & 1 & 0 & 0 & 0 \\
		0 & 1 & 1 & 0 & 0 \\
		0 & 0 & 1 & 1 & 0 \\
		0 & 0 & 0 & 1 & 0 \\
		0 & 0 & 0 & 0 & 1
	\end{matrix}\right)
\end{equation*} \newline
Where matrix $A$ has entries
\begin{align*}
	&d_{1} = 1, d_{2} = \frac{1}{x_{12}}, d_{3} = \frac{1}{x_{12} x_{23}}, d_{4} = \frac{1}{x_{12} x_{23} x_{34}}, d_{5} = 1,\\ 
	&a_{12} = 1, a_{13} = 1, a_{14} = 1, a_{15} = 1,\\ 
	&a_{23} = \frac{x_{12} x_{23} - x_{13}}{x_{12}^{2} x_{23}}, a_{24} = -\frac{x_{12} x_{14} x_{23} - x_{12} x_{13} x_{24} - {\left(x_{12}^{2} x_{23}^{2} - x_{12} x_{13} x_{23} + x_{13}^{2}\right)} x_{34}}{x_{12}^{3} x_{23}^{2} x_{34}}, a_{25} = \frac{x_{13} x_{25} x_{34} + {\left(x_{14} x_{23} - x_{13} x_{24}\right)} x_{35}}{x_{12} x_{23} x_{34}},\\ 
	&a_{34} = -\frac{x_{12} x_{24} - {\left(x_{12} x_{23} - x_{13}\right)} x_{34}}{x_{12}^{2} x_{23}^{2} x_{34}}, a_{35} = -\frac{x_{25} x_{34} - x_{24} x_{35}}{x_{23} x_{34}},\\ 
	&a_{45} = -\frac{x_{35}}{x_{34}}
\end{align*}

\begin{equation*}x=\left(\begin{matrix}
		1 & x_{12} & x_{13} & x_{14} & x_{15} \\
		0 & 1 & x_{23} & x_{24} & x_{25} \\
		0 & 0 & 1 & x_{34} & x_{35}  \\
		0 & 0 & 0 & 1 & 0  \\
		0 & 0 & 0 & 0 & 1 \\
	\end{matrix}\right),x^A=\left(\begin{matrix}
		1 & 1 & 0 & 0 & 0 \\
		0 & 1 & 1 & 0 & 0 \\
		0 & 0 & 1 & 1 & 0 \\
		0 & 0 & 0 & 1 & 0 \\
		0 & 0 & 0 & 0 & 1
	\end{matrix}\right)
\end{equation*} \newline
Where matrix $A$ has entries
\begin{align*}
	&d_{1} = 1, d_{2} = \frac{1}{x_{12}}, d_{3} = \frac{1}{x_{12} x_{23}}, d_{4} = \frac{1}{x_{12} x_{23} x_{34}}, d_{5} = 1,\\ 
	&a_{12} = 1, a_{13} = 1, a_{14} = 1, a_{15} = 1,\\ 
	&a_{23} = \frac{x_{12} x_{23} - x_{13}}{x_{12}^{2} x_{23}}, a_{24} = -\frac{x_{12} x_{14} x_{23} - x_{12} x_{13} x_{24} - {\left(x_{12}^{2} x_{23}^{2} - x_{12} x_{13} x_{23} + x_{13}^{2}\right)} x_{34}}{x_{12}^{3} x_{23}^{2} x_{34}},\\,
	&a_{25} = -\frac{{\left(x_{15} x_{23} - x_{13} x_{25}\right)} x_{34} - {\left(x_{14} x_{23} - x_{13} x_{24}\right)} x_{35}}{x_{12} x_{23} x_{34}},\\ 
	&a_{34} = -\frac{x_{12} x_{24} - {\left(x_{12} x_{23} - x_{13}\right)} x_{34}}{x_{12}^{2} x_{23}^{2} x_{34}}, a_{35} = -\frac{x_{25} x_{34} - x_{24} x_{35}}{x_{23} x_{34}},\\ 
	&a_{45} = -\frac{x_{35}}{x_{34}}
\end{align*}

		\section{Subcases of $Y_8$}

\begin{equation*}x=\left(
\right)
\end{equation*} \newline
Where matrix $A$ has entries
\begin{align*}
	&d_{1} = 1, d_{2} = \frac{1}{x_{12}}, d_{3} = 1, d_{4} = \frac{x_{45}}{x_{12} x_{25} + x_{14} x_{45}}, d_{5} = \frac{1}{x_{12} x_{25} + x_{14} x_{45}},\\ 
	&a_{12} = 1, a_{13} = 1, a_{14} = 1, a_{15} = 1,\\ 
	&a_{23} = 0, a_{24} = \frac{x_{25}}{x_{12} x_{25} + x_{14} x_{45}}, a_{25} = 1,\\ 
	&a_{34} = 0, a_{35} = 1,\\ 
	&a_{45} = -\frac{{\left(x_{12} - 1\right)} x_{14} x_{45} + x_{15} + {\left(x_{12}^{2} - x_{12}\right)} x_{25}}{x_{12} x_{14} x_{25} + x_{14}^{2} x_{45}}
\end{align*}

Now assume $x_{14}= \frac{-x_{12}x_{25}}{x_{45}}$
\begin{equation*}x=\left(
\right)
\end{equation*} \newline
Where matrix $A$ has entries
\begin{align*}
	&d_{1} = 1, d_{2} = \frac{1}{x_{12}}, d_{3} = 1, d_{4} = \frac{x_{45}}{x_{12} x_{25} + x_{14} x_{45}}, d_{5} = \frac{1}{x_{12} x_{25} + x_{14} x_{45}},\\ 
	&a_{12} = 1, a_{13} = 1, a_{14} = 1, a_{15} = 1,\\ 
	&a_{23} = -\frac{x_{13}}{x_{12}}, a_{24} = \frac{x_{25}}{x_{12} x_{25} + x_{14} x_{45}}, a_{25} = 1,\\ 
	&a_{34} = 0, a_{35} = 1,\\ 
	&a_{45} = -\frac{{\left(x_{12} + x_{13} - 1\right)} x_{14} x_{45} + x_{15} + {\left(x_{12}^{2} + x_{12} x_{13} - x_{12}\right)} x_{25}}{x_{12} x_{14} x_{25} + x_{14}^{2} x_{45}}
\end{align*}

Now assume $x_{14}= \frac{-x_{12}x_{25}}{x_{45}}$
\begin{equation*}x=\left(
\right)
\end{equation*} \newline
Where matrix $A$ has entries
\begin{align*}
	&d_{1} = 1, d_{2} = \frac{1}{x_{12}}, d_{3} = 1, d_{4} = \frac{1}{x_{12} x_{24}}, d_{5} = \frac{1}{x_{12} x_{24} x_{45}},\\ 
	&a_{12} = 1, a_{13} = 1, a_{14} = 1, a_{15} = 1,\\ 
	&a_{23} = 0, a_{24} = \frac{x_{12} x_{24} - x_{14}}{x_{12}^{2} x_{24}}, a_{25} = -\frac{x_{12} x_{15} x_{24} - {\left(x_{12}^{2} x_{24}^{2} - x_{12} x_{14} x_{24} + x_{14}^{2}\right)} x_{45}}{x_{12}^{3} x_{24}^{2} x_{45}},\\ 
	&a_{34} = 0, a_{35} = 1,\\ 
	&a_{45} = \frac{x_{12} x_{24} - x_{14}}{x_{12}^{2} x_{24}^{2}}
\end{align*}

\begin{equation*}x=\left(\begin{matrix}
		1 & x_{12} & x_{13} & 0 & 0 \\
		0 & 1 & 0 & x_{24} & 0 \\
		0 & 0 & 1 & 0 & 0  \\
		0 & 0 & 0 & 1 & x_{45}  \\
		0 & 0 & 0 & 0 & 1 \\
	\end{matrix}\right),x^A=\left(\begin{matrix}
		1 & 1 & 0 & 0 & 0 \\
		0 & 1 & 0 & 1 & 0 \\
		0 & 0 & 1 & 0 & 0 \\
		0 & 0 & 0 & 1 & 1 \\
		0 & 0 & 0 & 0 & 1
	\end{matrix}\right)
\end{equation*} \newline
Where matrix $A$ has entries
\begin{align*}
	&d_{1} = 1, d_{2} = \frac{1}{x_{12}}, d_{3} = 1, d_{4} = \frac{1}{x_{12} x_{24}}, d_{5} = \frac{1}{x_{12} x_{24} x_{45}},\\ 
	&a_{12} = 1, a_{13} = 1, a_{14} = 1, a_{15} = 1,\\ 
	&a_{23} = -\frac{x_{13}}{x_{12}}, a_{24} = \frac{1}{x_{12}}, a_{25} = 1,\\ 
	&a_{34} = 0, a_{35} = -\frac{x_{12} - 1}{x_{13}},\\ 
	&a_{45} = \frac{1}{x_{12} x_{24}}
\end{align*}

\begin{equation*}x=\left(\begin{matrix}
		1 & x_{12} & x_{13} & 0 & x_{15} \\
		0 & 1 & 0 & x_{24} & 0 \\
		0 & 0 & 1 & 0 & 0  \\
		0 & 0 & 0 & 1 & x_{45}  \\
		0 & 0 & 0 & 0 & 1 \\
	\end{matrix}\right),x^A=\left(\begin{matrix}
		1 & 1 & 0 & 0 & 0 \\
		0 & 1 & 0 & 1 & 0 \\
		0 & 0 & 1 & 0 & 0 \\
		0 & 0 & 0 & 1 & 1 \\
		0 & 0 & 0 & 0 & 1
	\end{matrix}\right)
\end{equation*} \newline
Where matrix $A$ has entries
\begin{align*}
	&d_{1} = 1, d_{2} = \frac{1}{x_{12}}, d_{3} = 1, d_{4} = \frac{1}{x_{12} x_{24}}, d_{5} = \frac{1}{x_{12} x_{24} x_{45}},\\ 
	&a_{12} = 1, a_{13} = 1, a_{14} = 1, a_{15} = 1,\\ 
	&a_{23} = -\frac{x_{13}}{x_{12}}, a_{24} = \frac{1}{x_{12}}, a_{25} = 1,\\ 
	&a_{34} = 0, a_{35} = -\frac{{\left(x_{12}^{2} - x_{12}\right)} x_{24} x_{45} + x_{15}}{x_{12} x_{13} x_{24} x_{45}},\\ 
	&a_{45} = \frac{1}{x_{12} x_{24}}
\end{align*}

\begin{equation*}x=\left(\begin{matrix}
		1 & x_{12} & x_{13} & x_{14} & 0 \\
		0 & 1 & 0 & x_{24} & 0 \\
		0 & 0 & 1 & 0 & 0  \\
		0 & 0 & 0 & 1 & x_{45}  \\
		0 & 0 & 0 & 0 & 1 \\
	\end{matrix}\right),x^A=\left(\begin{matrix}
		1 & 1 & 0 & 0 & 0 \\
		0 & 1 & 0 & 1 & 0 \\
		0 & 0 & 1 & 0 & 0 \\
		0 & 0 & 0 & 1 & 1 \\
		0 & 0 & 0 & 0 & 1
	\end{matrix}\right)
\end{equation*} \newline
Where matrix $A$ has entries
\begin{align*}
	&d_{1} = 1, d_{2} = \frac{1}{x_{12}}, d_{3} = 1, d_{4} = \frac{1}{x_{12} x_{24}}, d_{5} = \frac{1}{x_{12} x_{24} x_{45}},\\ 
	&a_{12} = 1, a_{13} = 1, a_{14} = 1, a_{15} = 1,\\ 
	&a_{23} = -\frac{x_{13}}{x_{12}}, a_{24} = \frac{x_{12} x_{24} - x_{14}}{x_{12}^{2} x_{24}}, a_{25} = 1,\\ 
	&a_{34} = 0, a_{35} = -\frac{x_{12} x_{14} x_{24} - x_{14}^{2} + {\left(x_{12}^{3} - x_{12}^{2}\right)} x_{24}^{2}}{x_{12}^{2} x_{13} x_{24}^{2}},\\ 
	&a_{45} = \frac{x_{12} x_{24} - x_{14}}{x_{12}^{2} x_{24}^{2}}
\end{align*}

\begin{equation*}x=\left(\begin{matrix}
		1 & x_{12} & x_{13} & x_{14} & x_{15} \\
		0 & 1 & 0 & x_{24} & 0 \\
		0 & 0 & 1 & 0 & 0  \\
		0 & 0 & 0 & 1 & x_{45}  \\
		0 & 0 & 0 & 0 & 1 \\
	\end{matrix}\right),x^A=\left(\begin{matrix}
		1 & 1 & 0 & 0 & 0 \\
		0 & 1 & 0 & 1 & 0 \\
		0 & 0 & 1 & 0 & 0 \\
		0 & 0 & 0 & 1 & 1 \\
		0 & 0 & 0 & 0 & 1
	\end{matrix}\right)
\end{equation*} \newline
Where matrix $A$ has entries
\begin{align*}
	&d_{1} = 1, d_{2} = \frac{1}{x_{12}}, d_{3} = 1, d_{4} = \frac{1}{x_{12} x_{24}}, d_{5} = \frac{1}{x_{12} x_{24} x_{45}},\\ 
	&a_{12} = 1, a_{13} = 1, a_{14} = 1, a_{15} = 1,\\ 
	&a_{23} = -\frac{x_{13}}{x_{12}}, a_{24} = \frac{x_{12} x_{24} - x_{14}}{x_{12}^{2} x_{24}}, a_{25} = 1,\\ 
	&a_{34} = 0, a_{35} = -\frac{x_{12} x_{15} x_{24} + {\left(x_{12} x_{14} x_{24} - x_{14}^{2} + {\left(x_{12}^{3} - x_{12}^{2}\right)} x_{24}^{2}\right)} x_{45}}{x_{12}^{2} x_{13} x_{24}^{2} x_{45}},\\ 
	&a_{45} = \frac{x_{12} x_{24} - x_{14}}{x_{12}^{2} x_{24}^{2}}
\end{align*}

\begin{equation*}x=\left(\begin{matrix}
		1 & x_{12} & 0 & 0 & 0 \\
		0 & 1 & 0 & x_{24} & x_{25} \\
		0 & 0 & 1 & 0 & 0  \\
		0 & 0 & 0 & 1 & x_{45}  \\
		0 & 0 & 0 & 0 & 1 \\
	\end{matrix}\right),x^A=\left(\begin{matrix}
		1 & 1 & 0 & 0 & 0 \\
		0 & 1 & 0 & 1 & 0 \\
		0 & 0 & 1 & 0 & 0 \\
		0 & 0 & 0 & 1 & 1 \\
		0 & 0 & 0 & 0 & 1
	\end{matrix}\right)
\end{equation*} \newline
Where matrix $A$ has entries
\begin{align*}
	&d_{1} = 1, d_{2} = \frac{1}{x_{12}}, d_{3} = 1, d_{4} = \frac{1}{x_{12} x_{24}}, d_{5} = \frac{1}{x_{12} x_{24} x_{45}},\\ 
	&a_{12} = 1, a_{13} = 1, a_{14} = 1, a_{15} = 1,\\ 
	&a_{23} = 0, a_{24} = \frac{1}{x_{12}}, a_{25} = \frac{1}{x_{12}},\\ 
	&a_{34} = 0, a_{35} = 1,\\ 
	&a_{45} = \frac{x_{24} x_{45} - x_{25}}{x_{12} x_{24}^{2} x_{45}}
\end{align*}

\begin{equation*}x=\left(\begin{matrix}
		1 & x_{12} & 0 & 0 & x_{15} \\
		0 & 1 & 0 & x_{24} & x_{25} \\
		0 & 0 & 1 & 0 & 0  \\
		0 & 0 & 0 & 1 & x_{45}  \\
		0 & 0 & 0 & 0 & 1 \\
	\end{matrix}\right),x^A=\left(\begin{matrix}
		1 & 1 & 0 & 0 & 0 \\
		0 & 1 & 0 & 1 & 0 \\
		0 & 0 & 1 & 0 & 0 \\
		0 & 0 & 0 & 1 & 1 \\
		0 & 0 & 0 & 0 & 1
	\end{matrix}\right)
\end{equation*} \newline
Where matrix $A$ has entries
\begin{align*}
	&d_{1} = 1, d_{2} = \frac{1}{x_{12}}, d_{3} = 1, d_{4} = \frac{1}{x_{12} x_{24}}, d_{5} = \frac{1}{x_{12} x_{24} x_{45}},\\ 
	&a_{12} = 1, a_{13} = 1, a_{14} = 1, a_{15} = 1,\\ 
	&a_{23} = 0, a_{24} = \frac{1}{x_{12}}, a_{25} = \frac{x_{12} x_{24} x_{45} - x_{15}}{x_{12}^{2} x_{24} x_{45}},\\ 
	&a_{34} = 0, a_{35} = 1,\\ 
	&a_{45} = \frac{x_{24} x_{45} - x_{25}}{x_{12} x_{24}^{2} x_{45}}
\end{align*}

\begin{equation*}x=\left(\begin{matrix}
		1 & x_{12} & 0 & x_{14} & 0 \\
		0 & 1 & 0 & x_{24} & x_{25} \\
		0 & 0 & 1 & 0 & 0  \\
		0 & 0 & 0 & 1 & x_{45}  \\
		0 & 0 & 0 & 0 & 1 \\
	\end{matrix}\right),x^A=\left(\begin{matrix}
		1 & 1 & 0 & 0 & 0 \\
		0 & 1 & 0 & 1 & 0 \\
		0 & 0 & 1 & 0 & 0 \\
		0 & 0 & 0 & 1 & 1 \\
		0 & 0 & 0 & 0 & 1
	\end{matrix}\right)
\end{equation*} \newline
Where matrix $A$ has entries
\begin{align*}
	&d_{1} = 1, d_{2} = \frac{1}{x_{12}}, d_{3} = 1, d_{4} = \frac{1}{x_{12} x_{24}}, d_{5} = \frac{1}{x_{12} x_{24} x_{45}},\\ 
	&a_{12} = 1, a_{13} = 1, a_{14} = 1, a_{15} = 1,\\ 
	&a_{23} = 0, a_{24} = \frac{x_{12} x_{24} - x_{14}}{x_{12}^{2} x_{24}}, a_{25} = \frac{x_{12} x_{14} x_{25} + {\left(x_{12}^{2} x_{24}^{2} - x_{12} x_{14} x_{24} + x_{14}^{2}\right)} x_{45}}{x_{12}^{3} x_{24}^{2} x_{45}},\\ 
	&a_{34} = 0, a_{35} = 1,\\ 
	&a_{45} = -\frac{x_{12} x_{25} - {\left(x_{12} x_{24} - x_{14}\right)} x_{45}}{x_{12}^{2} x_{24}^{2} x_{45}}
\end{align*}

\begin{equation*}x=\left(\begin{matrix}
		1 & x_{12} & 0 & x_{14} & x_{15} \\
		0 & 1 & 0 & x_{24} & x_{25} \\
		0 & 0 & 1 & 0 & 0  \\
		0 & 0 & 0 & 1 & x_{45}  \\
		0 & 0 & 0 & 0 & 1 \\
	\end{matrix}\right),x^A=\left(\begin{matrix}
		1 & 1 & 0 & 0 & 0 \\
		0 & 1 & 0 & 1 & 0 \\
		0 & 0 & 1 & 0 & 0 \\
		0 & 0 & 0 & 1 & 1 \\
		0 & 0 & 0 & 0 & 1
	\end{matrix}\right)
\end{equation*} \newline
Where matrix $A$ has entries
\begin{align*}
	&d_{1} = 1, d_{2} = \frac{1}{x_{12}}, d_{3} = 1, d_{4} = \frac{1}{x_{12} x_{24}}, d_{5} = \frac{1}{x_{12} x_{24} x_{45}},\\ 
	&a_{12} = 1, a_{13} = 1, a_{14} = 1, a_{15} = 1,\\ 
	&a_{23} = 0, a_{24} = \frac{x_{12} x_{24} - x_{14}}{x_{12}^{2} x_{24}}, a_{25} = -\frac{x_{12} x_{15} x_{24} - x_{12} x_{14} x_{25} - {\left(x_{12}^{2} x_{24}^{2} - x_{12} x_{14} x_{24} + x_{14}^{2}\right)} x_{45}}{x_{12}^{3} x_{24}^{2} x_{45}},\\ 
	&a_{34} = 0, a_{35} = 1,\\ 
	&a_{45} = -\frac{x_{12} x_{25} - {\left(x_{12} x_{24} - x_{14}\right)} x_{45}}{x_{12}^{2} x_{24}^{2} x_{45}}
\end{align*}

\begin{equation*}x=\left(
\right)
\end{equation*} \newline
Where matrix $A$ has entries
\begin{align*}
	&d_{1} = 1, d_{2} = \frac{1}{x_{12}}, d_{3} = 1, d_{4} = \frac{1}{x_{12} x_{24}}, d_{5} = \frac{1}{x_{12} x_{24} x_{45}},\\ 
	&a_{12} = 1, a_{13} = 1, a_{14} = 1, a_{15} = 1,\\ 
	&a_{23} = -\frac{x_{13}}{x_{12}}, a_{24} = \frac{x_{12} x_{24} - x_{14}}{x_{12}^{2} x_{24}}, a_{25} = 1,\\ 
	&a_{34} = 0, a_{35} = \frac{x_{12} x_{14} x_{25} - {\left(x_{12} x_{14} x_{24} - x_{14}^{2} + {\left(x_{12}^{3} - x_{12}^{2}\right)} x_{24}^{2}\right)} x_{45}}{x_{12}^{2} x_{13} x_{24}^{2} x_{45}},\\ 
	&a_{45} = -\frac{x_{12} x_{25} - {\left(x_{12} x_{24} - x_{14}\right)} x_{45}}{x_{12}^{2} x_{24}^{2} x_{45}}
\end{align*}

\begin{equation*}x=\left(\begin{matrix}
		1 & x_{12} & x_{13} & x_{14} & x_{15} \\
		0 & 1 & 0 & x_{24} & x_{25} \\
		0 & 0 & 1 & 0 & 0  \\
		0 & 0 & 0 & 1 & x_{45}  \\
		0 & 0 & 0 & 0 & 1 \\
	\end{matrix}\right),x^A=\left(\begin{matrix}
		1 & 1 & 0 & 0 & 0 \\
		0 & 1 & 0 & 1 & 0 \\
		0 & 0 & 1 & 0 & 0 \\
		0 & 0 & 0 & 1 & 1 \\
		0 & 0 & 0 & 0 & 1
	\end{matrix}\right)
\end{equation*} \newline
Where matrix $A$ has entries
\begin{align*}
	&d_{1} = 1, d_{2} = \frac{1}{x_{12}}, d_{3} = 1, d_{4} = \frac{1}{x_{12} x_{24}}, d_{5} = \frac{1}{x_{12} x_{24} x_{45}},\\ 
	&a_{12} = 1, a_{13} = 1, a_{14} = 1, a_{15} = 1,\\ 
	&a_{23} = -\frac{x_{13}}{x_{12}}, a_{24} = \frac{x_{12} x_{24} - x_{14}}{x_{12}^{2} x_{24}}, a_{25} = 1,\\ 
	&a_{34} = 0, a_{35} = -\frac{x_{12} x_{15} x_{24} - x_{12} x_{14} x_{25} + {\left(x_{12} x_{14} x_{24} - x_{14}^{2} + {\left(x_{12}^{3} - x_{12}^{2}\right)} x_{24}^{2}\right)} x_{45}}{x_{12}^{2} x_{13} x_{24}^{2} x_{45}},\\ 
	&a_{45} = -\frac{x_{12} x_{25} - {\left(x_{12} x_{24} - x_{14}\right)} x_{45}}{x_{12}^{2} x_{24}^{2} x_{45}}
\end{align*}

\begin{equation*}x=\left(
\right)
\end{equation*} \newline
Where matrix $A$ has entries
\begin{align*}
	&d_{1} = 1, d_{2} = \frac{1}{x_{12}}, d_{3} = 1, d_{4} = \frac{x_{45}}{x_{13} x_{35} + x_{14} x_{45}}, d_{5} = \frac{1}{x_{13} x_{35} + x_{14} x_{45}},\\ 
	&a_{12} = 1, a_{13} = 1, a_{14} = 1, a_{15} = 1,\\ 
	&a_{23} = -\frac{x_{13}}{x_{12}}, a_{24} = 0, a_{25} = 1,\\ 
	&a_{34} = \frac{x_{35}}{x_{13} x_{35} + x_{14} x_{45}}, a_{35} = 1,\\ 
	&a_{45} = -\frac{{\left(x_{12} + x_{13} - 1\right)} x_{14} x_{45} + x_{15} + {\left(x_{13}^{2} + {\left(x_{12} - 1\right)} x_{13}\right)} x_{35}}{x_{13} x_{14} x_{35} + x_{14}^{2} x_{45}}
\end{align*}

Now assume $x_{14}= \frac{-x_{13}x_{35}}{x_{45}}$
\begin{equation*}x=\left(
\right)
\end{equation*} \newline
Where matrix $A$ has entries
\begin{align*}
	&d_{1} = 1, d_{2} = \frac{1}{x_{12}}, d_{3} = 1, d_{4} = \frac{x_{45}}{x_{12} x_{25} + x_{14} x_{45}}, d_{5} = \frac{1}{x_{12} x_{25} + x_{14} x_{45}},\\ 
	&a_{12} = 1, a_{13} = 1, a_{14} = 1, a_{15} = 1,\\ 
	&a_{23} = 0, a_{24} = \frac{x_{25}}{x_{12} x_{25} + x_{14} x_{45}}, a_{25} = 1,\\ 
	&a_{34} = \frac{x_{35}}{x_{12} x_{25} + x_{14} x_{45}}, a_{35} = 1,\\ 
	&a_{45} = -\frac{{\left(x_{12} - 1\right)} x_{14} x_{45} + x_{15} + {\left(x_{12}^{2} - x_{12}\right)} x_{25}}{x_{12} x_{14} x_{25} + x_{14}^{2} x_{45}}
\end{align*}

Now assume $x_{14}= \frac{-x_{12}x_{25}}{x_{45}}$
\begin{equation*}x=\left(\begin{matrix}
		1 & x_{12} & 0 & x_{14} & x_{15} \\
		0 & 1 & 0 & 0 & x_{25} \\
		0 & 0 & 1 & 0 & x_{35}  \\
		0 & 0 & 0 & 1 & x_{45}  \\
		0 & 0 & 0 & 0 & 1 \\
	\end{matrix}\right),x^A=\left(\begin{matrix}
		1 & 1 & 0 & 0 & 0 \\
		0 & 1 & 0 & 0 & 0 \\
		0 & 0 & 1 & 0 & 0 \\
		0 & 0 & 0 & 1 & 1 \\
		0 & 0 & 0 & 0 & 1
	\end{matrix}\right)
\end{equation*} \newline
Where matrix $A$ has entries
\begin{align*}
	&d_{1} = 1, d_{2} = \frac{1}{x_{12}}, d_{3} = 1, d_{4} = 1, d_{5} = \frac{1}{x_{45}},\\ 
	&a_{12} = 1, a_{13} = 1, a_{14} = 1, a_{15} = 1,\\ 
	&a_{23} = 0, a_{24} = \frac{x_{25}}{x_{45}}, a_{25} = 1,\\ 
	&a_{34} = \frac{x_{35}}{x_{45}}, a_{35} = 1,\\ 
	&a_{45} = \frac{x_{15} + {\left(x_{12} - 1\right)} x_{45}}{x_{12} x_{25}}
\end{align*}


First assume $x_{25}\neq \frac{-x_{13}x_{35}}{x_{12}}$
\begin{equation*}x=\left(\begin{matrix}
		1 & x_{12} & x_{13} & 0 & 0 \\
		0 & 1 & 0 & 0 & x_{25} \\
		0 & 0 & 1 & 0 & x_{35}  \\
		0 & 0 & 0 & 1 & x_{45}  \\
		0 & 0 & 0 & 0 & 1 \\
	\end{matrix}\right),x^A=\left(\begin{matrix}
		1 & 1 & 0 & 1 & 0 \\
		0 & 1 & 0 & 0 & 0 \\
		0 & 0 & 1 & 0 & 0 \\
		0 & 0 & 0 & 1 & 1 \\
		0 & 0 & 0 & 0 & 1
	\end{matrix}\right)
\end{equation*} \newline
Where matrix $A$ has entries
\begin{align*}
	&d_{1} = 1, d_{2} = \frac{1}{x_{12}}, d_{3} = 1, d_{4} = \frac{x_{45}}{x_{12} x_{25} + x_{13} x_{35}}, d_{5} = \frac{1}{x_{12} x_{25} + x_{13} x_{35}},\\ 
	&a_{12} = 1, a_{13} = 1, a_{14} = 1, a_{15} = 1,\\ 
	&a_{23} = -\frac{x_{13}}{x_{12}}, a_{24} = \frac{x_{25}}{x_{12} x_{25} + x_{13} x_{35}}, a_{25} = 1,\\ 
	&a_{34} = \frac{x_{35}}{x_{12} x_{25} + x_{13} x_{35}}, a_{35} = -\frac{x_{12} - 1}{x_{13}},\\ 
	&a_{45} = 1
\end{align*}

Now assume $x_{25}= \frac{-x_{13}x_{35}}{x_{12}}$
\begin{equation*}x=\left(\begin{matrix}
		1 & x_{12} & x_{13} & 0 & 0 \\
		0 & 1 & 0 & 0 & x_{25} \\
		0 & 0 & 1 & 0 & x_{35}  \\
		0 & 0 & 0 & 1 & x_{45}  \\
		0 & 0 & 0 & 0 & 1 \\
	\end{matrix}\right),x^A=\left(\begin{matrix}
		1 & 1 & 0 & 0 & 0 \\
		0 & 1 & 0 & 0 & 0 \\
		0 & 0 & 1 & 0 & 0 \\
		0 & 0 & 0 & 1 & 1 \\
		0 & 0 & 0 & 0 & 1
	\end{matrix}\right)
\end{equation*} \newline
Where matrix $A$ has entries
\begin{align*}
	&d_{1} = 1, d_{2} = \frac{1}{x_{12}}, d_{3} = 1, d_{4} = 1, d_{5} = \frac{1}{x_{45}},\\ 
	&a_{12} = 1, a_{13} = 1, a_{14} = 1, a_{15} = 1,\\ 
	&a_{23} = -\frac{x_{13}}{x_{12}}, a_{24} = -\frac{x_{13} x_{35}}{x_{12} x_{45}}, a_{25} = 1,\\ 
	&a_{34} = \frac{x_{35}}{x_{45}}, a_{35} = -\frac{x_{12} - 1}{x_{13}},\\ 
	&a_{45} = 1
\end{align*}

First assume $x_{25}\neq \frac{-x_{13}x_{35}}{x_{12}}$
\begin{equation*}x=\left(\begin{matrix}
		1 & x_{12} & x_{13} & 0 & x_{15} \\
		0 & 1 & 0 & 0 & x_{25} \\
		0 & 0 & 1 & 0 & x_{35}  \\
		0 & 0 & 0 & 1 & x_{45}  \\
		0 & 0 & 0 & 0 & 1 \\
	\end{matrix}\right),x^A=\left(\begin{matrix}
		1 & 1 & 0 & 1 & 0 \\
		0 & 1 & 0 & 0 & 0 \\
		0 & 0 & 1 & 0 & 0 \\
		0 & 0 & 0 & 1 & 1 \\
		0 & 0 & 0 & 0 & 1
	\end{matrix}\right)
\end{equation*} \newline
Where matrix $A$ has entries
\begin{align*}
	&d_{1} = 1, d_{2} = \frac{1}{x_{12}}, d_{3} = 1, d_{4} = \frac{x_{45}}{x_{12} x_{25} + x_{13} x_{35}}, d_{5} = \frac{1}{x_{12} x_{25} + x_{13} x_{35}},\\ 
	&a_{12} = 1, a_{13} = 1, a_{14} = 1, a_{15} = 1,\\ 
	&a_{23} = -\frac{x_{13}}{x_{12}}, a_{24} = \frac{x_{25}}{x_{12} x_{25} + x_{13} x_{35}}, a_{25} = 1,\\ 
	&a_{34} = \frac{x_{35}}{x_{12} x_{25} + x_{13} x_{35}}, a_{35} = -\frac{{\left(x_{12} - 1\right)} x_{13} x_{35} + x_{15} + {\left(x_{12}^{2} - x_{12}\right)} x_{25}}{x_{12} x_{13} x_{25} + x_{13}^{2} x_{35}},\\ 
	&a_{45} = 1
\end{align*}

Now assume $x_{25}= \frac{-x_{13}x_{35}}{x_{12}}$
\begin{equation*}x=\left(\begin{matrix}
		1 & x_{12} & x_{13} & 0 & x_{15} \\
		0 & 1 & 0 & 0 & x_{25} \\
		0 & 0 & 1 & 0 & x_{35}  \\
		0 & 0 & 0 & 1 & x_{45}  \\
		0 & 0 & 0 & 0 & 1 \\
	\end{matrix}\right),x^A=\left(\begin{matrix}
		1 & 1 & 0 & 0 & 0 \\
		0 & 1 & 0 & 0 & 0 \\
		0 & 0 & 1 & 0 & 0 \\
		0 & 0 & 0 & 1 & 1 \\
		0 & 0 & 0 & 0 & 1
	\end{matrix}\right)
\end{equation*} \newline
Where matrix $A$ has entries
\begin{align*}
	&d_{1} = 1, d_{2} = \frac{1}{x_{12}}, d_{3} = 1, d_{4} = 1, d_{5} = \frac{1}{x_{45}},\\ 
	&a_{12} = 1, a_{13} = 1, a_{14} = 1, a_{15} = 1,\\ 
	&a_{23} = -\frac{x_{13}}{x_{12}}, a_{24} = -\frac{x_{13} x_{35}}{x_{12} x_{45}}, a_{25} = 1,\\ 
	&a_{34} = \frac{x_{35}}{x_{45}}, a_{35} = -\frac{x_{15} + {\left(x_{12} - 1\right)} x_{45}}{x_{13} x_{45}},\\ 
	&a_{45} = 1
\end{align*}

First assume $x_{14} \neq \frac{-x_{12}x_{25}-x_{13}x_{35}}{x_{45}}$
\begin{equation*}x=\left(\begin{matrix}
		1 & x_{12} & x_{13} & x_{14} & 0 \\
		0 & 1 & 0 & 0 & x_{25} \\
		0 & 0 & 1 & 0 & x_{35}  \\
		0 & 0 & 0 & 1 & x_{45}  \\
		0 & 0 & 0 & 0 & 1 \\
	\end{matrix}\right),x^A=\left(\begin{matrix}
		1 & 1 & 0 & 0 & 0 \\
		0 & 1 & 0 & 1 & 0 \\
		0 & 0 & 1 & 0 & 0 \\
		0 & 0 & 0 & 1 & 1 \\
		0 & 0 & 0 & 0 & 1
	\end{matrix}\right)
\end{equation*} \newline
Where matrix $A$ has entries
\begin{align*}
	&d_{1} = 1, d_{2} = \frac{1}{x_{12}}, d_{3} = 1, d_{4} = \frac{x_{45}}{x_{12} x_{25} + x_{13} x_{35} + x_{14} x_{45}}, d_{5} = \frac{1}{x_{12} x_{25} + x_{13} x_{35} + x_{14} x_{45}},\\ 
	&a_{12} = 1, a_{13} = 1, a_{14} = 1, a_{15} = 1,\\ 
	&a_{23} = -\frac{x_{13}}{x_{12}}, a_{24} = \frac{x_{25}}{x_{12} x_{25} + x_{13} x_{35} + x_{14} x_{45}}, a_{25} = 1,\\ 
	&a_{34} = \frac{x_{35}}{x_{12} x_{25} + x_{13} x_{35} + x_{14} x_{45}}, a_{35} = 1,\\ 
	&a_{45} = -\frac{x_{12} + x_{13} - 1}{x_{14}}
\end{align*}

Now assume $x_{14} = \frac{-x_{12}x_{25}-x_{13}x_{35}}{x_{45}}$ and $x_{25} \neq \frac{-x_{13}x_{35}}{x_{12}}$
\begin{equation*}x=\left(\begin{matrix}
		1 & x_{12} & x_{13} & x_{14} & 0 \\
		0 & 1 & 0 & 0 & x_{25} \\
		0 & 0 & 1 & 0 & x_{35}  \\
		0 & 0 & 0 & 1 & x_{45}  \\
		0 & 0 & 0 & 0 & 1 \\
	\end{matrix}\right),x^A=\left(\begin{matrix}
		1 & 1 & 0 & 0 & 0 \\
		0 & 1 & 0 & 0 & 0 \\
		0 & 0 & 1 & 0 & 0 \\
		0 & 0 & 0 & 1 & 1 \\
		0 & 0 & 0 & 0 & 1
	\end{matrix}\right)
\end{equation*} \newline
Where matrix $A$ has entries
\begin{align*}
	&d_{1} = 1, d_{2} = \frac{1}{x_{12}}, d_{3} = 1, d_{4} = 1, d_{5} = \frac{1}{x_{45}},\\ 
	&a_{12} = 1, a_{13} = 1, a_{14} = 1, a_{15} = 1,\\ 
	&a_{23} = -\frac{x_{13}}{x_{12}}, a_{24} = \frac{x_{25}}{x_{45}}, a_{25} = 1,\\ 
	&a_{34} = \frac{x_{35}}{x_{45}}, a_{35} = 1,\\ 
	&a_{45} = \frac{{\left(x_{12} + x_{13} - 1\right)} x_{45}}{x_{12} x_{25} + x_{13} x_{35}}
\end{align*}

Now assume $x_{14} = \frac{-x_{12}x_{25}-x_{13}x_{35}}{x_{45}}$ and $x_{25} = \frac{-x_{13}x_{35}}{x_{12}}$
\begin{equation*}x=\left(\begin{matrix}
		1 & x_{12} & x_{13} & x_{14} & 0 \\
		0 & 1 & 0 & 0 & x_{25} \\
		0 & 0 & 1 & 0 & x_{35}  \\
		0 & 0 & 0 & 1 & x_{45}  \\
		0 & 0 & 0 & 0 & 1 \\
	\end{matrix}\right),x^A=\left(\begin{matrix}
		1 & 1 & 0 & 0 & 0 \\
		0 & 1 & 0 & 0 & 0 \\
		0 & 0 & 1 & 0 & 0 \\
		0 & 0 & 0 & 1 & 1 \\
		0 & 0 & 0 & 0 & 1
	\end{matrix}\right)
\end{equation*} \newline
Where matrix $A$ has entries
\begin{align*}
	&d_{1} = 1, d_{2} = \frac{1}{x_{12}}, d_{3} = 1, d_{4} = 1, d_{5} = \frac{1}{x_{45}},\\ 
	&a_{12} = 1, a_{13} = 1, a_{14} = 1, a_{15} = 1,\\ 
	&a_{23} = -\frac{x_{13}}{x_{12}}, a_{24} = -\frac{x_{13} x_{35}}{x_{12} x_{45}}, a_{25} = 1,\\ 
	&a_{34} = \frac{x_{35}}{x_{45}}, a_{35} = -\frac{x_{12} - 1}{x_{13}},\\ 
	&a_{45} = 1
\end{align*}

First assume $x_{14} \neq \frac{-x_{12}*x_{25}-x_{13}x_{35}}{x_{45}}$
\begin{equation*}x=\left(\begin{matrix}
		1 & x_{12} & x_{13} & x_{14} & x_{15} \\
		0 & 1 & 0 & 0 & x_{25} \\
		0 & 0 & 1 & 0 & x_{35}  \\
		0 & 0 & 0 & 1 & x_{45}  \\
		0 & 0 & 0 & 0 & 1 \\
	\end{matrix}\right),x^A=\left(\begin{matrix}
		1 & 1 & 0 & 1 & 0 \\
		0 & 1 & 0 & 0 & 0 \\
		0 & 0 & 1 & 0 & 0 \\
		0 & 0 & 0 & 1 & 1 \\
		0 & 0 & 0 & 0 & 1
	\end{matrix}\right)
\end{equation*} \newline
Where matrix $A$ has entries
\begin{align*}
	&d_{1} = 1, d_{2} = \frac{1}{x_{12}}, d_{3} = 1, d_{4} = \frac{x_{45}}{x_{12} x_{25} + x_{13} x_{35} + x_{14} x_{45}}, d_{5} = \frac{1}{x_{12} x_{25} + x_{13} x_{35} + x_{14} x_{45}},\\ 
	&a_{12} = 1, a_{13} = 1, a_{14} = 1, a_{15} = 1,\\ 
	&a_{23} = -\frac{x_{13}}{x_{12}}, a_{24} = \frac{x_{25}}{x_{12} x_{25} + x_{13} x_{35} + x_{14} x_{45}}, a_{25} = 1,\\ 
	&a_{34} = \frac{x_{35}}{x_{12} x_{25} + x_{13} x_{35} + x_{14} x_{45}}, a_{35} = 1,\\ 
	&a_{45} = -\frac{{\left(x_{12} + x_{13} - 1\right)} x_{14} x_{45} + x_{15} + {\left(x_{12}^{2} + x_{12} x_{13} - x_{12}\right)} x_{25} + {\left(x_{13}^{2} + {\left(x_{12} - 1\right)} x_{13}\right)} x_{35}}{x_{12} x_{14} x_{25} + x_{13} x_{14} x_{35} + x_{14}^{2} x_{45}}
\end{align*}

Now assume $x_{14} = \frac{-x_{12}*x_{25}-x_{13}x_{35}}{x_{45}}$ and $x_{25} \neq \frac{-x_{13}x_{35}}{x_{12}}$
\begin{equation*}x=\left(\begin{matrix}
		1 & x_{12} & x_{13} & x_{14} & x_{15} \\
		0 & 1 & 0 & 0 & x_{25} \\
		0 & 0 & 1 & 0 & x_{35}  \\
		0 & 0 & 0 & 1 & x_{45}  \\
		0 & 0 & 0 & 0 & 1 \\
	\end{matrix}\right),x^A=\left(\begin{matrix}
		1 & 1 & 0 & 0 & 0 \\
		0 & 1 & 0 & 0 & 0 \\
		0 & 0 & 1 & 0 & 0 \\
		0 & 0 & 0 & 1 & 1 \\
		0 & 0 & 0 & 0 & 1
	\end{matrix}\right)
\end{equation*} \newline
Where matrix $A$ has entries
\begin{align*}
	&d_{1} = 1, d_{2} = \frac{1}{x_{12}}, d_{3} = 1, d_{4} = 1, d_{5} = \frac{1}{x_{45}},\\ 
	&a_{12} = 1, a_{13} = 1, a_{14} = 1, a_{15} = 1,\\ 
	&a_{23} = -\frac{x_{13}}{x_{12}}, a_{24} = \frac{x_{25}}{x_{45}}, a_{25} = 1,\\ 
	&a_{34} = \frac{x_{35}}{x_{45}}, a_{35} = 1,\\ 
	&a_{45} = \frac{x_{15} + {\left(x_{12} + x_{13} - 1\right)} x_{45}}{x_{12} x_{25} + x_{13} x_{35}}
\end{align*}

Now assume $x_{14} = \frac{-x_{12}x_{25}-x_{13}x_{35}}{x_{45}}$ and $x_{25} = \frac{-x_{13}x_{35}}{x_{12}}$
\begin{equation*}x=\left(
\right)
\end{equation*} \newline
Where matrix $A$ has entries
\begin{align*}
	&d_{1} = 1, d_{2} = \frac{1}{x_{12}}, d_{3} = 1, d_{4} = \frac{1}{x_{12} x_{24}}, d_{5} = \frac{1}{x_{12} x_{24} x_{45}},\\ 
	&a_{12} = 1, a_{13} = 1, a_{14} = 1, a_{15} = 1,\\ 
	&a_{23} = 0, a_{24} = \frac{x_{12} x_{24} - x_{14}}{x_{12}^{2} x_{24}}, a_{25} = \frac{x_{12}^{2} x_{24}^{2} - x_{12} x_{14} x_{24} + x_{14}^{2}}{x_{12}^{3} x_{24}^{2}},\\ 
	&a_{34} = \frac{x_{35}}{x_{12} x_{24} x_{45}}, a_{35} = 1,\\ 
	&a_{45} = \frac{x_{12} x_{24} - x_{14}}{x_{12}^{2} x_{24}^{2}}
\end{align*}

\begin{equation*}x=\left(\begin{matrix}
		1 & x_{12} & 0 & x_{14} & x_{15} \\
		0 & 1 & 0 & x_{24} & 0 \\
		0 & 0 & 1 & 0 & x_{35}  \\
		0 & 0 & 0 & 1 & x_{45}  \\
		0 & 0 & 0 & 0 & 1 \\
	\end{matrix}\right),x^A=\left(\begin{matrix}
		1 & 1 & 0 & 0 & 0 \\
		0 & 1 & 0 & 1 & 0 \\
		0 & 0 & 1 & 0 & 0 \\
		0 & 0 & 0 & 1 & 1 \\
		0 & 0 & 0 & 0 & 1
	\end{matrix}\right)
\end{equation*} \newline
Where matrix $A$ has entries
\begin{align*}
	&d_{1} = 1, d_{2} = \frac{1}{x_{12}}, d_{3} = 1, d_{4} = \frac{1}{x_{12} x_{24}}, d_{5} = \frac{1}{x_{12} x_{24} x_{45}},\\ 
	&a_{12} = 1, a_{13} = 1, a_{14} = 1, a_{15} = 1,\\ 
	&a_{23} = 0, a_{24} = \frac{x_{12} x_{24} - x_{14}}{x_{12}^{2} x_{24}}, a_{25} = -\frac{x_{12} x_{15} x_{24} - {\left(x_{12}^{2} x_{24}^{2} - x_{12} x_{14} x_{24} + x_{14}^{2}\right)} x_{45}}{x_{12}^{3} x_{24}^{2} x_{45}},\\ 
	&a_{34} = \frac{x_{35}}{x_{12} x_{24} x_{45}}, a_{35} = 1,\\ 
	&a_{45} = \frac{x_{12} x_{24} - x_{14}}{x_{12}^{2} x_{24}^{2}}
\end{align*}

\begin{equation*}x=\left(\begin{matrix}
		1 & x_{12} & x_{13} & 0 & 0 \\
		0 & 1 & 0 & x_{24} & 0 \\
		0 & 0 & 1 & 0 & x_{35}  \\
		0 & 0 & 0 & 1 & x_{45}  \\
		0 & 0 & 0 & 0 & 1 \\
	\end{matrix}\right),x^A=\left(\begin{matrix}
		1 & 1 & 0 & 0 & 0 \\
		0 & 1 & 0 & 1 & 0 \\
		0 & 0 & 1 & 0 & 0 \\
		0 & 0 & 0 & 1 & 1 \\
		0 & 0 & 0 & 0 & 1
	\end{matrix}\right)
\end{equation*} \newline
Where matrix $A$ has entries
\begin{align*}
	&d_{1} = 1, d_{2} = \frac{1}{x_{12}}, d_{3} = 1, d_{4} = \frac{1}{x_{12} x_{24}}, d_{5} = \frac{1}{x_{12} x_{24} x_{45}},\\ 
	&a_{12} = 1, a_{13} = 1, a_{14} = 1, a_{15} = 1,\\ 
	&a_{23} = -\frac{x_{13}}{x_{12}}, a_{24} = \frac{x_{12} x_{24} x_{45} - x_{13} x_{35}}{x_{12}^{2} x_{24} x_{45}}, a_{25} = 1,\\ 
	&a_{34} = \frac{x_{35}}{x_{12} x_{24} x_{45}}, a_{35} = -\frac{x_{12} - 1}{x_{13}},\\ 
	&a_{45} = \frac{x_{12} x_{24} x_{45} - x_{13} x_{35}}{x_{12}^{2} x_{24}^{2} x_{45}}
\end{align*}

\begin{equation*}x=\left(\begin{matrix}
		1 & x_{12} & x_{13} & 0 & x_{15} \\
		0 & 1 & 0 & x_{24} & 0 \\
		0 & 0 & 1 & 0 & x_{35}  \\
		0 & 0 & 0 & 1 & x_{45}  \\
		0 & 0 & 0 & 0 & 1 \\
	\end{matrix}\right),x^A=\left(\begin{matrix}
		1 & 1 & 0 & 0 & 0 \\
		0 & 1 & 0 & 1 & 0 \\
		0 & 0 & 1 & 0 & 0 \\
		0 & 0 & 0 & 1 & 1 \\
		0 & 0 & 0 & 0 & 1
	\end{matrix}\right)
\end{equation*} \newline
Where matrix $A$ has entries
\begin{align*}
	&d_{1} = 1, d_{2} = \frac{1}{x_{12}}, d_{3} = 1, d_{4} = \frac{1}{x_{12} x_{24}}, d_{5} = \frac{1}{x_{12} x_{24} x_{45}},\\ 
	&a_{12} = 1, a_{13} = 1, a_{14} = 1, a_{15} = 1,\\ 
	&a_{23} = -\frac{x_{13}}{x_{12}}, a_{24} = \frac{x_{12} x_{24} x_{45} - x_{13} x_{35}}{x_{12}^{2} x_{24} x_{45}}, a_{25} = 1,\\ 
	&a_{34} = \frac{x_{35}}{x_{12} x_{24} x_{45}}, a_{35} = -\frac{{\left(x_{12}^{2} - x_{12}\right)} x_{24} x_{45} + x_{15}}{x_{12} x_{13} x_{24} x_{45}},\\ 
	&a_{45} = \frac{x_{12} x_{24} x_{45} - x_{13} x_{35}}{x_{12}^{2} x_{24}^{2} x_{45}}
\end{align*}

\begin{equation*}x=\left(\begin{matrix}
		1 & x_{12} & x_{13} & x_{14} & 0 \\
		0 & 1 & 0 & x_{24} & 0 \\
		0 & 0 & 1 & 0 & x_{35}  \\
		0 & 0 & 0 & 1 & x_{45}  \\
		0 & 0 & 0 & 0 & 1 \\
	\end{matrix}\right),x^A=\left(\begin{matrix}
		1 & 1 & 0 & 0 & 0 \\
		0 & 1 & 0 & 1 & 0 \\
		0 & 0 & 1 & 0 & 0 \\
		0 & 0 & 0 & 1 & 1 \\
		0 & 0 & 0 & 0 & 1
	\end{matrix}\right)
\end{equation*} \newline
Where matrix $A$ has entries
\begin{align*}
	&d_{1} = 1, d_{2} = \frac{1}{x_{12}}, d_{3} = 1, d_{4} = \frac{1}{x_{12} x_{24}}, d_{5} = \frac{1}{x_{12} x_{24} x_{45}},\\ 
	&a_{12} = 1, a_{13} = 1, a_{14} = 1, a_{15} = 1,\\ 
	&a_{23} = -\frac{x_{13}}{x_{12}}, a_{24} = -\frac{x_{13} x_{35} - {\left(x_{12} x_{24} - x_{14}\right)} x_{45}}{x_{12}^{2} x_{24} x_{45}}, a_{25} = 1,\\ 
	&a_{34} = \frac{x_{35}}{x_{12} x_{24} x_{45}}, a_{35} = \frac{x_{13} x_{14} x_{35} - {\left(x_{12} x_{14} x_{24} - x_{14}^{2} + {\left(x_{12}^{3} - x_{12}^{2}\right)} x_{24}^{2}\right)} x_{45}}{x_{12}^{2} x_{13} x_{24}^{2} x_{45}},\\ 
	&a_{45} = -\frac{x_{13} x_{35} - {\left(x_{12} x_{24} - x_{14}\right)} x_{45}}{x_{12}^{2} x_{24}^{2} x_{45}}
\end{align*}

\begin{equation*}x=\left(\begin{matrix}
		1 & x_{12} & x_{13} & x_{14} & x_{15} \\
		0 & 1 & 0 & x_{24} & 0 \\
		0 & 0 & 1 & 0 & x_{35}  \\
		0 & 0 & 0 & 1 & x_{45}  \\
		0 & 0 & 0 & 0 & 1 \\
	\end{matrix}\right),x^A=\left(\begin{matrix}
		1 & 1 & 0 & 0 & 0 \\
		0 & 1 & 0 & 1 & 0 \\
		0 & 0 & 1 & 0 & 0 \\
		0 & 0 & 0 & 1 & 1 \\
		0 & 0 & 0 & 0 & 1
	\end{matrix}\right)
\end{equation*} \newline
Where matrix $A$ has entries
\begin{align*}
	&d_{1} = 1, d_{2} = \frac{1}{x_{12}}, d_{3} = 1, d_{4} = \frac{1}{x_{12} x_{24}}, d_{5} = \frac{1}{x_{12} x_{24} x_{45}},\\ 
	&a_{12} = 1, a_{13} = 1, a_{14} = 1, a_{15} = 1,\\ 
	&a_{23} = -\frac{x_{13}}{x_{12}}, a_{24} = -\frac{x_{13} x_{35} - {\left(x_{12} x_{24} - x_{14}\right)} x_{45}}{x_{12}^{2} x_{24} x_{45}}, a_{25} = 1,\\ 
	&a_{34} = \frac{x_{35}}{x_{12} x_{24} x_{45}}, a_{35} = -\frac{x_{12} x_{15} x_{24} - x_{13} x_{14} x_{35} + {\left(x_{12} x_{14} x_{24} - x_{14}^{2} + {\left(x_{12}^{3} - x_{12}^{2}\right)} x_{24}^{2}\right)} x_{45}}{x_{12}^{2} x_{13} x_{24}^{2} x_{45}},\\ 
	&a_{45} = -\frac{x_{13} x_{35} - {\left(x_{12} x_{24} - x_{14}\right)} x_{45}}{x_{12}^{2} x_{24}^{2} x_{45}}
\end{align*}

\begin{equation*}x=\left(
\right)
\end{equation*} \newline
Where matrix $A$ has entries
\begin{align*}
	&d_{1} = 1, d_{2} = \frac{1}{x_{12}}, d_{3} = 1, d_{4} = \frac{1}{x_{12} x_{24}}, d_{5} = \frac{1}{x_{12} x_{24} x_{45}},\\ 
	&a_{12} = 1, a_{13} = 1, a_{14} = 1, a_{15} = 1,\\ 
	&a_{23} = 0, a_{24} = \frac{x_{12} x_{24} - x_{14}}{x_{12}^{2} x_{24}}, a_{25} = \frac{x_{12} x_{14} x_{25} + {\left(x_{12}^{2} x_{24}^{2} - x_{12} x_{14} x_{24} + x_{14}^{2}\right)} x_{45}}{x_{12}^{3} x_{24}^{2} x_{45}},\\ 
	&a_{34} = \frac{x_{35}}{x_{12} x_{24} x_{45}}, a_{35} = 1,\\ 
	&a_{45} = -\frac{x_{12} x_{25} - {\left(x_{12} x_{24} - x_{14}\right)} x_{45}}{x_{12}^{2} x_{24}^{2} x_{45}}
\end{align*}

\begin{equation*}x=\left(\begin{matrix}
		1 & x_{12} & 0 & x_{14} & x_{15} \\
		0 & 1 & 0 & x_{24} & x_{25} \\
		0 & 0 & 1 & 0 & x_{35}  \\
		0 & 0 & 0 & 1 & x_{45}  \\
		0 & 0 & 0 & 0 & 1 \\
	\end{matrix}\right),x^A=\left(\begin{matrix}
		1 & 1 & 0 & 0 & 0 \\
		0 & 1 & 0 & 1 & 0 \\
		0 & 0 & 1 & 0 & 0 \\
		0 & 0 & 0 & 1 & 1 \\
		0 & 0 & 0 & 0 & 1
	\end{matrix}\right)
\end{equation*} \newline
Where matrix $A$ has entries
\begin{align*}
	&d_{1} = 1, d_{2} = \frac{1}{x_{12}}, d_{3} = 1, d_{4} = \frac{1}{x_{12} x_{24}}, d_{5} = \frac{1}{x_{12} x_{24} x_{45}},\\ 
	&a_{12} = 1, a_{13} = 1, a_{14} = 1, a_{15} = 1,\\ 
	&a_{23} = 0, a_{24} = \frac{x_{12} x_{24} - x_{14}}{x_{12}^{2} x_{24}}, a_{25} = -\frac{x_{12} x_{15} x_{24} - x_{12} x_{14} x_{25} - {\left(x_{12}^{2} x_{24}^{2} - x_{12} x_{14} x_{24} + x_{14}^{2}\right)} x_{45}}{x_{12}^{3} x_{24}^{2} x_{45}},\\ 
	&a_{34} = \frac{x_{35}}{x_{12} x_{24} x_{45}}, a_{35} = 1,\\ 
	&a_{45} = -\frac{x_{12} x_{25} - {\left(x_{12} x_{24} - x_{14}\right)} x_{45}}{x_{12}^{2} x_{24}^{2} x_{45}}
\end{align*}

\begin{equation*}x=\left(\begin{matrix}
		1 & x_{12} & x_{13} & 0 & 0 \\
		0 & 1 & 0 & x_{24} & x_{25} \\
		0 & 0 & 1 & 0 & x_{35}  \\
		0 & 0 & 0 & 1 & x_{45}  \\
		0 & 0 & 0 & 0 & 1 \\
	\end{matrix}\right),x^A=\left(\begin{matrix}
		1 & 1 & 0 & 0 & 0 \\
		0 & 1 & 0 & 1 & 0 \\
		0 & 0 & 1 & 0 & 0 \\
		0 & 0 & 0 & 1 & 1 \\
		0 & 0 & 0 & 0 & 1
	\end{matrix}\right)
\end{equation*} \newline
Where matrix $A$ has entries
\begin{align*}
	&d_{1} = 1, d_{2} = \frac{1}{x_{12}}, d_{3} = 1, d_{4} = \frac{1}{x_{12} x_{24}}, d_{5} = \frac{1}{x_{12} x_{24} x_{45}},\\ 
	&a_{12} = 1, a_{13} = 1, a_{14} = 1, a_{15} = 1,\\ 
	&a_{23} = -\frac{x_{13}}{x_{12}}, a_{24} = \frac{x_{12} x_{24} x_{45} - x_{13} x_{35}}{x_{12}^{2} x_{24} x_{45}}, a_{25} = 1,\\ 
	&a_{34} = \frac{x_{35}}{x_{12} x_{24} x_{45}}, a_{35} = -\frac{x_{12} - 1}{x_{13}},\\ 
	&a_{45} = \frac{x_{12} x_{24} x_{45} - x_{12} x_{25} - x_{13} x_{35}}{x_{12}^{2} x_{24}^{2} x_{45}}
\end{align*}

\begin{equation*}x=\left(\begin{matrix}
		1 & x_{12} & x_{13} & 0 & x_{15} \\
		0 & 1 & 0 & x_{24} & x_{25} \\
		0 & 0 & 1 & 0 & x_{35}  \\
		0 & 0 & 0 & 1 & x_{45}  \\
		0 & 0 & 0 & 0 & 1 \\
	\end{matrix}\right),x^A=\left(\begin{matrix}
		1 & 1 & 0 & 0 & 0 \\
		0 & 1 & 0 & 1 & 0 \\
		0 & 0 & 1 & 0 & 0 \\
		0 & 0 & 0 & 1 & 1 \\
		0 & 0 & 0 & 0 & 1
	\end{matrix}\right)
\end{equation*} \newline
Where matrix $A$ has entries
\begin{align*}
	&d_{1} = 1, d_{2} = \frac{1}{x_{12}}, d_{3} = 1, d_{4} = \frac{1}{x_{12} x_{24}}, d_{5} = \frac{1}{x_{12} x_{24} x_{45}},\\ 
	&a_{12} = 1, a_{13} = 1, a_{14} = 1, a_{15} = 1,\\ 
	&a_{23} = -\frac{x_{13}}{x_{12}}, a_{24} = \frac{x_{12} x_{24} x_{45} - x_{13} x_{35}}{x_{12}^{2} x_{24} x_{45}}, a_{25} = 1,\\ 
	&a_{34} = \frac{x_{35}}{x_{12} x_{24} x_{45}}, a_{35} = -\frac{{\left(x_{12}^{2} - x_{12}\right)} x_{24} x_{45} + x_{15}}{x_{12} x_{13} x_{24} x_{45}},\\ 
	&a_{45} = \frac{x_{12} x_{24} x_{45} - x_{12} x_{25} - x_{13} x_{35}}{x_{12}^{2} x_{24}^{2} x_{45}}
\end{align*}

\begin{equation*}x=\left(\begin{matrix}
		1 & x_{12} & x_{13} & x_{14} & 0 \\
		0 & 1 & 0 & x_{24} & x_{25} \\
		0 & 0 & 1 & 0 & x_{35}  \\
		0 & 0 & 0 & 1 & x_{45}  \\
		0 & 0 & 0 & 0 & 1 \\
	\end{matrix}\right),x^A=\left(\begin{matrix}
		1 & 1 & 0 & 0 & 0 \\
		0 & 1 & 0 & 1 & 0 \\
		0 & 0 & 1 & 0 & 0 \\
		0 & 0 & 0 & 1 & 1 \\
		0 & 0 & 0 & 0 & 1
	\end{matrix}\right)
\end{equation*} \newline
Where matrix $A$ has entries
\begin{align*}
	&d_{1} = 1, d_{2} = \frac{1}{x_{12}}, d_{3} = 1, d_{4} = \frac{1}{x_{12} x_{24}}, d_{5} = \frac{1}{x_{12} x_{24} x_{45}},\\ 
	&a_{12} = 1, a_{13} = 1, a_{14} = 1, a_{15} = 1,\\ 
	&a_{23} = -\frac{x_{13}}{x_{12}}, a_{24} = -\frac{x_{13} x_{35} - {\left(x_{12} x_{24} - x_{14}\right)} x_{45}}{x_{12}^{2} x_{24} x_{45}}, a_{25} = 1,\\ 
	&a_{34} = \frac{x_{35}}{x_{12} x_{24} x_{45}}, a_{35} = \frac{x_{12} x_{14} x_{25} + x_{13} x_{14} x_{35} - {\left(x_{12} x_{14} x_{24} - x_{14}^{2} + {\left(x_{12}^{3} - x_{12}^{2}\right)} x_{24}^{2}\right)} x_{45}}{x_{12}^{2} x_{13} x_{24}^{2} x_{45}},\\ 
	&a_{45} = -\frac{x_{12} x_{25} + x_{13} x_{35} - {\left(x_{12} x_{24} - x_{14}\right)} x_{45}}{x_{12}^{2} x_{24}^{2} x_{45}}
\end{align*}

\begin{equation*}x=\left(\begin{matrix}
		1 & x_{12} & x_{13} & x_{14} & x_{15} \\
		0 & 1 & 0 & x_{24} & x_{25} \\
		0 & 0 & 1 & 0 & x_{35}  \\
		0 & 0 & 0 & 1 & x_{45}  \\
		0 & 0 & 0 & 0 & 1 \\
	\end{matrix}\right),x^A=\left(\begin{matrix}
		1 & 1 & 0 & 0 & 0 \\
		0 & 1 & 0 & 1 & 0 \\
		0 & 0 & 1 & 0 & 0 \\
		0 & 0 & 0 & 1 & 1 \\
		0 & 0 & 0 & 0 & 1
	\end{matrix}\right)
\end{equation*} \newline
Where matrix $A$ has entries
\begin{align*}
	&d_{1} = 1, d_{2} = \frac{1}{x_{12}}, d_{3} = 1, d_{4} = \frac{1}{x_{12} x_{24}}, d_{5} = \frac{1}{x_{12} x_{24} x_{45}},\\ 
	&a_{12} = 1, a_{13} = 1, a_{14} = 1, a_{15} = 1,\\ 
	&a_{23} = -\frac{x_{13}}{x_{12}}, a_{24} = -\frac{x_{13} x_{35} - {\left(x_{12} x_{24} - x_{14}\right)} x_{45}}{x_{12}^{2} x_{24} x_{45}}, a_{25} = 1,\\ 
	&a_{34} = \frac{x_{35}}{x_{12} x_{24} x_{45}}, a_{35} = -\frac{x_{12} x_{15} x_{24} - x_{12} x_{14} x_{25} - x_{13} x_{14} x_{35} + {\left(x_{12} x_{14} x_{24} - x_{14}^{2} + {\left(x_{12}^{3} - x_{12}^{2}\right)} x_{24}^{2}\right)} x_{45}}{x_{12}^{2} x_{13} x_{24}^{2} x_{45}},\\ 
	&a_{45} = -\frac{x_{12} x_{25} + x_{13} x_{35} - {\left(x_{12} x_{24} - x_{14}\right)} x_{45}}{x_{12}^{2} x_{24}^{2} x_{45}}
\end{align*}

		\section{Subcases of $Y_9$}

\begin{equation*}x=\left(
\right)
\end{equation*} \newline
Where matrix $A$ has entries
\begin{align*}
	&d_{1} = 1, d_{2} = \frac{1}{x_{12}}, d_{3} = \frac{1}{x_{12} x_{23}}, d_{4} = 1, d_{5} = \frac{1}{x_{45}},\\ 
	&a_{12} = 1, a_{13} = 1, a_{14} = 1, a_{15} = 1,\\ 
	&a_{23} = \frac{x_{12} x_{23} - x_{13}}{x_{12}^{2} x_{23}}, a_{24} = -\frac{x_{14}}{x_{12}}, a_{25} = 1,\\ 
	&a_{34} = 0, a_{35} = -\frac{x_{14}}{x_{12} x_{23}},\\ 
	&a_{45} = -\frac{x_{12} x_{15} x_{23} - {\left(x_{13} x_{14} - {\left(x_{12}^{2} - x_{12}\right)} x_{23}\right)} x_{45}}{x_{12} x_{14} x_{23} x_{45}}
\end{align*}

\begin{equation*}x=\left(
\right)
\end{equation*} \newline
Where matrix $A$ has entries
\begin{align*}
	&d_{1} = 1, d_{2} = \frac{1}{x_{12}}, d_{3} = \frac{1}{x_{12} x_{23}}, d_{4} = 1, d_{5} = \frac{1}{x_{45}},\\ 
	&a_{12} = 1, a_{13} = 1, a_{14} = 1, a_{15} = 1,\\ 
	&a_{23} = \frac{x_{12} x_{23} - x_{13}}{x_{12}^{2} x_{23}}, a_{24} = -\frac{x_{14}}{x_{12}}, a_{25} = 1,\\ 
	&a_{34} = 0, a_{35} = -\frac{x_{12} x_{25} + x_{14} x_{45}}{x_{12} x_{23} x_{45}},\\ 
	&a_{45} = \frac{x_{12} x_{13} x_{25} + {\left(x_{13} x_{14} - {\left(x_{12}^{2} - x_{12}\right)} x_{23}\right)} x_{45}}{x_{12} x_{14} x_{23} x_{45}}
\end{align*}

\begin{equation*}x=\left(\begin{matrix}
		1 & x_{12} & x_{13} & x_{14} & x_{15} \\
		0 & 1 & x_{23} & 0 & x_{25} \\
		0 & 0 & 1 & 0 & 0  \\
		0 & 0 & 0 & 1 & x_{45}  \\
		0 & 0 & 0 & 0 & 1 \\
	\end{matrix}\right),x^A=\left(\begin{matrix}
		1 & 1 & 0 & 0 & 0 \\
		0 & 1 & 1 & 0 & 0 \\
		0 & 0 & 1 & 0 & 0 \\
		0 & 0 & 0 & 1 & 1 \\
		0 & 0 & 0 & 0 & 1
	\end{matrix}\right)
\end{equation*} \newline
Where matrix $A$ has entries
\begin{align*}
	&d_{1} = 1, d_{2} = \frac{1}{x_{12}}, d_{3} = \frac{1}{x_{12} x_{23}}, d_{4} = 1, d_{5} = \frac{1}{x_{45}},\\ 
	&a_{12} = 1, a_{13} = 1, a_{14} = 1, a_{15} = 1,\\ 
	&a_{23} = \frac{x_{12} x_{23} - x_{13}}{x_{12}^{2} x_{23}}, a_{24} = -\frac{x_{14}}{x_{12}}, a_{25} = 1,\\ 
	&a_{34} = 0, a_{35} = -\frac{x_{12} x_{25} + x_{14} x_{45}}{x_{12} x_{23} x_{45}},\\ 
	&a_{45} = -\frac{x_{12} x_{15} x_{23} - x_{12} x_{13} x_{25} - {\left(x_{13} x_{14} - {\left(x_{12}^{2} - x_{12}\right)} x_{23}\right)} x_{45}}{x_{12} x_{14} x_{23} x_{45}}
\end{align*}

\begin{equation*}x=\left(
\right)
\end{equation*} \newline
Where matrix $A$ has entries
\begin{align*}
	&d_{1} = 1, d_{2} = \frac{1}{x_{12}}, d_{3} = \frac{1}{x_{12} x_{23}}, d_{4} = \frac{1}{x_{12} x_{24}}, d_{5} = \frac{1}{x_{12} x_{24} x_{45}},\\ 
	&a_{12} = 1, a_{13} = 1, a_{14} = 1, a_{15} = 1,\\ 
	&a_{23} = \frac{1}{x_{12}}, a_{24} = -\frac{x_{14}}{x_{12}^{2} x_{24}}, a_{25} = 1,\\ 
	&a_{34} = -\frac{1}{x_{12} x_{23}}, a_{35} = \frac{x_{12} x_{14} x_{24} - x_{14}^{2} + {\left(x_{12}^{3} - {2} \, x_{12}^{2}\right)} x_{24}^{2}}{x_{12}^{2} x_{14} x_{23} x_{24}},\\ 
	&a_{45} = -\frac{x_{12} - 2}{x_{14}}
\end{align*}

\begin{equation*}x=\left(\begin{matrix}
		1 & x_{12} & 0 & x_{14} & x_{15} \\
		0 & 1 & x_{23} & x_{24} & 0 \\
		0 & 0 & 1 & 0 & 0  \\
		0 & 0 & 0 & 1 & x_{45}  \\
		0 & 0 & 0 & 0 & 1 \\
	\end{matrix}\right),x^A=\left(\begin{matrix}
		1 & 1 & 0 & 0 & 0 \\
		0 & 1 & 1 & 0 & 0 \\
		0 & 0 & 1 & 0 & 1 \\
		0 & 0 & 0 & 1 & 1 \\
		0 & 0 & 0 & 0 & 1
	\end{matrix}\right)
\end{equation*} \newline
Where matrix $A$ has entries
\begin{align*}
	&d_{1} = 1, d_{2} = \frac{1}{x_{12}}, d_{3} = \frac{1}{x_{12} x_{23}}, d_{4} = \frac{1}{x_{12} x_{24}}, d_{5} = \frac{1}{x_{12} x_{24} x_{45}},\\ 
	&a_{12} = 1, a_{13} = 1, a_{14} = 1, a_{15} = 1,\\ 
	&a_{23} = \frac{1}{x_{12}}, a_{24} = -\frac{x_{14}}{x_{12}^{2} x_{24}}, a_{25} = 1,\\ 
	&a_{34} = -\frac{1}{x_{12} x_{23}}, a_{35} = \frac{x_{12} x_{15} x_{24} + {\left(x_{12} x_{14} x_{24} - x_{14}^{2} + {\left(x_{12}^{3} - {2} \, x_{12}^{2}\right)} x_{24}^{2}\right)} x_{45}}{x_{12}^{2} x_{14} x_{23} x_{24} x_{45}},\\ 
	&a_{45} = -\frac{{\left(x_{12}^{2} - {2} \, x_{12}\right)} x_{24} x_{45} + x_{15}}{x_{12} x_{14} x_{24} x_{45}}
\end{align*}

\begin{equation*}x=\left(\begin{matrix}
		1 & x_{12} & x_{13} & 0 & 0 \\
		0 & 1 & x_{23} & x_{24} & 0 \\
		0 & 0 & 1 & 0 & 0  \\
		0 & 0 & 0 & 1 & x_{45}  \\
		0 & 0 & 0 & 0 & 1 \\
	\end{matrix}\right),x^A=\left(\begin{matrix}
		1 & 1 & 0 & 0 & 0 \\
		0 & 1 & 1 & 0 & 0 \\
		0 & 0 & 1 & 0 & 1 \\
		0 & 0 & 0 & 1 & 1 \\
		0 & 0 & 0 & 0 & 1
	\end{matrix}\right)
\end{equation*} \newline
Where matrix $A$ has entries
\begin{align*}
	&d_{1} = 1, d_{2} = \frac{1}{x_{12}}, d_{3} = \frac{1}{x_{12} x_{23}}, d_{4} = \frac{1}{x_{12} x_{24}}, d_{5} = \frac{1}{x_{12} x_{24} x_{45}},\\ 
	&a_{12} = 1, a_{13} = 1, a_{14} = 1, a_{15} = 1,\\ 
	&a_{23} = \frac{x_{12} x_{23} - x_{13}}{x_{12}^{2} x_{23}}, a_{24} = \frac{x_{13}}{x_{12}^{2} x_{23}}, a_{25} = 1,\\ 
	&a_{34} = -\frac{1}{x_{12} x_{23}}, a_{35} = -\frac{x_{12} - 2}{x_{13}},\\ 
	&a_{45} = \frac{x_{13} + {\left(x_{12}^{2} - {2} \, x_{12}\right)} x_{23}}{x_{12} x_{13} x_{24}}
\end{align*}

\begin{equation*}x=\left(\begin{matrix}
		1 & x_{12} & x_{13} & 0 & x_{15} \\
		0 & 1 & x_{23} & x_{24} & 0 \\
		0 & 0 & 1 & 0 & 0  \\
		0 & 0 & 0 & 1 & x_{45}  \\
		0 & 0 & 0 & 0 & 1 \\
	\end{matrix}\right),x^A=\left(\begin{matrix}
		1 & 1 & 0 & 0 & 0 \\
		0 & 1 & 1 & 0 & 0 \\
		0 & 0 & 1 & 0 & 1 \\
		0 & 0 & 0 & 1 & 1 \\
		0 & 0 & 0 & 0 & 1
	\end{matrix}\right)
\end{equation*} \newline
Where matrix $A$ has entries
\begin{align*}
	&d_{1} = 1, d_{2} = \frac{1}{x_{12}}, d_{3} = \frac{1}{x_{12} x_{23}}, d_{4} = \frac{1}{x_{12} x_{24}}, d_{5} = \frac{1}{x_{12} x_{24} x_{45}},\\ 
	&a_{12} = 1, a_{13} = 1, a_{14} = 1, a_{15} = 1,\\ 
	&a_{23} = \frac{x_{12} x_{23} - x_{13}}{x_{12}^{2} x_{23}}, a_{24} = \frac{x_{13}}{x_{12}^{2} x_{23}}, a_{25} = 1,\\ 
	&a_{34} = -\frac{1}{x_{12} x_{23}}, a_{35} = -\frac{{\left(x_{12}^{2} - {2} \, x_{12}\right)} x_{24} x_{45} + x_{15}}{x_{12} x_{13} x_{24} x_{45}},\\ 
	&a_{45} = \frac{x_{15} x_{23} + {\left(x_{13} + {\left(x_{12}^{2} - {2} \, x_{12}\right)} x_{23}\right)} x_{24} x_{45}}{x_{12} x_{13} x_{24}^{2} x_{45}}
\end{align*}


First assume $x_{14}\neq \frac{x_{13}x_{24}}{x_{23}}$
\begin{equation*}x=\left(\begin{matrix}
		1 & x_{12} & x_{13} & x_{14} & 0 \\
		0 & 1 & x_{23} & x_{24} & 0 \\
		0 & 0 & 1 & 0 & 0  \\
		0 & 0 & 0 & 1 & x_{45}  \\
		0 & 0 & 0 & 0 & 1 \\
	\end{matrix}\right),x^A=\left(\begin{matrix}
		1 & 1 & 0 & 0 & 0 \\
		0 & 1 & 1 & 0 & 0 \\
		0 & 0 & 1 & 0 & 1 \\
		0 & 0 & 0 & 1 & 1 \\
		0 & 0 & 0 & 0 & 1
	\end{matrix}\right)
\end{equation*} \newline
Where matrix $A$ has entries
\begin{align*}
	&d_{1} = 1, d_{2} = \frac{1}{x_{12}}, d_{3} = \frac{1}{x_{12} x_{23}}, d_{4} = \frac{1}{x_{12} x_{24}}, d_{5} = \frac{1}{x_{12} x_{24} x_{45}},\\ 
	&a_{12} = 1, a_{13} = 1, a_{14} = 1, a_{15} = 1,\\ 
	&a_{23} = \frac{x_{12} x_{23} - x_{13}}{x_{12}^{2} x_{23}}, a_{24} = -\frac{x_{14} x_{23} - x_{13} x_{24}}{x_{12}^{2} x_{23} x_{24}}, a_{25} = 1,\\ 
	&a_{34} = -\frac{1}{x_{12} x_{23}}, a_{35} = \frac{x_{12} x_{14} x_{24} - x_{14}^{2} + {\left(x_{12}^{3} - {2} \, x_{12}^{2}\right)} x_{24}^{2}}{x_{12}^{2} x_{14} x_{23} x_{24} - x_{12}^{2} x_{13} x_{24}^{2}},\\ 
	&a_{45} = \frac{x_{13} x_{14} - {\left(x_{12} x_{13} + {\left(x_{12}^{3} - {2} \, x_{12}^{2}\right)} x_{23}\right)} x_{24}}{x_{12}^{2} x_{14} x_{23} x_{24} - x_{12}^{2} x_{13} x_{24}^{2}}
\end{align*}

Now assume $x_{14}= \frac{x_{13}x_{24}}{x_{23}}$ 

\begin{equation*}x=\left(\begin{matrix}
		1 & x_{12} & x_{13} & x_{14} & 0 \\
		0 & 1 & x_{23} & x_{24} & 0 \\
		0 & 0 & 1 & 0 & 0  \\
		0 & 0 & 0 & 1 & x_{45}  \\
		0 & 0 & 0 & 0 & 1 \\
	\end{matrix}\right),x^A=\left(\begin{matrix}
		1 & 1 & 0 & 0 & 0 \\
		0 & 1 & 1 & 0 & 0 \\
		0 & 0 & 1 & 0 & 1 \\
		0 & 0 & 0 & 1 & 1 \\
		0 & 0 & 0 & 0 & 1
	\end{matrix}\right)
\end{equation*} \newline
Where matrix $A$ has entries
\begin{align*}
	&d_{1} = 1, d_{2} = \frac{1}{x_{12}}, d_{3} = \frac{1}{x_{12} x_{23}}, d_{4} = \frac{1}{x_{12} x_{24}}, d_{5} = \frac{1}{x_{12} x_{24} x_{45}},\\ 
	&a_{12} = 1, a_{13} = 1, a_{14} = 1, a_{15} = 1,\\ 
	&a_{23} = \frac{x_{12} x_{23} - x_{13}}{x_{12}^{2} x_{23}}, a_{24} = 0, a_{25} = \frac{{2} \, x_{12}^{2} x_{23}^{2} - x_{12} x_{13} x_{23} + x_{13}^{2}}{x_{12}^{3} x_{23}^{2}},\\ 
	&a_{34} = -\frac{1}{x_{12} x_{23}}, a_{35} = 1,\\ 
	&a_{45} = -\frac{x_{12}^{2} x_{23}^{2} - x_{12} x_{23} + x_{13}}{x_{12}^{2} x_{23} x_{24}}
\end{align*}


First assume $x_{14}\neq \frac{x_{13}x_{24}}{x_{23}}$
\begin{equation*}x=\left(\begin{matrix}
		1 & x_{12} & x_{13} & x_{14} & x_{15} \\
		0 & 1 & x_{23} & x_{24} & 0 \\
		0 & 0 & 1 & 0 & 0  \\
		0 & 0 & 0 & 1 & x_{45}  \\
		0 & 0 & 0 & 0 & 1 \\
	\end{matrix}\right),x^A=\left(\begin{matrix}
		1 & 1 & 0 & 0 & 0 \\
		0 & 1 & 1 & 0 & 0 \\
		0 & 0 & 1 & 0 & 1 \\
		0 & 0 & 0 & 1 & 1 \\
		0 & 0 & 0 & 0 & 1
	\end{matrix}\right)
\end{equation*} \newline
Where matrix $A$ has entries
\begin{align*}
	&d_{1} = 1, d_{2} = \frac{1}{x_{12}}, d_{3} = \frac{1}{x_{12} x_{23}}, d_{4} = \frac{1}{x_{12} x_{24}}, d_{5} = \frac{1}{x_{12} x_{24} x_{45}},\\ 
	&a_{12} = 1, a_{13} = 1, a_{14} = 1, a_{15} = 1,\\ 
	&a_{23} = \frac{x_{12} x_{23} - x_{13}}{x_{12}^{2} x_{23}}, a_{24} = -\frac{x_{14} x_{23} - x_{13} x_{24}}{x_{12}^{2} x_{23} x_{24}}, a_{25} = 1,\\ 
	&a_{34} = -\frac{1}{x_{12} x_{23}}, a_{35} = \frac{x_{12} x_{15} x_{24} + {\left(x_{12} x_{14} x_{24} - x_{14}^{2} + {\left(x_{12}^{3} - {2} \, x_{12}^{2}\right)} x_{24}^{2}\right)} x_{45}}{{\left(x_{12}^{2} x_{14} x_{23} x_{24} - x_{12}^{2} x_{13} x_{24}^{2}\right)} x_{45}},\\ 
	&a_{45} = -\frac{x_{12} x_{15} x_{23} - {\left(x_{13} x_{14} - {\left(x_{12} x_{13} + {\left(x_{12}^{3} - {2} \, x_{12}^{2}\right)} x_{23}\right)} x_{24}\right)} x_{45}}{{\left(x_{12}^{2} x_{14} x_{23} x_{24} - x_{12}^{2} x_{13} x_{24}^{2}\right)} x_{45}}
\end{align*}

Now assume $x_{14}= \frac{x_{13}x_{24}}{x_{23}}$
\begin{equation*}x=\left(\begin{matrix}
		1 & x_{12} & x_{13} & x_{14} & x_{15} \\
		0 & 1 & x_{23} & x_{24} & 0 \\
		0 & 0 & 1 & 0 & 0  \\
		0 & 0 & 0 & 1 & x_{45}  \\
		0 & 0 & 0 & 0 & 1 \\
	\end{matrix}\right),x^A=\left(\begin{matrix}
		1 & 1 & 0 & 0 & 0 \\
		0 & 1 & 1 & 0 & 0 \\
		0 & 0 & 1 & 0 & 1 \\
		0 & 0 & 0 & 1 & 1 \\
		0 & 0 & 0 & 0 & 1
	\end{matrix}\right)
\end{equation*} \newline
Where matrix $A$ has entries
\begin{align*}
	&d_{1} = 1, d_{2} = \frac{1}{x_{12}}, d_{3} = \frac{1}{x_{12} x_{23}}, d_{4} = \frac{1}{x_{12} x_{24}}, d_{5} = \frac{1}{x_{12} x_{24} x_{45}},\\ 
	&a_{12} = 1, a_{13} = 1, a_{14} = 1, a_{15} = 1,\\ 
	&a_{23} = \frac{x_{12} x_{23} - x_{13}}{x_{12}^{2} x_{23}}, a_{24} = 0, a_{25} = -\frac{x_{12} x_{15} x_{23}^{2} - {\left({2} \, x_{12}^{2} x_{23}^{2} - x_{12} x_{13} x_{23} + x_{13}^{2}\right)} x_{24} x_{45}}{x_{12}^{3} x_{23}^{2} x_{24} x_{45}},\\ 
	&a_{34} = -\frac{1}{x_{12} x_{23}}, a_{35} = 1,\\ 
	&a_{45} = -\frac{x_{12}^{2} x_{23}^{2} - x_{12} x_{23} + x_{13}}{x_{12}^{2} x_{23} x_{24}}
\end{align*}
\begin{equation*}x=\left(\begin{matrix}
		1 & x_{12} & 0 & 0 & 0 \\
		0 & 1 & x_{23} & x_{24} & x_{25} \\
		0 & 0 & 1 & 0 & 0  \\
		0 & 0 & 0 & 1 & x_{45}  \\
		0 & 0 & 0 & 0 & 1 \\
	\end{matrix}\right),x^A=\left(\begin{matrix}
		1 & 1 & 0 & 0 & 0 \\
		0 & 1 & 1 & 0 & 0 \\
		0 & 0 & 1 & 0 & 1 \\
		0 & 0 & 0 & 1 & 1 \\
		0 & 0 & 0 & 0 & 1
	\end{matrix}\right)
\end{equation*} \newline
Where matrix $A$ has entries
\begin{align*}
	&d_{1} = 1, d_{2} = \frac{1}{x_{12}}, d_{3} = \frac{1}{x_{12} x_{23}}, d_{4} = \frac{1}{x_{12} x_{24}}, d_{5} = \frac{1}{x_{12} x_{24} x_{45}},\\ 
	&a_{12} = 1, a_{13} = 1, a_{14} = 1, a_{15} = 1,\\ 
	&a_{23} = \frac{1}{x_{12}}, a_{24} = 0, a_{25} = \frac{2}{x_{12}},\\ 
	&a_{34} = -\frac{1}{x_{12} x_{23}}, a_{35} = 1,\\ 
	&a_{45} = -\frac{{\left(x_{12} x_{23} - 1\right)} x_{24} x_{45} + x_{25}}{x_{12} x_{24}^{2} x_{45}}
\end{align*}

\begin{equation*}x=\left(\begin{matrix}
		1 & x_{12} & 0 & 0 & x_{15} \\
		0 & 1 & x_{23} & x_{24} & x_{25} \\
		0 & 0 & 1 & 0 & 0  \\
		0 & 0 & 0 & 1 & x_{45}  \\
		0 & 0 & 0 & 0 & 1 \\
	\end{matrix}\right),x^A=\left(\begin{matrix}
		1 & 1 & 0 & 0 & 0 \\
		0 & 1 & 1 & 0 & 0 \\
		0 & 0 & 1 & 0 & 1 \\
		0 & 0 & 0 & 1 & 1 \\
		0 & 0 & 0 & 0 & 1
	\end{matrix}\right)
\end{equation*} \newline
Where matrix $A$ has entries
\begin{align*}
	&d_{1} = 1, d_{2} = \frac{1}{x_{12}}, d_{3} = \frac{1}{x_{12} x_{23}}, d_{4} = \frac{1}{x_{12} x_{24}}, d_{5} = \frac{1}{x_{12} x_{24} x_{45}},\\ 
	&a_{12} = 1, a_{13} = 1, a_{14} = 1, a_{15} = 1,\\ 
	&a_{23} = \frac{1}{x_{12}}, a_{24} = 0, a_{25} = \frac{{2} \, x_{12} x_{24} x_{45} - x_{15}}{x_{12}^{2} x_{24} x_{45}},\\ 
	&a_{34} = -\frac{1}{x_{12} x_{23}}, a_{35} = 1,\\ 
	&a_{45} = -\frac{{\left(x_{12} x_{23} - 1\right)} x_{24} x_{45} + x_{25}}{x_{12} x_{24}^{2} x_{45}}
\end{align*}

\begin{equation*}x=\left(\begin{matrix}
		1 & x_{12} & 0 & x_{14} & 0 \\
		0 & 1 & x_{23} & x_{24} & x_{25} \\
		0 & 0 & 1 & 0 & 0  \\
		0 & 0 & 0 & 1 & x_{45}  \\
		0 & 0 & 0 & 0 & 1 \\
	\end{matrix}\right),x^A=\left(\begin{matrix}
		1 & 1 & 0 & 0 & 0 \\
		0 & 1 & 1 & 0 & 0 \\
		0 & 0 & 1 & 0 & 1 \\
		0 & 0 & 0 & 1 & 1 \\
		0 & 0 & 0 & 0 & 1
	\end{matrix}\right)
\end{equation*} \newline
Where matrix $A$ has entries
\begin{align*}
	&d_{1} = 1, d_{2} = \frac{1}{x_{12}}, d_{3} = \frac{1}{x_{12} x_{23}}, d_{4} = \frac{1}{x_{12} x_{24}}, d_{5} = \frac{1}{x_{12} x_{24} x_{45}},\\ 
	&a_{12} = 1, a_{13} = 1, a_{14} = 1, a_{15} = 1,\\ 
	&a_{23} = \frac{1}{x_{12}}, a_{24} = -\frac{x_{14}}{x_{12}^{2} x_{24}}, a_{25} = 1,\\ 
	&a_{34} = -\frac{1}{x_{12} x_{23}}, a_{35} = -\frac{x_{12} x_{14} x_{25} - {\left(x_{12} x_{14} x_{24} - x_{14}^{2} + {\left(x_{12}^{3} - {2} \, x_{12}^{2}\right)} x_{24}^{2}\right)} x_{45}}{x_{12}^{2} x_{14} x_{23} x_{24} x_{45}},\\ 
	&a_{45} = -\frac{x_{12} - 2}{x_{14}}
\end{align*}

\begin{equation*}x=\left(\begin{matrix}
		1 & x_{12} & 0 & x_{14} & x_{15} \\
		0 & 1 & x_{23} & x_{24} & x_{25} \\
		0 & 0 & 1 & 0 & 0  \\
		0 & 0 & 0 & 1 & x_{45}  \\
		0 & 0 & 0 & 0 & 1 \\
	\end{matrix}\right),x^A=\left(\begin{matrix}
		1 & 1 & 0 & 0 & 0 \\
		0 & 1 & 1 & 0 & 0 \\
		0 & 0 & 1 & 0 & 1 \\
		0 & 0 & 0 & 1 & 1 \\
		0 & 0 & 0 & 0 & 1
	\end{matrix}\right)
\end{equation*} \newline
Where matrix $A$ has entries
\begin{align*}
	&d_{1} = 1, d_{2} = \frac{1}{x_{12}}, d_{3} = \frac{1}{x_{12} x_{23}}, d_{4} = \frac{1}{x_{12} x_{24}}, d_{5} = \frac{1}{x_{12} x_{24} x_{45}},\\ 
	&a_{12} = 1, a_{13} = 1, a_{14} = 1, a_{15} = 1,\\ 
	&a_{23} = \frac{1}{x_{12}}, a_{24} = -\frac{x_{14}}{x_{12}^{2} x_{24}}, a_{25} = 1,\\ 
	&a_{34} = -\frac{1}{x_{12} x_{23}}, a_{35} = \frac{x_{12} x_{15} x_{24} - x_{12} x_{14} x_{25} + {\left(x_{12} x_{14} x_{24} - x_{14}^{2} + {\left(x_{12}^{3} - {2} \, x_{12}^{2}\right)} x_{24}^{2}\right)} x_{45}}{x_{12}^{2} x_{14} x_{23} x_{24} x_{45}},\\ 
	&a_{45} = -\frac{{\left(x_{12}^{2} - {2} \, x_{12}\right)} x_{24} x_{45} + x_{15}}{x_{12} x_{14} x_{24} x_{45}}
\end{align*}

\begin{equation*}x=\left(\begin{matrix}
		1 & x_{12} & x_{13} & 0 & 0 \\
		0 & 1 & x_{23} & x_{24} & x_{25} \\
		0 & 0 & 1 & 0 & 0  \\
		0 & 0 & 0 & 1 & x_{45}  \\
		0 & 0 & 0 & 0 & 1 \\
	\end{matrix}\right),x^A=\left(\begin{matrix}
		1 & 1 & 0 & 0 & 0 \\
		0 & 1 & 1 & 0 & 0 \\
		0 & 0 & 1 & 0 & 1 \\
		0 & 0 & 0 & 1 & 1 \\
		0 & 0 & 0 & 0 & 1
	\end{matrix}\right)
\end{equation*} \newline
Where matrix $A$ has entries
\begin{align*}
	&d_{1} = 1, d_{2} = \frac{1}{x_{12}}, d_{3} = \frac{1}{x_{12} x_{23}}, d_{4} = \frac{1}{x_{12} x_{24}}, d_{5} = \frac{1}{x_{12} x_{24} x_{45}},\\ 
	&a_{12} = 1, a_{13} = 1, a_{14} = 1, a_{15} = 1,\\ 
	&a_{23} = \frac{x_{12} x_{23} - x_{13}}{x_{12}^{2} x_{23}}, a_{24} = \frac{x_{13}}{x_{12}^{2} x_{23}}, a_{25} = 1,\\ 
	&a_{34} = -\frac{1}{x_{12} x_{23}}, a_{35} = -\frac{x_{12} - 2}{x_{13}},\\ 
	&a_{45} = -\frac{x_{13} x_{25} - {\left(x_{13} + {\left(x_{12}^{2} - {2} \, x_{12}\right)} x_{23}\right)} x_{24} x_{45}}{x_{12} x_{13} x_{24}^{2} x_{45}}
\end{align*}

\begin{equation*}x=\left(\begin{matrix}
		1 & x_{12} & x_{13} & 0 & x_{15} \\
		0 & 1 & x_{23} & x_{24} & x_{25} \\
		0 & 0 & 1 & 0 & 0  \\
		0 & 0 & 0 & 1 & x_{45}  \\
		0 & 0 & 0 & 0 & 1 \\
	\end{matrix}\right),x^A=\left(\begin{matrix}
		1 & 1 & 0 & 0 & 0 \\
		0 & 1 & 1 & 0 & 0 \\
		0 & 0 & 1 & 0 & 1 \\
		0 & 0 & 0 & 1 & 1 \\
		0 & 0 & 0 & 0 & 1
	\end{matrix}\right)
\end{equation*} \newline
Where matrix $A$ has entries
\begin{align*}
	&d_{1} = 1, d_{2} = \frac{1}{x_{12}}, d_{3} = \frac{1}{x_{12} x_{23}}, d_{4} = \frac{1}{x_{12} x_{24}}, d_{5} = \frac{1}{x_{12} x_{24} x_{45}},\\ 
	&a_{12} = 1, a_{13} = 1, a_{14} = 1, a_{15} = 1,\\ 
	&a_{23} = \frac{x_{12} x_{23} - x_{13}}{x_{12}^{2} x_{23}}, a_{24} = \frac{x_{13}}{x_{12}^{2} x_{23}}, a_{25} = 1,\\ 
	&a_{34} = -\frac{1}{x_{12} x_{23}}, a_{35} = -\frac{{\left(x_{12}^{2} - {2} \, x_{12}\right)} x_{24} x_{45} + x_{15}}{x_{12} x_{13} x_{24} x_{45}},\\ 
	&a_{45} = \frac{x_{15} x_{23} - x_{13} x_{25} + {\left(x_{13} + {\left(x_{12}^{2} - {2} \, x_{12}\right)} x_{23}\right)} x_{24} x_{45}}{x_{12} x_{13} x_{24}^{2} x_{45}}
\end{align*}


First assume $x_{14}\neq \frac{x_{13}x_{24}}{x_{23}}$
\begin{equation*}x=\left(\begin{matrix}
		1 & x_{12} & x_{13} & x_{14} & 0 \\
		0 & 1 & x_{23} & x_{24} & x_{25} \\
		0 & 0 & 1 & 0 & 0  \\
		0 & 0 & 0 & 1 & x_{45}  \\
		0 & 0 & 0 & 0 & 1 \\
	\end{matrix}\right),x^A=\left(\begin{matrix}
		1 & 1 & 0 & 0 & 0 \\
		0 & 1 & 1 & 0 & 0 \\
		0 & 0 & 1 & 0 & 1 \\
		0 & 0 & 0 & 1 & 1 \\
		0 & 0 & 0 & 0 & 1
	\end{matrix}\right)
\end{equation*} \newline
Where matrix $A$ has entries
\begin{align*}
	&d_{1} = 1, d_{2} = \frac{1}{x_{12}}, d_{3} = \frac{1}{x_{12} x_{23}}, d_{4} = \frac{1}{x_{12} x_{24}}, d_{5} = \frac{1}{x_{12} x_{24} x_{45}},\\ 
	&a_{12} = 1, a_{13} = 1, a_{14} = 1, a_{15} = 1,\\ 
	&a_{23} = \frac{x_{12} x_{23} - x_{13}}{x_{12}^{2} x_{23}}, a_{24} = -\frac{x_{14} x_{23} - x_{13} x_{24}}{x_{12}^{2} x_{23} x_{24}}, a_{25} = 1,\\ 
	&a_{34} = -\frac{1}{x_{12} x_{23}}, a_{35} = -\frac{x_{12} x_{14} x_{25} - {\left(x_{12} x_{14} x_{24} - x_{14}^{2} + {\left(x_{12}^{3} - {2} \, x_{12}^{2}\right)} x_{24}^{2}\right)} x_{45}}{{\left(x_{12}^{2} x_{14} x_{23} x_{24} - x_{12}^{2} x_{13} x_{24}^{2}\right)} x_{45}},\\ 
	&a_{45} = \frac{x_{12} x_{13} x_{25} + {\left(x_{13} x_{14} - {\left(x_{12} x_{13} + {\left(x_{12}^{3} - {2} \, x_{12}^{2}\right)} x_{23}\right)} x_{24}\right)} x_{45}}{{\left(x_{12}^{2} x_{14} x_{23} x_{24} - x_{12}^{2} x_{13} x_{24}^{2}\right)} x_{45}}
\end{align*}

Now assume $x_{14}= \frac{x_{13}x_{24}}{x_{23}}$
\begin{equation*}x=\left(\begin{matrix}
		1 & x_{12} & x_{13} & x_{14} & 0 \\
		0 & 1 & x_{23} & x_{24} & x_{25} \\
		0 & 0 & 1 & 0 & 0  \\
		0 & 0 & 0 & 1 & x_{45}  \\
		0 & 0 & 0 & 0 & 1 \\
	\end{matrix}\right),x^A=\left(\begin{matrix}
		1 & 1 & 0 & 0 & 0 \\
		0 & 1 & 1 & 0 & 0 \\
		0 & 0 & 1 & 0 & 1 \\
		0 & 0 & 0 & 1 & 1 \\
		0 & 0 & 0 & 0 & 1
	\end{matrix}\right)
\end{equation*} \newline
Where matrix $A$ has entries
\begin{align*}
	&d_{1} = 1, d_{2} = \frac{1}{x_{12}}, d_{3} = \frac{1}{x_{12} x_{23}}, d_{4} = \frac{1}{x_{12} x_{24}}, d_{5} = \frac{1}{x_{12} x_{24} x_{45}},\\ 
	&a_{12} = 1, a_{13} = 1, a_{14} = 1, a_{15} = 1,\\ 
	&a_{23} = \frac{x_{12} x_{23} - x_{13}}{x_{12}^{2} x_{23}}, a_{24} = 0, a_{25} = \frac{x_{12} x_{13} x_{23} x_{25} + {\left({2} \, x_{12}^{2} x_{23}^{2} - x_{12} x_{13} x_{23} + x_{13}^{2}\right)} x_{24} x_{45}}{x_{12}^{3} x_{23}^{2} x_{24} x_{45}},\\ 
	&a_{34} = -\frac{1}{x_{12} x_{23}}, a_{35} = 1,\\ 
	&a_{45} = -\frac{x_{12} x_{23} x_{25} + {\left(x_{12}^{2} x_{23}^{2} - x_{12} x_{23} + x_{13}\right)} x_{24} x_{45}}{x_{12}^{2} x_{23} x_{24}^{2} x_{45}}
\end{align*}


First assume $x_{14}\neq \frac{x_{13}x_{24}}{x_{23}}$
\begin{equation*}x=\left(\begin{matrix}
		1 & x_{12} & x_{13} & x_{14} & x_{15} \\
		0 & 1 & x_{23} & x_{24} & x_{25} \\
		0 & 0 & 1 & 0 & 0  \\
		0 & 0 & 0 & 1 & x_{45}  \\
		0 & 0 & 0 & 0 & 1 \\
	\end{matrix}\right),x^A=\left(\begin{matrix}
		1 & 1 & 0 & 0 & 0 \\
		0 & 1 & 1 & 0 & 0 \\
		0 & 0 & 1 & 0 & 1 \\
		0 & 0 & 0 & 1 & 1 \\
		0 & 0 & 0 & 0 & 1
	\end{matrix}\right)
\end{equation*} \newline
Where matrix $A$ has entries
\begin{align*}
	&d_{1} = 1, d_{2} = \frac{1}{x_{12}}, d_{3} = \frac{1}{x_{12} x_{23}}, d_{4} = \frac{1}{x_{12} x_{24}}, d_{5} = \frac{1}{x_{12} x_{24} x_{45}},\\ 
	&a_{12} = 1, a_{13} = 1, a_{14} = 1, a_{15} = 1,\\ 
	&a_{23} = \frac{x_{12} x_{23} - x_{13}}{x_{12}^{2} x_{23}}, a_{24} = -\frac{x_{14} x_{23} - x_{13} x_{24}}{x_{12}^{2} x_{23} x_{24}}, a_{25} = 1,\\ 
	&a_{34} = -\frac{1}{x_{12} x_{23}}, a_{35} = \frac{x_{12} x_{15} x_{24} - x_{12} x_{14} x_{25} + {\left(x_{12} x_{14} x_{24} - x_{14}^{2} + {\left(x_{12}^{3} - {2} \, x_{12}^{2}\right)} x_{24}^{2}\right)} x_{45}}{{\left(x_{12}^{2} x_{14} x_{23} x_{24} - x_{12}^{2} x_{13} x_{24}^{2}\right)} x_{45}},\\ 
	&a_{45} = -\frac{x_{12} x_{15} x_{23} - x_{12} x_{13} x_{25} - {\left(x_{13} x_{14} - {\left(x_{12} x_{13} + {\left(x_{12}^{3} - {2} \, x_{12}^{2}\right)} x_{23}\right)} x_{24}\right)} x_{45}}{{\left(x_{12}^{2} x_{14} x_{23} x_{24} - x_{12}^{2} x_{13} x_{24}^{2}\right)} x_{45}}
\end{align*}
Now assume $x_{14}= \frac{x_{13}x_{24}}{x_{23}}$
\begin{equation*}x=\left(\begin{matrix}
		1 & x_{12} & x_{13} & x_{14} & x_{15} \\
		0 & 1 & x_{23} & x_{24} & x_{25} \\
		0 & 0 & 1 & 0 & 0  \\
		0 & 0 & 0 & 1 & x_{45}  \\
		0 & 0 & 0 & 0 & 1 \\
	\end{matrix}\right),x^A=\left(\begin{matrix}
		1 & 1 & 0 & 0 & 0 \\
		0 & 1 & 1 & 0 & 0 \\
		0 & 0 & 1 & 0 & 1 \\
		0 & 0 & 0 & 1 & 1 \\
		0 & 0 & 0 & 0 & 1
	\end{matrix}\right)
\end{equation*} \newline
Where matrix $A$ has entries
\begin{align*}
	&d_{1} = 1, d_{2} = \frac{1}{x_{12}}, d_{3} = \frac{1}{x_{12} x_{23}}, d_{4} = \frac{1}{x_{12} x_{24}}, d_{5} = \frac{1}{x_{12} x_{24} x_{45}},\\ 
	&a_{12} = 1, a_{13} = 1, a_{14} = 1, a_{15} = 1,\\ 
	&a_{23} = \frac{x_{12} x_{23} - x_{13}}{x_{12}^{2} x_{23}}, a_{24} = 0, a_{25} = -\frac{x_{12} x_{15} x_{23}^{2} - x_{12} x_{13} x_{23} x_{25} - {\left({2} \, x_{12}^{2} x_{23}^{2} - x_{12} x_{13} x_{23} + x_{13}^{2}\right)} x_{24} x_{45}}{x_{12}^{3} x_{23}^{2} x_{24} x_{45}},\\ 
	&a_{34} = -\frac{1}{x_{12} x_{23}}, a_{35} = 1,\\ 
	&a_{45} = -\frac{x_{12} x_{23} x_{25} + {\left(x_{12}^{2} x_{23}^{2} - x_{12} x_{23} + x_{13}\right)} x_{24} x_{45}}{x_{12}^{2} x_{23} x_{24}^{2} x_{45}}
\end{align*}

\begin{equation*}x=\left(
\right)
\end{equation*} \newline
Where matrix $A$ has entries
\begin{align*}
	&d_{1} = 1, d_{2} = \frac{1}{x_{12}}, d_{3} = \frac{1}{x_{12} x_{23}}, d_{4} = \frac{x_{45}}{x_{12} x_{23} x_{35}}, d_{5} = \frac{1}{x_{12} x_{23} x_{35}},\\ 
	&a_{12} = 1, a_{13} = 1, a_{14} = 1, a_{15} = 1,\\ 
	&a_{23} = \frac{1}{x_{12}}, a_{24} = -\frac{x_{14} x_{45}}{x_{12}^{2} x_{23} x_{35}}, a_{25} = 1,\\ 
	&a_{34} = 0, a_{35} = \frac{x_{12} x_{23} x_{35} - x_{14} x_{45}}{x_{12}^{2} x_{23}^{2} x_{35}},\\ 
	&a_{45} = -\frac{{\left(x_{12}^{2} - {2} \, x_{12}\right)} x_{23} x_{35} + x_{15}}{x_{12} x_{14} x_{23} x_{35}}
\end{align*}

\begin{equation*}x=\left(\begin{matrix}
		1 & x_{12} & x_{13} & 0 & 0 \\
		0 & 1 & x_{23} & 0 & 0 \\
		0 & 0 & 1 & 0 & x_{35}  \\
		0 & 0 & 0 & 1 & x_{45}  \\
		0 & 0 & 0 & 0 & 1 \\
	\end{matrix}\right),x^A=\left(\begin{matrix}
		1 & 1 & 0 & 0 & 0 \\
		0 & 1 & 1 & 0 & 0 \\
		0 & 0 & 1 & 0 & 1 \\
		0 & 0 & 0 & 1 & 1 \\
		0 & 0 & 0 & 0 & 1
	\end{matrix}\right)
\end{equation*} \newline
Where matrix $A$ has entries
\begin{align*}
	&d_{1} = 1, d_{2} = \frac{1}{x_{12}}, d_{3} = \frac{1}{x_{12} x_{23}}, d_{4} = \frac{x_{45}}{x_{12} x_{23} x_{35}}, d_{5} = \frac{1}{x_{12} x_{23} x_{35}},\\ 
	&a_{12} = 1, a_{13} = 1, a_{14} = 1, a_{15} = 1,\\ 
	&a_{23} = \frac{x_{12} x_{23} - x_{13}}{x_{12}^{2} x_{23}}, a_{24} = 0, a_{25} = \frac{{2} \, x_{12}^{2} x_{23}^{2} - x_{12} x_{13} x_{23} + x_{13}^{2}}{x_{12}^{3} x_{23}^{2}},\\ 
	&a_{34} = 0, a_{35} = \frac{x_{12} x_{23} - x_{13}}{x_{12}^{2} x_{23}^{2}},\\ 
	&a_{45} = 1
\end{align*}

\begin{equation*}x=\left(\begin{matrix}
		1 & x_{12} & x_{13} & 0 & x_{15} \\
		0 & 1 & x_{23} & 0 & 0 \\
		0 & 0 & 1 & 0 & x_{35}  \\
		0 & 0 & 0 & 1 & x_{45}  \\
		0 & 0 & 0 & 0 & 1 \\
	\end{matrix}\right),x^A=\left(\begin{matrix}
		1 & 1 & 0 & 0 & 0 \\
		0 & 1 & 1 & 0 & 0 \\
		0 & 0 & 1 & 0 & 1 \\
		0 & 0 & 0 & 1 & 1 \\
		0 & 0 & 0 & 0 & 1
	\end{matrix}\right)
\end{equation*} \newline
Where matrix $A$ has entries
\begin{align*}
	&d_{1} = 1, d_{2} = \frac{1}{x_{12}}, d_{3} = \frac{1}{x_{12} x_{23}}, d_{4} = \frac{x_{45}}{x_{12} x_{23} x_{35}}, d_{5} = \frac{1}{x_{12} x_{23} x_{35}},\\ 
	&a_{12} = 1, a_{13} = 1, a_{14} = 1, a_{15} = 1,\\ 
	&a_{23} = \frac{x_{12} x_{23} - x_{13}}{x_{12}^{2} x_{23}}, a_{24} = 0, a_{25} = -\frac{x_{12} x_{15} x_{23} - {\left({2} \, x_{12}^{2} x_{23}^{2} - x_{12} x_{13} x_{23} + x_{13}^{2}\right)} x_{35}}{x_{12}^{3} x_{23}^{2} x_{35}},\\ 
	&a_{34} = 0, a_{35} = \frac{x_{12} x_{23} - x_{13}}{x_{12}^{2} x_{23}^{2}},\\ 
	&a_{45} = 1
\end{align*}

\begin{equation*}x=\left(\begin{matrix}
		1 & x_{12} & x_{13} & x_{14} & 0 \\
		0 & 1 & x_{23} & 0 & 0 \\
		0 & 0 & 1 & 0 & x_{35}  \\
		0 & 0 & 0 & 1 & x_{45}  \\
		0 & 0 & 0 & 0 & 1 \\
	\end{matrix}\right),x^A=\left(\begin{matrix}
		1 & 1 & 0 & 0 & 0 \\
		0 & 1 & 1 & 0 & 0 \\
		0 & 0 & 1 & 0 & 1 \\
		0 & 0 & 0 & 1 & 1 \\
		0 & 0 & 0 & 0 & 1
	\end{matrix}\right)
\end{equation*} \newline
Where matrix $A$ has entries
\begin{align*}
	&d_{1} = 1, d_{2} = \frac{1}{x_{12}}, d_{3} = \frac{1}{x_{12} x_{23}}, d_{4} = \frac{x_{45}}{x_{12} x_{23} x_{35}}, d_{5} = \frac{1}{x_{12} x_{23} x_{35}},\\ 
	&a_{12} = 1, a_{13} = 1, a_{14} = 1, a_{15} = 1,\\ 
	&a_{23} = \frac{x_{12} x_{23} - x_{13}}{x_{12}^{2} x_{23}}, a_{24} = -\frac{x_{14} x_{45}}{x_{12}^{2} x_{23} x_{35}}, a_{25} = 1,\\ 
	&a_{34} = 0, a_{35} = -\frac{x_{14} x_{45} - {\left(x_{12} x_{23} - x_{13}\right)} x_{35}}{x_{12}^{2} x_{23}^{2} x_{35}},\\ 
	&a_{45} = \frac{x_{13} x_{14} x_{45} - {\left(x_{12} x_{13} x_{23} - x_{13}^{2} + {\left(x_{12}^{3} - {2} \, x_{12}^{2}\right)} x_{23}^{2}\right)} x_{35}}{x_{12}^{2} x_{14} x_{23}^{2} x_{35}}
\end{align*}

\begin{equation*}x=\left(\begin{matrix}
		1 & x_{12} & x_{13} & x_{14} & x_{15} \\
		0 & 1 & x_{23} & 0 & 0 \\
		0 & 0 & 1 & 0 & x_{35}  \\
		0 & 0 & 0 & 1 & x_{45}  \\
		0 & 0 & 0 & 0 & 1 \\
	\end{matrix}\right),x^A=\left(\begin{matrix}
		1 & 1 & 0 & 0 & 0 \\
		0 & 1 & 1 & 0 & 0 \\
		0 & 0 & 1 & 0 & 1 \\
		0 & 0 & 0 & 1 & 1 \\
		0 & 0 & 0 & 0 & 1
	\end{matrix}\right)
\end{equation*} \newline
Where matrix $A$ has entries
\begin{align*}
	&d_{1} = 1, d_{2} = \frac{1}{x_{12}}, d_{3} = \frac{1}{x_{12} x_{23}}, d_{4} = \frac{x_{45}}{x_{12} x_{23} x_{35}}, d_{5} = \frac{1}{x_{12} x_{23} x_{35}},\\ 
	&a_{12} = 1, a_{13} = 1, a_{14} = 1, a_{15} = 1,\\ 
	&a_{23} = \frac{x_{12} x_{23} - x_{13}}{x_{12}^{2} x_{23}}, a_{24} = -\frac{x_{14} x_{45}}{x_{12}^{2} x_{23} x_{35}}, a_{25} = 1,\\ 
	&a_{34} = 0, a_{35} = -\frac{x_{14} x_{45} - {\left(x_{12} x_{23} - x_{13}\right)} x_{35}}{x_{12}^{2} x_{23}^{2} x_{35}},\\ 
	&a_{45} = -\frac{x_{12} x_{15} x_{23} - x_{13} x_{14} x_{45} + {\left(x_{12} x_{13} x_{23} - x_{13}^{2} + {\left(x_{12}^{3} - {2} \, x_{12}^{2}\right)} x_{23}^{2}\right)} x_{35}}{x_{12}^{2} x_{14} x_{23}^{2} x_{35}}
\end{align*}

\begin{equation*}x=\left(
\right)
\end{equation*} \newline
Where matrix $A$ has entries
\begin{align*}
	&d_{1} = 1, d_{2} = \frac{1}{x_{12}}, d_{3} = \frac{1}{x_{12} x_{23}}, d_{4} = \frac{x_{45}}{x_{12} x_{23} x_{35}}, d_{5} = \frac{1}{x_{12} x_{23} x_{35}},\\ 
	&a_{12} = 1, a_{13} = 1, a_{14} = 1, a_{15} = 1,\\ 
	&a_{23} = \frac{1}{x_{12}}, a_{24} = -\frac{x_{14} x_{45}}{x_{12}^{2} x_{23} x_{35}}, a_{25} = 1,\\ 
	&a_{34} = 0, a_{35} = \frac{x_{12} x_{23} x_{35} - x_{12} x_{25} - x_{14} x_{45}}{x_{12}^{2} x_{23}^{2} x_{35}},\\ 
	&a_{45} = -\frac{x_{12} - 2}{x_{14}}
\end{align*}

\begin{equation*}x=\left(\begin{matrix}
		1 & x_{12} & 0 & x_{14} & x_{15} \\
		0 & 1 & x_{23} & 0 & x_{25} \\
		0 & 0 & 1 & 0 & x_{35}  \\
		0 & 0 & 0 & 1 & x_{45}  \\
		0 & 0 & 0 & 0 & 1 \\
	\end{matrix}\right),x^A=\left(\begin{matrix}
		1 & 1 & 0 & 0 & 0 \\
		0 & 1 & 1 & 0 & 0 \\
		0 & 0 & 1 & 0 & 1 \\
		0 & 0 & 0 & 1 & 1 \\
		0 & 0 & 0 & 0 & 1
	\end{matrix}\right)
\end{equation*} \newline
Where matrix $A$ has entries
\begin{align*}
	&d_{1} = 1, d_{2} = \frac{1}{x_{12}}, d_{3} = \frac{1}{x_{12} x_{23}}, d_{4} = \frac{x_{45}}{x_{12} x_{23} x_{35}}, d_{5} = \frac{1}{x_{12} x_{23} x_{35}},\\ 
	&a_{12} = 1, a_{13} = 1, a_{14} = 1, a_{15} = 1,\\ 
	&a_{23} = \frac{1}{x_{12}}, a_{24} = -\frac{x_{14} x_{45}}{x_{12}^{2} x_{23} x_{35}}, a_{25} = 1,\\ 
	&a_{34} = 0, a_{35} = \frac{x_{12} x_{23} x_{35} - x_{12} x_{25} - x_{14} x_{45}}{x_{12}^{2} x_{23}^{2} x_{35}},\\ 
	&a_{45} = -\frac{{\left(x_{12}^{2} - {2} \, x_{12}\right)} x_{23} x_{35} + x_{15}}{x_{12} x_{14} x_{23} x_{35}}
\end{align*}

\begin{equation*}x=\left(\begin{matrix}
		1 & x_{12} & x_{13} & 0 & 0 \\
		0 & 1 & x_{23} & 0 & x_{25} \\
		0 & 0 & 1 & 0 & x_{35}  \\
		0 & 0 & 0 & 1 & x_{45}  \\
		0 & 0 & 0 & 0 & 1 \\
	\end{matrix}\right),x^A=\left(\begin{matrix}
		1 & 1 & 0 & 0 & 0 \\
		0 & 1 & 1 & 0 & 0 \\
		0 & 0 & 1 & 0 & 1 \\
		0 & 0 & 0 & 1 & 1 \\
		0 & 0 & 0 & 0 & 1
	\end{matrix}\right)
\end{equation*} \newline
Where matrix $A$ has entries
\begin{align*}
	&d_{1} = 1, d_{2} = \frac{1}{x_{12}}, d_{3} = \frac{1}{x_{12} x_{23}}, d_{4} = \frac{x_{45}}{x_{12} x_{23} x_{35}}, d_{5} = \frac{1}{x_{12} x_{23} x_{35}},\\ 
	&a_{12} = 1, a_{13} = 1, a_{14} = 1, a_{15} = 1,\\ 
	&a_{23} = \frac{x_{12} x_{23} - x_{13}}{x_{12}^{2} x_{23}}, a_{24} = 0, a_{25} = \frac{x_{12} x_{13} x_{25} + {\left({2} \, x_{12}^{2} x_{23}^{2} - x_{12} x_{13} x_{23} + x_{13}^{2}\right)} x_{35}}{x_{12}^{3} x_{23}^{2} x_{35}},\\ 
	&a_{34} = 0, a_{35} = -\frac{x_{12} x_{25} - {\left(x_{12} x_{23} - x_{13}\right)} x_{35}}{x_{12}^{2} x_{23}^{2} x_{35}},\\ 
	&a_{45} = 1
\end{align*}

\begin{equation*}x=\left(\begin{matrix}
		1 & x_{12} & x_{13} & 0 & x_{15} \\
		0 & 1 & x_{23} & 0 & x_{25} \\
		0 & 0 & 1 & 0 & x_{35}  \\
		0 & 0 & 0 & 1 & x_{45}  \\
		0 & 0 & 0 & 0 & 1 \\
	\end{matrix}\right),x^A=\left(\begin{matrix}
		1 & 1 & 0 & 0 & 0 \\
		0 & 1 & 1 & 0 & 0 \\
		0 & 0 & 1 & 0 & 1 \\
		0 & 0 & 0 & 1 & 1 \\
		0 & 0 & 0 & 0 & 1
	\end{matrix}\right)
\end{equation*} \newline
Where matrix $A$ has entries
\begin{align*}
	&d_{1} = 1, d_{2} = \frac{1}{x_{12}}, d_{3} = \frac{1}{x_{12} x_{23}}, d_{4} = \frac{x_{45}}{x_{12} x_{23} x_{35}}, d_{5} = \frac{1}{x_{12} x_{23} x_{35}},\\ 
	&a_{12} = 1, a_{13} = 1, a_{14} = 1, a_{15} = 1,\\ 
	&a_{23} = \frac{x_{12} x_{23} - x_{13}}{x_{12}^{2} x_{23}}, a_{24} = 0, a_{25} = -\frac{x_{12} x_{15} x_{23} - x_{12} x_{13} x_{25} - {\left({2} \, x_{12}^{2} x_{23}^{2} - x_{12} x_{13} x_{23} + x_{13}^{2}\right)} x_{35}}{x_{12}^{3} x_{23}^{2} x_{35}},\\ 
	&a_{34} = 0, a_{35} = -\frac{x_{12} x_{25} - {\left(x_{12} x_{23} - x_{13}\right)} x_{35}}{x_{12}^{2} x_{23}^{2} x_{35}},\\ 
	&a_{45} = 1
\end{align*}

\begin{equation*}x=\left(\begin{matrix}
		1 & x_{12} & x_{13} & x_{14} & 0 \\
		0 & 1 & x_{23} & 0 & x_{25} \\
		0 & 0 & 1 & 0 & x_{35}  \\
		0 & 0 & 0 & 1 & x_{45}  \\
		0 & 0 & 0 & 0 & 1 \\
	\end{matrix}\right),x^A=\left(\begin{matrix}
		1 & 1 & 0 & 0 & 0 \\
		0 & 1 & 1 & 0 & 0 \\
		0 & 0 & 1 & 0 & 1 \\
		0 & 0 & 0 & 1 & 1 \\
		0 & 0 & 0 & 0 & 1
	\end{matrix}\right)
\end{equation*} \newline
Where matrix $A$ has entries
\begin{align*}
	&d_{1} = 1, d_{2} = \frac{1}{x_{12}}, d_{3} = \frac{1}{x_{12} x_{23}}, d_{4} = \frac{x_{45}}{x_{12} x_{23} x_{35}}, d_{5} = \frac{1}{x_{12} x_{23} x_{35}},\\ 
	&a_{12} = 1, a_{13} = 1, a_{14} = 1, a_{15} = 1,\\ 
	&a_{23} = \frac{x_{12} x_{23} - x_{13}}{x_{12}^{2} x_{23}}, a_{24} = -\frac{x_{14} x_{45}}{x_{12}^{2} x_{23} x_{35}}, a_{25} = 1,\\ 
	&a_{34} = 0, a_{35} = -\frac{x_{12} x_{25} + x_{14} x_{45} - {\left(x_{12} x_{23} - x_{13}\right)} x_{35}}{x_{12}^{2} x_{23}^{2} x_{35}},\\ 
	&a_{45} = \frac{x_{12} x_{13} x_{25} + x_{13} x_{14} x_{45} - {\left(x_{12} x_{13} x_{23} - x_{13}^{2} + {\left(x_{12}^{3} - {2} \, x_{12}^{2}\right)} x_{23}^{2}\right)} x_{35}}{x_{12}^{2} x_{14} x_{23}^{2} x_{35}}
\end{align*}

\begin{equation*}x=\left(\begin{matrix}
		1 & x_{12} & x_{13} & x_{14} & x_{15} \\
		0 & 1 & x_{23} & 0 & x_{25} \\
		0 & 0 & 1 & 0 & x_{35}  \\
		0 & 0 & 0 & 1 & x_{45}  \\
		0 & 0 & 0 & 0 & 1 \\
	\end{matrix}\right),x^A=\left(\begin{matrix}
		1 & 1 & 0 & 0 & 0 \\
		0 & 1 & 1 & 0 & 0 \\
		0 & 0 & 1 & 0 & 1 \\
		0 & 0 & 0 & 1 & 1 \\
		0 & 0 & 0 & 0 & 1
	\end{matrix}\right)
\end{equation*} \newline
Where matrix $A$ has entries
\begin{align*}
	&d_{1} = 1, d_{2} = \frac{1}{x_{12}}, d_{3} = \frac{1}{x_{12} x_{23}}, d_{4} = \frac{x_{45}}{x_{12} x_{23} x_{35}}, d_{5} = \frac{1}{x_{12} x_{23} x_{35}},\\ 
	&a_{12} = 1, a_{13} = 1, a_{14} = 1, a_{15} = 1,\\ 
	&a_{23} = \frac{x_{12} x_{23} - x_{13}}{x_{12}^{2} x_{23}}, a_{24} = -\frac{x_{14} x_{45}}{x_{12}^{2} x_{23} x_{35}}, a_{25} = 1,\\ 
	&a_{34} = 0, a_{35} = -\frac{x_{12} x_{25} + x_{14} x_{45} - {\left(x_{12} x_{23} - x_{13}\right)} x_{35}}{x_{12}^{2} x_{23}^{2} x_{35}},\\ 
	&a_{45} = -\frac{x_{12} x_{15} x_{23} - x_{12} x_{13} x_{25} - x_{13} x_{14} x_{45} + {\left(x_{12} x_{13} x_{23} - x_{13}^{2} + {\left(x_{12}^{3} - {2} \, x_{12}^{2}\right)} x_{23}^{2}\right)} x_{35}}{x_{12}^{2} x_{14} x_{23}^{2} x_{35}}
\end{align*}


First assume $x_{35}\neq \frac{-x_{24}x_{45}}{x_{23}}$
\begin{equation*}x=\left(\begin{matrix}
		1 & x_{12} & 0 & 0 & 0 \\
		0 & 1 & x_{23} & x_{24} & 0 \\
		0 & 0 & 1 & 0 & x_{35}  \\
		0 & 0 & 0 & 1 & x_{45}  \\
		0 & 0 & 0 & 0 & 1 \\
	\end{matrix}\right),x^A=\left(\begin{matrix}
		1 & 1 & 0 & 0 & 0 \\
		0 & 1 & 1 & 0 & 0 \\
		0 & 0 & 1 & 0 & 1 \\
		0 & 0 & 0 & 1 & 1 \\
		0 & 0 & 0 & 0 & 1
	\end{matrix}\right)
\end{equation*} \newline
Where matrix $A$ has entries
\begin{align*}
	&d_{1} = 1, d_{2} = \frac{1}{x_{12}}, d_{3} = \frac{1}{x_{12} x_{23}}, d_{4} = \frac{x_{45}}{x_{12} x_{23} x_{35} + x_{12} x_{24} x_{45}}, d_{5} = \frac{1}{x_{12} x_{23} x_{35} + x_{12} x_{24} x_{45}},\\ 
	&a_{12} = 1, a_{13} = 1, a_{14} = 1, a_{15} = 1,\\ 
	&a_{23} = \frac{1}{x_{12}}, a_{24} = 0, a_{25} = \frac{2}{x_{12}},\\ 
	&a_{34} = -\frac{x_{24} x_{45}}{x_{12} x_{23}^{2} x_{35} + x_{12} x_{23} x_{24} x_{45}}, a_{35} = 1,\\ 
	&a_{45} = -\frac{x_{12} x_{23} - 1}{x_{12} x_{24}}
\end{align*}

Now assume $x_{35}= \frac{-x_{24}x_{45}}{x_{23}}$
\begin{equation*}x=\left(\begin{matrix}
		1 & x_{12} & 0 & 0 & 0 \\
		0 & 1 & x_{23} & x_{24} & 0 \\
		0 & 0 & 1 & 0 & x_{35}  \\
		0 & 0 & 0 & 1 & x_{45}  \\
		0 & 0 & 0 & 0 & 1 \\
	\end{matrix}\right),x^A=\left(\begin{matrix}
		1 & 1 & 0 & 0 & 0 \\
		0 & 1 & 1 & 0 & 0 \\
		0 & 0 & 1 & 0 & 0 \\
		0 & 0 & 0 & 1 & 1 \\
		0 & 0 & 0 & 0 & 1
	\end{matrix}\right)
\end{equation*} \newline
Where matrix $A$ has entries
\begin{align*}
	&d_{1} = 1, d_{2} = \frac{1}{x_{12}}, d_{3} = \frac{1}{x_{12} x_{23}}, d_{4} = 1, d_{5} = \frac{1}{x_{45}},\\ 
	&a_{12} = 1, a_{13} = 1, a_{14} = 1, a_{15} = 1,\\ 
	&a_{23} = \frac{1}{x_{12}}, a_{24} = 0, a_{25} = \frac{1}{x_{12}},\\ 
	&a_{34} = -\frac{x_{24}}{x_{23}}, a_{35} = 1,\\ 
	&a_{45} = -\frac{x_{23}}{x_{24}}
\end{align*}


First assume $x_{35}\neq \frac{-x_{24}x_{45}}{x_{23}}$
\begin{equation*}x=\left(\begin{matrix}
		1 & x_{12} & 0 & 0 & x_{15} \\
		0 & 1 & x_{23} & x_{24} & 0 \\
		0 & 0 & 1 & 0 & x_{35}  \\
		0 & 0 & 0 & 1 & x_{45}  \\
		0 & 0 & 0 & 0 & 1 \\
	\end{matrix}\right),x^A=\left(\begin{matrix}
		1 & 1 & 0 & 0 & 0 \\
		0 & 1 & 1 & 0 & 0 \\
		0 & 0 & 1 & 0 & 1 \\
		0 & 0 & 0 & 1 & 1 \\
		0 & 0 & 0 & 0 & 1
	\end{matrix}\right)
\end{equation*} \newline
Where matrix $A$ has entries
\begin{align*}
	&d_{1} = 1, d_{2} = \frac{1}{x_{12}}, d_{3} = \frac{1}{x_{12} x_{23}}, d_{4} = \frac{x_{45}}{x_{12} x_{23} x_{35} + x_{12} x_{24} x_{45}}, d_{5} = \frac{1}{x_{12} x_{23} x_{35} + x_{12} x_{24} x_{45}},\\ 
	&a_{12} = 1, a_{13} = 1, a_{14} = 1, a_{15} = 1,\\ 
	&a_{23} = \frac{1}{x_{12}}, a_{24} = 0, a_{25} = \frac{{2} \, x_{12} x_{23} x_{35} + {2} \, x_{12} x_{24} x_{45} - x_{15}}{x_{12}^{2} x_{23} x_{35} + x_{12}^{2} x_{24} x_{45}},\\ 
	&a_{34} = -\frac{x_{24} x_{45}}{x_{12} x_{23}^{2} x_{35} + x_{12} x_{23} x_{24} x_{45}}, a_{35} = 1,\\ 
	&a_{45} = -\frac{x_{12} x_{23} - 1}{x_{12} x_{24}}
\end{align*}

Now assume $x_{35}= \frac{-x_{24}x_{45}}{x_{23}}$
\begin{equation*}x=\left(\begin{matrix}
		1 & x_{12} & 0 & 0 & x_{15} \\
		0 & 1 & x_{23} & x_{24} & 0 \\
		0 & 0 & 1 & 0 & x_{35}  \\
		0 & 0 & 0 & 1 & x_{45}  \\
		0 & 0 & 0 & 0 & 1 \\
	\end{matrix}\right),x^A=\left(\begin{matrix}
		1 & 1 & 0 & 0 & 0 \\
		0 & 1 & 1 & 0 & 0 \\
		0 & 0 & 1 & 0 & 0 \\
		0 & 0 & 0 & 1 & 1 \\
		0 & 0 & 0 & 0 & 1
	\end{matrix}\right)
\end{equation*} \newline
Where matrix $A$ has entries
\begin{align*}
	&d_{1} = 1, d_{2} = \frac{1}{x_{12}}, d_{3} = \frac{1}{x_{12} x_{23}}, d_{4} = 1, d_{5} = \frac{1}{x_{45}},\\ 
	&a_{12} = 1, a_{13} = 1, a_{14} = 1, a_{15} = 1,\\ 
	&a_{23} = \frac{1}{x_{12}}, a_{24} = 0, a_{25} = -\frac{x_{15} - x_{45}}{x_{12} x_{45}},\\ 
	&a_{34} = -\frac{x_{24}}{x_{23}}, a_{35} = 1,\\ 
	&a_{45} = -\frac{x_{23}}{x_{24}}
\end{align*}


First assume $x_{35}\neq \frac{-x_{24}x_{45}}{x_{23}}$
\begin{equation*}x=\left(\begin{matrix}
		1 & x_{12} & 0 & x_{14} & 0 \\
		0 & 1 & x_{23} & x_{24} & 0 \\
		0 & 0 & 1 & 0 & x_{35}  \\
		0 & 0 & 0 & 1 & x_{45}  \\
		0 & 0 & 0 & 0 & 1 \\
	\end{matrix}\right),x^A=\left(\begin{matrix}
		1 & 1 & 0 & 0 & 0 \\
		0 & 1 & 1 & 0 & 0 \\
		0 & 0 & 1 & 0 & 1 \\
		0 & 0 & 0 & 1 & 1 \\
		0 & 0 & 0 & 0 & 1
	\end{matrix}\right)
\end{equation*} \newline
Where matrix $A$ has entries
\begin{align*}
	&d_{1} = 1, d_{2} = \frac{1}{x_{12}}, d_{3} = \frac{1}{x_{12} x_{23}}, d_{4} = \frac{x_{45}}{x_{12} x_{23} x_{35} + x_{12} x_{24} x_{45}}, d_{5} = \frac{1}{x_{12} x_{23} x_{35} + x_{12} x_{24} x_{45}},\\ 
	&a_{12} = 1, a_{13} = 1, a_{14} = 1, a_{15} = 1,\\ 
	&a_{23} = \frac{1}{x_{12}}, a_{24} = -\frac{x_{14} x_{45}}{x_{12}^{2} x_{23} x_{35} + x_{12}^{2} x_{24} x_{45}}, a_{25} = 1,\\ 
	&a_{34} = -\frac{x_{24} x_{45}}{x_{12} x_{23}^{2} x_{35} + x_{12} x_{23} x_{24} x_{45}}, a_{35} = \frac{{\left(x_{12} x_{14} x_{23} + {\left(x_{12}^{3} - {2} \, x_{12}^{2}\right)} x_{23} x_{24}\right)} x_{35} + {\left(x_{12} x_{14} x_{24} - x_{14}^{2} + {\left(x_{12}^{3} - {2} \, x_{12}^{2}\right)} x_{24}^{2}\right)} x_{45}}{x_{12}^{2} x_{14} x_{23}^{2} x_{35} + x_{12}^{2} x_{14} x_{23} x_{24} x_{45}},\\ 
	&a_{45} = -\frac{x_{12} - 2}{x_{14}}
\end{align*}

Now assume $x_{35}= \frac{-x_{24}x_{45}}{x_{23}}$
\begin{equation*}x=\left(\begin{matrix}
		1 & x_{12} & 0 & x_{14} & 0 \\
		0 & 1 & x_{23} & x_{24} & 0 \\
		0 & 0 & 1 & 0 & x_{35}  \\
		0 & 0 & 0 & 1 & x_{45}  \\
		0 & 0 & 0 & 0 & 1 \\
	\end{matrix}\right),x^A=\left(\begin{matrix}
		1 & 1 & 0 & 0 & 0 \\
		0 & 1 & 1 & 0 & 0 \\
		0 & 0 & 1 & 0 & 0 \\
		0 & 0 & 0 & 1 & 1 \\
		0 & 0 & 0 & 0 & 1
	\end{matrix}\right)
\end{equation*} \newline
Where matrix $A$ has entries
\begin{align*}
	&d_{1} = 1, d_{2} = \frac{1}{x_{12}}, d_{3} = \frac{1}{x_{12} x_{23}}, d_{4} = 1, d_{5} = \frac{1}{x_{45}},\\ 
	&a_{12} = 1, a_{13} = 1, a_{14} = 1, a_{15} = 1,\\ 
	&a_{23} = \frac{1}{x_{12}}, a_{24} = -\frac{x_{14}}{x_{12}}, a_{25} = 1,\\ 
	&a_{34} = -\frac{x_{24}}{x_{23}}, a_{35} = -\frac{x_{14}^{2} - {\left(x_{12}^{2} - x_{12}\right)} x_{24}}{x_{12} x_{14} x_{23}},\\ 
	&a_{45} = -\frac{x_{12} - 1}{x_{14}}
\end{align*}


First assume $x_{35}\neq \frac{-x_{24}x_{45}}{x_{23}}$
\begin{equation*}x=\left(\begin{matrix}
		1 & x_{12} & 0 & x_{14} & x_{15} \\
		0 & 1 & x_{23} & x_{24} & 0 \\
		0 & 0 & 1 & 0 & x_{35}  \\
		0 & 0 & 0 & 1 & x_{45}  \\
		0 & 0 & 0 & 0 & 1 \\
	\end{matrix}\right),x^A=\left(\begin{matrix}
		1 & 1 & 0 & 0 & 0 \\
		0 & 1 & 1 & 0 & 0 \\
		0 & 0 & 1 & 0 & 1 \\
		0 & 0 & 0 & 1 & 1 \\
		0 & 0 & 0 & 0 & 1
	\end{matrix}\right)
\end{equation*} \newline
Where matrix $A$ has entries
\begin{align*}
	&d_{1} = 1, d_{2} = \frac{1}{x_{12}}, d_{3} = \frac{1}{x_{12} x_{23}}, d_{4} = \frac{x_{45}}{x_{12} x_{23} x_{35} + x_{12} x_{24} x_{45}}, d_{5} = \frac{1}{x_{12} x_{23} x_{35} + x_{12} x_{24} x_{45}},\\ 
	&a_{12} = 1, a_{13} = 1, a_{14} = 1, a_{15} = 1,\\ 
	&a_{23} = \frac{1}{x_{12}}, a_{24} = -\frac{x_{14} x_{45}}{x_{12}^{2} x_{23} x_{35} + x_{12}^{2} x_{24} x_{45}}, a_{25} = 1,\\ 
	&a_{34} = -\frac{x_{24} x_{45}}{x_{12} x_{23}^{2} x_{35} + x_{12} x_{23} x_{24} x_{45}}, \\
	&a_{35} = \frac{x_{12} x_{15} x_{24} + {\left(x_{12} x_{14} x_{23} + {\left(x_{12}^{3} - {2} \, x_{12}^{2}\right)} x_{23} x_{24}\right)} x_{35} + {\left(x_{12} x_{14} x_{24} - x_{14}^{2} + {\left(x_{12}^{3} - {2} \, x_{12}^{2}\right)} x_{24}^{2}\right)} x_{45}}{x_{12}^{2} x_{14} x_{23}^{2} x_{35} + x_{12}^{2} x_{14} x_{23} x_{24} x_{45}},\\ 
	&a_{45} = -\frac{{\left(x_{12}^{2} - {2} \, x_{12}\right)} x_{23} x_{35} + {\left(x_{12}^{2} - {2} \, x_{12}\right)} x_{24} x_{45} + x_{15}}{x_{12} x_{14} x_{23} x_{35} + x_{12} x_{14} x_{24} x_{45}}
\end{align*}

Now assume $x_{35}= \frac{-x_{24}x_{45}}{x_{23}}$
\begin{equation*}x=\left(\begin{matrix}
		1 & x_{12} & 0 & x_{14} & x_{15} \\
		0 & 1 & x_{23} & x_{24} & 0 \\
		0 & 0 & 1 & 0 & x_{35}  \\
		0 & 0 & 0 & 1 & x_{45}  \\
		0 & 0 & 0 & 0 & 1 \\
	\end{matrix}\right),x^A=\left(\begin{matrix}
		1 & 1 & 0 & 0 & 0 \\
		0 & 1 & 1 & 0 & 0 \\
		0 & 0 & 1 & 0 & 0 \\
		0 & 0 & 0 & 1 & 1 \\
		0 & 0 & 0 & 0 & 1
	\end{matrix}\right)
\end{equation*} \newline
Where matrix $A$ has entries
\begin{align*}
	&d_{1} = 1, d_{2} = \frac{1}{x_{12}}, d_{3} = \frac{1}{x_{12} x_{23}}, d_{4} = 1, d_{5} = \frac{1}{x_{45}},\\ 
	&a_{12} = 1, a_{13} = 1, a_{14} = 1, a_{15} = 1,\\ 
	&a_{23} = \frac{1}{x_{12}}, a_{24} = -\frac{x_{14}}{x_{12}}, a_{25} = 1,\\ 
	&a_{34} = -\frac{x_{24}}{x_{23}}, a_{35} = \frac{x_{12} x_{15} x_{24} - {\left(x_{14}^{2} - {\left(x_{12}^{2} - x_{12}\right)} x_{24}\right)} x_{45}}{x_{12} x_{14} x_{23} x_{45}},\\ 
	&a_{45} = -\frac{x_{15} + {\left(x_{12} - 1\right)} x_{45}}{x_{14} x_{45}}
\end{align*}


First assume $x_{35}\neq \frac{-x_{24}x_{45}}{x_{23}}$
\begin{equation*}x=\left(\begin{matrix}
		1 & x_{12} & x_{13} & 0 & 0 \\
		0 & 1 & x_{23} & x_{24} & 0 \\
		0 & 0 & 1 & 0 & x_{35}  \\
		0 & 0 & 0 & 1 & x_{45}  \\
		0 & 0 & 0 & 0 & 1 \\
	\end{matrix}\right),x^A=\left(\begin{matrix}
		1 & 1 & 0 & 0 & 0 \\
		0 & 1 & 1 & 0 & 0 \\
		0 & 0 & 1 & 0 & 1 \\
		0 & 0 & 0 & 1 & 1 \\
		0 & 0 & 0 & 0 & 1
	\end{matrix}\right)
\end{equation*} \newline
Where matrix $A$ has entries
\begin{align*}
	&d_{1} = 1, d_{2} = \frac{1}{x_{12}}, d_{3} = \frac{1}{x_{12} x_{23}}, d_{4} = \frac{x_{45}}{x_{12} x_{23} x_{35} + x_{12} x_{24} x_{45}}, d_{5} = \frac{1}{x_{12} x_{23} x_{35} + x_{12} x_{24} x_{45}},\\ 
	&a_{12} = 1, a_{13} = 1, a_{14} = 1, a_{15} = 1,\\ 
	&a_{23} = \frac{x_{12} x_{23} - x_{13}}{x_{12}^{2} x_{23}}, a_{24} = \frac{x_{13} x_{24} x_{45}}{x_{12}^{2} x_{23}^{2} x_{35} + x_{12}^{2} x_{23} x_{24} x_{45}}, a_{25} = 1,\\ 
	&a_{34} = -\frac{x_{24} x_{45}}{x_{12} x_{23}^{2} x_{35} + x_{12} x_{23} x_{24} x_{45}}, a_{35} = -\frac{x_{12} - 2}{x_{13}},\\ 
	&a_{45} = \frac{{\left(x_{12} x_{13} + {\left(x_{12}^{3} - {2} \, x_{12}^{2}\right)} x_{23}\right)} x_{24} x_{45} + {\left(x_{12} x_{13} x_{23} - x_{13}^{2} + {\left(x_{12}^{3} - {2} \, x_{12}^{2}\right)} x_{23}^{2}\right)} x_{35}}{x_{12}^{2} x_{13} x_{23} x_{24} x_{35} + x_{12}^{2} x_{13} x_{24}^{2} x_{45}}
\end{align*}

Now assume $x_{35}= \frac{-x_{24}x_{45}}{x_{23}}$
\begin{equation*}x=\left(\begin{matrix}
		1 & x_{12} & x_{13} & 0 & 0 \\
		0 & 1 & x_{23} & x_{24} & 0 \\
		0 & 0 & 1 & 0 & x_{35}  \\
		0 & 0 & 0 & 1 & x_{45}  \\
		0 & 0 & 0 & 0 & 1 \\
	\end{matrix}\right),x^A=\left(\begin{matrix}
		1 & 1 & 0 & 0 & 0 \\
		0 & 1 & 1 & 0 & 0 \\
		0 & 0 & 1 & 0 & 0 \\
		0 & 0 & 0 & 1 & 1 \\
		0 & 0 & 0 & 0 & 1
	\end{matrix}\right)
\end{equation*} \newline
Where matrix $A$ has entries
\begin{align*}
	&d_{1} = 1, d_{2} = \frac{1}{x_{12}}, d_{3} = \frac{1}{x_{12} x_{23}}, d_{4} = 1, d_{5} = \frac{1}{x_{45}},\\ 
	&a_{12} = 1, a_{13} = 1, a_{14} = 1, a_{15} = 1,\\ 
	&a_{23} = \frac{x_{12} x_{23} - x_{13}}{x_{12}^{2} x_{23}}, a_{24} = \frac{x_{13} x_{24}}{x_{12} x_{23}}, a_{25} = 1,\\ 
	&a_{34} = -\frac{x_{24}}{x_{23}}, a_{35} = -\frac{x_{12} - 1}{x_{13}},\\ 
	&a_{45} = \frac{x_{13}^{2} x_{24} + {\left(x_{12}^{2} - x_{12}\right)} x_{23}^{2}}{x_{12} x_{13} x_{23} x_{24}}
\end{align*}


First assume $x_{35}\neq \frac{-x_{24}x_{45}}{x_{23}}$
\begin{equation*}x=\left(\begin{matrix}
		1 & x_{12} & x_{13} & 0 & x_{15} \\
		0 & 1 & x_{23} & x_{24} & 0 \\
		0 & 0 & 1 & 0 & x_{35}  \\
		0 & 0 & 0 & 1 & x_{45}  \\
		0 & 0 & 0 & 0 & 1 \\
	\end{matrix}\right),x^A=\left(\begin{matrix}
		1 & 1 & 0 & 0 & 0 \\
		0 & 1 & 1 & 0 & 0 \\
		0 & 0 & 1 & 0 & 1 \\
		0 & 0 & 0 & 1 & 1 \\
		0 & 0 & 0 & 0 & 1
	\end{matrix}\right)
\end{equation*} \newline
Where matrix $A$ has entries
\begin{align*}
	&d_{1} = 1, d_{2} = \frac{1}{x_{12}}, d_{3} = \frac{1}{x_{12} x_{23}}, d_{4} = \frac{x_{45}}{x_{12} x_{23} x_{35} + x_{12} x_{24} x_{45}}, d_{5} = \frac{1}{x_{12} x_{23} x_{35} + x_{12} x_{24} x_{45}},\\ 
	&a_{12} = 1, a_{13} = 1, a_{14} = 1, a_{15} = 1,\\ 
	&a_{23} = \frac{x_{12} x_{23} - x_{13}}{x_{12}^{2} x_{23}}, a_{24} = \frac{x_{13} x_{24} x_{45}}{x_{12}^{2} x_{23}^{2} x_{35} + x_{12}^{2} x_{23} x_{24} x_{45}}, a_{25} = 1,\\ 
	&a_{34} = -\frac{x_{24} x_{45}}{x_{12} x_{23}^{2} x_{35} + x_{12} x_{23} x_{24} x_{45}}, a_{35} = -\frac{{\left(x_{12}^{2} - {2} \, x_{12}\right)} x_{23} x_{35} + {\left(x_{12}^{2} - {2} \, x_{12}\right)} x_{24} x_{45} + x_{15}}{x_{12} x_{13} x_{23} x_{35} + x_{12} x_{13} x_{24} x_{45}},\\ 
	&a_{45} = \frac{x_{12} x_{15} x_{23} + {\left(x_{12} x_{13} + {\left(x_{12}^{3} - {2} \, x_{12}^{2}\right)} x_{23}\right)} x_{24} x_{45} + {\left(x_{12} x_{13} x_{23} - x_{13}^{2} + {\left(x_{12}^{3} - {2} \, x_{12}^{2}\right)} x_{23}^{2}\right)} x_{35}}{x_{12}^{2} x_{13} x_{23} x_{24} x_{35} + x_{12}^{2} x_{13} x_{24}^{2} x_{45}}
\end{align*}

Now assume $x_{35}= \frac{-x_{24}x_{45}}{x_{23}}$
\begin{equation*}x=\left(\begin{matrix}
		1 & x_{12} & x_{13} & 0 & x_{15} \\
		0 & 1 & x_{23} & x_{24} & 0 \\
		0 & 0 & 1 & 0 & x_{35}  \\
		0 & 0 & 0 & 1 & x_{45}  \\
		0 & 0 & 0 & 0 & 1 \\
	\end{matrix}\right),x^A=\left(\begin{matrix}
		1 & 1 & 0 & 0 & 0 \\
		0 & 1 & 1 & 0 & 0 \\
		0 & 0 & 1 & 0 & 0 \\
		0 & 0 & 0 & 1 & 1 \\
		0 & 0 & 0 & 0 & 1
	\end{matrix}\right)
\end{equation*} \newline
Where matrix $A$ has entries
\begin{align*}
	&d_{1} = 1, d_{2} = \frac{1}{x_{12}}, d_{3} = \frac{1}{x_{12} x_{23}}, d_{4} = 1, d_{5} = \frac{1}{x_{45}},\\ 
	&a_{12} = 1, a_{13} = 1, a_{14} = 1, a_{15} = 1,\\ 
	&a_{23} = \frac{x_{12} x_{23} - x_{13}}{x_{12}^{2} x_{23}}, a_{24} = \frac{x_{13} x_{24}}{x_{12} x_{23}}, a_{25} = 1,\\ 
	&a_{34} = -\frac{x_{24}}{x_{23}}, a_{35} = -\frac{x_{15} + {\left(x_{12} - 1\right)} x_{45}}{x_{13} x_{45}},\\ 
	&a_{45} = \frac{x_{12} x_{15} x_{23}^{2} + {\left(x_{13}^{2} x_{24} + {\left(x_{12}^{2} - x_{12}\right)} x_{23}^{2}\right)} x_{45}}{x_{12} x_{13} x_{23} x_{24} x_{45}}
\end{align*}


First assume $x_{23}\neq \frac{-x_{24}x_{45}}{x_{35}}$ and $x_{23}\neq \frac{x_{13}x_{24}}{x_{14}}$
\begin{equation*}x=\left(\begin{matrix}
		1 & x_{12} & x_{13} & x_{14} & 0 \\
		0 & 1 & x_{23} & x_{24} & 0 \\
		0 & 0 & 1 & 0 & x_{35}  \\
		0 & 0 & 0 & 1 & x_{45}  \\
		0 & 0 & 0 & 0 & 1 \\
	\end{matrix}\right),x^A=\left(\begin{matrix}
		1 & 1 & 0 & 0 & 0 \\
		0 & 1 & 1 & 0 & 0 \\
		0 & 0 & 1 & 0 & 1 \\
		0 & 0 & 0 & 1 & 1 \\
		0 & 0 & 0 & 0 & 1
	\end{matrix}\right)
\end{equation*} \newline
Where matrix $A$ has entries
\begin{align*}
	&d_{1} = 1, d_{2} = \frac{1}{x_{12}}, d_{3} = \frac{1}{x_{12} x_{23}}, d_{4} = \frac{x_{45}}{x_{12} x_{23} x_{35} + x_{12} x_{24} x_{45}}, d_{5} = \frac{1}{x_{12} x_{23} x_{35} + x_{12} x_{24} x_{45}},\\ 
	&a_{12} = 1, a_{13} = 1, a_{14} = 1, a_{15} = 1,\\ 
	&a_{23} = \frac{x_{12} x_{23} - x_{13}}{x_{12}^{2} x_{23}}, a_{24} = -\frac{{\left(x_{14} x_{23} - x_{13} x_{24}\right)} x_{45}}{x_{12}^{2} x_{23}^{2} x_{35} + x_{12}^{2} x_{23} x_{24} x_{45}}, a_{25} = 1,\\ 
	&a_{34} = -\frac{x_{24} x_{45}}{x_{12} x_{23}^{2} x_{35} + x_{12} x_{23} x_{24} x_{45}},\\
	&a_{35} = \frac{{\left(x_{12} x_{14} x_{23} - x_{13} x_{14} + {\left(x_{12}^{3} - {2} \, x_{12}^{2}\right)} x_{23} x_{24}\right)} x_{35} + {\left(x_{12} x_{14} x_{24} - x_{14}^{2} + {\left(x_{12}^{3} - {2} \, x_{12}^{2}\right)} x_{24}^{2}\right)} x_{45}}{{\left(x_{12}^{2} x_{14} x_{23}^{2} - x_{12}^{2} x_{13} x_{23} x_{24}\right)} x_{35} + {\left(x_{12}^{2} x_{14} x_{23} x_{24} - x_{12}^{2} x_{13} x_{24}^{2}\right)} x_{45}},\\ 
	&a_{45} = -\frac{{\left(x_{12} x_{13} x_{23} - x_{13}^{2} + {\left(x_{12}^{3} - {2} \, x_{12}^{2}\right)} x_{23}^{2}\right)} x_{35} - {\left(x_{13} x_{14} - {\left(x_{12} x_{13} + {\left(x_{12}^{3} - {2} \, x_{12}^{2}\right)} x_{23}\right)} x_{24}\right)} x_{45}}{{\left(x_{12}^{2} x_{14} x_{23}^{2} - x_{12}^{2} x_{13} x_{23} x_{24}\right)} x_{35} + {\left(x_{12}^{2} x_{14} x_{23} x_{24} - x_{12}^{2} x_{13} x_{24}^{2}\right)} x_{45}}
\end{align*}

Now assume $x_{23}\neq \frac{-x_{24}x_{45}}{x_{35}}$ and $x_{23}= \frac{x_{13}x_{24}}{x_{14}}$ and $x_{45}\neq \frac{-x_{13}x_{35}}{x_{14}}$
\begin{equation*}x=\left(\begin{matrix} 
		1 & x_{12} & x_{13} & x_{14} & 0 \\
		0 & 1 & x_{23} & x_{24} & 0 \\
		0 & 0 & 1 & 0 & x_{35}  \\
		0 & 0 & 0 & 1 & x_{45}  \\
		0 & 0 & 0 & 0 & 1 \\
	\end{matrix}\right),x^A=\left(\begin{matrix}
		1 & 1 & 0 & 0 & 0 \\
		0 & 1 & 1 & 0 & 0 \\
		0 & 0 & 1 & 0 & 1 \\
		0 & 0 & 0 & 1 & 1 \\
		0 & 0 & 0 & 0 & 1
	\end{matrix}\right)
\end{equation*} \newline
Where matrix $A$ has entries
\begin{align*}
	&d_{1} = 1, d_{2} = \frac{1}{x_{12}}, d_{3} = \frac{x_{14}}{x_{12} x_{13} x_{24}}, d_{4} = \frac{x_{14} x_{45}}{x_{12} x_{13} x_{24} x_{35} + x_{12} x_{14} x_{24} x_{45}}, d_{5} = \frac{x_{14}}{x_{12} x_{13} x_{24} x_{35} + x_{12} x_{14} x_{24} x_{45}},\\ 
	&a_{12} = 1, a_{13} = 1, a_{14} = 1, a_{15} = 1,\\ 
	&a_{23} = \frac{x_{12} x_{24} - x_{14}}{x_{12}^{2} x_{24}}, a_{24} = 0, a_{25} = \frac{{2} \, x_{12}^{2} x_{24}^{2} - x_{12} x_{14} x_{24} + x_{14}^{2}}{x_{12}^{3} x_{24}^{2}},\\ 
	&a_{34} = -\frac{x_{14}^{2} x_{45}}{x_{12} x_{13}^{2} x_{24} x_{35} + x_{12} x_{13} x_{14} x_{24} x_{45}}, a_{35} = 1,\\ 
	&a_{45} = -\frac{x_{12}^{2} x_{13} x_{24}^{2} - x_{12} x_{14} x_{24} + x_{14}^{2}}{x_{12}^{2} x_{14} x_{24}^{2}}
\end{align*}

Now assume $x_{23}\neq \frac{-x_{24}x_{45}}{x_{35}}$ and $x_{23}= \frac{x_{13}x_{24}}{x_{14}}$
and $x_{45}= \frac{-x_{13}x_{35}}{x_{14}}$
\begin{equation*}x=\left(\begin{matrix}
		1 & x_{12} & x_{13} & x_{14} & 0 \\
		0 & 1 & x_{23} & x_{24} & 0 \\
		0 & 0 & 1 & 0 & x_{35}  \\
		0 & 0 & 0 & 1 & x_{45}  \\
		0 & 0 & 0 & 0 & 1 \\
	\end{matrix}\right),x^A=\left(\begin{matrix}
		1 & 1 & 0 & 0 & 0 \\
		0 & 1 & 1 & 0 & 0 \\
		0 & 0 & 1 & 0 & 0 \\
		0 & 0 & 0 & 1 & 1 \\
		0 & 0 & 0 & 0 & 1
	\end{matrix}\right)
\end{equation*} \newline
Where matrix $A$ has entries
\begin{align*}
	&d_{1} = 1, d_{2} = \frac{1}{x_{12}}, d_{3} = \frac{x_{14}}{x_{12} x_{13} x_{24}}, d_{4} = 1, d_{5} = -\frac{x_{14}}{x_{13} x_{35}},\\ 
	&a_{12} = 1, a_{13} = 1, a_{14} = 1, a_{15} = 1,\\ 
	&a_{23} = \frac{x_{12} x_{24} - x_{14}}{x_{12}^{2} x_{24}}, a_{24} = 0, a_{25} = \frac{1}{x_{12}},\\ 
	&a_{34} = -\frac{x_{14}}{x_{13}}, a_{35} = 1,\\ 
	&a_{45} = -\frac{x_{13}}{x_{14}}
\end{align*}

Now assume $x_{23}= \frac{-x_{24}x_{45}}{x_{35}}$ and $x_{45}\neq \frac{-x_{13}x_{35}}{x_{14}}$
\begin{equation*}x=\left(\begin{matrix}
		1 & x_{12} & x_{13} & x_{14} & 0 \\
		0 & 1 & x_{23} & x_{24} & 0 \\
		0 & 0 & 1 & 0 & x_{35}  \\
		0 & 0 & 0 & 1 & x_{45}  \\
		0 & 0 & 0 & 0 & 1 \\
	\end{matrix}\right),x^A=\left(\begin{matrix}
		1 & 1 & 0 & 0 & 0 \\
		0 & 1 & 1 & 0 & 0 \\
		0 & 0 & 1 & 0 & 0 \\
		0 & 0 & 0 & 1 & 1 \\
		0 & 0 & 0 & 0 & 1
	\end{matrix}\right)
\end{equation*} \newline
Where matrix $A$ has entries
\begin{align*}
	&d_{1} = 1, d_{2} = \frac{1}{x_{12}}, d_{3} = -\frac{x_{35}}{x_{12} x_{24} x_{45}}, d_{4} = 1, d_{5} = \frac{1}{x_{45}},\\ 
	&a_{12} = 1, a_{13} = 1, a_{14} = 1, a_{15} = 1,\\ 
	&a_{23} = \frac{x_{12} x_{24} x_{45} + x_{13} x_{35}}{x_{12}^{2} x_{24} x_{45}}, a_{24} = -\frac{x_{13} x_{35} + x_{14} x_{45}}{x_{12} x_{45}}, a_{25} = 1,\\ 
	&a_{34} = \frac{x_{35}}{x_{45}}, a_{35} = \frac{x_{13} x_{14} x_{35}^{2} + {\left(x_{14}^{2} - {\left(x_{12}^{2} - x_{12}\right)} x_{24}\right)} x_{35} x_{45}}{x_{12} x_{13} x_{24} x_{35} x_{45} + x_{12} x_{14} x_{24} x_{45}^{2}},\\ 
	&a_{45} = -\frac{x_{13}^{2} x_{35}^{2} + x_{13} x_{14} x_{35} x_{45} + {\left(x_{12}^{2} - x_{12}\right)} x_{24} x_{45}^{2}}{x_{12} x_{13} x_{24} x_{35} x_{45} + x_{12} x_{14} x_{24} x_{45}^{2}}
\end{align*}

Now assume $x_{23}= \frac{-x_{24}x_{45}}{x_{35}}$ and $x_{45}= \frac{-x_{13}x_{35}}{x_{14}}$
\begin{equation*}x=\left(\begin{matrix}
		1 & x_{12} & x_{13} & x_{14} & 0 \\
		0 & 1 & x_{23} & x_{24} & 0 \\
		0 & 0 & 1 & 0 & x_{35}  \\
		0 & 0 & 0 & 1 & x_{45}  \\
		0 & 0 & 0 & 0 & 1 \\
	\end{matrix}\right),x^A=\left(\begin{matrix}
		1 & 1 & 0 & 0 & 0 \\
		0 & 1 & 1 & 0 & 0 \\
		0 & 0 & 1 & 0 & 0 \\
		0 & 0 & 0 & 1 & 1 \\
		0 & 0 & 0 & 0 & 1
	\end{matrix}\right)
\end{equation*} \newline
Where matrix $A$ has entries
\begin{align*}
	&d_{1} = 1, d_{2} = \frac{1}{x_{12}}, d_{3} = \frac{x_{14}}{x_{12} x_{13} x_{24}}, d_{4} = 1, d_{5} = -\frac{x_{14}}{x_{13} x_{35}},\\ 
	&a_{12} = 1, a_{13} = 1, a_{14} = 1, a_{15} = 1,\\ 
	&a_{23} = \frac{x_{12} x_{24} - x_{14}}{x_{12}^{2} x_{24}}, a_{24} = 0, a_{25} = \frac{1}{x_{12}},\\ 
	&a_{34} = -\frac{x_{14}}{x_{13}}, a_{35} = 1,\\ 
	&a_{45} = -\frac{x_{13}}{x_{14}}
\end{align*}


First assume $x_{23}\neq \frac{-x_{24}x_{45}}{x_{35}}$ and $x_{23}\neq \frac{x_{13}x_{24}}{x_{14}}$
\begin{equation*}x=\left(\begin{matrix}
		1 & x_{12} & x_{13} & x_{14} & x_{15} \\
		0 & 1 & x_{23} & x_{24} & 0 \\
		0 & 0 & 1 & 0 & x_{35}  \\
		0 & 0 & 0 & 1 & x_{45}  \\
		0 & 0 & 0 & 0 & 1 \\
	\end{matrix}\right),x^A=\left(\begin{matrix}
		1 & 1 & 0 & 0 & 0 \\
		0 & 1 & 1 & 0 & 0 \\
		0 & 0 & 1 & 0 & 1 \\
		0 & 0 & 0 & 1 & 1 \\
		0 & 0 & 0 & 0 & 1
	\end{matrix}\right)
\end{equation*} \newline
Where matrix $A$ has entries
\begin{align*}
	&d_{1} = 1, d_{2} = \frac{1}{x_{12}}, d_{3} = \frac{1}{x_{12} x_{23}}, d_{4} = \frac{x_{45}}{x_{12} x_{23} x_{35} + x_{12} x_{24} x_{45}}, d_{5} = \frac{1}{x_{12} x_{23} x_{35} + x_{12} x_{24} x_{45}},\\ 
	&a_{12} = 1, a_{13} = 1, a_{14} = 1, a_{15} = 1,\\ 
	&a_{23} = \frac{x_{12} x_{23} - x_{13}}{x_{12}^{2} x_{23}}, a_{24} = -\frac{{\left(x_{14} x_{23} - x_{13} x_{24}\right)} x_{45}}{x_{12}^{2} x_{23}^{2} x_{35} + x_{12}^{2} x_{23} x_{24} x_{45}}, a_{25} = 1,\\ 
	&a_{34} = -\frac{x_{24} x_{45}}{x_{12} x_{23}^{2} x_{35} + x_{12} x_{23} x_{24} x_{45}}, \\
	&a_{35} = \frac{x_{12} x_{15} x_{24} + {\left(x_{12} x_{14} x_{23} - x_{13} x_{14} + {\left(x_{12}^{3} - {2} \, x_{12}^{2}\right)} x_{23} x_{24}\right)} x_{35} + {\left(x_{12} x_{14} x_{24} - x_{14}^{2} + {\left(x_{12}^{3} - {2} \, x_{12}^{2}\right)} x_{24}^{2}\right)} x_{45}}{{\left(x_{12}^{2} x_{14} x_{23}^{2} - x_{12}^{2} x_{13} x_{23} x_{24}\right)} x_{35} + {\left(x_{12}^{2} x_{14} x_{23} x_{24} - x_{12}^{2} x_{13} x_{24}^{2}\right)} x_{45}},\\ 
	&a_{45} = -\frac{x_{12} x_{15} x_{23} + {\left(x_{12} x_{13} x_{23} - x_{13}^{2} + {\left(x_{12}^{3} - {2} \, x_{12}^{2}\right)} x_{23}^{2}\right)} x_{35} - {\left(x_{13} x_{14} - {\left(x_{12} x_{13} + {\left(x_{12}^{3} - {2} \, x_{12}^{2}\right)} x_{23}\right)} x_{24}\right)} x_{45}}{{\left(x_{12}^{2} x_{14} x_{23}^{2} - x_{12}^{2} x_{13} x_{23} x_{24}\right)} x_{35} + {\left(x_{12}^{2} x_{14} x_{23} x_{24} - x_{12}^{2} x_{13} x_{24}^{2}\right)} x_{45}}
\end{align*}

Now assume $x_{23}\neq \frac{-x_{24}x_{45}}{x_{35}}$ and $x_{23}= \frac{x_{13}x_{24}}{x_{14}}$ and $x_{45}\neq \frac{-x_{13}x_{35}}{x_{14}}$

\begin{equation*}x=\left(\begin{matrix}
		1 & x_{12} & x_{13} & x_{14} & x_{15} \\
		0 & 1 & x_{23} & x_{24} & 0 \\
		0 & 0 & 1 & 0 & x_{35}  \\
		0 & 0 & 0 & 1 & x_{45}  \\
		0 & 0 & 0 & 0 & 1 \\
	\end{matrix}\right),x^A=\left(\begin{matrix}
		1 & 1 & 0 & 0 & 0 \\
		0 & 1 & 1 & 0 & 0 \\
		0 & 0 & 1 & 0 & 1 \\
		0 & 0 & 0 & 1 & 1 \\
		0 & 0 & 0 & 0 & 1
	\end{matrix}\right)
\end{equation*} \newline
Where matrix $A$ has entries
\begin{align*}
	&d_{1} = 1, d_{2} = \frac{1}{x_{12}}, d_{3} = \frac{x_{14}}{x_{12} x_{13} x_{24}}, d_{4} = \frac{x_{14} x_{45}}{x_{12} x_{13} x_{24} x_{35} + x_{12} x_{14} x_{24} x_{45}}, d_{5} = \frac{x_{14}}{x_{12} x_{13} x_{24} x_{35} + x_{12} x_{14} x_{24} x_{45}},\\ 
	&a_{12} = 1, a_{13} = 1, a_{14} = 1, a_{15} = 1,\\ 
	&a_{23} = \frac{x_{12} x_{24} - x_{14}}{x_{12}^{2} x_{24}}, a_{24} = 0, \\
	&a_{25} = -\frac{x_{12} x_{14} x_{15} x_{24} - {\left({2} \, x_{12}^{2} x_{13} x_{24}^{2} - x_{12} x_{13} x_{14} x_{24} + x_{13} x_{14}^{2}\right)} x_{35} - {\left({2} \, x_{12}^{2} x_{14} x_{24}^{2} - x_{12} x_{14}^{2} x_{24} + x_{14}^{3}\right)} x_{45}}{x_{12}^{3} x_{13} x_{24}^{2} x_{35} + x_{12}^{3} x_{14} x_{24}^{2} x_{45}},\\ 
	&a_{34} = -\frac{x_{14}^{2} x_{45}}{x_{12} x_{13}^{2} x_{24} x_{35} + x_{12} x_{13} x_{14} x_{24} x_{45}}, a_{35} = 1,\\ 
	&a_{45} = -\frac{x_{12}^{2} x_{13} x_{24}^{2} - x_{12} x_{14} x_{24} + x_{14}^{2}}{x_{12}^{2} x_{14} x_{24}^{2}}
\end{align*}

Now assume $x_{23}\neq \frac{-x_{24}x_{45}}{x_{35}}$ and $x_{23}= \frac{x_{13}x_{24}}{x_{14}}$
and $x_{45}= \frac{-x_{13}x_{35}}{x_{14}}$
\begin{equation*}x=\left(\begin{matrix}
		1 & x_{12} & x_{13} & x_{14} & x_{15} \\
		0 & 1 & x_{23} & x_{24} & 0 \\
		0 & 0 & 1 & 0 & x_{35}  \\
		0 & 0 & 0 & 1 & x_{45}  \\
		0 & 0 & 0 & 0 & 1 \\
	\end{matrix}\right),x^A=\left(\begin{matrix}
		1 & 1 & 0 & 0 & 0 \\
		0 & 1 & 1 & 0 & 0 \\
		0 & 0 & 1 & 0 & 0 \\
		0 & 0 & 0 & 1 & 1 \\
		0 & 0 & 0 & 0 & 1 
	\end{matrix}\right)
\end{equation*} \newline
Where matrix $A$ has entries
\begin{align*}
	&d_{1} = 1, d_{2} = \frac{1}{x_{12}}, d_{3} = \frac{x_{14}}{x_{12} x_{13} x_{24}}, d_{4} = 1, d_{5} = -\frac{x_{14}}{x_{13} x_{35}},\\ 
	&a_{12} = 1, a_{13} = 1, a_{14} = 1, a_{15} = 1,\\ 
	&a_{23} = \frac{x_{12} x_{24} - x_{14}}{x_{12}^{2} x_{24}}, a_{24} = 0, a_{25} = \frac{x_{14} x_{15} + x_{13} x_{35}}{x_{12} x_{13} x_{35}},\\ 
	&a_{34} = -\frac{x_{14}}{x_{13}}, a_{35} = 1,\\ 
	&a_{45} = -\frac{x_{13}}{x_{14}}
\end{align*}

Now assume $x_{23}= \frac{-x_{24}x_{45}}{x_{35}}$ and $x_{45}\neq \frac{-x_{13}x_{35}}{x_{14}}$

\begin{equation*}x=\left(\begin{matrix}
		1 & x_{12} & x_{13} & x_{14} & x_{15} \\
		0 & 1 & x_{23} & x_{24} & 0 \\
		0 & 0 & 1 & 0 & x_{35}  \\
		0 & 0 & 0 & 1 & x_{45}  \\
		0 & 0 & 0 & 0 & 1 \\
	\end{matrix}\right),x^A=\left(\begin{matrix}
		1 & 1 & 0 & 0 & 0 \\
		0 & 1 & 1 & 0 & 0 \\
		0 & 0 & 1 & 0 & 0 \\
		0 & 0 & 0 & 1 & 1 \\
		0 & 0 & 0 & 0 & 1
	\end{matrix}\right)
\end{equation*} \newline
Where matrix $A$ has entries
\begin{align*}
	&d_{1} = 1, d_{2} = \frac{1}{x_{12}}, d_{3} = -\frac{x_{35}}{x_{12} x_{24} x_{45}}, d_{4} = 1, d_{5} = \frac{1}{x_{45}},\\ 
	&a_{12} = 1, a_{13} = 1, a_{14} = 1, a_{15} = 1,\\ 
	&a_{23} = \frac{x_{12} x_{24} x_{45} + x_{13} x_{35}}{x_{12}^{2} x_{24} x_{45}}, a_{24} = -\frac{x_{13} x_{35} + x_{14} x_{45}}{x_{12} x_{45}}, a_{25} = 1,\\ 
	&a_{34} = \frac{x_{35}}{x_{45}}, a_{35} = -\frac{x_{12} x_{15} x_{24} x_{35} - x_{13} x_{14} x_{35}^{2} - {\left(x_{14}^{2} - {\left(x_{12}^{2} - x_{12}\right)} x_{24}\right)} x_{35} x_{45}}{x_{12} x_{13} x_{24} x_{35} x_{45} + x_{12} x_{14} x_{24} x_{45}^{2}},\\ 
	&a_{45} = -\frac{x_{13}^{2} x_{35}^{2} + {\left(x_{12}^{2} - x_{12}\right)} x_{24} x_{45}^{2} + {\left(x_{12} x_{15} x_{24} + x_{13} x_{14} x_{35}\right)} x_{45}}{x_{12} x_{13} x_{24} x_{35} x_{45} + x_{12} x_{14} x_{24} x_{45}^{2}}
\end{align*}

Now assume $x_{23}= \frac{-x_{24}x_{45}}{x_{35}}$ and $x_{45}= \frac{-x_{13}x_{35}}{x_{14}}$

\begin{equation*}x=\left(\begin{matrix}
		1 & x_{12} & x_{13} & x_{14} & x_{15} \\
		0 & 1 & x_{23} & x_{24} & 0 \\
		0 & 0 & 1 & 0 & x_{35}  \\
		0 & 0 & 0 & 1 & x_{45}  \\
		0 & 0 & 0 & 0 & 1 \\
	\end{matrix}\right),x^A=\left(\begin{matrix}
		1 & 1 & 0 & 0 & 0 \\
		0 & 1 & 1 & 0 & 0 \\
		0 & 0 & 1 & 0 & 0 \\
		0 & 0 & 0 & 1 & 1 \\
		0 & 0 & 0 & 0 & 1
	\end{matrix}\right)
\end{equation*} \newline
Where matrix $A$ has entries
\begin{align*}
	&d_{1} = 1, d_{2} = \frac{1}{x_{12}}, d_{3} = \frac{x_{14}}{x_{12} x_{13} x_{24}}, d_{4} = 1, d_{5} = -\frac{x_{14}}{x_{13} x_{35}},\\ 
	&a_{12} = 1, a_{13} = 1, a_{14} = 1, a_{15} = 1,\\ 
	&a_{23} = \frac{x_{12} x_{24} - x_{14}}{x_{12}^{2} x_{24}}, a_{24} = 0, a_{25} = \frac{x_{14} x_{15} + x_{13} x_{35}}{x_{12} x_{13} x_{35}},\\ 
	&a_{34} = -\frac{x_{14}}{x_{13}}, a_{35} = 1,\\ 
	&a_{45} = -\frac{x_{13}}{x_{14}}
\end{align*}


First assume $x_{23}\neq \frac{-x_{24}x_{45}}{x_{35}}$
\begin{equation*}x=\left(\begin{matrix}
		1 & x_{12} & 0 & 0 & 0 \\
		0 & 1 & x_{23} & x_{24} & x_{25} \\
		0 & 0 & 1 & 0 & x_{35}  \\
		0 & 0 & 0 & 1 & x_{45}  \\
		0 & 0 & 0 & 0 & 1 \\
	\end{matrix}\right),x^A=\left(\begin{matrix}
		1 & 1 & 0 & 0 & 0 \\
		0 & 1 & 1 & 0 & 0 \\
		0 & 0 & 1 & 0 & 1 \\
		0 & 0 & 0 & 1 & 1 \\
		0 & 0 & 0 & 0 & 1
	\end{matrix}\right)
\end{equation*} \newline
Where matrix $A$ has entries
\begin{align*}
	&d_{1} = 1, d_{2} = \frac{1}{x_{12}}, d_{3} = \frac{1}{x_{12} x_{23}}, d_{4} = \frac{x_{45}}{x_{12} x_{23} x_{35} + x_{12} x_{24} x_{45}}, d_{5} = \frac{1}{x_{12} x_{23} x_{35} + x_{12} x_{24} x_{45}},\\ 
	&a_{12} = 1, a_{13} = 1, a_{14} = 1, a_{15} = 1,\\ 
	&a_{23} = \frac{1}{x_{12}}, a_{24} = 0, a_{25} = \frac{2}{x_{12}},\\ 
	&a_{34} = -\frac{x_{24} x_{45}}{x_{12} x_{23}^{2} x_{35} + x_{12} x_{23} x_{24} x_{45}}, a_{35} = 1,\\ 
	&a_{45} = -\frac{{\left(x_{12} x_{23} - 1\right)} x_{24} x_{45} + x_{25} + {\left(x_{12} x_{23}^{2} - x_{23}\right)} x_{35}}{x_{12} x_{23} x_{24} x_{35} + x_{12} x_{24}^{2} x_{45}}
\end{align*}

Now assume $x_{23}= \frac{-x_{24}x_{45}}{x_{35}}$
\begin{equation*}x=\left(\begin{matrix}
		1 & x_{12} & 0 & 0 & 0 \\
		0 & 1 & x_{23} & x_{24} & x_{25} \\
		0 & 0 & 1 & 0 & x_{35}  \\
		0 & 0 & 0 & 1 & x_{45}  \\
		0 & 0 & 0 & 0 & 1 \\
	\end{matrix}\right),x^A=\left(\begin{matrix}
		1 & 1 & 0 & 0 & 0 \\
		0 & 1 & 1 & 0 & 0 \\
		0 & 0 & 1 & 0 & 0 \\
		0 & 0 & 0 & 1 & 1 \\
		0 & 0 & 0 & 0 & 1
	\end{matrix}\right)
\end{equation*} \newline
Where matrix $A$ has entries
\begin{align*}
	&d_{1} = 1, d_{2} = \frac{1}{x_{12}}, d_{3} = -\frac{x_{35}}{x_{12} x_{24} x_{45}}, d_{4} = 1, d_{5} = \frac{1}{x_{45}},\\ 
	&a_{12} = 1, a_{13} = 1, a_{14} = 1, a_{15} = 1,\\ 
	&a_{23} = \frac{1}{x_{12}}, a_{24} = 0, a_{25} = \frac{1}{x_{12}},\\ 
	&a_{34} = \frac{x_{35}}{x_{45}}, a_{35} = 1,\\ 
	&a_{45} = \frac{x_{24} x_{45}^{2} - x_{25} x_{35}}{x_{24} x_{35} x_{45}}
\end{align*}


First assume $x_{23}\neq \frac{-x_{24}x_{45}}{x_{35}}$
\begin{equation*}x=\left(\begin{matrix}
		1 & x_{12} & 0 & 0 & x_{15} \\
		0 & 1 & x_{23} & x_{24} & x_{25} \\
		0 & 0 & 1 & 0 & x_{35}  \\
		0 & 0 & 0 & 1 & x_{45}  \\
		0 & 0 & 0 & 0 & 1 \\
	\end{matrix}\right),x^A=\left(\begin{matrix}
		1 & 1 & 0 & 0 & 0 \\
		0 & 1 & 1 & 0 & 0 \\
		0 & 0 & 1 & 0 & 1 \\
		0 & 0 & 0 & 1 & 1 \\
		0 & 0 & 0 & 0 & 1
	\end{matrix}\right)
\end{equation*} \newline
Where matrix $A$ has entries
\begin{align*}
	&d_{1} = 1, d_{2} = \frac{1}{x_{12}}, d_{3} = \frac{1}{x_{12} x_{23}}, d_{4} = \frac{x_{45}}{x_{12} x_{23} x_{35} + x_{12} x_{24} x_{45}}, d_{5} = \frac{1}{x_{12} x_{23} x_{35} + x_{12} x_{24} x_{45}},\\ 
	&a_{12} = 1, a_{13} = 1, a_{14} = 1, a_{15} = 1,\\ 
	&a_{23} = \frac{1}{x_{12}}, a_{24} = 0, a_{25} = \frac{{2} \, x_{12} x_{23} x_{35} + {2} \, x_{12} x_{24} x_{45} - x_{15}}{x_{12}^{2} x_{23} x_{35} + x_{12}^{2} x_{24} x_{45}},\\ 
	&a_{34} = -\frac{x_{24} x_{45}}{x_{12} x_{23}^{2} x_{35} + x_{12} x_{23} x_{24} x_{45}}, a_{35} = 1,\\ 
	&a_{45} = -\frac{{\left(x_{12} x_{23} - 1\right)} x_{24} x_{45} + x_{25} + {\left(x_{12} x_{23}^{2} - x_{23}\right)} x_{35}}{x_{12} x_{23} x_{24} x_{35} + x_{12} x_{24}^{2} x_{45}}
\end{align*}

Now assume $x_{23}= \frac{-x_{24}x_{45}}{x_{35}}$
\begin{equation*}x=\left(\begin{matrix}
		1 & x_{12} & 0 & 0 & x_{15} \\
		0 & 1 & x_{23} & x_{24} & x_{25} \\
		0 & 0 & 1 & 0 & x_{35}  \\
		0 & 0 & 0 & 1 & x_{45}  \\
		0 & 0 & 0 & 0 & 1 \\
	\end{matrix}\right),x^A=\left(\begin{matrix}
		1 & 1 & 0 & 0 & 0 \\
		0 & 1 & 1 & 0 & 0 \\
		0 & 0 & 1 & 0 & 0 \\
		0 & 0 & 0 & 1 & 1 \\
		0 & 0 & 0 & 0 & 1
	\end{matrix}\right)
\end{equation*} \newline
Where matrix $A$ has entries
\begin{align*}
	&d_{1} = 1, d_{2} = \frac{1}{x_{12}}, d_{3} = -\frac{x_{35}}{x_{12} x_{24} x_{45}}, d_{4} = 1, d_{5} = \frac{1}{x_{45}},\\ 
	&a_{12} = 1, a_{13} = 1, a_{14} = 1, a_{15} = 1,\\ 
	&a_{23} = \frac{1}{x_{12}}, a_{24} = 0, a_{25} = -\frac{x_{15} - x_{45}}{x_{12} x_{45}},\\ 
	&a_{34} = \frac{x_{35}}{x_{45}}, a_{35} = 1,\\ 
	&a_{45} = \frac{x_{24} x_{45}^{2} - x_{25} x_{35}}{x_{24} x_{35} x_{45}}
\end{align*}


First assume $x_{23}\neq \frac{-x_{24}x_{45}}{x_{35}}$
\begin{equation*}x=\left(\begin{matrix}
		1 & x_{12} & 0 & x_{14} & 0 \\
		0 & 1 & x_{23} & x_{24} & x_{25} \\
		0 & 0 & 1 & 0 & x_{35}  \\
		0 & 0 & 0 & 1 & x_{45}  \\
		0 & 0 & 0 & 0 & 1 \\
	\end{matrix}\right),x^A=\left(\begin{matrix}
		1 & 1 & 0 & 0 & 0 \\
		0 & 1 & 1 & 0 & 0 \\
		0 & 0 & 1 & 0 & 1 \\
		0 & 0 & 0 & 1 & 1 \\
		0 & 0 & 0 & 0 & 1
	\end{matrix}\right)
\end{equation*} \newline
Where matrix $A$ has entries
\begin{align*}
	&d_{1} = 1, d_{2} = \frac{1}{x_{12}}, d_{3} = \frac{1}{x_{12} x_{23}}, d_{4} = \frac{x_{45}}{x_{12} x_{23} x_{35} + x_{12} x_{24} x_{45}}, d_{5} = \frac{1}{x_{12} x_{23} x_{35} + x_{12} x_{24} x_{45}},\\ 
	&a_{12} = 1, a_{13} = 1, a_{14} = 1, a_{15} = 1,\\ 
	&a_{23} = \frac{1}{x_{12}}, a_{24} = -\frac{x_{14} x_{45}}{x_{12}^{2} x_{23} x_{35} + x_{12}^{2} x_{24} x_{45}}, a_{25} = 1,\\ 
	&a_{34} = -\frac{x_{24} x_{45}}{x_{12} x_{23}^{2} x_{35} + x_{12} x_{23} x_{24} x_{45}}, \\
	&a_{35} = -\frac{x_{12} x_{14} x_{25} - {\left(x_{12} x_{14} x_{23} + {\left(x_{12}^{3} - {2} \, x_{12}^{2}\right)} x_{23} x_{24}\right)} x_{35} - {\left(x_{12} x_{14} x_{24} - x_{14}^{2} + {\left(x_{12}^{3} - {2} \, x_{12}^{2}\right)} x_{24}^{2}\right)} x_{45}}{x_{12}^{2} x_{14} x_{23}^{2} x_{35} + x_{12}^{2} x_{14} x_{23} x_{24} x_{45}},\\ 
	&a_{45} = -\frac{x_{12} - 2}{x_{14}}
\end{align*}

Now assume $x_{23}= \frac{-x_{24}x_{45}}{x_{35}}$
\begin{equation*}x=\left(\begin{matrix}
		1 & x_{12} & 0 & x_{14} & 0 \\
		0 & 1 & x_{23} & x_{24} & x_{25} \\
		0 & 0 & 1 & 0 & x_{35}  \\
		0 & 0 & 0 & 1 & x_{45}  \\
		0 & 0 & 0 & 0 & 1 \\
	\end{matrix}\right),x^A=\left(\begin{matrix}
		1 & 1 & 0 & 0 & 0 \\
		0 & 1 & 1 & 0 & 0 \\
		0 & 0 & 1 & 0 & 0 \\
		0 & 0 & 0 & 1 & 1 \\
		0 & 0 & 0 & 0 & 1
	\end{matrix}\right)
\end{equation*} \newline
Where matrix $A$ has entries
\begin{align*}
	&d_{1} = 1, d_{2} = \frac{1}{x_{12}}, d_{3} = -\frac{x_{35}}{x_{12} x_{24} x_{45}}, d_{4} = 1, d_{5} = \frac{1}{x_{45}},\\ 
	&a_{12} = 1, a_{13} = 1, a_{14} = 1, a_{15} = 1,\\ 
	&a_{23} = \frac{1}{x_{12}}, a_{24} = -\frac{x_{14}}{x_{12}}, a_{25} = 1,\\ 
	&a_{34} = \frac{x_{35}}{x_{45}}, a_{35} = \frac{x_{12} x_{14} x_{25} x_{35} + {\left(x_{14}^{2} - {\left(x_{12}^{2} - x_{12}\right)} x_{24}\right)} x_{35} x_{45}}{x_{12} x_{14} x_{24} x_{45}^{2}},\\ 
	&a_{45} = -\frac{x_{12} - 1}{x_{14}}
\end{align*}


First assume $x_{23}\neq \frac{-x_{24}x_{45}}{x_{35}}$
\begin{equation*}x=\left(\begin{matrix}
		1 & x_{12} & 0 & x_{14} & x_{15} \\
		0 & 1 & x_{23} & x_{24} & x_{25} \\
		0 & 0 & 1 & 0 & x_{35}  \\
		0 & 0 & 0 & 1 & x_{45}  \\
		0 & 0 & 0 & 0 & 1 \\
	\end{matrix}\right),x^A=\left(\begin{matrix}
		1 & 1 & 0 & 0 & 0 \\
		0 & 1 & 1 & 0 & 0 \\
		0 & 0 & 1 & 0 & 1 \\
		0 & 0 & 0 & 1 & 1 \\
		0 & 0 & 0 & 0 & 1
	\end{matrix}\right)
\end{equation*} \newline
Where matrix $A$ has entries
\begin{align*}
	&d_{1} = 1, d_{2} = \frac{1}{x_{12}}, d_{3} = \frac{1}{x_{12} x_{23}}, d_{4} = \frac{x_{45}}{x_{12} x_{23} x_{35} + x_{12} x_{24} x_{45}}, d_{5} = \frac{1}{x_{12} x_{23} x_{35} + x_{12} x_{24} x_{45}},\\ 
	&a_{12} = 1, a_{13} = 1, a_{14} = 1, a_{15} = 1,\\ 
	&a_{23} = \frac{1}{x_{12}}, a_{24} = -\frac{x_{14} x_{45}}{x_{12}^{2} x_{23} x_{35} + x_{12}^{2} x_{24} x_{45}}, a_{25} = 1,\\ 
	&a_{34} = -\frac{x_{24} x_{45}}{x_{12} x_{23}^{2} x_{35} + x_{12} x_{23} x_{24} x_{45}}, \\
	&a_{35} = \frac{x_{12} x_{15} x_{24} - x_{12} x_{14} x_{25} + {\left(x_{12} x_{14} x_{23} + {\left(x_{12}^{3} - {2} \, x_{12}^{2}\right)} x_{23} x_{24}\right)} x_{35} + {\left(x_{12} x_{14} x_{24} - x_{14}^{2} + {\left(x_{12}^{3} - {2} \, x_{12}^{2}\right)} x_{24}^{2}\right)} x_{45}}{x_{12}^{2} x_{14} x_{23}^{2} x_{35} + x_{12}^{2} x_{14} x_{23} x_{24} x_{45}},\\ 
	&a_{45} = -\frac{{\left(x_{12}^{2} - {2} \, x_{12}\right)} x_{23} x_{35} + {\left(x_{12}^{2} - {2} \, x_{12}\right)} x_{24} x_{45} + x_{15}}{x_{12} x_{14} x_{23} x_{35} + x_{12} x_{14} x_{24} x_{45}}
\end{align*}
Now assume $x_{23}=\frac{-x_{24}x_{45}}{x_{35}}$
\begin{equation*}x=\left(\begin{matrix}
		1 & x_{12} & 0 & x_{14} & x_{15} \\
		0 & 1 & x_{23} & x_{24} & x_{25} \\
		0 & 0 & 1 & 0 & x_{35}  \\
		0 & 0 & 0 & 1 & x_{45}  \\
		0 & 0 & 0 & 0 & 1 \\
	\end{matrix}\right),x^A=\left(\begin{matrix}
		1 & 1 & 0 & 0 & 0 \\
		0 & 1 & 1 & 0 & 0 \\
		0 & 0 & 1 & 0 & 0 \\
		0 & 0 & 0 & 1 & 1 \\
		0 & 0 & 0 & 0 & 1
	\end{matrix}\right)
\end{equation*} \newline
Where matrix $A$ has entries
\begin{align*}
	&d_{1} = 1, d_{2} = \frac{1}{x_{12}}, d_{3} = -\frac{x_{35}}{x_{12} x_{24} x_{45}}, d_{4} = 1, d_{5} = \frac{1}{x_{45}},\\ 
	&a_{12} = 1, a_{13} = 1, a_{14} = 1, a_{15} = 1,\\ 
	&a_{23} = \frac{1}{x_{12}}, a_{24} = -\frac{x_{14}}{x_{12}}, a_{25} = 1,\\ 
	&a_{34} = \frac{x_{35}}{x_{45}}, a_{35} = \frac{{\left(x_{14}^{2} - {\left(x_{12}^{2} - x_{12}\right)} x_{24}\right)} x_{35} x_{45} - {\left(x_{12} x_{15} x_{24} - x_{12} x_{14} x_{25}\right)} x_{35}}{x_{12} x_{14} x_{24} x_{45}^{2}},\\ 
	&a_{45} = -\frac{x_{15} + {\left(x_{12} - 1\right)} x_{45}}{x_{14} x_{45}}
\end{align*}

First assume $x_{23}\neq \frac{-x_{24}x_{45}}{x_{35}}$
\begin{equation*}x=\left(\begin{matrix}
		1 & x_{12} & x_{13} & 0 & 0 \\
		0 & 1 & x_{23} & x_{24} & x_{25} \\
		0 & 0 & 1 & 0 & x_{35}  \\
		0 & 0 & 0 & 1 & x_{45}  \\
		0 & 0 & 0 & 0 & 1 \\
	\end{matrix}\right),x^A=\left(\begin{matrix}
		1 & 1 & 0 & 0 & 0 \\
		0 & 1 & 1 & 0 & 0 \\
		0 & 0 & 1 & 0 & 1 \\
		0 & 0 & 0 & 1 & 1 \\
		0 & 0 & 0 & 0 & 1
	\end{matrix}\right)
\end{equation*} \newline
Where matrix $A$ has entries
\begin{align*}
	&d_{1} = 1, d_{2} = \frac{1}{x_{12}}, d_{3} = \frac{1}{x_{12} x_{23}}, d_{4} = \frac{x_{45}}{x_{12} x_{23} x_{35} + x_{12} x_{24} x_{45}}, d_{5} = \frac{1}{x_{12} x_{23} x_{35} + x_{12} x_{24} x_{45}},\\ 
	&a_{12} = 1, a_{13} = 1, a_{14} = 1, a_{15} = 1,\\ 
	&a_{23} = \frac{x_{12} x_{23} - x_{13}}{x_{12}^{2} x_{23}}, a_{24} = \frac{x_{13} x_{24} x_{45}}{x_{12}^{2} x_{23}^{2} x_{35} + x_{12}^{2} x_{23} x_{24} x_{45}}, a_{25} = 1,\\ 
	&a_{34} = -\frac{x_{24} x_{45}}{x_{12} x_{23}^{2} x_{35} + x_{12} x_{23} x_{24} x_{45}}, a_{35} = -\frac{x_{12} - 2}{x_{13}},\\ 
	&a_{45} = -\frac{x_{12} x_{13} x_{25} - {\left(x_{12} x_{13} + {\left(x_{12}^{3} - {2} \, x_{12}^{2}\right)} x_{23}\right)} x_{24} x_{45} - {\left(x_{12} x_{13} x_{23} - x_{13}^{2} + {\left(x_{12}^{3} - {2} \, x_{12}^{2}\right)} x_{23}^{2}\right)} x_{35}}{x_{12}^{2} x_{13} x_{23} x_{24} x_{35} + x_{12}^{2} x_{13} x_{24}^{2} x_{45}}
\end{align*}

Now assume $x_{23}= \frac{-x_{24}x_{45}}{x_{35}}$
\begin{equation*}x=\left(\begin{matrix}
		1 & x_{12} & x_{13} & 0 & 0 \\
		0 & 1 & x_{23} & x_{24} & x_{25} \\
		0 & 0 & 1 & 0 & x_{35}  \\
		0 & 0 & 0 & 1 & x_{45}  \\
		0 & 0 & 0 & 0 & 1 \\
	\end{matrix}\right),x^A=\left(\begin{matrix}
		1 & 1 & 0 & 0 & 0 \\
		0 & 1 & 1 & 0 & 0 \\
		0 & 0 & 1 & 0 & 0 \\
		0 & 0 & 0 & 1 & 1 \\
		0 & 0 & 0 & 0 & 1
	\end{matrix}\right)
\end{equation*} \newline
Where matrix $A$ has entries
\begin{align*}
	&d_{1} = 1, d_{2} = \frac{1}{x_{12}}, d_{3} = -\frac{x_{35}}{x_{12} x_{24} x_{45}}, d_{4} = 1, d_{5} = \frac{1}{x_{45}},\\ 
	&a_{12} = 1, a_{13} = 1, a_{14} = 1, a_{15} = 1,\\ 
	&a_{23} = \frac{x_{12} x_{24} x_{45} + x_{13} x_{35}}{x_{12}^{2} x_{24} x_{45}}, a_{24} = -\frac{x_{13} x_{35}}{x_{12} x_{45}}, a_{25} = 1,\\ 
	&a_{34} = \frac{x_{35}}{x_{45}}, a_{35} = -\frac{x_{12} - 1}{x_{13}},\\ 
	&a_{45} = -\frac{x_{12} x_{13} x_{25} x_{35} + x_{13}^{2} x_{35}^{2} + {\left(x_{12}^{2} - x_{12}\right)} x_{24} x_{45}^{2}}{x_{12} x_{13} x_{24} x_{35} x_{45}}
\end{align*}

First assume $x_{23}\neq \frac{-x_{24}x_{45}}{x_{35}}$
\begin{equation*}x=\left(\begin{matrix}
		1 & x_{12} & x_{13} & 0 & x_{15} \\
		0 & 1 & x_{23} & x_{24} & x_{25} \\
		0 & 0 & 1 & 0 & x_{35}  \\
		0 & 0 & 0 & 1 & x_{45}  \\
		0 & 0 & 0 & 0 & 1 \\
	\end{matrix}\right),x^A=\left(\begin{matrix}
		1 & 1 & 0 & 0 & 0 \\
		0 & 1 & 1 & 0 & 0 \\
		0 & 0 & 1 & 0 & 1 \\
		0 & 0 & 0 & 1 & 1 \\
		0 & 0 & 0 & 0 & 1
	\end{matrix}\right)
\end{equation*} \newline
Where matrix $A$ has entries
\begin{align*}
	&d_{1} = 1, d_{2} = \frac{1}{x_{12}}, d_{3} = \frac{1}{x_{12} x_{23}}, d_{4} = \frac{x_{45}}{x_{12} x_{23} x_{35} + x_{12} x_{24} x_{45}}, d_{5} = \frac{1}{x_{12} x_{23} x_{35} + x_{12} x_{24} x_{45}},\\ 
	&a_{12} = 1, a_{13} = 1, a_{14} = 1, a_{15} = 1,\\ 
	&a_{23} = \frac{x_{12} x_{23} - x_{13}}{x_{12}^{2} x_{23}}, a_{24} = \frac{x_{13} x_{24} x_{45}}{x_{12}^{2} x_{23}^{2} x_{35} + x_{12}^{2} x_{23} x_{24} x_{45}}, a_{25} = 1,\\ 
	&a_{34} = -\frac{x_{24} x_{45}}{x_{12} x_{23}^{2} x_{35} + x_{12} x_{23} x_{24} x_{45}}, a_{35} = -\frac{{\left(x_{12}^{2} - {2} \, x_{12}\right)} x_{23} x_{35} + {\left(x_{12}^{2} - {2} \, x_{12}\right)} x_{24} x_{45} + x_{15}}{x_{12} x_{13} x_{23} x_{35} + x_{12} x_{13} x_{24} x_{45}},\\ 
	&a_{45} = \frac{x_{12} x_{15} x_{23} - x_{12} x_{13} x_{25} + {\left(x_{12} x_{13} + {\left(x_{12}^{3} - {2} \, x_{12}^{2}\right)} x_{23}\right)} x_{24} x_{45} + {\left(x_{12} x_{13} x_{23} - x_{13}^{2} + {\left(x_{12}^{3} - {2} \, x_{12}^{2}\right)} x_{23}^{2}\right)} x_{35}}{x_{12}^{2} x_{13} x_{23} x_{24} x_{35} + x_{12}^{2} x_{13} x_{24}^{2} x_{45}}
\end{align*}

Now assume $x_{23}= \frac{-x_{24}x_{45}}{x_{35}}$
\begin{equation*}x=\left(\begin{matrix}
		1 & x_{12} & x_{13} & 0 & x_{15} \\
		0 & 1 & x_{23} & x_{24} & x_{25} \\
		0 & 0 & 1 & 0 & x_{35}  \\
		0 & 0 & 0 & 1 & x_{45}  \\
		0 & 0 & 0 & 0 & 1 \\
	\end{matrix}\right),x^A=\left(\begin{matrix}
		1 & 1 & 0 & 0 & 0 \\
		0 & 1 & 1 & 0 & 0 \\
		0 & 0 & 1 & 0 & 0 \\
		0 & 0 & 0 & 1 & 1 \\
		0 & 0 & 0 & 0 & 1
	\end{matrix}\right)
\end{equation*} \newline
Where matrix $A$ has entries
\begin{align*}
	&d_{1} = 1, d_{2} = \frac{1}{x_{12}}, d_{3} = -\frac{x_{35}}{x_{12} x_{24} x_{45}}, d_{4} = 1, d_{5} = \frac{1}{x_{45}},\\ 
	&a_{12} = 1, a_{13} = 1, a_{14} = 1, a_{15} = 1,\\ 
	&a_{23} = \frac{x_{12} x_{24} x_{45} + x_{13} x_{35}}{x_{12}^{2} x_{24} x_{45}}, a_{24} = -\frac{x_{13} x_{35}}{x_{12} x_{45}}, a_{25} = 1,\\ 
	&a_{34} = \frac{x_{35}}{x_{45}}, a_{35} = -\frac{x_{15} + {\left(x_{12} - 1\right)} x_{45}}{x_{13} x_{45}},\\ 
	&a_{45} = -\frac{x_{12} x_{13} x_{25} x_{35} + x_{13}^{2} x_{35}^{2} + x_{12} x_{15} x_{24} x_{45} + {\left(x_{12}^{2} - x_{12}\right)} x_{24} x_{45}^{2}}{x_{12} x_{13} x_{24} x_{35} x_{45}}
\end{align*}


First assume $x_{23}\neq \frac{-x_{24}x_{45}}{x_{35}}$ and $x_{23}\neq \frac{x_{13}x_{24}}{x_{14}}$
\begin{equation*}x=\left(\begin{matrix}
		1 & x_{12} & x_{13} & x_{14} & 0 \\
		0 & 1 & x_{23} & x_{24} & x_{25} \\
		0 & 0 & 1 & 0 & x_{35}  \\
		0 & 0 & 0 & 1 & x_{45}  \\
		0 & 0 & 0 & 0 & 1 \\
	\end{matrix}\right),x^A=\left(\begin{matrix}
		1 & 1 & 0 & 0 & 0 \\
		0 & 1 & 1 & 0 & 0 \\
		0 & 0 & 1 & 0 & 1 \\
		0 & 0 & 0 & 1 & 1 \\
		0 & 0 & 0 & 0 & 1
	\end{matrix}\right)
\end{equation*} \newline
Where matrix $A$ has entries
\begin{align*}
	&d_{1} = 1, d_{2} = \frac{1}{x_{12}}, d_{3} = \frac{1}{x_{12} x_{23}}, d_{4} = \frac{x_{45}}{x_{12} x_{23} x_{35} + x_{12} x_{24} x_{45}}, d_{5} = \frac{1}{x_{12} x_{23} x_{35} + x_{12} x_{24} x_{45}}, \\ 
	&a_{12} = 1, a_{13} = 1, a_{14} = 1, a_{15} = 1, \\ 
	&a_{23} = \frac{x_{12} x_{23} - x_{13}}{x_{12}^{2} x_{23}}, a_{24} = -\frac{{\left(x_{14} x_{23} - x_{13} x_{24}\right)} x_{45}}{x_{12}^{2} x_{23}^{2} x_{35} + x_{12}^{2} x_{23} x_{24} x_{45}}, a_{25} = 1, \\ 
	&a_{34} = -\frac{x_{24} x_{45}}{x_{12} x_{23}^{2} x_{35} + x_{12} x_{23} x_{24} x_{45}}, \\
	&a_{35} = -\frac{x_{12} x_{14} x_{25} - {\left(x_{12} x_{14} x_{23} - x_{13} x_{14} + {\left(x_{12}^{3} - 2 x_{12}^{2}\right)} x_{23} x_{24}\right)} x_{35} - {\left(x_{12} x_{14} x_{24} - x_{14}^{2} + {\left(x_{12}^{3} - 2 x_{12}^{2}\right)} x_{24}^{2}\right)} x_{45}}{{\left(x_{12}^{2} x_{14} x_{23}^{2} - x_{12}^{2} x_{13} x_{23} x_{24}\right)} x_{35} + {\left(x_{12}^{2} x_{14} x_{23} x_{24} - x_{12}^{2} x_{13} x_{24}^{2}\right)} x_{45}}, \\ 
	&a_{45} = \frac{x_{12} x_{13} x_{25} - {\left(x_{12} x_{13} x_{23} - x_{13}^{2} + {\left(x_{12}^{3} - 2 x_{12}^{2}\right)} x_{23}^{2}\right)} x_{35} + {\left(x_{13} x_{14} - {\left(x_{12} x_{13} + {\left(x_{12}^{3} - 2 x_{12}^{2}\right)} x_{23}\right)} x_{24}\right)} x_{45}}{{\left(x_{12}^{2} x_{14} x_{23}^{2} - x_{12}^{2} x_{13} x_{23} x_{24}\right)} x_{35} + {\left(x_{12}^{2} x_{14} x_{23} x_{24} - x_{12}^{2} x_{13} x_{24}^{2}\right)} x_{45}}
\end{align*}

Now assume $x_{23}\neq \frac{-x_{24}x_{45}}{x_{35}}$ and $x_{23}= \frac{x_{13}x_{24}}{x_{14}}$ and $x_{35}\neq \frac{-x_{14}x_{45}}{x_{13}}$ 
\begin{equation*}x=\left(\begin{matrix}
		1 & x_{12} & x_{13} & x_{14} & 0 \\
		0 & 1 & x_{23} & x_{24} & x_{25} \\
		0 & 0 & 1 & 0 & x_{35}  \\
		0 & 0 & 0 & 1 & x_{45}  \\
		0 & 0 & 0 & 0 & 1 \\
	\end{matrix}\right),x^A=\left(\begin{matrix}
		1 & 1 & 0 & 0 & 0 \\
		0 & 1 & 1 & 0 & 0 \\
		0 & 0 & 1 & 0 & 1 \\
		0 & 0 & 0 & 1 & 1 \\
		0 & 0 & 0 & 0 & 1
	\end{matrix}\right)
\end{equation*} \newline
Where matrix $A$ has entries
\begin{align*}
	&d_{1} = 1, d_{2} = \frac{1}{x_{12}}, d_{3} = \frac{x_{14}}{x_{12} x_{13} x_{24}}, d_{4} = \frac{x_{14} x_{45}}{x_{12} x_{13} x_{24} x_{35} + x_{12} x_{14} x_{24} x_{45}}, d_{5} = \frac{x_{14}}{x_{12} x_{13} x_{24} x_{35} + x_{12} x_{14} x_{24} x_{45}}, \\ 
	&a_{12} = 1, a_{13} = 1, a_{14} = 1, a_{15} = 1, \\ 
	&a_{23} = \frac{x_{12} x_{24} - x_{14}}{x_{12}^{2} x_{24}}, a_{24} = 0, \\
	&a_{25} = \frac{x_{12} x_{14}^{2} x_{25} + {\left(2 x_{12}^{2} x_{13} x_{24}^{2} - x_{12} x_{13} x_{14} x_{24} + x_{13} x_{14}^{2}\right)} x_{35} + {\left(2 x_{12}^{2} x_{14} x_{24}^{2} - x_{12} x_{14}^{2} x_{24} + x_{14}^{3}\right)} x_{45}}{x_{12}^{3} x_{13} x_{24}^{2} x_{35} + x_{12}^{3} x_{14} x_{24}^{2} x_{45}}, \\ 
	&a_{34} = -\frac{x_{14}^{2} x_{45}}{x_{12} x_{13}^{2} x_{24} x_{35} + x_{12} x_{13} x_{14} x_{24} x_{45}}, a_{35} = 1, \\ 
	&a_{45} = -\frac{x_{12} x_{14}^{2} x_{25} + {\left(x_{12}^{2} x_{13}^{2} x_{24}^{2} - x_{12} x_{13} x_{14} x_{24} + x_{13} x_{14}^{2}\right)} x_{35} + {\left(x_{12}^{2} x_{13} x_{14} x_{24}^{2} - x_{12} x_{14}^{2} x_{24} + x_{14}^{3}\right)} x_{45}}{x_{12}^{2} x_{13} x_{14} x_{24}^{2} x_{35} + x_{12}^{2} x_{14}^{2} x_{24}^{2} x_{45}}
\end{align*}

Now assume $x_{23}\neq \frac{-x_{24}x_{45}}{x_{35}}$ and $x_{23}= \frac{x_{13}x_{24}}{x_{14}}$ and $x_{35}= \frac{-x_{14}x_{45}}{x_{13}}$ 
\begin{equation*}x=\left(\begin{matrix}
		1 & x_{12} & x_{13} & x_{14} & 0 \\
		0 & 1 & x_{23} & x_{24} & x_{25} \\
		0 & 0 & 1 & 0 & x_{35}  \\
		0 & 0 & 0 & 1 & x_{45}  \\
		0 & 0 & 0 & 0 & 1 \\
	\end{matrix}\right),x^A=\left(\begin{matrix}
		1 & 1 & 0 & 0 & 0 \\
		0 & 1 & 1 & 0 & 0 \\
		0 & 0 & 1 & 0 & 0 \\
		0 & 0 & 0 & 1 & 1 \\
		0 & 0 & 0 & 0 & 1
	\end{matrix}\right)
\end{equation*} \newline
Where matrix $A$ has entries
\begin{align*}
	&d_{1} = 1, d_{2} = \frac{1}{x_{12}}, d_{3} = \frac{x_{14}}{x_{12} x_{13} x_{24}}, d_{4} = 1, d_{5} = \frac{1}{x_{45}}, \\ 
	&a_{12} = 1, a_{13} = 1, a_{14} = 1, a_{15} = 1, \\ 
	&a_{23} = \frac{x_{12} x_{24} - x_{14}}{x_{12}^{2} x_{24}}, a_{24} = 0, a_{25} = \frac{x_{14} x_{25} + x_{24} x_{45}}{x_{12} x_{24} x_{45}}, \\ 
	&a_{34} = -\frac{x_{14}}{x_{13}}, a_{35} = 1, \\ 
	&a_{45} = -\frac{x_{13} x_{24} x_{45} + x_{14} x_{25}}{x_{14} x_{24} x_{45}}
\end{align*}

Now assume $x_{23}= \frac{-x_{24}x_{45}}{x_{35}}$ and $x_{45}\neq \frac{-x_{13}x_{35}}{x_{14}}$
\begin{equation*}x=\left(\begin{matrix}
		1 & x_{12} & x_{13} & x_{14} & 0 \\
		0 & 1 & x_{23} & x_{24} & x_{25} \\
		0 & 0 & 1 & 0 & x_{35}  \\
		0 & 0 & 0 & 1 & x_{45}  \\
		0 & 0 & 0 & 0 & 1 \\
	\end{matrix}\right),x^A=\left(\begin{matrix}
		1 & 1 & 0 & 0 & 0 \\
		0 & 1 & 1 & 0 & 0 \\
		0 & 0 & 1 & 0 & 0 \\
		0 & 0 & 0 & 1 & 1 \\
		0 & 0 & 0 & 0 & 1
	\end{matrix}\right)
\end{equation*} \newline
Where matrix $A$ has entries
\begin{align*}
	&d_{1} = 1, d_{2} = \frac{1}{x_{12}}, d_{3} = -\frac{x_{35}}{x_{12} x_{24} x_{45}}, d_{4} = 1, d_{5} = \frac{1}{x_{45}}, \\ 
	&a_{12} = 1, a_{13} = 1, a_{14} = 1, a_{15} = 1, \\ 
	&a_{23} = \frac{x_{12} x_{24} x_{45} + x_{13} x_{35}}{x_{12}^{2} x_{24} x_{45}}, a_{24} = -\frac{x_{13} x_{35} + x_{14} x_{45}}{x_{12} x_{45}}, a_{25} = 1, \\ 
	&a_{34} = \frac{x_{35}}{x_{45}}, a_{35} = \frac{x_{12} x_{14} x_{25} x_{35} + x_{13} x_{14} x_{35}^{2} + {\left(x_{14}^{2} - {\left(x_{12}^{2} - x_{12}\right)} x_{24}\right)} x_{35} x_{45}}{x_{12} x_{13} x_{24} x_{35} x_{45} + x_{12} x_{14} x_{24} x_{45}^{2}}, \\ 
	&a_{45} = -\frac{x_{12} x_{13} x_{25} x_{35} + x_{13}^{2} x_{35}^{2} + x_{13} x_{14} x_{35} x_{45} + {\left(x_{12}^{2} - x_{12}\right)} x_{24} x_{45}^{2}}{x_{12} x_{13} x_{24} x_{35} x_{45} + x_{12} x_{14} x_{24} x_{45}^{2}}
\end{align*}

Now assume $x_{23}= \frac{-x_{24}x_{45}}{x_{35}}$ and $x_{45}= \frac{-x_{13}x_{35}}{x_{14}}$
\begin{equation*}x=\left(\begin{matrix}
		1 & x_{12} & x_{13} & x_{14} & 0 \\
		0 & 1 & x_{23} & x_{24} & x_{25} \\
		0 & 0 & 1 & 0 & x_{35}  \\
		0 & 0 & 0 & 1 & x_{45}  \\
		0 & 0 & 0 & 0 & 1 \\
	\end{matrix}\right),x^A=\left(\begin{matrix}
		1 & 1 & 0 & 0 & 0 \\
		0 & 1 & 1 & 0 & 0 \\
		0 & 0 & 1 & 0 & 0 \\
		0 & 0 & 0 & 1 & 1 \\
		0 & 0 & 0 & 0 & 1
	\end{matrix}\right)
\end{equation*} \newline
Where matrix $A$ has entries
\begin{align*}
	&d_{1} = 1, d_{2} = \frac{1}{x_{12}}, d_{3} = \frac{x_{14}}{x_{12} x_{13} x_{24}}, d_{4} = 1, d_{5} = -\frac{x_{14}}{x_{13} x_{35}}, \\ 
	&a_{12} = 1, a_{13} = 1, a_{14} = 1, a_{15} = 1, \\ 
	&a_{23} = \frac{x_{12} x_{24} - x_{14}}{x_{12}^{2} x_{24}}, a_{24} = 0, a_{25} = -\frac{x_{14}^{2} x_{25} - x_{13} x_{24} x_{35}}{x_{12} x_{13} x_{24} x_{35}}, \\ 
	&a_{34} = -\frac{x_{14}}{x_{13}}, a_{35} = 1, \\ 
	&a_{45} = -\frac{x_{13}^{2} x_{24} x_{35} - x_{14}^{2} x_{25}}{x_{13} x_{14} x_{24} x_{35}}
\end{align*}


First assume $x_{23}\neq \frac{-x_{24}x_{45}}{x_{35}}$ and $x_{23}\neq \frac{x_{13}x_{24}}{x_{14}}$

\begin{equation*}x=\left(\begin{matrix}
		1 & x_{12} & x_{13} & x_{14} & x_{15} \\
		0 & 1 & x_{23} & x_{24} & x_{25} \\
		0 & 0 & 1 & 0 & x_{35}  \\
		0 & 0 & 0 & 1 & x_{45}  \\
		0 & 0 & 0 & 0 & 1 \\
	\end{matrix}\right),x^A=\left(\begin{matrix}
		1 & 1 & 0 & 0 & 0 \\
		0 & 1 & 1 & 0 & 0 \\
		0 & 0 & 1 & 0 & 1 \\
		0 & 0 & 0 & 1 & 1 \\
		0 & 0 & 0 & 0 & 1
	\end{matrix}\right)
\end{equation*} \newline
Where matrix $A$ has entries
\begin{align*}
	&d_{1} = 1, d_{2} = \frac{1}{x_{12}}, d_{3} = \frac{1}{x_{12} x_{23}}, d_{4} = \frac{x_{45}}{x_{12} x_{23} x_{35} + x_{12} x_{24} x_{45}}, d_{5} = \frac{1}{x_{12} x_{23} x_{35} + x_{12} x_{24} x_{45}}, \\ 
	&a_{12} = 1, a_{13} = 1, a_{14} = 1, a_{15} = 1, \\ 
	&a_{23} = \frac{x_{12} x_{23} - x_{13}}{x_{12}^{2} x_{23}}, a_{24} = -\frac{{\left(x_{14} x_{23} - x_{13} x_{24}\right)} x_{45}}{x_{12}^{2} x_{23}^{2} x_{35} + x_{12}^{2} x_{23} x_{24} x_{45}}, a_{25} = 1, \\ 
	&a_{34} = -\frac{x_{24} x_{45}}{x_{12} x_{23}^{2} x_{35} + x_{12} x_{23} x_{24} x_{45}}, \\
	&a_{35} = \frac{x_{12} x_{15} x_{24} - x_{12} x_{14} x_{25} + {\left(x_{12} x_{14} x_{23} - x_{13} x_{14} + {\left(x_{12}^{3} - 2 x_{12}^{2}\right)} x_{23} x_{24}\right)} x_{35} + {\left(x_{12} x_{14} x_{24} - x_{14}^{2} + {\left(x_{12}^{3} - 2 x_{12}^{2}\right)} x_{24}^{2}\right)} x_{45}}{{\left(x_{12}^{2} x_{14} x_{23}^{2} - x_{12}^{2} x_{13} x_{23} x_{24}\right)} x_{35} + {\left(x_{12}^{2} x_{14} x_{23} x_{24} - x_{12}^{2} x_{13} x_{24}^{2}\right)} x_{45}}, \\ 
	&a_{45} = -\frac{x_{12} x_{15} x_{23} - x_{12} x_{13} x_{25} + {\left(x_{12} x_{13} x_{23} - x_{13}^{2} + {\left(x_{12}^{3} - 2 x_{12}^{2}\right)} x_{23}^{2}\right)} x_{35} - {\left(x_{13} x_{14} - {\left(x_{12} x_{13} + {\left(x_{12}^{3} - 2 x_{12}^{2}\right)} x_{23}\right)} x_{24}\right)} x_{45}}{{\left(x_{12}^{2} x_{14} x_{23}^{2} - x_{12}^{2} x_{13} x_{23} x_{24}\right)} x_{35} + {\left(x_{12}^{2} x_{14} x_{23} x_{24} - x_{12}^{2} x_{13} x_{24}^{2}\right)} x_{45}}
\end{align*}

Now assume $x_{23}\neq \frac{-x_{24}x_{45}}{x_{35}}$ and $x_{23}= \frac{x_{13}x_{24}}{x_{14}}$ and $x_{35}\neq \frac{-x_{14}x_{45}}{x_{13}}$ 
\begin{equation*}x=\left(\begin{matrix}
		1 & x_{12} & x_{13} & x_{14} & x_{15} \\
		0 & 1 & x_{23} & x_{24} & x_{25} \\
		0 & 0 & 1 & 0 & x_{35}  \\
		0 & 0 & 0 & 1 & x_{45}  \\
		0 & 0 & 0 & 0 & 1 \\
	\end{matrix}\right),x^A=\left(\begin{matrix}
		1 & 1 & 0 & 0 & 0 \\
		0 & 1 & 1 & 0 & 0 \\
		0 & 0 & 1 & 0 & 0 \\
		0 & 0 & 0 & 1 & 1 \\
		0 & 0 & 0 & 0 & 1
	\end{matrix}\right)
\end{equation*} \newline
Where matrix $A$ has entries
\begin{align*}
	&d_{1} = 1, d_{2} = \frac{1}{x_{12}}, d_{3} = \frac{x_{14}}{x_{12} x_{13} x_{24}}, d_{4} = 1, d_{5} = \frac{1}{x_{45}}, \\ 
	&a_{12} = 1, a_{13} = 1, a_{14} = 1, a_{15} = 1, \\ 
	&a_{23} = \frac{x_{12} x_{24} - x_{14}}{x_{12}^{2} x_{24}}, a_{24} = 0, a_{25} = -\frac{x_{15} x_{24} - x_{14} x_{25} - x_{24} x_{45}}{x_{12} x_{24} x_{45}}, \\ 
	&a_{34} = -\frac{x_{14}}{x_{13}}, a_{35} = 1, \\ 
	&a_{45} = -\frac{x_{13} x_{24} x_{45} + x_{14} x_{25}}{x_{14} x_{24} x_{45}}
\end{align*}

Now assume $x_{23}\neq \frac{-x_{24}x_{45}}{x_{35}}$ and $x_{23}= \frac{x_{13}x_{24}}{x_{14}}$ and $x_{35}= \frac{-x_{14}x_{45}}{x_{13}}$ 

\begin{equation*}x=\left(\begin{matrix}
		1 & x_{12} & x_{13} & x_{14} & x_{15} \\
		0 & 1 & x_{23} & x_{24} & x_{25} \\
		0 & 0 & 1 & 0 & x_{35}  \\
		0 & 0 & 0 & 1 & x_{45}  \\
		0 & 0 & 0 & 0 & 1 \\
	\end{matrix}\right),x^A=\left(\begin{matrix}
		1 & 1 & 0 & 0 & 0 \\
		0 & 1 & 1 & 0 & 0 \\
		0 & 0 & 1 & 0 & 0 \\
		0 & 0 & 0 & 1 & 1 \\
		0 & 0 & 0 & 0 & 1
	\end{matrix}\right)
\end{equation*} \newline
Where matrix $A$ has entries

\begin{align*}
	&d_{1} = 1, d_{2} = \frac{1}{x_{12}}, d_{3} = \frac{x_{14}}{x_{12} x_{13} x_{24}}, d_{4} = 1, d_{5} = \frac{1}{x_{45}}, \\ 
	&a_{12} = 1, a_{13} = 1, a_{14} = 1, a_{15} = 1, \\ 
	&a_{23} = \frac{x_{12} x_{24} - x_{14}}{x_{12}^{2} x_{24}}, a_{24} = 0, a_{25} = -\frac{x_{15} x_{24} - x_{14} x_{25} - x_{24} x_{45}}{x_{12} x_{24} x_{45}}, \\ 
	&a_{34} = -\frac{x_{14}}{x_{13}}, a_{35} = 1, \\ 
	&a_{45} = -\frac{x_{13} x_{24} x_{45} + x_{14} x_{25}}{x_{14} x_{24} x_{45}}
\end{align*}

Now assume $x_{23}= \frac{-x_{24}x_{45}}{x_{35}}$ and $x_{45}\neq \frac{-x_{13}x_{35}}{x_{14}}$
\begin{equation*}x=\left(\begin{matrix}
		1 & x_{12} & x_{13} & x_{14} & x_{15} \\
		0 & 1 & x_{23} & x_{24} & x_{25} \\
		0 & 0 & 1 & 0 & x_{35}  \\
		0 & 0 & 0 & 1 & x_{45}  \\
		0 & 0 & 0 & 0 & 1 \\
	\end{matrix}\right),x^A=\left(\begin{matrix}
		1 & 1 & 0 & 0 & 0 \\
		0 & 1 & 1 & 0 & 0 \\
		0 & 0 & 1 & 0 & 0 \\
		0 & 0 & 0 & 1 & 1 \\
		0 & 0 & 0 & 0 & 1
	\end{matrix}\right)
\end{equation*} \newline
Where matrix $A$ has entries
\begin{align*}
	&d_{1} = 1, d_{2} = \frac{1}{x_{12}}, d_{3} = -\frac{x_{35}}{x_{12} x_{24} x_{45}}, d_{4} = 1, d_{5} = \frac{1}{x_{45}}, \\ 
	&a_{12} = 1, a_{13} = 1, a_{14} = 1, a_{15} = 1, \\ 
	&a_{23} = \frac{x_{12} x_{24} x_{45} + x_{13} x_{35}}{x_{12}^{2} x_{24} x_{45}}, a_{24} = -\frac{x_{13} x_{35} + x_{14} x_{45}}{x_{12} x_{45}}, a_{25} = 1, \\ 
	&a_{34} = \frac{x_{35}}{x_{45}}, a_{35} = \frac{x_{13} x_{14} x_{35}^{2} + {\left(x_{14}^{2} - {\left(x_{12}^{2} - x_{12}\right)} x_{24}\right)} x_{35} x_{45} - {\left(x_{12} x_{15} x_{24} - x_{12} x_{14} x_{25}\right)} x_{35}}{x_{12} x_{13} x_{24} x_{35} x_{45} + x_{12} x_{14} x_{24} x_{45}^{2}}, \\ 
	&a_{45} = -\frac{x_{12} x_{13} x_{25} x_{35} + x_{13}^{2} x_{35}^{2} + {\left(x_{12}^{2} - x_{12}\right)} x_{24} x_{45}^{2} + {\left(x_{12} x_{15} x_{24} + x_{13} x_{14} x_{35}\right)} x_{45}}{x_{12} x_{13} x_{24} x_{35} x_{45} + x_{12} x_{14} x_{24} x_{45}^{2}}
\end{align*}

Now assume $x_{23}= \frac{-x_{24}x_{45}}{x_{35}}$ and $x_{45}= \frac{-x_{13}x_{35}}{x_{14}}$
\begin{equation*}x=\left(\begin{matrix}
		1 & x_{12} & x_{13} & x_{14} & x_{15} \\
		0 & 1 & x_{23} & x_{24} & x_{25} \\
		0 & 0 & 1 & 0 & x_{35}  \\
		0 & 0 & 0 & 1 & x_{45}  \\
		0 & 0 & 0 & 0 & 1 \\
	\end{matrix}\right),x^A=\left(\begin{matrix}
		1 & 1 & 0 & 0 & 0 \\
		0 & 1 & 1 & 0 & 0 \\
		0 & 0 & 1 & 0 & 0 \\
		0 & 0 & 0 & 1 & 1 \\
		0 & 0 & 0 & 0 & 1
	\end{matrix}\right)
\end{equation*} \newline
Where matrix $A$ has entries
\begin{align*}
	&d_{1} = 1, d_{2} = \frac{1}{x_{12}}, d_{3} = \frac{x_{14}}{x_{12} x_{13} x_{24}}, d_{4} = 1, d_{5} = -\frac{x_{14}}{x_{13} x_{35}}, \\ 
	&a_{12} = 1, a_{13} = 1, a_{14} = 1, a_{15} = 1, \\ 
	&a_{23} = \frac{x_{12} x_{24} - x_{14}}{x_{12}^{2} x_{24}}, a_{24} = 0, a_{25} = \frac{x_{14} x_{15} x_{24} - x_{14}^{2} x_{25} + x_{13} x_{24} x_{35}}{x_{12} x_{13} x_{24} x_{35}}, \\ 
	&a_{34} = -\frac{x_{14}}{x_{13}}, a_{35} = 1, \\ 
	&a_{45} = -\frac{x_{13}^{2} x_{24} x_{35} - x_{14}^{2} x_{25}}{x_{13} x_{14} x_{24} x_{35}}
\end{align*}
		
		\section{Subcases of $Y_{10}$}

\begin{equation*}x=\left(
\right)
\end{equation*} \newline
Where matrix $A$ has entries
\begin{align*}
	&d_{1} = 1, d_{2} = \frac{1}{x_{12}}, d_{3} = \frac{1}{x_{12} x_{23}}, d_{4} = \frac{1}{x_{12} x_{23} x_{34}}, d_{5} = \frac{1}{x_{12} x_{23} x_{34} x_{45}},\\ 
	&a_{12} = 0, a_{13} = 1, a_{14} = 1, a_{15} = 1,\\ 
	&a_{23} = 0, a_{24} = \frac{x_{12} x_{23} x_{34} - x_{14}}{x_{12}^{2} x_{23} x_{34}}, a_{25} = \frac{x_{12} x_{23} x_{34} x_{45} - x_{15}}{x_{12}^{2} x_{23} x_{34} x_{45}},\\ 
	&a_{34} = 0, a_{35} = \frac{x_{12} x_{23} x_{34} - x_{14}}{x_{12}^{2} x_{23}^{2} x_{34}},\\ 
	&a_{45} = 0
\end{align*}

\begin{equation*}x=\left(\begin{matrix}
		1 & x_{12} & x_{13} & 0 & 0 \\
		0 & 1 & x_{23} & 0 & 0 \\
		0 & 0 & 1 & x_{34} & 0  \\
		0 & 0 & 0 & 1 & x_{45}  \\
		0 & 0 & 0 & 1 & 1 \\
	\end{matrix}\right),x^A=\left(\begin{matrix}
		1 & 1 & 0 & 0 & 0 \\
		0 & 1 & 1 & 0 & 0 \\
		0 & 0 & 1 & 1 & 0 \\
		0 & 0 & 0 & 1 & 1 \\
		0 & 0 & 0 & 0 & 1
	\end{matrix}\right)
\end{equation*} \newline
Where matrix $A$ has entries
\begin{align*}
	&d_{1} = 1, d_{2} = \frac{1}{x_{12}}, d_{3} = \frac{1}{x_{12} x_{23}}, d_{4} = \frac{1}{x_{12} x_{23} x_{34}}, d_{5} = \frac{1}{x_{12} x_{23} x_{34} x_{45}},\\ 
	&a_{12} = \frac{x_{13}}{x_{12} x_{23}}, a_{13} = 1, a_{14} = 1, a_{15} = 1,\\ 
	&a_{23} = 0, a_{24} = \frac{1}{x_{12}}, a_{25} = \frac{x_{12}^{3} x_{23}^{3} - x_{12}^{2} x_{13} x_{23}^{2}}{x_{12}^{4} x_{23}^{3}},\\ 
	&a_{34} = 0, a_{35} = \frac{1}{x_{12} x_{23}},\\ 
	&a_{45} = 0
\end{align*}

\begin{equation*}x=\left(\begin{matrix}
		1 & x_{12} & x_{13} & 0 & x_{15} \\
		0 & 1 & x_{23} & 0 & 0 \\
		0 & 0 & 1 & x_{34} & 0  \\
		0 & 0 & 0 & 1 & x_{45}  \\
		0 & 0 & 0 & 1 & 1 \\
	\end{matrix}\right),x^A=\left(\begin{matrix}
		1 & 1 & 0 & 0 & 0 \\
		0 & 1 & 1 & 0 & 0 \\
		0 & 0 & 1 & 1 & 0 \\
		0 & 0 & 0 & 1 & 1 \\
		0 & 0 & 0 & 0 & 1
	\end{matrix}\right)
\end{equation*} \newline
Where matrix $A$ has entries
\begin{align*}
	&d_{1} = 1, d_{2} = \frac{1}{x_{12}}, d_{3} = \frac{1}{x_{12} x_{23}}, d_{4} = \frac{1}{x_{12} x_{23} x_{34}}, d_{5} = \frac{1}{x_{12} x_{23} x_{34} x_{45}},\\ 
	&a_{12} = \frac{x_{13}}{x_{12} x_{23}}, a_{13} = 1, a_{14} = 1, a_{15} = 1,\\ 
	&a_{23} = 0, a_{24} = \frac{1}{x_{12}}, a_{25} = -\frac{x_{12}^{2} x_{15} x_{23}^{2} - {\left(x_{12}^{3} x_{23}^{3} - x_{12}^{2} x_{13} x_{23}^{2}\right)} x_{34} x_{45}}{x_{12}^{4} x_{23}^{3} x_{34} x_{45}},\\ 
	&a_{34} = 0, a_{35} = \frac{1}{x_{12} x_{23}},\\ 
	&a_{45} = 0
\end{align*}

\begin{equation*}x=\left(\begin{matrix}
		1 & x_{12} & x_{13} & x_{14} & 0 \\
		0 & 1 & x_{23} & 0 & 0 \\
		0 & 0 & 1 & x_{34} & 0  \\
		0 & 0 & 0 & 1 & x_{45}  \\
		0 & 0 & 0 & 1 & 1 \\
	\end{matrix}\right),x^A=\left(\begin{matrix}
		1 & 1 & 0 & 0 & 0 \\
		0 & 1 & 1 & 0 & 0 \\
		0 & 0 & 1 & 1 & 0 \\
		0 & 0 & 0 & 1 & 1 \\
		0 & 0 & 0 & 0 & 1
	\end{matrix}\right)
\end{equation*} \newline
Where matrix $A$ has entries
\begin{align*}
	&d_{1} = 1, d_{2} = \frac{1}{x_{12}}, d_{3} = \frac{1}{x_{12} x_{23}}, d_{4} = \frac{1}{x_{12} x_{23} x_{34}}, d_{5} = \frac{1}{x_{12} x_{23} x_{34} x_{45}},\\ 
	&a_{12} = \frac{x_{13}}{x_{12} x_{23}}, a_{13} = 1, a_{14} = 1, a_{15} = 1,\\ 
	&a_{23} = 0, a_{24} = \frac{x_{12}^{2} x_{23}^{2} x_{34} - x_{12} x_{14} x_{23}}{x_{12}^{3} x_{23}^{2} x_{34}}, a_{25} = \frac{x_{12} x_{13} x_{14} x_{23} + {\left(x_{12}^{3} x_{23}^{3} - x_{12}^{2} x_{13} x_{23}^{2}\right)} x_{34}}{x_{12}^{4} x_{23}^{3} x_{34}},\\ 
	&a_{34} = 0, a_{35} = \frac{x_{12}^{2} x_{23}^{2} x_{34} - x_{12} x_{14} x_{23}}{x_{12}^{3} x_{23}^{3} x_{34}},\\ 
	&a_{45} = 0
\end{align*}

\begin{equation*}x=\left(\begin{matrix}
		1 & x_{12} & x_{13} & x_{14} & x_{15} \\
		0 & 1 & x_{23} & 0 & 0 \\
		0 & 0 & 1 & x_{34} & 0  \\
		0 & 0 & 0 & 1 & x_{45}  \\
		0 & 0 & 0 & 1 & 1 \\
	\end{matrix}\right),x^A=\left(\begin{matrix}
		1 & 1 & 0 & 0 & 0 \\
		0 & 1 & 1 & 0 & 0 \\
		0 & 0 & 1 & 1 & 0 \\
		0 & 0 & 0 & 1 & 1 \\
		0 & 0 & 0 & 0 & 1
	\end{matrix}\right)
\end{equation*} \newline
Where matrix $A$ has entries
\begin{align*}
	&d_{1} = 1, d_{2} = \frac{1}{x_{12}}, d_{3} = \frac{1}{x_{12} x_{23}}, d_{4} = \frac{1}{x_{12} x_{23} x_{34}}, d_{5} = \frac{1}{x_{12} x_{23} x_{34} x_{45}},\\ 
	&a_{12} = \frac{x_{13}}{x_{12} x_{23}}, a_{13} = 1, a_{14} = 1, a_{15} = 1,\\ 
	&a_{23} = 0, a_{24} = \frac{x_{12}^{2} x_{23}^{2} x_{34} - x_{12} x_{14} x_{23}}{x_{12}^{3} x_{23}^{2} x_{34}}, a_{25} = -\frac{x_{12}^{2} x_{15} x_{23}^{2} - {\left(x_{12} x_{13} x_{14} x_{23} + {\left(x_{12}^{3} x_{23}^{3} - x_{12}^{2} x_{13} x_{23}^{2}\right)} x_{34}\right)} x_{45}}{x_{12}^{4} x_{23}^{3} x_{34} x_{45}},\\ 
	&a_{34} = 0, a_{35} = \frac{x_{12}^{2} x_{23}^{2} x_{34} - x_{12} x_{14} x_{23}}{x_{12}^{3} x_{23}^{3} x_{34}},\\ 
	&a_{45} = 0
\end{align*}

\begin{equation*}x=\left(
\right)
\end{equation*} \newline
Where matrix $A$ has entries
\begin{align*}
	&d_{1} = 1, d_{2} = \frac{1}{x_{12}}, d_{3} = \frac{1}{x_{12} x_{23}}, d_{4} = \frac{1}{x_{12} x_{23} x_{34}}, d_{5} = \frac{1}{x_{12} x_{23} x_{34} x_{45}},\\ 
	&a_{12} = 0, a_{13} = 1, a_{14} = 1, a_{15} = 1,\\ 
	&a_{23} = 0, a_{24} = \frac{x_{12} x_{23} x_{34} - x_{14}}{x_{12}^{2} x_{23} x_{34}}, a_{25} = \frac{1}{x_{12}},\\ 
	&a_{34} = 0, a_{35} = -\frac{x_{12} x_{25} - {\left(x_{12} x_{23} x_{34} - x_{14}\right)} x_{45}}{x_{12}^{2} x_{23}^{2} x_{34} x_{45}},\\ 
	&a_{45} = 0
\end{align*}

\begin{equation*}x=\left(\begin{matrix}
		1 & x_{12} & 0 & x_{14} & x_{15} \\
		0 & 1 & x_{23} & 0 & x_{25} \\
		0 & 0 & 1 & x_{34} & 0  \\
		0 & 0 & 0 & 1 & x_{45}  \\
		0 & 0 & 0 & 1 & 1 \\
	\end{matrix}\right),x^A=\left(\begin{matrix}
		1 & 1 & 0 & 0 & 0 \\
		0 & 1 & 1 & 0 & 0 \\
		0 & 0 & 1 & 1 & 0 \\
		0 & 0 & 0 & 1 & 1 \\
		0 & 0 & 0 & 0 & 1
	\end{matrix}\right)
\end{equation*} \newline
Where matrix $A$ has entries
\begin{align*}
	&d_{1} = 1, d_{2} = \frac{1}{x_{12}}, d_{3} = \frac{1}{x_{12} x_{23}}, d_{4} = \frac{1}{x_{12} x_{23} x_{34}}, d_{5} = \frac{1}{x_{12} x_{23} x_{34} x_{45}},\\ 
	&a_{12} = 0, a_{13} = 1, a_{14} = 1, a_{15} = 1,\\ 
	&a_{23} = 0, a_{24} = \frac{x_{12} x_{23} x_{34} - x_{14}}{x_{12}^{2} x_{23} x_{34}}, a_{25} = \frac{x_{12} x_{23} x_{34} x_{45} - x_{15}}{x_{12}^{2} x_{23} x_{34} x_{45}},\\ 
	&a_{34} = 0, a_{35} = -\frac{x_{12} x_{25} - {\left(x_{12} x_{23} x_{34} - x_{14}\right)} x_{45}}{x_{12}^{2} x_{23}^{2} x_{34} x_{45}},\\ 
	&a_{45} = 0
\end{align*}

\begin{equation*}x=\left(\begin{matrix}
		1 & x_{12} & x_{13} & 0 & 0 \\
		0 & 1 & x_{23} & 0 & x_{25} \\
		0 & 0 & 1 & x_{34} & 0  \\
		0 & 0 & 0 & 1 & x_{45}  \\
		0 & 0 & 0 & 1 & 1 \\
	\end{matrix}\right),x^A=\left(\begin{matrix}
		1 & 1 & 0 & 0 & 0 \\
		0 & 1 & 1 & 0 & 0 \\
		0 & 0 & 1 & 1 & 0 \\
		0 & 0 & 0 & 1 & 1 \\
		0 & 0 & 0 & 0 & 1
	\end{matrix}\right)
\end{equation*} \newline
Where matrix $A$ has entries
\begin{align*}
	&d_{1} = 1, d_{2} = \frac{1}{x_{12}}, d_{3} = \frac{1}{x_{12} x_{23}}, d_{4} = \frac{1}{x_{12} x_{23} x_{34}}, d_{5} = \frac{1}{x_{12} x_{23} x_{34} x_{45}},\\ 
	&a_{12} = \frac{x_{13}}{x_{12} x_{23}}, a_{13} = 1, a_{14} = 1, a_{15} = 1,\\ 
	&a_{23} = 0, a_{24} = \frac{1}{x_{12}}, a_{25} = \frac{x_{12}^{2} x_{13} x_{23} x_{25} + {\left(x_{12}^{3} x_{23}^{3} - x_{12}^{2} x_{13} x_{23}^{2}\right)} x_{34} x_{45}}{x_{12}^{4} x_{23}^{3} x_{34} x_{45}},\\ 
	&a_{34} = 0, a_{35} = \frac{x_{12}^{2} x_{23}^{2} x_{34} x_{45} - x_{12}^{2} x_{23} x_{25}}{x_{12}^{3} x_{23}^{3} x_{34} x_{45}},\\ 
	&a_{45} = 0
\end{align*}

\begin{equation*}x=\left(\begin{matrix}
		1 & x_{12} & x_{13} & 0 & x_{15} \\
		0 & 1 & x_{23} & 0 & x_{25} \\
		0 & 0 & 1 & x_{34} & 0  \\
		0 & 0 & 0 & 1 & x_{45}  \\
		0 & 0 & 0 & 1 & 1 \\
	\end{matrix}\right),x^A=\left(\begin{matrix}
		1 & 1 & 0 & 0 & 0 \\
		0 & 1 & 1 & 0 & 0 \\
		0 & 0 & 1 & 1 & 0 \\
		0 & 0 & 0 & 1 & 1 \\
		0 & 0 & 0 & 0 & 1
	\end{matrix}\right)
\end{equation*} \newline
Where matrix $A$ has entries
\begin{align*}
	&d_{1} = 1, d_{2} = \frac{1}{x_{12}}, d_{3} = \frac{1}{x_{12} x_{23}}, d_{4} = \frac{1}{x_{12} x_{23} x_{34}}, d_{5} = \frac{1}{x_{12} x_{23} x_{34} x_{45}},\\ 
	&a_{12} = \frac{x_{13}}{x_{12} x_{23}}, a_{13} = 1, a_{14} = 1, a_{15} = 1,\\ 
	&a_{23} = 0, a_{24} = \frac{1}{x_{12}}, a_{25} = -\frac{x_{12}^{2} x_{15} x_{23}^{2} - x_{12}^{2} x_{13} x_{23} x_{25} - {\left(x_{12}^{3} x_{23}^{3} - x_{12}^{2} x_{13} x_{23}^{2}\right)} x_{34} x_{45}}{x_{12}^{4} x_{23}^{3} x_{34} x_{45}},\\ 
	&a_{34} = 0, a_{35} = \frac{x_{12}^{2} x_{23}^{2} x_{34} x_{45} - x_{12}^{2} x_{23} x_{25}}{x_{12}^{3} x_{23}^{3} x_{34} x_{45}},\\ 
	&a_{45} = 0
\end{align*}

\begin{equation*}x=\left(\begin{matrix}
		1 & x_{12} & x_{13} & x_{14} & 0 \\
		0 & 1 & x_{23} & 0 & x_{25} \\
		0 & 0 & 1 & x_{34} & 0  \\
		0 & 0 & 0 & 1 & x_{45}  \\
		0 & 0 & 0 & 1 & 1 \\
	\end{matrix}\right),x^A=\left(\begin{matrix}
		1 & 1 & 0 & 0 & 0 \\
		0 & 1 & 1 & 0 & 0 \\
		0 & 0 & 1 & 1 & 0 \\
		0 & 0 & 0 & 1 & 1 \\
		0 & 0 & 0 & 0 & 1
	\end{matrix}\right)
\end{equation*} \newline
Where matrix $A$ has entries
\begin{align*}
	&d_{1} = 1, d_{2} = \frac{1}{x_{12}}, d_{3} = \frac{1}{x_{12} x_{23}}, d_{4} = \frac{1}{x_{12} x_{23} x_{34}}, d_{5} = \frac{1}{x_{12} x_{23} x_{34} x_{45}},\\ 
	&a_{12} = \frac{x_{13}}{x_{12} x_{23}}, a_{13} = 1, a_{14} = 1, a_{15} = 1,\\ 
	&a_{23} = 0, a_{24} = \frac{x_{12}^{2} x_{23}^{2} x_{34} - x_{12} x_{14} x_{23}}{x_{12}^{3} x_{23}^{2} x_{34}}, a_{25} = \frac{x_{12}^{2} x_{13} x_{23} x_{25} + {\left(x_{12} x_{13} x_{14} x_{23} + {\left(x_{12}^{3} x_{23}^{3} - x_{12}^{2} x_{13} x_{23}^{2}\right)} x_{34}\right)} x_{45}}{x_{12}^{4} x_{23}^{3} x_{34} x_{45}},\\ 
	&a_{34} = 0, a_{35} = -\frac{x_{12}^{2} x_{23} x_{25} - {\left(x_{12}^{2} x_{23}^{2} x_{34} - x_{12} x_{14} x_{23}\right)} x_{45}}{x_{12}^{3} x_{23}^{3} x_{34} x_{45}},\\ 
	&a_{45} = 0
\end{align*}

\begin{equation*}x=\left(\begin{matrix}
		1 & x_{12} & x_{13} & x_{14} & x_{15} \\
		0 & 1 & x_{23} & 0 & x_{25} \\
		0 & 0 & 1 & x_{34} & 0  \\
		0 & 0 & 0 & 1 & x_{45}  \\
		0 & 0 & 0 & 1 & 1 \\
	\end{matrix}\right),x^A=\left(\begin{matrix}
		1 & 1 & 0 & 0 & 0 \\
		0 & 1 & 1 & 0 & 0 \\
		0 & 0 & 1 & 1 & 0 \\
		0 & 0 & 0 & 1 & 1 \\
		0 & 0 & 0 & 0 & 1
	\end{matrix}\right)
\end{equation*} \newline
Where matrix $A$ has entries
\begin{align*}
	&d_{1} = 1, d_{2} = \frac{1}{x_{12}}, d_{3} = \frac{1}{x_{12} x_{23}}, d_{4} = \frac{1}{x_{12} x_{23} x_{34}}, d_{5} = \frac{1}{x_{12} x_{23} x_{34} x_{45}},\\ 
	&a_{12} = \frac{x_{13}}{x_{12} x_{23}}, a_{13} = 1, a_{14} = 1, a_{15} = 1,\\ 
	&a_{23} = 0, a_{24} = \frac{x_{12}^{2} x_{23}^{2} x_{34} - x_{12} x_{14} x_{23}}{x_{12}^{3} x_{23}^{2} x_{34}}, a_{25} = -\frac{x_{12}^{2} x_{15} x_{23}^{2} - x_{12}^{2} x_{13} x_{23} x_{25} - {\left(x_{12} x_{13} x_{14} x_{23} + {\left(x_{12}^{3} x_{23}^{3} - x_{12}^{2} x_{13} x_{23}^{2}\right)} x_{34}\right)} x_{45}}{x_{12}^{4} x_{23}^{3} x_{34} x_{45}},\\ 
	&a_{34} = 0, a_{35} = -\frac{x_{12}^{2} x_{23} x_{25} - {\left(x_{12}^{2} x_{23}^{2} x_{34} - x_{12} x_{14} x_{23}\right)} x_{45}}{x_{12}^{3} x_{23}^{3} x_{34} x_{45}},\\ 
	&a_{45} = 0
\end{align*}

\begin{equation*}x=\left(\begin{matrix}
		1 & x_{12} & 0 & 0 & 0 \\
		0 & 1 & x_{23} & x_{24} & 0 \\
		0 & 0 & 1 & x_{34} & 0  \\
		0 & 0 & 0 & 1 & x_{45}  \\
		0 & 0 & 0 & 1 & 1 \\
	\end{matrix}\right),x^A=\left(\begin{matrix}
		1 & 1 & 0 & 0 & 0 \\
		0 & 1 & 1 & 0 & 0 \\
		0 & 0 & 1 & 1 & 0 \\
		0 & 0 & 0 & 1 & 1 \\
		0 & 0 & 0 & 0 & 1
	\end{matrix}\right)
\end{equation*} \newline
Where matrix $A$ has entries
\begin{align*}
	&d_{1} = 1, d_{2} = \frac{1}{x_{12}}, d_{3} = \frac{1}{x_{12} x_{23}}, d_{4} = \frac{1}{x_{12} x_{23} x_{34}}, d_{5} = \frac{1}{x_{12} x_{23} x_{34} x_{45}},\\ 
	&a_{12} = 0, a_{13} = 1, a_{14} = 1, a_{15} = 1,\\ 
	&a_{23} = 0, a_{24} = \frac{1}{x_{12}}, a_{25} = \frac{1}{x_{12}},\\ 
	&a_{34} = -\frac{x_{24}}{x_{12} x_{23}^{2} x_{34}}, a_{35} = \frac{x_{23}^{2} x_{34}^{2} + x_{24}^{2}}{x_{12} x_{23}^{3} x_{34}^{2}},\\ 
	&a_{45} = -\frac{x_{24}}{x_{12} x_{23}^{2} x_{34}^{2}}
\end{align*}

\begin{equation*}x=\left(\begin{matrix}
		1 & x_{12} & 0 & 0 & x_{15} \\
		0 & 1 & x_{23} & x_{24} & 0 \\
		0 & 0 & 1 & x_{34} & 0  \\
		0 & 0 & 0 & 1 & x_{45}  \\
		0 & 0 & 0 & 1 & 1 \\
	\end{matrix}\right),x^A=\left(\begin{matrix}
		1 & 1 & 0 & 0 & 0 \\
		0 & 1 & 1 & 0 & 0 \\
		0 & 0 & 1 & 1 & 0 \\
		0 & 0 & 0 & 1 & 1 \\
		0 & 0 & 0 & 0 & 1
	\end{matrix}\right)
\end{equation*} \newline
Where matrix $A$ has entries
\begin{align*}
	&d_{1} = 1, d_{2} = \frac{1}{x_{12}}, d_{3} = \frac{1}{x_{12} x_{23}}, d_{4} = \frac{1}{x_{12} x_{23} x_{34}}, d_{5} = \frac{1}{x_{12} x_{23} x_{34} x_{45}},\\ 
	&a_{12} = 0, a_{13} = 1, a_{14} = 1, a_{15} = 1,\\ 
	&a_{23} = 0, a_{24} = \frac{1}{x_{12}}, a_{25} = \frac{x_{12} x_{23} x_{34} x_{45} - x_{15}}{x_{12}^{2} x_{23} x_{34} x_{45}},\\ 
	&a_{34} = -\frac{x_{24}}{x_{12} x_{23}^{2} x_{34}}, a_{35} = \frac{x_{23}^{2} x_{34}^{2} + x_{24}^{2}}{x_{12} x_{23}^{3} x_{34}^{2}},\\ 
	&a_{45} = -\frac{x_{24}}{x_{12} x_{23}^{2} x_{34}^{2}}
\end{align*}

\begin{equation*}x=\left(\begin{matrix}
		1 & x_{12} & 0 & x_{14} & 0 \\
		0 & 1 & x_{23} & x_{24} & 0 \\
		0 & 0 & 1 & x_{34} & 0  \\
		0 & 0 & 0 & 1 & x_{45}  \\
		0 & 0 & 0 & 1 & 1 \\
	\end{matrix}\right),x^A=\left(\begin{matrix}
		1 & 1 & 0 & 0 & 0 \\
		0 & 1 & 1 & 0 & 0 \\
		0 & 0 & 1 & 1 & 0 \\
		0 & 0 & 0 & 1 & 1 \\
		0 & 0 & 0 & 0 & 1
	\end{matrix}\right)
\end{equation*} \newline
Where matrix $A$ has entries
\begin{align*}
	&d_{1} = 1, d_{2} = \frac{1}{x_{12}}, d_{3} = \frac{1}{x_{12} x_{23}}, d_{4} = \frac{1}{x_{12} x_{23} x_{34}}, d_{5} = \frac{1}{x_{12} x_{23} x_{34} x_{45}},\\ 
	&a_{12} = 0, a_{13} = 1, a_{14} = 1, a_{15} = 1,\\ 
	&a_{23} = 0, a_{24} = \frac{x_{12} x_{23} x_{34} - x_{14}}{x_{12}^{2} x_{23} x_{34}}, a_{25} = \frac{x_{12} x_{23}^{2} x_{34}^{2} + x_{14} x_{24}}{x_{12}^{2} x_{23}^{2} x_{34}^{2}},\\ 
	&a_{34} = -\frac{x_{24}}{x_{12} x_{23}^{2} x_{34}}, a_{35} = \frac{x_{12} x_{23}^{2} x_{34}^{2} + x_{12} x_{24}^{2} - x_{14} x_{23} x_{34}}{x_{12}^{2} x_{23}^{3} x_{34}^{2}},\\ 
	&a_{45} = -\frac{x_{24}}{x_{12} x_{23}^{2} x_{34}^{2}}
\end{align*}

\begin{equation*}x=\left(\begin{matrix}
		1 & x_{12} & 0 & x_{14} & x_{15} \\
		0 & 1 & x_{23} & x_{24} & 0 \\
		0 & 0 & 1 & x_{34} & 0  \\
		0 & 0 & 0 & 1 & x_{45}  \\
		0 & 0 & 0 & 1 & 1 \\
	\end{matrix}\right),x^A=\left(\begin{matrix}
		1 & 1 & 0 & 0 & 0 \\
		0 & 1 & 1 & 0 & 0 \\
		0 & 0 & 1 & 1 & 0 \\
		0 & 0 & 0 & 1 & 1 \\
		0 & 0 & 0 & 0 & 1
	\end{matrix}\right)
\end{equation*} \newline
Where matrix $A$ has entries
\begin{align*}
	&d_{1} = 1, d_{2} = \frac{1}{x_{12}}, d_{3} = \frac{1}{x_{12} x_{23}}, d_{4} = \frac{1}{x_{12} x_{23} x_{34}}, d_{5} = \frac{1}{x_{12} x_{23} x_{34} x_{45}},\\ 
	&a_{12} = 0, a_{13} = 1, a_{14} = 1, a_{15} = 1,\\ 
	&a_{23} = 0, a_{24} = \frac{x_{12} x_{23} x_{34} - x_{14}}{x_{12}^{2} x_{23} x_{34}}, a_{25} = -\frac{x_{15} x_{23} x_{34} - {\left(x_{12} x_{23}^{2} x_{34}^{2} + x_{14} x_{24}\right)} x_{45}}{x_{12}^{2} x_{23}^{2} x_{34}^{2} x_{45}},\\ 
	&a_{34} = -\frac{x_{24}}{x_{12} x_{23}^{2} x_{34}}, a_{35} = \frac{x_{12} x_{23}^{2} x_{34}^{2} + x_{12} x_{24}^{2} - x_{14} x_{23} x_{34}}{x_{12}^{2} x_{23}^{3} x_{34}^{2}},\\ 
	&a_{45} = -\frac{x_{24}}{x_{12} x_{23}^{2} x_{34}^{2}}
\end{align*}

\begin{equation*}x=\left(\begin{matrix}
		1 & x_{12} & x_{13} & 0 & 0 \\
		0 & 1 & x_{23} & x_{24} & 0 \\
		0 & 0 & 1 & x_{34} & 0  \\
		0 & 0 & 0 & 1 & x_{45}  \\
		0 & 0 & 0 & 1 & 1 \\
	\end{matrix}\right),x^A=\left(\begin{matrix}
		1 & 1 & 0 & 0 & 0 \\
		0 & 1 & 1 & 0 & 0 \\
		0 & 0 & 1 & 1 & 0 \\
		0 & 0 & 0 & 1 & 1 \\
		0 & 0 & 0 & 0 & 1
	\end{matrix}\right)
\end{equation*} \newline
Where matrix $A$ has entries
\begin{align*}
	&d_{1} = 1, d_{2} = \frac{1}{x_{12}}, d_{3} = \frac{1}{x_{12} x_{23}}, d_{4} = \frac{1}{x_{12} x_{23} x_{34}}, d_{5} = \frac{1}{x_{12} x_{23} x_{34} x_{45}},\\ 
	&a_{12} = \frac{x_{13}}{x_{12} x_{23}}, a_{13} = 1, a_{14} = 1, a_{15} = 1,\\ 
	&a_{23} = 0, a_{24} = \frac{x_{12}^{2} x_{23}^{2} x_{34} + x_{12} x_{13} x_{24}}{x_{12}^{3} x_{23}^{2} x_{34}}, a_{25} = -\frac{x_{12}^{2} x_{13} x_{24}^{2} + x_{12} x_{13}^{2} x_{24} x_{34} - {\left(x_{12}^{3} x_{23}^{3} - x_{12}^{2} x_{13} x_{23}^{2}\right)} x_{34}^{2}}{x_{12}^{4} x_{23}^{3} x_{34}^{2}},\\ 
	&a_{34} = -\frac{x_{24}}{x_{12} x_{23}^{2} x_{34}}, a_{35} = \frac{x_{12}^{2} x_{23}^{2} x_{34}^{2} + x_{12}^{2} x_{24}^{2} + x_{12} x_{13} x_{24} x_{34}}{x_{12}^{3} x_{23}^{3} x_{34}^{2}},\\ 
	&a_{45} = -\frac{x_{24}}{x_{12} x_{23}^{2} x_{34}^{2}}
\end{align*}

\begin{equation*}x=\left(\begin{matrix}
		1 & x_{12} & x_{13} & 0 & x_{15} \\
		0 & 1 & x_{23} & x_{24} & 0 \\
		0 & 0 & 1 & x_{34} & 0  \\
		0 & 0 & 0 & 1 & x_{45}  \\
		0 & 0 & 0 & 1 & 1 \\
	\end{matrix}\right),x^A=\left(\begin{matrix}
		1 & 1 & 0 & 0 & 0 \\
		0 & 1 & 1 & 0 & 0 \\
		0 & 0 & 1 & 1 & 0 \\
		0 & 0 & 0 & 1 & 1 \\
		0 & 0 & 0 & 0 & 1
	\end{matrix}\right)
\end{equation*} \newline
Where matrix $A$ has entries
\begin{align*}
	&d_{1} = 1, d_{2} = \frac{1}{x_{12}}, d_{3} = \frac{1}{x_{12} x_{23}}, d_{4} = \frac{1}{x_{12} x_{23} x_{34}}, d_{5} = \frac{1}{x_{12} x_{23} x_{34} x_{45}},\\ 
	&a_{12} = \frac{x_{13}}{x_{12} x_{23}}, a_{13} = 1, a_{14} = 1, a_{15} = 1,\\ 
	&a_{23} = 0, a_{24} = \frac{x_{12}^{2} x_{23}^{2} x_{34} + x_{12} x_{13} x_{24}}{x_{12}^{3} x_{23}^{2} x_{34}}, a_{25} = -\frac{x_{12}^{2} x_{15} x_{23}^{2} x_{34} + {\left(x_{12}^{2} x_{13} x_{24}^{2} + x_{12} x_{13}^{2} x_{24} x_{34} - {\left(x_{12}^{3} x_{23}^{3} - x_{12}^{2} x_{13} x_{23}^{2}\right)} x_{34}^{2}\right)} x_{45}}{x_{12}^{4} x_{23}^{3} x_{34}^{2} x_{45}},\\ 
	&a_{34} = -\frac{x_{24}}{x_{12} x_{23}^{2} x_{34}}, a_{35} = \frac{x_{12}^{2} x_{23}^{2} x_{34}^{2} + x_{12}^{2} x_{24}^{2} + x_{12} x_{13} x_{24} x_{34}}{x_{12}^{3} x_{23}^{3} x_{34}^{2}},\\ 
	&a_{45} = -\frac{x_{24}}{x_{12} x_{23}^{2} x_{34}^{2}}
\end{align*}

\begin{equation*}x=\left(\begin{matrix}
		1 & x_{12} & x_{13} & x_{14} & 0 \\
		0 & 1 & x_{23} & x_{24} & 0 \\
		0 & 0 & 1 & x_{34} & 0  \\
		0 & 0 & 0 & 1 & x_{45}  \\
		0 & 0 & 0 & 1 & 1 \\
	\end{matrix}\right),x^A=\left(\begin{matrix}
		1 & 1 & 0 & 0 & 0 \\
		0 & 1 & 1 & 0 & 0 \\
		0 & 0 & 1 & 1 & 0 \\
		0 & 0 & 0 & 1 & 1 \\
		0 & 0 & 0 & 0 & 1
	\end{matrix}\right)
\end{equation*} \newline
Where matrix $A$ has entries
\begin{align*}
	&d_{1} = 1, d_{2} = \frac{1}{x_{12}}, d_{3} = \frac{1}{x_{12} x_{23}}, d_{4} = \frac{1}{x_{12} x_{23} x_{34}}, d_{5} = \frac{1}{x_{12} x_{23} x_{34} x_{45}},\\ 
	&a_{12} = \frac{x_{13}}{x_{12} x_{23}}, a_{13} = 1, a_{14} = 1, a_{15} = 1,\\ 
	&a_{23} = 0, a_{24} = \frac{x_{12}^{2} x_{23}^{2} x_{34} - x_{12} x_{14} x_{23} + x_{12} x_{13} x_{24}}{x_{12}^{3} x_{23}^{2} x_{34}}, \\
	&a_{25} = \frac{x_{12}^{2} x_{14} x_{23} x_{24} - x_{12}^{2} x_{13} x_{24}^{2} + {\left(x_{12}^{3} x_{23}^{3} - x_{12}^{2} x_{13} x_{23}^{2}\right)} x_{34}^{2} + {\left(x_{12} x_{13} x_{14} x_{23} - x_{12} x_{13}^{2} x_{24}\right)} x_{34}}{x_{12}^{4} x_{23}^{3} x_{34}^{2}},\\ 
	&a_{34} = -\frac{x_{24}}{x_{12} x_{23}^{2} x_{34}}, a_{35} = \frac{x_{12}^{2} x_{23}^{2} x_{34}^{2} + x_{12}^{2} x_{24}^{2} - {\left(x_{12} x_{14} x_{23} - x_{12} x_{13} x_{24}\right)} x_{34}}{x_{12}^{3} x_{23}^{3} x_{34}^{2}},\\ 
	&a_{45} = -\frac{x_{24}}{x_{12} x_{23}^{2} x_{34}^{2}}
\end{align*}

\begin{equation*}x=\left(\begin{matrix}
		1 & x_{12} & x_{13} & x_{14} & x_{15} \\
		0 & 1 & x_{23} & x_{24} & 0 \\
		0 & 0 & 1 & x_{34} & 0  \\
		0 & 0 & 0 & 1 & x_{45}  \\
		0 & 0 & 0 & 1 & 1 \\
	\end{matrix}\right),x^A=\left(\begin{matrix}
		1 & 1 & 0 & 0 & 0 \\
		0 & 1 & 1 & 0 & 0 \\
		0 & 0 & 1 & 1 & 0 \\
		0 & 0 & 0 & 1 & 1 \\
		0 & 0 & 0 & 0 & 1
	\end{matrix}\right)
\end{equation*} \newline
Where matrix $A$ has entries
\begin{align*}
	&d_{1} = 1, d_{2} = \frac{1}{x_{12}}, d_{3} = \frac{1}{x_{12} x_{23}}, d_{4} = \frac{1}{x_{12} x_{23} x_{34}}, d_{5} = \frac{1}{x_{12} x_{23} x_{34} x_{45}},\\ 
	&a_{12} = \frac{x_{13}}{x_{12} x_{23}}, a_{13} = 1, a_{14} = 1, a_{15} = 1,\\ 
	&a_{23} = 0, a_{24} = \frac{x_{12}^{2} x_{23}^{2} x_{34} - x_{12} x_{14} x_{23} + x_{12} x_{13} x_{24}}{x_{12}^{3} x_{23}^{2} x_{34}}, \\
	&a_{25} = -\frac{x_{12}^{2} x_{15} x_{23}^{2} x_{34} - {\left(x_{12}^{2} x_{14} x_{23} x_{24} - x_{12}^{2} x_{13} x_{24}^{2} + {\left(x_{12}^{3} x_{23}^{3} - x_{12}^{2} x_{13} x_{23}^{2}\right)} x_{34}^{2} + {\left(x_{12} x_{13} x_{14} x_{23} - x_{12} x_{13}^{2} x_{24}\right)} x_{34}\right)} x_{45}}{x_{12}^{4} x_{23}^{3} x_{34}^{2} x_{45}},\\ 
	&a_{34} = -\frac{x_{24}}{x_{12} x_{23}^{2} x_{34}}, a_{35} = \frac{x_{12}^{2} x_{23}^{2} x_{34}^{2} + x_{12}^{2} x_{24}^{2} - {\left(x_{12} x_{14} x_{23} - x_{12} x_{13} x_{24}\right)} x_{34}}{x_{12}^{3} x_{23}^{3} x_{34}^{2}},\\ 
	&a_{45} = -\frac{x_{24}}{x_{12} x_{23}^{2} x_{34}^{2}}
\end{align*}

\begin{equation*}x=\left(\begin{matrix}
		1 & x_{12} & 0 & 0 & 0 \\
		0 & 1 & x_{23} & x_{24} & x_{25} \\
		0 & 0 & 1 & x_{34} & 0  \\
		0 & 0 & 0 & 1 & x_{45}  \\
		0 & 0 & 0 & 1 & 1 \\
	\end{matrix}\right),x^A=\left(\begin{matrix}
		1 & 1 & 0 & 0 & 0 \\
		0 & 1 & 1 & 0 & 0 \\
		0 & 0 & 1 & 1 & 0 \\
		0 & 0 & 0 & 1 & 1 \\
		0 & 0 & 0 & 0 & 1
	\end{matrix}\right)
\end{equation*} \newline
Where matrix $A$ has entries
\begin{align*}
	&d_{1} = 1, d_{2} = \frac{1}{x_{12}}, d_{3} = \frac{1}{x_{12} x_{23}}, d_{4} = \frac{1}{x_{12} x_{23} x_{34}}, d_{5} = \frac{1}{x_{12} x_{23} x_{34} x_{45}},\\ 
	&a_{12} = 0, a_{13} = 1, a_{14} = 1, a_{15} = 1,\\ 
	&a_{23} = 0, a_{24} = \frac{1}{x_{12}}, a_{25} = \frac{1}{x_{12}},\\ 
	&a_{34} = -\frac{x_{24}}{x_{12} x_{23}^{2} x_{34}}, a_{35} = -\frac{x_{23} x_{25} x_{34} - {\left(x_{23}^{2} x_{34}^{2} + x_{24}^{2}\right)} x_{45}}{x_{12} x_{23}^{3} x_{34}^{2} x_{45}},\\ 
	&a_{45} = -\frac{x_{24}}{x_{12} x_{23}^{2} x_{34}^{2}}
\end{align*}

\begin{equation*}x=\left(\begin{matrix}
		1 & x_{12} & 0 & 0 & x_{15} \\
		0 & 1 & x_{23} & x_{24} & x_{25} \\
		0 & 0 & 1 & x_{34} & 0  \\
		0 & 0 & 0 & 1 & x_{45}  \\
		0 & 0 & 0 & 1 & 1 \\
	\end{matrix}\right),x^A=\left(\begin{matrix}
		1 & 1 & 0 & 0 & 0 \\
		0 & 1 & 1 & 0 & 0 \\
		0 & 0 & 1 & 1 & 0 \\
		0 & 0 & 0 & 1 & 1 \\
		0 & 0 & 0 & 0 & 1
	\end{matrix}\right)
\end{equation*} \newline
Where matrix $A$ has entries
\begin{align*}
	&d_{1} = 1, d_{2} = \frac{1}{x_{12}}, d_{3} = \frac{1}{x_{12} x_{23}}, d_{4} = \frac{1}{x_{12} x_{23} x_{34}}, d_{5} = \frac{1}{x_{12} x_{23} x_{34} x_{45}},\\ 
	&a_{12} = 0, a_{13} = 1, a_{14} = 1, a_{15} = 1,\\ 
	&a_{23} = 0, a_{24} = \frac{1}{x_{12}}, a_{25} = \frac{x_{12} x_{23} x_{34} x_{45} - x_{15}}{x_{12}^{2} x_{23} x_{34} x_{45}},\\ 
	&a_{34} = -\frac{x_{24}}{x_{12} x_{23}^{2} x_{34}}, a_{35} = -\frac{x_{23} x_{25} x_{34} - {\left(x_{23}^{2} x_{34}^{2} + x_{24}^{2}\right)} x_{45}}{x_{12} x_{23}^{3} x_{34}^{2} x_{45}},\\ 
	&a_{45} = -\frac{x_{24}}{x_{12} x_{23}^{2} x_{34}^{2}}
\end{align*}

\begin{equation*}x=\left(\begin{matrix}
		1 & x_{12} & 0 & x_{14} & 0 \\
		0 & 1 & x_{23} & x_{24} & x_{25} \\
		0 & 0 & 1 & x_{34} & 0  \\
		0 & 0 & 0 & 1 & x_{45}  \\
		0 & 0 & 0 & 1 & 1 \\
	\end{matrix}\right),x^A=\left(\begin{matrix}
		1 & 1 & 0 & 0 & 0 \\
		0 & 1 & 1 & 0 & 0 \\
		0 & 0 & 1 & 1 & 0 \\
		0 & 0 & 0 & 1 & 1 \\
		0 & 0 & 0 & 0 & 1
	\end{matrix}\right)
\end{equation*} \newline
Where matrix $A$ has entries
\begin{align*}
	&d_{1} = 1, d_{2} = \frac{1}{x_{12}}, d_{3} = \frac{1}{x_{12} x_{23}}, d_{4} = \frac{1}{x_{12} x_{23} x_{34}}, d_{5} = \frac{1}{x_{12} x_{23} x_{34} x_{45}},\\ 
	&a_{12} = 0, a_{13} = 1, a_{14} = 1, a_{15} = 1,\\ 
	&a_{23} = 0, a_{24} = \frac{x_{12} x_{23} x_{34} - x_{14}}{x_{12}^{2} x_{23} x_{34}}, a_{25} = \frac{x_{12} x_{23}^{2} x_{34}^{2} + x_{14} x_{24}}{x_{12}^{2} x_{23}^{2} x_{34}^{2}},\\ 
	&a_{34} = -\frac{x_{24}}{x_{12} x_{23}^{2} x_{34}}, a_{35} = -\frac{x_{12} x_{23} x_{25} x_{34} - {\left(x_{12} x_{23}^{2} x_{34}^{2} + x_{12} x_{24}^{2} - x_{14} x_{23} x_{34}\right)} x_{45}}{x_{12}^{2} x_{23}^{3} x_{34}^{2} x_{45}},\\ 
	&a_{45} = -\frac{x_{24}}{x_{12} x_{23}^{2} x_{34}^{2}}
\end{align*}

\begin{equation*}x=\left(\begin{matrix}
		1 & x_{12} & 0 & x_{14} & x_{15} \\
		0 & 1 & x_{23} & x_{24} & x_{25} \\
		0 & 0 & 1 & x_{34} & 0  \\
		0 & 0 & 0 & 1 & x_{45}  \\
		0 & 0 & 0 & 1 & 1 \\
	\end{matrix}\right),x^A=\left(\begin{matrix}
		1 & 1 & 0 & 0 & 0 \\
		0 & 1 & 1 & 0 & 0 \\
		0 & 0 & 1 & 1 & 0 \\
		0 & 0 & 0 & 1 & 1 \\
		0 & 0 & 0 & 0 & 1
	\end{matrix}\right)
\end{equation*} \newline
Where matrix $A$ has entries
\begin{align*}
	&d_{1} = 1, d_{2} = \frac{1}{x_{12}}, d_{3} = \frac{1}{x_{12} x_{23}}, d_{4} = \frac{1}{x_{12} x_{23} x_{34}}, d_{5} = \frac{1}{x_{12} x_{23} x_{34} x_{45}},\\ 
	&a_{12} = 0, a_{13} = 1, a_{14} = 1, a_{15} = 1,\\ 
	&a_{23} = 0, a_{24} = \frac{x_{12} x_{23} x_{34} - x_{14}}{x_{12}^{2} x_{23} x_{34}}, a_{25} = -\frac{x_{15} x_{23} x_{34} - {\left(x_{12} x_{23}^{2} x_{34}^{2} + x_{14} x_{24}\right)} x_{45}}{x_{12}^{2} x_{23}^{2} x_{34}^{2} x_{45}},\\ 
	&a_{34} = -\frac{x_{24}}{x_{12} x_{23}^{2} x_{34}}, a_{35} = -\frac{x_{12} x_{23} x_{25} x_{34} - {\left(x_{12} x_{23}^{2} x_{34}^{2} + x_{12} x_{24}^{2} - x_{14} x_{23} x_{34}\right)} x_{45}}{x_{12}^{2} x_{23}^{3} x_{34}^{2} x_{45}},\\ 
	&a_{45} = -\frac{x_{24}}{x_{12} x_{23}^{2} x_{34}^{2}}
\end{align*}

\begin{equation*}x=\left(\begin{matrix}
		1 & x_{12} & x_{13} & 0 & 0 \\
		0 & 1 & x_{23} & x_{24} & x_{25} \\
		0 & 0 & 1 & x_{34} & 0  \\
		0 & 0 & 0 & 1 & x_{45}  \\
		0 & 0 & 0 & 1 & 1 \\
	\end{matrix}\right),x^A=\left(\begin{matrix}
		1 & 1 & 0 & 0 & 0 \\
		0 & 1 & 1 & 0 & 0 \\
		0 & 0 & 1 & 1 & 0 \\
		0 & 0 & 0 & 1 & 1 \\
		0 & 0 & 0 & 0 & 1
	\end{matrix}\right)
\end{equation*} \newline
Where matrix $A$ has entries
\begin{align*}
	&d_{1} = 1, d_{2} = \frac{1}{x_{12}}, d_{3} = \frac{1}{x_{12} x_{23}}, d_{4} = \frac{1}{x_{12} x_{23} x_{34}}, d_{5} = \frac{1}{x_{12} x_{23} x_{34} x_{45}},\\ 
	&a_{12} = \frac{x_{13}}{x_{12} x_{23}}, a_{13} = 1, a_{14} = 1, a_{15} = 1,\\ 
	&a_{23} = 0, a_{24} = \frac{x_{12}^{2} x_{23}^{2} x_{34} + x_{12} x_{13} x_{24}}{x_{12}^{3} x_{23}^{2} x_{34}}, a_{25} = \frac{x_{12}^{2} x_{13} x_{23} x_{25} x_{34} - {\left(x_{12}^{2} x_{13} x_{24}^{2} + x_{12} x_{13}^{2} x_{24} x_{34} - {\left(x_{12}^{3} x_{23}^{3} - x_{12}^{2} x_{13} x_{23}^{2}\right)} x_{34}^{2}\right)} x_{45}}{x_{12}^{4} x_{23}^{3} x_{34}^{2} x_{45}},\\ 
	&a_{34} = -\frac{x_{24}}{x_{12} x_{23}^{2} x_{34}}, a_{35} = -\frac{x_{12}^{2} x_{23} x_{25} x_{34} - {\left(x_{12}^{2} x_{23}^{2} x_{34}^{2} + x_{12}^{2} x_{24}^{2} + x_{12} x_{13} x_{24} x_{34}\right)} x_{45}}{x_{12}^{3} x_{23}^{3} x_{34}^{2} x_{45}},\\ 
	&a_{45} = -\frac{x_{24}}{x_{12} x_{23}^{2} x_{34}^{2}}
\end{align*}

\begin{equation*}x=\left(\begin{matrix}
		1 & x_{12} & x_{13} & 0 & x_{15} \\
		0 & 1 & x_{23} & x_{24} & x_{25} \\
		0 & 0 & 1 & x_{34} & 0  \\
		0 & 0 & 0 & 1 & x_{45}  \\
		0 & 0 & 0 & 1 & 1 \\
	\end{matrix}\right),x^A=\left(\begin{matrix}
		1 & 1 & 0 & 0 & 0 \\
		0 & 1 & 1 & 0 & 0 \\
		0 & 0 & 1 & 1 & 0 \\
		0 & 0 & 0 & 1 & 1 \\
		0 & 0 & 0 & 0 & 1
	\end{matrix}\right)
\end{equation*} \newline
Where matrix $A$ has entries
\begin{align*}
	&d_{1} = 1, d_{2} = \frac{1}{x_{12}}, d_{3} = \frac{1}{x_{12} x_{23}}, d_{4} = \frac{1}{x_{12} x_{23} x_{34}}, d_{5} = \frac{1}{x_{12} x_{23} x_{34} x_{45}},\\ 
	&a_{12} = \frac{x_{13}}{x_{12} x_{23}}, a_{13} = 1, a_{14} = 1, a_{15} = 1,\\ 
	&a_{23} = 0, a_{24} = \frac{x_{12}^{2} x_{23}^{2} x_{34} + x_{12} x_{13} x_{24}}{x_{12}^{3} x_{23}^{2} x_{34}}, \\
	&a_{25} = -\frac{{\left(x_{12}^{2} x_{15} x_{23}^{2} - x_{12}^{2} x_{13} x_{23} x_{25}\right)} x_{34} + {\left(x_{12}^{2} x_{13} x_{24}^{2} + x_{12} x_{13}^{2} x_{24} x_{34} - {\left(x_{12}^{3} x_{23}^{3} - x_{12}^{2} x_{13} x_{23}^{2}\right)} x_{34}^{2}\right)} x_{45}}{x_{12}^{4} x_{23}^{3} x_{34}^{2} x_{45}},\\ 
	&a_{34} = -\frac{x_{24}}{x_{12} x_{23}^{2} x_{34}}, a_{35} = -\frac{x_{12}^{2} x_{23} x_{25} x_{34} - {\left(x_{12}^{2} x_{23}^{2} x_{34}^{2} + x_{12}^{2} x_{24}^{2} + x_{12} x_{13} x_{24} x_{34}\right)} x_{45}}{x_{12}^{3} x_{23}^{3} x_{34}^{2} x_{45}},\\ 
	&a_{45} = -\frac{x_{24}}{x_{12} x_{23}^{2} x_{34}^{2}}
\end{align*}

\begin{equation*}x=\left(\begin{matrix}
		1 & x_{12} & x_{13} & x_{14} & 0 \\
		0 & 1 & x_{23} & x_{24} & x_{25} \\
		0 & 0 & 1 & x_{34} & 0  \\
		0 & 0 & 0 & 1 & x_{45}  \\
		0 & 0 & 0 & 1 & 1 \\
	\end{matrix}\right),x^A=\left(\begin{matrix}
		1 & 1 & 0 & 0 & 0 \\
		0 & 1 & 1 & 0 & 0 \\
		0 & 0 & 1 & 1 & 0 \\
		0 & 0 & 0 & 1 & 1 \\
		0 & 0 & 0 & 0 & 1
	\end{matrix}\right)
\end{equation*} \newline
Where matrix $A$ has entries
\begin{align*}
	&d_{1} = 1, d_{2} = \frac{1}{x_{12}}, d_{3} = \frac{1}{x_{12} x_{23}}, d_{4} = \frac{1}{x_{12} x_{23} x_{34}}, d_{5} = \frac{1}{x_{12} x_{23} x_{34} x_{45}},\\ 
	&a_{12} = \frac{x_{13}}{x_{12} x_{23}}, a_{13} = 1, a_{14} = 1, a_{15} = 1,\\ 
	&a_{23} = 0, a_{24} = \frac{x_{12}^{2} x_{23}^{2} x_{34} - x_{12} x_{14} x_{23} + x_{12} x_{13} x_{24}}{x_{12}^{3} x_{23}^{2} x_{34}}, \\
	&a_{25} = \frac{x_{12}^{2} x_{13} x_{23} x_{25} x_{34} + {\left(x_{12}^{2} x_{14} x_{23} x_{24} - x_{12}^{2} x_{13} x_{24}^{2} + {\left(x_{12}^{3} x_{23}^{3} - x_{12}^{2} x_{13} x_{23}^{2}\right)} x_{34}^{2} + {\left(x_{12} x_{13} x_{14} x_{23} - x_{12} x_{13}^{2} x_{24}\right)} x_{34}\right)} x_{45}}{x_{12}^{4} x_{23}^{3} x_{34}^{2} x_{45}},\\ 
	&a_{34} = -\frac{x_{24}}{x_{12} x_{23}^{2} x_{34}}, a_{35} = -\frac{x_{12}^{2} x_{23} x_{25} x_{34} - {\left(x_{12}^{2} x_{23}^{2} x_{34}^{2} + x_{12}^{2} x_{24}^{2} - {\left(x_{12} x_{14} x_{23} - x_{12} x_{13} x_{24}\right)} x_{34}\right)} x_{45}}{x_{12}^{3} x_{23}^{3} x_{34}^{2} x_{45}},\\ 
	&a_{45} = -\frac{x_{24}}{x_{12} x_{23}^{2} x_{34}^{2}}
\end{align*}

\begin{equation*}x=\left(\begin{matrix}
		1 & x_{12} & x_{13} & x_{14} & x_{15} \\
		0 & 1 & x_{23} & x_{24} & x_{25} \\
		0 & 0 & 1 & x_{34} & 0  \\
		0 & 0 & 0 & 1 & x_{45}  \\
		0 & 0 & 0 & 1 & 1 \\
	\end{matrix}\right),x^A=\left(\begin{matrix}
		1 & 1 & 0 & 0 & 0 \\
		0 & 1 & 1 & 0 & 0 \\
		0 & 0 & 1 & 1 & 0 \\
		0 & 0 & 0 & 1 & 1 \\
		0 & 0 & 0 & 0 & 1
	\end{matrix}\right)
\end{equation*} \newline
Where matrix $A$ has entries
\begin{align*}
	&d_{1} = 1, d_{2} = \frac{1}{x_{12}}, d_{3} = \frac{1}{x_{12} x_{23}}, d_{4} = \frac{1}{x_{12} x_{23} x_{34}}, d_{5} = \frac{1}{x_{12} x_{23} x_{34} x_{45}},\\ 
	&a_{12} = \frac{x_{13}}{x_{12} x_{23}}, a_{13} = 1, a_{14} = 1, a_{15} = 1,\\ 
	&a_{23} = 0, a_{24} = \frac{x_{12}^{2} x_{23}^{2} x_{34} - x_{12} x_{14} x_{23} + x_{12} x_{13} x_{24}}{x_{12}^{3} x_{23}^{2} x_{34}}, \\
	&a_{25} = -\frac{\splitfrac{{\left(x_{12}^{2} x_{15} x_{23}^{2} - x_{12}^{2} x_{13} x_{23} x_{25}\right)} x_{34} -}
		{	- {\left(x_{12}^{2} x_{14} x_{23} x_{24} - x_{12}^{2} x_{13} x_{24}^{2} + {\left(x_{12}^{3} x_{23}^{3} - x_{12}^{2} x_{13} x_{23}^{2}\right)} x_{34}^{2} + {\left(x_{12} x_{13} x_{14} x_{23} - x_{12} x_{13}^{2} x_{24}\right)} x_{34}\right)} x_{45}}}{x_{12}^{4} x_{23}^{3} x_{34}^{2} x_{45}},\\ 
	&a_{34} = -\frac{x_{24}}{x_{12} x_{23}^{2} x_{34}}, a_{35} = -\frac{x_{12}^{2} x_{23} x_{25} x_{34} - {\left(x_{12}^{2} x_{23}^{2} x_{34}^{2} + x_{12}^{2} x_{24}^{2} - {\left(x_{12} x_{14} x_{23} - x_{12} x_{13} x_{24}\right)} x_{34}\right)} x_{45}}{x_{12}^{3} x_{23}^{3} x_{34}^{2} x_{45}},\\ 
	&a_{45} = -\frac{x_{24}}{x_{12} x_{23}^{2} x_{34}^{2}}
\end{align*}

\begin{equation*}x=\left(
\right)
\end{equation*} \newline
Where matrix $A$ has entries
\begin{align*}
	&d_{1} = 1, d_{2} = \frac{1}{x_{12}}, d_{3} = \frac{1}{x_{12} x_{23}}, d_{4} = \frac{1}{x_{12} x_{23} x_{34}}, d_{5} = \frac{1}{x_{12} x_{23} x_{34} x_{45}},\\ 
	&a_{12} = 0, a_{13} = 1, a_{14} = 1, a_{15} = 1,\\ 
	&a_{23} = 0, a_{24} = \frac{x_{12} x_{23} x_{34} - x_{14}}{x_{12}^{2} x_{23} x_{34}}, a_{25} = \frac{x_{12} x_{23} x_{34}^{2} x_{45} + x_{14} x_{35}}{x_{12}^{2} x_{23} x_{34}^{2} x_{45}},\\ 
	&a_{34} = 0, a_{35} = \frac{x_{12} x_{23} x_{34} - x_{14}}{x_{12}^{2} x_{23}^{2} x_{34}},\\ 
	&a_{45} = -\frac{x_{35}}{x_{12} x_{23} x_{34}^{2} x_{45}}
\end{align*}

\begin{equation*}x=\left(\begin{matrix}
		1 & x_{12} & 0 & x_{14} & x_{15} \\
		0 & 1 & x_{23} & 0 & 0 \\
		0 & 0 & 1 & x_{34} & x_{35}  \\
		0 & 0 & 0 & 1 & x_{45}  \\
		0 & 0 & 0 & 1 & 1 \\
	\end{matrix}\right),x^A=\left(\begin{matrix}
		1 & 1 & 0 & 0 & 0 \\
		0 & 1 & 1 & 0 & 0 \\
		0 & 0 & 1 & 1 & 0 \\
		0 & 0 & 0 & 1 & 1 \\
		0 & 0 & 0 & 0 & 1
	\end{matrix}\right)
\end{equation*} \newline
Where matrix $A$ has entries
\begin{align*}
	&d_{1} = 1, d_{2} = \frac{1}{x_{12}}, d_{3} = \frac{1}{x_{12} x_{23}}, d_{4} = \frac{1}{x_{12} x_{23} x_{34}}, d_{5} = \frac{1}{x_{12} x_{23} x_{34} x_{45}},\\ 
	&a_{12} = 0, a_{13} = 1, a_{14} = 1, a_{15} = 1,\\ 
	&a_{23} = 0, a_{24} = \frac{x_{12} x_{23} x_{34} - x_{14}}{x_{12}^{2} x_{23} x_{34}}, a_{25} = \frac{x_{12} x_{23} x_{34}^{2} x_{45} - x_{15} x_{34} + x_{14} x_{35}}{x_{12}^{2} x_{23} x_{34}^{2} x_{45}},\\ 
	&a_{34} = 0, a_{35} = \frac{x_{12} x_{23} x_{34} - x_{14}}{x_{12}^{2} x_{23}^{2} x_{34}},\\ 
	&a_{45} = -\frac{x_{35}}{x_{12} x_{23} x_{34}^{2} x_{45}}
\end{align*}

\begin{equation*}x=\left(\begin{matrix}
		1 & x_{12} & x_{13} & 0 & 0 \\
		0 & 1 & x_{23} & 0 & 0 \\
		0 & 0 & 1 & x_{34} & x_{35}  \\
		0 & 0 & 0 & 1 & x_{45}  \\
		0 & 0 & 0 & 1 & 1 \\
	\end{matrix}\right),x^A=\left(\begin{matrix}
		1 & 1 & 0 & 0 & 0 \\
		0 & 1 & 1 & 0 & 0 \\
		0 & 0 & 1 & 1 & 0 \\
		0 & 0 & 0 & 1 & 1 \\
		0 & 0 & 0 & 0 & 1
	\end{matrix}\right)
\end{equation*} \newline
Where matrix $A$ has entries
\begin{align*}
	&d_{1} = 1, d_{2} = \frac{1}{x_{12}}, d_{3} = \frac{1}{x_{12} x_{23}}, d_{4} = \frac{1}{x_{12} x_{23} x_{34}}, d_{5} = \frac{1}{x_{12} x_{23} x_{34} x_{45}},\\ 
	&a_{12} = \frac{x_{13}}{x_{12} x_{23}}, a_{13} = 1, a_{14} = 1, a_{15} = 1,\\ 
	&a_{23} = 0, a_{24} = \frac{1}{x_{12}}, a_{25} = \frac{x_{12}^{3} x_{23}^{3} - x_{12}^{2} x_{13} x_{23}^{2}}{x_{12}^{4} x_{23}^{3}},\\ 
	&a_{34} = 0, a_{35} = \frac{1}{x_{12} x_{23}},\\ 
	&a_{45} = -\frac{x_{35}}{x_{12} x_{23} x_{34}^{2} x_{45}}
\end{align*}

\begin{equation*}x=\left(\begin{matrix}
		1 & x_{12} & x_{13} & 0 & x_{15} \\
		0 & 1 & x_{23} & 0 & 0 \\
		0 & 0 & 1 & x_{34} & x_{35}  \\
		0 & 0 & 0 & 1 & x_{45}  \\
		0 & 0 & 0 & 1 & 1 \\
	\end{matrix}\right),x^A=\left(\begin{matrix}
		1 & 1 & 0 & 0 & 0 \\
		0 & 1 & 1 & 0 & 0 \\
		0 & 0 & 1 & 1 & 0 \\
		0 & 0 & 0 & 1 & 1 \\
		0 & 0 & 0 & 0 & 1
	\end{matrix}\right)
\end{equation*} \newline
Where matrix $A$ has entries
\begin{align*}
	&d_{1} = 1, d_{2} = \frac{1}{x_{12}}, d_{3} = \frac{1}{x_{12} x_{23}}, d_{4} = \frac{1}{x_{12} x_{23} x_{34}}, d_{5} = \frac{1}{x_{12} x_{23} x_{34} x_{45}},\\ 
	&a_{12} = \frac{x_{13}}{x_{12} x_{23}}, a_{13} = 1, a_{14} = 1, a_{15} = 1,\\ 
	&a_{23} = 0, a_{24} = \frac{1}{x_{12}}, a_{25} = -\frac{x_{12}^{2} x_{15} x_{23}^{2} - {\left(x_{12}^{3} x_{23}^{3} - x_{12}^{2} x_{13} x_{23}^{2}\right)} x_{34} x_{45}}{x_{12}^{4} x_{23}^{3} x_{34} x_{45}},\\ 
	&a_{34} = 0, a_{35} = \frac{1}{x_{12} x_{23}},\\ 
	&a_{45} = -\frac{x_{35}}{x_{12} x_{23} x_{34}^{2} x_{45}}
\end{align*}

\begin{equation*}x=\left(\begin{matrix}
		1 & x_{12} & x_{13} & x_{14} & 0 \\
		0 & 1 & x_{23} & 0 & 0 \\
		0 & 0 & 1 & x_{34} & x_{35}  \\
		0 & 0 & 0 & 1 & x_{45}  \\
		0 & 0 & 0 & 1 & 1 \\
	\end{matrix}\right),x^A=\left(\begin{matrix}
		1 & 1 & 0 & 0 & 0 \\
		0 & 1 & 1 & 0 & 0 \\
		0 & 0 & 1 & 1 & 0 \\
		0 & 0 & 0 & 1 & 1 \\
		0 & 0 & 0 & 0 & 1
	\end{matrix}\right)
\end{equation*} \newline
Where matrix $A$ has entries
\begin{align*}
	&d_{1} = 1, d_{2} = \frac{1}{x_{12}}, d_{3} = \frac{1}{x_{12} x_{23}}, d_{4} = \frac{1}{x_{12} x_{23} x_{34}}, d_{5} = \frac{1}{x_{12} x_{23} x_{34} x_{45}},\\ 
	&a_{12} = \frac{x_{13}}{x_{12} x_{23}}, a_{13} = 1, a_{14} = 1, a_{15} = 1,\\ 
	&a_{23} = 0, a_{24} = \frac{x_{12}^{2} x_{23}^{2} x_{34} - x_{12} x_{14} x_{23}}{x_{12}^{3} x_{23}^{2} x_{34}}, a_{25} = \frac{x_{12}^{2} x_{14} x_{23}^{2} x_{35} + {\left(x_{12} x_{13} x_{14} x_{23} x_{34} + {\left(x_{12}^{3} x_{23}^{3} - x_{12}^{2} x_{13} x_{23}^{2}\right)} x_{34}^{2}\right)} x_{45}}{x_{12}^{4} x_{23}^{3} x_{34}^{2} x_{45}},\\ 
	&a_{34} = 0, a_{35} = \frac{x_{12}^{2} x_{23}^{2} x_{34} - x_{12} x_{14} x_{23}}{x_{12}^{3} x_{23}^{3} x_{34}},\\ 
	&a_{45} = -\frac{x_{35}}{x_{12} x_{23} x_{34}^{2} x_{45}}
\end{align*}

\begin{equation*}x=\left(\begin{matrix}
		1 & x_{12} & x_{13} & x_{14} & x_{15} \\
		0 & 1 & x_{23} & 0 & 0 \\
		0 & 0 & 1 & x_{34} & x_{35}  \\
		0 & 0 & 0 & 1 & x_{45}  \\
		0 & 0 & 0 & 1 & 1 \\
	\end{matrix}\right),x^A=\left(\begin{matrix}
		1 & 1 & 0 & 0 & 0 \\
		0 & 1 & 1 & 0 & 0 \\
		0 & 0 & 1 & 1 & 0 \\
		0 & 0 & 0 & 1 & 1 \\
		0 & 0 & 0 & 0 & 1
	\end{matrix}\right)
\end{equation*} \newline
Where matrix $A$ has entries
\begin{align*}
	&d_{1} = 1, d_{2} = \frac{1}{x_{12}}, d_{3} = \frac{1}{x_{12} x_{23}}, d_{4} = \frac{1}{x_{12} x_{23} x_{34}}, d_{5} = \frac{1}{x_{12} x_{23} x_{34} x_{45}},\\ 
	&a_{12} = \frac{x_{13}}{x_{12} x_{23}}, a_{13} = 1, a_{14} = 1, a_{15} = 1,\\ 
	&a_{23} = 0, a_{24} = \frac{x_{12}^{2} x_{23}^{2} x_{34} - x_{12} x_{14} x_{23}}{x_{12}^{3} x_{23}^{2} x_{34}}, a_{25} = -\frac{x_{12}^{2} x_{15} x_{23}^{2} x_{34} - x_{12}^{2} x_{14} x_{23}^{2} x_{35} - {\left(x_{12} x_{13} x_{14} x_{23} x_{34} + {\left(x_{12}^{3} x_{23}^{3} - x_{12}^{2} x_{13} x_{23}^{2}\right)} x_{34}^{2}\right)} x_{45}}{x_{12}^{4} x_{23}^{3} x_{34}^{2} x_{45}},\\ 
	&a_{34} = 0, a_{35} = \frac{x_{12}^{2} x_{23}^{2} x_{34} - x_{12} x_{14} x_{23}}{x_{12}^{3} x_{23}^{3} x_{34}},\\ 
	&a_{45} = -\frac{x_{35}}{x_{12} x_{23} x_{34}^{2} x_{45}}
\end{align*}

\begin{equation*}x=\left(\begin{matrix}
		1 & x_{12} & 0 & 0 & 0 \\
		0 & 1 & x_{23} & 0 & x_{25} \\
		0 & 0 & 1 & x_{34} & x_{35}  \\
		0 & 0 & 0 & 1 & x_{45}  \\
		0 & 0 & 0 & 1 & 1 \\
	\end{matrix}\right),x^A=\left(\begin{matrix}
		1 & 1 & 0 & 0 & 0 \\
		0 & 1 & 1 & 0 & 0 \\
		0 & 0 & 1 & 1 & 0 \\
		0 & 0 & 0 & 1 & 1 \\
		0 & 0 & 0 & 0 & 1
	\end{matrix}\right)
\end{equation*} \newline
Where matrix $A$ has entries
\begin{align*}
	&d_{1} = 1, d_{2} = \frac{1}{x_{12}}, d_{3} = \frac{1}{x_{12} x_{23}}, d_{4} = \frac{1}{x_{12} x_{23} x_{34}}, d_{5} = \frac{1}{x_{12} x_{23} x_{34} x_{45}},\\ 
	&a_{12} = 0, a_{13} = 1, a_{14} = 1, a_{15} = 1,\\ 
	&a_{23} = 0, a_{24} = \frac{1}{x_{12}}, a_{25} = \frac{1}{x_{12}},\\ 
	&a_{34} = 0, a_{35} = \frac{x_{23} x_{34} x_{45} - x_{25}}{x_{12} x_{23}^{2} x_{34} x_{45}},\\ 
	&a_{45} = -\frac{x_{35}}{x_{12} x_{23} x_{34}^{2} x_{45}}
\end{align*}

\begin{equation*}x=\left(\begin{matrix}
		1 & x_{12} & 0 & 0 & x_{15} \\
		0 & 1 & x_{23} & 0 & x_{25} \\
		0 & 0 & 1 & x_{34} & x_{35}  \\
		0 & 0 & 0 & 1 & x_{45}  \\
		0 & 0 & 0 & 1 & 1 \\
	\end{matrix}\right),x^A=\left(\begin{matrix}
		1 & 1 & 0 & 0 & 0 \\
		0 & 1 & 1 & 0 & 0 \\
		0 & 0 & 1 & 1 & 0 \\
		0 & 0 & 0 & 1 & 1 \\
		0 & 0 & 0 & 0 & 1
	\end{matrix}\right)
\end{equation*} \newline
Where matrix $A$ has entries
\begin{align*}
	&d_{1} = 1, d_{2} = \frac{1}{x_{12}}, d_{3} = \frac{1}{x_{12} x_{23}}, d_{4} = \frac{1}{x_{12} x_{23} x_{34}}, d_{5} = \frac{1}{x_{12} x_{23} x_{34} x_{45}},\\ 
	&a_{12} = 0, a_{13} = 1, a_{14} = 1, a_{15} = 1,\\ 
	&a_{23} = 0, a_{24} = \frac{1}{x_{12}}, a_{25} = \frac{x_{12} x_{23} x_{34} x_{45} - x_{15}}{x_{12}^{2} x_{23} x_{34} x_{45}},\\ 
	&a_{34} = 0, a_{35} = \frac{x_{23} x_{34} x_{45} - x_{25}}{x_{12} x_{23}^{2} x_{34} x_{45}},\\ 
	&a_{45} = -\frac{x_{35}}{x_{12} x_{23} x_{34}^{2} x_{45}}
\end{align*}

\begin{equation*}x=\left(\begin{matrix}
		1 & x_{12} & 0 & x_{14} & 0 \\
		0 & 1 & x_{23} & 0 & x_{25} \\
		0 & 0 & 1 & x_{34} & x_{35}  \\
		0 & 0 & 0 & 1 & x_{45}  \\
		0 & 0 & 0 & 1 & 1 \\
	\end{matrix}\right),x^A=\left(\begin{matrix}
		1 & 1 & 0 & 0 & 0 \\
		0 & 1 & 1 & 0 & 0 \\
		0 & 0 & 1 & 1 & 0 \\
		0 & 0 & 0 & 1 & 1 \\
		0 & 0 & 0 & 0 & 1
	\end{matrix}\right)
\end{equation*} \newline
Where matrix $A$ has entries
\begin{align*}
	&d_{1} = 1, d_{2} = \frac{1}{x_{12}}, d_{3} = \frac{1}{x_{12} x_{23}}, d_{4} = \frac{1}{x_{12} x_{23} x_{34}}, d_{5} = \frac{1}{x_{12} x_{23} x_{34} x_{45}},\\ 
	&a_{12} = 0, a_{13} = 1, a_{14} = 1, a_{15} = 1,\\ 
	&a_{23} = 0, a_{24} = \frac{x_{12} x_{23} x_{34} - x_{14}}{x_{12}^{2} x_{23} x_{34}}, a_{25} = \frac{x_{12} x_{23} x_{34}^{2} x_{45} + x_{14} x_{35}}{x_{12}^{2} x_{23} x_{34}^{2} x_{45}},\\ 
	&a_{34} = 0, a_{35} = -\frac{x_{12} x_{25} - {\left(x_{12} x_{23} x_{34} - x_{14}\right)} x_{45}}{x_{12}^{2} x_{23}^{2} x_{34} x_{45}},\\ 
	&a_{45} = -\frac{x_{35}}{x_{12} x_{23} x_{34}^{2} x_{45}}
\end{align*}

\begin{equation*}x=\left(\begin{matrix}
		1 & x_{12} & 0 & x_{14} & x_{15} \\
		0 & 1 & x_{23} & 0 & x_{25} \\
		0 & 0 & 1 & x_{34} & x_{35}  \\
		0 & 0 & 0 & 1 & x_{45}  \\
		0 & 0 & 0 & 1 & 1 \\
	\end{matrix}\right),x^A=\left(\begin{matrix}
		1 & 1 & 0 & 0 & 0 \\
		0 & 1 & 1 & 0 & 0 \\
		0 & 0 & 1 & 1 & 0 \\
		0 & 0 & 0 & 1 & 1 \\
		0 & 0 & 0 & 0 & 1
	\end{matrix}\right)
\end{equation*} \newline
Where matrix $A$ has entries
\begin{align*}
	&d_{1} = 1, d_{2} = \frac{1}{x_{12}}, d_{3} = \frac{1}{x_{12} x_{23}}, d_{4} = \frac{1}{x_{12} x_{23} x_{34}}, d_{5} = \frac{1}{x_{12} x_{23} x_{34} x_{45}},\\ 
	&a_{12} = 0, a_{13} = 1, a_{14} = 1, a_{15} = 1,\\ 
	&a_{23} = 0, a_{24} = \frac{x_{12} x_{23} x_{34} - x_{14}}{x_{12}^{2} x_{23} x_{34}}, a_{25} = \frac{x_{12} x_{23} x_{34}^{2} x_{45} - x_{15} x_{34} + x_{14} x_{35}}{x_{12}^{2} x_{23} x_{34}^{2} x_{45}},\\ 
	&a_{34} = 0, a_{35} = -\frac{x_{12} x_{25} - {\left(x_{12} x_{23} x_{34} - x_{14}\right)} x_{45}}{x_{12}^{2} x_{23}^{2} x_{34} x_{45}},\\ 
	&a_{45} = -\frac{x_{35}}{x_{12} x_{23} x_{34}^{2} x_{45}}
\end{align*}

\begin{equation*}x=\left(\begin{matrix}
		1 & x_{12} & x_{13} & 0 & 0 \\
		0 & 1 & x_{23} & 0 & x_{25} \\
		0 & 0 & 1 & x_{34} & x_{35}  \\
		0 & 0 & 0 & 1 & x_{45}  \\
		0 & 0 & 0 & 1 & 1 \\
	\end{matrix}\right),x^A=\left(\begin{matrix}
		1 & 1 & 0 & 0 & 0 \\
		0 & 1 & 1 & 0 & 0 \\
		0 & 0 & 1 & 1 & 0 \\
		0 & 0 & 0 & 1 & 1 \\
		0 & 0 & 0 & 0 & 1
	\end{matrix}\right)
\end{equation*} \newline
Where matrix $A$ has entries
\begin{align*}
	&d_{1} = 1, d_{2} = \frac{1}{x_{12}}, d_{3} = \frac{1}{x_{12} x_{23}}, d_{4} = \frac{1}{x_{12} x_{23} x_{34}}, d_{5} = \frac{1}{x_{12} x_{23} x_{34} x_{45}},\\ 
	&a_{12} = \frac{x_{13}}{x_{12} x_{23}}, a_{13} = 1, a_{14} = 1, a_{15} = 1,\\ 
	&a_{23} = 0, a_{24} = \frac{1}{x_{12}}, a_{25} = \frac{x_{12}^{2} x_{13} x_{23} x_{25} + {\left(x_{12}^{3} x_{23}^{3} - x_{12}^{2} x_{13} x_{23}^{2}\right)} x_{34} x_{45}}{x_{12}^{4} x_{23}^{3} x_{34} x_{45}},\\ 
	&a_{34} = 0, a_{35} = \frac{x_{12}^{2} x_{23}^{2} x_{34} x_{45} - x_{12}^{2} x_{23} x_{25}}{x_{12}^{3} x_{23}^{3} x_{34} x_{45}},\\ 
	&a_{45} = -\frac{x_{35}}{x_{12} x_{23} x_{34}^{2} x_{45}}
\end{align*}

\begin{equation*}x=\left(\begin{matrix}
		1 & x_{12} & x_{13} & 0 & x_{15} \\
		0 & 1 & x_{23} & 0 & x_{25} \\
		0 & 0 & 1 & x_{34} & x_{35}  \\
		0 & 0 & 0 & 1 & x_{45}  \\
		0 & 0 & 0 & 1 & 1 \\
	\end{matrix}\right),x^A=\left(\begin{matrix}
		1 & 1 & 0 & 0 & 0 \\
		0 & 1 & 1 & 0 & 0 \\
		0 & 0 & 1 & 1 & 0 \\
		0 & 0 & 0 & 1 & 1 \\
		0 & 0 & 0 & 0 & 1
	\end{matrix}\right)
\end{equation*} \newline
Where matrix $A$ has entries
\begin{align*}
	&d_{1} = 1, d_{2} = \frac{1}{x_{12}}, d_{3} = \frac{1}{x_{12} x_{23}}, d_{4} = \frac{1}{x_{12} x_{23} x_{34}}, d_{5} = \frac{1}{x_{12} x_{23} x_{34} x_{45}},\\ 
	&a_{12} = \frac{x_{13}}{x_{12} x_{23}}, a_{13} = 1, a_{14} = 1, a_{15} = 1,\\ 
	&a_{23} = 0, a_{24} = \frac{1}{x_{12}}, a_{25} = -\frac{x_{12}^{2} x_{15} x_{23}^{2} - x_{12}^{2} x_{13} x_{23} x_{25} - {\left(x_{12}^{3} x_{23}^{3} - x_{12}^{2} x_{13} x_{23}^{2}\right)} x_{34} x_{45}}{x_{12}^{4} x_{23}^{3} x_{34} x_{45}},\\ 
	&a_{34} = 0, a_{35} = \frac{x_{12}^{2} x_{23}^{2} x_{34} x_{45} - x_{12}^{2} x_{23} x_{25}}{x_{12}^{3} x_{23}^{3} x_{34} x_{45}},\\ 
	&a_{45} = -\frac{x_{35}}{x_{12} x_{23} x_{34}^{2} x_{45}}
\end{align*}

\begin{equation*}x=\left(\begin{matrix}
		1 & x_{12} & x_{13} & x_{14} & 0 \\
		0 & 1 & x_{23} & 0 & x_{25} \\
		0 & 0 & 1 & x_{34} & x_{35}  \\
		0 & 0 & 0 & 1 & x_{45}  \\
		0 & 0 & 0 & 1 & 1 \\
	\end{matrix}\right),x^A=\left(\begin{matrix}
		1 & 1 & 0 & 0 & 0 \\
		0 & 1 & 1 & 0 & 0 \\
		0 & 0 & 1 & 1 & 0 \\
		0 & 0 & 0 & 1 & 1 \\
		0 & 0 & 0 & 0 & 1
	\end{matrix}\right)
\end{equation*} \newline
Where matrix $A$ has entries
\begin{align*}
	&d_{1} = 1, d_{2} = \frac{1}{x_{12}}, d_{3} = \frac{1}{x_{12} x_{23}}, d_{4} = \frac{1}{x_{12} x_{23} x_{34}}, d_{5} = \frac{1}{x_{12} x_{23} x_{34} x_{45}},\\ 
	&a_{12} = \frac{x_{13}}{x_{12} x_{23}}, a_{13} = 1, a_{14} = 1, a_{15} = 1,\\ 
	&a_{23} = 0, a_{24} = \frac{x_{12}^{2} x_{23}^{2} x_{34} - x_{12} x_{14} x_{23}}{x_{12}^{3} x_{23}^{2} x_{34}}, \\
	&a_{25} = \frac{x_{12}^{2} x_{13} x_{23} x_{25} x_{34} + x_{12}^{2} x_{14} x_{23}^{2} x_{35} + {\left(x_{12} x_{13} x_{14} x_{23} x_{34} + {\left(x_{12}^{3} x_{23}^{3} - x_{12}^{2} x_{13} x_{23}^{2}\right)} x_{34}^{2}\right)} x_{45}}{x_{12}^{4} x_{23}^{3} x_{34}^{2} x_{45}},\\ 
	&a_{34} = 0, a_{35} = -\frac{x_{12}^{2} x_{23} x_{25} - {\left(x_{12}^{2} x_{23}^{2} x_{34} - x_{12} x_{14} x_{23}\right)} x_{45}}{x_{12}^{3} x_{23}^{3} x_{34} x_{45}},\\ 
	&a_{45} = -\frac{x_{35}}{x_{12} x_{23} x_{34}^{2} x_{45}}
\end{align*}

\begin{equation*}x=\left(\begin{matrix}
		1 & x_{12} & x_{13} & x_{14} & x_{15} \\
		0 & 1 & x_{23} & 0 & x_{25} \\
		0 & 0 & 1 & x_{34} & x_{35}  \\
		0 & 0 & 0 & 1 & x_{45}  \\
		0 & 0 & 0 & 1 & 1 \\
	\end{matrix}\right),x^A=\left(\begin{matrix}
		1 & 1 & 0 & 0 & 0 \\
		0 & 1 & 1 & 0 & 0 \\
		0 & 0 & 1 & 1 & 0 \\
		0 & 0 & 0 & 1 & 1 \\
		0 & 0 & 0 & 0 & 1
	\end{matrix}\right)
\end{equation*} \newline
Where matrix $A$ has entries
\begin{align*}
	&d_{1} = 1, d_{2} = \frac{1}{x_{12}}, d_{3} = \frac{1}{x_{12} x_{23}}, d_{4} = \frac{1}{x_{12} x_{23} x_{34}}, d_{5} = \frac{1}{x_{12} x_{23} x_{34} x_{45}},\\ 
	&a_{12} = \frac{x_{13}}{x_{12} x_{23}}, a_{13} = 1, a_{14} = 1, a_{15} = 1,\\ 
	&a_{23} = 0, a_{24} = \frac{x_{12}^{2} x_{23}^{2} x_{34} - x_{12} x_{14} x_{23}}{x_{12}^{3} x_{23}^{2} x_{34}}, \\
	&a_{25} = \frac{x_{12}^{2} x_{14} x_{23}^{2} x_{35} - {\left(x_{12}^{2} x_{15} x_{23}^{2} - x_{12}^{2} x_{13} x_{23} x_{25}\right)} x_{34} + {\left(x_{12} x_{13} x_{14} x_{23} x_{34} + {\left(x_{12}^{3} x_{23}^{3} - x_{12}^{2} x_{13} x_{23}^{2}\right)} x_{34}^{2}\right)} x_{45}}{x_{12}^{4} x_{23}^{3} x_{34}^{2} x_{45}},\\ 
	&a_{34} = 0, a_{35} = -\frac{x_{12}^{2} x_{23} x_{25} - {\left(x_{12}^{2} x_{23}^{2} x_{34} - x_{12} x_{14} x_{23}\right)} x_{45}}{x_{12}^{3} x_{23}^{3} x_{34} x_{45}},\\ 
	&a_{45} = -\frac{x_{35}}{x_{12} x_{23} x_{34}^{2} x_{45}}
\end{align*}

\begin{equation*}x=\left(\begin{matrix}
		1 & x_{12} & 0 & 0 & 0 \\
		0 & 1 & x_{23} & x_{24} & 0 \\
		0 & 0 & 1 & x_{34} & x_{35}  \\
		0 & 0 & 0 & 1 & x_{45}  \\
		0 & 0 & 0 & 1 & 1 \\
	\end{matrix}\right),x^A=\left(\begin{matrix}
		1 & 1 & 0 & 0 & 0 \\
		0 & 1 & 1 & 0 & 0 \\
		0 & 0 & 1 & 1 & 0 \\
		0 & 0 & 0 & 1 & 1 \\
		0 & 0 & 0 & 0 & 1
	\end{matrix}\right)
\end{equation*} \newline
Where matrix $A$ has entries
\begin{align*}
	&d_{1} = 1, d_{2} = \frac{1}{x_{12}}, d_{3} = \frac{1}{x_{12} x_{23}}, d_{4} = \frac{1}{x_{12} x_{23} x_{34}}, d_{5} = \frac{1}{x_{12} x_{23} x_{34} x_{45}},\\ 
	&a_{12} = 0, a_{13} = 1, a_{14} = 1, a_{15} = 1,\\ 
	&a_{23} = 0, a_{24} = \frac{1}{x_{12}}, a_{25} = \frac{1}{x_{12}},\\ 
	&a_{34} = -\frac{x_{24}}{x_{12} x_{23}^{2} x_{34}}, a_{35} = \frac{x_{23} x_{24} x_{35} + {\left(x_{23}^{2} x_{34}^{2} + x_{24}^{2}\right)} x_{45}}{x_{12} x_{23}^{3} x_{34}^{2} x_{45}},\\ 
	&a_{45} = -\frac{x_{23} x_{35} + x_{24} x_{45}}{x_{12} x_{23}^{2} x_{34}^{2} x_{45}}
\end{align*}

\begin{equation*}x=\left(\begin{matrix}
		1 & x_{12} & 0 & 0 & x_{15} \\
		0 & 1 & x_{23} & x_{24} & 0 \\
		0 & 0 & 1 & x_{34} & x_{35}  \\
		0 & 0 & 0 & 1 & x_{45}  \\
		0 & 0 & 0 & 1 & 1 \\
	\end{matrix}\right),x^A=\left(\begin{matrix}
		1 & 1 & 0 & 0 & 0 \\
		0 & 1 & 1 & 0 & 0 \\
		0 & 0 & 1 & 1 & 0 \\
		0 & 0 & 0 & 1 & 1 \\
		0 & 0 & 0 & 0 & 1
	\end{matrix}\right)
\end{equation*} \newline
Where matrix $A$ has entries
\begin{align*}
	&d_{1} = 1, d_{2} = \frac{1}{x_{12}}, d_{3} = \frac{1}{x_{12} x_{23}}, d_{4} = \frac{1}{x_{12} x_{23} x_{34}}, d_{5} = \frac{1}{x_{12} x_{23} x_{34} x_{45}},\\ 
	&a_{12} = 0, a_{13} = 1, a_{14} = 1, a_{15} = 1,\\ 
	&a_{23} = 0, a_{24} = \frac{1}{x_{12}}, a_{25} = \frac{x_{12} x_{23} x_{34} x_{45} - x_{15}}{x_{12}^{2} x_{23} x_{34} x_{45}},\\ 
	&a_{34} = -\frac{x_{24}}{x_{12} x_{23}^{2} x_{34}}, a_{35} = \frac{x_{23} x_{24} x_{35} + {\left(x_{23}^{2} x_{34}^{2} + x_{24}^{2}\right)} x_{45}}{x_{12} x_{23}^{3} x_{34}^{2} x_{45}},\\ 
	&a_{45} = -\frac{x_{23} x_{35} + x_{24} x_{45}}{x_{12} x_{23}^{2} x_{34}^{2} x_{45}}
\end{align*}

\begin{equation*}x=\left(\begin{matrix}
		1 & x_{12} & 0 & x_{14} & 0 \\
		0 & 1 & x_{23} & x_{24} & 0 \\
		0 & 0 & 1 & x_{34} & x_{35}  \\
		0 & 0 & 0 & 1 & x_{45}  \\
		0 & 0 & 0 & 1 & 1 \\
	\end{matrix}\right),x^A=\left(\begin{matrix}
		1 & 1 & 0 & 0 & 0 \\
		0 & 1 & 1 & 0 & 0 \\
		0 & 0 & 1 & 1 & 0 \\
		0 & 0 & 0 & 1 & 1 \\
		0 & 0 & 0 & 0 & 1
	\end{matrix}\right)
\end{equation*} \newline
Where matrix $A$ has entries
\begin{align*}
	&d_{1} = 1, d_{2} = \frac{1}{x_{12}}, d_{3} = \frac{1}{x_{12} x_{23}}, d_{4} = \frac{1}{x_{12} x_{23} x_{34}}, d_{5} = \frac{1}{x_{12} x_{23} x_{34} x_{45}},\\ 
	&a_{12} = 0, a_{13} = 1, a_{14} = 1, a_{15} = 1,\\ 
	&a_{23} = 0, a_{24} = \frac{x_{12} x_{23} x_{34} - x_{14}}{x_{12}^{2} x_{23} x_{34}}, a_{25} = \frac{x_{14} x_{23} x_{35} + {\left(x_{12} x_{23}^{2} x_{34}^{2} + x_{14} x_{24}\right)} x_{45}}{x_{12}^{2} x_{23}^{2} x_{34}^{2} x_{45}},\\ 
	&a_{34} = -\frac{x_{24}}{x_{12} x_{23}^{2} x_{34}}, a_{35} = \frac{x_{12} x_{23} x_{24} x_{35} + {\left(x_{12} x_{23}^{2} x_{34}^{2} + x_{12} x_{24}^{2} - x_{14} x_{23} x_{34}\right)} x_{45}}{x_{12}^{2} x_{23}^{3} x_{34}^{2} x_{45}},\\ 
	&a_{45} = -\frac{x_{23} x_{35} + x_{24} x_{45}}{x_{12} x_{23}^{2} x_{34}^{2} x_{45}}
\end{align*}

\begin{equation*}x=\left(\begin{matrix}
		1 & x_{12} & 0 & x_{14} & x_{15} \\
		0 & 1 & x_{23} & x_{24} & 0 \\
		0 & 0 & 1 & x_{34} & x_{35}  \\
		0 & 0 & 0 & 1 & x_{45}  \\
		0 & 0 & 0 & 1 & 1 \\
	\end{matrix}\right),x^A=\left(\begin{matrix}
		1 & 1 & 0 & 0 & 0 \\
		0 & 1 & 1 & 0 & 0 \\
		0 & 0 & 1 & 1 & 0 \\
		0 & 0 & 0 & 1 & 1 \\
		0 & 0 & 0 & 0 & 1
	\end{matrix}\right)
\end{equation*} \newline
Where matrix $A$ has entries
\begin{align*}
	&d_{1} = 1, d_{2} = \frac{1}{x_{12}}, d_{3} = \frac{1}{x_{12} x_{23}}, d_{4} = \frac{1}{x_{12} x_{23} x_{34}}, d_{5} = \frac{1}{x_{12} x_{23} x_{34} x_{45}},\\ 
	&a_{12} = 0, a_{13} = 1, a_{14} = 1, a_{15} = 1,\\ 
	&a_{23} = 0, a_{24} = \frac{x_{12} x_{23} x_{34} - x_{14}}{x_{12}^{2} x_{23} x_{34}}, a_{25} = -\frac{x_{15} x_{23} x_{34} - x_{14} x_{23} x_{35} - {\left(x_{12} x_{23}^{2} x_{34}^{2} + x_{14} x_{24}\right)} x_{45}}{x_{12}^{2} x_{23}^{2} x_{34}^{2} x_{45}},\\ 
	&a_{34} = -\frac{x_{24}}{x_{12} x_{23}^{2} x_{34}}, a_{35} = \frac{x_{12} x_{23} x_{24} x_{35} + {\left(x_{12} x_{23}^{2} x_{34}^{2} + x_{12} x_{24}^{2} - x_{14} x_{23} x_{34}\right)} x_{45}}{x_{12}^{2} x_{23}^{3} x_{34}^{2} x_{45}},\\ 
	&a_{45} = -\frac{x_{23} x_{35} + x_{24} x_{45}}{x_{12} x_{23}^{2} x_{34}^{2} x_{45}}
\end{align*}

\begin{equation*}x=\left(\begin{matrix}
		1 & x_{12} & x_{13} & 0 & 0 \\
		0 & 1 & x_{23} & x_{24} & 0 \\
		0 & 0 & 1 & x_{34} & x_{35}  \\
		0 & 0 & 0 & 1 & x_{45}  \\
		0 & 0 & 0 & 1 & 1 \\
	\end{matrix}\right),x^A=\left(\begin{matrix}
		1 & 1 & 0 & 0 & 0 \\
		0 & 1 & 1 & 0 & 0 \\
		0 & 0 & 1 & 1 & 0 \\
		0 & 0 & 0 & 1 & 1 \\
		0 & 0 & 0 & 0 & 1
	\end{matrix}\right)
\end{equation*} \newline
Where matrix $A$ has entries
\begin{align*}
	&d_{1} = 1, d_{2} = \frac{1}{x_{12}}, d_{3} = \frac{1}{x_{12} x_{23}}, d_{4} = \frac{1}{x_{12} x_{23} x_{34}}, d_{5} = \frac{1}{x_{12} x_{23} x_{34} x_{45}},\\ 
	&a_{12} = \frac{x_{13}}{x_{12} x_{23}}, a_{13} = 1, a_{14} = 1, a_{15} = 1,\\ 
	&a_{23} = 0, a_{24} = \frac{x_{12}^{2} x_{23}^{2} x_{34} + x_{12} x_{13} x_{24}}{x_{12}^{3} x_{23}^{2} x_{34}}, a_{25} = -\frac{x_{12}^{2} x_{13} x_{23} x_{24} x_{35} + {\left(x_{12}^{2} x_{13} x_{24}^{2} + x_{12} x_{13}^{2} x_{24} x_{34} - {\left(x_{12}^{3} x_{23}^{3} - x_{12}^{2} x_{13} x_{23}^{2}\right)} x_{34}^{2}\right)} x_{45}}{x_{12}^{4} x_{23}^{3} x_{34}^{2} x_{45}},\\ 
	&a_{34} = -\frac{x_{24}}{x_{12} x_{23}^{2} x_{34}}, a_{35} = \frac{x_{12}^{2} x_{23} x_{24} x_{35} + {\left(x_{12}^{2} x_{23}^{2} x_{34}^{2} + x_{12}^{2} x_{24}^{2} + x_{12} x_{13} x_{24} x_{34}\right)} x_{45}}{x_{12}^{3} x_{23}^{3} x_{34}^{2} x_{45}},\\ 
	&a_{45} = -\frac{x_{23} x_{35} + x_{24} x_{45}}{x_{12} x_{23}^{2} x_{34}^{2} x_{45}}
\end{align*}

\begin{equation*}x=\left(\begin{matrix}
		1 & x_{12} & x_{13} & 0 & x_{15} \\
		0 & 1 & x_{23} & x_{24} & 0 \\
		0 & 0 & 1 & x_{34} & x_{35}  \\
		0 & 0 & 0 & 1 & x_{45}  \\
		0 & 0 & 0 & 1 & 1 \\
	\end{matrix}\right),x^A=\left(\begin{matrix}
		1 & 1 & 0 & 0 & 0 \\
		0 & 1 & 1 & 0 & 0 \\
		0 & 0 & 1 & 1 & 0 \\
		0 & 0 & 0 & 1 & 1 \\
		0 & 0 & 0 & 0 & 1
	\end{matrix}\right)
\end{equation*} \newline
Where matrix $A$ has entries
\begin{align*}
	&d_{1} = 1, d_{2} = \frac{1}{x_{12}}, d_{3} = \frac{1}{x_{12} x_{23}}, d_{4} = \frac{1}{x_{12} x_{23} x_{34}}, d_{5} = \frac{1}{x_{12} x_{23} x_{34} x_{45}},\\ 
	&a_{12} = \frac{x_{13}}{x_{12} x_{23}}, a_{13} = 1, a_{14} = 1, a_{15} = 1,\\ 
	&a_{23} = 0, a_{24} = \frac{x_{12}^{2} x_{23}^{2} x_{34} + x_{12} x_{13} x_{24}}{x_{12}^{3} x_{23}^{2} x_{34}}, \\
	&a_{25} = -\frac{x_{12}^{2} x_{15} x_{23}^{2} x_{34} + x_{12}^{2} x_{13} x_{23} x_{24} x_{35} + {\left(x_{12}^{2} x_{13} x_{24}^{2} + x_{12} x_{13}^{2} x_{24} x_{34} - {\left(x_{12}^{3} x_{23}^{3} - x_{12}^{2} x_{13} x_{23}^{2}\right)} x_{34}^{2}\right)} x_{45}}{x_{12}^{4} x_{23}^{3} x_{34}^{2} x_{45}},\\ 
	&a_{34} = -\frac{x_{24}}{x_{12} x_{23}^{2} x_{34}}, a_{35} = \frac{x_{12}^{2} x_{23} x_{24} x_{35} + {\left(x_{12}^{2} x_{23}^{2} x_{34}^{2} + x_{12}^{2} x_{24}^{2} + x_{12} x_{13} x_{24} x_{34}\right)} x_{45}}{x_{12}^{3} x_{23}^{3} x_{34}^{2} x_{45}},\\ 
	&a_{45} = -\frac{x_{23} x_{35} + x_{24} x_{45}}{x_{12} x_{23}^{2} x_{34}^{2} x_{45}}
\end{align*}

\begin{equation*}x=\left(\begin{matrix}
		1 & x_{12} & x_{13} & x_{14} & 0 \\
		0 & 1 & x_{23} & x_{24} & 0 \\
		0 & 0 & 1 & x_{34} & x_{35}  \\
		0 & 0 & 0 & 1 & x_{45}  \\
		0 & 0 & 0 & 1 & 1 \\
	\end{matrix}\right),x^A=\left(\begin{matrix}
		1 & 1 & 0 & 0 & 0 \\
		0 & 1 & 1 & 0 & 0 \\
		0 & 0 & 1 & 1 & 0 \\
		0 & 0 & 0 & 1 & 1 \\
		0 & 0 & 0 & 0 & 1
	\end{matrix}\right)
\end{equation*} \newline
Where matrix $A$ has entries
\begin{align*}
	&d_{1} = 1, d_{2} = \frac{1}{x_{12}}, d_{3} = \frac{1}{x_{12} x_{23}}, d_{4} = \frac{1}{x_{12} x_{23} x_{34}}, d_{5} = \frac{1}{x_{12} x_{23} x_{34} x_{45}},\\ 
	&a_{12} = \frac{x_{13}}{x_{12} x_{23}}, a_{13} = 1, a_{14} = 1, a_{15} = 1,\\ 
	&a_{23} = 0, a_{24} = \frac{x_{12}^{2} x_{23}^{2} x_{34} - x_{12} x_{14} x_{23} + x_{12} x_{13} x_{24}}{x_{12}^{3} x_{23}^{2} x_{34}}, \\
	&a_{25} = \frac{\splitfrac{{\left(x_{12}^{2} x_{14} x_{23}^{2} - x_{12}^{2} x_{13} x_{23} x_{24}\right)} x_{35} +}
		{+ {\left(x_{12}^{2} x_{14} x_{23} x_{24} - x_{12}^{2} x_{13} x_{24}^{2} + {\left(x_{12}^{3} x_{23}^{3} - x_{12}^{2} x_{13} x_{23}^{2}\right)} x_{34}^{2} + {\left(x_{12} x_{13} x_{14} x_{23} - x_{12} x_{13}^{2} x_{24}\right)} x_{34}\right)} x_{45}}}{x_{12}^{4} x_{23}^{3} x_{34}^{2} x_{45}},\\ 
	&a_{34} = -\frac{x_{24}}{x_{12} x_{23}^{2} x_{34}}, a_{35} = \frac{x_{12}^{2} x_{23} x_{24} x_{35} + {\left(x_{12}^{2} x_{23}^{2} x_{34}^{2} + x_{12}^{2} x_{24}^{2} - {\left(x_{12} x_{14} x_{23} - x_{12} x_{13} x_{24}\right)} x_{34}\right)} x_{45}}{x_{12}^{3} x_{23}^{3} x_{34}^{2} x_{45}},\\ 
	&a_{45} = -\frac{x_{23} x_{35} + x_{24} x_{45}}{x_{12} x_{23}^{2} x_{34}^{2} x_{45}}
\end{align*}

\begin{equation*}x=\left(\begin{matrix}
		1 & x_{12} & x_{13} & x_{14} & x_{15} \\
		0 & 1 & x_{23} & x_{24} & 0 \\
		0 & 0 & 1 & x_{34} & x_{35}  \\
		0 & 0 & 0 & 1 & x_{45}  \\
		0 & 0 & 0 & 1 & 1 \\
	\end{matrix}\right),x^A=\left(\begin{matrix}
		1 & 1 & 0 & 0 & 0 \\
		0 & 1 & 1 & 0 & 0 \\
		0 & 0 & 1 & 1 & 0 \\
		0 & 0 & 0 & 1 & 1 \\
		0 & 0 & 0 & 0 & 1
	\end{matrix}\right)
\end{equation*} \newline
Where matrix $A$ has entries
\begin{align*}
	&d_{1} = 1, d_{2} = \frac{1}{x_{12}}, d_{3} = \frac{1}{x_{12} x_{23}}, d_{4} = \frac{1}{x_{12} x_{23} x_{34}}, d_{5} = \frac{1}{x_{12} x_{23} x_{34} x_{45}},\\ 
	&a_{12} = \frac{x_{13}}{x_{12} x_{23}}, a_{13} = 1, a_{14} = 1, a_{15} = 1,\\ 
	&a_{23} = 0, a_{24} = \frac{x_{12}^{2} x_{23}^{2} x_{34} - x_{12} x_{14} x_{23} + x_{12} x_{13} x_{24}}{x_{12}^{3} x_{23}^{2} x_{34}}, \\
	&a_{25} = -\frac{\splitfrac{x_{12}^{2} x_{15} x_{23}^{2} x_{34} - {\left(x_{12}^{2} x_{14} x_{23}^{2} - x_{12}^{2} x_{13} x_{23} x_{24}\right)} x_{35} -}
		{- {\left(x_{12}^{2} x_{14} x_{23} x_{24} - x_{12}^{2} x_{13} x_{24}^{2} + {\left(x_{12}^{3} x_{23}^{3} - x_{12}^{2} x_{13} x_{23}^{2}\right)} x_{34}^{2} + {\left(x_{12} x_{13} x_{14} x_{23} - x_{12} x_{13}^{2} x_{24}\right)} x_{34}\right)} x_{45}}}{x_{12}^{4} x_{23}^{3} x_{34}^{2} x_{45}},\\ 
	&a_{34} = -\frac{x_{24}}{x_{12} x_{23}^{2} x_{34}}, a_{35} = \frac{x_{12}^{2} x_{23} x_{24} x_{35} + {\left(x_{12}^{2} x_{23}^{2} x_{34}^{2} + x_{12}^{2} x_{24}^{2} - {\left(x_{12} x_{14} x_{23} - x_{12} x_{13} x_{24}\right)} x_{34}\right)} x_{45}}{x_{12}^{3} x_{23}^{3} x_{34}^{2} x_{45}},\\ 
	&a_{45} = -\frac{x_{23} x_{35} + x_{24} x_{45}}{x_{12} x_{23}^{2} x_{34}^{2} x_{45}}
\end{align*}

\begin{equation*}x=\left(\begin{matrix}
		1 & x_{12} & 0 & 0 & 0 \\
		0 & 1 & x_{23} & x_{24} & x_{25} \\
		0 & 0 & 1 & x_{34} & x_{35}  \\
		0 & 0 & 0 & 1 & x_{45}  \\
		0 & 0 & 0 & 1 & 1 \\
	\end{matrix}\right),x^A=\left(\begin{matrix}
		1 & 1 & 0 & 0 & 0 \\
		0 & 1 & 1 & 0 & 0 \\
		0 & 0 & 1 & 1 & 0 \\
		0 & 0 & 0 & 1 & 1 \\
		0 & 0 & 0 & 0 & 1
	\end{matrix}\right)
\end{equation*} \newline
Where matrix $A$ has entries
\begin{align*}
	&d_{1} = 1, d_{2} = \frac{1}{x_{12}}, d_{3} = \frac{1}{x_{12} x_{23}}, d_{4} = \frac{1}{x_{12} x_{23} x_{34}}, d_{5} = \frac{1}{x_{12} x_{23} x_{34} x_{45}},\\ 
	&a_{12} = 0, a_{13} = 1, a_{14} = 1, a_{15} = 1,\\ 
	&a_{23} = 0, a_{24} = \frac{1}{x_{12}}, a_{25} = \frac{1}{x_{12}},\\ 
	&a_{34} = -\frac{x_{24}}{x_{12} x_{23}^{2} x_{34}}, a_{35} = -\frac{x_{23} x_{25} x_{34} - x_{23} x_{24} x_{35} - {\left(x_{23}^{2} x_{34}^{2} + x_{24}^{2}\right)} x_{45}}{x_{12} x_{23}^{3} x_{34}^{2} x_{45}},\\ 
	&a_{45} = -\frac{x_{23} x_{35} + x_{24} x_{45}}{x_{12} x_{23}^{2} x_{34}^{2} x_{45}}
\end{align*}

\begin{equation*}x=\left(\begin{matrix}
		1 & x_{12} & 0 & 0 & x_{15} \\
		0 & 1 & x_{23} & x_{24} & x_{25} \\
		0 & 0 & 1 & x_{34} & x_{35}  \\
		0 & 0 & 0 & 1 & x_{45}  \\
		0 & 0 & 0 & 1 & 1 \\
	\end{matrix}\right),x^A=\left(\begin{matrix}
		1 & 1 & 0 & 0 & 0 \\
		0 & 1 & 1 & 0 & 0 \\
		0 & 0 & 1 & 1 & 0 \\
		0 & 0 & 0 & 1 & 1 \\
		0 & 0 & 0 & 0 & 1
	\end{matrix}\right)
\end{equation*} \newline
Where matrix $A$ has entries
\begin{align*}
	&d_{1} = 1, d_{2} = \frac{1}{x_{12}}, d_{3} = \frac{1}{x_{12} x_{23}}, d_{4} = \frac{1}{x_{12} x_{23} x_{34}}, d_{5} = \frac{1}{x_{12} x_{23} x_{34} x_{45}},\\ 
	&a_{12} = 0, a_{13} = 1, a_{14} = 1, a_{15} = 1,\\ 
	&a_{23} = 0, a_{24} = \frac{1}{x_{12}}, a_{25} = \frac{x_{12} x_{23} x_{34} x_{45} - x_{15}}{x_{12}^{2} x_{23} x_{34} x_{45}},\\ 
	&a_{34} = -\frac{x_{24}}{x_{12} x_{23}^{2} x_{34}}, a_{35} = -\frac{x_{23} x_{25} x_{34} - x_{23} x_{24} x_{35} - {\left(x_{23}^{2} x_{34}^{2} + x_{24}^{2}\right)} x_{45}}{x_{12} x_{23}^{3} x_{34}^{2} x_{45}},\\ 
	&a_{45} = -\frac{x_{23} x_{35} + x_{24} x_{45}}{x_{12} x_{23}^{2} x_{34}^{2} x_{45}}
\end{align*}

\begin{equation*}x=\left(\begin{matrix}
		1 & x_{12} & 0 & x_{14} & 0 \\
		0 & 1 & x_{23} & x_{24} & x_{25} \\
		0 & 0 & 1 & x_{34} & x_{35}  \\
		0 & 0 & 0 & 1 & x_{45}  \\
		0 & 0 & 0 & 1 & 1 \\
	\end{matrix}\right),x^A=\left(\begin{matrix}
		1 & 1 & 0 & 0 & 0 \\
		0 & 1 & 1 & 0 & 0 \\
		0 & 0 & 1 & 1 & 0 \\
		0 & 0 & 0 & 1 & 1 \\
		0 & 0 & 0 & 0 & 1
	\end{matrix}\right)
\end{equation*} \newline
Where matrix $A$ has entries
\begin{align*}
	&d_{1} = 1, d_{2} = \frac{1}{x_{12}}, d_{3} = \frac{1}{x_{12} x_{23}}, d_{4} = \frac{1}{x_{12} x_{23} x_{34}}, d_{5} = \frac{1}{x_{12} x_{23} x_{34} x_{45}},\\ 
	&a_{12} = 0, a_{13} = 1, a_{14} = 1, a_{15} = 1,\\ 
	&a_{23} = 0, a_{24} = \frac{x_{12} x_{23} x_{34} - x_{14}}{x_{12}^{2} x_{23} x_{34}}, a_{25} = \frac{x_{14} x_{23} x_{35} + {\left(x_{12} x_{23}^{2} x_{34}^{2} + x_{14} x_{24}\right)} x_{45}}{x_{12}^{2} x_{23}^{2} x_{34}^{2} x_{45}},\\ 
	&a_{34} = -\frac{x_{24}}{x_{12} x_{23}^{2} x_{34}}, a_{35} = -\frac{x_{12} x_{23} x_{25} x_{34} - x_{12} x_{23} x_{24} x_{35} - {\left(x_{12} x_{23}^{2} x_{34}^{2} + x_{12} x_{24}^{2} - x_{14} x_{23} x_{34}\right)} x_{45}}{x_{12}^{2} x_{23}^{3} x_{34}^{2} x_{45}},\\ 
	&a_{45} = -\frac{x_{23} x_{35} + x_{24} x_{45}}{x_{12} x_{23}^{2} x_{34}^{2} x_{45}}
\end{align*}

\begin{equation*}x=\left(\begin{matrix}
		1 & x_{12} & 0 & x_{14} & x_{15} \\
		0 & 1 & x_{23} & x_{24} & x_{25} \\
		0 & 0 & 1 & x_{34} & x_{35}  \\
		0 & 0 & 0 & 1 & x_{45}  \\
		0 & 0 & 0 & 1 & 1 \\
	\end{matrix}\right),x^A=\left(\begin{matrix}
		1 & 1 & 0 & 0 & 0 \\
		0 & 1 & 1 & 0 & 0 \\
		0 & 0 & 1 & 1 & 0 \\
		0 & 0 & 0 & 1 & 1 \\
		0 & 0 & 0 & 0 & 1
	\end{matrix}\right)
\end{equation*} \newline
Where matrix $A$ has entries
\begin{align*}
	&d_{1} = 1, d_{2} = \frac{1}{x_{12}}, d_{3} = \frac{1}{x_{12} x_{23}}, d_{4} = \frac{1}{x_{12} x_{23} x_{34}}, d_{5} = \frac{1}{x_{12} x_{23} x_{34} x_{45}},\\ 
	&a_{12} = 0, a_{13} = 1, a_{14} = 1, a_{15} = 1,\\ 
	&a_{23} = 0, a_{24} = \frac{x_{12} x_{23} x_{34} - x_{14}}{x_{12}^{2} x_{23} x_{34}}, a_{25} = -\frac{x_{15} x_{23} x_{34} - x_{14} x_{23} x_{35} - {\left(x_{12} x_{23}^{2} x_{34}^{2} + x_{14} x_{24}\right)} x_{45}}{x_{12}^{2} x_{23}^{2} x_{34}^{2} x_{45}},\\ 
	&a_{34} = -\frac{x_{24}}{x_{12} x_{23}^{2} x_{34}}, a_{35} = -\frac{x_{12} x_{23} x_{25} x_{34} - x_{12} x_{23} x_{24} x_{35} - {\left(x_{12} x_{23}^{2} x_{34}^{2} + x_{12} x_{24}^{2} - x_{14} x_{23} x_{34}\right)} x_{45}}{x_{12}^{2} x_{23}^{3} x_{34}^{2} x_{45}},\\ 
	&a_{45} = -\frac{x_{23} x_{35} + x_{24} x_{45}}{x_{12} x_{23}^{2} x_{34}^{2} x_{45}}
\end{align*}

\begin{equation*}x=\left(\begin{matrix}
		1 & x_{12} & x_{13} & 0 & 0 \\
		0 & 1 & x_{23} & x_{24} & x_{25} \\
		0 & 0 & 1 & x_{34} & x_{35}  \\
		0 & 0 & 0 & 1 & x_{45}  \\
		0 & 0 & 0 & 1 & 1 \\
	\end{matrix}\right),x^A=\left(\begin{matrix}
		1 & 1 & 0 & 0 & 0 \\
		0 & 1 & 1 & 0 & 0 \\
		0 & 0 & 1 & 1 & 0 \\
		0 & 0 & 0 & 1 & 1 \\
		0 & 0 & 0 & 0 & 1
	\end{matrix}\right)
\end{equation*} \newline
Where matrix $A$ has entries
\begin{align*}
	&d_{1} = 1, d_{2} = \frac{1}{x_{12}}, d_{3} = \frac{1}{x_{12} x_{23}}, d_{4} = \frac{1}{x_{12} x_{23} x_{34}}, d_{5} = \frac{1}{x_{12} x_{23} x_{34} x_{45}},\\ 
	&a_{12} = \frac{x_{13}}{x_{12} x_{23}}, a_{13} = 1, a_{14} = 1, a_{15} = 1,\\ 
	&a_{23} = 0, a_{24} = \frac{x_{12}^{2} x_{23}^{2} x_{34} + x_{12} x_{13} x_{24}}{x_{12}^{3} x_{23}^{2} x_{34}}, \\
	&a_{25} = \frac{x_{12}^{2} x_{13} x_{23} x_{25} x_{34} - x_{12}^{2} x_{13} x_{23} x_{24} x_{35} - {\left(x_{12}^{2} x_{13} x_{24}^{2} + x_{12} x_{13}^{2} x_{24} x_{34} - {\left(x_{12}^{3} x_{23}^{3} - x_{12}^{2} x_{13} x_{23}^{2}\right)} x_{34}^{2}\right)} x_{45}}{x_{12}^{4} x_{23}^{3} x_{34}^{2} x_{45}},\\ 
	&a_{34} = -\frac{x_{24}}{x_{12} x_{23}^{2} x_{34}}, a_{35} = -\frac{x_{12}^{2} x_{23} x_{25} x_{34} - x_{12}^{2} x_{23} x_{24} x_{35} - {\left(x_{12}^{2} x_{23}^{2} x_{34}^{2} + x_{12}^{2} x_{24}^{2} + x_{12} x_{13} x_{24} x_{34}\right)} x_{45}}{x_{12}^{3} x_{23}^{3} x_{34}^{2} x_{45}},\\ 
	&a_{45} = -\frac{x_{23} x_{35} + x_{24} x_{45}}{x_{12} x_{23}^{2} x_{34}^{2} x_{45}}
\end{align*}

\begin{equation*}x=\left(\begin{matrix}
		1 & x_{12} & x_{13} & 0 & x_{15} \\
		0 & 1 & x_{23} & x_{24} & x_{25} \\
		0 & 0 & 1 & x_{34} & x_{35}  \\
		0 & 0 & 0 & 1 & x_{45}  \\
		0 & 0 & 0 & 1 & 1 \\
	\end{matrix}\right),x^A=\left(\begin{matrix}
		1 & 1 & 0 & 0 & 0 \\
		0 & 1 & 1 & 0 & 0 \\
		0 & 0 & 1 & 1 & 0 \\
		0 & 0 & 0 & 1 & 1 \\
		0 & 0 & 0 & 0 & 1
	\end{matrix}\right)
\end{equation*} \newline
Where matrix $A$ has entries
\begin{align*}
	&d_{1} = 1, d_{2} = \frac{1}{x_{12}}, d_{3} = \frac{1}{x_{12} x_{23}}, d_{4} = \frac{1}{x_{12} x_{23} x_{34}}, d_{5} = \frac{1}{x_{12} x_{23} x_{34} x_{45}},\\ 
	&a_{12} = \frac{x_{13}}{x_{12} x_{23}}, a_{13} = 1, a_{14} = 1, a_{15} = 1,\\ 
	&a_{23} = 0, a_{24} = \frac{x_{12}^{2} x_{23}^{2} x_{34} + x_{12} x_{13} x_{24}}{x_{12}^{3} x_{23}^{2} x_{34}}, \\
	&a_{25} = -\frac{x_{12}^{2} x_{13} x_{23} x_{24} x_{35} + {\left(x_{12}^{2} x_{15} x_{23}^{2} - x_{12}^{2} x_{13} x_{23} x_{25}\right)} x_{34} + {\left(x_{12}^{2} x_{13} x_{24}^{2} + x_{12} x_{13}^{2} x_{24} x_{34} - {\left(x_{12}^{3} x_{23}^{3} - x_{12}^{2} x_{13} x_{23}^{2}\right)} x_{34}^{2}\right)} x_{45}}{x_{12}^{4} x_{23}^{3} x_{34}^{2} x_{45}},\\ 
	&a_{34} = -\frac{x_{24}}{x_{12} x_{23}^{2} x_{34}}, a_{35} = -\frac{x_{12}^{2} x_{23} x_{25} x_{34} - x_{12}^{2} x_{23} x_{24} x_{35} - {\left(x_{12}^{2} x_{23}^{2} x_{34}^{2} + x_{12}^{2} x_{24}^{2} + x_{12} x_{13} x_{24} x_{34}\right)} x_{45}}{x_{12}^{3} x_{23}^{3} x_{34}^{2} x_{45}},\\ 
	&a_{45} = -\frac{x_{23} x_{35} + x_{24} x_{45}}{x_{12} x_{23}^{2} x_{34}^{2} x_{45}}
\end{align*}

\begin{equation*}x=\left(\begin{matrix}
		1 & x_{12} & x_{13} & x_{14} & 0 \\
		0 & 1 & x_{23} & x_{24} & x_{25} \\
		0 & 0 & 1 & x_{34} & x_{35}  \\
		0 & 0 & 0 & 1 & x_{45}  \\
		0 & 0 & 0 & 1 & 1 \\
	\end{matrix}\right),x^A=\left(\begin{matrix}
		1 & 1 & 0 & 0 & 0 \\
		0 & 1 & 1 & 0 & 0 \\
		0 & 0 & 1 & 1 & 0 \\
		0 & 0 & 0 & 1 & 1 \\
		0 & 0 & 0 & 0 & 1
	\end{matrix}\right)
\end{equation*} \newline
Where matrix $A$ has entries
\begin{align*}
	&d_{1} = 1, d_{2} = \frac{1}{x_{12}}, d_{3} = \frac{1}{x_{12} x_{23}}, d_{4} = \frac{1}{x_{12} x_{23} x_{34}}, d_{5} = \frac{1}{x_{12} x_{23} x_{34} x_{45}},\\ 
	&a_{12} = \frac{x_{13}}{x_{12} x_{23}}, a_{13} = 1, a_{14} = 1, a_{15} = 1,\\ 
	&a_{23} = 0, a_{24} = \frac{x_{12}^{2} x_{23}^{2} x_{34} - x_{12} x_{14} x_{23} + x_{12} x_{13} x_{24}}{x_{12}^{3} x_{23}^{2} x_{34}}, \\
	&a_{25} = \frac{\splitfrac{x_{12}^{2} x_{13} x_{23} x_{25} x_{34} + {\left(x_{12}^{2} x_{14} x_{23}^{2} - x_{12}^{2} x_{13} x_{23} x_{24}\right)} x_{35} +}
		{+ {\left(x_{12}^{2} x_{14} x_{23} x_{24} - x_{12}^{2} x_{13} x_{24}^{2} + {\left(x_{12}^{3} x_{23}^{3} - x_{12}^{2} x_{13} x_{23}^{2}\right)} x_{34}^{2} + {\left(x_{12} x_{13} x_{14} x_{23} - x_{12} x_{13}^{2} x_{24}\right)} x_{34}\right)} x_{45}}}{x_{12}^{4} x_{23}^{3} x_{34}^{2} x_{45}},\\ 
	&a_{34} = -\frac{x_{24}}{x_{12} x_{23}^{2} x_{34}}, a_{35} = -\frac{x_{12}^{2} x_{23} x_{25} x_{34} - x_{12}^{2} x_{23} x_{24} x_{35} - {\left(x_{12}^{2} x_{23}^{2} x_{34}^{2} + x_{12}^{2} x_{24}^{2} - {\left(x_{12} x_{14} x_{23} - x_{12} x_{13} x_{24}\right)} x_{34}\right)} x_{45}}{x_{12}^{3} x_{23}^{3} x_{34}^{2} x_{45}},\\ 
	&a_{45} = -\frac{x_{23} x_{35} + x_{24} x_{45}}{x_{12} x_{23}^{2} x_{34}^{2} x_{45}}
\end{align*}

\begin{equation*}x=\left(\begin{matrix}
		1 & x_{12} & x_{13} & x_{14} & x_{15} \\
		0 & 1 & x_{23} & x_{24} & x_{25} \\
		0 & 0 & 1 & x_{34} & x_{35}  \\
		0 & 0 & 0 & 1 & x_{45}  \\
		0 & 0 & 0 & 1 & 1 \\
	\end{matrix}\right),x^A=\left(\begin{matrix}
		1 & 1 & 0 & 0 & 0 \\
		0 & 1 & 1 & 0 & 0 \\
		0 & 0 & 1 & 1 & 0 \\
		0 & 0 & 0 & 1 & 1 \\
		0 & 0 & 0 & 0 & 1
	\end{matrix}\right)
\end{equation*} \newline
Where matrix $A$ has entries
\begin{align*}
	&d_{1} = 1, d_{2} = \frac{1}{x_{12}}, d_{3} = \frac{1}{x_{12} x_{23}}, d_{4} = \frac{1}{x_{12} x_{23} x_{34}}, d_{5} = \frac{1}{x_{12} x_{23} x_{34} x_{45}}, \\ 
	&a_{12} = 1, a_{13} = 1, a_{14} = 1, a_{15} = 1, \\ 
	&a_{23} = \frac{x_{12} x_{23} - x_{13}}{x_{12}^{2} x_{23}}, a_{24} = -\frac{x_{12} x_{14} x_{23} - x_{12} x_{13} x_{24} - {\left(x_{12}^{2} x_{23}^{2} - x_{12} x_{13} x_{23} + x_{13}^{2}\right)} x_{34}}{x_{12}^{3} x_{23}^{2} x_{34}},\\
	&a_{25} = -\frac{\splitfrac{
			{x_{12}^{2} x_{15} x_{23}^{2} - x_{12}^{2} x_{13} x_{23} x_{25})} x_{34} - {(x_{12}^{2} x_{14} x_{23}^{2} - x_{12}^{2} x_{13} x_{23} x_{24})} x_{35} - (x_{12}^{2} x_{14} x_{23} x_{24} - x_{12}^{2} x_{13} x_{24}^{2} + (x_{12}^{3} x_{23}^{3} - x_{12}^{2} x_{13} x_{23}^{2} +
		}{
			+x_{12} x_{13}^{2} x_{23} - x_{13}^{3}) x_{34}^{2} - {(x_{12}^{2} x_{14} x_{23}^{2} - 2 \, x_{12} x_{13} x_{14} x_{23} - {(x_{12}^{2} x_{13} x_{23} - 2 \, x_{12} x_{13}^{2})} x_{24})} x_{34}) x_{45}
	}}{
		x_{12}^{4} x_{23}^{3} x_{34}^{2} x_{45}
	}, \\ 
	&a_{34} = -\frac{x_{12} x_{24} - {\left(x_{12} x_{23} - x_{13}\right)} x_{34}}{x_{12}^{2} x_{23}^{2} x_{34}}, \\
	&a_{35} = -\frac{x_{12}^{2} x_{23} x_{25} x_{34} - x_{12}^{2} x_{23} x_{24} x_{35} - {\left(x_{12}^{2} x_{24}^{2} + {\left(x_{12}^{2} x_{23}^{2} - x_{12} x_{13} x_{23} + x_{13}^{2}\right)} x_{34}^{2} - {\left(x_{12} x_{14} x_{23} + {\left(x_{12}^{2} x_{23} - 2 \, x_{12} x_{13}\right)} x_{24}\right)} x_{34}\right)} x_{45}}{x_{12}^{3} x_{23}^{3} x_{34}^{2} x_{45}}, \\ 
	&a_{45} = -\frac{x_{12} x_{23} x_{35} + {\left(x_{12} x_{24} - {\left(x_{12} x_{23} - x_{13}\right)} x_{34}\right)} x_{45}}{x_{12}^{2} x_{23}^{2} x_{34}^{2} x_{45}}
\end{align*}

		\section{Subcases of $Y_{11}$}

\begin{equation*}x=\left(
\right)
\end{equation*} \newline
Where matrix $A$ has entries
\begin{align*}
	&d_{1} = 1, d_{2} = 1, d_{3} = \frac{1}{x_{13}}, d_{4} = \frac{x_{45}}{x_{13} x_{35} + x_{14} x_{45}}, d_{5} = \frac{1}{x_{13} x_{35} + x_{14} x_{45}},\\ 
	&a_{12} = 1, a_{13} = 1, a_{14} = 1, a_{15} = 1,\\ 
	&a_{23} = 1, a_{24} = 0, a_{25} = 1,\\ 
	&a_{34} = \frac{x_{35}}{x_{13} x_{35} + x_{14} x_{45}}, a_{35} = 1,\\ 
	&a_{45} = -\frac{{\left(x_{13} - 1\right)} x_{14} x_{45} + x_{15} + {\left(x_{13}^{2} - x_{13}\right)} x_{35}}{x_{13} x_{14} x_{35} + x_{14}^{2} x_{45}}
\end{align*}

Now assume $x_{14}= \frac{-x_{13}x_{35}}{x_{45}}$.
\begin{equation*}x=\left(
\right)
\end{equation*} \newline
Where matrix $A$ has entries
\begin{align*}
	&d_{1} = 1, d_{2} = 1, d_{3} = \frac{1}{x_{13}}, d_{4} = \frac{x_{45}}{x_{13} x_{35} + x_{14} x_{45}}, d_{5} = \frac{1}{x_{13} x_{35} + x_{14} x_{45}},\\ 
	&a_{12} = 1, a_{13} = 1, a_{14} = 1, a_{15} = 1,\\ 
	&a_{23} = 1, a_{24} = \frac{x_{25}}{x_{13} x_{35} + x_{14} x_{45}}, a_{25} = 1,\\ 
	&a_{34} = \frac{x_{35}}{x_{13} x_{35} + x_{14} x_{45}}, a_{35} = 1,\\ 
	&a_{45} = -\frac{{\left(x_{13} - 1\right)} x_{14} x_{45} + x_{15} + {\left(x_{13}^{2} - x_{13}\right)} x_{35}}{x_{13} x_{14} x_{35} + x_{14}^{2} x_{45}}
\end{align*}

Now assume $x_{14}= \frac{-x_{13}x_{35}}{x_{45}}$
\begin{equation*}x=\left(
\right)
\end{equation*} \newline
Where matrix $A$ has entries
\begin{align*}
	&d_{1} = 1, d_{2} = 1, d_{3} = \frac{1}{x_{13}}, d_{4} = \frac{1}{x_{24}}, d_{5} = \frac{1}{x_{24} x_{45}},\\ 
	&a_{12} = \frac{x_{13} x_{35} + x_{14} x_{45}}{x_{24} x_{45}}, a_{13} = 1, a_{14} = 1, a_{15} = 1,\\ 
	&a_{23} = 1, a_{24} = 1, a_{25} = 1,\\ 
	&a_{34} = \frac{x_{35}}{x_{24} x_{45}}, a_{35} = \frac{x_{14} x_{25} - {\left(x_{14} x_{24} - x_{24}^{2}\right)} x_{45}}{x_{13} x_{24}^{2} x_{45}},\\ 
	&a_{45} = \frac{x_{24} x_{45} - x_{25}}{x_{24}^{2} x_{45}}
\end{align*}

\begin{equation*}x=\left(\begin{matrix}
		1 & 0 & x_{13} & x_{14} & x_{15} \\
		0 & 1 & 0 & x_{24} & x_{25} \\
		0 & 0 & 1 & 0 & x_{35}  \\
		0 & 0 & 0 & 1 & x_{45}  \\
		0 & 0 & 0 & 0 & 1 \\
	\end{matrix}\right),x^A=\left(\begin{matrix}
		1 & 0 & 1 & 0 & 0 \\
		0 & 1 & 0 & 1 & 0 \\
		0 & 0 & 1 & 0 & 0 \\
		0 & 0 & 0 & 1 & 1 \\
		0 & 0 & 0 & 0 & 1
	\end{matrix}\right)
\end{equation*} \newline
Where matrix $A$ has entries
\begin{align*}
	&d_{1} = 1, d_{2} = 1, d_{3} = \frac{1}{x_{13}}, d_{4} = \frac{1}{x_{24}}, d_{5} = \frac{1}{x_{24} x_{45}},\\ 
	&a_{12} = \frac{x_{13} x_{35} + x_{14} x_{45}}{x_{24} x_{45}}, a_{13} = 1, a_{14} = 1, a_{15} = 1,\\ 
	&a_{23} = 1, a_{24} = 1, a_{25} = 1,\\ 
	&a_{34} = \frac{x_{35}}{x_{24} x_{45}}, a_{35} = -\frac{x_{15} x_{24} - x_{14} x_{25} + {\left(x_{14} x_{24} - x_{24}^{2}\right)} x_{45}}{x_{13} x_{24}^{2} x_{45}},\\ 
	&a_{45} = \frac{x_{24} x_{45} - x_{25}}{x_{24}^{2} x_{45}}
\end{align*}

		\section{Subcases of $Y_{12}$}

		\begin{equation*}x=\left(
\right)
		\end{equation*} \newline
		Where matrix $A$ has entries
		\begin{align*}
			&d_{1} = 1, d_{2} = 1, d_{3} = \frac{1}{x_{13}}, d_{4} = \frac{1}{x_{13} x_{34}}, d_{5} = -\frac{x_{34}}{x_{24} x_{35}},\\ 
			&a_{12} = 1, a_{13} = 1, a_{14} = 1, a_{15} = 1,\\ 
			&a_{23} = \frac{x_{24}}{x_{13} x_{34}}, a_{24} = 1, a_{25} = 1,\\ 
			&a_{34} = \frac{x_{13} x_{34} - x_{14}}{x_{13}^{2} x_{34}}, a_{35} = \frac{x_{15} x_{34} - {\left(x_{14} - x_{24}\right)} x_{35}}{x_{13} x_{24} x_{35}},\\ 
			&a_{45} = \frac{1}{x_{24}}
		\end{align*}
		
		
		First assume $x_{25}\neq \frac{x_{24}x_{35}}{x_{34}}$.
		\begin{equation*}x=\left(
\right)
		\end{equation*} \newline
		Where matrix $A$ has entries
		\begin{align*}
			&d_{1} = 1, d_{2} = 1, d_{3} = 1, d_{4} = \frac{1}{x_{34}}, d_{5} = \frac{x_{34}}{x_{25} x_{34} - x_{24} x_{35}},\\ 
			&a_{12} = \frac{x_{15} x_{34} - x_{14} x_{35}}{x_{25} x_{34} - x_{24} x_{35}}, a_{13} = \frac{x_{14}}{x_{34}}, a_{14} = 1, a_{15} = 1,\\ 
			&a_{23} = \frac{x_{24}}{x_{34}}, a_{24} = 1, a_{25} = 1,\\ 
			&a_{34} = 1, a_{35} = 1,\\ 
			&a_{45} = -\frac{x_{35}}{x_{25} x_{34} - x_{24} x_{35}}
		\end{align*}

		Now assume $x_{25}= \frac{x_{24}x_{35}}{x_{34}}$ and $x_{15}\neq \frac{x_{14}x_{35}}{x_{34}}$.
		\begin{equation*}x=\left(
\right)
		\end{equation*} \newline
		Where matrix $A$ has entries
		\begin{align*}
			&d_{1} = 1, d_{2} = 1, d_{3} = \frac{1}{x_{13}}, d_{4} = \frac{1}{x_{13} x_{34}}, d_{5} = \frac{x_{34}}{x_{25} x_{34} - x_{24} x_{35}},\\ 
			&a_{12} = 1, a_{13} = 1, a_{14} = 1, a_{15} = 1,\\ 
			&a_{23} = \frac{x_{24}}{x_{13} x_{34}}, a_{24} = 1, a_{25} = 1,\\ 
			&a_{34} = \frac{1}{x_{13}}, a_{35} = -\frac{x_{24} x_{35} + {\left(x_{15} - x_{25}\right)} x_{34}}{x_{13} x_{25} x_{34} - x_{13} x_{24} x_{35}},\\ 
			&a_{45} = -\frac{x_{35}}{x_{25} x_{34} - x_{24} x_{35}}
		\end{align*}
		
		Now assume $x_{25}= \frac{x_{24}x_{35}}{x_{34}}$.
		\begin{equation*}x=\left(\begin{matrix}
				1 & 0 & x_{13} & 0 & x_{15} \\
				0 & 1 & 0 & x_{24} & x_{25} \\
				0 & 0 & 1 & x_{34} & x_{35}  \\
				0 & 0 & 0 & 1 & 0  \\
				0 & 0 & 0 & 0 & 1 \\
			\end{matrix}\right),x^A=\left(\begin{matrix}
				1 & 0 & 1 & 0 & 0 \\
				0 & 1 & 0 & 0 & 0 \\
				0 & 0 & 1 & 1 & 0 \\
				0 & 0 & 0 & 1 & 0 \\
				0 & 0 & 0 & 0 & 1
			\end{matrix}\right)
		\end{equation*} \newline
		Where matrix $A$ has entries
		\begin{align*}
			&d_{1} = 1, d_{2} = 1, d_{3} = \frac{1}{x_{13}}, d_{4} = \frac{1}{x_{13} x_{34}}, d_{5} = 1,\\ 
			&a_{12} = 1, a_{13} = 1, a_{14} = 1, a_{15} = 1,\\ 
			&a_{23} = \frac{x_{24}}{x_{13} x_{34}}, a_{24} = 1, a_{25} = 1,\\ 
			&a_{34} = \frac{1}{x_{13}}, a_{35} = -\frac{x_{15}}{x_{13}},\\ 
			&a_{45} = -\frac{x_{35}}{x_{34}}
		\end{align*}
		
		
		First assume $x_{25}\neq \frac{x_{24}x_{35}}{x_{34}}$.
		\begin{equation*}x=\left(\begin{matrix}
				1 & 0 & x_{13} & x_{14} & 0 \\
				0 & 1 & 0 & x_{24} & x_{25} \\
				0 & 0 & 1 & x_{34} & x_{35}  \\
				0 & 0 & 0 & 1 & 0  \\
				0 & 0 & 0 & 0 & 1 \\
			\end{matrix}\right),x^A=\left(\begin{matrix}
				1 & 0 & 1 & 0 & 0 \\
				0 & 1 & 0 & 0 & 1 \\
				0 & 0 & 1 & 1 & 0 \\
				0 & 0 & 0 & 1 & 0 \\
				0 & 0 & 0 & 0 & 1
			\end{matrix}\right)
		\end{equation*} \newline
		Where matrix $A$ has entries
		\begin{align*}
			&d_{1} = 1, d_{2} = 1, d_{3} = \frac{1}{x_{13}}, d_{4} = \frac{1}{x_{13} x_{34}}, d_{5} = \frac{x_{34}}{x_{25} x_{34} - x_{24} x_{35}},\\ 
			&a_{12} = 1, a_{13} = 1, a_{14} = 1, a_{15} = 1,\\ 
			&a_{23} = \frac{x_{24}}{x_{13} x_{34}}, a_{24} = 1, a_{25} = 1,\\ 
			&a_{34} = \frac{x_{13} x_{34} - x_{14}}{x_{13}^{2} x_{34}}, a_{35} = \frac{x_{25} x_{34} + {\left(x_{14} - x_{24}\right)} x_{35}}{x_{13} x_{25} x_{34} - x_{13} x_{24} x_{35}},\\ 
			&a_{45} = -\frac{x_{35}}{x_{25} x_{34} - x_{24} x_{35}}
		\end{align*}
		
		Now assume $x_{25}=\frac{x_{24}x_{35}}{x_{34}}$.
		\begin{equation*}x=\left(\begin{matrix}
				1 & 0 & x_{13} & x_{14} & 0 \\
				0 & 1 & 0 & x_{24} & x_{25} \\
				0 & 0 & 1 & x_{34} & x_{35}  \\
				0 & 0 & 0 & 1 & 0  \\
				0 & 0 & 0 & 0 & 1 \\
			\end{matrix}\right),x^A=\left(\begin{matrix}
				1 & 0 & 1 & 0 & 0 \\
				0 & 1 & 0 & 0 & 0 \\
				0 & 0 & 1 & 1 & 0 \\
				0 & 0 & 0 & 1 & 0 \\
				0 & 0 & 0 & 0 & 1
			\end{matrix}\right)
		\end{equation*} \newline
		Where matrix $A$ has entries
		\begin{align*}
			&d_{1} = 1, d_{2} = 1, d_{3} = \frac{1}{x_{13}}, d_{4} = \frac{1}{x_{13} x_{34}}, d_{5} = 1,\\ 
			&a_{12} = 1, a_{13} = 1, a_{14} = 1, a_{15} = 1,\\ 
			&a_{23} = \frac{x_{24}}{x_{13} x_{34}}, a_{24} = 1, a_{25} = 1,\\ 
			&a_{34} = \frac{x_{13} x_{34} - x_{14}}{x_{13}^{2} x_{34}}, a_{35} = \frac{x_{14} x_{35}}{x_{13} x_{34}},\\ 
			&a_{45} = -\frac{x_{35}}{x_{34}}
		\end{align*}
		
		
		First assume $x_{25}\neq \frac{x_{24}x_{35}}{x_{34}}$.
		\begin{equation*}x=\left(\begin{matrix}
				1 & 0 & x_{13} & x_{14} & x_{15} \\
				0 & 1 & 0 & x_{24} & x_{25} \\
				0 & 0 & 1 & x_{34} & x_{35}  \\
				0 & 0 & 0 & 1 & 0  \\
				0 & 0 & 0 & 0 & 1 \\
			\end{matrix}\right),x^A=\left(\begin{matrix}
				1 & 0 & 1 & 0 & 0 \\
				0 & 1 & 0 & 0 & 1 \\
				0 & 0 & 1 & 1 & 0 \\
				0 & 0 & 0 & 1 & 0 \\
				0 & 0 & 0 & 0 & 1
			\end{matrix}\right)
		\end{equation*} \newline
		Where matrix $A$ has entries
		\begin{align*}
			&d_{1} = 1, d_{2} = 1, d_{3} = \frac{1}{x_{13}}, d_{4} = \frac{1}{x_{13} x_{34}}, d_{5} = \frac{x_{34}}{x_{25} x_{34} - x_{24} x_{35}},\\ 
			&a_{12} = 1, a_{13} = 1, a_{14} = 1, a_{15} = 1,\\ 
			&a_{23} = \frac{x_{24}}{x_{13} x_{34}}, a_{24} = 1, a_{25} = 1,\\ 
			&a_{34} = \frac{x_{13} x_{34} - x_{14}}{x_{13}^{2} x_{34}}, a_{35} = -\frac{{\left(x_{15} - x_{25}\right)} x_{34} - {\left(x_{14} - x_{24}\right)} x_{35}}{x_{13} x_{25} x_{34} - x_{13} x_{24} x_{35}},\\ 
			&a_{45} = -\frac{x_{35}}{x_{25} x_{34} - x_{24} x_{35}}
		\end{align*}

		Now assume $x_{25}= \frac{x_{24}x_{35}}{x_{34}}$.
		\begin{equation*}x=\left(
\right)
		\end{equation*} \newline
		Where matrix $A$ has entries
		\begin{align*}
			&d_{1} = 1, d_{2} = 1, d_{3} = \frac{1}{x_{13}}, d_{4} = \frac{1}{x_{13} x_{34}}, d_{5} = 1,\\ 
			&a_{12} = 1, a_{13} = 1, a_{14} = 1, a_{15} = 1,\\ 
			&a_{23} = \frac{x_{24}}{x_{13} x_{34}}, a_{24} = 1, a_{25} = 1,\\ 
			&a_{34} = \frac{x_{13} x_{34} - x_{14}}{x_{13}^{2} x_{34}}, a_{35} = -\frac{x_{15} x_{34} - x_{14} x_{35}}{x_{13} x_{34}},\\ 
			&a_{45} = -\frac{x_{35}}{x_{34}}
		\end{align*}
		\section{Subcases of $Y_{13}$}

\begin{equation*}x=\left(%
\right)
\end{equation*} \newline
Where matrix $A$ has entries
\begin{align*}
	&d_{1} = 1, d_{2} = 1, d_{3} = \frac{1}{x_{13}}, d_{4} = \frac{1}{x_{13} x_{34}}, d_{5} = \frac{1}{x_{13} x_{34} x_{45}}, \\ 
	&a_{12} = 1, a_{13} = 1, a_{14} = 1, a_{15} = 1, \\ 
	&a_{23} = 0, a_{24} = 0, a_{25} = 1, \\ 
	&a_{34} = \frac{x_{13} x_{34} - x_{14}}{x_{13}^{2} x_{34}}, a_{35} = -\frac{x_{13} x_{15} x_{34} - {\left(x_{13}^{2} x_{34}^{2} - x_{13} x_{14} x_{34} + x_{14}^{2}\right)} x_{45}}{x_{13}^{3} x_{34}^{2} x_{45}}, \\ 
	&a_{45} = \frac{x_{13} x_{34} - x_{14}}{x_{13}^{2} x_{34}^{2}}
\end{align*}

\begin{equation*}x=\left(
\right)
\end{equation*} \newline
Where matrix $A$ has entries
\begin{align*}
	&d_{1} = 1, d_{2} = 1, d_{3} = \frac{1}{x_{13}}, d_{4} = \frac{1}{x_{13} x_{34}}, d_{5} = \frac{1}{x_{13} x_{34} x_{45}}, \\ 
	&a_{12} = 1, a_{13} = 1, a_{14} = 1, a_{15} = 1, \\ 
	&a_{23} = 0, a_{24} = \frac{x_{25}}{x_{13} x_{34} x_{45}}, a_{25} = 1, \\ 
	&a_{34} = \frac{x_{13} x_{34} - x_{14}}{x_{13}^{2} x_{34}}, a_{35} = \frac{x_{13}^{2} x_{34}^{2} - x_{13} x_{14} x_{34} + x_{14}^{2}}{x_{13}^{3} x_{34}^{2}}, \\ 
	&a_{45} = \frac{x_{13} x_{34} - x_{14}}{x_{13}^{2} x_{34}^{2}}
\end{align*}

\begin{equation*}x=\left(\begin{matrix}
		1 & 0 & x_{13} & x_{14} & x_{15} \\
		0 & 1 & 0 & 0 & x_{25} \\
		0 & 0 & 1 & x_{34} & 0  \\
		0 & 0 & 0 & 1 & x_{45}  \\
		0 & 0 & 0 & 0 & 1 \\
	\end{matrix}\right),x^A=\left(\begin{matrix}
		1 & 0 & 1 & 0 & 0 \\
		0 & 1 & 0 & 0 & 0 \\
		0 & 0 & 1 & 1 & 0 \\
		0 & 0 & 0 & 1 & 1 \\
		0 & 0 & 0 & 0 & 1
	\end{matrix}\right)
\end{equation*} \newline
Where matrix $A$ has entries
\begin{align*}
	&d_{1} = 1, d_{2} = 1, d_{3} = \frac{1}{x_{13}}, d_{4} = \frac{1}{x_{13} x_{34}}, d_{5} = \frac{1}{x_{13} x_{34} x_{45}}, \\ 
	&a_{12} = 1, a_{13} = 1, a_{14} = 1, a_{15} = 1, \\ 
	&a_{23} = 0, a_{24} = \frac{x_{25}}{x_{13} x_{34} x_{45}}, a_{25} = 1, \\ 
	&a_{34} = \frac{x_{13} x_{34} - x_{14}}{x_{13}^{2} x_{34}}, a_{35} = -\frac{x_{13} x_{15} x_{34} - {\left(x_{13}^{2} x_{34}^{2} - x_{13} x_{14} x_{34} + x_{14}^{2}\right)} x_{45}}{x_{13}^{3} x_{34}^{2} x_{45}}, \\ 
	&a_{45} = \frac{x_{13} x_{34} - x_{14}}{x_{13}^{2} x_{34}^{2}}
\end{align*}

\begin{equation*}x=\left(
\right)
\end{equation*} \newline
Where matrix $A$ has entries
\begin{align*}
	&d_{1} = 1, d_{2} = 1, d_{3} = \frac{1}{x_{13}}, d_{4} = \frac{1}{x_{13} x_{34}}, d_{5} = \frac{1}{x_{13} x_{34} x_{45}}, \\ 
	&a_{12} = 1, a_{13} = 1, a_{14} = 1, a_{15} = 1, \\ 
	&a_{23} = \frac{x_{24}}{x_{13} x_{34}}, a_{24} = \frac{x_{13} x_{24} x_{34} - x_{14} x_{24}}{x_{13}^{2} x_{34}^{2}}, a_{25} = 1, \\ 
	&a_{34} = \frac{x_{13} x_{34} - x_{14}}{x_{13}^{2} x_{34}}, a_{35} = \frac{x_{13}^{2} x_{34}^{2} - x_{13} x_{14} x_{34} + x_{14}^{2}}{x_{13}^{3} x_{34}^{2}}, \\ 
	&a_{45} = \frac{x_{13} x_{34} - x_{14}}{x_{13}^{2} x_{34}^{2}}
\end{align*}

\begin{equation*}x=\left(\begin{matrix}
		1 & 0 & x_{13} & x_{14} & x_{15} \\
		0 & 1 & 0 & x_{24} & 0 \\
		0 & 0 & 1 & x_{34} & 0  \\
		0 & 0 & 0 & 1 & x_{45}  \\
		0 & 0 & 0 & 0 & 1 \\
	\end{matrix}\right),x^A=\left(\begin{matrix}
		1 & 0 & 1 & 0 & 0 \\
		0 & 1 & 0 & 0 & 0 \\
		0 & 0 & 1 & 1 & 0 \\
		0 & 0 & 0 & 1 & 1 \\
		0 & 0 & 0 & 0 & 1
	\end{matrix}\right)
\end{equation*} \newline
Where matrix $A$ has entries
\begin{align*}
	&d_{1} = 1, d_{2} = 1, d_{3} = \frac{1}{x_{13}}, d_{4} = \frac{1}{x_{13} x_{34}}, d_{5} = \frac{1}{x_{13} x_{34} x_{45}}, \\ 
	&a_{12} = 1, a_{13} = 1, a_{14} = 1, a_{15} = 1, \\ 
	&a_{23} = \frac{x_{24}}{x_{13} x_{34}}, a_{24} = \frac{x_{13} x_{24} x_{34} - x_{14} x_{24}}{x_{13}^{2} x_{34}^{2}}, a_{25} = 1, \\ 
	&a_{34} = \frac{x_{13} x_{34} - x_{14}}{x_{13}^{2} x_{34}}, a_{35} = -\frac{x_{13} x_{15} x_{34} - {\left(x_{13}^{2} x_{34}^{2} - x_{13} x_{14} x_{34} + x_{14}^{2}\right)} x_{45}}{x_{13}^{3} x_{34}^{2} x_{45}}, \\ 
	&a_{45} = \frac{x_{13} x_{34} - x_{14}}{x_{13}^{2} x_{34}^{2}}
\end{align*}

\begin{equation*}x=\left(
\right)
\end{equation*} \newline
Where matrix $A$ has entries
\begin{align*}
	&d_{1} = 1, d_{2} = 1, d_{3} = \frac{1}{x_{13}}, d_{4} = \frac{1}{x_{13} x_{34}}, d_{5} = \frac{1}{x_{13} x_{34} x_{45}}, \\ 
	&a_{12} = 1, a_{13} = 1, a_{14} = 1, a_{15} = 1, \\ 
	&a_{23} = \frac{x_{24}}{x_{13} x_{34}}, a_{24} = \frac{x_{13} x_{25} x_{34} + {\left(x_{13} x_{24} x_{34} - x_{14} x_{24}\right)} x_{45}}{x_{13}^{2} x_{34}^{2} x_{45}}, a_{25} = 1, \\ 
	&a_{34} = \frac{x_{13} x_{34} - x_{14}}{x_{13}^{2} x_{34}}, a_{35} = \frac{x_{13}^{2} x_{34}^{2} - x_{13} x_{14} x_{34} + x_{14}^{2}}{x_{13}^{3} x_{34}^{2}}, \\ 
	&a_{45} = \frac{x_{13} x_{34} - x_{14}}{x_{13}^{2} x_{34}^{2}}
\end{align*}

\begin{equation*}x=\left(\begin{matrix}
		1 & 0 & x_{13} & x_{14} & x_{15} \\
		0 & 1 & 0 & x_{24} & x_{25} \\
		0 & 0 & 1 & x_{34} & 0  \\
		0 & 0 & 0 & 1 & x_{45}  \\
		0 & 0 & 0 & 0 & 1 \\
	\end{matrix}\right),x^A=\left(\begin{matrix}
		1 & 0 & 1 & 0 & 0 \\
		0 & 1 & 0 & 0 & 0 \\
		0 & 0 & 1 & 1 & 0 \\
		0 & 0 & 0 & 1 & 1 \\
		0 & 0 & 0 & 0 & 1
	\end{matrix}\right)
\end{equation*} \newline
Where matrix $A$ has entries
\begin{align*}
	&d_{1} = 1, d_{2} = 1, d_{3} = \frac{1}{x_{13}}, d_{4} = \frac{1}{x_{13} x_{34}}, d_{5} = \frac{1}{x_{13} x_{34} x_{45}}, \\ 
	&a_{12} = 1, a_{13} = 1, a_{14} = 1, a_{15} = 1, \\ 
	&a_{23} = \frac{x_{24}}{x_{13} x_{34}}, a_{24} = \frac{x_{13} x_{25} x_{34} + {\left(x_{13} x_{24} x_{34} - x_{14} x_{24}\right)} x_{45}}{x_{13}^{2} x_{34}^{2} x_{45}}, a_{25} = 1, \\ 
	&a_{34} = \frac{x_{13} x_{34} - x_{14}}{x_{13}^{2} x_{34}}, a_{35} = -\frac{x_{13} x_{15} x_{34} - {\left(x_{13}^{2} x_{34}^{2} - x_{13} x_{14} x_{34} + x_{14}^{2}\right)} x_{45}}{x_{13}^{3} x_{34}^{2} x_{45}}, \\ 
	&a_{45} = \frac{x_{13} x_{34} - x_{14}}{x_{13}^{2} x_{34}^{2}}
\end{align*}

\begin{equation*}x=\left(
\right)
\end{equation*} \newline
Where matrix $A$ has entries
\begin{align*}
	&d_{1} = 1, d_{2} = 1, d_{3} = \frac{1}{x_{13}}, d_{4} = \frac{1}{x_{13} x_{34}}, d_{5} = \frac{1}{x_{13} x_{34} x_{45}}, \\ 
	&a_{12} = 1, a_{13} = 1, a_{14} = 1, a_{15} = 1, \\ 
	&a_{23} = 0, a_{24} = 0, a_{25} = 1, \\ 
	&a_{34} = \frac{1 x_{13} x_{34} - 1 x_{14}}{x_{13}^{2} x_{34}}, a_{35} = \frac{1 x_{13} x_{14} x_{35} + {\left(1 x_{13}^{2} x_{34}^{2} - 1 x_{13} x_{14} x_{34} + 1 x_{14}^{2}\right)} x_{45}}{x_{13}^{3} x_{34}^{2} x_{45}}, \\ 
	&a_{45} = -\frac{1 x_{13} x_{35} - {\left(1 x_{13} x_{34} - 1 x_{14}\right)} x_{45}}{x_{13}^{2} x_{34}^{2} x_{45}}
\end{align*}

\begin{equation*}x=\left(\begin{matrix}
		1 & 0 & x_{13} & x_{14} & x_{15} \\
		0 & 1 & 0 & 0 & 0 \\
		0 & 0 & 1 & x_{34} & x_{35}  \\
		0 & 0 & 0 & 1 & x_{45}  \\
		0 & 0 & 0 & 0 & 1 \\
	\end{matrix}\right),x^A=\left(\begin{matrix}
		1 & 0 & 1 & 0 & 0 \\
		0 & 1 & 0 & 0 & 0 \\
		0 & 0 & 1 & 1 & 0 \\
		0 & 0 & 0 & 1 & 1 \\
		0 & 0 & 0 & 0 & 1
	\end{matrix}\right)
\end{equation*} \newline
Where matrix $A$ has entries
\begin{align*}
	&d_{1} = 1, d_{2} = 1, d_{3} = \frac{1}{x_{13}}, d_{4} = \frac{1}{x_{13} x_{34}}, d_{5} = \frac{1}{x_{13} x_{34} x_{45}}, \\ 
	&a_{12} = 1, a_{13} = 1, a_{14} = 1, a_{15} = 1, \\ 
	&a_{23} = 0, a_{24} = 0, a_{25} = 1, \\ 
	&a_{34} = \frac{1 x_{13} x_{34} - 1 x_{14}}{x_{13}^{2} x_{34}}, a_{35} = -\frac{1 x_{13} x_{15} x_{34} - 1 x_{13} x_{14} x_{35} - {\left(1 x_{13}^{2} x_{34}^{2} - 1 x_{13} x_{14} x_{34} + 1 x_{14}^{2}\right)} x_{45}}{x_{13}^{3} x_{34}^{2} x_{45}}, \\ 
	&a_{45} = -\frac{1 x_{13} x_{35} - {\left(1 x_{13} x_{34} - 1 x_{14}\right)} x_{45}}{x_{13}^{2} x_{34}^{2} x_{45}}
\end{align*}

\begin{equation*}x=\left(
\right)
\end{equation*} \newline
Where matrix $A$ has entries
\begin{align*}
	&d_{1} = 1, d_{2} = 1, d_{3} = \frac{1}{x_{13}}, d_{4} = \frac{1}{x_{13} x_{34}}, d_{5} = \frac{1}{x_{13} x_{34} x_{45}}, \\ 
	&a_{12} = 1, a_{13} = 1, a_{14} = 1, a_{15} = 1, \\ 
	&a_{23} = 0, a_{24} = \frac{x_{25}}{x_{13} x_{34} x_{45}}, a_{25} = 1, \\ 
	&a_{34} = \frac{x_{13} x_{34} - x_{14}}{x_{13}^{2} x_{34}}, a_{35} = \frac{x_{13} x_{14} x_{35} + {\left(x_{13}^{2} x_{34}^{2} - x_{13} x_{14} x_{34} + x_{14}^{2}\right)} x_{45}}{x_{13}^{3} x_{34}^{2} x_{45}}, \\ 
	&a_{45} = -\frac{x_{13} x_{35} - {\left(x_{13} x_{34} - x_{14}\right)} x_{45}}{x_{13}^{2} x_{34}^{2} x_{45}}
\end{align*}

\begin{equation*}x=\left(\begin{matrix}
		1 & 0 & x_{13} & x_{14} & x_{15} \\
		0 & 1 & 0 & 0 & x_{25} \\
		0 & 0 & 1 & x_{34} & x_{35}  \\
		0 & 0 & 0 & 1 & x_{45}  \\
		0 & 0 & 0 & 0 & 1 \\
	\end{matrix}\right),x^A=\left(\begin{matrix}
		1 & 0 & 1 & 0 & 0 \\
		0 & 1 & 0 & 0 & 0 \\
		0 & 0 & 1 & 1 & 0 \\
		0 & 0 & 0 & 1 & 1 \\
		0 & 0 & 0 & 0 & 1
	\end{matrix}\right)
\end{equation*} \newline
Where matrix $A$ has entries
\begin{align*}
	&d_{1} = 1, d_{2} = 1, d_{3} = \frac{1}{x_{13}}, d_{4} = \frac{1}{x_{13} x_{34}}, d_{5} = \frac{1}{x_{13} x_{34} x_{45}}, \\ 
	&a_{12} = 1, a_{13} = 1, a_{14} = 1, a_{15} = 1, \\ 
	&a_{23} = 0, a_{24} = \frac{x_{25}}{x_{13} x_{34} x_{45}}, a_{25} = 1, \\ 
	&a_{34} = \frac{x_{13} x_{34} - x_{14}}{x_{13}^{2} x_{34}}, a_{35} = -\frac{x_{13} x_{15} x_{34} - x_{13} x_{14} x_{35} - {\left(x_{13}^{2} x_{34}^{2} - x_{13} x_{14} x_{34} + x_{14}^{2}\right)} x_{45}}{x_{13}^{3} x_{34}^{2} x_{45}}, \\ 
	&a_{45} = -\frac{x_{13} x_{35} - {\left(x_{13} x_{34} - x_{14}\right)} x_{45}}{x_{13}^{2} x_{34}^{2} x_{45}}
\end{align*}

\begin{equation*}x=\left(
\right)
\end{equation*} \newline
Where matrix $A$ has entries
\begin{align*}
	&d_{1} = 1, d_{2} = 1, d_{3} = \frac{1}{x_{13}}, d_{4} = \frac{1}{x_{13} x_{34}}, d_{5} = \frac{1}{x_{13} x_{34} x_{45}}, \\ 
	&a_{12} = 1, a_{13} = 1, a_{14} = 1, a_{15} = 1, \\ 
	&a_{23} = \frac{x_{24}}{x_{13} x_{34}}, a_{24} = \frac{x_{24} x_{34} x_{45} - x_{24} x_{35}}{x_{13} x_{34}^{2} x_{45}}, a_{25} = 1, \\ 
	&a_{34} = \frac{1}{x_{13}}, a_{35} = \frac{x_{13} x_{34} x_{45} - x_{15}}{x_{13}^{2} x_{34} x_{45}}, \\ 
	&a_{45} = \frac{x_{34} x_{45} - x_{35}}{x_{13} x_{34}^{2} x_{45}}
\end{align*}

\begin{equation*}x=\left(\begin{matrix}
		1 & 0 & x_{13} & x_{14} & 0 \\
		0 & 1 & 0 & x_{24} & 0 \\
		0 & 0 & 1 & x_{34} & x_{35}  \\
		0 & 0 & 0 & 1 & x_{45}  \\
		0 & 0 & 0 & 0 & 1 \\
	\end{matrix}\right),x^A=\left(\begin{matrix}
		1 & 0 & 1 & 0 & 0 \\
		0 & 1 & 0 & 0 & 0 \\
		0 & 0 & 1 & 1 & 0 \\
		0 & 0 & 0 & 1 & 1 \\
		0 & 0 & 0 & 0 & 1
	\end{matrix}\right)
\end{equation*} \newline
Where matrix $A$ has entries
\begin{align*}
	&d_{1} = 1, d_{2} = 1, d_{3} = \frac{1}{x_{13}}, d_{4} = \frac{1}{x_{13} x_{34}}, d_{5} = \frac{1}{x_{13} x_{34} x_{45}}, \\ 
	&a_{12} = 1, a_{13} = 1, a_{14} = 1, a_{15} = 1, \\ 
	&a_{23} = \frac{x_{24}}{x_{13} x_{34}}, a_{24} = -\frac{x_{13} x_{24} x_{35} - {\left(x_{13} x_{24} x_{34} - x_{14} x_{24}\right)} x_{45}}{x_{13}^{2} x_{34}^{2} x_{45}}, a_{25} = 1, \\ 
	&a_{34} = \frac{x_{13} x_{34} - x_{14}}{x_{13}^{2} x_{34}}, a_{35} = \frac{x_{13} x_{14} x_{35} + {\left(x_{13}^{2} x_{34}^{2} - x_{13} x_{14} x_{34} + x_{14}^{2}\right)} x_{45}}{x_{13}^{3} x_{34}^{2} x_{45}}, \\ 
	&a_{45} = -\frac{x_{13} x_{35} - {\left(x_{13} x_{34} - x_{14}\right)} x_{45}}{x_{13}^{2} x_{34}^{2} x_{45}}
\end{align*}

\begin{equation*}x=\left(\begin{matrix}
		1 & 0 & x_{13} & x_{14} & x_{15} \\
		0 & 1 & 0 & x_{24} & 0 \\
		0 & 0 & 1 & x_{34} & x_{35}  \\
		0 & 0 & 0 & 1 & x_{45}  \\
		0 & 0 & 0 & 0 & 1 \\
	\end{matrix}\right),x^A=\left(\begin{matrix}
		1 & 0 & 1 & 0 & 0 \\
		0 & 1 & 0 & 0 & 0 \\
		0 & 0 & 1 & 1 & 0 \\
		0 & 0 & 0 & 1 & 1 \\
		0 & 0 & 0 & 0 & 1
	\end{matrix}\right)
\end{equation*} \newline
Where matrix $A$ has entries
\begin{align*}
	&d_{1} = 1, d_{2} = 1, d_{3} = \frac{1}{x_{13}}, d_{4} = \frac{1}{x_{13} x_{34}}, d_{5} = \frac{1}{x_{13} x_{34} x_{45}}, \\ 
	&a_{12} = 1, a_{13} = 1, a_{14} = 1, a_{15} = 1, \\ 
	&a_{23} = \frac{x_{24}}{x_{13} x_{34}}, a_{24} = -\frac{x_{13} x_{24} x_{35} - {\left(x_{13} x_{24} x_{34} - x_{14} x_{24}\right)} x_{45}}{x_{13}^{2} x_{34}^{2} x_{45}}, a_{25} = 1, \\ 
	&a_{34} = \frac{x_{13} x_{34} - x_{14}}{x_{13}^{2} x_{34}}, a_{35} = -\frac{x_{13} x_{15} x_{34} - x_{13} x_{14} x_{35} - {\left(x_{13}^{2} x_{34}^{2} - x_{13} x_{14} x_{34} + x_{14}^{2}\right)} x_{45}}{x_{13}^{3} x_{34}^{2} x_{45}}, \\ 
	&a_{45} = -\frac{x_{13} x_{35} - {\left(x_{13} x_{34} - x_{14}\right)} x_{45}}{x_{13}^{2} x_{34}^{2} x_{45}}
\end{align*}

\begin{equation*}x=\left(
\right)
\end{equation*} \newline
Where matrix $A$ has entries
\begin{align*}
	&d_{1} = 1, d_{2} = 1, d_{3} = 1, d_{4} = \frac{1}{x_{34}}, d_{5} = \frac{1}{x_{34} x_{45}}, \\ 
	&a_{12} = 1, a_{13} = \frac{x_{14}}{x_{34}}, a_{14} = 1, a_{15} = 1, \\ 
	&a_{23} = \frac{x_{24}}{x_{34}}, a_{24} = \frac{x_{24} x_{34} x_{45} - x_{15} x_{24} + x_{14} x_{25}}{x_{14} x_{34} x_{45}}, a_{25} = 1, \\ 
	&a_{34} = \frac{x_{34}^{2} x_{45} - x_{15} x_{34} + x_{14} x_{35}}{x_{14} x_{34} x_{45}}, a_{35} = 1, \\ 
	&a_{45} = \frac{x_{34} x_{45} - x_{15}}{x_{14} x_{34} x_{45}}
\end{align*}

\begin{equation*}x=\left(\begin{matrix}
		1 & 0 & x_{13} & 0 & 0 \\
		0 & 1 & 0 & x_{24} & x_{25} \\
		0 & 0 & 1 & x_{34} & x_{35}  \\
		0 & 0 & 0 & 1 & x_{45}  \\
		0 & 0 & 0 & 0 & 1 \\
	\end{matrix}\right),x^A=\left(\begin{matrix}
		1 & 0 & 1 & 0 & 0 \\
		0 & 1 & 0 & 0 & 0 \\
		0 & 0 & 1 & 1 & 0 \\
		0 & 0 & 0 & 1 & 1 \\
		0 & 0 & 0 & 0 & 1
	\end{matrix}\right)
\end{equation*} \newline
Where matrix $A$ has entries
\begin{align*}
	&d_{1} = 1, d_{2} = 1, d_{3} = \frac{1}{x_{13}}, d_{4} = \frac{1}{x_{13} x_{34}}, d_{5} = \frac{1}{x_{13} x_{34} x_{45}}, \\ 
	&a_{12} = 1, a_{13} = 1, a_{14} = 1, a_{15} = 1, \\ 
	&a_{23} = \frac{x_{24}}{x_{13} x_{34}}, a_{24} = \frac{x_{24} x_{34} x_{45} + x_{25} x_{34} - x_{24} x_{35}}{x_{13} x_{34}^{2} x_{45}}, a_{25} = 1, \\ 
	&a_{34} = \frac{1}{x_{13}}, a_{35} = \frac{1 + 1}{x_{13}}, \\ 
	&a_{45} = \frac{{\left(1 + 1\right)} x_{34} x_{45} - x_{35}}{x_{13} x_{34}^{2} x_{45}}
\end{align*}

\begin{equation*}x=\left(\begin{matrix}
		1 & 0 & x_{13} & 0 & x_{15} \\
		0 & 1 & 0 & x_{24} & x_{25} \\
		0 & 0 & 1 & x_{34} & x_{35}  \\
		0 & 0 & 0 & 1 & x_{45}  \\
		0 & 0 & 0 & 0 & 1 \\
	\end{matrix}\right),x^A=\left(\begin{matrix}
		1 & 0 & 1 & 0 & 0 \\
		0 & 1 & 0 & 0 & 0 \\
		0 & 0 & 1 & 1 & 0 \\
		0 & 0 & 0 & 1 & 1 \\
		0 & 0 & 0 & 0 & 1
	\end{matrix}\right)
\end{equation*} \newline
Where matrix $A$ has entries
\begin{align*}
	&d_{1} = 1, d_{2} = 1, d_{3} = \frac{1}{x_{13}}, d_{4} = \frac{1}{x_{13} x_{34}}, d_{5} = \frac{1}{x_{13} x_{34} x_{45}}, \\ 
	&a_{12} = 1, a_{13} = 1, a_{14} = 1, a_{15} = 1, \\ 
	&a_{23} = \frac{x_{24}}{x_{13} x_{34}}, a_{24} = \frac{x_{24} x_{34} x_{45} + x_{25} x_{34} - x_{24} x_{35}}{x_{13} x_{34}^{2} x_{45}}, a_{25} = 1, \\ 
	&a_{34} = \frac{1}{x_{13}}, a_{35} = \frac{{\left(1 + 1\right)} x_{13} x_{34} x_{45} - x_{15}}{x_{13}^{2} x_{34} x_{45}}, \\ 
	&a_{45} = \frac{{\left(1 + 1\right)} x_{34} x_{45} - x_{35}}{x_{13} x_{34}^{2} x_{45}}
\end{align*}

\begin{equation*}x=\left(\begin{matrix}
		1 & 0 & x_{13} & x_{14} & 0 \\
		0 & 1 & 0 & x_{24} & x_{25} \\
		0 & 0 & 1 & x_{34} & x_{35}  \\
		0 & 0 & 0 & 1 & x_{45}  \\
		0 & 0 & 0 & 0 & 1 \\
	\end{matrix}\right),x^A=\left(\begin{matrix}
		1 & 0 & 1 & 0 & 0 \\
		0 & 1 & 0 & 0 & 0 \\
		0 & 0 & 1 & 1 & 0 \\
		0 & 0 & 0 & 1 & 1 \\
		0 & 0 & 0 & 0 & 1
	\end{matrix}\right)
\end{equation*} \newline
Where matrix $A$ has entries
\begin{align*}
	&d_{1} = 1, d_{2} = 1, d_{3} = \frac{1}{x_{13}}, d_{4} = \frac{1}{x_{13} x_{34}}, d_{5} = \frac{1}{x_{13} x_{34} x_{45}}, \\ 
	&a_{12} = 1, a_{13} = 1, a_{14} = 1, a_{15} = 1, \\ 
	&a_{23} = \frac{x_{24}}{x_{13} x_{34}}, a_{24} = \frac{x_{13} x_{25} x_{34} - x_{13} x_{24} x_{35} + {\left(x_{13} x_{24} x_{34} - x_{14} x_{24}\right)} x_{45}}{x_{13}^{2} x_{34}^{2} x_{45}}, a_{25} = 1, \\ 
	&a_{34} = \frac{x_{13} x_{34} - x_{14}}{x_{13}^{2} x_{34}}, a_{35} = \frac{x_{13} x_{14} x_{35} + {\left(x_{13}^{2} x_{34}^{2} - x_{13} x_{14} x_{34} + x_{14}^{2}\right)} x_{45}}{x_{13}^{3} x_{34}^{2} x_{45}}, \\ 
	&a_{45} = -\frac{x_{13} x_{35} - {\left(x_{13} x_{34} - x_{14}\right)} x_{45}}{x_{13}^{2} x_{34}^{2} x_{45}}
\end{align*}

\begin{equation*}x=\left(\begin{matrix}
		1 & 0 & x_{13} & x_{14} & x_{15} \\
		0 & 1 & 0 & x_{24} & x_{25} \\
		0 & 0 & 1 & x_{34} & x_{35}  \\
		0 & 0 & 0 & 1 & x_{45}  \\
		0 & 0 & 0 & 0 & 1 \\
	\end{matrix}\right),x^A=\left(\begin{matrix}
		1 & 0 & 1 & 0 & 0 \\
		0 & 1 & 0 & 0 & 0 \\
		0 & 0 & 1 & 1 & 0 \\
		0 & 0 & 0 & 1 & 1 \\
		0 & 0 & 0 & 0 & 1
	\end{matrix}\right)
\end{equation*} \newline
Where matrix $A$ has entries
\begin{align*}
	&d_{1} = 1, d_{2} = 1, d_{3} = \frac{1}{x_{13}}, d_{4} = \frac{1}{x_{13} x_{34}}, d_{5} = \frac{1}{x_{13} x_{34} x_{45}}, \\ 
	&a_{12} = 1, a_{13} = 1, a_{14} = 1, a_{15} = 1, \\ 
	&a_{23} = \frac{x_{24}}{x_{13} x_{34}}, a_{24} = \frac{x_{13} x_{25} x_{34} - x_{13} x_{24} x_{35} + {\left(x_{13} x_{24} x_{34} - x_{14} x_{24}\right)} x_{45}}{x_{13}^{2} x_{34}^{2} x_{45}}, a_{25} = 1, \\ 
	&a_{34} = \frac{x_{13} x_{34} - x_{14}}{x_{13}^{2} x_{34}}, a_{35} = -\frac{x_{13} x_{15} x_{34} - x_{13} x_{14} x_{35} - {\left(x_{13}^{2} x_{34}^{2} - x_{13} x_{14} x_{34} + x_{14}^{2}\right)} x_{45}}{x_{13}^{3} x_{34}^{2} x_{45}}, \\ 
	&a_{45} = -\frac{x_{13} x_{35} - {\left(x_{13} x_{34} - x_{14}\right)} x_{45}}{x_{13}^{2} x_{34}^{2} x_{45}}
\end{align*}

		\section{Subcases of $Y_{14}$}

\begin{equation*}x=\left(
\right)
\end{equation*} \newline
Where matrix $A$ has entries
\begin{align*}
	&d_{1} = 1, d_{2} = \frac{1}{x_{12}}, d_{3} = \frac{x_{34}}{x_{12} x_{24} + x_{13} x_{34}}, d_{4} = \frac{1}{x_{12} x_{24} + x_{13} x_{34}}, d_{5} = 1,\\ 
	&a_{12} = 1, a_{13} = 1, a_{14} = 1, a_{15} = 1,\\ 
	&a_{23} = \frac{x_{24}}{x_{12} x_{24} + x_{13} x_{34}}, a_{24} = 1, a_{25} = 1,\\ 
	&a_{34} = -\frac{{\left(x_{12} - 1\right)} x_{13} x_{34} + x_{14} + {\left(x_{12}^{2} - x_{12}\right)} x_{24}}{x_{12} x_{13} x_{24} + x_{13}^{2} x_{34}}, a_{35} = -\frac{x_{12}}{x_{13}},\\ 
	&a_{45} = 0
\end{align*}

Now assume $x_{24}= \frac{-x_{13}x_{34}}{x_{12}}$
\begin{equation*}x=\left(\begin{matrix}
		1 & x_{12} & x_{13} & x_{14} & 0 \\
		0 & 1 & 0 & x_{24} & 0 \\
		0 & 0 & 1 & x_{34} & 0  \\
		0 & 0 & 0 & 1 & 0  \\
		0 & 0 & 0 & 0 & 1 \\
	\end{matrix}\right),x^A=\left(\begin{matrix}
		1 & 1 & 0 & 0 & 0 \\
		0 & 1 & 0 & 0 & 0 \\
		0 & 0 & 1 & 1 & 0 \\
		0 & 0 & 0 & 1 & 0 \\
		0 & 0 & 0 & 0 & 1
	\end{matrix}\right)
\end{equation*} \newline
Where matrix $A$ has entries
\begin{align*}
	&d_{1} = 1, d_{2} = \frac{1}{x_{12}}, d_{3} = 1, d_{4} = \frac{1}{x_{34}}, d_{5} = 1,\\ 
	&a_{12} = 1, a_{13} = 1, a_{14} = 1, a_{15} = 1,\\ 
	&a_{23} = -\frac{x_{13}}{x_{12}}, a_{24} = 1, a_{25} = 1,\\ 
	&a_{34} = -\frac{x_{14} + {\left(x_{12} - 1\right)} x_{34}}{x_{13} x_{34}}, a_{35} = -\frac{x_{12}}{x_{13}},\\ 
	&a_{45} = 0
\end{align*}


First assume $x_{24}\neq \frac{-x_{13}x_{34}}{x_{12}}$
\begin{equation*}x=\left(\begin{matrix}
		1 & x_{12} & x_{13} & x_{14} & x_{15} \\
		0 & 1 & 0 & x_{24} & 0 \\
		0 & 0 & 1 & x_{34} & 0  \\
		0 & 0 & 0 & 1 & 0  \\
		0 & 0 & 0 & 0 & 1 \\
	\end{matrix}\right),x^A=\left(\begin{matrix}
		1 & 1 & 1 & 0 & 0 \\
		0 & 1 & 0 & 0 & 0 \\
		0 & 0 & 1 & 1 & 0 \\
		0 & 0 & 0 & 1 & 0 \\
		0 & 0 & 0 & 0 & 1
	\end{matrix}\right)
\end{equation*} \newline
Where matrix $A$ has entries
\begin{align*}
	&d_{1} = 1, d_{2} = \frac{1}{x_{12}}, d_{3} = \frac{x_{34}}{x_{12} x_{24} + x_{13} x_{34}}, d_{4} = \frac{1}{x_{12} x_{24} + x_{13} x_{34}}, d_{5} = 1,\\ 
	&a_{12} = 1, a_{13} = 1, a_{14} = 1, a_{15} = 1,\\ 
	&a_{23} = \frac{x_{24}}{x_{12} x_{24} + x_{13} x_{34}}, a_{24} = 1, a_{25} = 1,\\ 
	&a_{34} = -\frac{{\left(x_{12} - 1\right)} x_{13} x_{34} + x_{14} + {\left(x_{12}^{2} - x_{12}\right)} x_{24}}{x_{12} x_{13} x_{24} + x_{13}^{2} x_{34}}, a_{35} = -\frac{x_{12} + x_{15}}{x_{13}},\\ 
	&a_{45} = 0
\end{align*}

Now assume $x_{24}= \frac{-x_{13}x_{34}}{x_{12}}$
\begin{equation*}x=\left(
\right)
\end{equation*} \newline
Where matrix $A$ has entries
\begin{align*}
	&d_{1} = 1, d_{2} = \frac{1}{x_{12}}, d_{3} = \frac{x_{34}}{x_{12} x_{24} + x_{13} x_{34}}, d_{4} = \frac{1}{x_{12} x_{24} + x_{13} x_{34}}, d_{5} = \frac{1}{x_{12} x_{25}},\\ 
	&a_{12} = 1, a_{13} = 1, a_{14} = 1, a_{15} = 1,\\ 
	&a_{23} = \frac{x_{24}}{x_{12} x_{24} + x_{13} x_{34}}, a_{24} = 1, a_{25} = 1,\\ 
	&a_{34} = -\frac{x_{12} - 1}{x_{13}}, a_{35} = -\frac{x_{15} + {\left(x_{12}^{2} - x_{12}\right)} x_{25}}{x_{12} x_{13} x_{25}},\\ 
	&a_{45} = 0
\end{align*}

Now assume $x_{24}= \frac{-x_{13}x_{34}}{x_{12}}$
\begin{equation*}x=\left(\begin{matrix}
		1 & x_{12} & x_{13} & 0 & x_{15} \\
		0 & 1 & 0 & x_{24} & x_{25} \\
		0 & 0 & 1 & x_{34} & 0  \\
		0 & 0 & 0 & 1 & 0  \\
		0 & 0 & 0 & 0 & 1 \\
	\end{matrix}\right),x^A=\left(\begin{matrix}
		1 & 1 & 0 & 0 & 0 \\
		0 & 1 & 0 & 0 & 1 \\
		0 & 0 & 1 & 1 & 0 \\
		0 & 0 & 0 & 1 & 0 \\
		0 & 0 & 0 & 0 & 1
	\end{matrix}\right)
\end{equation*} \newline
Where matrix $A$ has entries
\begin{align*}
	&d_{1} = 1, d_{2} = \frac{1}{x_{12}}, d_{3} = 1, d_{4} = \frac{1}{x_{34}}, d_{5} = \frac{1}{x_{12} x_{25}},\\ 
	&a_{12} = 1, a_{13} = 1, a_{14} = 1, a_{15} = 1,\\ 
	&a_{23} = -\frac{x_{13}}{x_{12}}, a_{24} = 1, a_{25} = 1,\\ 
	&a_{34} = -\frac{x_{12} - 1}{x_{13}}, a_{35} = -\frac{x_{15} + {\left(x_{12}^{2} - x_{12}\right)} x_{25}}{x_{12} x_{13} x_{25}},\\ 
	&a_{45} = 0
\end{align*}


First assume $x_{24}\neq \frac{-x_{13}x_{34}}{x_{12}}$
\begin{equation*}x=\left(\begin{matrix}
		1 & x_{12} & x_{13} & x_{14} & 0 \\
		0 & 1 & 0 & x_{24} & x_{25} \\
		0 & 0 & 1 & x_{34} & 0  \\
		0 & 0 & 0 & 1 & 0  \\
		0 & 0 & 0 & 0 & 1 \\
	\end{matrix}\right),x^A=\left(\begin{matrix}
		1 & 1 & 1 & 0 & 0 \\
		0 & 1 & 0 & 0 & 1 \\
		0 & 0 & 1 & 1 & 0 \\
		0 & 0 & 0 & 1 & 0 \\
		0 & 0 & 0 & 0 & 1
	\end{matrix}\right)
\end{equation*} \newline
Where matrix $A$ has entries
\begin{align*}
	&d_{1} = 1, d_{2} = \frac{1}{x_{12}}, d_{3} = \frac{x_{34}}{x_{12} x_{24} + x_{13} x_{34}}, d_{4} = \frac{1}{x_{12} x_{24} + x_{13} x_{34}}, d_{5} = \frac{1}{x_{12} x_{25}},\\ 
	&a_{12} = 1, a_{13} = 1, a_{14} = 1, a_{15} = 1,\\ 
	&a_{23} = \frac{x_{24}}{x_{12} x_{24} + x_{13} x_{34}}, a_{24} = 1, a_{25} = 1,\\ 
	&a_{34} = -\frac{{\left(x_{12} - 1\right)} x_{13} x_{34} + x_{14} + {\left(x_{12}^{2} - x_{12}\right)} x_{24}}{x_{12} x_{13} x_{24} + x_{13}^{2} x_{34}}, a_{35} = -\frac{x_{12} - 1}{x_{13}},\\ 
	&a_{45} = 0
\end{align*}

Now assume $x_{24}= \frac{-x_{13}x_{34}}{x_{12}}$
\begin{equation*}x=\left(\begin{matrix}
		1 & x_{12} & x_{13} & x_{14} & 0 \\
		0 & 1 & 0 & x_{24} & x_{25} \\
		0 & 0 & 1 & x_{34} & 0  \\
		0 & 0 & 0 & 0 & 0  \\
		0 & 0 & 0 & 0 & 0 \\
	\end{matrix}\right),x^A=\left(\begin{matrix}
		0 & 1 & 0 & 0 & 0 \\
		0 & 0 & 0 & 0 & 1 \\
		0 & 0 & 0 & 1 & 0 \\
		0 & 0 & 0 & 0 & 0 \\
		0 & 0 & 0 & 0 & 0
	\end{matrix}\right)
\end{equation*} \newline
Where matrix $A$ has entries
\begin{align*}
	&d_{1} = 1, d_{2} = \frac{1}{x_{12}}, d_{3} = 1, d_{4} = \frac{1}{x_{34}}, d_{5} = \frac{1}{x_{12} x_{25}},\\ 
	&a_{12} = 1, a_{13} = 1, a_{14} = 1, a_{15} = 1,\\ 
	&a_{23} = -\frac{x_{13}}{x_{12}}, a_{24} = 1, a_{25} = 1,\\ 
	&a_{34} = -\frac{x_{14} + {\left(x_{12} - 1\right)} x_{34}}{x_{13} x_{34}}, a_{35} = -\frac{x_{12} - 1}{x_{13}},\\ 
	&a_{45} = 0
\end{align*}


First assume $x_{24}\neq \frac{-x_{13}x_{34}}{x_{12}}$
\begin{equation*}x=\left(\begin{matrix}
		1 & x_{12} & x_{13} & x_{14} & x_{15} \\
		0 & 1 & 0 & x_{24} & x_{25} \\
		0 & 0 & 1 & x_{34} & 0  \\
		0 & 0 & 0 & 1 & 0  \\
		0 & 0 & 0 & 0 & 1 \\
	\end{matrix}\right),x^A=\left(\begin{matrix}
		1 & 1 & 1 & 0 & 0 \\
		0 & 1 & 0 & 0 & 1 \\
		0 & 0 & 1 & 1 & 0 \\
		0 & 0 & 0 & 1 & 0 \\
		0 & 0 & 0 & 0 & 1
	\end{matrix}\right)
\end{equation*} \newline
Where matrix $A$ has entries
\begin{align*}
	&d_{1} = 1, d_{2} = \frac{1}{x_{12}}, d_{3} = \frac{x_{34}}{x_{12} x_{24} + x_{13} x_{34}}, d_{4} = \frac{1}{x_{12} x_{24} + x_{13} x_{34}}, d_{5} = \frac{1}{x_{12} x_{25}},\\ 
	&a_{12} = 1, a_{13} = 1, a_{14} = 1, a_{15} = 1,\\ 
	&a_{23} = \frac{x_{24}}{x_{12} x_{24} + x_{13} x_{34}}, a_{24} = 1, a_{25} = 1,\\ 
	&a_{34} = -\frac{{\left(x_{12} - 1\right)} x_{13} x_{34} + x_{14} + {\left(x_{12}^{2} - x_{12}\right)} x_{24}}{x_{12} x_{13} x_{24} + x_{13}^{2} x_{34}}, a_{35} = -\frac{x_{15} + {\left(x_{12}^{2} - x_{12}\right)} x_{25}}{x_{12} x_{13} x_{25}},\\ 
	&a_{45} = 0
\end{align*}

Now assume $x_{24}= \frac{-x_{13}x_{34}}{x_{12}}$
\begin{equation*}x=\left(
\right)
\end{equation*} \newline
Where matrix $A$ has entries
\begin{align*}
	&d_{1} = 1, d_{2} = \frac{1}{x_{12}}, d_{3} = \frac{1}{x_{13}}, d_{4} = \frac{1}{x_{13} x_{34}}, d_{5} = \frac{1}{x_{12} x_{25}},\\ 
	&a_{12} = 1, a_{13} = 1, a_{14} = 1, a_{15} = 1,\\ 
	&a_{23} = 0, a_{24} = 1, a_{25} = 1,\\ 
	&a_{34} = -\frac{{\left(x_{12} - 1\right)} x_{13} x_{34} + x_{14}}{x_{13}^{2} x_{34}}, a_{35} = -\frac{{\left(x_{12}^{2} - x_{12}\right)} x_{25} x_{34} - x_{14} x_{35}}{x_{12} x_{13} x_{25} x_{34}},\\ 
	&a_{45} = -\frac{x_{35}}{x_{12} x_{25} x_{34}}
\end{align*}

\begin{equation*}x=\left(\begin{matrix}
		1 & x_{12} & x_{13} & x_{14} & x_{15} \\
		0 & 1 & 0 & 0 & x_{25} \\
		0 & 0 & 1 & x_{34} & x_{35}  \\
		0 & 0 & 0 & 1 & 0  \\
		0 & 0 & 0 & 0 & 1 \\
	\end{matrix}\right),x^A=\left(\begin{matrix}
		1 & 1 & 1 & 0 & 0 \\
		0 & 1 & 0 & 0 & 1 \\
		0 & 0 & 1 & 1 & 0 \\
		0 & 0 & 0 & 1 & 0 \\
		0 & 0 & 0 & 0 & 1
	\end{matrix}\right)
\end{equation*} \newline
Where matrix $A$ has entries
\begin{align*}
	&d_{1} = 1, d_{2} = \frac{1}{x_{12}}, d_{3} = \frac{1}{x_{13}}, d_{4} = \frac{1}{x_{13} x_{34}}, d_{5} = \frac{1}{x_{12} x_{25}},\\ 
	&a_{12} = 1, a_{13} = 1, a_{14} = 1, a_{15} = 1,\\ 
	&a_{23} = 0, a_{24} = 1, a_{25} = 1,\\ 
	&a_{34} = -\frac{{\left(x_{12} - 1\right)} x_{13} x_{34} + x_{14}}{x_{13}^{2} x_{34}}, a_{35} = \frac{x_{14} x_{35} - {\left(x_{15} + {\left(x_{12}^{2} - x_{12}\right)} x_{25}\right)} x_{34}}{x_{12} x_{13} x_{25} x_{34}},\\ 
	&a_{45} = -\frac{x_{35}}{x_{12} x_{25} x_{34}}
\end{align*}

\begin{equation*}x=\left(
\right)
\end{equation*} \newline
Where matrix $A$ has entries
\begin{align*}
	&d_{1} = 1, d_{2} = \frac{1}{x_{12}}, d_{3} = \frac{x_{34}}{x_{12} x_{24}}, d_{4} = \frac{1}{x_{12} x_{24}}, d_{5} = -\frac{x_{34}}{x_{12} x_{24} x_{35}},\\ 
	&a_{12} = 1, a_{13} = 1, a_{14} = 1, a_{15} = 1,\\ 
	&a_{23} = \frac{1}{x_{12}}, a_{24} = \frac{x_{12} x_{24} - x_{14}}{x_{12}^{2} x_{24}}, a_{25} = \frac{x_{15} x_{34} + {\left(x_{12} x_{24} - x_{14}\right)} x_{35}}{x_{12}^{2} x_{24} x_{35}},\\ 
	&a_{34} = 1, a_{35} = 1,\\ 
	&a_{45} = \frac{1}{x_{12} x_{24}}
\end{align*}


First assume $x_{24}\neq \frac{-x_{13}x_{34}}{x_{12}}$.
\begin{equation*}x=\left(\begin{matrix}
		1 & x_{12} & x_{13} & 0 & 0 \\
		0 & 1 & 0 & x_{24} & 0 \\
		0 & 0 & 1 & x_{34} & x_{35}  \\
		0 & 0 & 0 & 1 & 0  \\
		0 & 0 & 0 & 0 & 1 \\
	\end{matrix}\right),x^A=\left(\begin{matrix}
		1 & 1 & 1 & 0 & 0 \\
		0 & 1 & 0 & 0 & 1 \\
		0 & 0 & 1 & 1 & 0 \\
		0 & 0 & 0 & 1 & 0 \\
		0 & 0 & 0 & 0 & 1
	\end{matrix}\right)
\end{equation*} \newline
Where matrix $A$ has entries
\begin{align*}
	&d_{1} = 1, d_{2} = \frac{1}{x_{12}}, d_{3} = \frac{x_{34}}{x_{12} x_{24} + x_{13} x_{34}}, d_{4} = \frac{1}{x_{12} x_{24} + x_{13} x_{34}}, d_{5} = -\frac{x_{34}}{x_{12} x_{24} x_{35}},\\ 
	&a_{12} = 1, a_{13} = 1, a_{14} = 1, a_{15} = 1,\\ 
	&a_{23} = \frac{x_{24}}{x_{12} x_{24} + x_{13} x_{34}}, a_{24} = 1, a_{25} = 1,\\ 
	&a_{34} = -\frac{x_{12} - 1}{x_{13}}, a_{35} = -\frac{x_{12} - 1}{x_{13}},\\ 
	&a_{45} = \frac{1}{x_{12} x_{24}}
\end{align*}

Now assume $x_{24}= \frac{-x_{13}x_{34}}{x_{12}}$.
\begin{equation*}x=\left(\begin{matrix}
		1 & x_{12} & x_{13} & 0 & 0 \\
		0 & 1 & 0 & x_{24} & 0 \\
		0 & 0 & 1 & x_{34} & x_{35}  \\
		0 & 0 & 0 & 1 & 0  \\
		0 & 0 & 0 & 0 & 1 \\
	\end{matrix}\right),x^A=\left(\begin{matrix}
		1 & 1 & 0 & 0 & 0 \\
		0 & 1 & 0 & 0 & 1 \\
		0 & 0 & 1 & 1 & 0 \\
		0 & 0 & 0 & 1 & 0 \\
		0 & 0 & 0 & 0 & 1
	\end{matrix}\right)
\end{equation*} \newline
Where matrix $A$ has entries
\begin{align*}
	&d_{1} = 1, d_{2} = \frac{1}{x_{12}}, d_{3} = 1, d_{4} = \frac{1}{x_{34}}, d_{5} = \frac{1}{x_{13} x_{35}},\\ 
	&a_{12} = 1, a_{13} = 1, a_{14} = 1, a_{15} = 1,\\ 
	&a_{23} = -\frac{x_{13}}{x_{12}}, a_{24} = 1, a_{25} = 1,\\ 
	&a_{34} = -\frac{x_{12} - 1}{x_{13}}, a_{35} = -\frac{x_{12} - 1}{x_{13}},\\ 
	&a_{45} = -\frac{1}{x_{13} x_{34}}
\end{align*}


First assume $x_{24}\neq \frac{-x_{13}x_{34}}{x_{12}}$.
\begin{equation*}x=\left(\begin{matrix}
		1 & x_{12} & x_{13} & 0 & x_{15} \\
		0 & 1 & 0 & x_{24} & 0 \\
		0 & 0 & 1 & x_{34} & x_{35}  \\
		0 & 0 & 0 & 1 & 0  \\
		0 & 0 & 0 & 0 & 1 \\
	\end{matrix}\right),x^A=\left(\begin{matrix}
		1 & 1 & 1 & 0 & 0 \\
		0 & 1 & 0 & 0 & 1 \\
		0 & 0 & 1 & 1 & 0 \\
		0 & 0 & 0 & 1 & 0 \\
		0 & 0 & 0 & 0 & 1
	\end{matrix}\right)
\end{equation*} \newline
Where matrix $A$ has entries
\begin{align*}
	&d_{1} = 1, d_{2} = \frac{1}{x_{12}}, d_{3} = \frac{x_{34}}{x_{12} x_{24} + x_{13} x_{34}}, d_{4} = \frac{1}{x_{12} x_{24} + x_{13} x_{34}}, d_{5} = -\frac{x_{34}}{x_{12} x_{24} x_{35}},\\ 
	&a_{12} = 1, a_{13} = 1, a_{14} = 1, a_{15} = 1,\\ 
	&a_{23} = \frac{x_{24}}{x_{12} x_{24} + x_{13} x_{34}}, a_{24} = 1, a_{25} = 1,\\ 
	&a_{34} = -\frac{x_{12} - 1}{x_{13}}, a_{35} = \frac{x_{15} x_{34} - {\left(x_{12}^{2} - x_{12}\right)} x_{24} x_{35}}{x_{12} x_{13} x_{24} x_{35}},\\ 
	&a_{45} = \frac{1}{x_{12} x_{24}}
\end{align*}

Now assume $x_{24}= \frac{-x_{13}x_{34}}{x_{12}}$.
\begin{equation*}x=\left(\begin{matrix}
		1 & x_{12} & x_{13} & 0 & x_{15} \\
		0 & 1 & 0 & x_{24} & 0 \\
		0 & 0 & 1 & x_{34} & x_{35}  \\
		0 & 0 & 0 & 1 & 0  \\
		0 & 0 & 0 & 0 & 1 \\
	\end{matrix}\right),x^A=\left(\begin{matrix}
		1 & 1 & 0 & 0 & 0 \\
		0 & 1 & 0 & 0 & 1 \\
		0 & 0 & 1 & 1 & 0 \\
		0 & 0 & 0 & 1 & 0 \\
		0 & 0 & 0 & 0 & 1
	\end{matrix}\right)
\end{equation*} \newline
Where matrix $A$ has entries
\begin{align*}
	&d_{1} = 1, d_{2} = \frac{1}{x_{12}}, d_{3} = 1, d_{4} = \frac{1}{x_{34}}, d_{5} = \frac{1}{x_{13} x_{35}},\\ 
	&a_{12} = 1, a_{13} = 1, a_{14} = 1, a_{15} = 1,\\ 
	&a_{23} = -\frac{x_{13}}{x_{12}}, a_{24} = 1, a_{25} = 1,\\ 
	&a_{34} = -\frac{x_{12} - 1}{x_{13}}, a_{35} = -\frac{{\left(x_{12} - 1\right)} x_{13} x_{35} + x_{15}}{x_{13}^{2} x_{35}},\\ 
	&a_{45} = -\frac{1}{x_{13} x_{34}}
\end{align*}


First assume $x_{24}\neq \frac{-x_{13}x_{34}}{x_{12}}$.
\begin{equation*}x=\left(\begin{matrix}
		1 & x_{12} & x_{13} & x_{14} & 0 \\
		0 & 1 & 0 & x_{24} & 0 \\
		0 & 0 & 1 & x_{34} & x_{35}  \\
		0 & 0 & 0 & 1 & 0  \\
		0 & 0 & 0 & 0 & 1 \\
	\end{matrix}\right),x^A=\left(\begin{matrix}
		1 & 1 & 1 & 0 & 0 \\
		0 & 1 & 0 & 0 & 1 \\
		0 & 0 & 1 & 1 & 0 \\
		0 & 0 & 0 & 1 & 0 \\
		0 & 0 & 0 & 0 & 1
	\end{matrix}\right)
\end{equation*} \newline
Where matrix $A$ has entries
\begin{align*}
	&d_{1} = 1, d_{2} = \frac{1}{x_{12}}, d_{3} = \frac{x_{34}}{x_{12} x_{24} + x_{13} x_{34}}, d_{4} = \frac{1}{x_{12} x_{24} + x_{13} x_{34}}, d_{5} = -\frac{x_{34}}{x_{12} x_{24} x_{35}},\\ 
	&a_{12} = 1, a_{13} = 1, a_{14} = 1, a_{15} = 1,\\ 
	&a_{23} = \frac{x_{24}}{x_{12} x_{24} + x_{13} x_{34}}, a_{24} = 1, a_{25} = 1,\\ 
	&a_{34} = -\frac{{\left(x_{12} - 1\right)} x_{13} x_{34} + x_{14} + {\left(x_{12}^{2} - x_{12}\right)} x_{24}}{x_{12} x_{13} x_{24} + x_{13}^{2} x_{34}}, a_{35} = -\frac{x_{14} + {\left(x_{12}^{2} - x_{12}\right)} x_{24}}{x_{12} x_{13} x_{24}},\\ 
	&a_{45} = \frac{1}{x_{12} x_{24}}
\end{align*}

Now assume $x_{24}= \frac{-x_{13}x_{34}}{x_{12}}$.
\begin{equation*}x=\left(\begin{matrix}
		1 & x_{12} & x_{13} & x_{14} & 0 \\
		0 & 1 & 0 & x_{24} & 0 \\
		0 & 0 & 1 & x_{34} & x_{35}  \\
		0 & 0 & 0 & 1 & 0  \\
		0 & 0 & 0 & 0 & 1 \\
	\end{matrix}\right),x^A=\left(\begin{matrix}
		1 & 1 & 0 & 0 & 0 \\
		0 & 1 & 0 & 0 & 1 \\
		0 & 0 & 1 & 1 & 0 \\
		0 & 0 & 0 & 1 & 0 \\
		0 & 0 & 0 & 0 & 1
	\end{matrix}\right)
\end{equation*} \newline
Where matrix $A$ has entries
\begin{align*}
	&d_{1} = 1, d_{2} = \frac{1}{x_{12}}, d_{3} = 1, d_{4} = \frac{1}{x_{34}}, d_{5} = \frac{1}{x_{13} x_{35}},\\ 
	&a_{12} = 1, a_{13} = 1, a_{14} = 1, a_{15} = 1,\\ 
	&a_{23} = -\frac{x_{13}}{x_{12}}, a_{24} = 1, a_{25} = 1,\\ 
	&a_{34} = -\frac{x_{14} + {\left(x_{12} - 1\right)} x_{34}}{x_{13} x_{34}}, a_{35} = -\frac{{\left(x_{12} - 1\right)} x_{13} x_{34} - x_{14}}{x_{13}^{2} x_{34}},\\ 
	&a_{45} = -\frac{1}{x_{13} x_{34}}
\end{align*}


First assume $x_{24}\neq \frac{-x_{13}x_{34}}{x_{12}}$.
\begin{equation*}x=\left(\begin{matrix}
		1 & x_{12} & x_{13} & x_{14} & x_{15} \\
		0 & 1 & 0 & x_{24} & 0 \\
		0 & 0 & 1 & x_{34} & x_{35}  \\
		0 & 0 & 0 & 1 & 0  \\
		0 & 0 & 0 & 0 & 1 \\
	\end{matrix}\right),x^A=\left(\begin{matrix}
		1 & 1 & 1 & 0 & 0 \\
		0 & 1 & 0 & 0 & 1 \\
		0 & 0 & 1 & 1 & 0 \\
		0 & 0 & 0 & 1 & 0 \\
		0 & 0 & 0 & 0 & 1
	\end{matrix}\right)
\end{equation*} \newline
Where matrix $A$ has entries
\begin{align*}
	&d_{1} = 1, d_{2} = \frac{1}{x_{12}}, d_{3} = \frac{x_{34}}{x_{12} x_{24} + x_{13} x_{34}}, d_{4} = \frac{1}{x_{12} x_{24} + x_{13} x_{34}}, d_{5} = -\frac{x_{34}}{x_{12} x_{24} x_{35}},\\ 
	&a_{12} = 1, a_{13} = 1, a_{14} = 1, a_{15} = 1,\\ 
	&a_{23} = \frac{x_{24}}{x_{12} x_{24} + x_{13} x_{34}}, a_{24} = 1, a_{25} = 1,\\ 
	&a_{34} = -\frac{{\left(x_{12} - 1\right)} x_{13} x_{34} + x_{14} + {\left(x_{12}^{2} - x_{12}\right)} x_{24}}{x_{12} x_{13} x_{24} + x_{13}^{2} x_{34}}, a_{35} = \frac{x_{15} x_{34} - {\left(x_{14} + {\left(x_{12}^{2} - x_{12}\right)} x_{24}\right)} x_{35}}{x_{12} x_{13} x_{24} x_{35}},\\ 
	&a_{45} = \frac{1}{x_{12} x_{24}}
\end{align*}

Now assume $x_{24}= \frac{-x_{13}x_{34}}{x_{12}}$.
\begin{equation*}x=\left(\begin{matrix}
		1 & x_{12} & x_{13} & x_{14} & x_{15} \\
		0 & 1 & 0 & x_{24} & 0 \\
		0 & 0 & 1 & x_{34} & x_{35}  \\
		0 & 0 & 0 & 1 & 0  \\
		0 & 0 & 0 & 0 & 1 \\
	\end{matrix}\right),x^A=\left(\begin{matrix}
		1 & 1 & 0 & 0 & 0 \\
		0 & 1 & 0 & 0 & 1 \\
		0 & 0 & 1 & 1 & 0 \\
		0 & 0 & 0 & 1 & 0 \\
		0 & 0 & 0 & 0 & 1
	\end{matrix}\right)
\end{equation*} \newline
Where matrix $A$ has entries
\begin{align*}
	&d_{1} = 1, d_{2} = \frac{1}{x_{12}}, d_{3} = 1, d_{4} = \frac{1}{x_{34}}, d_{5} = \frac{1}{x_{13} x_{35}},\\ 
	&a_{12} = 1, a_{13} = 1, a_{14} = 1, a_{15} = 1,\\ 
	&a_{23} = -\frac{x_{13}}{x_{12}}, a_{24} = 1, a_{25} = 1,\\ 
	&a_{34} = -\frac{x_{14} + {\left(x_{12} - 1\right)} x_{34}}{x_{13} x_{34}}, a_{35} = -\frac{x_{15} x_{34} + {\left({\left(x_{12} - 1\right)} x_{13} x_{34} - x_{14}\right)} x_{35}}{x_{13}^{2} x_{34} x_{35}},\\ 
	&a_{45} = -\frac{1}{x_{13} x_{34}}
\end{align*}


First assume $x_{25}\neq \frac{x_{24}x_{35}}{x_{34}}$.
\begin{equation*}x=\left(\begin{matrix}
		1 & x_{12} & 0 & 0 & 0 \\
		0 & 1 & 0 & x_{24} & x_{25} \\
		0 & 0 & 1 & x_{34} & x_{35}  \\
		0 & 0 & 0 & 1 & 0  \\
		0 & 0 & 0 & 0 & 1 \\
	\end{matrix}\right),x^A=\left(\begin{matrix}
		1 & 1 & 1 & 0 & 0 \\
		0 & 1 & 0 & 0 & 1 \\
		0 & 0 & 1 & 1 & 0 \\
		0 & 0 & 0 & 1 & 0 \\
		0 & 0 & 0 & 0 & 1
	\end{matrix}\right)
\end{equation*} \newline
Where matrix $A$ has entries
\begin{align*}
	&d_{1} = 1, d_{2} = \frac{1}{x_{12}}, d_{3} = \frac{x_{34}}{x_{12} x_{24}}, d_{4} = \frac{1}{x_{12} x_{24}}, d_{5} = \frac{x_{34}}{x_{12} x_{25} x_{34} - x_{12} x_{24} x_{35}},\\ 
	&a_{12} = 1, a_{13} = 1, a_{14} = 1, a_{15} = 1,\\ 
	&a_{23} = \frac{1}{x_{12}}, a_{24} = \frac{1}{x_{12}}, a_{25} = \frac{1}{x_{12}},\\ 
	&a_{34} = 1, a_{35} = 1,\\ 
	&a_{45} = -\frac{x_{35}}{x_{12} x_{25} x_{34} - x_{12} x_{24} x_{35}}
\end{align*}

Now assume $x_{25}= \frac{x_{24}x_{35}}{x_{34}}$.
\begin{equation*}x=\left(\begin{matrix}
		1 & x_{12} & 0 & 0 & 0 \\
		0 & 1 & 0 & x_{24} & x_{25} \\
		0 & 0 & 1 & x_{34} & x_{35}  \\
		0 & 0 & 0 & 1 & 0  \\
		0 & 0 & 0 & 0 & 1 \\
	\end{matrix}\right),x^A=\left(\begin{matrix}
		1 & 1 & 1 & 0 & 0 \\
		0 & 1 & 0 & 0 & 0 \\
		0 & 0 & 1 & 1 & 0 \\
		0 & 0 & 0 & 1 & 0 \\
		0 & 0 & 0 & 0 & 1
	\end{matrix}\right)
\end{equation*} \newline
Where matrix $A$ has entries
\begin{align*}
	&d_{1} = 1, d_{2} = \frac{1}{x_{12}}, d_{3} = \frac{x_{34}}{x_{12} x_{24}}, d_{4} = \frac{1}{x_{12} x_{24}}, d_{5} = 1,\\ 
	&a_{12} = 1, a_{13} = 1, a_{14} = 1, a_{15} = 1,\\ 
	&a_{23} = \frac{1}{x_{12}}, a_{24} = \frac{1}{x_{12}}, a_{25} = 0,\\ 
	&a_{34} = 1, a_{35} = 1,\\ 
	&a_{45} = -\frac{x_{35}}{x_{34}}
\end{align*}


First assume $x_{25}\neq \frac{x_{24}x_{35}}{x_{34}}$.
\begin{equation*}x=\left(\begin{matrix}
		1 & x_{12} & 0 & 0 & x_{15} \\
		0 & 1 & 0 & x_{24} & x_{25} \\
		0 & 0 & 1 & x_{34} & x_{35}  \\
		0 & 0 & 0 & 1 & 0  \\
		0 & 0 & 0 & 0 & 1 \\
	\end{matrix}\right),x^A=\left(\begin{matrix}
		1 & 1 & 1 & 0 & 0 \\
		0 & 1 & 0 & 0 & 1 \\
		0 & 0 & 1 & 1 & 0 \\
		0 & 0 & 0 & 1 & 0 \\
		0 & 0 & 0 & 0 & 1
	\end{matrix}\right)
\end{equation*} \newline
Where matrix $A$ has entries
\begin{align*}
	&d_{1} = 1, d_{2} = \frac{1}{x_{12}}, d_{3} = \frac{x_{34}}{x_{12} x_{24}}, d_{4} = \frac{1}{x_{12} x_{24}}, d_{5} = \frac{x_{34}}{x_{12} x_{25} x_{34} - x_{12} x_{24} x_{35}},\\ 
	&a_{12} = 1, a_{13} = 1, a_{14} = 1, a_{15} = 1,\\ 
	&a_{23} = \frac{1}{x_{12}}, a_{24} = \frac{1}{x_{12}}, a_{25} = -\frac{x_{12} x_{24} x_{35} - {\left(x_{12} x_{25} - x_{15}\right)} x_{34}}{x_{12}^{2} x_{25} x_{34} - x_{12}^{2} x_{24} x_{35}},\\ 
	&a_{34} = 1, a_{35} = 1,\\ 
	&a_{45} = -\frac{x_{35}}{x_{12} x_{25} x_{34} - x_{12} x_{24} x_{35}}
\end{align*}

Now assume $x_{25}= \frac{x_{24}x_{35}}{x_{34}}$.
\begin{equation*}x=\left(\begin{matrix}
		1 & x_{12} & 0 & 0 & x_{15} \\
		0 & 1 & 0 & x_{24} & x_{25} \\
		0 & 0 & 1 & x_{34} & x_{35}  \\
		0 & 0 & 0 & 1 & 0  \\
		0 & 0 & 0 & 0 & 1 \\
	\end{matrix}\right),x^A=\left(\begin{matrix}
		1 & 1 & 1 & 0 & 0 \\
		0 & 1 & 0 & 0 & 0 \\
		0 & 0 & 1 & 1 & 0 \\
		0 & 0 & 0 & 1 & 0 \\
		0 & 0 & 0 & 0 & 1
	\end{matrix}\right)
\end{equation*} \newline
Where matrix $A$ has entries
\begin{align*}
	&d_{1} = 1, d_{2} = \frac{1}{x_{12}}, d_{3} = \frac{x_{34}}{x_{12} x_{24}}, d_{4} = \frac{1}{x_{12} x_{24}}, d_{5} = 1,\\ 
	&a_{12} = 1, a_{13} = 1, a_{14} = 1, a_{15} = 1,\\ 
	&a_{23} = \frac{1}{x_{12}}, a_{24} = \frac{1}{x_{12}}, a_{25} = -\frac{x_{15}}{x_{12}},\\ 
	&a_{34} = 1, a_{35} = 1,\\ 
	&a_{45} = -\frac{x_{35}}{x_{34}}
\end{align*}


First assume $x_{25}\neq \frac{x_{24}x_{35}}{x_{34}}$.
\begin{equation*}x=\left(\begin{matrix}
		1 & x_{12} & 0 & x_{14} & 0 \\
		0 & 1 & 0 & x_{24} & x_{25} \\
		0 & 0 & 1 & x_{34} & x_{35}  \\
		0 & 0 & 0 & 1 & 0  \\
		0 & 0 & 0 & 0 & 1 \\
	\end{matrix}\right),x^A=\left(\begin{matrix}
		1 & 1 & 1 & 0 & 0 \\
		0 & 1 & 0 & 0 & 1 \\
		0 & 0 & 1 & 1 & 0 \\
		0 & 0 & 0 & 1 & 0 \\
		0 & 0 & 0 & 0 & 1
	\end{matrix}\right)
\end{equation*} \newline
Where matrix $A$ has entries
\begin{align*}
	&d_{1} = 1, d_{2} = \frac{1}{x_{12}}, d_{3} = \frac{x_{34}}{x_{12} x_{24}}, d_{4} = \frac{1}{x_{12} x_{24}}, d_{5} = \frac{x_{34}}{x_{12} x_{25} x_{34} - x_{12} x_{24} x_{35}},\\ 
	&a_{12} = 1, a_{13} = 1, a_{14} = 1, a_{15} = 1,\\ 
	&a_{23} = \frac{1}{x_{12}}, a_{24} = \frac{x_{12} x_{24} - x_{14}}{x_{12}^{2} x_{24}}, a_{25} = \frac{x_{12} x_{25} x_{34} - {\left(x_{12} x_{24} - x_{14}\right)} x_{35}}{x_{12}^{2} x_{25} x_{34} - x_{12}^{2} x_{24} x_{35}},\\ 
	&a_{34} = 1, a_{35} = 1,\\ 
	&a_{45} = -\frac{x_{35}}{x_{12} x_{25} x_{34} - x_{12} x_{24} x_{35}}
\end{align*}

Now assume $x_{25}= \frac{x_{24}x_{35}}{x_{34}}$.
\begin{equation*}x=\left(\begin{matrix}
		1 & x_{12} & 0 & x_{14} & 0 \\
		0 & 1 & 0 & x_{24} & x_{25} \\
		0 & 0 & 1 & x_{34} & x_{35}  \\
		0 & 0 & 0 & 1 & 0  \\
		0 & 0 & 0 & 0 & 1 \\
	\end{matrix}\right),x^A=\left(\begin{matrix}
		1 & 1 & 1 & 0 & 0 \\
		0 & 1 & 0 & 0 & 0 \\
		0 & 0 & 1 & 1 & 0 \\
		0 & 0 & 0 & 1 & 0 \\
		0 & 0 & 0 & 0 & 1
	\end{matrix}\right)
\end{equation*} \newline
Where matrix $A$ has entries
\begin{align*}
	&d_{1} = 1, d_{2} = \frac{1}{x_{12}}, d_{3} = \frac{x_{34}}{x_{12} x_{24}}, d_{4} = \frac{1}{x_{12} x_{24}}, d_{5} = 1,\\ 
	&a_{12} = 1, a_{13} = 1, a_{14} = 1, a_{15} = 1,\\ 
	&a_{23} = \frac{1}{x_{12}}, a_{24} = \frac{x_{12} x_{24} - x_{14}}{x_{12}^{2} x_{24}}, a_{25} = \frac{x_{14} x_{35}}{x_{12} x_{34}},\\ 
	&a_{34} = 1, a_{35} = 1,\\ 
	&a_{45} = -\frac{x_{35}}{x_{34}}
\end{align*}


First assume $x_{25}\neq \frac{x_{24}x_{35}}{x_{34}}$.
\begin{equation*}x=\left(\begin{matrix}
		1 & x_{12} & 0 & x_{14} & x_{15} \\
		0 & 1 & 0 & x_{24} & x_{25} \\
		0 & 0 & 1 & x_{34} & x_{35}  \\
		0 & 0 & 0 & 1 & 0  \\
		0 & 0 & 0 & 0 & 1 \\
	\end{matrix}\right),x^A=\left(\begin{matrix}
		1 & 1 & 1 & 0 & 0 \\
		0 & 1 & 0 & 0 & 1 \\
		0 & 0 & 1 & 1 & 0 \\
		0 & 0 & 0 & 1 & 0 \\
		0 & 0 & 0 & 0 & 1
	\end{matrix}\right)
\end{equation*} \newline
Where matrix $A$ has entries
\begin{align*}
	&d_{1} = 1, d_{2} = \frac{1}{x_{12}}, d_{3} = \frac{x_{34}}{x_{12} x_{24}}, d_{4} = \frac{1}{x_{12} x_{24}}, d_{5} = \frac{x_{34}}{x_{12} x_{25} x_{34} - x_{12} x_{24} x_{35}},\\ 
	&a_{12} = 1, a_{13} = 1, a_{14} = 1, a_{15} = 1,\\ 
	&a_{23} = \frac{1}{x_{12}}, a_{24} = \frac{x_{12} x_{24} - x_{14}}{x_{12}^{2} x_{24}}, a_{25} = \frac{{\left(x_{12} x_{25} - x_{15}\right)} x_{34} - {\left(x_{12} x_{24} - x_{14}\right)} x_{35}}{x_{12}^{2} x_{25} x_{34} - x_{12}^{2} x_{24} x_{35}},\\ 
	&a_{34} = 1, a_{35} = 1,\\ 
	&a_{45} = -\frac{x_{35}}{x_{12} x_{25} x_{34} - x_{12} x_{24} x_{35}}
\end{align*}

Now assume $x_{25}= \frac{x_{24}x_{35}}{x_{34}}$.
\begin{equation*}x=\left(\begin{matrix}
		1 & x_{12} & 0 & x_{14} & x_{15} \\
		0 & 1 & 0 & x_{24} & x_{25} \\
		0 & 0 & 1 & x_{34} & x_{35}  \\
		0 & 0 & 0 & 1 & 0  \\
		0 & 0 & 0 & 0 & 1 \\
	\end{matrix}\right),x^A=\left(\begin{matrix}
		1 & 1 & 1 & 0 & 0 \\
		0 & 1 & 0 & 0 & 0 \\
		0 & 0 & 1 & 1 & 0 \\
		0 & 0 & 0 & 1 & 0 \\
		0 & 0 & 0 & 0 & 1
	\end{matrix}\right)
\end{equation*} \newline
Where matrix $A$ has entries
\begin{align*}
	&d_{1} = 1, d_{2} = \frac{1}{x_{12}}, d_{3} = \frac{x_{34}}{x_{12} x_{24}}, d_{4} = \frac{1}{x_{12} x_{24}}, d_{5} = 1,\\ 
	&a_{12} = 1, a_{13} = 1, a_{14} = 1, a_{15} = 1,\\ 
	&a_{23} = \frac{1}{x_{12}}, a_{24} = \frac{x_{12} x_{24} - x_{14}}{x_{12}^{2} x_{24}}, a_{25} = -\frac{x_{15} x_{34} - x_{14} x_{35}}{x_{12} x_{34}},\\ 
	&a_{34} = 1, a_{35} = 1,\\ 
	&a_{45} = -\frac{x_{35}}{x_{34}}
\end{align*}


First assume $x_{24}\neq \frac{-x_{13}x_{34}}{x_{12}}$ and $x_{24}\neq \frac{-x_{25}x_{34}}{x_{35}}$.
\begin{equation*}x=\left(\begin{matrix}
		1 & x_{12} & x_{13} & 0 & 0 \\
		0 & 1 & 0 & x_{24} & x_{25} \\
		0 & 0 & 1 & x_{34} & x_{35}  \\
		0 & 0 & 0 & 1 & 0  \\
		0 & 0 & 0 & 0 & 1 \\
	\end{matrix}\right),x^A=\left(\begin{matrix}
		1 & 1 & 1 & 0 & 0 \\
		0 & 1 & 0 & 0 & 1 \\
		0 & 0 & 1 & 1 & 0 \\
		0 & 0 & 0 & 1 & 0 \\
		0 & 0 & 0 & 0 & 1
	\end{matrix}\right)
\end{equation*} \newline
Where matrix $A$ has entries
\begin{align*}
	&d_{1} = 1, d_{2} = \frac{1}{x_{12}}, d_{3} = \frac{x_{34}}{x_{12} x_{24} + x_{13} x_{34}}, d_{4} = \frac{1}{x_{12} x_{24} + x_{13} x_{34}}, d_{5} = \frac{x_{34}}{x_{12} x_{25} x_{34} - x_{12} x_{24} x_{35}},\\ 
	&a_{12} = 1, a_{13} = 1, a_{14} = 1, a_{15} = 1,\\ 
	&a_{23} = \frac{x_{24}}{x_{12} x_{24} + x_{13} x_{34}}, a_{24} = 1, a_{25} = 1,\\ 
	&a_{34} = -\frac{x_{12} - 1}{x_{13}}, a_{35} = -\frac{x_{12} - 1}{x_{13}},\\ 
	&a_{45} = -\frac{x_{35}}{x_{12} x_{25} x_{34} - x_{12} x_{24} x_{35}}
\end{align*}

Now assume $x_{24}= \frac{-x_{13}x_{34}}{x_{12}}$ , $x_{24}\neq \frac{-x_{25}x_{34}}{x_{35}}$ and $x_{25}\neq \frac{-x_{13}x_{35}}{x_{12}}$.
\begin{equation*}x=\left(\begin{matrix}
		1 & x_{12} & x_{13} & 0 & 0 \\
		0 & 1 & 0 & x_{24} & x_{25} \\
		0 & 0 & 1 & x_{34} & x_{35}  \\
		0 & 0 & 0 & 1 & 0  \\
		0 & 0 & 0 & 0 & 1 \\
	\end{matrix}\right),x^A=\left(\begin{matrix}
		1 & 1 & 0 & 0 & 0 \\
		0 & 1 & 0 & 0 & 1 \\
		0 & 0 & 1 & 1 & 0 \\
		0 & 0 & 0 & 1 & 0 \\
		0 & 0 & 0 & 0 & 1
	\end{matrix}\right)
\end{equation*} \newline
Where matrix $A$ has entries
\begin{align*}
	&d_{1} = 1, d_{2} = \frac{1}{x_{12}}, d_{3} = 1, d_{4} = \frac{1}{x_{34}}, d_{5} = \frac{1}{x_{12} x_{25} + x_{13} x_{35}},\\ 
	&a_{12} = 1, a_{13} = 1, a_{14} = 1, a_{15} = 1,\\ 
	&a_{23} = -\frac{x_{13}}{x_{12}}, a_{24} = 1, a_{25} = 1,\\ 
	&a_{34} = -\frac{x_{12} - 1}{x_{13}}, a_{35} = -\frac{x_{12} - 1}{x_{13}},\\ 
	&a_{45} = -\frac{x_{35}}{x_{12} x_{25} x_{34} + x_{13} x_{34} x_{35}}
\end{align*}

Now assume $x_{24}= \frac{-x_{13}x_{34}}{x_{12}}$, $x_{24}\neq \frac{-x_{25}x_{34}}{x_{35}}$ and $x_{25}= \frac{-x_{13}x_{35}}{x_{12}}$.
\begin{equation*}x=\left(\begin{matrix}
		1 & x_{12} & x_{13} & 0 & 0 \\
		0 & 1 & 0 & x_{24} & x_{25} \\
		0 & 0 & 1 & x_{34} & x_{35}  \\
		0 & 0 & 0 & 1 & 0  \\
		0 & 0 & 0 & 0 & 1 \\
	\end{matrix}\right),x^A=\left(\begin{matrix}
		1 & 1 & 0 & 0 & 0 \\
		0 & 1 & 0 & 0 & 0 \\
		0 & 0 & 1 & 1 & 0 \\
		0 & 0 & 0 & 1 & 0 \\
		0 & 0 & 0 & 0 & 1
	\end{matrix}\right)
\end{equation*} \newline
Where matrix $A$ has entries
\begin{align*}
	&d_{1} = 1, d_{2} = \frac{1}{x_{12}}, d_{3} = 1, d_{4} = \frac{1}{x_{34}}, d_{5} = 1,\\ 
	&a_{12} = 1, a_{13} = 1, a_{14} = 1, a_{15} = 1,\\ 
	&a_{23} = -\frac{x_{13}}{x_{12}}, a_{24} = 1, a_{25} = 1,\\ 
	&a_{34} = -\frac{x_{12} - 1}{x_{13}}, a_{35} = -\frac{x_{12}}{x_{13}},\\ 
	&a_{45} = -\frac{x_{35}}{x_{34}}
\end{align*}

Now assume $x_{24}\neq \frac{-x_{13}x_{34}}{x_{12}}$, $x_{24}= \frac{-x_{25}x_{34}}{x_{35}}$ and $x_{25}\neq \frac{-x_{13}x_{35}}{x_{12}}$.
\begin{equation*}x=\left(\begin{matrix}
		1 & x_{12} & x_{13} & 0 & 0 \\
		0 & 1 & 0 & x_{24} & x_{25} \\
		0 & 0 & 1 & x_{34} & x_{35}  \\
		0 & 0 & 0 & 1 & 0  \\
		0 & 0 & 0 & 0 & 1 \\
	\end{matrix}\right),x^A=\left(\begin{matrix}
		1 & 1 & 1 & 0 & 0 \\
		0 & 1 & 0 & 0 & 0 \\
		0 & 0 & 1 & 1 & 0 \\
		0 & 0 & 0 & 1 & 0 \\
		0 & 0 & 0 & 0 & 1 
	\end{matrix}\right)
\end{equation*} \newline
Where matrix $A$ has entries
\begin{align*}
	&d_{1} = 1, d_{2} = \frac{1}{x_{12}}, d_{3} = \frac{x_{35}}{x_{12} x_{25} + x_{13} x_{35}}, d_{4} = \frac{x_{35}}{x_{12} x_{25} x_{34} + x_{13} x_{34} x_{35}}, d_{5} = 1,\\ 
	&a_{12} = 1, a_{13} = 1, a_{14} = 1, a_{15} = 1,\\ 
	&a_{23} = \frac{x_{25}}{x_{12} x_{25} + x_{13} x_{35}}, a_{24} = 1, a_{25} = 1,\\ 
	&a_{34} = -\frac{x_{12} - 1}{x_{13}}, a_{35} = -\frac{x_{12}}{x_{13}},\\ 
	&a_{45} = -\frac{x_{35}}{x_{34}}
\end{align*}

Now assume $x_{24}\neq \frac{-x_{13}x_{34}}{x_{12}}$, $x_{24}= \frac{-x_{25}x_{34}}{x_{35}}$ and $x_{25}= \frac{-x_{13}x_{35}}{x_{12}}$.
\begin{equation*}x=\left(
\right)
\end{equation*} \newline
Where matrix $A$ has entries
\begin{align*}
	&d_{1} = 1, d_{2} = \frac{1}{x_{12}}, d_{3} = \frac{x_{34}}{x_{12} x_{24} + x_{13} x_{34}}, d_{4} = \frac{1}{x_{12} x_{24} + x_{13} x_{34}}, d_{5} = \frac{x_{34}}{x_{12} x_{25} x_{34} - x_{12} x_{24} x_{35}},\\ 
	&a_{12} = 1, a_{13} = 1, a_{14} = 1, a_{15} = 1,\\ 
	&a_{23} = \frac{x_{24}}{x_{12} x_{24} + x_{13} x_{34}}, a_{24} = 1, a_{25} = 1,\\ 
	&a_{34} = -\frac{x_{12} - 1}{x_{13}}, a_{35} = \frac{{\left(x_{12}^{2} - x_{12}\right)} x_{24} x_{35} - {\left(x_{15} + {\left(x_{12}^{2} - x_{12}\right)} x_{25}\right)} x_{34}}{x_{12} x_{13} x_{25} x_{34} - x_{12} x_{13} x_{24} x_{35}},\\ 
	&a_{45} = -\frac{x_{35}}{x_{12} x_{25} x_{34} - x_{12} x_{24} x_{35}}
\end{align*}

Now assume $x_{24}= \frac{-x_{13}x_{34}}{x_{12}}$,$x_{24}\neq \frac{-x_{25}x_{34}}{x_{35}}$ and $x_{25}\neq \frac{-x_{13}x_{35}}{x_{12}}$.
\begin{equation*}x=\left(\begin{matrix}
		1 & x_{12} & x_{13} & 0 & x_{15} \\
		0 & 1 & 0 & x_{24} & x_{25} \\
		0 & 0 & 1 & x_{34} & x_{35}  \\
		0 & 0 & 0 & 1 & 0  \\
		0 & 0 & 0 & 0 & 1 \\
	\end{matrix}\right),x^A=\left(\begin{matrix}
		1 & 1 & 0 & 0 & 0 \\
		0 & 1 & 0 & 0 & 1 \\
		0 & 0 & 1 & 1 & 0 \\
		0 & 0 & 0 & 1 & 0 \\
		0 & 0 & 0 & 0 & 1
	\end{matrix}\right)
\end{equation*} \newline
Where matrix $A$ has entries
\begin{align*}
	&d_{1} = 1, d_{2} = \frac{1}{x_{12}}, d_{3} = 1, d_{4} = \frac{1}{x_{34}}, d_{5} = \frac{1}{x_{12} x_{25} + x_{13} x_{35}},\\ 
	&a_{12} = 1, a_{13} = 1, a_{14} = 1, a_{15} = 1,\\ 
	&a_{23} = -\frac{x_{13}}{x_{12}}, a_{24} = 1, a_{25} = 1,\\ 
	&a_{34} = -\frac{x_{12} - 1}{x_{13}}, a_{35} = -\frac{{\left(x_{12} - 1\right)} x_{13} x_{35} + x_{15} + {\left(x_{12}^{2} - x_{12}\right)} x_{25}}{x_{12} x_{13} x_{25} + x_{13}^{2} x_{35}},\\ 
	&a_{45} = -\frac{x_{35}}{x_{12} x_{25} x_{34} + x_{13} x_{34} x_{35}}
\end{align*}

Now assume $x_{24}= \frac{-x_{13}x_{34}}{x_{12}}$, $x_{24}\neq \frac{-x_{25}x_{34}}{x_{35}}$ and $x_{25}= \frac{-x_{13}x_{35}}{x_{12}}$.
\begin{equation*}x=\left(\begin{matrix}
		1 & x_{12} & x_{13} & 0 & x_{15} \\
		0 & 1 & 0 & x_{24} & x_{25} \\
		0 & 0 & 1 & x_{34} & x_{35}  \\
		0 & 0 & 0 & 1 & 0  \\
		0 & 0 & 0 & 0 & 1 \\
	\end{matrix}\right),x^A=\left(\begin{matrix}
		1 & 1 & 0 & 0 & 0 \\
		0 & 1 & 0 & 0 & 0 \\
		0 & 0 & 1 & 1 & 0 \\
		0 & 0 & 0 & 1 & 0 \\
		0 & 0 & 0 & 0 & 1
	\end{matrix}\right)
\end{equation*} \newline
Where matrix $A$ has entries
\begin{align*}
	&d_{1} = 1, d_{2} = \frac{1}{x_{12}}, d_{3} = 1, d_{4} = \frac{1}{x_{34}}, d_{5} = 1,\\ 
	&a_{12} = 1, a_{13} = 1, a_{14} = 1, a_{15} = 1,\\ 
	&a_{23} = -\frac{x_{13}}{x_{12}}, a_{24} = 1, a_{25} = 1,\\ 
	&a_{34} = -\frac{x_{12} - 1}{x_{13}}, a_{35} = -\frac{x_{12} + x_{15}}{x_{13}},\\ 
	&a_{45} = -\frac{x_{35}}{x_{34}}
\end{align*}

Now assume $x_{24}\neq \frac{-x_{13}x_{34}}{x_{12}}$, $x_{24}= \frac{-x_{25}x_{34}}{x_{35}}$ and $x_{25}\neq \frac{-x_{13}x_{35}}{x_{12}}$.
\begin{equation*}x=\left(\begin{matrix}
		1 & x_{12} & x_{13} & 0 & x_{15} \\
		0 & 1 & 0 & x_{24} & x_{25} \\
		0 & 0 & 1 & x_{34} & x_{35}  \\
		0 & 0 & 0 & 1 & 0  \\
		0 & 0 & 0 & 0 & 1 \\
	\end{matrix}\right),x^A=\left(\begin{matrix}
		1 & 1 & 1 & 0 & 0 \\
		0 & 1 & 0 & 0 & 0 \\
		0 & 0 & 1 & 1 & 0 \\
		0 & 0 & 0 & 1 & 0 \\
		0 & 0 & 0 & 0 & 1
	\end{matrix}\right)
\end{equation*} \newline
Where matrix $A$ has entries
\begin{align*}
	&d_{1} = 1, d_{2} = \frac{1}{x_{12}}, d_{3} = \frac{x_{35}}{x_{12} x_{25} + x_{13} x_{35}}, d_{4} = \frac{x_{35}}{x_{12} x_{25} x_{34} + x_{13} x_{34} x_{35}}, d_{5} = 1,\\ 
	&a_{12} = 1, a_{13} = 1, a_{14} = 1, a_{15} = 1,\\ 
	&a_{23} = \frac{x_{25}}{x_{12} x_{25} + x_{13} x_{35}}, a_{24} = 1, a_{25} = 1,\\ 
	&a_{34} = -\frac{x_{12} - 1}{x_{13}}, a_{35} = -\frac{x_{12} + x_{15}}{x_{13}},\\ 
	&a_{45} = -\frac{x_{35}}{x_{34}}
\end{align*}

Now assume $x_{24}\neq \frac{-x_{13}x_{34}}{x_{12}}$, $x_{24}= \frac{-x_{25}x_{34}}{x_{35}}$ and $x_{25}= \frac{-x_{13}x_{35}}{x_{12}}$.
\begin{equation*}x=\left(\begin{matrix}
		1 & x_{12} & x_{13} & 0 & x_{15} \\
		0 & 1 & 0 & x_{24} & x_{25} \\
		0 & 0 & 1 & x_{34} & x_{35}  \\
		0 & 0 & 0 & 1 & 0  \\
		0 & 0 & 0 & 0 & 1 \\
	\end{matrix}\right),x^A=\left(\begin{matrix}
		1 & 1 & 0 & 0 & 0 \\
		0 & 1 & 0 & 0 & 0 \\
		0 & 0 & 1 & 1 & 0 \\
		0 & 0 & 0 & 1 & 0 \\
		0 & 0 & 0 & 0 & 1
	\end{matrix}\right)
\end{equation*} \newline
Where matrix $A$ has entries
\begin{align*}
	&d_{1} = 1, d_{2} = \frac{1}{x_{12}}, d_{3} = 1, d_{4} = \frac{1}{x_{34}}, d_{5} = 1,\\ 
	&a_{12} = 1, a_{13} = 1, a_{14} = 1, a_{15} = 1,\\ 
	&a_{23} = -\frac{x_{13}}{x_{12}}, a_{24} = 1, a_{25} = 1,\\ 
	&a_{34} = -\frac{x_{12} - 1}{x_{13}}, a_{35} = -\frac{x_{12} + x_{15}}{x_{13}},\\ 
	&a_{45} = -\frac{x_{35}}{x_{34}}
\end{align*}

Now assume $x_{24}= \frac{-x_{13}x_{34}}{x_{12}}$, $x_{24}= \frac{-x_{25}x_{34}}{x_{35}}$.
\begin{equation*}x=\left(\begin{matrix}
		1 & x_{12} & x_{13} & 0 & x_{15} \\
		0 & 1 & 0 & x_{24} & x_{25} \\
		0 & 0 & 1 & x_{34} & x_{35}  \\
		0 & 0 & 0 & 1 & 0  \\
		0 & 0 & 0 & 0 & 1 \\
	\end{matrix}\right),x^A=\left(\begin{matrix}
		1 & 1 & 0 & 0 & 0 \\
		0 & 1 & 0 & 0 & 0 \\
		0 & 0 & 1 & 1 & 0 \\
		0 & 0 & 0 & 1 & 0 \\
		0 & 0 & 0 & 0 & 1
	\end{matrix}\right)
\end{equation*} \newline
Where matrix $A$ has entries
\begin{align*}
	&d_{1} = 1, d_{2} = \frac{1}{x_{12}}, d_{3} = 1, d_{4} = \frac{1}{x_{34}}, d_{5} = 1,\\ 
	&a_{12} = 1, a_{13} = 1, a_{14} = 1, a_{15} = 1,\\ 
	&a_{23} = \frac{x_{25}}{x_{35}}, a_{24} = 1, a_{25} = 1,\\ 
	&a_{34} = \frac{{\left(x_{12} - 1\right)} x_{35}}{x_{12} x_{25}}, a_{35} = \frac{{\left(x_{12} + x_{15}\right)} x_{35}}{x_{12} x_{25}},\\ 
	&a_{45} = -\frac{x_{35}}{x_{34}}
\end{align*}

First assume $x_{24}\neq \frac{-x_{13}x_{34}}{x_{12}}$ and $x_{25}\neq \frac{x_{24}x_{35}}{x_{34}}$.
\begin{equation*}x=\left(\begin{matrix}
		1 & x_{12} & x_{13} & x_{14} & 0 \\
		0 & 1 & 0 & x_{24} & x_{25} \\
		0 & 0 & 1 & x_{34} & x_{35}  \\
		0 & 0 & 0 & 1 & 0  \\
		0 & 0 & 0 & 0 & 1 \\
	\end{matrix}\right),x^A=\left(\begin{matrix}
		1 & 1 & 1 & 0 & 0 \\
		0 & 1 & 0 & 0 & 1 \\
		0 & 0 & 1 & 1 & 0 \\
		0 & 0 & 0 & 1 & 0 \\
		0 & 0 & 0 & 0 & 1
	\end{matrix}\right)
\end{equation*} \newline
Where matrix $A$ has entries
\begin{align*}
	&d_{1} = 1, d_{2} = \frac{1}{x_{12}}, d_{3} = \frac{x_{34}}{x_{12} x_{24} + x_{13} x_{34}}, d_{4} = \frac{1}{x_{12} x_{24} + x_{13} x_{34}}, d_{5} = \frac{x_{34}}{x_{12} x_{25} x_{34} - x_{12} x_{24} x_{35}},\\ 
	&a_{12} = 1, a_{13} = 1, a_{14} = 1, a_{15} = 1,\\ 
	&a_{23} = \frac{x_{24}}{x_{12} x_{24} + x_{13} x_{34}}, a_{24} = 1, a_{25} = 1,\\ 
	&a_{34} = -\frac{{\left(x_{12} - 1\right)} x_{13} x_{34} + x_{14} + {\left(x_{12}^{2} - x_{12}\right)} x_{24}}{x_{12} x_{13} x_{24} + x_{13}^{2} x_{34}}, a_{35} = -\frac{{\left(x_{12}^{2} - x_{12}\right)} x_{25} x_{34} - {\left(x_{14} + {\left(x_{12}^{2} - x_{12}\right)} x_{24}\right)} x_{35}}{x_{12} x_{13} x_{25} x_{34} - x_{12} x_{13} x_{24} x_{35}},\\ 
	&a_{45} = -\frac{x_{35}}{x_{12} x_{25} x_{34} - x_{12} x_{24} x_{35}}
\end{align*}

Now assume $x_{24}= \frac{-x_{13}x_{34}}{x_{12}}$ and $x_{25}\neq \frac{-x_{13}x_{35}}{x_{12}}$.
\begin{equation*}x=\left(\begin{matrix}
		1 & x_{12} & x_{13} & x_{14} & 0 \\
		0 & 1 & 0 & x_{24} & x_{25} \\
		0 & 0 & 1 & x_{34} & x_{35}  \\
		0 & 0 & 0 & 1 & 0  \\
		0 & 0 & 0 & 0 & 1 \\
	\end{matrix}\right),x^A=\left(\begin{matrix}
		1 & 1 & 0 & 0 & 0 \\
		0 & 1 & 0 & 0 & 1 \\
		0 & 0 & 1 & 1 & 0 \\
		0 & 0 & 0 & 1 & 0 \\
		0 & 0 & 0 & 0 & 1
	\end{matrix}\right)
\end{equation*} \newline
Where matrix $A$ has entries
\begin{align*}
	&d_{1} = 1, d_{2} = \frac{1}{x_{12}}, d_{3} = 1, d_{4} = \frac{1}{x_{34}}, d_{5} = \frac{1}{x_{12} x_{25} + x_{13} x_{35}},\\ 
	&a_{12} = 1, a_{13} = 1, a_{14} = 1, a_{15} = 1,\\ 
	&a_{23} = -\frac{x_{13}}{x_{12}}, a_{24} = 1, a_{25} = 1,\\ 
	&a_{34} = -\frac{x_{14} + {\left(x_{12} - 1\right)} x_{34}}{x_{13} x_{34}}, a_{35} = -\frac{{\left(x_{12}^{2} - x_{12}\right)} x_{25} x_{34} + {\left({\left(x_{12} - 1\right)} x_{13} x_{34} - x_{14}\right)} x_{35}}{x_{12} x_{13} x_{25} x_{34} + x_{13}^{2} x_{34} x_{35}},\\ 
	&a_{45} = -\frac{x_{35}}{x_{12} x_{25} x_{34} + x_{13} x_{34} x_{35}}
\end{align*}

Now assume $x_{24}= \frac{-x_{13}x_{34}}{x_{12}}$ and $x_{25}= \frac{-x_{13}x_{35}}{x_{12}}$.
\begin{equation*}x=\left(\begin{matrix}
		1 & x_{12} & x_{13} & x_{14} & 0 \\
		0 & 1 & 0 & x_{24} & x_{25} \\
		0 & 0 & 1 & x_{34} & x_{35}  \\
		0 & 0 & 0 & 1 & 0  \\
		0 & 0 & 0 & 0 & 1 \\
	\end{matrix}\right),x^A=\left(\begin{matrix}
		1 & 1 & 0 & 0 & 0 \\
		0 & 1 & 0 & 0 & 0 \\
		0 & 0 & 1 & 1 & 0 \\
		0 & 0 & 0 & 1 & 0 \\
		0 & 0 & 0 & 0 & 1
	\end{matrix}\right)
\end{equation*} \newline
Where matrix $A$ has entries
\begin{align*}
	&d_{1} = 1, d_{2} = \frac{1}{x_{12}}, d_{3} = 1, d_{4} = \frac{1}{x_{34}}, d_{5} = 1,\\ 
	&a_{12} = 1, a_{13} = 1, a_{14} = 1, a_{15} = 1,\\ 
	&a_{23} = -\frac{x_{13}}{x_{12}}, a_{24} = 1, a_{25} = 1,\\ 
	&a_{34} = -\frac{x_{14} + {\left(x_{12} - 1\right)} x_{34}}{x_{13} x_{34}}, a_{35} = -\frac{x_{12} x_{34} - x_{14} x_{35}}{x_{13} x_{34}},\\ 
	&a_{45} = -\frac{x_{35}}{x_{34}}
\end{align*}

Now assume $x_{25}= \frac{x_{24}x_{35}}{x_{34}}$ and $x_{12}\neq \frac{-x_{13}x_{34}}{x_{24}}$.
\begin{equation*}x=\left(\begin{matrix}
		1 & x_{12} & x_{13} & x_{14} & x_{15} \\
		0 & 1 & 0 & x_{24} & x_{25} \\
		0 & 0 & 1 & x_{34} & x_{35}  \\
		0 & 0 & 0 & 1 & 0  \\
		0 & 0 & 0 & 0 & 1 \\
	\end{matrix}\right),x^A=\left(\begin{matrix}
		1 & 1 & 1 & 0 & 0 \\
		0 & 1 & 0 & 0 & 0 \\
		0 & 0 & 1 & 1 & 0 \\
		0 & 0 & 0 & 1 & 0 \\
		0 & 0 & 0 & 0 & 1
	\end{matrix}\right)
\end{equation*} \newline
Where matrix $A$ has entries
\begin{align*}
	&d_{1} = 1, d_{2} = \frac{1}{x_{12}}, d_{3} = \frac{x_{34}}{x_{12} x_{24} + x_{13} x_{34}}, d_{4} = \frac{1}{x_{12} x_{24} + x_{13} x_{34}}, d_{5} = 1,\\ 
	&a_{12} = 1, a_{13} = 1, a_{14} = 1, a_{15} = 1,\\ 
	&a_{23} = \frac{x_{24}}{x_{12} x_{24} + x_{13} x_{34}}, a_{24} = 1, a_{25} = 1,\\ 
	&a_{34} = -\frac{{\left(x_{12} - 1\right)} x_{13} x_{34} + x_{14} + {\left(x_{12}^{2} - x_{12}\right)} x_{24}}{x_{12} x_{13} x_{24} + x_{13}^{2} x_{34}}, a_{35} = -\frac{x_{12} x_{34} - x_{14} x_{35}}{x_{13} x_{34}},\\ 
	&a_{45} = -\frac{x_{35}}{x_{34}}
\end{align*}

Now assume $x_{25}= \frac{x_{24}x_{35}}{x_{34}}$ and $x_{12}= \frac{-x_{13}x_{34}}{x_{24}}$.
\begin{equation*}x=\left(\begin{matrix}
		1 & x_{12} & x_{13} & x_{14} & x_{15} \\
		0 & 1 & 0 & x_{24} & x_{25} \\
		0 & 0 & 1 & x_{34} & x_{35}  \\
		0 & 0 & 0 & 1 & 0  \\
		0 & 0 & 0 & 0 & 1 \\
	\end{matrix}\right),x^A=\left(\begin{matrix}
		1 & 1 & 0 & 0 & 0 \\
		0 & 1 & 0 & 0 & 0 \\
		0 & 0 & 1 & 1 & 0 \\
		0 & 0 & 0 & 1 & 0 \\
		0 & 0 & 0 & 0 & 1
	\end{matrix}\right)
\end{equation*} \newline
Where matrix $A$ has entries
\begin{align*}
	&d_{1} = 1, d_{2} = -\frac{x_{24}}{x_{13} x_{34}}, d_{3} = 1, d_{4} = \frac{1}{x_{34}}, d_{5} = 1,\\ 
	&a_{12} = 1, a_{13} = 1, a_{14} = 1, a_{15} = 1,\\ 
	&a_{23} = \frac{x_{24}}{x_{34}}, a_{24} = 1, a_{25} = 1,\\ 
	&a_{34} = \frac{x_{13} x_{34}^{2} - x_{14} x_{24} + x_{24} x_{34}}{x_{13} x_{24} x_{34}}, a_{35} = \frac{x_{13} x_{34}^{2} + x_{14} x_{24} x_{35}}{x_{13} x_{24} x_{34}},\\ 
	&a_{45} = -\frac{x_{35}}{x_{34}}
\end{align*}


First assume $x_{24}\neq \frac{-x_{13}x_{34}}{x_{12}}$ and $x_{24}\neq \frac{-x_{25}x_{34}}{x_{35}}$.
\begin{equation*}x=\left(\begin{matrix}
		1 & x_{12} & x_{13} & x_{14} & x_{15} \\
		0 & 1 & 0 & x_{24} & x_{25} \\
		0 & 0 & 1 & x_{34} & x_{35}  \\
		0 & 0 & 0 & 1 & 0  \\
		0 & 0 & 0 & 0 & 1 \\
	\end{matrix}\right),x^A=\left(\begin{matrix}
		1 & 1 & 1 & 0 & 0 \\
		0 & 1 & 0 & 0 & 1 \\
		0 & 0 & 1 & 1 & 0 \\
		0 & 0 & 0 & 1 & 0 \\
		0 & 0 & 0 & 0 & 1
	\end{matrix}\right)
\end{equation*} \newline
Where matrix $A$ has entries
\begin{align*}
	&d_{1} = 1, d_{2} = \frac{1}{x_{12}}, d_{3} = \frac{x_{34}}{x_{12} x_{24} + x_{13} x_{34}}, d_{4} = \frac{1}{x_{12} x_{24} + x_{13} x_{34}}, d_{5} = \frac{x_{34}}{x_{12} x_{25} x_{34} - x_{12} x_{24} x_{35}},\\ 
	&a_{12} = 1, a_{13} = 1, a_{14} = 1, a_{15} = 1,\\ 
	&a_{23} = \frac{x_{24}}{x_{12} x_{24} + x_{13} x_{34}}, a_{24} = 1, a_{25} = 1,\\ 
	&a_{34} = -\frac{{\left(x_{12} - 1\right)} x_{13} x_{34} + x_{14} + {\left(x_{12}^{2} - x_{12}\right)} x_{24}}{x_{12} x_{13} x_{24} + x_{13}^{2} x_{34}}, a_{35} = -\frac{{\left(x_{15} + {\left(x_{12}^{2} - x_{12}\right)} x_{25}\right)} x_{34} - {\left(x_{14} + {\left(x_{12}^{2} - x_{12}\right)} x_{24}\right)} x_{35}}{x_{12} x_{13} x_{25} x_{34} - x_{12} x_{13} x_{24} x_{35}},\\ 
	&a_{45} = -\frac{x_{35}}{x_{12} x_{25} x_{34} - x_{12} x_{24} x_{35}}
\end{align*}

Now assume $x_{24}= \frac{-x_{13}x_{34}}{x_{12}}$, $x_{24}\neq \frac{-x_{25}x_{34}}{x_{35}}$ and $x_{25}\neq \frac{-x_{13}x_{35}}{x_{12}}$.
\begin{equation*}x=\left(\begin{matrix}
		1 & x_{12} & x_{13} & x_{14} & x_{15} \\
		0 & 1 & 0 & x_{24} & x_{25} \\
		0 & 0 & 1 & x_{34} & x_{35}  \\
		0 & 0 & 0 & 1 & 0  \\
		0 & 0 & 0 & 0 & 1 \\
	\end{matrix}\right),x^A=\left(\begin{matrix}
		1 & 1 & 0 & 0 & 0 \\
		0 & 1 & 0 & 0 & 1 \\
		0 & 0 & 1 & 1 & 0 \\
		0 & 0 & 0 & 1 & 0 \\
		0 & 0 & 0 & 0 & 1
	\end{matrix}\right)
\end{equation*} \newline
Where matrix $A$ has entries
\begin{align*}
	&d_{1} = 1, d_{2} = \frac{1}{x_{12}}, d_{3} = 1, d_{4} = \frac{1}{x_{34}}, d_{5} = \frac{1}{x_{12} x_{25} + x_{13} x_{35}},\\ 
	&a_{12} = 1, a_{13} = 1, a_{14} = 1, a_{15} = 1,\\ 
	&a_{23} = -\frac{x_{13}}{x_{12}}, a_{24} = 1, a_{25} = 1,\\ 
	&a_{34} = -\frac{x_{14} + {\left(x_{12} - 1\right)} x_{34}}{x_{13} x_{34}}, a_{35} = -\frac{{\left(x_{15} + {\left(x_{12}^{2} - x_{12}\right)} x_{25}\right)} x_{34} + {\left({\left(x_{12} - 1\right)} x_{13} x_{34} - x_{14}\right)} x_{35}}{x_{12} x_{13} x_{25} x_{34} + x_{13}^{2} x_{34} x_{35}},\\ 
	&a_{45} = -\frac{x_{35}}{x_{12} x_{25} x_{34} + x_{13} x_{34} x_{35}}
\end{align*}

Now assume $x_{24}= \frac{-x_{13}x_{34}}{x_{12}}$, $x_{24}\neq \frac{-x_{25}x_{34}}{x_{35}}$ and $x_{25}= \frac{-x_{13}x_{35}}{x_{12}}$.
\begin{equation*}x=\left(\begin{matrix}
		1 & x_{12} & x_{13} & x_{14} & x_{15} \\
		0 & 1 & 0 & x_{24} & x_{25} \\
		0 & 0 & 1 & x_{34} & x_{35}  \\
		0 & 0 & 0 & 1 & 0  \\
		0 & 0 & 0 & 0 & 1 \\
	\end{matrix}\right),x^A=\left(\begin{matrix}
		1 & 1 & 0 & 0 & 0 \\
		0 & 1 & 0 & 0 & 0 \\
		0 & 0 & 1 & 1 & 0 \\
		0 & 0 & 0 & 1 & 0 \\
		0 & 0 & 0 & 0 & 1
	\end{matrix}\right)
\end{equation*} \newline
Where matrix $A$ has entries
\begin{align*}
	&d_{1} = 1, d_{2} = \frac{1}{x_{12}}, d_{3} = 1, d_{4} = \frac{1}{x_{34}}, d_{5} = 1,\\ 
	&a_{12} = 1, a_{13} = 1, a_{14} = 1, a_{15} = 1,\\ 
	&a_{23} = -\frac{x_{13}}{x_{12}}, a_{24} = 1, a_{25} = 1,\\ 
	&a_{34} = -\frac{x_{14} + {\left(x_{12} - 1\right)} x_{34}}{x_{13} x_{34}}, a_{35} = \frac{x_{14} x_{35} - {\left(x_{12} + x_{15}\right)} x_{34}}{x_{13} x_{34}},\\ 
	&a_{45} = -\frac{x_{35}}{x_{34}}
\end{align*}

Now assume $x_{24}\neq \frac{-x_{13}x_{34}}{x_{12}}$, $x_{24}= \frac{-x_{25}x_{34}}{x_{35}}$ and $x_{25}\neq \frac{-x_{13}x_{35}}{x_{12}}$.
\begin{equation*}x=\left(\begin{matrix}
		1 & x_{12} & x_{13} & x_{14} & x_{15} \\
		0 & 1 & 0 & x_{24} & x_{25} \\
		0 & 0 & 1 & x_{34} & x_{35}  \\
		0 & 0 & 0 & 1 & 0  \\
		0 & 0 & 0 & 0 & 1 \\
	\end{matrix}\right),x^A=\left(\begin{matrix}
		1 & 1 & 1 & 0 & 0 \\
		0 & 1 & 0 & 0 & 0 \\
		0 & 0 & 1 & 1 & 0 \\
		0 & 0 & 0 & 1 & 0 \\
		0 & 0 & 0 & 0 & 1
	\end{matrix}\right)
\end{equation*} \newline
Where matrix $A$ has entries
\begin{align*}
	&d_{1} = 1, d_{2} = \frac{1}{x_{12}}, d_{3} = \frac{x_{35}}{x_{12} x_{25} + x_{13} x_{35}}, d_{4} = \frac{x_{35}}{x_{12} x_{25} x_{34} + x_{13} x_{34} x_{35}}, d_{5} = 1,\\ 
	&a_{12} = 1, a_{13} = 1, a_{14} = 1, a_{15} = 1,\\ 
	&a_{23} = \frac{x_{25}}{x_{12} x_{25} + x_{13} x_{35}}, a_{24} = 1, a_{25} = 1,\\ 
	&a_{34} = -\frac{{\left(x_{12}^{2} - x_{12}\right)} x_{25} x_{34} + {\left({\left(x_{12} - 1\right)} x_{13} x_{34} + x_{14}\right)} x_{35}}{x_{12} x_{13} x_{25} x_{34} + x_{13}^{2} x_{34} x_{35}}, a_{35} = \frac{x_{14} x_{35} - {\left(x_{12} + x_{15}\right)} x_{34}}{x_{13} x_{34}},\\ 
	&a_{45} = -\frac{x_{35}}{x_{34}}
\end{align*}

Now assume $x_{24}\neq \frac{-x_{13}x_{34}}{x_{12}}$, $x_{24}= \frac{-x_{25}x_{34}}{x_{35}}$ and $x_{25}= \frac{-x_{13}x_{35}}{x_{12}}$.
\begin{equation*}x=\left(\begin{matrix}
		1 & x_{12} & x_{13} & x_{14} & x_{15} \\
		0 & 1 & 0 & x_{24} & x_{25} \\
		0 & 0 & 1 & x_{34} & x_{35}  \\
		0 & 0 & 0 & 1 & 0  \\
		0 & 0 & 0 & 0 & 1 \\
	\end{matrix}\right),x^A=\left(\begin{matrix}
		1 & 1 & 0 & 0 & 0 \\
		0 & 1 & 0 & 0 & 0 \\
		0 & 0 & 1 & 1 & 0 \\
		0 & 0 & 0 & 1 & 0 \\
		0 & 0 & 0 & 0 & 1
	\end{matrix}\right)
\end{equation*} \newline
Where matrix $A$ has entries
\begin{align*}
	&d_{1} = 1, d_{2} = \frac{1}{x_{12}}, d_{3} = 1, d_{4} = \frac{1}{x_{34}}, d_{5} = 1,\\ 
	&a_{12} = 1, a_{13} = 1, a_{14} = 1, a_{15} = 1,\\ 
	&a_{23} = -\frac{x_{13}}{x_{12}}, a_{24} = 1, a_{25} = 1,\\ 
	&a_{34} = -\frac{x_{14} + {\left(x_{12} - 1\right)} x_{34}}{x_{13} x_{34}}, a_{35} = \frac{x_{14} x_{35} - {\left(x_{12} + x_{15}\right)} x_{34}}{x_{13} x_{34}},\\ 
	&a_{45} = -\frac{x_{35}}{x_{34}}
\end{align*}

Now assume $x_{24}= \frac{-x_{13}x_{34}}{x_{12}}$, $x_{24}= \frac{-x_{25}x_{34}}{x_{35}}$.
\begin{equation*}x=\left(\begin{matrix}
		1 & x_{12} & x_{13} & x_{14} & x_{15} \\
		0 & 1 & 0 & x_{24} & x_{25} \\
		0 & 0 & 1 & x_{34} & x_{35}  \\
		0 & 0 & 0 & 1 & 0  \\
		0 & 0 & 0 & 0 & 1 \\
	\end{matrix}\right),x^A=\left(\begin{matrix}
		1 & 1 & 0 & 0 & 0 \\
		0 & 1 & 0 & 0 & 0 \\
		0 & 0 & 1 & 1 & 0 \\
		0 & 0 & 0 & 1 & 0 \\
		0 & 0 & 0 & 0 & 1
	\end{matrix}\right)
\end{equation*} \newline
Where matrix $A$ has entries
\begin{align*}
	&d_{1} = 1, d_{2} = \frac{1}{x_{12}}, d_{3} = 1, d_{4} = \frac{1}{x_{34}}, d_{5} = 1,\\ 
	&a_{12} = 1, a_{13} = 1, a_{14} = 1, a_{15} = 1,\\ 
	&a_{23} = \frac{x_{25}}{x_{35}}, a_{24} = 1, a_{25} = 1,\\ 
	&a_{34} = \frac{{\left(x_{14} + {\left(x_{12} - 1\right)} x_{34}\right)} x_{35}}{x_{12} x_{25} x_{34}}, a_{35} = -\frac{x_{14} x_{35}^{2} - {\left(x_{12} + x_{15}\right)} x_{34} x_{35}}{x_{12} x_{25} x_{34}},\\ 
	&a_{45} = -\frac{x_{35}}{x_{34}}
\end{align*}

		\section{Subcases of $Y_{15}$}

\begin{equation*}x=\left(\begin{matrix}
		1 & 0 & 0 & 0 & x_{15} \\
		0 & 1 & x_{23} & 0 & 0 \\
		0 & 0 & 1 & x_{34} & 0  \\
		0 & 0 & 0 & 1 & x_{45}  \\
		0 & 0 & 0 & 0 & 1 \\
	\end{matrix}\right),x^A=\left(\begin{matrix}
		1 & 0 & 0 & 0 & 0 \\
		0 & 1 & 1 & 0 & 0 \\
		0 & 0 & 1 & 1 & 0 \\
		0 & 0 & 0 & 1 & 1 \\
		0 & 0 & 0 & 0 & 1
	\end{matrix}\right)
\end{equation*} \newline
Where matrix $A$ has entries
\begin{align*}
	&d_{1} = 1, d_{2} = 1, d_{3} = \frac{1}{x_{23}}, d_{4} = \frac{1}{x_{23} x_{34}}, d_{5} = \frac{1}{x_{23} x_{34} x_{45}},\\ 
	&a_{12} = 0, a_{13} = 0, a_{14} = \frac{1 x_{15}}{x_{23} x_{34} x_{45}}, a_{15} = 1,\\ 
	&a_{23} = 1, a_{24} = 1, a_{25} = 1,\\ 
	&a_{34} = \frac{1}{x_{23}}, a_{35} = \frac{1}{x_{23}},\\ 
	&a_{45} = \frac{1}{x_{23} x_{34}}
\end{align*}

\begin{equation*}x=\left(\begin{matrix}
		1 & 0 & 0 & x_{14} & 0 \\
		0 & 1 & x_{23} & 0 & 0 \\
		0 & 0 & 1 & x_{34} & 0  \\
		0 & 0 & 0 & 1 & x_{45}  \\
		0 & 0 & 0 & 0 & 1 \\
	\end{matrix}\right),x^A=\left(\begin{matrix}
		1 & 0 & 0 & 0 & 0 \\
		0 & 1 & 1 & 0 & 0 \\
		0 & 0 & 1 & 1 & 0 \\
		0 & 0 & 0 & 1 & 1 \\
		0 & 0 & 0 & 0 & 1
	\end{matrix}\right)
\end{equation*} \newline
Where matrix $A$ has entries
\begin{align*}
	&d_{1} = 1, d_{2} = 1, d_{3} = \frac{1}{x_{23}}, d_{4} = \frac{1}{x_{23} x_{34}}, d_{5} = \frac{1}{x_{23} x_{34} x_{45}},\\ 
	&a_{12} = 0, a_{13} = \frac{1 x_{14}}{x_{23} x_{34}}, a_{14} = 1, a_{15} = 1,\\ 
	&a_{23} = \frac{1 x_{23} x_{34}}{x_{14}}, a_{24} = 1, a_{25} = 1,\\ 
	&a_{34} = \frac{1 x_{34}}{x_{14}}, a_{35} = \frac{1}{x_{23}},\\ 
	&a_{45} = \frac{1}{x_{14}}
\end{align*}

\begin{equation*}x=\left(\begin{matrix}
		1 & 0 & 0 & x_{14} & x_{15} \\
		0 & 1 & x_{23} & 0 & 0 \\
		0 & 0 & 1 & x_{34} & 0  \\
		0 & 0 & 0 & 1 & x_{45}  \\
		0 & 0 & 0 & 0 & 1 \\
	\end{matrix}\right),x^A=\left(\begin{matrix}
		1 & 0 & 0 & 0 & 0 \\
		0 & 1 & 1 & 0 & 0 \\
		0 & 0 & 1 & 1 & 0 \\
		0 & 0 & 0 & 1 & 1 \\
		0 & 0 & 0 & 0 & 1
	\end{matrix}\right)
\end{equation*} \newline
Where matrix $A$ has entries
\begin{align*}
	&d_{1} = 1, d_{2} = 1, d_{3} = \frac{1}{x_{23}}, d_{4} = \frac{1}{x_{23} x_{34}}, d_{5} = \frac{1}{x_{23} x_{34} x_{45}},\\ 
	&a_{12} = 0, a_{13} = \frac{1 x_{14}}{x_{23} x_{34}}, a_{14} = 1, a_{15} = 1,\\ 
	&a_{23} = \frac{1 x_{23} x_{34} x_{45} - 1 x_{15}}{x_{14} x_{45}}, a_{24} = 1, a_{25} = 1,\\ 
	&a_{34} = \frac{1 x_{23} x_{34} x_{45} - 1 x_{15}}{x_{14} x_{23} x_{45}}, a_{35} = \frac{1}{x_{23}},\\ 
	&a_{45} = \frac{1 x_{23} x_{34} x_{45} - 1 x_{15}}{x_{14} x_{23} x_{34} x_{45}}
\end{align*}

\begin{equation*}x=\left(\begin{matrix}
		1 & 0 & x_{13} & 0 & 0 \\
		0 & 1 & x_{23} & 0 & 0 \\
		0 & 0 & 1 & x_{34} & 0  \\
		0 & 0 & 0 & 1 & x_{45}  \\
		0 & 0 & 0 & 0 & 1 \\
	\end{matrix}\right),x^A=\left(\begin{matrix}
		1 & 0 & 0 & 0 & 0 \\
		0 & 1 & 1 & 0 & 0 \\
		0 & 0 & 1 & 1 & 0 \\
		0 & 0 & 0 & 1 & 1 \\
		0 & 0 & 0 & 0 & 1
	\end{matrix}\right)
\end{equation*} \newline
Where matrix $A$ has entries
\begin{align*}
	&d_{1} = 1, d_{2} = 1, d_{3} = \frac{1}{x_{23}}, d_{4} = \frac{1}{x_{23} x_{34}}, d_{5} = \frac{1}{x_{23} x_{34} x_{45}},\\ 
	&a_{12} = \frac{x_{13}}{x_{23}}, a_{13} = 1, a_{14} = 1, a_{15} = 1,\\ 
	&a_{23} = \frac{x_{23}}{x_{13}}, a_{24} = \frac{x_{23}}{x_{13}}, a_{25} = 1,\\ 
	&a_{34} = \frac{1}{x_{13}}, a_{35} = \frac{1}{x_{13}},\\ 
	&a_{45} = \frac{1}{x_{13} x_{34}}
\end{align*}

\begin{equation*}x=\left(\begin{matrix}
		1 & 0 & x_{13} & 0 & x_{15} \\
		0 & 1 & x_{23} & 0 & 0 \\
		0 & 0 & 1 & x_{34} & 0  \\
		0 & 0 & 0 & 1 & x_{45}  \\
		0 & 0 & 0 & 0 & 1 \\
	\end{matrix}\right),x^A=\left(\begin{matrix}
		1 & 0 & 0 & 0 & 0 \\
		0 & 1 & 1 & 0 & 0 \\
		0 & 0 & 1 & 1 & 0 \\
		0 & 0 & 0 & 1 & 1 \\
		0 & 0 & 0 & 0 & 1
	\end{matrix}\right)
\end{equation*} \newline
Where matrix $A$ has entries
\begin{align*}
	&d_{1} = 1, d_{2} = 1, d_{3} = \frac{1}{x_{23}}, d_{4} = \frac{1}{x_{23} x_{34}}, d_{5} = \frac{1}{x_{23} x_{34} x_{45}},\\ 
	&a_{12} = \frac{x_{13}}{x_{23}}, a_{13} = 1, a_{14} = 1, a_{15} = 1,\\ 
	&a_{23} = \frac{x_{23}}{x_{13}}, a_{24} = \frac{x_{23} x_{34} x_{45} - x_{15}}{x_{13} x_{34} x_{45}}, a_{25} = 1,\\ 
	&a_{34} = \frac{1}{x_{13}}, a_{35} = \frac{x_{23} x_{34} x_{45} - x_{15}}{x_{13} x_{23} x_{34} x_{45}},\\ 
	&a_{45} = \frac{1}{x_{13} x_{34}}
\end{align*}

\begin{equation*}x=\left(\begin{matrix}
		1 & 0 & x_{13} & x_{14} & 0 \\
		0 & 1 & x_{23} & 0 & 0 \\
		0 & 0 & 1 & x_{34} & 0  \\
		0 & 0 & 0 & 1 & x_{45}  \\
		0 & 0 & 0 & 0 & 1 \\
	\end{matrix}\right),x^A=\left(\begin{matrix}
		1 & 0 & 0 & 0 & 0 \\
		0 & 1 & 1 & 0 & 0 \\
		0 & 0 & 1 & 1 & 0 \\
		0 & 0 & 0 & 1 & 1 \\
		0 & 0 & 0 & 0 & 1
	\end{matrix}\right)
\end{equation*} \newline
Where matrix $A$ has entries
\begin{align*}
	&d_{1} = 1, d_{2} = 1, d_{3} = \frac{1}{x_{23}}, d_{4} = \frac{1}{x_{23} x_{34}}, d_{5} = \frac{1}{x_{23} x_{34} x_{45}},\\ 
	&a_{12} = \frac{x_{13}}{x_{23}}, a_{13} = 1, a_{14} = 1, a_{15} = 1,\\ 
	&a_{23} = \frac{x_{23} x_{34} - x_{14}}{x_{13} x_{34}}, a_{24} = \frac{x_{13} x_{23} x_{34}^{2} - x_{14} x_{23} x_{34} + x_{14}^{2}}{x_{13}^{2} x_{34}^{2}}, a_{25} = 1,\\ 
	&a_{34} = \frac{x_{23} x_{34} - x_{14}}{x_{13} x_{23} x_{34}}, a_{35} = \frac{x_{13} x_{23} x_{34}^{2} - x_{14} x_{23} x_{34} + x_{14}^{2}}{x_{13}^{2} x_{23} x_{34}^{2}},\\ 
	&a_{45} = \frac{x_{23} x_{34} - x_{14}}{x_{13} x_{23} x_{34}^{2}}
\end{align*}

\begin{equation*}x=\left(\begin{matrix}
		1 & 0 & x_{13} & x_{14} & x_{15} \\
		0 & 1 & x_{23} & 0 & 0 \\
		0 & 0 & 1 & x_{34} & 0  \\
		0 & 0 & 0 & 1 & x_{45}  \\
		0 & 0 & 0 & 0 & 1 \\
	\end{matrix}\right),x^A=\left(\begin{matrix}
		1 & 0 & 0 & 0 & 0 \\
		0 & 1 & 1 & 0 & 0 \\
		0 & 0 & 1 & 1 & 0 \\
		0 & 0 & 0 & 1 & 1 \\
		0 & 0 & 0 & 0 & 1
	\end{matrix}\right)
\end{equation*} \newline
Where matrix $A$ has entries
\begin{align*}
	&d_{1} = 1, d_{2} = 1, d_{3} = \frac{1}{x_{23}}, d_{4} = \frac{1}{x_{23} x_{34}}, d_{5} = \frac{1}{x_{23} x_{34} x_{45}}, \\ 
	&a_{12} = \frac{x_{13}}{x_{23}}, a_{13} = 1, a_{14} = 1, a_{15} = 1, \\ 
	&a_{23} = \frac{x_{23} x_{34} - x_{14}}{x_{13} x_{34}}, a_{24} = -\frac{x_{13} x_{15} x_{34} - {\left(x_{13} x_{23} x_{34}^{2} - x_{14} x_{23} x_{34} + x_{14}^{2}\right)} x_{45}}{x_{13}^{2} x_{34}^{2} x_{45}}, a_{25} = 1, \\ 
	&a_{34} = \frac{x_{23} x_{34} - x_{14}}{x_{13} x_{23} x_{34}}, a_{35} = -\frac{x_{13} x_{15} x_{34} - {\left(x_{13} x_{23} x_{34}^{2} - x_{14} x_{23} x_{34} + x_{14}^{2}\right)} x_{45}}{x_{13}^{2} x_{23} x_{34}^{2} x_{45}}, \\ 
	&a_{45} = \frac{x_{23} x_{34} - x_{14}}{x_{13} x_{23} x_{34}^{2}}
\end{align*}

\begin{equation*}x=\left(
\right)
\end{equation*} \newline
Where matrix $A$ has entries
\begin{align*}
	&d_{1} = 1, d_{2} = 1, d_{3} = \frac{1}{x_{23}}, d_{4} = \frac{1}{x_{23} x_{34}}, d_{5} = \frac{1}{x_{23} x_{34} x_{45}},\\ 
	&a_{12} = 0, a_{13} = \frac{x_{14}}{x_{23} x_{34}}, a_{14} = 1, a_{15} = 1,\\ 
	&a_{23} = \frac{x_{23} x_{34} x_{45} - x_{15}}{x_{14} x_{45}}, a_{24} = 1, a_{25} = 1,\\ 
	&a_{34} = \frac{x_{23} x_{34} x_{45} - x_{15}}{x_{14} x_{23} x_{45}}, a_{35} = \frac{x_{23} x_{34} x_{45} - x_{25}}{x_{23}^{2} x_{34} x_{45}},\\ 
	&a_{45} = \frac{x_{23} x_{34} x_{45} - x_{15}}{x_{14} x_{23} x_{34} x_{45}}
\end{align*}

\begin{equation*}x=\left(\begin{matrix}
		1 & 0 & x_{13} & 0 & 0 \\
		0 & 1 & x_{23} & 0 & x_{25} \\
		0 & 0 & 1 & x_{34} & 0  \\
		0 & 0 & 0 & 1 & x_{45}  \\
		0 & 0 & 0 & 0 & 1 \\
	\end{matrix}\right),x^A=\left(\begin{matrix}
		1 & 0 & 0 & 0 & 0 \\
		0 & 1 & 1 & 0 & 0 \\
		0 & 0 & 1 & 1 & 0 \\
		0 & 0 & 0 & 1 & 1 \\
		0 & 0 & 0 & 0 & 1
	\end{matrix}\right)
\end{equation*} \newline
Where matrix $A$ has entries
\begin{align*}
	&d_{1} = 1, d_{2} = 1, d_{3} = \frac{1}{x_{23}}, d_{4} = \frac{1}{x_{23} x_{34}}, d_{5} = \frac{1}{x_{23} x_{34} x_{45}},\\ 
	&a_{12} = \frac{x_{13}}{x_{23}}, a_{13} = 1, a_{14} = 1, a_{15} = 1,\\ 
	&a_{23} = \frac{x_{23}}{x_{13}}, a_{24} = \frac{x_{23}^{2} x_{34} x_{45} + x_{13} x_{25}}{x_{13} x_{23} x_{34} x_{45}}, a_{25} = 1,\\ 
	&a_{34} = \frac{1}{x_{13}}, a_{35} = \frac{1}{x_{13}},\\ 
	&a_{45} = \frac{1}{x_{13} x_{34}}
\end{align*}

\begin{equation*}x=\left(\begin{matrix}
		1 & 0 & x_{13} & 0 & x_{15} \\
		0 & 1 & x_{23} & 0 & x_{25} \\
		0 & 0 & 1 & x_{34} & 0  \\
		0 & 0 & 0 & 1 & x_{45}  \\
		0 & 0 & 0 & 0 & 1 \\
	\end{matrix}\right),x^A=\left(\begin{matrix}
		1 & 0 & 0 & 0 & 0 \\
		0 & 1 & 1 & 0 & 0 \\
		0 & 0 & 1 & 1 & 0 \\
		0 & 0 & 0 & 1 & 1 \\
		0 & 0 & 0 & 0 & 1
	\end{matrix}\right)
\end{equation*} \newline
Where matrix $A$ has entries
\begin{align*}
	&d_{1} = 1, d_{2} = 1, d_{3} = \frac{1}{x_{23}}, d_{4} = \frac{1}{x_{23} x_{34}}, d_{5} = \frac{1}{x_{23} x_{34} x_{45}},\\ 
	&a_{12} = \frac{x_{13}}{x_{23}}, a_{13} = 1, a_{14} = 1, a_{15} = 1,\\ 
	&a_{23} = \frac{x_{23}}{x_{13}}, a_{24} = \frac{x_{23}^{2} x_{34} x_{45} - x_{15} x_{23} + x_{13} x_{25}}{x_{13} x_{23} x_{34} x_{45}}, a_{25} = 1,\\ 
	&a_{34} = \frac{1}{x_{13}}, a_{35} = \frac{x_{23} x_{34} x_{45} - x_{15}}{x_{13} x_{23} x_{34} x_{45}},\\ 
	&a_{45} = \frac{1}{x_{13} x_{34}}
\end{align*}

\begin{equation*}x=\left(\begin{matrix}
		1 & 0 & x_{13} & x_{14} & 0 \\
		0 & 1 & x_{23} & 0 & x_{25} \\
		0 & 0 & 1 & x_{34} & 0  \\
		0 & 0 & 0 & 1 & x_{45}  \\
		0 & 0 & 0 & 0 & 1 \\
	\end{matrix}\right),x^A=\left(\begin{matrix}
		1 & 0 & 0 & 0 & 0 \\
		0 & 1 & 1 & 0 & 0 \\
		0 & 0 & 1 & 1 & 0 \\
		0 & 0 & 0 & 1 & 1 \\
		0 & 0 & 0 & 0 & 1
	\end{matrix}\right)
\end{equation*} \newline
Where matrix $A$ has entries
\begin{align*}
	&d_{1} = 1, d_{2} = 1, d_{3} = \frac{1}{x_{23}}, d_{4} = \frac{1}{x_{23} x_{34}}, d_{5} = \frac{1}{x_{23} x_{34} x_{45}},\\ 
	&a_{12} = \frac{x_{13}}{x_{23}}, a_{13} = 1, a_{14} = 1, a_{15} = 1,\\ 
	&a_{23} = \frac{x_{23} x_{34} - x_{14}}{x_{13} x_{34}}, a_{24} = \frac{x_{13}^{2} x_{25} x_{34} + {\left(x_{13} x_{23}^{2} x_{34}^{2} - x_{14} x_{23}^{2} x_{34} + x_{14}^{2} x_{23}\right)} x_{45}}{x_{13}^{2} x_{23} x_{34}^{2} x_{45}}, a_{25} = 1,\\ 
	&a_{34} = \frac{x_{23} x_{34} - x_{14}}{x_{13} x_{23} x_{34}}, a_{35} = \frac{x_{13} x_{23} x_{34}^{2} - x_{14} x_{23} x_{34} + x_{14}^{2}}{x_{13}^{2} x_{23} x_{34}^{2}},\\ 
	&a_{45} = \frac{x_{23} x_{34} - x_{14}}{x_{13} x_{23} x_{34}^{2}}
\end{align*}

\begin{equation*}x=\left(\begin{matrix}
		1 & 0 & x_{13} & x_{14} & x_{15} \\
		0 & 1 & x_{23} & 0 & x_{25} \\
		0 & 0 & 1 & x_{34} & 0  \\
		0 & 0 & 0 & 1 & x_{45}  \\
		0 & 0 & 0 & 0 & 1 \\
	\end{matrix}\right),x^A=\left(\begin{matrix}
		1 & 0 & 0 & 0 & 0 \\
		0 & 1 & 1 & 0 & 0 \\
		0 & 0 & 1 & 1 & 0 \\
		0 & 0 & 0 & 1 & 1 \\
		0 & 0 & 0 & 0 & 1
	\end{matrix}\right)
\end{equation*} \newline
Where matrix $A$ has entries
\begin{align*}
	&d_{1} = 1, d_{2} = 1, d_{3} = \frac{1}{x_{23}}, d_{4} = \frac{1}{x_{23} x_{34}}, d_{5} = \frac{1}{x_{23} x_{34} x_{45}},\\ 
	&a_{12} = \frac{x_{13}}{x_{23}}, a_{13} = 1, a_{14} = 1, a_{15} = 1,\\ 
	&a_{23} = \frac{x_{23} x_{34} - x_{14}}{x_{13} x_{34}}, a_{24} = -\frac{{\left(x_{13} x_{15} x_{23} - x_{13}^{2} x_{25}\right)} x_{34} - {\left(x_{13} x_{23}^{2} x_{34}^{2} - x_{14} x_{23}^{2} x_{34} + x_{14}^{2} x_{23}\right)} x_{45}}{x_{13}^{2} x_{23} x_{34}^{2} x_{45}}, a_{25} = 1,\\ 
	&a_{34} = \frac{x_{23} x_{34} - x_{14}}{x_{13} x_{23} x_{34}}, a_{35} = -\frac{x_{13} x_{15} x_{34} - {\left(x_{13} x_{23} x_{34}^{2} - x_{14} x_{23} x_{34} + x_{14}^{2}\right)} x_{45}}{x_{13}^{2} x_{23} x_{34}^{2} x_{45}},\\ 
	&a_{45} = \frac{x_{23} x_{34} - x_{14}}{x_{13} x_{23} x_{34}^{2}}
\end{align*}

\begin{equation*}x=\left(\begin{matrix}
		1 & 0 & 0 & 0 & 0 \\
		0 & 1 & x_{23} & x_{24} & 0 \\
		0 & 0 & 1 & x_{34} & 0  \\
		0 & 0 & 0 & 1 & x_{45}  \\
		0 & 0 & 0 & 0 & 1 \\
	\end{matrix}\right),x^A=\left(\begin{matrix}
		1 & 0 & 0 & 0 & 0 \\
		0 & 1 & 1 & 0 & 0 \\
		0 & 0 & 1 & 1 & 0 \\
		0 & 0 & 0 & 1 & 1 \\
		0 & 0 & 0 & 0 & 1
	\end{matrix}\right)
\end{equation*} \newline
Where matrix $A$ has entries
\begin{align*}
	&d_{1} = 1, d_{2} = 1, d_{3} = \frac{1}{x_{23}}, d_{4} = \frac{1}{x_{23} x_{34}}, d_{5} = \frac{1}{x_{23} x_{34} x_{45}},\\ 
	&a_{12} = 0, a_{13} = 0, a_{14} = 0, a_{15} = 1,\\ 
	&a_{23} = 1, a_{24} = 1, a_{25} = 1,\\ 
	&a_{34} = \frac{x_{23} x_{34} - x_{24}}{x_{23}^{2} x_{34}}, a_{35} = \frac{x_{23}^{2} x_{34}^{2} - x_{23} x_{24} x_{34} + x_{24}^{2}}{x_{23}^{3} x_{34}^{2}},\\ 
	&a_{45} = \frac{x_{23} x_{34} - x_{24}}{x_{23}^{2} x_{34}^{2}}
\end{align*}

\begin{equation*}x=\left(\begin{matrix}
		1 & 0 & 0 & 0 & x_{15} \\
		0 & 1 & x_{23} & x_{24} & 0 \\
		0 & 0 & 1 & x_{34} & 0  \\
		0 & 0 & 0 & 1 & x_{45}  \\
		0 & 0 & 0 & 0 & 1 \\
	\end{matrix}\right),x^A=\left(\begin{matrix}
		1 & 0 & 0 & 0 & 0 \\
		0 & 1 & 1 & 0 & 0 \\
		0 & 0 & 1 & 1 & 0 \\
		0 & 0 & 0 & 1 & 1 \\
		0 & 0 & 0 & 0 & 1
	\end{matrix}\right)
\end{equation*} \newline
Where matrix $A$ has entries
\begin{align*}
	&d_{1} = 1, d_{2} = 1, d_{3} = \frac{1}{x_{23}}, d_{4} = \frac{1}{x_{23} x_{34}}, d_{5} = \frac{1}{x_{23} x_{34} x_{45}},\\ 
	&a_{12} = 0, a_{13} = 0, a_{14} = \frac{x_{15}}{x_{23} x_{34} x_{45}}, a_{15} = 1,\\ 
	&a_{23} = 1, a_{24} = 1, a_{25} = 1,\\ 
	&a_{34} = \frac{x_{23} x_{34} - x_{24}}{x_{23}^{2} x_{34}}, a_{35} = \frac{x_{23}^{2} x_{34}^{2} - x_{23} x_{24} x_{34} + x_{24}^{2}}{x_{23}^{3} x_{34}^{2}},\\ 
	&a_{45} = \frac{x_{23} x_{34} - x_{24}}{x_{23}^{2} x_{34}^{2}}
\end{align*}

\begin{equation*}x=\left(\begin{matrix}
		1 & 0 & 0 & x_{14} & 0 \\
		0 & 1 & x_{23} & x_{24} & 0 \\
		0 & 0 & 1 & x_{34} & 0  \\
		0 & 0 & 0 & 1 & x_{45}  \\
		0 & 0 & 0 & 0 & 1 \\
	\end{matrix}\right),x^A=\left(\begin{matrix}
		1 & 0 & 0 & 0 & 0 \\
		0 & 1 & 1 & 0 & 0 \\
		0 & 0 & 1 & 1 & 0 \\
		0 & 0 & 0 & 1 & 1 \\
		0 & 0 & 0 & 0 & 1
	\end{matrix}\right)
\end{equation*} \newline
Where matrix $A$ has entries
\begin{align*}
	&d_{1} = 1, d_{2} = 1, d_{3} = \frac{1}{x_{23}}, d_{4} = \frac{1}{x_{23} x_{34}}, d_{5} = \frac{1}{x_{23} x_{34} x_{45}},\\ 
	&a_{12} = 0, a_{13} = \frac{x_{14}}{x_{23} x_{34}}, a_{14} = 1, a_{15} = 1,\\ 
	&a_{23} = \frac{x_{23}^{2} x_{34}^{2} + x_{14} x_{24}}{x_{14} x_{23} x_{34}}, a_{24} = 1, a_{25} = 1,\\ 
	&a_{34} = \frac{x_{34}}{x_{14}}, a_{35} = \frac{x_{14} - x_{24}}{x_{14} x_{23}},\\ 
	&a_{45} = \frac{1}{x_{14}}
\end{align*}

\begin{equation*}x=\left(\begin{matrix}
		1 & 0 & 0 & x_{14} & x_{15} \\
		0 & 1 & x_{23} & x_{24} & 0 \\
		0 & 0 & 1 & x_{34} & 0  \\
		0 & 0 & 0 & 1 & x_{45}  \\
		0 & 0 & 0 & 0 & 1 \\
	\end{matrix}\right),x^A=\left(\begin{matrix}
		1 & 0 & 0 & 0 & 0 \\
		0 & 1 & 1 & 0 & 0 \\
		0 & 0 & 1 & 1 & 0 \\
		0 & 0 & 0 & 1 & 1 \\
		0 & 0 & 0 & 0 & 1
	\end{matrix}\right)
\end{equation*} \newline
Where matrix $A$ has entries
\begin{align*}
	&d_{1} = 1, d_{2} = 1, d_{3} = \frac{1}{x_{23}}, d_{4} = \frac{1}{x_{23} x_{34}}, d_{5} = \frac{1}{x_{23} x_{34} x_{45}},\\ 
	&a_{12} = 0, a_{13} = \frac{x_{14}}{x_{23} x_{34}}, a_{14} = 1, a_{15} = 1,\\ 
	&a_{23} = -\frac{x_{15} x_{23} x_{34} - {\left(x_{23}^{2} x_{34}^{2} + x_{14} x_{24}\right)} x_{45}}{x_{14} x_{23} x_{34} x_{45}}, a_{24} = 1, a_{25} = 1,\\ 
	&a_{34} = \frac{x_{23} x_{34} x_{45} - x_{15}}{x_{14} x_{23} x_{45}}, a_{35} = \frac{x_{15} x_{24} + {\left(x_{14} x_{23} - x_{23} x_{24}\right)} x_{34} x_{45}}{x_{14} x_{23}^{2} x_{34} x_{45}},\\ 
	&a_{45} = \frac{x_{23} x_{34} x_{45} - x_{15}}{x_{14} x_{23} x_{34} x_{45}}
\end{align*}

\begin{equation*}x=\left(\begin{matrix}
		1 & 0 & x_{13} & 0 & 0 \\
		0 & 1 & x_{23} & x_{24} & 0 \\
		0 & 0 & 1 & x_{34} & 0  \\
		0 & 0 & 0 & 1 & x_{45}  \\
		0 & 0 & 0 & 0 & 1 \\
	\end{matrix}\right),x^A=\left(\begin{matrix}
		1 & 0 & 0 & 0 & 0 \\
		0 & 1 & 1 & 0 & 0 \\
		0 & 0 & 1 & 1 & 0 \\
		0 & 0 & 0 & 1 & 1 \\
		0 & 0 & 0 & 0 & 1
	\end{matrix}\right)
\end{equation*} \newline
Where matrix $A$ has entries
\begin{align*}
	&d_{1} = 1, d_{2} = 1, d_{3} = \frac{1}{x_{23}}, d_{4} = \frac{1}{x_{23} x_{34}}, d_{5} = \frac{1}{x_{23} x_{34} x_{45}},\\ 
	&a_{12} = \frac{x_{13}}{x_{23}}, a_{13} = 1, a_{14} = 1, a_{15} = 1,\\ 
	&a_{23} = \frac{x_{23}^{2} x_{34} + x_{13} x_{24}}{x_{13} x_{23} x_{34}}, a_{24} = \frac{x_{23} x_{34} + x_{24}}{x_{13} x_{34}}, a_{25} = 1,\\ 
	&a_{34} = \frac{1}{x_{13}}, a_{35} = \frac{1}{x_{13}},\\ 
	&a_{45} = \frac{1}{x_{13} x_{34}}
\end{align*}

\begin{equation*}x=\left(\begin{matrix}
		1 & 0 & x_{13} & 0 & x_{15} \\
		0 & 1 & x_{23} & x_{24} & 0 \\
		0 & 0 & 1 & x_{34} & 0  \\
		0 & 0 & 0 & 1 & x_{45}  \\
		0 & 0 & 0 & 0 & 1 \\
	\end{matrix}\right),x^A=\left(\begin{matrix}
		1 & 0 & 0 & 0 & 0 \\
		0 & 1 & 1 & 0 & 0 \\
		0 & 0 & 1 & 1 & 0 \\
		0 & 0 & 0 & 1 & 1 \\
		0 & 0 & 0 & 0 & 1
	\end{matrix}\right)
\end{equation*} \newline
Where matrix $A$ has entries
\begin{align*}
	&d_{1} = 1, d_{2} = 1, d_{3} = \frac{1}{x_{23}}, d_{4} = \frac{1}{x_{23} x_{34}}, d_{5} = \frac{1}{x_{23} x_{34} x_{45}},\\ 
	&a_{12} = \frac{x_{13}}{x_{23}}, a_{13} = 1, a_{14} = 1, a_{15} = 1,\\ 
	&a_{23} = \frac{x_{23}^{2} x_{34} + x_{13} x_{24}}{x_{13} x_{23} x_{34}}, a_{24} = -\frac{x_{15} - {\left(x_{23} x_{34} + x_{24}\right)} x_{45}}{x_{13} x_{34} x_{45}}, a_{25} = 1,\\ 
	&a_{34} = \frac{1}{x_{13}}, a_{35} = \frac{x_{23} x_{34} x_{45} - x_{15}}{x_{13} x_{23} x_{34} x_{45}},\\ 
	&a_{45} = \frac{1}{x_{13} x_{34}}
\end{align*}

\begin{equation*}x=\left(\begin{matrix}
		1 & 0 & x_{13} & x_{14} & 0 \\
		0 & 1 & x_{23} & x_{24} & 0 \\
		0 & 0 & 1 & x_{34} & 0  \\
		0 & 0 & 0 & 1 & x_{45}  \\
		0 & 0 & 0 & 0 & 1 \\
	\end{matrix}\right),x^A=\left(\begin{matrix}
		1 & 0 & 0 & 0 & 0 \\
		0 & 1 & 1 & 0 & 0 \\
		0 & 0 & 1 & 1 & 0 \\
		0 & 0 & 0 & 1 & 1 \\
		0 & 0 & 0 & 0 & 1
	\end{matrix}\right)
\end{equation*} \newline
Where matrix $A$ has entries
\begin{align*}
	&d_{1} = 1, d_{2} = 1, d_{3} = \frac{1}{x_{23}}, d_{4} = \frac{1}{x_{23} x_{34}}, d_{5} = \frac{1}{x_{23} x_{34} x_{45}},\\ 
	&a_{12} = \frac{x_{13}}{x_{23}}, a_{13} = 1, a_{14} = 1, a_{15} = 1,\\ 
	&a_{23} = \frac{x_{23}^{2} x_{34} - x_{14} x_{23} + x_{13} x_{24}}{x_{13} x_{23} x_{34}}, a_{24} = \frac{x_{13} x_{23}^{2} x_{34}^{2} + x_{14}^{2} x_{23} - x_{13} x_{14} x_{24} - {\left(x_{14} x_{23}^{2} - x_{13} x_{23} x_{24}\right)} x_{34}}{x_{13}^{2} x_{23} x_{34}^{2}}, a_{25} = 1,\\ 
	&a_{34} = \frac{x_{23} x_{34} - x_{14}}{x_{13} x_{23} x_{34}}, a_{35} = \frac{x_{13} x_{23} x_{34}^{2} - x_{14} x_{23} x_{34} + x_{14}^{2}}{x_{13}^{2} x_{23} x_{34}^{2}},\\ 
	&a_{45} = \frac{x_{23} x_{34} - x_{14}}{x_{13} x_{23} x_{34}^{2}}
\end{align*}

\begin{equation*}x=\left(\begin{matrix}
		1 & 0 & x_{13} & x_{14} & x_{15} \\
		0 & 1 & x_{23} & x_{24} & 0 \\
		0 & 0 & 1 & x_{34} & 0  \\
		0 & 0 & 0 & 1 & x_{45}  \\
		0 & 0 & 0 & 0 & 1 \\
	\end{matrix}\right),x^A=\left(\begin{matrix}
		1 & 0 & 0 & 0 & 0 \\
		0 & 1 & 1 & 0 & 0 \\
		0 & 0 & 1 & 1 & 0 \\
		0 & 0 & 0 & 1 & 1 \\
		0 & 0 & 0 & 0 & 1
	\end{matrix}\right)
\end{equation*} \newline
Where matrix $A$ has entries
\begin{align*}
	&d_{1} = 1, d_{2} = 1, d_{3} = \frac{1}{x_{23}}, d_{4} = \frac{1}{x_{23} x_{34}}, d_{5} = \frac{1}{x_{23} x_{34} x_{45}},\\ 
	&a_{12} = \frac{x_{13}}{x_{23}}, a_{13} = 1, a_{14} = 1, a_{15} = 1,\\ 
	&a_{23} = \frac{x_{23}^{2} x_{34} - x_{14} x_{23} + x_{13} x_{24}}{x_{13} x_{23} x_{34}}, a_{24} = -\frac{x_{13} x_{15} x_{23} x_{34} - {\left(x_{13} x_{23}^{2} x_{34}^{2} + x_{14}^{2} x_{23} - x_{13} x_{14} x_{24} - {\left(x_{14} x_{23}^{2} - x_{13} x_{23} x_{24}\right)} x_{34}\right)} x_{45}}{x_{13}^{2} x_{23} x_{34}^{2} x_{45}}, a_{25} = 1,\\ 
	&a_{34} = \frac{x_{23} x_{34} - x_{14}}{x_{13} x_{23} x_{34}}, a_{35} = -\frac{x_{13} x_{15} x_{34} - {\left(x_{13} x_{23} x_{34}^{2} - x_{14} x_{23} x_{34} + x_{14}^{2}\right)} x_{45}}{x_{13}^{2} x_{23} x_{34}^{2} x_{45}},\\ 
	&a_{45} = \frac{x_{23} x_{34} - x_{14}}{x_{13} x_{23} x_{34}^{2}}
\end{align*}

\begin{equation*}x=\left(\begin{matrix}
		1 & 0 & 0 & 0 & 0 \\
		0 & 1 & x_{23} & x_{24} & x_{25} \\
		0 & 0 & 1 & x_{34} & 0  \\
		0 & 0 & 0 & 1 & x_{45}  \\
		0 & 0 & 0 & 0 & 1 \\
	\end{matrix}\right),x^A=\left(\begin{matrix}
		1 & 0 & 0 & 0 & 0 \\
		0 & 1 & 1 & 0 & 0 \\
		0 & 0 & 1 & 1 & 0 \\
		0 & 0 & 0 & 1 & 1 \\
		0 & 0 & 0 & 0 & 1
	\end{matrix}\right)
\end{equation*} \newline
Where matrix $A$ has entries
\begin{align*}
	&d_{1} = 1, d_{2} = 1, d_{3} = \frac{1}{x_{23}}, d_{4} = \frac{1}{x_{23} x_{34}}, d_{5} = \frac{1}{x_{23} x_{34} x_{45}},\\ 
	&a_{12} = 0, a_{13} = 0, a_{14} = 0, a_{15} = 1,\\ 
	&a_{23} = 1, a_{24} = 1, a_{25} = 1,\\ 
	&a_{34} = \frac{x_{23} x_{34} - x_{24}}{x_{23}^{2} x_{34}}, a_{35} = -\frac{x_{23} x_{25} x_{34} - {\left(x_{23}^{2} x_{34}^{2} - x_{23} x_{24} x_{34} + x_{24}^{2}\right)} x_{45}}{x_{23}^{3} x_{34}^{2} x_{45}},\\ 
	&a_{45} = \frac{x_{23} x_{34} - x_{24}}{x_{23}^{2} x_{34}^{2}}
\end{align*}

\begin{equation*}x=\left(\begin{matrix}
		1 & 0 & 0 & 0 & x_{15} \\
		0 & 1 & x_{23} & x_{24} & x_{25} \\
		0 & 0 & 1 & x_{34} & 0  \\
		0 & 0 & 0 & 1 & x_{45}  \\
		0 & 0 & 0 & 0 & 1 \\
	\end{matrix}\right),x^A=\left(\begin{matrix}
		1 & 0 & 0 & 0 & 0 \\
		0 & 1 & 1 & 0 & 0 \\
		0 & 0 & 1 & 1 & 0 \\
		0 & 0 & 0 & 1 & 1 \\
		0 & 0 & 0 & 0 & 1
	\end{matrix}\right)
\end{equation*} \newline
Where matrix $A$ has entries
\begin{align*}
	&d_{1} = 1, d_{2} = 1, d_{3} = \frac{1}{x_{23}}, d_{4} = \frac{1}{x_{23} x_{34}}, d_{5} = \frac{1}{x_{23} x_{34} x_{45}},\\ 
	&a_{12} = 0, a_{13} = 0, a_{14} = \frac{x_{15}}{x_{23} x_{34} x_{45}}, a_{15} = 1,\\ 
	&a_{23} = 1, a_{24} = 1, a_{25} = 1,\\ 
	&a_{34} = \frac{x_{23} x_{34} - x_{24}}{x_{23}^{2} x_{34}}, a_{35} = -\frac{x_{23} x_{25} x_{34} - {\left(x_{23}^{2} x_{34}^{2} - x_{23} x_{24} x_{34} + x_{24}^{2}\right)} x_{45}}{x_{23}^{3} x_{34}^{2} x_{45}},\\ 
	&a_{45} = \frac{x_{23} x_{34} - x_{24}}{x_{23}^{2} x_{34}^{2}}
\end{align*}

\begin{equation*}x=\left(\begin{matrix}
		1 & 0 & 0 & x_{14} & 0 \\
		0 & 1 & x_{23} & x_{24} & x_{25} \\
		0 & 0 & 1 & x_{34} & 0  \\
		0 & 0 & 0 & 1 & x_{45}  \\
		0 & 0 & 0 & 0 & 1 \\
	\end{matrix}\right),x^A=\left(\begin{matrix}
		1 & 0 & 0 & 0 & 0 \\
		0 & 1 & 1 & 0 & 0 \\
		0 & 0 & 1 & 1 & 0 \\
		0 & 0 & 0 & 1 & 1 \\
		0 & 0 & 0 & 0 & 1
	\end{matrix}\right)
\end{equation*} \newline
Where matrix $A$ has entries
\begin{align*}
	&d_{1} = 1, d_{2} = 1, d_{3} = \frac{1}{x_{23}}, d_{4} = \frac{1}{x_{23} x_{34}}, d_{5} = \frac{1}{x_{23} x_{34} x_{45}},\\ 
	&a_{12} = 0, a_{13} = \frac{x_{14}}{x_{23} x_{34}}, a_{14} = 1, a_{15} = 1,\\ 
	&a_{23} = \frac{x_{23}^{2} x_{34}^{2} + x_{14} x_{24}}{x_{14} x_{23} x_{34}}, a_{24} = 1, a_{25} = 1,\\ 
	&a_{34} = \frac{x_{34}}{x_{14}}, a_{35} = -\frac{x_{14} x_{25} - {\left(x_{14} x_{23} - x_{23} x_{24}\right)} x_{34} x_{45}}{x_{14} x_{23}^{2} x_{34} x_{45}},\\ 
	&a_{45} = \frac{1}{x_{14}}
\end{align*}

\begin{equation*}x=\left(\begin{matrix}
		1 & 0 & 0 & x_{14} & x_{15} \\
		0 & 1 & x_{23} & x_{24} & x_{25} \\
		0 & 0 & 1 & x_{34} & 0  \\
		0 & 0 & 0 & 1 & x_{45}  \\
		0 & 0 & 0 & 0 & 1 \\
	\end{matrix}\right),x^A=\left(\begin{matrix}
		1 & 0 & 0 & 0 & 0 \\
		0 & 1 & 1 & 0 & 0 \\
		0 & 0 & 1 & 1 & 0 \\
		0 & 0 & 0 & 1 & 1 \\
		0 & 0 & 0 & 0 & 1
	\end{matrix}\right)
\end{equation*} \newline
Where matrix $A$ has entries
\begin{align*}
	&d_{1} = 1, d_{2} = 1, d_{3} = \frac{1}{x_{23}}, d_{4} = \frac{1}{x_{23} x_{34}}, d_{5} = \frac{1}{x_{23} x_{34} x_{45}},\\ 
	&a_{12} = 0, a_{13} = \frac{x_{14}}{x_{23} x_{34}}, a_{14} = 1, a_{15} = 1,\\ 
	&a_{23} = -\frac{x_{15} x_{23} x_{34} - {\left(x_{23}^{2} x_{34}^{2} + x_{14} x_{24}\right)} x_{45}}{x_{14} x_{23} x_{34} x_{45}}, a_{24} = 1, a_{25} = 1,\\ 
	&a_{34} = \frac{x_{23} x_{34} x_{45} - x_{15}}{x_{14} x_{23} x_{45}}, a_{35} = \frac{x_{15} x_{24} - x_{14} x_{25} + {\left(x_{14} x_{23} - x_{23} x_{24}\right)} x_{34} x_{45}}{x_{14} x_{23}^{2} x_{34} x_{45}},\\ 
	&a_{45} = \frac{x_{23} x_{34} x_{45} - x_{15}}{x_{14} x_{23} x_{34} x_{45}}
\end{align*}

\begin{equation*}x=\left(\begin{matrix}
		1 & 0 & x_{13} & 0 & 0 \\
		0 & 1 & x_{23} & x_{24} & x_{25} \\
		0 & 0 & 1 & x_{34} & 0  \\
		0 & 0 & 0 & 1 & x_{45}  \\
		0 & 0 & 0 & 0 & 1 \\
	\end{matrix}\right),x^A=\left(\begin{matrix}
		1 & 0 & 0 & 0 & 0 \\
		0 & 1 & 1 & 0 & 0 \\
		0 & 0 & 1 & 1 & 0 \\
		0 & 0 & 0 & 1 & 1 \\
		0 & 0 & 0 & 0 & 1
	\end{matrix}\right)
\end{equation*} \newline
Where matrix $A$ has entries
\begin{align*}
	&d_{1} = 1, d_{2} = 1, d_{3} = \frac{1}{x_{23}}, d_{4} = \frac{1}{x_{23} x_{34}}, d_{5} = \frac{1}{x_{23} x_{34} x_{45}},\\ 
	&a_{12} = \frac{x_{13}}{x_{23}}, a_{13} = 1, a_{14} = 1, a_{15} = 1,\\ 
	&a_{23} = \frac{x_{23}^{2} x_{34} + x_{13} x_{24}}{x_{13} x_{23} x_{34}}, a_{24} = \frac{x_{13} x_{25} + {\left(x_{23}^{2} x_{34} + x_{23} x_{24}\right)} x_{45}}{x_{13} x_{23} x_{34} x_{45}}, a_{25} = 1,\\ 
	&a_{34} = \frac{1}{x_{13}}, a_{35} = \frac{1}{x_{13}},\\ 
	&a_{45} = \frac{1}{x_{13} x_{34}}
\end{align*}

\begin{equation*}x=\left(\begin{matrix}
		1 & 0 & x_{13} & 0 & x_{15} \\
		0 & 1 & x_{23} & x_{24} & x_{25} \\
		0 & 0 & 1 & x_{34} & 0  \\
		0 & 0 & 0 & 1 & x_{45}  \\
		0 & 0 & 0 & 0 & 1 \\
	\end{matrix}\right),x^A=\left(\begin{matrix}
		1 & 0 & 0 & 0 & 0 \\
		0 & 1 & 1 & 0 & 0 \\
		0 & 0 & 1 & 1 & 0 \\
		0 & 0 & 0 & 1 & 1 \\
		0 & 0 & 0 & 0 & 1
	\end{matrix}\right)
\end{equation*} \newline
Where matrix $A$ has entries
\begin{align*}
	&d_{1} = 1, d_{2} = 1, d_{3} = \frac{1}{x_{23}}, d_{4} = \frac{1}{x_{23} x_{34}}, d_{5} = \frac{1}{x_{23} x_{34} x_{45}},\\ 
	&a_{12} = \frac{x_{13}}{x_{23}}, a_{13} = 1, a_{14} = 1, a_{15} = 1,\\ 
	&a_{23} = \frac{x_{23}^{2} x_{34} + x_{13} x_{24}}{x_{13} x_{23} x_{34}}, a_{24} = -\frac{x_{15} x_{23} - x_{13} x_{25} - {\left(x_{23}^{2} x_{34} + x_{23} x_{24}\right)} x_{45}}{x_{13} x_{23} x_{34} x_{45}}, a_{25} = 1,\\ 
	&a_{34} = \frac{1}{x_{13}}, a_{35} = \frac{x_{23} x_{34} x_{45} - x_{15}}{x_{13} x_{23} x_{34} x_{45}},\\ 
	&a_{45} = \frac{1}{x_{13} x_{34}}
\end{align*}

\begin{equation*}x=\left(\begin{matrix}
		1 & 0 & x_{13} & x_{14} & 0 \\
		0 & 1 & x_{23} & x_{24} & x_{25} \\
		0 & 0 & 1 & x_{34} & 0  \\
		0 & 0 & 0 & 1 & x_{45}  \\
		0 & 0 & 0 & 0 & 1 \\
	\end{matrix}\right),x^A=\left(\begin{matrix}
		1 & 0 & 0 & 0 & 0 \\
		0 & 1 & 1 & 0 & 0 \\
		0 & 0 & 1 & 1 & 0 \\
		0 & 0 & 0 & 1 & 1 \\
		0 & 0 & 0 & 0 & 1
	\end{matrix}\right)
\end{equation*} \newline
Where matrix $A$ has entries
\begin{align*}
	&d_{1} = 1, d_{2} = 1, d_{3} = \frac{1}{x_{23}}, d_{4} = \frac{1}{x_{23} x_{34}}, d_{5} = \frac{1}{x_{23} x_{34} x_{45}},\\ 
	&a_{12} = \frac{x_{13}}{x_{23}}, a_{13} = 1, a_{14} = 1, a_{15} = 1,\\ 
	&a_{23} = \frac{x_{23}^{2} x_{34} - x_{14} x_{23} + x_{13} x_{24}}{x_{13} x_{23} x_{34}}, a_{24} = \frac{x_{13}^{2} x_{25} x_{34} + {\left(x_{13} x_{23}^{2} x_{34}^{2} + x_{14}^{2} x_{23} - x_{13} x_{14} x_{24} - {\left(x_{14} x_{23}^{2} - x_{13} x_{23} x_{24}\right)} x_{34}\right)} x_{45}}{x_{13}^{2} x_{23} x_{34}^{2} x_{45}}, a_{25} = 1,\\ 
	&a_{34} = \frac{x_{23} x_{34} - x_{14}}{x_{13} x_{23} x_{34}}, a_{35} = \frac{x_{13} x_{23} x_{34}^{2} - x_{14} x_{23} x_{34} + x_{14}^{2}}{x_{13}^{2} x_{23} x_{34}^{2}},\\ 
	&a_{45} = \frac{x_{23} x_{34} - x_{14}}{x_{13} x_{23} x_{34}^{2}}
\end{align*}

\begin{equation*}x=\left(\begin{matrix}
		1 & 0 & x_{13} & x_{14} & x_{15} \\
		0 & 1 & x_{23} & x_{24} & x_{25} \\
		0 & 0 & 1 & x_{34} & 0  \\
		0 & 0 & 0 & 1 & x_{45}  \\
		0 & 0 & 0 & 0 & 1 \\
	\end{matrix}\right),x^A=\left(\begin{matrix}
		1 & 0 & 0 & 0 & 0 \\
		0 & 1 & 1 & 0 & 0 \\
		0 & 0 & 1 & 1 & 0 \\
		0 & 0 & 0 & 1 & 1 \\
		0 & 0 & 0 & 0 & 1
	\end{matrix}\right)
\end{equation*} \newline
Where matrix $A$ has entries
\begin{align*}
	&d_{1} = 1, d_{2} = 1, d_{3} = \frac{1}{x_{23}}, d_{4} = \frac{1}{x_{23} x_{34}}, d_{5} = \frac{1}{x_{23} x_{34} x_{45}},\\ 
	&a_{12} = \frac{x_{13}}{x_{23}}, a_{13} = 1, a_{14} = 1, a_{15} = 1,\\ 
	&a_{23} = \frac{x_{23}^{2} x_{34} - x_{14} x_{23} + x_{13} x_{24}}{x_{13} x_{23} x_{34}},\\
	& a_{24} = -\frac{{\left(x_{13} x_{15} x_{23} - x_{13}^{2} x_{25}\right)} x_{34} - {\left(x_{13} x_{23}^{2} x_{34}^{2} + x_{14}^{2} x_{23} - x_{13} x_{14} x_{24} - {\left(x_{14} x_{23}^{2} - x_{13} x_{23} x_{24}\right)} x_{34}\right)} x_{45}}{x_{13}^{2} x_{23} x_{34}^{2} x_{45}}, a_{25} = 1,\\ 
	&a_{34} = \frac{x_{23} x_{34} - x_{14}}{x_{13} x_{23} x_{34}}, a_{35} = -\frac{x_{13} x_{15} x_{34} - {\left(x_{13} x_{23} x_{34}^{2} - x_{14} x_{23} x_{34} + x_{14}^{2}\right)} x_{45}}{x_{13}^{2} x_{23} x_{34}^{2} x_{45}},\\ 
	&a_{45} = \frac{x_{23} x_{34} - x_{14}}{x_{13} x_{23} x_{34}^{2}}
\end{align*}

\begin{equation*}x=\left(\begin{matrix}
		1 & 0 & 0 & 0 & 0 \\
		0 & 1 & x_{23} & 0 & 0 \\
		0 & 0 & 1 & x_{34} & x_{35}  \\
		0 & 0 & 0 & 1 & x_{45}  \\
		0 & 0 & 0 & 0 & 1 \\
	\end{matrix}\right),x^A=\left(\begin{matrix}
		1 & 0 & 0 & 0 & 0 \\
		0 & 1 & 1 & 0 & 0 \\
		0 & 0 & 1 & 1 & 0 \\
		0 & 0 & 0 & 1 & 1 \\
		0 & 0 & 0 & 0 & 1
	\end{matrix}\right)
\end{equation*} \newline
Where matrix $A$ has entries
\begin{align*}
	&d_{1} = 1, d_{2} = 1, d_{3} = \frac{1}{x_{23}}, d_{4} = \frac{1}{x_{23} x_{34}}, d_{5} = \frac{1}{x_{23} x_{34} x_{45}},\\ 
	&a_{12} = 0, a_{13} = 0, a_{14} = 0, a_{15} = 1,\\ 
	&a_{23} = 1, a_{24} = 1, a_{25} = 1,\\ 
	&a_{34} = \frac{1}{x_{23}}, a_{35} = \frac{1}{x_{23}},\\ 
	&a_{45} = \frac{x_{34} x_{45} - x_{35}}{x_{23} x_{34}^{2} x_{45}}
\end{align*}

\begin{equation*}x=\left(\begin{matrix}
		1 & 0 & 0 & 0 & x_{15} \\
		0 & 1 & x_{23} & 0 & 0 \\
		0 & 0 & 1 & x_{34} & x_{35}  \\
		0 & 0 & 0 & 1 & x_{45}  \\
		0 & 0 & 0 & 0 & 1 \\
	\end{matrix}\right),x^A=\left(\begin{matrix}
		1 & 0 & 0 & 1 & 1 \\
		0 & 1 & 1 & 0 & 0 \\
		0 & 0 & 1 & 1 & 0 \\
		0 & 0 & 0 & 1 & 1 \\
		0 & 0 & 0 & 0 & 1
	\end{matrix}\right)
\end{equation*} \newline
Where matrix $A$ has entries
\begin{align*}
	&d_{1} = 1, d_{2} = 1, d_{3} = \frac{1}{x_{23}}, d_{4} = \frac{1}{x_{23} x_{34}}, d_{5} = \frac{1}{x_{23} x_{34} x_{45}},\\ 
	&a_{12} = 0, a_{13} = 0, a_{14} = \frac{x_{15}}{x_{23} x_{34} x_{45}}, a_{15} = 1,\\ 
	&a_{23} = 1, a_{24} = 1, a_{25} = 1,\\ 
	&a_{34} = \frac{1}{x_{23}}, a_{35} = \frac{1}{x_{23}},\\ 
	&a_{45} = \frac{x_{34} x_{45} - x_{35}}{x_{23} x_{34}^{2} x_{45}}
\end{align*}

\begin{equation*}x=\left(\begin{matrix}
		1 & 0 & 0 & x_{14} & 0 \\
		0 & 1 & x_{23} & 0 & 0 \\
		0 & 0 & 1 & x_{34} & x_{35}  \\
		0 & 0 & 0 & 1 & x_{45}  \\
		0 & 0 & 0 & 0 & 1 \\
	\end{matrix}\right),x^A=\left(\begin{matrix}
		1 & 0 & 0 & 0 & 0 \\
		0 & 1 & 1 & 0 & 0 \\
		0 & 0 & 1 & 1 & 0 \\
		0 & 0 & 0 & 1 & 1 \\
		0 & 0 & 0 & 0 & 1
	\end{matrix}\right)
\end{equation*} \newline
Where matrix $A$ has entries
\begin{align*}
	&d_{1} = 1, d_{2} = 1, d_{3} = \frac{1}{x_{23}}, d_{4} = \frac{1}{x_{23} x_{34}}, d_{5} = \frac{1}{x_{23} x_{34} x_{45}},\\ 
	&a_{12} = 0, a_{13} = \frac{x_{14}}{x_{23} x_{34}}, a_{14} = 1, a_{15} = 1,\\ 
	&a_{23} = \frac{x_{23} x_{34}^{2} x_{45} + x_{14} x_{35}}{x_{14} x_{34} x_{45}}, a_{24} = 1, a_{25} = 1,\\ 
	&a_{34} = \frac{x_{23} x_{34}^{2} x_{45} + x_{14} x_{35}}{x_{14} x_{23} x_{34} x_{45}}, a_{35} = \frac{1}{x_{23}},\\ 
	&a_{45} = \frac{1}{x_{14}}
\end{align*}

\begin{equation*}x=\left(\begin{matrix}
		1 & 0 & 0 & x_{14} & x_{15} \\
		0 & 1 & x_{23} & 0 & 0 \\
		0 & 0 & 1 & x_{34} & x_{35}  \\
		0 & 0 & 0 & 1 & x_{45}  \\
		0 & 0 & 0 & 0 & 1 \\
	\end{matrix}\right),x^A=\left(\begin{matrix}
		1 & 0 & 0 & 0 & 0 \\
		0 & 1 & 1 & 0 & 0 \\
		0 & 0 & 1 & 1 & 0 \\
		0 & 0 & 0 & 1 & 1 \\
		0 & 0 & 0 & 0 & 1
	\end{matrix}\right)
\end{equation*} \newline
Where matrix $A$ has entries
\begin{align*}
	&d_{1} = 1, d_{2} = 1, d_{3} = \frac{1}{x_{23}}, d_{4} = \frac{1}{x_{23} x_{34}}, d_{5} = \frac{1}{x_{23} x_{34} x_{45}},\\ 
	&a_{12} = 0, a_{13} = \frac{x_{14}}{x_{23} x_{34}}, a_{14} = 1, a_{15} = 1,\\ 
	&a_{23} = \frac{x_{23} x_{34}^{2} x_{45} - x_{15} x_{34} + x_{14} x_{35}}{x_{14} x_{34} x_{45}}, a_{24} = 1, a_{25} = 1,\\ 
	&a_{34} = \frac{x_{23} x_{34}^{2} x_{45} - x_{15} x_{34} + x_{14} x_{35}}{x_{14} x_{23} x_{34} x_{45}}, a_{35} = \frac{1}{x_{23}},\\ 
	&a_{45} = \frac{x_{23} x_{34} x_{45} - x_{15}}{x_{14} x_{23} x_{34} x_{45}}
\end{align*}

\begin{equation*}x=\left(\begin{matrix}
		1 & 0 & x_{13} & 0 & 0 \\
		0 & 1 & x_{23} & 0 & 0 \\
		0 & 0 & 1 & x_{34} & x_{35}  \\
		0 & 0 & 0 & 1 & x_{45}  \\
		0 & 0 & 0 & 0 & 1 \\
	\end{matrix}\right),x^A=\left(\begin{matrix}
		1 & 0 & 0 & 0 & 0 \\
		0 & 1 & 1 & 0 & 0 \\
		0 & 0 & 1 & 1 & 0 \\
		0 & 0 & 0 & 1 & 1 \\
		0 & 0 & 0 & 0 & 1
	\end{matrix}\right)
\end{equation*} \newline
Where matrix $A$ has entries
\begin{align*}
	&d_{1} = 1, d_{2} = 1, d_{3} = \frac{1}{x_{23}}, d_{4} = \frac{1}{x_{23} x_{34}}, d_{5} = \frac{1}{x_{23} x_{34} x_{45}},\\ 
	&a_{12} = \frac{x_{13}}{x_{23}}, a_{13} = 1, a_{14} = 1, a_{15} = 1,\\ 
	&a_{23} = \frac{x_{23}}{x_{13}}, a_{24} = \frac{x_{23}}{x_{13}}, a_{25} = 1,\\ 
	&a_{34} = \frac{1}{x_{13}}, a_{35} = \frac{1}{x_{13}},\\ 
	&a_{45} = \frac{x_{23} x_{34} x_{45} - x_{13} x_{35}}{x_{13} x_{23} x_{34}^{2} x_{45}}
\end{align*}

\begin{equation*}x=\left(\begin{matrix}
		1 & 0 & x_{13} & 0 & x_{15} \\
		0 & 1 & x_{23} & 0 & 0 \\
		0 & 0 & 1 & x_{34} & x_{35}  \\
		0 & 0 & 0 & 1 & x_{45}  \\
		0 & 0 & 0 & 0 & 1 \\
	\end{matrix}\right),x^A=\left(\begin{matrix}
		1 & 0 & 0 & 0 & 0 \\
		0 & 1 & 1 & 0 & 0 \\
		0 & 0 & 1 & 1 & 0 \\
		0 & 0 & 0 & 1 & 1 \\
		0 & 0 & 0 & 0 & 1
	\end{matrix}\right)
\end{equation*} \newline
Where matrix $A$ has entries
\begin{align*}
	&d_{1} = 1, d_{2} = 1, d_{3} = \frac{1}{x_{23}}, d_{4} = \frac{1}{x_{23} x_{34}}, d_{5} = \frac{1}{x_{23} x_{34} x_{45}},\\ 
	&a_{12} = \frac{x_{13}}{x_{23}}, a_{13} = 1, a_{14} = 1, a_{15} = 1,\\ 
	&a_{23} = \frac{x_{23}}{x_{13}}, a_{24} = \frac{x_{23} x_{34} x_{45} - x_{15}}{x_{13} x_{34} x_{45}}, a_{25} = 1,\\ 
	&a_{34} = \frac{1}{x_{13}}, a_{35} = \frac{x_{23} x_{34} x_{45} - x_{15}}{x_{13} x_{23} x_{34} x_{45}},\\ 
	&a_{45} = \frac{x_{23} x_{34} x_{45} - x_{13} x_{35}}{x_{13} x_{23} x_{34}^{2} x_{45}}
\end{align*}

\begin{equation*}x=\left(\begin{matrix}
		1 & 0 & x_{13} & x_{14} & 0 \\
		0 & 1 & x_{23} & 0 & 0 \\
		0 & 0 & 1 & x_{34} & x_{35}  \\
		0 & 0 & 0 & 1 & x_{45}  \\
		0 & 0 & 0 & 0 & 1 \\
	\end{matrix}\right),x^A=\left(\begin{matrix}
		1 & 0 & 0 & 0 & 0 \\
		0 & 1 & 1 & 0 & 0 \\
		0 & 0 & 1 & 1 & 0 \\
		0 & 0 & 0 & 1 & 1 \\
		0 & 0 & 0 & 0 & 1
	\end{matrix}\right)
\end{equation*} \newline
Where matrix $A$ has entries
\begin{align*}
	&d_{1} = 1, d_{2} = 1, d_{3} = \frac{1}{x_{23}}, d_{4} = \frac{1}{x_{23} x_{34}}, d_{5} = \frac{1}{x_{23} x_{34} x_{45}},\\ 
	&a_{12} = \frac{x_{13}}{x_{23}}, a_{13} = 1, a_{14} = 1, a_{15} = 1,\\ 
	&a_{23} = \frac{x_{23} x_{34} - x_{14}}{x_{13} x_{34}}, a_{24} = \frac{x_{13} x_{14} x_{35} + {\left(x_{13} x_{23} x_{34}^{2} - x_{14} x_{23} x_{34} + x_{14}^{2}\right)} x_{45}}{x_{13}^{2} x_{34}^{2} x_{45}}, a_{25} = 1,\\ 
	&a_{34} = \frac{x_{23} x_{34} - x_{14}}{x_{13} x_{23} x_{34}}, a_{35} = \frac{x_{13} x_{14} x_{35} + {\left(x_{13} x_{23} x_{34}^{2} - x_{14} x_{23} x_{34} + x_{14}^{2}\right)} x_{45}}{x_{13}^{2} x_{23} x_{34}^{2} x_{45}},\\ 
	&a_{45} = -\frac{x_{13} x_{35} - {\left(x_{23} x_{34} - x_{14}\right)} x_{45}}{x_{13} x_{23} x_{34}^{2} x_{45}}
\end{align*}

\begin{equation*}x=\left(\begin{matrix}
		1 & 0 & x_{13} & x_{14} & x_{15} \\
		0 & 1 & x_{23} & 0 & 0 \\
		0 & 0 & 1 & x_{34} & x_{35}  \\
		0 & 0 & 0 & 1 & x_{45}  \\
		0 & 0 & 0 & 0 & 1 \\
	\end{matrix}\right),x^A=\left(\begin{matrix}
		1 & 0 & 0 & 0 & 0 \\
		0 & 1 & 1 & 0 & 0 \\
		0 & 0 & 1 & 1 & 0 \\
		0 & 0 & 0 & 1 & 1 \\
		0 & 0 & 0 & 0 & 1
	\end{matrix}\right)
\end{equation*} \newline
Where matrix $A$ has entries
\begin{align*}
	&d_{1} = 1, d_{2} = 1, d_{3} = \frac{1}{x_{23}}, d_{4} = \frac{1}{x_{23} x_{34}}, d_{5} = \frac{1}{x_{23} x_{34} x_{45}},\\ 
	&a_{12} = \frac{x_{13}}{x_{23}}, a_{13} = 1, a_{14} = 1, a_{15} = 1,\\ 
	&a_{23} = \frac{x_{23} x_{34} - x_{14}}{x_{13} x_{34}}, a_{24} = -\frac{x_{13} x_{15} x_{34} - x_{13} x_{14} x_{35} - {\left(x_{13} x_{23} x_{34}^{2} - x_{14} x_{23} x_{34} + x_{14}^{2}\right)} x_{45}}{x_{13}^{2} x_{34}^{2} x_{45}}, a_{25} = 1,\\ 
	&a_{34} = \frac{x_{23} x_{34} - x_{14}}{x_{13} x_{23} x_{34}}, a_{35} = -\frac{x_{13} x_{15} x_{34} - x_{13} x_{14} x_{35} - {\left(x_{13} x_{23} x_{34}^{2} - x_{14} x_{23} x_{34} + x_{14}^{2}\right)} x_{45}}{x_{13}^{2} x_{23} x_{34}^{2} x_{45}},\\ 
	&a_{45} = -\frac{x_{13} x_{35} - {\left(x_{23} x_{34} - x_{14}\right)} x_{45}}{x_{13} x_{23} x_{34}^{2} x_{45}}
\end{align*}

\begin{equation*}x=\left(
\right)
\end{equation*} \newline
Where matrix $A$ has entries
\begin{align*}
	&d_{1} = 1, d_{2} = 1, d_{3} = \frac{1}{x_{23}}, d_{4} = \frac{1}{x_{23} x_{34}}, d_{5} = \frac{1}{x_{23} x_{34} x_{45}},\\ 
	&a_{12} = 0, a_{13} = \frac{x_{14}}{x_{23} x_{34}}, a_{14} = 1, a_{15} = 1,\\ 
	&a_{23} = \frac{x_{23} x_{34}^{2} x_{45} + x_{14} x_{35}}{x_{14} x_{34} x_{45}}, a_{24} = 1, a_{25} = 1,\\ 
	&a_{34} = \frac{x_{23} x_{34}^{2} x_{45} + x_{14} x_{35}}{x_{14} x_{23} x_{34} x_{45}}, a_{35} = \frac{x_{23} x_{34} x_{45} - x_{25}}{x_{23}^{2} x_{34} x_{45}},\\ 
	&a_{45} = \frac{1}{x_{14}}
\end{align*}

\begin{equation*}x=\left(\begin{matrix}
		1 & 0 & 0 & x_{14} & x_{15} \\
		0 & 1 & x_{23} & 0 & x_{25} \\
		0 & 0 & 1 & x_{34} & x_{35}  \\
		0 & 0 & 0 & 1 & x_{45}  \\
		0 & 0 & 0 & 0 & 1 \\
	\end{matrix}\right),x^A=\left(\begin{matrix}
		1 & 0 & 0 & 0 & 0 \\
		0 & 1 & 1 & 0 & 0 \\
		0 & 0 & 1 & 1 & 0 \\
		0 & 0 & 0 & 1 & 1 \\
		0 & 0 & 0 & 0 & 1
	\end{matrix}\right)
\end{equation*} \newline
Where matrix $A$ has entries
\begin{align*}
	&d_{1} = 1, d_{2} = 1, d_{3} = \frac{1}{x_{23}}, d_{4} = \frac{1}{x_{23} x_{34}}, d_{5} = \frac{1}{x_{23} x_{34} x_{45}},\\ 
	&a_{12} = 0, a_{13} = \frac{x_{14}}{x_{23} x_{34}}, a_{14} = 1, a_{15} = 1,\\ 
	&a_{23} = \frac{x_{23} x_{34}^{2} x_{45} - x_{15} x_{34} + x_{14} x_{35}}{x_{14} x_{34} x_{45}}, a_{24} = 1, a_{25} = 1,\\ 
	&a_{34} = \frac{x_{23} x_{34}^{2} x_{45} - x_{15} x_{34} + x_{14} x_{35}}{x_{14} x_{23} x_{34} x_{45}}, a_{35} = \frac{x_{23} x_{34} x_{45} - x_{25}}{x_{23}^{2} x_{34} x_{45}},\\ 
	&a_{45} = \frac{x_{23} x_{34} x_{45} - x_{15}}{x_{14} x_{23} x_{34} x_{45}}
\end{align*}

\begin{equation*}x=\left(\begin{matrix}
		1 & 0 & x_{13} & 0 & 0 \\
		0 & 1 & x_{23} & 0 & x_{25} \\
		0 & 0 & 1 & x_{34} & x_{35}  \\
		0 & 0 & 0 & 1 & x_{45}  \\
		0 & 0 & 0 & 0 & 1 \\
	\end{matrix}\right),x^A=\left(\begin{matrix}
		1 & 0 & 0 & 0 & 0 \\
		0 & 1 & 1 & 0 & 0 \\
		0 & 0 & 1 & 1 & 0 \\
		0 & 0 & 0 & 1 & 1 \\
		0 & 0 & 0 & 0 & 1
	\end{matrix}\right)
\end{equation*} \newline
Where matrix $A$ has entries
\begin{align*}
	&d_{1} = 1, d_{2} = 1, d_{3} = \frac{1}{x_{23}}, d_{4} = \frac{1}{x_{23} x_{34}}, d_{5} = \frac{1}{x_{23} x_{34} x_{45}},\\ 
	&a_{12} = \frac{x_{13}}{x_{23}}, a_{13} = 1, a_{14} = 1, a_{15} = 1,\\ 
	&a_{23} = \frac{x_{23}}{x_{13}}, a_{24} = \frac{x_{23}^{2} x_{34} x_{45} + x_{13} x_{25}}{x_{13} x_{23} x_{34} x_{45}}, a_{25} = 1,\\ 
	&a_{34} = \frac{1}{x_{13}}, a_{35} = \frac{1}{x_{13}},\\ 
	&a_{45} = \frac{x_{23} x_{34} x_{45} - x_{13} x_{35}}{x_{13} x_{23} x_{34}^{2} x_{45}}
\end{align*}

\begin{equation*}x=\left(\begin{matrix}
		1 & 0 & x_{13} & 0 & x_{15} \\
		0 & 1 & x_{23} & 0 & x_{25} \\
		0 & 0 & 1 & x_{34} & x_{35}  \\
		0 & 0 & 0 & 1 & x_{45}  \\
		0 & 0 & 0 & 0 & 1 \\
	\end{matrix}\right),x^A=\left(\begin{matrix}
		1 & 0 & 0 & 0 & 0 \\
		0 & 1 & 1 & 0 & 0 \\
		0 & 0 & 1 & 1 & 0 \\
		0 & 0 & 0 & 1 & 1 \\
		0 & 0 & 0 & 0 & 1
	\end{matrix}\right)
\end{equation*} \newline
Where matrix $A$ has entries
\begin{align*}
	&d_{1} = 1, d_{2} = 1, d_{3} = \frac{1}{x_{23}}, d_{4} = \frac{1}{x_{23} x_{34}}, d_{5} = \frac{1}{x_{23} x_{34} x_{45}},\\ 
	&a_{12} = \frac{x_{13}}{x_{23}}, a_{13} = 1, a_{14} = 1, a_{15} = 1,\\ 
	&a_{23} = \frac{x_{23}}{x_{13}}, a_{24} = \frac{x_{23}^{2} x_{34} x_{45} - x_{15} x_{23} + x_{13} x_{25}}{x_{13} x_{23} x_{34} x_{45}}, a_{25} = 1,\\ 
	&a_{34} = \frac{1}{x_{13}}, a_{35} = \frac{x_{23} x_{34} x_{45} - x_{15}}{x_{13} x_{23} x_{34} x_{45}},\\ 
	&a_{45} = \frac{x_{23} x_{34} x_{45} - x_{13} x_{35}}{x_{13} x_{23} x_{34}^{2} x_{45}}
\end{align*}

\begin{equation*}x=\left(\begin{matrix}
		1 & 0 & x_{13} & x_{14} & 0 \\
		0 & 1 & x_{23} & 0 & x_{25} \\
		0 & 0 & 1 & x_{34} & x_{35}  \\
		0 & 0 & 0 & 1 & x_{45}  \\
		0 & 0 & 0 & 0 & 1 \\
	\end{matrix}\right),x^A=\left(\begin{matrix}
		1 & 0 & 0 & 0 & 0 \\
		0 & 1 & 1 & 0 & 0 \\
		0 & 0 & 1 & 1 & 0 \\
		0 & 0 & 0 & 1 & 1 \\
		0 & 0 & 0 & 0 & 1
	\end{matrix}\right)
\end{equation*} \newline
Where matrix $A$ has entries
\begin{align*}
	&d_{1} = 1, d_{2} = 1, d_{3} = \frac{1}{x_{23}}, d_{4} = \frac{1}{x_{23} x_{34}}, d_{5} = \frac{1}{x_{23} x_{34} x_{45}},\\ 
	&a_{12} = \frac{x_{13}}{x_{23}}, a_{13} = 1, a_{14} = 1, a_{15} = 1,\\ 
	&a_{23} = \frac{x_{23} x_{34} - x_{14}}{x_{13} x_{34}}, a_{24} = \frac{x_{13}^{2} x_{25} x_{34} + x_{13} x_{14} x_{23} x_{35} + {\left(x_{13} x_{23}^{2} x_{34}^{2} - x_{14} x_{23}^{2} x_{34} + x_{14}^{2} x_{23}\right)} x_{45}}{x_{13}^{2} x_{23} x_{34}^{2} x_{45}}, a_{25} = 1,\\ 
	&a_{34} = \frac{x_{23} x_{34} - x_{14}}{x_{13} x_{23} x_{34}}, a_{35} = \frac{x_{13} x_{14} x_{35} + {\left(x_{13} x_{23} x_{34}^{2} - x_{14} x_{23} x_{34} + x_{14}^{2}\right)} x_{45}}{x_{13}^{2} x_{23} x_{34}^{2} x_{45}},\\ 
	&a_{45} = -\frac{x_{13} x_{35} - {\left(x_{23} x_{34} - x_{14}\right)} x_{45}}{x_{13} x_{23} x_{34}^{2} x_{45}}
\end{align*}

\begin{equation*}x=\left(\begin{matrix}
		1 & 0 & x_{13} & x_{14} & x_{15} \\
		0 & 1 & x_{23} & 0 & x_{25} \\
		0 & 0 & 1 & x_{34} & x_{35}  \\
		0 & 0 & 0 & 1 & x_{45}  \\
		0 & 0 & 0 & 0 & 1 \\
	\end{matrix}\right),x^A=\left(\begin{matrix}
		1 & 0 & 0 & 0 & 0 \\
		0 & 1 & 1 & 0 & 0 \\
		0 & 0 & 1 & 1 & 0 \\
		0 & 0 & 0 & 1 & 1 \\
		0 & 0 & 0 & 0 & 1
	\end{matrix}\right)
\end{equation*} \newline
Where matrix $A$ has entries
\begin{align*}
	&d_{1} = 1, d_{2} = 1, d_{3} = \frac{1}{x_{23}}, d_{4} = \frac{1}{x_{23} x_{34}}, d_{5} = \frac{1}{x_{23} x_{34} x_{45}},\\ 
	&a_{12} = \frac{x_{13}}{x_{23}}, a_{13} = 1, a_{14} = 1, a_{15} = 1,\\ 
	&a_{23} = \frac{x_{23} x_{34} - x_{14}}{x_{13} x_{34}}, a_{24} = \frac{x_{13} x_{14} x_{23} x_{35} - {\left(x_{13} x_{15} x_{23} - x_{13}^{2} x_{25}\right)} x_{34} + {\left(x_{13} x_{23}^{2} x_{34}^{2} - x_{14} x_{23}^{2} x_{34} + x_{14}^{2} x_{23}\right)} x_{45}}{x_{13}^{2} x_{23} x_{34}^{2} x_{45}}, a_{25} = 1,\\ 
	&a_{34} = \frac{x_{23} x_{34} - x_{14}}{x_{13} x_{23} x_{34}}, a_{35} = -\frac{x_{13} x_{15} x_{34} - x_{13} x_{14} x_{35} - {\left(x_{13} x_{23} x_{34}^{2} - x_{14} x_{23} x_{34} + x_{14}^{2}\right)} x_{45}}{x_{13}^{2} x_{23} x_{34}^{2} x_{45}},\\ 
	&a_{45} = -\frac{x_{13} x_{35} - {\left(x_{23} x_{34} - x_{14}\right)} x_{45}}{x_{13} x_{23} x_{34}^{2} x_{45}}
\end{align*}

\begin{equation*}x=\left(\begin{matrix}
		1 & 0 & 0 & 0 & 0 \\
		0 & 1 & x_{23} & x_{24} & 0 \\
		0 & 0 & 1 & x_{34} & x_{35}  \\
		0 & 0 & 0 & 1 & x_{45}  \\
		0 & 0 & 0 & 0 & 1 \\
	\end{matrix}\right),x^A=\left(\begin{matrix}
		1 & 0 & 0 & 0 & 0 \\
		0 & 1 & 1 & 0 & 0 \\
		0 & 0 & 1 & 1 & 0 \\
		0 & 0 & 0 & 1 & 1 \\
		0 & 0 & 0 & 0 & 1
	\end{matrix}\right)
\end{equation*} \newline
Where matrix $A$ has entries
\begin{align*}
	&d_{1} = 1, d_{2} = 1, d_{3} = \frac{1}{x_{23}}, d_{4} = \frac{1}{x_{23} x_{34}}, d_{5} = \frac{1}{x_{23} x_{34} x_{45}},\\ 
	&a_{12} = 0, a_{13} = 0, a_{14} = 0, a_{15} = 1,\\ 
	&a_{23} = 1, a_{24} = 1, a_{25} = 1,\\ 
	&a_{34} = \frac{x_{23} x_{34} - x_{24}}{x_{23}^{2} x_{34}}, a_{35} = \frac{x_{23} x_{24} x_{35} + {\left(x_{23}^{2} x_{34}^{2} - x_{23} x_{24} x_{34} + x_{24}^{2}\right)} x_{45}}{x_{23}^{3} x_{34}^{2} x_{45}},\\ 
	&a_{45} = -\frac{x_{23} x_{35} - {\left(x_{23} x_{34} - x_{24}\right)} x_{45}}{x_{23}^{2} x_{34}^{2} x_{45}}
\end{align*}

\begin{equation*}x=\left(\begin{matrix}
		1 & 0 & 0 & 0 & x_{15} \\
		0 & 1 & x_{23} & x_{24} & 0 \\
		0 & 0 & 1 & x_{34} & x_{35}  \\
		0 & 0 & 0 & 1 & x_{45}  \\
		0 & 0 & 0 & 0 & 1 \\
	\end{matrix}\right),x^A=\left(\begin{matrix}
		1 & 0 & 0 & 0 & 0 \\
		0 & 1 & 1 & 0 & 0 \\
		0 & 0 & 1 & 1 & 0 \\
		0 & 0 & 0 & 1 & 1 \\
		0 & 0 & 0 & 0 & 1
	\end{matrix}\right)
\end{equation*} \newline
Where matrix $A$ has entries
\begin{align*}
	&d_{1} = 1, d_{2} = 1, d_{3} = \frac{1}{x_{23}}, d_{4} = \frac{1}{x_{23} x_{34}}, d_{5} = \frac{1}{x_{23} x_{34} x_{45}},\\ 
	&a_{12} = 0, a_{13} = 0, a_{14} = \frac{x_{15}}{x_{23} x_{34} x_{45}}, a_{15} = 1,\\ 
	&a_{23} = 1, a_{24} = 1, a_{25} = 1,\\ 
	&a_{34} = \frac{x_{23} x_{34} - x_{24}}{x_{23}^{2} x_{34}}, a_{35} = \frac{x_{23} x_{24} x_{35} + {\left(x_{23}^{2} x_{34}^{2} - x_{23} x_{24} x_{34} + x_{24}^{2}\right)} x_{45}}{x_{23}^{3} x_{34}^{2} x_{45}},\\ 
	&a_{45} = -\frac{x_{23} x_{35} - {\left(x_{23} x_{34} - x_{24}\right)} x_{45}}{x_{23}^{2} x_{34}^{2} x_{45}}
\end{align*}

\begin{equation*}x=\left(\begin{matrix}
		1 & 0 & 0 & x_{14} & 0 \\
		0 & 1 & x_{23} & x_{24} & 0 \\
		0 & 0 & 1 & x_{34} & x_{35}  \\
		0 & 0 & 0 & 1 & x_{45}  \\
		0 & 0 & 0 & 0 & 1 \\
	\end{matrix}\right),x^A=\left(\begin{matrix}
		1 & 0 & 0 & 0 & 0 \\
		0 & 1 & 1 & 0 & 0 \\
		0 & 0 & 1 & 1 & 0 \\
		0 & 0 & 0 & 1 & 1 \\
		0 & 0 & 0 & 0 & 1
	\end{matrix}\right)
\end{equation*} \newline
Where matrix $A$ has entries
\begin{align*}
	&d_{1} = 1, d_{2} = 1, d_{3} = \frac{1}{x_{23}}, d_{4} = \frac{1}{x_{23} x_{34}}, d_{5} = \frac{1}{x_{23} x_{34} x_{45}},\\ 
	&a_{12} = 0, a_{13} = \frac{x_{14}}{x_{23} x_{34}}, a_{14} = 1, a_{15} = 1,\\ 
	&a_{23} = \frac{x_{14} x_{23} x_{35} + {\left(x_{23}^{2} x_{34}^{2} + x_{14} x_{24}\right)} x_{45}}{x_{14} x_{23} x_{34} x_{45}}, a_{24} = 1, a_{25} = 1,\\ 
	&a_{34} = \frac{x_{23} x_{34}^{2} x_{45} + x_{14} x_{35}}{x_{14} x_{23} x_{34} x_{45}}, a_{35} = \frac{x_{14} - x_{24}}{x_{14} x_{23}},\\ 
	&a_{45} = \frac{1}{x_{14}}
\end{align*}

\begin{equation*}x=\left(\begin{matrix}
		1 & 0 & 0 & x_{14} & x_{15} \\
		0 & 1 & x_{23} & x_{24} & 0 \\
		0 & 0 & 1 & x_{34} & x_{35}  \\
		0 & 0 & 0 & 1 & x_{45}  \\
		0 & 0 & 0 & 0 & 1 \\
	\end{matrix}\right),x^A=\left(\begin{matrix}
		1 & 0 & 0 & 0 & 0 \\
		0 & 1 & 1 & 0 & 0 \\
		0 & 0 & 1 & 1 & 0 \\
		0 & 0 & 0 & 1 & 1 \\
		0 & 0 & 0 & 0 & 1
	\end{matrix}\right)
\end{equation*} \newline
Where matrix $A$ has entries
\begin{align*}
	&d_{1} = 1, d_{2} = 1, d_{3} = \frac{1}{x_{23}}, d_{4} = \frac{1}{x_{23} x_{34}}, d_{5} = \frac{1}{x_{23} x_{34} x_{45}},\\ 
	&a_{12} = 0, a_{13} = \frac{x_{14}}{x_{23} x_{34}}, a_{14} = 1, a_{15} = 1,\\ 
	&a_{23} = -\frac{x_{15} x_{23} x_{34} - x_{14} x_{23} x_{35} - {\left(x_{23}^{2} x_{34}^{2} + x_{14} x_{24}\right)} x_{45}}{x_{14} x_{23} x_{34} x_{45}}, a_{24} = 1, a_{25} = 1,\\ 
	&a_{34} = \frac{x_{23} x_{34}^{2} x_{45} - x_{15} x_{34} + x_{14} x_{35}}{x_{14} x_{23} x_{34} x_{45}}, a_{35} = \frac{x_{15} x_{24} + {\left(x_{14} x_{23} - x_{23} x_{24}\right)} x_{34} x_{45}}{x_{14} x_{23}^{2} x_{34} x_{45}},\\ 
	&a_{45} = \frac{x_{23} x_{34} x_{45} - x_{15}}{x_{14} x_{23} x_{34} x_{45}}
\end{align*}

\begin{equation*}x=\left(\begin{matrix}
		1 & 0 & x_{13} & 0 & 0 \\
		0 & 1 & x_{23} & x_{24} & 0 \\
		0 & 0 & 1 & x_{34} & x_{35}  \\
		0 & 0 & 0 & 1 & x_{45}  \\
		0 & 0 & 0 & 0 & 1 \\
	\end{matrix}\right),x^A=\left(\begin{matrix}
		1 & 0 & 0 & 0 & 0 \\
		0 & 1 & 1 & 0 & 0 \\
		0 & 0 & 1 & 1 & 0 \\
		0 & 0 & 0 & 1 & 1 \\
		0 & 0 & 0 & 0 & 1
	\end{matrix}\right)
\end{equation*} \newline
Where matrix $A$ has entries
\begin{align*}
	&d_{1} = 1, d_{2} = 1, d_{3} = \frac{1}{x_{23}}, d_{4} = \frac{1}{x_{23} x_{34}}, d_{5} = \frac{1}{x_{23} x_{34} x_{45}},\\ 
	&a_{12} = \frac{x_{13}}{x_{23}}, a_{13} = 1, a_{14} = 1, a_{15} = 1,\\ 
	&a_{23} = \frac{x_{23}^{2} x_{34} + x_{13} x_{24}}{x_{13} x_{23} x_{34}}, a_{24} = -\frac{x_{13} x_{24} x_{35} - {\left(x_{23}^{2} x_{34}^{2} + x_{23} x_{24} x_{34}\right)} x_{45}}{x_{13} x_{23} x_{34}^{2} x_{45}}, a_{25} = 1,\\ 
	&a_{34} = \frac{1}{x_{13}}, a_{35} = \frac{1}{x_{13}},\\ 
	&a_{45} = \frac{x_{23} x_{34} x_{45} - x_{13} x_{35}}{x_{13} x_{23} x_{34}^{2} x_{45}}
\end{align*}

\begin{equation*}x=\left(\begin{matrix}
		1 & 0 & x_{13} & 0 & x_{15} \\
		0 & 1 & x_{23} & x_{24} & 0 \\
		0 & 0 & 1 & x_{34} & x_{35}  \\
		0 & 0 & 0 & 1 & x_{45}  \\
		0 & 0 & 0 & 0 & 1 \\
	\end{matrix}\right),x^A=\left(\begin{matrix}
		1 & 0 & 0 & 0 & 0 \\
		0 & 1 & 1 & 0 & 0 \\
		0 & 0 & 1 & 1 & 0 \\
		0 & 0 & 0 & 1 & 1 \\
		0 & 0 & 0 & 0 & 1
	\end{matrix}\right)
\end{equation*} \newline
Where matrix $A$ has entries
\begin{align*}
	&d_{1} = 1, d_{2} = 1, d_{3} = \frac{1}{x_{23}}, d_{4} = \frac{1}{x_{23} x_{34}}, d_{5} = \frac{1}{x_{23} x_{34} x_{45}},\\ 
	&a_{12} = \frac{x_{13}}{x_{23}}, a_{13} = 1, a_{14} = 1, a_{15} = 1,\\ 
	&a_{23} = \frac{x_{23}^{2} x_{34} + x_{13} x_{24}}{x_{13} x_{23} x_{34}}, a_{24} = -\frac{x_{15} x_{23} x_{34} + x_{13} x_{24} x_{35} - {\left(x_{23}^{2} x_{34}^{2} + x_{23} x_{24} x_{34}\right)} x_{45}}{x_{13} x_{23} x_{34}^{2} x_{45}}, a_{25} = 1,\\ 
	&a_{34} = \frac{1}{x_{13}}, a_{35} = \frac{x_{23} x_{34} x_{45} - x_{15}}{x_{13} x_{23} x_{34} x_{45}},\\ 
	&a_{45} = \frac{x_{23} x_{34} x_{45} - x_{13} x_{35}}{x_{13} x_{23} x_{34}^{2} x_{45}}
\end{align*}

\begin{equation*}x=\left(\begin{matrix}
		1 & 0 & x_{13} & x_{14} & 0 \\
		0 & 1 & x_{23} & x_{24} & 0 \\
		0 & 0 & 1 & x_{34} & x_{35}  \\
		0 & 0 & 0 & 1 & x_{45}  \\
		0 & 0 & 0 & 0 & 1 \\
	\end{matrix}\right),x^A=\left(\begin{matrix}
		1 & 0 & 0 & 0 & 0 \\
		0 & 1 & 1 & 0 & 0 \\
		0 & 0 & 1 & 1 & 0 \\
		0 & 0 & 0 & 1 & 1 \\
		0 & 0 & 0 & 0 & 1
	\end{matrix}\right)
\end{equation*} \newline
Where matrix $A$ has entries
\begin{align*}
	&d_{1} = 1, d_{2} = 1, d_{3} = \frac{1}{x_{23}}, d_{4} = \frac{1}{x_{23} x_{34}}, d_{5} = \frac{1}{x_{23} x_{34} x_{45}},\\ 
	&a_{12} = \frac{x_{13}}{x_{23}}, a_{13} = 1, a_{14} = 1, a_{15} = 1,\\ 
	&a_{23} = \frac{x_{23}^{2} x_{34} - x_{14} x_{23} + x_{13} x_{24}}{x_{13} x_{23} x_{34}},\\
	& a_{24} = \frac{{\left(x_{13} x_{14} x_{23} - x_{13}^{2} x_{24}\right)} x_{35} + {\left(x_{13} x_{23}^{2} x_{34}^{2} + x_{14}^{2} x_{23} - x_{13} x_{14} x_{24} - {\left(x_{14} x_{23}^{2} - x_{13} x_{23} x_{24}\right)} x_{34}\right)} x_{45}}{x_{13}^{2} x_{23} x_{34}^{2} x_{45}}, a_{25} = 1,\\ 
	&a_{34} = \frac{x_{23} x_{34} - x_{14}}{x_{13} x_{23} x_{34}}, a_{35} = \frac{x_{13} x_{14} x_{35} + {\left(x_{13} x_{23} x_{34}^{2} - x_{14} x_{23} x_{34} + x_{14}^{2}\right)} x_{45}}{x_{13}^{2} x_{23} x_{34}^{2} x_{45}},\\ 
	&a_{45} = -\frac{x_{13} x_{35} - {\left(x_{23} x_{34} - x_{14}\right)} x_{45}}{x_{13} x_{23} x_{34}^{2} x_{45}}
\end{align*}

\begin{equation*}x=\left(\begin{matrix}
		0 & 0 & x_{13} & x_{14} & x_{15} \\
		0 & 1 & x_{23} & x_{24} & 0 \\
		0 & 0 & 1 & x_{34} & x_{35}  \\
		0 & 0 & 0 & 1 & x_{45}  \\
		0 & 0 & 0 & 0 & 1 \\
	\end{matrix}\right),x^A=\left(\begin{matrix}
		1 & 0 & 0 & 0 & 0 \\
		0 & 1 & 1 & 0 & 0 \\
		0 & 0 & 1 & 1 & 0 \\
		0 & 0 & 0 & 1 & 1 \\
		0 & 0 & 0 & 0 & 1
	\end{matrix}\right)
\end{equation*} \newline
Where matrix $A$ has entries
\begin{align*}
	&d_{1} = 1, d_{2} = 1, d_{3} = \frac{1}{x_{23}}, d_{4} = \frac{1}{x_{23} x_{34}}, d_{5} = \frac{1}{x_{23} x_{34} x_{45}},\\ 
	&a_{12} = \frac{x_{13}}{x_{23}}, a_{13} = 1, a_{14} = 1, a_{15} = 1,\\ 
	&a_{23} = \frac{x_{23}^{2} x_{34} - x_{14} x_{23} + x_{13} x_{24}}{x_{13} x_{23} x_{34}},\\
	& a_{24} = -\frac{x_{13} x_{15} x_{23} x_{34} - {\left(x_{13} x_{14} x_{23} - x_{13}^{2} x_{24}\right)} x_{35} - {\left(x_{13} x_{23}^{2} x_{34}^{2} + x_{14}^{2} x_{23} - x_{13} x_{14} x_{24} - {\left(x_{14} x_{23}^{2} - x_{13} x_{23} x_{24}\right)} x_{34}\right)} x_{45}}{x_{13}^{2} x_{23} x_{34}^{2} x_{45}}, a_{25} = 1,\\ 
	&a_{34} = \frac{x_{23} x_{34} - x_{14}}{x_{13} x_{23} x_{34}}, a_{35} = -\frac{x_{13} x_{15} x_{34} - x_{13} x_{14} x_{35} - {\left(x_{13} x_{23} x_{34}^{2} - x_{14} x_{23} x_{34} + x_{14}^{2}\right)} x_{45}}{x_{13}^{2} x_{23} x_{34}^{2} x_{45}},\\ 
	&a_{45} = -\frac{x_{13} x_{35} - {\left(x_{23} x_{34} - x_{14}\right)} x_{45}}{x_{13} x_{23} x_{34}^{2} x_{45}}
\end{align*}

\begin{equation*}x=\left(\begin{matrix}
		1 & 0 & 0 & 0 & 0 \\
		0 & 1 & x_{23} & x_{24} & x_{25} \\
		0 & 0 & 1 & x_{34} & x_{35}  \\
		0 & 0 & 0 & 1 & x_{45}  \\
		0 & 0 & 0 & 0 & 1 \\
	\end{matrix}\right),x^A=\left(\begin{matrix}
		1 & 0 & 0 & 0 & 0 \\
		0 & 1 & 1 & 0 & 0 \\
		0 & 0 & 1 & 1 & 0 \\
		0 & 0 & 0 & 1 & 1 \\
		0 & 0 & 0 & 0 & 1
	\end{matrix}\right)
\end{equation*} \newline
Where matrix $A$ has entries
\begin{align*}
	&d_{1} = 1, d_{2} = 1, d_{3} = \frac{1}{x_{23}}, d_{4} = \frac{1}{x_{23} x_{34}}, d_{5} = \frac{1}{x_{23} x_{34} x_{45}},\\ 
	&a_{12} = 0, a_{13} = 0, a_{14} = 0, a_{15} = 1,\\ 
	&a_{23} = 1, a_{24} = 1, a_{25} = 1,\\ 
	&a_{34} = \frac{x_{23} x_{34} - x_{24}}{x_{23}^{2} x_{34}}, a_{35} = -\frac{x_{23} x_{25} x_{34} - x_{23} x_{24} x_{35} - {\left(x_{23}^{2} x_{34}^{2} - x_{23} x_{24} x_{34} + x_{24}^{2}\right)} x_{45}}{x_{23}^{3} x_{34}^{2} x_{45}},\\ 
	&a_{45} = -\frac{x_{23} x_{35} - {\left(x_{23} x_{34} - x_{24}\right)} x_{45}}{x_{23}^{2} x_{34}^{2} x_{45}}
\end{align*}

\begin{equation*}x=\left(\begin{matrix}
		1 & 0 & 0 & 0 & x_{15} \\
		0 & 1 & x_{23} & x_{24} & x_{25} \\
		0 & 0 & 1 & x_{34} & x_{35}  \\
		0 & 0 & 0 & 1 & x_{45}  \\
		0 & 0 & 0 & 0 & 1 \\
	\end{matrix}\right),x^A=\left(\begin{matrix}
		1 & 0 & 0 & 0 & 0 \\
		0 & 1 & 1 & 0 & 0 \\
		0 & 0 & 1 & 1 & 0 \\
		0 & 0 & 0 & 1 & 1 \\
		0 & 0 & 0 & 0 & 1
	\end{matrix}\right)
\end{equation*} \newline
Where matrix $A$ has entries
\begin{align*}
	&d_{1} = 1, d_{2} = 1, d_{3} = \frac{1}{x_{23}}, d_{4} = \frac{1}{x_{23} x_{34}}, d_{5} = \frac{1}{x_{23} x_{34} x_{45}},\\ 
	&a_{12} = 0, a_{13} = 0, a_{14} = \frac{x_{15}}{x_{23} x_{34} x_{45}}, a_{15} = 1,\\ 
	&a_{23} = 1, a_{24} = 1, a_{25} = 1,\\ 
	&a_{34} = \frac{x_{23} x_{34} - x_{24}}{x_{23}^{2} x_{34}}, a_{35} = -\frac{x_{23} x_{25} x_{34} - x_{23} x_{24} x_{35} - {\left(x_{23}^{2} x_{34}^{2} - x_{23} x_{24} x_{34} + x_{24}^{2}\right)} x_{45}}{x_{23}^{3} x_{34}^{2} x_{45}},\\ 
	&a_{45} = -\frac{x_{23} x_{35} - {\left(x_{23} x_{34} - x_{24}\right)} x_{45}}{x_{23}^{2} x_{34}^{2} x_{45}}
\end{align*}

\begin{equation*}x=\left(\begin{matrix}
		1 & 0 & 0 & x_{14} & 0 \\
		0 & 1 & x_{23} & x_{24} & x_{25} \\
		0 & 0 & 1 & x_{34} & x_{35}  \\
		0 & 0 & 0 & 1 & x_{45}  \\
		0 & 0 & 0 & 0 & 1 \\
	\end{matrix}\right),x^A=\left(\begin{matrix}
		1 & 0 & 0 & 0 & 0 \\
		0 & 1 & 1 & 0 & 0 \\
		0 & 0 & 1 & 1 & 0 \\
		0 & 0 & 0 & 1 & 1 \\
		0 & 0 & 0 & 0 & 1
	\end{matrix}\right)
\end{equation*} \newline
Where matrix $A$ has entries
\begin{align*}
	&d_{1} = 1, d_{2} = 1, d_{3} = \frac{1}{x_{23}}, d_{4} = \frac{1}{x_{23} x_{34}}, d_{5} = \frac{1}{x_{23} x_{34} x_{45}},\\ 
	&a_{12} = 0, a_{13} = \frac{x_{14}}{x_{23} x_{34}}, a_{14} = 1, a_{15} = 1,\\ 
	&a_{23} = \frac{x_{14} x_{23} x_{35} + {\left(x_{23}^{2} x_{34}^{2} + x_{14} x_{24}\right)} x_{45}}{x_{14} x_{23} x_{34} x_{45}}, a_{24} = 1, a_{25} = 1,\\ 
	&a_{34} = \frac{x_{23} x_{34}^{2} x_{45} + x_{14} x_{35}}{x_{14} x_{23} x_{34} x_{45}}, a_{35} = -\frac{x_{14} x_{25} - {\left(x_{14} x_{23} - x_{23} x_{24}\right)} x_{34} x_{45}}{x_{14} x_{23}^{2} x_{34} x_{45}},\\ 
	&a_{45} = \frac{1}{x_{14}}
\end{align*}

\begin{equation*}x=\left(\begin{matrix}
		1 & 0 & 0 & x_{14} & x_{15} \\
		0 & 1 & x_{23} & x_{24} & x_{25} \\
		0 & 0 & 1 & x_{34} & x_{35}  \\
		0 & 0 & 0 & 1 & x_{45}  \\
		0 & 0 & 0 & 0 & 1 \\
	\end{matrix}\right),x^A=\left(\begin{matrix}
		1 & 0 & 0 & 0 & 0 \\
		0 & 1 & 1 & 0 & 0 \\
		0 & 0 & 1 & 1 & 0 \\
		0 & 0 & 0 & 1 & 1 \\
		0 & 0 & 0 & 0 & 1
	\end{matrix}\right)
\end{equation*} \newline
Where matrix $A$ has entries
\begin{align*}
	&d_{1} = 1, d_{2} = 1, d_{3} = \frac{1}{x_{23}}, d_{4} = \frac{1}{x_{23} x_{34}}, d_{5} = \frac{1}{x_{23} x_{34} x_{45}},\\ 
	&a_{12} = 0, a_{13} = \frac{x_{14}}{x_{23} x_{34}}, a_{14} = 1, a_{15} = 1,\\ 
	&a_{23} = -\frac{x_{15} x_{23} x_{34} - x_{14} x_{23} x_{35} - {\left(x_{23}^{2} x_{34}^{2} + x_{14} x_{24}\right)} x_{45}}{x_{14} x_{23} x_{34} x_{45}}, a_{24} = 1, a_{25} = 1,\\ 
	&a_{34} = \frac{x_{23} x_{34}^{2} x_{45} - x_{15} x_{34} + x_{14} x_{35}}{x_{14} x_{23} x_{34} x_{45}}, a_{35} = \frac{x_{15} x_{24} - x_{14} x_{25} + {\left(x_{14} x_{23} - x_{23} x_{24}\right)} x_{34} x_{45}}{x_{14} x_{23}^{2} x_{34} x_{45}},\\ 
	&a_{45} = \frac{x_{23} x_{34} x_{45} - x_{15}}{x_{14} x_{23} x_{34} x_{45}}
\end{align*}

\begin{equation*}x=\left(\begin{matrix}
		1 & 0 & x_{13} & 0 & 0 \\
		0 & 1 & x_{23} & x_{24} & x_{25} \\
		0 & 0 & 1 & x_{34} & x_{35}  \\
		0 & 0 & 0 & 1 & x_{45}  \\
		0 & 0 & 0 & 0 & 1 \\
	\end{matrix}\right),x^A=\left(\begin{matrix}
		1 & 0 & 0 & 0 & 0 \\
		0 & 1 & 1 & 0 & 0 \\
		0 & 0 & 1 & 1 & 0 \\
		0 & 0 & 0 & 1 & 1 \\
		0 & 0 & 0 & 0 & 1
	\end{matrix}\right)
\end{equation*} \newline
Where matrix $A$ has entries
\begin{align*}
	&d_{1} = 1, d_{2} = 1, d_{3} = \frac{1}{x_{23}}, d_{4} = \frac{1}{x_{23} x_{34}}, d_{5} = \frac{1}{x_{23} x_{34} x_{45}},\\ 
	&a_{12} = \frac{x_{13}}{x_{23}}, a_{13} = 1, a_{14} = 1, a_{15} = 1,\\ 
	&a_{23} = \frac{x_{23}^{2} x_{34} + x_{13} x_{24}}{x_{13} x_{23} x_{34}}, a_{24} = \frac{x_{13} x_{25} x_{34} - x_{13} x_{24} x_{35} + {\left(x_{23}^{2} x_{34}^{2} + x_{23} x_{24} x_{34}\right)} x_{45}}{x_{13} x_{23} x_{34}^{2} x_{45}}, a_{25} = 1,\\ 
	&a_{34} = \frac{1}{x_{13}}, a_{35} = \frac{1}{x_{13}},\\ 
	&a_{45} = \frac{x_{23} x_{34} x_{45} - x_{13} x_{35}}{x_{13} x_{23} x_{34}^{2} x_{45}}
\end{align*}

\begin{equation*}x=\left(\begin{matrix}
		1 & 0 & x_{13} & 0 & x_{15} \\
		0 & 1 & x_{23} & x_{24} & x_{25} \\
		0 & 0 & 1 & x_{34} & x_{35}  \\
		0 & 0 & 0 & 1 & x_{45}  \\
		0 & 0 & 0 & 0 & 1 \\
	\end{matrix}\right),x^A=\left(\begin{matrix}
		1 & 0 & 0 & 0 & 0 \\
		0 & 1 & 1 & 0 & 0 \\
		0 & 0 & 1 & 1 & 0 \\
		0 & 0 & 0 & 1 & 1 \\
		0 & 0 & 0 & 0 & 1
	\end{matrix}\right)
\end{equation*} \newline
Where matrix $A$ has entries
\begin{align*}
	&d_{1} = 1, d_{2} = 1, d_{3} = \frac{1}{x_{23}}, d_{4} = \frac{1}{x_{23} x_{34}}, d_{5} = \frac{1}{x_{23} x_{34} x_{45}},\\ 
	&a_{12} = \frac{x_{13}}{x_{23}}, a_{13} = 1, a_{14} = 1, a_{15} = 1,\\ 
	&a_{23} = \frac{x_{23}^{2} x_{34} + x_{13} x_{24}}{x_{13} x_{23} x_{34}}, a_{24} = -\frac{x_{13} x_{24} x_{35} + {\left(x_{15} x_{23} - x_{13} x_{25}\right)} x_{34} - {\left(x_{23}^{2} x_{34}^{2} + x_{23} x_{24} x_{34}\right)} x_{45}}{x_{13} x_{23} x_{34}^{2} x_{45}}, a_{25} = 1,\\ 
	&a_{34} = \frac{1}{x_{13}}, a_{35} = \frac{x_{23} x_{34} x_{45} - x_{15}}{x_{13} x_{23} x_{34} x_{45}},\\ 
	&a_{45} = \frac{x_{23} x_{34} x_{45} - x_{13} x_{35}}{x_{13} x_{23} x_{34}^{2} x_{45}}
\end{align*}

\begin{equation*}x=\left(\begin{matrix}
		1 & 0 & x_{13} & x_{14} & 0 \\
		0 & 1 & x_{23} & x_{24} & x_{25} \\
		0 & 0 & 1 & x_{34} & x_{35}  \\
		0 & 0 & 0 & 1 & x_{45}  \\
		0 & 0 & 0 & 0 & 1 \\
	\end{matrix}\right),x^A=\left(\begin{matrix}
		1 & 0 & 0 & 0 & 0 \\
		0 & 1 & 1 & 0 & 0 \\
		0 & 0 & 1 & 1 & 0 \\
		0 & 0 & 0 & 1 & 1 \\
		0 & 0 & 0 & 0 & 1
	\end{matrix}\right)
\end{equation*} \newline
Where matrix $A$ has entries
\begin{align*}
	&d_{1} = 1, d_{2} = 1, d_{3} = \frac{1}{x_{23}}, d_{4} = \frac{1}{x_{23} x_{34}}, d_{5} = \frac{1}{x_{23} x_{34} x_{45}},\\ 
	&a_{12} = \frac{x_{13}}{x_{23}}, a_{13} = 1, a_{14} = 1, a_{15} = 1,\\ 
	&a_{23} = \frac{x_{23}^{2} x_{34} - x_{14} x_{23} + x_{13} x_{24}}{x_{13} x_{23} x_{34}}, \\
	&a_{24} = \frac{x_{13}^{2} x_{25} x_{34} + {\left(x_{13} x_{14} x_{23} - x_{13}^{2} x_{24}\right)} x_{35} + {\left(x_{13} x_{23}^{2} x_{34}^{2} + x_{14}^{2} x_{23} - x_{13} x_{14} x_{24} - {\left(x_{14} x_{23}^{2} - x_{13} x_{23} x_{24}\right)} x_{34}\right)} x_{45}}{x_{13}^{2} x_{23} x_{34}^{2} x_{45}}, a_{25} = 1,\\ 
	&a_{34} = \frac{x_{23} x_{34} - x_{14}}{x_{13} x_{23} x_{34}}, a_{35} = \frac{x_{13} x_{14} x_{35} + {\left(x_{13} x_{23} x_{34}^{2} - x_{14} x_{23} x_{34} + x_{14}^{2}\right)} x_{45}}{x_{13}^{2} x_{23} x_{34}^{2} x_{45}},\\ 
	&a_{45} = -\frac{x_{13} x_{35} - {\left(x_{23} x_{34} - x_{14}\right)} x_{45}}{x_{13} x_{23} x_{34}^{2} x_{45}}
\end{align*}

\begin{equation*}x=\left(\begin{matrix}
		1 & 0 & x_{13} & x_{14} & x_{15} \\
		0 & 1 & x_{23} & x_{24} & x_{25} \\
		0 & 0 & 1 & x_{34} & x_{35}  \\
		0 & 0 & 0 & 1 & x_{45}  \\
		0 & 0 & 0 & 0 & 1 \\
	\end{matrix}\right),x^A=\left(\begin{matrix}
		1 & 0 & 0 & 0 & 0 \\
		0 & 1 & 1 & 0 & 0 \\
		0 & 0 & 1 & 1 & 0 \\
		0 & 0 & 0 & 1 & 1 \\
		0 & 0 & 0 & 0 & 1
	\end{matrix}\right)
\end{equation*} \newline
Where matrix $A$ has entries
\begin{align*}
	&d_{1} = 1, d_{2} = 1, d_{3} = \frac{1}{x_{23}}, d_{4} = \frac{1}{x_{23} x_{34}}, d_{5} = \frac{1}{x_{23} x_{34} x_{45}},\\ 
	&a_{12} = \frac{x_{13}}{x_{23}}, a_{13} = 1, a_{14} = 1, a_{15} = 1,\\ 
	&a_{23} = \frac{x_{23}^{2} x_{34} - x_{14} x_{23} + x_{13} x_{24}}{x_{13} x_{23} x_{34}}, a_{25} = 1, \\
	&a_{24} = -\frac{{\left(x_{13} x_{15} x_{23} - x_{13}^{2} x_{25}\right)} x_{34} - {\left(x_{13} x_{14} x_{23} - x_{13}^{2} x_{24}\right)} x_{35} - {\left(x_{13} x_{23}^{2} x_{34}^{2} + x_{14}^{2} x_{23} - x_{13} x_{14} x_{24} - {\left(x_{14} x_{23}^{2} - x_{13} x_{23} x_{24}\right)} x_{34}\right)} x_{45}}{x_{13}^{2} x_{23} x_{34}^{2} x_{45}}, \\ 
	&a_{34} = \frac{x_{23} x_{34} - x_{14}}{x_{13} x_{23} x_{34}}, a_{35} = -\frac{x_{13} x_{15} x_{34} - x_{13} x_{14} x_{35} - {\left(x_{13} x_{23} x_{34}^{2} - x_{14} x_{23} x_{34} + x_{14}^{2}\right)} x_{45}}{x_{13}^{2} x_{23} x_{34}^{2} x_{45}},\\ 
	&a_{45} = -\frac{x_{13} x_{35} - {\left(x_{23} x_{34} - x_{14}\right)} x_{45}}{x_{13} x_{23} x_{34}^{2} x_{45}}
\end{align*}

		\section{Subcases of $Y_{16}$}

\begin{equation*}x=\left(
\right)
\end{equation*} \newline
Where matrix $A$ has entries
\begin{align*}
	&d_{1} = 1, d_{2} = \frac{1}{x_{12}}, d_{3} = \frac{1}{x_{13}}, d_{4} = \frac{1}{x_{13} x_{34}}, d_{5} = \frac{1}{x_{13} x_{34} x_{45}},\\ 
	&a_{12} = 1, a_{13} = 1, a_{14} = 1, a_{15} = 1,\\ 
	&a_{23} = 0, a_{24} = 0, a_{25} = 1,\\ 
	&a_{34} = \frac{x_{13} x_{34} - x_{14}}{x_{13}^{2} x_{34}}, a_{35} = -\frac{{\left(x_{12} - 1\right)} x_{13}^{2} x_{34}^{2} + x_{13} x_{14} x_{34} - x_{14}^{2}}{x_{13}^{3} x_{34}^{2}},\\ 
	&a_{45} = \frac{x_{13} x_{34} - x_{14}}{x_{13}^{2} x_{34}^{2}}
\end{align*}

\begin{equation*}x=\left(\begin{matrix}
		1 & x_{12} & x_{13} & x_{14} & x_{15} \\
		0 & 1 & 0 & 0 & 0 \\
		0 & 0 & 1 & x_{34} & 0  \\
		0 & 0 & 0 & 1 & x_{45}  \\
		0 & 0 & 0 & 0 & 1 \\
	\end{matrix}\right),x^A=\left(\begin{matrix}
		1 & 1 & 1 & 0 & 0 \\
		0 & 1 & 0 & 0 & 0 \\
		0 & 0 & 1 & 1 & 0 \\
		0 & 0 & 0 & 1 & 1 \\
		0 & 0 & 0 & 0 & 1
	\end{matrix}\right)
\end{equation*} \newline
Where matrix $A$ has entries
\begin{align*}
	&d_{1} = 1, d_{2} = \frac{1}{x_{12}}, d_{3} = \frac{1}{x_{13}}, d_{4} = \frac{1}{x_{13} x_{34}}, d_{5} = \frac{1}{x_{13} x_{34} x_{45}},\\ 
	&a_{12} = 1, a_{13} = 1, a_{14} = 1, a_{15} = 1,\\ 
	&a_{23} = 0, a_{24} = 0, a_{25} = 1,\\ 
	&a_{34} = \frac{x_{13} x_{34} - x_{14}}{x_{13}^{2} x_{34}}, a_{35} = -\frac{x_{13} x_{15} x_{34} + {\left({\left(x_{12} - 1\right)} x_{13}^{2} x_{34}^{2} + x_{13} x_{14} x_{34} - x_{14}^{2}\right)} x_{45}}{x_{13}^{3} x_{34}^{2} x_{45}},\\ 
	&a_{45} = \frac{x_{13} x_{34} - x_{14}}{x_{13}^{2} x_{34}^{2}}
\end{align*}

\begin{equation*}x=\left(
\right)
\end{equation*} \newline
Where matrix $A$ has entries
\begin{align*}
	&d_{1} = 1, d_{2} = \frac{1}{x_{12}}, d_{3} = 1, d_{4} = \frac{1}{x_{34}}, d_{5} = \frac{1}{x_{34} x_{45}},\\ 
	&a_{12} = 1, a_{13} = \frac{x_{12} x_{25} + x_{14} x_{45}}{x_{34} x_{45}}, a_{14} = 1, a_{15} = 1,\\ 
	&a_{23} = 0, a_{24} = \frac{x_{25}}{x_{34} x_{45}}, a_{25} = 1,\\ 
	&a_{34} = -\frac{{\left(x_{12} - 1\right)} x_{34} x_{45} + x_{15}}{x_{14} x_{45}}, a_{35} = 1,\\ 
	&a_{45} = -\frac{{\left(x_{12} - 1\right)} x_{34} x_{45} + x_{15}}{x_{14} x_{34} x_{45}}
\end{align*}

\begin{equation*}x=\left(\begin{matrix}
		1 & x_{12} & x_{13} & 0 & 0 \\
		0 & 1 & 0 & 0 & x_{25} \\
		0 & 0 & 1 & x_{34} & 0  \\
		0 & 0 & 0 & 1 & x_{45}  \\
		0 & 0 & 0 & 0 & 1 \\
	\end{matrix}\right),x^A=\left(\begin{matrix}
		1 & 1 & 1 & 0 & 0 \\
		0 & 1 & 0 & 0 & 0 \\
		0 & 0 & 1 & 1 & 0 \\
		0 & 0 & 0 & 1 & 1 \\
		0 & 0 & 0 & 0 & 1
	\end{matrix}\right)
\end{equation*} \newline
Where matrix $A$ has entries
\begin{align*}
	&d_{1} = 1, d_{2} = \frac{1}{x_{12}}, d_{3} = \frac{1}{x_{13}}, d_{4} = \frac{1}{x_{13} x_{34}}, d_{5} = \frac{1}{x_{13} x_{34} x_{45}},\\ 
	&a_{12} = 1, a_{13} = 1, a_{14} = 1, a_{15} = 1,\\ 
	&a_{23} = 0, a_{24} = \frac{x_{25}}{x_{13} x_{34} x_{45}}, a_{25} = 1,\\ 
	&a_{34} = \frac{x_{13} x_{34} x_{45} - x_{12} x_{25}}{x_{13}^{2} x_{34} x_{45}}, a_{35} = -\frac{x_{12} - 1}{x_{13}},\\ 
	&a_{45} = \frac{x_{13} x_{34} x_{45} - x_{12} x_{25}}{x_{13}^{2} x_{34}^{2} x_{45}}
\end{align*}

\begin{equation*}x=\left(\begin{matrix}
		1 & x_{12} & x_{13} & 0 & x_{15} \\
		0 & 1 & 0 & 0 & x_{25} \\
		0 & 0 & 1 & x_{34} & 0  \\
		0 & 0 & 0 & 1 & x_{45}  \\
		0 & 0 & 0 & 0 & 1 \\
	\end{matrix}\right),x^A=\left(\begin{matrix}
		1 & 1 & 1 & 0 & 0 \\
		0 & 1 & 0 & 0 & 0 \\
		0 & 0 & 1 & 1 & 0 \\
		0 & 0 & 0 & 1 & 1 \\
		0 & 0 & 0 & 0 & 1
	\end{matrix}\right)
\end{equation*} \newline
Where matrix $A$ has entries
\begin{align*}
	&d_{1} = 1, d_{2} = \frac{1}{x_{12}}, d_{3} = \frac{1}{x_{13}}, d_{4} = \frac{1}{x_{13} x_{34}}, d_{5} = \frac{1}{x_{13} x_{34} x_{45}},\\ 
	&a_{12} = 1, a_{13} = 1, a_{14} = 1, a_{15} = 1,\\ 
	&a_{23} = 0, a_{24} = \frac{x_{25}}{x_{13} x_{34} x_{45}}, a_{25} = 1,\\ 
	&a_{34} = \frac{x_{13} x_{34} x_{45} - x_{12} x_{25}}{x_{13}^{2} x_{34} x_{45}}, a_{35} = -\frac{{\left(x_{12} - 1\right)} x_{13} x_{34} x_{45} + x_{15}}{x_{13}^{2} x_{34} x_{45}},\\ 
	&a_{45} = \frac{x_{13} x_{34} x_{45} - x_{12} x_{25}}{x_{13}^{2} x_{34}^{2} x_{45}}
\end{align*}

\begin{equation*}x=\left(\begin{matrix}
		1 & x_{12} & x_{13} & x_{14} & 0 \\
		0 & 1 & 0 & 0 & x_{25} \\
		0 & 0 & 1 & x_{34} & 0  \\
		0 & 0 & 0 & 1 & x_{45}  \\
		0 & 0 & 0 & 0 & 1 \\
	\end{matrix}\right),x^A=\left(\begin{matrix}
		1 & 1 & 1 & 0 & 0 \\
		0 & 1 & 0 & 0 & 0 \\
		0 & 0 & 1 & 1 & 0 \\
		0 & 0 & 0 & 1 & 1 \\
		0 & 0 & 0 & 0 & 1
	\end{matrix}\right)
\end{equation*} \newline
Where matrix $A$ has entries
\begin{align*}
	&d_{1} = 1, d_{2} = \frac{1}{x_{12}}, d_{3} = \frac{1}{x_{13}}, d_{4} = \frac{1}{x_{13} x_{34}}, d_{5} = \frac{1}{x_{13} x_{34} x_{45}},\\ 
	&a_{12} = 1, a_{13} = 1, a_{14} = 1, a_{15} = 1,\\ 
	&a_{23} = 0, a_{24} = \frac{x_{25}}{x_{13} x_{34} x_{45}}, a_{25} = 1,\\ 
	&a_{34} = -\frac{x_{12} x_{25} - {\left(x_{13} x_{34} - x_{14}\right)} x_{45}}{x_{13}^{2} x_{34} x_{45}}, a_{35} = \frac{x_{12} x_{14} x_{25} - {\left({\left(x_{12} - 1\right)} x_{13}^{2} x_{34}^{2} + x_{13} x_{14} x_{34} - x_{14}^{2}\right)} x_{45}}{x_{13}^{3} x_{34}^{2} x_{45}},\\ 
	&a_{45} = -\frac{x_{12} x_{25} - {\left(x_{13} x_{34} - x_{14}\right)} x_{45}}{x_{13}^{2} x_{34}^{2} x_{45}}
\end{align*}

\begin{equation*}x=\left(\begin{matrix}
		1 & x_{12} & x_{13} & x_{14} & x_{15} \\
		0 & 1 & 0 & 0 & x_{25} \\
		0 & 0 & 1 & x_{34} & 0  \\
		0 & 0 & 0 & 1 & x_{45}  \\
		0 & 0 & 0 & 0 & 1 \\
	\end{matrix}\right),x^A=\left(\begin{matrix}
		1 & 1 & 1 & 0 & 0 \\
		0 & 1 & 0 & 0 & 0 \\
		0 & 0 & 1 & 1 & 0 \\
		0 & 0 & 0 & 1 & 1 \\
		0 & 0 & 0 & 0 & 1
	\end{matrix}\right)
\end{equation*} \newline
Where matrix $A$ has entries
\begin{align*}
	&d_{1} = 1, d_{2} = \frac{1}{x_{12}}, d_{3} = \frac{1}{x_{13}}, d_{4} = \frac{1}{x_{13} x_{34}}, d_{5} = \frac{1}{x_{13} x_{34} x_{45}},\\ 
	&a_{12} = 1, a_{13} = 1, a_{14} = 1, a_{15} = 1,\\ 
	&a_{23} = 0, a_{24} = \frac{x_{25}}{x_{13} x_{34} x_{45}}, a_{25} = 1,\\ 
	&a_{34} = -\frac{x_{12} x_{25} - {\left(x_{13} x_{34} - x_{14}\right)} x_{45}}{x_{13}^{2} x_{34} x_{45}}, a_{35} = \frac{x_{12} x_{14} x_{25} - x_{13} x_{15} x_{34} - {\left({\left(x_{12} - 1\right)} x_{13}^{2} x_{34}^{2} + x_{13} x_{14} x_{34} - x_{14}^{2}\right)} x_{45}}{x_{13}^{3} x_{34}^{2} x_{45}},\\ 
	&a_{45} = -\frac{x_{12} x_{25} - {\left(x_{13} x_{34} - x_{14}\right)} x_{45}}{x_{13}^{2} x_{34}^{2} x_{45}}
\end{align*}

\begin{equation*}x=\left(
\right)
\end{equation*} \newline
Where matrix $A$ has entries
\begin{align*}
	&d_{1} = 1, d_{2} = \frac{1}{x_{12}}, d_{3} = \frac{x_{34}}{x_{12} x_{24}}, d_{4} = \frac{1}{x_{12} x_{24}}, d_{5} = \frac{1}{x_{12} x_{24} x_{45}},\\ 
	&a_{12} = 1, a_{13} = 1, a_{14} = 1, a_{15} = 1,\\ 
	&a_{23} = \frac{1}{x_{12}}, a_{24} = \frac{x_{12} x_{24} - x_{14}}{x_{12}^{2} x_{24}}, a_{25} = \frac{x_{12}^{2} x_{24}^{2} - x_{12} x_{14} x_{24} + x_{14}^{2}}{x_{12}^{3} x_{24}^{2}},\\ 
	&a_{34} = \frac{{\left(x_{12} x_{24} - x_{14}\right)} x_{34}}{x_{12}^{2} x_{24}^{2}}, a_{35} = 1,\\ 
	&a_{45} = \frac{x_{12} x_{24} - x_{14}}{x_{12}^{2} x_{24}^{2}}
\end{align*}

\begin{equation*}x=\left(\begin{matrix}
		1 & x_{12} & 0 & x_{14} & x_{15} \\
		0 & 1 & 0 & x_{24} & 0 \\
		0 & 0 & 1 & x_{34} & 0  \\
		0 & 0 & 0 & 1 & x_{45}  \\
		0 & 0 & 0 & 0 & 1 \\
	\end{matrix}\right),x^A=\left(\begin{matrix}
		1 & 1 & 1 & 0 & 0 \\
		0 & 1 & 0 & 0 & 0 \\
		0 & 0 & 1 & 1 & 0 \\
		0 & 0 & 0 & 1 & 1 \\
		0 & 0 & 0 & 0 & 1
	\end{matrix}\right)
\end{equation*} \newline
Where matrix $A$ has entries
\begin{align*}
	&d_{1} = 1, d_{2} = \frac{1}{x_{12}}, d_{3} = \frac{x_{34}}{x_{12} x_{24}}, d_{4} = \frac{1}{x_{12} x_{24}}, d_{5} = \frac{1}{x_{12} x_{24} x_{45}},\\ 
	&a_{12} = 1, a_{13} = 1, a_{14} = 1, a_{15} = 1,\\ 
	&a_{23} = \frac{1}{x_{12}}, a_{24} = \frac{x_{12} x_{24} - x_{14}}{x_{12}^{2} x_{24}}, a_{25} = -\frac{x_{12} x_{15} x_{24} - {\left(x_{12}^{2} x_{24}^{2} - x_{12} x_{14} x_{24} + x_{14}^{2}\right)} x_{45}}{x_{12}^{3} x_{24}^{2} x_{45}},\\ 
	&a_{34} = \frac{{\left(x_{12} x_{24} - x_{14}\right)} x_{34}}{x_{12}^{2} x_{24}^{2}}, a_{35} = 1,\\ 
	&a_{45} = \frac{x_{12} x_{24} - x_{14}}{x_{12}^{2} x_{24}^{2}}
\end{align*}


First assume $x_{24}\neq \frac{-x_{13}x_{34}}{x_{12}}$
\begin{equation*}x=\left(\begin{matrix}
		1 & x_{12} & x_{13} & 0 & 0 \\
		0 & 1 & 0 & x_{24} & 0 \\
		0 & 0 & 1 & x_{34} & 0  \\
		0 & 0 & 0 & 1 & x_{45}  \\
		0 & 0 & 0 & 0 & 1 \\
	\end{matrix}\right),x^A=\left(\begin{matrix}
		1 & 1 & 1 & 0 & 0 \\
		0 & 1 & 0 & 0 & 0 \\
		0 & 0 & 1 & 1 & 0 \\
		0 & 0 & 0 & 1 & 1 \\
		0 & 0 & 0 & 0 & 1
	\end{matrix}\right)
\end{equation*} \newline
Where matrix $A$ has entries
\begin{align*}
	&d_{1} = 1, d_{2} = \frac{1}{x_{12}}, d_{3} = \frac{x_{34}}{x_{12} x_{24} + x_{13} x_{34}}, d_{4} = \frac{1}{x_{12} x_{24} + x_{13} x_{34}}, d_{5} = \frac{1}{{\left(x_{12} x_{24} + x_{13} x_{34}\right)} x_{45}},\\ 
	&a_{12} = 1, a_{13} = 1, a_{14} = 1, a_{15} = 1,\\ 
	&a_{23} = \frac{x_{24}}{x_{12} x_{24} + x_{13} x_{34}}, a_{24} = \frac{x_{24}}{x_{12} x_{24} + x_{13} x_{34}}, a_{25} = 1,\\ 
	&a_{34} = \frac{x_{34}}{x_{12} x_{24} + x_{13} x_{34}}, a_{35} = -\frac{x_{12} - 1}{x_{13}},\\ 
	&a_{45} = \frac{1}{x_{12} x_{24} + x_{13} x_{34}}
\end{align*}

Now assume $x_{24}= \frac{-x_{13}x_{34}}{x_{12}}$
\begin{equation*}x=\left(\begin{matrix}
		1 & x_{12} & x_{13} & 0 & 0 \\
		0 & 1 & 0 & x_{24} & 0 \\
		0 & 0 & 1 & x_{34} & 0  \\
		0 & 0 & 0 & 1 & x_{45}  \\
		0 & 0 & 0 & 0 & 1 \\
	\end{matrix}\right),x^A=\left(\begin{matrix}
		1 & 1 & 0 & 0 & 0 \\
		0 & 1 & 0 & 0 & 0 \\
		0 & 0 & 1 & 1 & 0 \\
		0 & 0 & 0 & 1 & 1 \\
		0 & 0 & 0 & 0 & 1
	\end{matrix}\right)
\end{equation*} \newline
Where matrix $A$ has entries
\begin{align*}
	&d_{1} = 1, d_{2} = \frac{1}{x_{12}}, d_{3} = 1, d_{4} = \frac{1}{x_{34}}, d_{5} = \frac{1}{x_{34} x_{45}},\\ 
	&a_{12} = 1, a_{13} = 0, a_{14} = 1, a_{15} = 1,\\ 
	&a_{23} = -\frac{x_{13}}{x_{12}}, a_{24} = 1, a_{25} = 1,\\ 
	&a_{34} = -\frac{x_{12}}{x_{13}}, a_{35} = -\frac{x_{12} - 1}{x_{13}},\\ 
	&a_{45} = -\frac{x_{12}}{x_{13} x_{34}}
\end{align*}


First assume $x_{24}\neq \frac{-x_{13}x_{34}}{x_{12}}$
\begin{equation*}x=\left(\begin{matrix}
		1 & x_{12} & x_{13} & 0 & x_{15} \\
		0 & 1 & 0 & x_{24} & 0 \\
		0 & 0 & 1 & x_{34} & 0  \\
		0 & 0 & 0 & 1 & x_{45}  \\
		0 & 0 & 0 & 0 & 1 \\
	\end{matrix}\right),x^A=\left(\begin{matrix}
		1 & 1 & 1 & 0 & 0 \\
		0 & 1 & 0 & 0 & 0 \\
		0 & 0 & 1 & 1 & 0 \\
		0 & 0 & 0 & 1 & 1 \\
		0 & 0 & 0 & 0 & 1
	\end{matrix}\right)
\end{equation*} \newline
Where matrix $A$ has entries
\begin{align*}
	&d_{1} = 1, d_{2} = \frac{1}{x_{12}}, d_{3} = \frac{x_{34}}{x_{12} x_{24} + x_{13} x_{34}}, d_{4} = \frac{1}{x_{12} x_{24} + x_{13} x_{34}}, d_{5} = \frac{1}{{\left(x_{12} x_{24} + x_{13} x_{34}\right)} x_{45}},\\ 
	&a_{12} = 1, a_{13} = 1, a_{14} = 1, a_{15} = 1,\\ 
	&a_{23} = \frac{x_{24}}{x_{12} x_{24} + x_{13} x_{34}}, a_{24} = \frac{x_{24}}{x_{12} x_{24} + x_{13} x_{34}}, a_{25} = 1,\\ 
	&a_{34} = \frac{x_{34}}{x_{12} x_{24} + x_{13} x_{34}}, a_{35} = -\frac{x_{15} + {\left({\left(x_{12} - 1\right)} x_{13} x_{34} + {\left(x_{12}^{2} - x_{12}\right)} x_{24}\right)} x_{45}}{{\left(x_{12} x_{13} x_{24} + x_{13}^{2} x_{34}\right)} x_{45}},\\ 
	&a_{45} = \frac{1}{x_{12} x_{24} + x_{13} x_{34}}
\end{align*}

Now assume $x_{24}= \frac{-x_{13}x_{34}}{x_{12}}$
\begin{equation*}x=\left(\begin{matrix}
		1 & x_{12} & x_{13} & 0 & x_{15} \\
		0 & 1 & 0 & x_{24} & 0 \\
		0 & 0 & 1 & x_{34} & 0  \\
		0 & 0 & 0 & 1 & x_{45}  \\
		0 & 0 & 0 & 0 & 1 \\
	\end{matrix}\right),x^A=\left(\begin{matrix}
		1 & 1 & 0 & 0 & 0 \\
		0 & 1 & 0 & 0 & 0 \\
		0 & 0 & 1 & 1 & 0 \\
		0 & 0 & 0 & 1 & 1 \\
		0 & 0 & 0 & 0 & 1
	\end{matrix}\right)
\end{equation*} \newline
Where matrix $A$ has entries
\begin{align*}
	&d_{1} = 1, d_{2} = \frac{1}{x_{12}}, d_{3} = 1, d_{4} = \frac{1}{x_{34}}, d_{5} = \frac{1}{x_{34} x_{45}},\\ 
	&a_{12} = 1, a_{13} = 0, a_{14} = 1, a_{15} = 1,\\ 
	&a_{23} = -\frac{x_{13}}{x_{12}}, a_{24} = 1, a_{25} = 1,\\ 
	&a_{34} = -\frac{x_{12}}{x_{13}}, a_{35} = -\frac{{\left(x_{12} - 1\right)} x_{34} x_{45} + x_{15}}{x_{13} x_{34} x_{45}},\\ 
	&a_{45} = -\frac{x_{12}}{x_{13} x_{34}}
\end{align*}


First assume $x_{24}\neq \frac{-x_{13}x_{34}}{x_{12}}$
\begin{equation*}x=\left(\begin{matrix}
		1 & x_{12} & x_{13} & x_{14} & 0 \\
		0 & 1 & 0 & x_{24} & 0 \\
		0 & 0 & 1 & x_{34} & 0  \\
		0 & 0 & 0 & 1 & x_{45}  \\
		0 & 0 & 0 & 0 & 1 \\
	\end{matrix}\right),x^A=\left(\begin{matrix}
		1 & 1 & 1 & 0 & 0 \\
		0 & 1 & 0 & 0 & 0 \\
		0 & 0 & 1 & 1 & 0 \\
		0 & 0 & 0 & 1 & 1 \\
		0 & 0 & 0 & 0 & 1
	\end{matrix}\right)
\end{equation*} \newline
Where matrix $A$ has entries
\begin{align*}
	&d_{1} = 1, d_{2} = \frac{1}{x_{12}}, d_{3} = \frac{x_{34}}{x_{12} x_{24} + x_{13} x_{34}}, d_{4} = \frac{1}{x_{12} x_{24} + x_{13} x_{34}}, d_{5} = \frac{1}{{\left(x_{12} x_{24} + x_{13} x_{34}\right)} x_{45}},\\ 
	&a_{12} = 1, a_{13} = 1, a_{14} = 1, a_{15} = 1,\\ 
	&a_{23} = \frac{x_{24}}{x_{12} x_{24} + x_{13} x_{34}}, a_{24} = \frac{x_{12} x_{24}^{2} + x_{13} x_{24} x_{34} - x_{14} x_{24}}{x_{12}^{2} x_{24}^{2} + 2 \, x_{12} x_{13} x_{24} x_{34} + x_{13}^{2} x_{34}^{2}}, a_{25} = 1,\\ 
	&a_{34} = \frac{x_{13} x_{34}^{2} + {\left(x_{12} x_{24} - x_{14}\right)} x_{34}}{x_{12}^{2} x_{24}^{2} + 2 \, x_{12} x_{13} x_{24} x_{34} + x_{13}^{2} x_{34}^{2}}, \\
	&a_{35} = -\frac{{\left(x_{12} - 1\right)} x_{13}^{2} x_{34}^{2} + x_{12} x_{14} x_{24} - x_{14}^{2} + {\left(x_{12}^{3} - x_{12}^{2}\right)} x_{24}^{2} + {\left(x_{13} x_{14} + 2 \, {\left(x_{12}^{2} - x_{12}\right)} x_{13} x_{24}\right)} x_{34}}{x_{12}^{2} x_{13} x_{24}^{2} + 2 \, x_{12} x_{13}^{2} x_{24} x_{34} + x_{13}^{3} x_{34}^{2}},\\ 
	&a_{45} = \frac{x_{12} x_{24} + x_{13} x_{34} - x_{14}}{x_{12}^{2} x_{24}^{2} + 2 \, x_{12} x_{13} x_{24} x_{34} + x_{13}^{2} x_{34}^{2}}
\end{align*}

First assume $x_{24}\neq \frac{-x_{13}x_{34}}{x_{12}}$
\begin{equation*}x=\left(\begin{matrix}
		1 & x_{12} & x_{13} & x_{14} & 0 \\
		0 & 1 & 0 & x_{24} & 0 \\
		0 & 0 & 1 & x_{34} & 0  \\
		0 & 0 & 0 & 1 & x_{45}  \\
		0 & 0 & 0 & 0 & 1 \\
	\end{matrix}\right),x^A=\left(\begin{matrix}
		1 & 1 & 0 & 0 & 0 \\
		0 & 1 & 0 & 0 & 0 \\
		0 & 0 & 1 & 1 & 0 \\
		0 & 0 & 0 & 1 & 1 \\
		0 & 0 & 0 & 0 & 1
	\end{matrix}\right)
\end{equation*} \newline
Where matrix $A$ has entries
\begin{align*}
	&d_{1} = 1, d_{2} = \frac{1}{x_{12}}, d_{3} = 1, d_{4} = \frac{1}{x_{34}}, d_{5} = \frac{1}{x_{34} x_{45}},\\ 
	&a_{12} = 1, a_{13} = \frac{x_{14}}{x_{34}}, a_{14} = 1, a_{15} = 1,\\ 
	&a_{23} = -\frac{x_{13}}{x_{12}}, a_{24} = 1, a_{25} = 1,\\ 
	&a_{34} = -\frac{x_{12}}{x_{13}}, a_{35} = \frac{x_{12} x_{14} - {\left(x_{12} - 1\right)} x_{13} x_{34}}{x_{13}^{2} x_{34}},\\ 
	&a_{45} = -\frac{x_{12}}{x_{13} x_{34}}
\end{align*}


First assume $x_{24}\neq \frac{-x_{13}x_{34}}{x_{12}}$
\begin{equation*}x=\left(\begin{matrix}
		1 & x_{12} & x_{13} & x_{14} & x_{15} \\
		0 & 1 & 0 & x_{24} & 0 \\
		0 & 0 & 1 & x_{34} & 0  \\
		0 & 0 & 0 & 1 & x_{45}  \\
		0 & 0 & 0 & 0 & 1 \\
	\end{matrix}\right),x^A=\left(\begin{matrix}
		1 & 1 & 1 & 0 & 0 \\
		0 & 1 & 0 & 0 & 0 \\
		0 & 0 & 1 & 1 & 0 \\
		0 & 0 & 0 & 1 & 1 \\
		0 & 0 & 0 & 0 & 1
	\end{matrix}\right)
\end{equation*} \newline
Where matrix $A$ has entries
\begin{align*}
	&d_{1} = 1, d_{2} = \frac{1}{x_{12}}, d_{3} = \frac{x_{34}}{x_{12} x_{24} + x_{13} x_{34}}, d_{4} = \frac{1}{x_{12} x_{24} + x_{13} x_{34}}, d_{5} = \frac{1}{{\left(x_{12} x_{24} + x_{13} x_{34}\right)} x_{45}},\\ 
	&a_{12} = 1, a_{13} = 1, a_{14} = 1, a_{15} = 1,\\ 
	&a_{23} = \frac{x_{24}}{x_{12} x_{24} + x_{13} x_{34}}, a_{24} = \frac{x_{12} x_{24}^{2} + x_{13} x_{24} x_{34} - x_{14} x_{24}}{x_{12}^{2} x_{24}^{2} + 2 \, x_{12} x_{13} x_{24} x_{34} + x_{13}^{2} x_{34}^{2}}, a_{25} = 1,\\ 
	&a_{34} = \frac{x_{13} x_{34}^{2} + {\left(x_{12} x_{24} - x_{14}\right)} x_{34}}{x_{12}^{2} x_{24}^{2} + 2 \, x_{12} x_{13} x_{24} x_{34} + x_{13}^{2} x_{34}^{2}}, \\
	&a_{35} = -\frac{x_{12} x_{15} x_{24} + x_{13} x_{15} x_{34} + {\left({\left(x_{12} - 1\right)} x_{13}^{2} x_{34}^{2} + x_{12} x_{14} x_{24} - x_{14}^{2} + {\left(x_{12}^{3} - x_{12}^{2}\right)} x_{24}^{2} + {\left(x_{13} x_{14} + 2 \, {\left(x_{12}^{2} - x_{12}\right)} x_{13} x_{24}\right)} x_{34}\right)} x_{45}}{{\left(x_{12}^{2} x_{13} x_{24}^{2} + 2 \, x_{12} x_{13}^{2} x_{24} x_{34} + x_{13}^{3} x_{34}^{2}\right)} x_{45}},\\ 
	&a_{45} = \frac{x_{12} x_{24} + x_{13} x_{34} - x_{14}}{x_{12}^{2} x_{24}^{2} + 2 \, x_{12} x_{13} x_{24} x_{34} + x_{13}^{2} x_{34}^{2}}
\end{align*}

First assume $x_{24}\neq \frac{-x_{13}x_{34}}{x_{12}}$
\begin{equation*}x=\left(\begin{matrix}
		1 & x_{12} & x_{13} & x_{14} & 0 \\
		0 & 1 & 0 & x_{24} & 0 \\
		0 & 0 & 1 & x_{34} & 0  \\
		0 & 0 & 0 & 1 & x_{45}  \\
		0 & 0 & 0 & 0 & 1 \\
	\end{matrix}\right),x^A=\left(\begin{matrix}
		1 & 1 & 0 & 0 & 0 \\
		0 & 1 & 0 & 0 & 0 \\
		0 & 0 & 1 & 1 & 0 \\
		0 & 0 & 0 & 1 & 1 \\
		0 & 0 & 0 & 0 & 1
	\end{matrix}\right)
\end{equation*} \newline
Where matrix $A$ has entries
\begin{align*}
	&d_{1} = 1, d_{2} = \frac{1}{x_{12}}, d_{3} = 1, d_{4} = \frac{1}{x_{34}}, d_{5} = \frac{1}{x_{34} x_{45}},\\ 
	&a_{12} = 1, a_{13} = \frac{x_{14}}{x_{34}}, a_{14} = 1, a_{15} = 1,\\ 
	&a_{23} = -\frac{x_{13}}{x_{12}}, a_{24} = 1, a_{25} = 1,\\ 
	&a_{34} = -\frac{x_{12}}{x_{13}}, a_{35} = -\frac{x_{13} x_{15} - {\left(x_{12} x_{14} - {\left(x_{12} - 1\right)} x_{13} x_{34}\right)} x_{45}}{x_{13}^{2} x_{34} x_{45}},\\ 
	&a_{45} = -\frac{x_{12}}{x_{13} x_{34}}
\end{align*}

\begin{equation*}x=\left(\begin{matrix}
		1 & x_{12} & 0 & 0 & 0 \\
		0 & 1 & 0 & x_{24} & x_{25} \\
		0 & 0 & 1 & x_{34} & 0  \\
		0 & 0 & 0 & 1 & x_{45}  \\
		0 & 0 & 0 & 0 & 1 \\
	\end{matrix}\right),x^A=\left(\begin{matrix}
		1 & 1 & 1 & 0 & 0 \\
		0 & 1 & 0 & 0 & 0 \\
		0 & 0 & 1 & 1 & 0 \\
		0 & 0 & 0 & 1 & 1 \\
		0 & 0 & 0 & 0 & 1
	\end{matrix}\right)
\end{equation*} \newline
Where matrix $A$ has entries
\begin{align*}
	&d_{1} = 1, d_{2} = \frac{1}{x_{12}}, d_{3} = \frac{x_{34}}{x_{12} x_{24}}, d_{4} = \frac{1}{x_{12} x_{24}}, d_{5} = \frac{1}{x_{12} x_{24} x_{45}},\\ 
	&a_{12} = 1, a_{13} = 1, a_{14} = 1, a_{15} = 1,\\ 
	&a_{23} = \frac{1}{x_{12}}, a_{24} = \frac{1}{x_{12}}, a_{25} = \frac{1}{x_{12}},\\ 
	&a_{34} = \frac{x_{24} x_{34} x_{45} - x_{25} x_{34}}{x_{12} x_{24}^{2} x_{45}}, a_{35} = 1,\\ 
	&a_{45} = \frac{x_{24} x_{45} - x_{25}}{x_{12} x_{24}^{2} x_{45}}
\end{align*}

\begin{equation*}x=\left(\begin{matrix}
		1 & x_{12} & 0 & 0 & x_{15} \\
		0 & 1 & 0 & x_{24} & x_{25} \\
		0 & 0 & 1 & x_{34} & 0  \\
		0 & 0 & 0 & 1 & x_{45}  \\
		0 & 0 & 0 & 0 & 1 \\
	\end{matrix}\right),x^A=\left(\begin{matrix}
		1 & 1 & 1 & 0 & 0 \\
		0 & 1 & 0 & 0 & 0 \\
		0 & 0 & 1 & 1 & 0 \\
		0 & 0 & 0 & 1 & 1 \\
		0 & 0 & 0 & 0 & 1
	\end{matrix}\right)
\end{equation*} \newline
Where matrix $A$ has entries
\begin{align*}
	&d_{1} = 1, d_{2} = \frac{1}{x_{12}}, d_{3} = \frac{x_{34}}{x_{12} x_{24}}, d_{4} = \frac{1}{x_{12} x_{24}}, d_{5} = \frac{1}{x_{12} x_{24} x_{45}},\\ 
	&a_{12} = 1, a_{13} = 1, a_{14} = 1, a_{15} = 1,\\ 
	&a_{23} = \frac{1}{x_{12}}, a_{24} = \frac{1}{x_{12}}, a_{25} = \frac{x_{12} x_{24} x_{45} - x_{15}}{x_{12}^{2} x_{24} x_{45}},\\ 
	&a_{34} = \frac{x_{24} x_{34} x_{45} - x_{25} x_{34}}{x_{12} x_{24}^{2} x_{45}}, a_{35} = 1,\\ 
	&a_{45} = \frac{x_{24} x_{45} - x_{25}}{x_{12} x_{24}^{2} x_{45}}
\end{align*}

\begin{equation*}x=\left(\begin{matrix}
		1 & x_{12} & 0 & x_{14} & 0 \\
		0 & 1 & 0 & x_{24} & x_{25} \\
		0 & 0 & 1 & x_{34} & 0  \\
		0 & 0 & 0 & 1 & x_{45}  \\
		0 & 0 & 0 & 0 & 1 \\
	\end{matrix}\right),x^A=\left(\begin{matrix}
		1 & 1 & 1 & 0 & 0 \\
		0 & 1 & 0 & 0 & 0 \\
		0 & 0 & 1 & 1 & 0 \\
		0 & 0 & 0 & 1 & 1 \\
		0 & 0 & 0 & 0 & 1
	\end{matrix}\right)
\end{equation*} \newline
Where matrix $A$ has entries
\begin{align*}
	&d_{1} = 1, d_{2} = \frac{1}{x_{12}}, d_{3} = \frac{x_{34}}{x_{12} x_{24}}, d_{4} = \frac{1}{x_{12} x_{24}}, d_{5} = \frac{1}{x_{12} x_{24} x_{45}},\\ 
	&a_{12} = 1, a_{13} = 1, a_{14} = 1, a_{15} = 1,\\ 
	&a_{23} = \frac{1}{x_{12}}, a_{24} = \frac{x_{12} x_{24} - x_{14}}{x_{12}^{2} x_{24}}, a_{25} = \frac{x_{12} x_{14} x_{25} + {\left(x_{12}^{2} x_{24}^{2} - x_{12} x_{14} x_{24} + x_{14}^{2}\right)} x_{45}}{x_{12}^{3} x_{24}^{2} x_{45}},\\ 
	&a_{34} = -\frac{x_{12} x_{25} x_{34} - {\left(x_{12} x_{24} - x_{14}\right)} x_{34} x_{45}}{x_{12}^{2} x_{24}^{2} x_{45}}, a_{35} = 1,\\ 
	&a_{45} = -\frac{x_{12} x_{25} - {\left(x_{12} x_{24} - x_{14}\right)} x_{45}}{x_{12}^{2} x_{24}^{2} x_{45}}
\end{align*}

\begin{equation*}x=\left(\begin{matrix}
		1 & x_{12} & 0 & x_{14} & x_{15} \\
		0 & 1 & 0 & x_{24} & x_{25} \\
		0 & 0 & 1 & x_{34} & 0  \\
		0 & 0 & 0 & 1 & x_{45}  \\
		0 & 0 & 0 & 0 & 1 \\
	\end{matrix}\right),x^A=\left(\begin{matrix}
		1 & 1 & 1 & 0 & 0 \\
		0 & 1 & 0 & 0 & 0 \\
		0 & 0 & 1 & 1 & 0 \\
		0 & 0 & 0 & 1 & 1 \\
		0 & 0 & 0 & 0 & 1
	\end{matrix}\right)
\end{equation*} \newline
Where matrix $A$ has entries
\begin{align*}
	&d_{1} = 1, d_{2} = \frac{1}{x_{12}}, d_{3} = \frac{x_{34}}{x_{12} x_{24}}, d_{4} = \frac{1}{x_{12} x_{24}}, d_{5} = \frac{1}{x_{12} x_{24} x_{45}},\\ 
	&a_{12} = 1, a_{13} = 1, a_{14} = 1, a_{15} = 1,\\ 
	&a_{23} = \frac{1}{x_{12}}, a_{24} = \frac{x_{12} x_{24} - x_{14}}{x_{12}^{2} x_{24}}, a_{25} = -\frac{x_{12} x_{15} x_{24} - x_{12} x_{14} x_{25} - {\left(x_{12}^{2} x_{24}^{2} - x_{12} x_{14} x_{24} + x_{14}^{2}\right)} x_{45}}{x_{12}^{3} x_{24}^{2} x_{45}},\\ 
	&a_{34} = -\frac{x_{12} x_{25} x_{34} - {\left(x_{12} x_{24} - x_{14}\right)} x_{34} x_{45}}{x_{12}^{2} x_{24}^{2} x_{45}}, a_{35} = 1,\\ 
	&a_{45} = -\frac{x_{12} x_{25} - {\left(x_{12} x_{24} - x_{14}\right)} x_{45}}{x_{12}^{2} x_{24}^{2} x_{45}}
\end{align*}


First assume $x_{24}\neq \frac{-x_{13}x_{34}}{x_{12}}$.
\begin{equation*}x=\left(\begin{matrix}
		1 & x_{12} & x_{13} & 0 & 0 \\
		0 & 1 & 0 & x_{24} & x_{25} \\
		0 & 0 & 1 & x_{34} & 0  \\
		0 & 0 & 0 & 1 & x_{45}  \\
		0 & 0 & 0 & 0 & 1 \\
	\end{matrix}\right),x^A=\left(\begin{matrix}
		1 & 1 & 1 & 0 & 0 \\
		0 & 1 & 0 & 0 & 0 \\
		0 & 0 & 1 & 1 & 0 \\
		0 & 0 & 0 & 1 & 1 \\
		0 & 0 & 0 & 0 & 1
	\end{matrix}\right)
\end{equation*} \newline
Where matrix $A$ has entries
\begin{align*}
	&d_{1} = 1, d_{2} = \frac{1}{x_{12}}, d_{3} = \frac{x_{34}}{x_{12} x_{24} + x_{13} x_{34}}, d_{4} = \frac{1}{x_{12} x_{24} + x_{13} x_{34}}, d_{5} = \frac{1}{{\left(x_{12} x_{24} + x_{13} x_{34}\right)} x_{45}},\\ 
	&a_{12} = 1, a_{13} = 1, a_{14} = 1, a_{15} = 1,\\ 
	&a_{23} = \frac{x_{24}}{x_{12} x_{24} + x_{13} x_{34}}, a_{24} = \frac{x_{13} x_{25} x_{34} + {\left(x_{12} x_{24}^{2} + x_{13} x_{24} x_{34}\right)} x_{45}}{{\left(x_{12}^{2} x_{24}^{2} + 2 \, x_{12} x_{13} x_{24} x_{34} + x_{13}^{2} x_{34}^{2}\right)} x_{45}}, a_{25} = 1,\\ 
	&a_{34} = -\frac{x_{12} x_{25} x_{34} - {\left(x_{12} x_{24} x_{34} + x_{13} x_{34}^{2}\right)} x_{45}}{{\left(x_{12}^{2} x_{24}^{2} + 2 \, x_{12} x_{13} x_{24} x_{34} + x_{13}^{2} x_{34}^{2}\right)} x_{45}}, a_{35} = -\frac{x_{12} - 1}{x_{13}},\\ 
	&a_{45} = -\frac{x_{12} x_{25} - {\left(x_{12} x_{24} + x_{13} x_{34}\right)} x_{45}}{{\left(x_{12}^{2} x_{24}^{2} + 2 \, x_{12} x_{13} x_{24} x_{34} + x_{13}^{2} x_{34}^{2}\right)} x_{45}}
\end{align*}

Now assume $x_{24}=\frac{-x_{13}x_{34}}{x_{12}}$.
\begin{equation*}x=\left(\begin{matrix}
		1 & x_{12} & x_{13} & 0 & 0 \\
		0 & 1 & 0 & x_{24} & x_{25} \\
		0 & 0 & 1 & x_{34} & 0  \\
		0 & 0 & 0 & 1 & x_{45}  \\
		0 & 0 & 0 & 0 & 1 \\
	\end{matrix}\right),x^A=\left(\begin{matrix}
		1 & 1 & 0 & 0 & 0 \\
		0 & 1 & 0 & 0 & 0 \\
		0 & 0 & 1 & 1 & 0 \\
		0 & 0 & 0 & 1 & 1 \\
		0 & 0 & 0 & 0 & 1
	\end{matrix}\right)
\end{equation*} \newline
Where matrix $A$ has entries
\begin{align*}
	&d_{1} = 1, d_{2} = \frac{1}{x_{12}}, d_{3} = 1, d_{4} = \frac{1}{x_{34}}, d_{5} = \frac{1}{x_{34} x_{45}},\\ 
	&a_{12} = 1, a_{13} = \frac{x_{12} x_{25}}{x_{34} x_{45}}, a_{14} = 1, a_{15} = 1,\\ 
	&a_{23} = -\frac{x_{13}}{x_{12}}, a_{24} = 1, a_{25} = 1,\\ 
	&a_{34} = -\frac{x_{12} x_{34} x_{45} - x_{12} x_{25}}{x_{13} x_{34} x_{45}}, a_{35} = -\frac{x_{12} - 1}{x_{13}},\\ 
	&a_{45} = -\frac{x_{12} x_{34} x_{45} - x_{12} x_{25}}{x_{13} x_{34}^{2} x_{45}}
\end{align*}


First assume $x_{24}\neq \frac{-x_{13}x_{34}}{x_{12}}$.
\begin{equation*}x=\left(\begin{matrix}
		1 & x_{12} & x_{13} & 0 & x_{15} \\
		0 & 1 & 0 & x_{24} & x_{25} \\
		0 & 0 & 1 & x_{34} & 0  \\
		0 & 0 & 0 & 1 & x_{45}  \\
		0 & 0 & 0 & 0 & 1 \\
	\end{matrix}\right),x^A=\left(\begin{matrix}
		1 & 1 & 1 & 0 & 0 \\
		0 & 1 & 0 & 0 & 0 \\
		0 & 0 & 1 & 1 & 0 \\
		0 & 0 & 0 & 1 & 1 \\
		0 & 0 & 0 & 0 & 1
	\end{matrix}\right)
\end{equation*} \newline
Where matrix $A$ has entries
\begin{align*}
	&d_{1} = 1, d_{2} = \frac{1}{x_{12}}, d_{3} = \frac{x_{34}}{x_{12} x_{24} + x_{13} x_{34}}, d_{4} = \frac{1}{x_{12} x_{24} + x_{13} x_{34}}, d_{5} = \frac{1}{{\left(x_{12} x_{24} + x_{13} x_{34}\right)} x_{45}},\\ 
	&a_{12} = 1, a_{13} = 1, a_{14} = 1, a_{15} = 1,\\ 
	&a_{23} = \frac{x_{24}}{x_{12} x_{24} + x_{13} x_{34}}, a_{24} = \frac{x_{13} x_{25} x_{34} + {\left(x_{12} x_{24}^{2} + x_{13} x_{24} x_{34}\right)} x_{45}}{{\left(x_{12}^{2} x_{24}^{2} + 2 \, x_{12} x_{13} x_{24} x_{34} + x_{13}^{2} x_{34}^{2}\right)} x_{45}}, a_{25} = 1,\\ 
	&a_{34} = -\frac{x_{12} x_{25} x_{34} - {\left(x_{12} x_{24} x_{34} + x_{13} x_{34}^{2}\right)} x_{45}}{{\left(x_{12}^{2} x_{24}^{2} + 2 \, x_{12} x_{13} x_{24} x_{34} + x_{13}^{2} x_{34}^{2}\right)} x_{45}}, a_{35} = -\frac{x_{15} + {\left({\left(x_{12} - 1\right)} x_{13} x_{34} + {\left(x_{12}^{2} - x_{12}\right)} x_{24}\right)} x_{45}}{{\left(x_{12} x_{13} x_{24} + x_{13}^{2} x_{34}\right)} x_{45}},\\ 
	&a_{45} = -\frac{x_{12} x_{25} - {\left(x_{12} x_{24} + x_{13} x_{34}\right)} x_{45}}{{\left(x_{12}^{2} x_{24}^{2} + 2 \, x_{12} x_{13} x_{24} x_{34} + x_{13}^{2} x_{34}^{2}\right)} x_{45}}
\end{align*}

Now assume $x_{24}= \frac{-x_{13}x_{34}}{x_{12}}$.
\begin{equation*}x=\left(\begin{matrix}
		1 & x_{12} & x_{13} & 0 & x_{15} \\
		0 & 1 & 0 & x_{24} & x_{25} \\
		0 & 0 & 1 & x_{34} & 0  \\
		0 & 0 & 0 & 1 & x_{45}  \\
		0 & 0 & 0 & 0 & 1 \\
	\end{matrix}\right),x^A=\left(\begin{matrix}
		1 & 1 & 0 & 0 & 0 \\
		0 & 1 & 0 & 0 & 0 \\
		0 & 0 & 1 & 1 & 0 \\
		0 & 0 & 0 & 1 & 1 \\
		0 & 0 & 0 & 0 & 1
	\end{matrix}\right)
\end{equation*} \newline
Where matrix $A$ has entries
\begin{align*}
	&d_{1} = 1, d_{2} = \frac{1}{x_{12}}, d_{3} = 1, d_{4} = \frac{1}{x_{34}}, d_{5} = \frac{1}{x_{34} x_{45}},\\ 
	&a_{12} = 1, a_{13} = \frac{x_{12} x_{25}}{x_{34} x_{45}}, a_{14} = 1, a_{15} = 1,\\ 
	&a_{23} = -\frac{x_{13}}{x_{12}}, a_{24} = 1, a_{25} = 1,\\ 
	&a_{34} = -\frac{x_{12} x_{34} x_{45} - x_{12} x_{25}}{x_{13} x_{34} x_{45}}, a_{35} = -\frac{{\left(x_{12} - 1\right)} x_{34} x_{45} + x_{15}}{x_{13} x_{34} x_{45}},\\ 
	&a_{45} = -\frac{x_{12} x_{34} x_{45} - x_{12} x_{25}}{x_{13} x_{34}^{2} x_{45}}
\end{align*}


First assume $x_{24}\neq \frac{-x_{13}x_{34}}{x_{12}}$.
\begin{equation*}x=\left(\begin{matrix}
		1 & x_{12} & x_{13} & x_{14} & 0 \\
		0 & 1 & 0 & x_{24} & x_{25} \\
		0 & 0 & 1 & x_{34} & 0  \\
		0 & 0 & 0 & 1 & x_{45}  \\
		0 & 0 & 0 & 0 & 1 \\
	\end{matrix}\right),x^A=\left(\begin{matrix}
		1 & 1 & 1 & 0 & 0 \\
		0 & 1 & 0 & 0 & 0 \\
		0 & 0 & 1 & 1 & 0 \\
		0 & 0 & 0 & 1 & 1 \\
		0 & 0 & 0 & 0 & 1
	\end{matrix}\right)
\end{equation*} \newline
Where matrix $A$ has entries
\begin{align*}
	&d_{1} = 1, d_{2} = \frac{1}{x_{12}}, d_{3} = \frac{x_{34}}{x_{12} x_{24} + x_{13} x_{34}}, d_{4} = \frac{1}{x_{12} x_{24} + x_{13} x_{34}}, d_{5} = \frac{1}{{\left(x_{12} x_{24} + x_{13} x_{34}\right)} x_{45}},\\ 
	&a_{12} = 1, a_{13} = 1, a_{14} = 1, a_{15} = 1,\\ 
	&a_{23} = \frac{x_{24}}{x_{12} x_{24} + x_{13} x_{34}}, a_{24} = \frac{x_{13} x_{25} x_{34} + {\left(x_{12} x_{24}^{2} + x_{13} x_{24} x_{34} - x_{14} x_{24}\right)} x_{45}}{{\left(x_{12}^{2} x_{24}^{2} + 2 \, x_{12} x_{13} x_{24} x_{34} + x_{13}^{2} x_{34}^{2}\right)} x_{45}}, a_{25} = 1,\\ 
	&a_{34} = -\frac{x_{12} x_{25} x_{34} - {\left(x_{13} x_{34}^{2} + {\left(x_{12} x_{24} - x_{14}\right)} x_{34}\right)} x_{45}}{{\left(x_{12}^{2} x_{24}^{2} + 2 \, x_{12} x_{13} x_{24} x_{34} + x_{13}^{2} x_{34}^{2}\right)} x_{45}}, \\
	&a_{35} = \frac{x_{12} x_{14} x_{25} - {\left({\left(x_{12} - 1\right)} x_{13}^{2} x_{34}^{2} + x_{12} x_{14} x_{24} - x_{14}^{2} + {\left(x_{12}^{3} - x_{12}^{2}\right)} x_{24}^{2} + {\left(x_{13} x_{14} + 2 \, {\left(x_{12}^{2} - x_{12}\right)} x_{13} x_{24}\right)} x_{34}\right)} x_{45}}{{\left(x_{12}^{2} x_{13} x_{24}^{2} + 2 \, x_{12} x_{13}^{2} x_{24} x_{34} + x_{13}^{3} x_{34}^{2}\right)} x_{45}},\\ 
	&a_{45} = -\frac{x_{12} x_{25} - {\left(x_{12} x_{24} + x_{13} x_{34} - x_{14}\right)} x_{45}}{{\left(x_{12}^{2} x_{24}^{2} + 2 \, x_{12} x_{13} x_{24} x_{34} + x_{13}^{2} x_{34}^{2}\right)} x_{45}}
\end{align*}

Now assume $x_{24}= \frac{-x_{13}x_{34}}{x_{12}}$.
\begin{equation*}x=\left(\begin{matrix}
		1 & x_{12} & x_{13} & x_{14} & 0 \\
		0 & 1 & 0 & x_{24} & x_{25} \\
		0 & 0 & 1 & x_{34} & 0  \\
		0 & 0 & 0 & 1 & x_{45}  \\
		0 & 0 & 0 & 0 & 1 \\
	\end{matrix}\right),x^A=\left(\begin{matrix}
		1 & 1 & 0 & 0 & 0 \\
		0 & 1 & 0 & 0 & 0 \\
		0 & 0 & 1 & 1 & 0 \\
		0 & 0 & 0 & 1 & 1 \\
		0 & 0 & 0 & 0 & 1
	\end{matrix}\right)
\end{equation*} \newline
Where matrix $A$ has entries
\begin{align*}
	&d_{1} = 1, d_{2} = \frac{1}{x_{12}}, d_{3} = 1, d_{4} = \frac{1}{x_{34}}, d_{5} = \frac{1}{x_{34} x_{45}},\\ 
	&a_{12} = 1, a_{13} = \frac{x_{12} x_{25} + x_{14} x_{45}}{x_{34} x_{45}}, a_{14} = 1, a_{15} = 1,\\ 
	&a_{23} = -\frac{x_{13}}{x_{12}}, a_{24} = 1, a_{25} = 1,\\ 
	&a_{34} = -\frac{x_{12} x_{34} x_{45} - x_{12} x_{25}}{x_{13} x_{34} x_{45}}, a_{35} = -\frac{x_{12} x_{14} x_{25} - {\left(x_{12} x_{14} x_{34} - {\left(x_{12} - 1\right)} x_{13} x_{34}^{2}\right)} x_{45}}{x_{13}^{2} x_{34}^{2} x_{45}},\\ 
	&a_{45} = -\frac{x_{12} x_{34} x_{45} - x_{12} x_{25}}{x_{13} x_{34}^{2} x_{45}}
\end{align*}


First assume $x_{24}\neq \frac{-x_{13}x_{34}}{x_{12}}$.
\begin{equation*}x=\left(\begin{matrix}
		1 & x_{12} & x_{13} & x_{14} & x_{15} \\
		0 & 1 & 0 & x_{24} & x_{25} \\
		0 & 0 & 1 & x_{34} & 0  \\
		0 & 0 & 0 & 1 & x_{45}  \\
		0 & 0 & 0 & 0 & 1 \\
	\end{matrix}\right),x^A=\left(\begin{matrix}
		1 & 1 & 1 & 0 & 0 \\
		0 & 1 & 0 & 0 & 0 \\
		0 & 0 & 1 & 1 & 0 \\
		0 & 0 & 0 & 1 & 1 \\
		0 & 0 & 0 & 0 & 1
	\end{matrix}\right)
\end{equation*} \newline
Where matrix $A$ has entries
\begin{align*}
	&d_{1} = 1, d_{2} = \frac{1}{x_{12}}, d_{3} = \frac{x_{34}}{x_{12} x_{24} + x_{13} x_{34}}, d_{4} = \frac{1}{x_{12} x_{24} + x_{13} x_{34}}, d_{5} = \frac{1}{{\left(x_{12} x_{24} + x_{13} x_{34}\right)} x_{45}},\\ 
	&a_{12} = 1, a_{13} = 1, a_{14} = 1, a_{15} = 1,\\ 
	&a_{23} = \frac{x_{24}}{x_{12} x_{24} + x_{13} x_{34}}, a_{24} = \frac{x_{13} x_{25} x_{34} + {\left(x_{12} x_{24}^{2} + x_{13} x_{24} x_{34} - x_{14} x_{24}\right)} x_{45}}{{\left(x_{12}^{2} x_{24}^{2} + 2 \, x_{12} x_{13} x_{24} x_{34} + x_{13}^{2} x_{34}^{2}\right)} x_{45}}, a_{25} = 1,\\ 
	&a_{34} = -\frac{x_{12} x_{25} x_{34} - {\left(x_{13} x_{34}^{2} + {\left(x_{12} x_{24} - x_{14}\right)} x_{34}\right)} x_{45}}{{\left(x_{12}^{2} x_{24}^{2} + 2 \, x_{12} x_{13} x_{24} x_{34} + x_{13}^{2} x_{34}^{2}\right)} x_{45}}, \\
	&a_{35} = -\frac{\splitfrac{x_{12} x_{15} x_{24} - x_{12} x_{14} x_{25} +}
		{+ x_{13} x_{15} x_{34} + {\left({\left(x_{12} - 1\right)} x_{13}^{2} x_{34}^{2} + x_{12} x_{14} x_{24} - x_{14}^{2} + {\left(x_{12}^{3} - x_{12}^{2}\right)} x_{24}^{2} + {\left(x_{13} x_{14} + 2 \, {\left(x_{12}^{2} - x_{12}\right)} x_{13} x_{24}\right)} x_{34}\right)} x_{45}}}{{\left(x_{12}^{2} x_{13} x_{24}^{2} + 2 \, x_{12} x_{13}^{2} x_{24} x_{34} + x_{13}^{3} x_{34}^{2}\right)} x_{45}},\\ 
	&a_{45} = -\frac{x_{12} x_{25} - {\left(x_{12} x_{24} + x_{13} x_{34} - x_{14}\right)} x_{45}}{{\left(x_{12}^{2} x_{24}^{2} + 2 \, x_{12} x_{13} x_{24} x_{34} + x_{13}^{2} x_{34}^{2}\right)} x_{45}}
\end{align*}

Now assume $x_{24}= \frac{-x_{13}x_{34}}{x_{12}}$.
\begin{equation*}x=\left(\begin{matrix}
		1 & x_{12} & x_{13} & x_{14} & x_{15} \\
		0 & 1 & 0 & x_{24} & x_{25} \\
		0 & 0 & 1 & x_{34} & 0  \\
		0 & 0 & 0 & 1 & x_{45}  \\
		0 & 0 & 0 & 0 & 1 \\
	\end{matrix}\right),x^A=\left(\begin{matrix}
		1 & 1 & 0 & 0 & 0 \\
		0 & 1 & 0 & 0 & 0 \\
		0 & 0 & 1 & 1 & 0 \\
		0 & 0 & 0 & 1 & 1 \\
		0 & 0 & 0 & 0 & 1
	\end{matrix}\right)
\end{equation*} \newline
Where matrix $A$ has entries
\begin{align*}
	&d_{1} = 1, d_{2} = \frac{1}{x_{12}}, d_{3} = 1, d_{4} = \frac{1}{x_{34}}, d_{5} = \frac{1}{x_{34} x_{45}},\\ 
	&a_{12} = 1, a_{13} = \frac{x_{12} x_{25} + x_{14} x_{45}}{x_{34} x_{45}}, a_{14} = 1, a_{15} = 1,\\ 
	&a_{23} = -\frac{x_{13}}{x_{12}}, a_{24} = 1, a_{25} = 1,\\ 
	&a_{34} = -\frac{x_{12} x_{34} x_{45} - x_{12} x_{25}}{x_{13} x_{34} x_{45}}, a_{35} = -\frac{x_{12} x_{14} x_{25} + x_{13} x_{15} x_{34} - {\left(x_{12} x_{14} x_{34} - {\left(x_{12} - 1\right)} x_{13} x_{34}^{2}\right)} x_{45}}{x_{13}^{2} x_{34}^{2} x_{45}},\\ 
	&a_{45} = -\frac{x_{12} x_{34} x_{45} - x_{12} x_{25}}{x_{13} x_{34}^{2} x_{45}}
\end{align*}

\begin{equation*}x=\left(
\right)
\end{equation*} \newline
Where matrix $A$ has entries
\begin{align*}
	&d_{1} = 1, d_{2} = \frac{1}{x_{12}}, d_{3} = \frac{1}{x_{13}}, d_{4} = \frac{1}{x_{13} x_{34}}, d_{5} = \frac{1}{x_{13} x_{34} x_{45}},\\ 
	&a_{12} = 1, a_{13} = 1, a_{14} = 1, a_{15} = 1,\\ 
	&a_{23} = 0, a_{24} = 0, a_{25} = 1,\\ 
	&a_{34} = \frac{x_{13} x_{34} - x_{14}}{x_{13}^{2} x_{34}}, a_{35} = \frac{x_{13} x_{14} x_{35} - {\left({\left(x_{12} - 1\right)} x_{13}^{2} x_{34}^{2} + x_{13} x_{14} x_{34} - x_{14}^{2}\right)} x_{45}}{x_{13}^{3} x_{34}^{2} x_{45}},\\ 
	&a_{45} = -\frac{x_{13} x_{35} - {\left(x_{13} x_{34} - x_{14}\right)} x_{45}}{x_{13}^{2} x_{34}^{2} x_{45}}
\end{align*}

\begin{equation*}x=\left(\begin{matrix}
		1 & x_{12} & x_{13} & x_{14} & x_{15} \\
		0 & 1 & 0 & 0 & 0 \\
		0 & 0 & 1 & x_{34} & x_{35}  \\
		0 & 0 & 0 & 1 & x_{45}  \\
		0 & 0 & 0 & 0 & 1 \\
	\end{matrix}\right),x^A=\left(\begin{matrix}
		1 & 1 & 1 & 0 & 0 \\
		0 & 1 & 0 & 0 & 0 \\
		0 & 0 & 1 & 1 & 0 \\
		0 & 0 & 0 & 1 & 1 \\
		0 & 0 & 0 & 0 & 1
	\end{matrix}\right)
\end{equation*} \newline
Where matrix $A$ has entries
\begin{align*}
	&d_{1} = 1, d_{2} = \frac{1}{x_{12}}, d_{3} = \frac{1}{x_{13}}, d_{4} = \frac{1}{x_{13} x_{34}}, d_{5} = \frac{1}{x_{13} x_{34} x_{45}},\\ 
	&a_{12} = 1, a_{13} = 1, a_{14} = 1, a_{15} = 1,\\ 
	&a_{23} = 0, a_{24} = 0, a_{25} = 1,\\ 
	&a_{34} = \frac{x_{13} x_{34} - x_{14}}{x_{13}^{2} x_{34}}, a_{35} = -\frac{x_{13} x_{15} x_{34} - x_{13} x_{14} x_{35} + {\left({\left(x_{12} - 1\right)} x_{13}^{2} x_{34}^{2} + x_{13} x_{14} x_{34} - x_{14}^{2}\right)} x_{45}}{x_{13}^{3} x_{34}^{2} x_{45}},\\ 
	&a_{45} = -\frac{x_{13} x_{35} - {\left(x_{13} x_{34} - x_{14}\right)} x_{45}}{x_{13}^{2} x_{34}^{2} x_{45}}
\end{align*}

\begin{equation*}x=\left(
\right)
\end{equation*} \newline
Where matrix $A$ has entries
\begin{align*}
	&d_{1} = 1, d_{2} = \frac{1}{x_{12}}, d_{3} = 1, d_{4} = \frac{1}{x_{34}}, d_{5} = \frac{1}{x_{34} x_{45}},\\ 
	&a_{12} = 1, a_{13} = \frac{x_{12} x_{25} + x_{14} x_{45}}{x_{34} x_{45}}, a_{14} = 1, a_{15} = 1,\\ 
	&a_{23} = 0, a_{24} = \frac{x_{25}}{x_{34} x_{45}}, a_{25} = 1,\\ 
	&a_{34} = -\frac{{\left(x_{12} - 1\right)} x_{34}^{2} x_{45} + x_{15} x_{34} - x_{14} x_{35}}{x_{14} x_{34} x_{45}}, a_{35} = 1,\\ 
	&a_{45} = -\frac{{\left(x_{12} - 1\right)} x_{34} x_{45} + x_{15}}{x_{14} x_{34} x_{45}}
\end{align*}

\begin{equation*}x=\left(\begin{matrix}
		1 & x_{12} & x_{13} & 0 & 0 \\
		0 & 1 & 0 & 0 & x_{25} \\
		0 & 0 & 1 & x_{34} & x_{35}  \\
		0 & 0 & 0 & 1 & x_{45}  \\
		0 & 0 & 0 & 0 & 1 \\
	\end{matrix}\right),x^A=\left(\begin{matrix}
		1 & 1 & 1 & 0 & 0 \\
		0 & 1 & 0 & 0 & 0 \\
		0 & 0 & 1 & 1 & 0 \\
		0 & 0 & 0 & 1 & 1 \\
		0 & 0 & 0 & 0 & 1
	\end{matrix}\right)
\end{equation*} \newline
Where matrix $A$ has entries
\begin{align*}
	&d_{1} = 1, d_{2} = \frac{1}{x_{12}}, d_{3} = \frac{1}{x_{13}}, d_{4} = \frac{1}{x_{13} x_{34}}, d_{5} = \frac{1}{x_{13} x_{34} x_{45}},\\ 
	&a_{12} = 1, a_{13} = 1, a_{14} = 1, a_{15} = 1,\\ 
	&a_{23} = 0, a_{24} = \frac{x_{25}}{x_{13} x_{34} x_{45}}, a_{25} = 1,\\ 
	&a_{34} = \frac{x_{13} x_{34} x_{45} - x_{12} x_{25}}{x_{13}^{2} x_{34} x_{45}}, a_{35} = -\frac{x_{12} - 1}{x_{13}},\\ 
	&a_{45} = \frac{x_{13} x_{34} x_{45} - x_{12} x_{25} - x_{13} x_{35}}{x_{13}^{2} x_{34}^{2} x_{45}}
\end{align*}

\begin{equation*}x=\left(\begin{matrix}
		1 & x_{12} & x_{13} & 0 & x_{15} \\
		0 & 1 & 0 & 0 & x_{25} \\
		0 & 0 & 1 & x_{34} & x_{35}  \\
		0 & 0 & 0 & 1 & x_{45}  \\
		0 & 0 & 0 & 0 & 1 \\
	\end{matrix}\right),x^A=\left(\begin{matrix}
		1 & 1 & 1 & 0 & 0 \\
		0 & 1 & 0 & 0 & 0 \\
		0 & 0 & 1 & 1 & 0 \\
		0 & 0 & 0 & 1 & 1 \\
		0 & 0 & 0 & 0 & 1
	\end{matrix}\right)
\end{equation*} \newline
Where matrix $A$ has entries
\begin{align*}
	&d_{1} = 1, d_{2} = \frac{1}{x_{12}}, d_{3} = \frac{1}{x_{13}}, d_{4} = \frac{1}{x_{13} x_{34}}, d_{5} = \frac{1}{x_{13} x_{34} x_{45}},\\ 
	&a_{12} = 1, a_{13} = 1, a_{14} = 1, a_{15} = 1,\\ 
	&a_{23} = 0, a_{24} = \frac{x_{25}}{x_{13} x_{34} x_{45}}, a_{25} = 1,\\ 
	&a_{34} = \frac{x_{13} x_{34} x_{45} - x_{12} x_{25}}{x_{13}^{2} x_{34} x_{45}}, a_{35} = -\frac{{\left(x_{12} - 1\right)} x_{13} x_{34} x_{45} + x_{15}}{x_{13}^{2} x_{34} x_{45}},\\ 
	&a_{45} = \frac{x_{13} x_{34} x_{45} - x_{12} x_{25} - x_{13} x_{35}}{x_{13}^{2} x_{34}^{2} x_{45}}
\end{align*}

\begin{equation*}x=\left(\begin{matrix}
		1 & x_{12} & x_{13} & x_{14} & 0 \\
		0 & 1 & 0 & 0 & x_{25} \\
		0 & 0 & 1 & x_{34} & x_{35}  \\
		0 & 0 & 0 & 1 & x_{45}  \\
		0 & 0 & 0 & 0 & 1 \\
	\end{matrix}\right),x^A=\left(\begin{matrix}
		1 & 1 & 1 & 0 & 0 \\
		0 & 1 & 0 & 0 & 0 \\
		0 & 0 & 1 & 1 & 0 \\
		0 & 0 & 0 & 1 & 1 \\
		0 & 0 & 0 & 0 & 1
	\end{matrix}\right)
\end{equation*} \newline
Where matrix $A$ has entries
\begin{align*}
	&d_{1} = 1, d_{2} = \frac{1}{x_{12}}, d_{3} = \frac{1}{x_{13}}, d_{4} = \frac{1}{x_{13} x_{34}}, d_{5} = \frac{1}{x_{13} x_{34} x_{45}},\\ 
	&a_{12} = 1, a_{13} = 1, a_{14} = 1, a_{15} = 1,\\ 
	&a_{23} = 0, a_{24} = \frac{x_{25}}{x_{13} x_{34} x_{45}}, a_{25} = 1,\\ 
	&a_{34} = -\frac{x_{12} x_{25} - {\left(x_{13} x_{34} - x_{14}\right)} x_{45}}{x_{13}^{2} x_{34} x_{45}}, a_{35} = \frac{x_{12} x_{14} x_{25} + x_{13} x_{14} x_{35} - {\left({\left(x_{12} - 1\right)} x_{13}^{2} x_{34}^{2} + x_{13} x_{14} x_{34} - x_{14}^{2}\right)} x_{45}}{x_{13}^{3} x_{34}^{2} x_{45}},\\ 
	&a_{45} = -\frac{x_{12} x_{25} + x_{13} x_{35} - {\left(x_{13} x_{34} - x_{14}\right)} x_{45}}{x_{13}^{2} x_{34}^{2} x_{45}}
\end{align*}

\begin{equation*}x=\left(\begin{matrix}
		1 & x_{12} & x_{13} & x_{14} & x_{15} \\
		0 & 1 & 0 & 0 & x_{25} \\
		0 & 0 & 1 & x_{34} & x_{35}  \\
		0 & 0 & 0 & 1 & x_{45}  \\
		0 & 0 & 0 & 0 & 1 \\
	\end{matrix}\right),x^A=\left(\begin{matrix}
		1 & 1 & 1 & 0 & 0 \\
		0 & 1 & 0 & 0 & 0 \\
		0 & 0 & 1 & 1 & 0 \\
		0 & 0 & 0 & 1 & 1 \\
		0 & 0 & 0 & 0 & 1
	\end{matrix}\right)
\end{equation*} \newline
Where matrix $A$ has entries
\begin{align*}
	&d_{1} = 1, d_{2} = \frac{1}{x_{12}}, d_{3} = \frac{1}{x_{13}}, d_{4} = \frac{1}{x_{13} x_{34}}, d_{5} = \frac{1}{x_{13} x_{34} x_{45}},\\ 
	&a_{12} = 1, a_{13} = 1, a_{14} = 1, a_{15} = 1,\\ 
	&a_{23} = 0, a_{24} = \frac{x_{25}}{x_{13} x_{34} x_{45}}, a_{25} = 1,\\ 
	&a_{34} = -\frac{x_{12} x_{25} - {\left(x_{13} x_{34} - x_{14}\right)} x_{45}}{x_{13}^{2} x_{34} x_{45}}, a_{35} = \frac{x_{12} x_{14} x_{25} - x_{13} x_{15} x_{34} + x_{13} x_{14} x_{35} - {\left({\left(x_{12} - 1\right)} x_{13}^{2} x_{34}^{2} + x_{13} x_{14} x_{34} - x_{14}^{2}\right)} x_{45}}{x_{13}^{3} x_{34}^{2} x_{45}},\\ 
	&a_{45} = -\frac{x_{12} x_{25} + x_{13} x_{35} - {\left(x_{13} x_{34} - x_{14}\right)} x_{45}}{x_{13}^{2} x_{34}^{2} x_{45}}
\end{align*}

\begin{equation*}x=\left(
\right)
\end{equation*} \newline
Where matrix $A$ has entries
\begin{align*}
	&d_{1} = 1, d_{2} = \frac{1}{x_{12}}, d_{3} = \frac{x_{34}}{x_{12} x_{24}}, d_{4} = \frac{1}{x_{12} x_{24}}, d_{5} = \frac{1}{x_{12} x_{24} x_{45}},\\ 
	&a_{12} = 1, a_{13} = 1, a_{14} = 1, a_{15} = 1,\\ 
	&a_{23} = \frac{1}{x_{12}}, a_{24} = \frac{x_{12} x_{24} - x_{14}}{x_{12}^{2} x_{24}}, a_{25} = \frac{x_{12}^{2} x_{24}^{2} - x_{12} x_{14} x_{24} + x_{14}^{2}}{x_{12}^{3} x_{24}^{2}},\\ 
	&a_{34} = \frac{x_{12} x_{24} x_{35} + {\left(x_{12} x_{24} - x_{14}\right)} x_{34} x_{45}}{x_{12}^{2} x_{24}^{2} x_{45}}, a_{35} = 1,\\ 
	&a_{45} = \frac{x_{12} x_{24} - x_{14}}{x_{12}^{2} x_{24}^{2}}
\end{align*}

\begin{equation*}x=\left(\begin{matrix}
		1 & x_{12} & 0 & x_{14} & x_{15} \\
		0 & 1 & 0 & x_{24} & 0 \\
		0 & 0 & 1 & x_{34} & x_{35}  \\
		0 & 0 & 0 & 1 & x_{45}  \\
		0 & 0 & 0 & 0 & 1 \\
	\end{matrix}\right),x^A=\left(\begin{matrix}
		1 & 1 & 1 & 0 & 0 \\
		0 & 1 & 0 & 0 & 0 \\
		0 & 0 & 1 & 1 & 0 \\
		0 & 0 & 0 & 1 & 1 \\
		0 & 0 & 0 & 0 & 1
	\end{matrix}\right)
\end{equation*} \newline
Where matrix $A$ has entries
\begin{align*}
	&d_{1} = 1, d_{2} = \frac{1}{x_{12}}, d_{3} = \frac{x_{34}}{x_{12} x_{24}}, d_{4} = \frac{1}{x_{12} x_{24}}, d_{5} = \frac{1}{x_{12} x_{24} x_{45}},\\ 
	&a_{12} = 1, a_{13} = 1, a_{14} = 1, a_{15} = 1,\\ 
	&a_{23} = \frac{1}{x_{12}}, a_{24} = \frac{x_{12} x_{24} - x_{14}}{x_{12}^{2} x_{24}}, a_{25} = -\frac{x_{12} x_{15} x_{24} - {\left(x_{12}^{2} x_{24}^{2} - x_{12} x_{14} x_{24} + x_{14}^{2}\right)} x_{45}}{x_{12}^{3} x_{24}^{2} x_{45}},\\ 
	&a_{34} = \frac{x_{12} x_{24} x_{35} + {\left(x_{12} x_{24} - x_{14}\right)} x_{34} x_{45}}{x_{12}^{2} x_{24}^{2} x_{45}}, a_{35} = 1,\\ 
	&a_{45} = \frac{x_{12} x_{24} - x_{14}}{x_{12}^{2} x_{24}^{2}}
\end{align*}


First assume $x_{24}\neq \frac{-x_{13}x_{34}}{x_{12}}$.
\begin{equation*}x=\left(\begin{matrix}
		1 & x_{12} & x_{13} & 0 & 0 \\
		0 & 1 & 0 & x_{24} & 0 \\
		0 & 0 & 1 & x_{34} & x_{35}  \\
		0 & 0 & 0 & 1 & x_{45}  \\
		0 & 0 & 0 & 0 & 1 \\
	\end{matrix}\right),x^A=\left(\begin{matrix}
		1 & 1 & 1 & 0 & 0 \\
		0 & 1 & 0 & 0 & 0 \\
		0 & 0 & 1 & 1 & 0 \\
		0 & 0 & 0 & 1 & 1 \\
		0 & 0 & 0 & 0 & 1
	\end{matrix}\right)
\end{equation*} \newline
Where matrix $A$ has entries
\begin{align*}
	&d_{1} = 1, d_{2} = \frac{1}{x_{12}}, d_{3} = \frac{x_{34}}{x_{12} x_{24} + x_{13} x_{34}}, d_{4} = \frac{1}{x_{12} x_{24} + x_{13} x_{34}}, d_{5} = \frac{1}{{\left(x_{12} x_{24} + x_{13} x_{34}\right)} x_{45}},\\ 
	&a_{12} = 1, a_{13} = 1, a_{14} = 1, a_{15} = 1,\\ 
	&a_{23} = \frac{x_{24}}{x_{12} x_{24} + x_{13} x_{34}}, a_{24} = -\frac{x_{13} x_{24} x_{35} - {\left(x_{12} x_{24}^{2} + x_{13} x_{24} x_{34}\right)} x_{45}}{{\left(x_{12}^{2} x_{24}^{2} + 2 \, x_{12} x_{13} x_{24} x_{34} + x_{13}^{2} x_{34}^{2}\right)} x_{45}}, a_{25} = 1,\\ 
	&a_{34} = \frac{x_{12} x_{24} x_{35} + {\left(x_{12} x_{24} x_{34} + x_{13} x_{34}^{2}\right)} x_{45}}{{\left(x_{12}^{2} x_{24}^{2} + 2 \, x_{12} x_{13} x_{24} x_{34} + x_{13}^{2} x_{34}^{2}\right)} x_{45}}, a_{35} = -\frac{x_{12} - 1}{x_{13}},\\ 
	&a_{45} = -\frac{x_{13} x_{35} - {\left(x_{12} x_{24} + x_{13} x_{34}\right)} x_{45}}{{\left(x_{12}^{2} x_{24}^{2} + 2 \, x_{12} x_{13} x_{24} x_{34} + x_{13}^{2} x_{34}^{2}\right)} x_{45}}
\end{align*}

Now assume $x_{24}= \frac{-x_{13}x_{34}}{x_{12}}$.
\begin{equation*}x=\left(\begin{matrix}
		1 & x_{12} & x_{13} & 0 & 0 \\
		0 & 1 & 0 & x_{24} & 0 \\
		0 & 0 & 1 & x_{34} & x_{35}  \\
		0 & 0 & 0 & 1 & x_{45}  \\
		0 & 0 & 0 & 0 & 1 \\
	\end{matrix}\right),x^A=\left(\begin{matrix}
		1 & 1 & 0 & 0 & 0 \\
		0 & 1 & 0 & 0 & 0 \\
		0 & 0 & 1 & 1 & 0 \\
		0 & 0 & 0 & 1 & 1 \\
		0 & 0 & 0 & 0 & 1
	\end{matrix}\right)
\end{equation*} \newline
Where matrix $A$ has entries
\begin{align*}
	&d_{1} = 1, d_{2} = \frac{1}{x_{12}}, d_{3} = 1, d_{4} = \frac{1}{x_{34}}, d_{5} = \frac{1}{x_{34} x_{45}},\\ 
	&a_{12} = 1, a_{13} = \frac{x_{13} x_{35}}{x_{34} x_{45}}, a_{14} = 1, a_{15} = 1,\\ 
	&a_{23} = -\frac{x_{13}}{x_{12}}, a_{24} = 1, a_{25} = 1,\\ 
	&a_{34} = -\frac{x_{12} x_{34} x_{45} - x_{13} x_{35}}{x_{13} x_{34} x_{45}}, a_{35} = -\frac{x_{12} - 1}{x_{13}},\\ 
	&a_{45} = -\frac{x_{12}}{x_{13} x_{34}}
\end{align*}


First assume $x_{24}\neq \frac{-x_{13}x_{34}}{x_{12}}$.
\begin{equation*}x=\left(\begin{matrix}
		1 & x_{12} & x_{13} & 0 & x_{15} \\
		0 & 1 & 0 & x_{24} & 0 \\
		0 & 0 & 1 & x_{34} & x_{35}  \\
		0 & 0 & 0 & 1 & x_{45}  \\
		0 & 0 & 0 & 0 & 1 \\
	\end{matrix}\right),x^A=\left(\begin{matrix}
		1 & 1 & 1 & 0 & 0 \\
		0 & 1 & 0 & 0 & 0 \\
		0 & 0 & 1 & 1 & 0 \\
		0 & 0 & 0 & 1 & 1 \\
		0 & 0 & 0 & 0 & 1
	\end{matrix}\right)
\end{equation*} \newline
Where matrix $A$ has entries
\begin{align*}
	&d_{1} = 1, d_{2} = \frac{1}{x_{12}}, d_{3} = \frac{x_{34}}{x_{12} x_{24} + x_{13} x_{34}}, d_{4} = \frac{1}{x_{12} x_{24} + x_{13} x_{34}}, d_{5} = \frac{1}{{\left(x_{12} x_{24} + x_{13} x_{34}\right)} x_{45}},\\ 
	&a_{12} = 1, a_{13} = 1, a_{14} = 1, a_{15} = 1,\\ 
	&a_{23} = \frac{x_{24}}{x_{12} x_{24} + x_{13} x_{34}}, a_{24} = -\frac{x_{13} x_{24} x_{35} - {\left(x_{12} x_{24}^{2} + x_{13} x_{24} x_{34}\right)} x_{45}}{{\left(x_{12}^{2} x_{24}^{2} + 2 \, x_{12} x_{13} x_{24} x_{34} + x_{13}^{2} x_{34}^{2}\right)} x_{45}}, a_{25} = 1,\\ 
	&a_{34} = \frac{x_{12} x_{24} x_{35} + {\left(x_{12} x_{24} x_{34} + x_{13} x_{34}^{2}\right)} x_{45}}{{\left(x_{12}^{2} x_{24}^{2} + 2 \, x_{12} x_{13} x_{24} x_{34} + x_{13}^{2} x_{34}^{2}\right)} x_{45}}, a_{35} = -\frac{x_{15} + {\left({\left(x_{12} - 1\right)} x_{13} x_{34} + {\left(x_{12}^{2} - x_{12}\right)} x_{24}\right)} x_{45}}{{\left(x_{12} x_{13} x_{24} + x_{13}^{2} x_{34}\right)} x_{45}},\\ 
	&a_{45} = -\frac{x_{13} x_{35} - {\left(x_{12} x_{24} + x_{13} x_{34}\right)} x_{45}}{{\left(x_{12}^{2} x_{24}^{2} + 2 \, x_{12} x_{13} x_{24} x_{34} + x_{13}^{2} x_{34}^{2}\right)} x_{45}}
\end{align*}

Now  assume $x_{24}= \frac{-x_{13}x_{34}}{x_{12}}$.
\begin{equation*}x=\left(\begin{matrix}
		1 & x_{12} & x_{13} & 0 & x_{15} \\
		0 & 1 & 0 & x_{24} & 0 \\
		0 & 0 & 1 & x_{34} & x_{35}  \\
		0 & 0 & 0 & 1 & x_{45}  \\
		0 & 0 & 0 & 0 & 1 \\
	\end{matrix}\right),x^A=\left(\begin{matrix}
		1 & 1 & 0 & 0 & 0 \\
		0 & 1 & 0 & 0 & 0 \\
		0 & 0 & 1 & 1 & 0 \\
		0 & 0 & 0 & 1 & 1 \\
		0 & 0 & 0 & 0 & 1
	\end{matrix}\right)
\end{equation*} \newline
Where matrix $A$ has entries
\begin{align*}
	&d_{1} = 1, d_{2} = \frac{1}{x_{12}}, d_{3} = 1, d_{4} = \frac{1}{x_{34}}, d_{5} = \frac{1}{x_{34} x_{45}},\\ 
	&a_{12} = 1, a_{13} = \frac{x_{13} x_{35}}{x_{34} x_{45}}, a_{14} = 1, a_{15} = 1,\\ 
	&a_{23} = -\frac{x_{13}}{x_{12}}, a_{24} = 1, a_{25} = 1,\\ 
	&a_{34} = -\frac{x_{12} x_{34} x_{45} - x_{13} x_{35}}{x_{13} x_{34} x_{45}}, a_{35} = -\frac{{\left(x_{12} - 1\right)} x_{34} x_{45} + x_{15}}{x_{13} x_{34} x_{45}},\\ 
	&a_{45} = -\frac{x_{12}}{x_{13} x_{34}}
\end{align*}


First assume $x_{24}\neq \frac{-x_{13}x_{34}}{x_{12}}$.
\begin{equation*}x=\left(\begin{matrix}
		1 & x_{12} & x_{13} & x_{14} & 0 \\
		0 & 1 & 0 & x_{24} & 0 \\
		0 & 0 & 1 & x_{34} & x_{35}  \\
		0 & 0 & 0 & 1 & x_{45}  \\
		0 & 0 & 0 & 0 & 1 \\
	\end{matrix}\right),x^A=\left(\begin{matrix}
		1 & 1 & 1 & 0 & 0 \\
		0 & 1 & 0 & 0 & 0 \\
		0 & 0 & 1 & 1 & 0 \\
		0 & 0 & 0 & 1 & 1 \\
		0 & 0 & 0 & 0 & 1
	\end{matrix}\right)
\end{equation*} \newline
Where matrix $A$ has entries
\begin{align*}
	&d_{1} = 1, d_{2} = \frac{1}{x_{12}}, d_{3} = \frac{x_{34}}{x_{12} x_{24} + x_{13} x_{34}}, d_{4} = \frac{1}{x_{12} x_{24} + x_{13} x_{34}}, d_{5} = \frac{1}{{\left(x_{12} x_{24} + x_{13} x_{34}\right)} x_{45}},\\ 
	&a_{12} = 1, a_{13} = 1, a_{14} = 1, a_{15} = 1,\\ 
	&a_{23} = \frac{x_{24}}{x_{12} x_{24} + x_{13} x_{34}}, a_{24} = -\frac{x_{13} x_{24} x_{35} - {\left(x_{12} x_{24}^{2} + x_{13} x_{24} x_{34} - x_{14} x_{24}\right)} x_{45}}{{\left(x_{12}^{2} x_{24}^{2} + 2 \, x_{12} x_{13} x_{24} x_{34} + x_{13}^{2} x_{34}^{2}\right)} x_{45}}, a_{25} = 1,\\ 
	&a_{34} = \frac{x_{12} x_{24} x_{35} + {\left(x_{13} x_{34}^{2} + {\left(x_{12} x_{24} - x_{14}\right)} x_{34}\right)} x_{45}}{{\left(x_{12}^{2} x_{24}^{2} + 2 \, x_{12} x_{13} x_{24} x_{34} + x_{13}^{2} x_{34}^{2}\right)} x_{45}}, \\
	&a_{35} = \frac{x_{13} x_{14} x_{35} - {\left({\left(x_{12} - 1\right)} x_{13}^{2} x_{34}^{2} + x_{12} x_{14} x_{24} - x_{14}^{2} + {\left(x_{12}^{3} - x_{12}^{2}\right)} x_{24}^{2} + {\left(x_{13} x_{14} + 2 \, {\left(x_{12}^{2} - x_{12}\right)} x_{13} x_{24}\right)} x_{34}\right)} x_{45}}{{\left(x_{12}^{2} x_{13} x_{24}^{2} + 2 \, x_{12} x_{13}^{2} x_{24} x_{34} + x_{13}^{3} x_{34}^{2}\right)} x_{45}},\\ 
	&a_{45} = -\frac{x_{13} x_{35} - {\left(x_{12} x_{24} + x_{13} x_{34} - x_{14}\right)} x_{45}}{{\left(x_{12}^{2} x_{24}^{2} + 2 \, x_{12} x_{13} x_{24} x_{34} + x_{13}^{2} x_{34}^{2}\right)} x_{45}}
\end{align*}

Now assume $x_{24}= \frac{-x_{13}x_{34}}{x_{12}}$.
\begin{equation*}x=\left(\begin{matrix}
		1 & x_{12} & x_{13} & x_{14} & 0 \\
		0 & 1 & 0 & x_{24} & 0 \\
		0 & 0 & 1 & x_{34} & x_{35}  \\
		0 & 0 & 0 & 1 & x_{45}  \\
		0 & 0 & 0 & 0 & 1 \\
	\end{matrix}\right),x^A=\left(\begin{matrix}
		1 & 1 & 0 & 0 & 0 \\
		0 & 1 & 0 & 0 & 0 \\
		0 & 0 & 1 & 1 & 0 \\
		0 & 0 & 0 & 1 & 1 \\
		0 & 0 & 0 & 0 & 1
	\end{matrix}\right)
\end{equation*} \newline
Where matrix $A$ has entries
\begin{align*}
	&d_{1} = 1, d_{2} = \frac{1}{x_{12}}, d_{3} = 1, d_{4} = \frac{1}{x_{34}}, d_{5} = \frac{1}{x_{34} x_{45}},\\ 
	&a_{12} = 1, a_{13} = \frac{x_{13} x_{35} + x_{14} x_{45}}{x_{34} x_{45}}, a_{14} = 1, a_{15} = 1,\\ 
	&a_{23} = -\frac{x_{13}}{x_{12}}, a_{24} = 1, a_{25} = 1,\\ 
	&a_{34} = -\frac{x_{12} x_{34} x_{45} - x_{13} x_{35}}{x_{13} x_{34} x_{45}}, a_{35} = \frac{x_{12} x_{14} - {\left(x_{12} - 1\right)} x_{13} x_{34}}{x_{13}^{2} x_{34}},\\ 
	&a_{45} = -\frac{x_{12}}{x_{13} x_{34}}
\end{align*}


First assume $x_{24}\neq \frac{-x_{13}x_{34}}{x_{12}}$.
\begin{equation*}x=\left(\begin{matrix}
		1 & x_{12} & x_{13} & x_{14} & x_{15} \\
		0 & 1 & 0 & x_{24} & 0 \\
		0 & 0 & 1 & x_{34} & x_{35}  \\
		0 & 0 & 0 & 1 & x_{45}  \\
		0 & 0 & 0 & 0 & 1 \\
	\end{matrix}\right),x^A=\left(\begin{matrix}
		1 & 1 & 1 & 0 & 0 \\
		0 & 1 & 0 & 0 & 0 \\
		0 & 0 & 1 & 1 & 0 \\
		0 & 0 & 0 & 1 & 1 \\
		0 & 0 & 0 & 0 & 1
	\end{matrix}\right)
\end{equation*} \newline
Where matrix $A$ has entries
\begin{align*}
	&d_{1} = 1, d_{2} = \frac{1}{x_{12}}, d_{3} = \frac{x_{34}}{x_{12} x_{24} + x_{13} x_{34}}, d_{4} = \frac{1}{x_{12} x_{24} + x_{13} x_{34}}, d_{5} = \frac{1}{{\left(x_{12} x_{24} + x_{13} x_{34}\right)} x_{45}},\\ 
	&a_{12} = 1, a_{13} = 1, a_{14} = 1, a_{15} = 1,\\ 
	&a_{23} = \frac{x_{24}}{x_{12} x_{24} + x_{13} x_{34}}, a_{24} = -\frac{x_{13} x_{24} x_{35} - {\left(x_{12} x_{24}^{2} + x_{13} x_{24} x_{34} - x_{14} x_{24}\right)} x_{45}}{{\left(x_{12}^{2} x_{24}^{2} + 2 \, x_{12} x_{13} x_{24} x_{34} + x_{13}^{2} x_{34}^{2}\right)} x_{45}}, a_{25} = 1,\\ 
	&a_{34} = \frac{x_{12} x_{24} x_{35} + {\left(x_{13} x_{34}^{2} + {\left(x_{12} x_{24} - x_{14}\right)} x_{34}\right)} x_{45}}{{\left(x_{12}^{2} x_{24}^{2} + 2 \, x_{12} x_{13} x_{24} x_{34} + x_{13}^{2} x_{34}^{2}\right)} x_{45}}, \\
	&a_{35} = -\frac{\splitfrac{x_{12} x_{15} x_{24} + x_{13} x_{15} x_{34} - x_{13} x_{14} x_{35} +}
		{+{\left({\left(x_{12} - 1\right)} x_{13}^{2} x_{34}^{2} + x_{12} x_{14} x_{24} - x_{14}^{2} + {\left(x_{12}^{3} - x_{12}^{2}\right)} x_{24}^{2} + {\left(x_{13} x_{14} + 2 \, {\left(x_{12}^{2} - x_{12}\right)} x_{13} x_{24}\right)} x_{34}\right)} x_{45}}}{{\left(x_{12}^{2} x_{13} x_{24}^{2} + 2 \, x_{12} x_{13}^{2} x_{24} x_{34} + x_{13}^{3} x_{34}^{2}\right)} x_{45}},\\ 
	&a_{45} = -\frac{x_{13} x_{35} - {\left(x_{12} x_{24} + x_{13} x_{34} - x_{14}\right)} x_{45}}{{\left(x_{12}^{2} x_{24}^{2} + 2 \, x_{12} x_{13} x_{24} x_{34} + x_{13}^{2} x_{34}^{2}\right)} x_{45}}
\end{align*}

Now assume $x_{24}= \frac{-x_{13}x_{34}}{x_{12}}$.
\begin{equation*}x=\left(\begin{matrix}
		1 & x_{12} & x_{13} & x_{14} & x_{15} \\
		0 & 1 & 0 & x_{24} & 0 \\
		0 & 0 & 1 & x_{34} & x_{35}  \\
		0 & 0 & 0 & 1 & x_{45}  \\
		0 & 0 & 0 & 0 & 1 \\
	\end{matrix}\right),x^A=\left(\begin{matrix}
		1 & 1 & 0 & 0 & 0 \\
		0 & 1 & 0 & 0 & 0 \\
		0 & 0 & 1 & 1 & 0 \\
		0 & 0 & 0 & 1 & 1 \\
		0 & 0 & 0 & 0 & 1
	\end{matrix}\right)
\end{equation*} \newline
Where matrix $A$ has entries
\begin{align*}
	&d_{1} = 1, d_{2} = \frac{1}{x_{12}}, d_{3} = 1, d_{4} = \frac{1}{x_{34}}, d_{5} = \frac{1}{x_{34} x_{45}},\\ 
	&a_{12} = 1, a_{13} = \frac{x_{13} x_{35} + x_{14} x_{45}}{x_{34} x_{45}}, a_{14} = 1, a_{15} = 1,\\ 
	&a_{23} = -\frac{x_{13}}{x_{12}}, a_{24} = 1, a_{25} = 1,\\ 
	&a_{34} = -\frac{x_{12} x_{34} x_{45} - x_{13} x_{35}}{x_{13} x_{34} x_{45}}, a_{35} = -\frac{x_{13} x_{15} - {\left(x_{12} x_{14} - {\left(x_{12} - 1\right)} x_{13} x_{34}\right)} x_{45}}{x_{13}^{2} x_{34} x_{45}},\\ 
	&a_{45} = -\frac{x_{12}}{x_{13} x_{34}}
\end{align*}

\begin{equation*}x=\left(\begin{matrix}
		1 & x_{12} & 0 & 0 & 0 \\
		0 & 1 & 0 & x_{24} & x_{25} \\
		0 & 0 & 1 & x_{34} & x_{35}  \\
		0 & 0 & 0 & 1 & x_{45}  \\
		0 & 0 & 0 & 0 & 1 \\
	\end{matrix}\right),x^A=\left(\begin{matrix}
		1 & 1 & 1 & 0 & 0 \\
		0 & 1 & 0 & 0 & 0 \\
		0 & 0 & 1 & 1 & 0 \\
		0 & 0 & 0 & 1 & 1 \\
		0 & 0 & 0 & 0 & 1
	\end{matrix}\right)
\end{equation*} \newline
Where matrix $A$ has entries
\begin{align*}
	&d_{1} = 1, d_{2} = \frac{1}{x_{12}}, d_{3} = \frac{x_{34}}{x_{12} x_{24}}, d_{4} = \frac{1}{x_{12} x_{24}}, d_{5} = \frac{1}{x_{12} x_{24} x_{45}},\\ 
	&a_{12} = 1, a_{13} = 1, a_{14} = 1, a_{15} = 1,\\ 
	&a_{23} = \frac{1}{x_{12}}, a_{24} = \frac{1}{x_{12}}, a_{25} = \frac{1}{x_{12}},\\ 
	&a_{34} = \frac{x_{24} x_{34} x_{45} - x_{25} x_{34} + x_{24} x_{35}}{x_{12} x_{24}^{2} x_{45}}, a_{35} = 1,\\ 
	&a_{45} = \frac{x_{24} x_{45} - x_{25}}{x_{12} x_{24}^{2} x_{45}}
\end{align*}

\begin{equation*}x=\left(\begin{matrix}
		1 & x_{12} & 0 & 0 & x_{15} \\
		0 & 1 & 0 & x_{24} & x_{25} \\
		0 & 0 & 1 & x_{34} & x_{35}  \\
		0 & 0 & 0 & 1 & x_{45}  \\
		0 & 0 & 0 & 0 & 1 \\
	\end{matrix}\right),x^A=\left(\begin{matrix}
		1 & 1 & 1 & 0 & 0 \\
		0 & 1 & 0 & 0 & 0 \\
		0 & 0 & 1 & 1 & 0 \\
		0 & 0 & 0 & 1 & 1 \\
		0 & 0 & 0 & 0 & 1
	\end{matrix}\right)
\end{equation*} \newline
Where matrix $A$ has entries
\begin{align*}
	&d_{1} = 1, d_{2} = \frac{1}{x_{12}}, d_{3} = \frac{x_{34}}{x_{12} x_{24}}, d_{4} = \frac{1}{x_{12} x_{24}}, d_{5} = \frac{1}{x_{12} x_{24} x_{45}},\\ 
	&a_{12} = 1, a_{13} = 1, a_{14} = 1, a_{15} = 1,\\ 
	&a_{23} = \frac{1}{x_{12}}, a_{24} = \frac{1}{x_{12}}, a_{25} = \frac{x_{12} x_{24} x_{45} - x_{15}}{x_{12}^{2} x_{24} x_{45}},\\ 
	&a_{34} = \frac{x_{24} x_{34} x_{45} - x_{25} x_{34} + x_{24} x_{35}}{x_{12} x_{24}^{2} x_{45}}, a_{35} = 1,\\ 
	&a_{45} = \frac{x_{24} x_{45} - x_{25}}{x_{12} x_{24}^{2} x_{45}}
\end{align*}

\begin{equation*}x=\left(\begin{matrix}
		1 & x_{12} & 0 & x_{14} & 0 \\
		0 & 1 & 0 & x_{24} & x_{25} \\
		0 & 0 & 1 & x_{34} & x_{35}  \\
		0 & 0 & 0 & 1 & x_{45}  \\
		0 & 0 & 0 & 0 & 1 \\
	\end{matrix}\right),x^A=\left(\begin{matrix}
		1 & 1 & 1 & 0 & 0 \\
		0 & 1 & 0 & 0 & 0 \\
		0 & 0 & 1 & 1 & 0 \\
		0 & 0 & 0 & 1 & 1 \\
		0 & 0 & 0 & 0 & 1
	\end{matrix}\right)
\end{equation*} \newline
Where matrix $A$ has entries
\begin{align*}
	&d_{1} = 1, d_{2} = \frac{1}{x_{12}}, d_{3} = \frac{x_{34}}{x_{12} x_{24}}, d_{4} = \frac{1}{x_{12} x_{24}}, d_{5} = \frac{1}{x_{12} x_{24} x_{45}},\\ 
	&a_{12} = 1, a_{13} = 1, a_{14} = 1, a_{15} = 1,\\ 
	&a_{23} = \frac{1}{x_{12}}, a_{24} = \frac{x_{12} x_{24} - x_{14}}{x_{12}^{2} x_{24}}, a_{25} = \frac{x_{12} x_{14} x_{25} + {\left(x_{12}^{2} x_{24}^{2} - x_{12} x_{14} x_{24} + x_{14}^{2}\right)} x_{45}}{x_{12}^{3} x_{24}^{2} x_{45}},\\ 
	&a_{34} = -\frac{x_{12} x_{25} x_{34} - x_{12} x_{24} x_{35} - {\left(x_{12} x_{24} - x_{14}\right)} x_{34} x_{45}}{x_{12}^{2} x_{24}^{2} x_{45}}, a_{35} = 1,\\ 
	&a_{45} = -\frac{x_{12} x_{25} - {\left(x_{12} x_{24} - x_{14}\right)} x_{45}}{x_{12}^{2} x_{24}^{2} x_{45}}
\end{align*}

\begin{equation*}x=\left(\begin{matrix}
		1 & x_{12} & 0 & x_{14} & x_{15} \\
		0 & 1 & 0 & x_{24} & x_{25} \\
		0 & 0 & 1 & x_{34} & x_{35}  \\
		0 & 0 & 0 & 1 & x_{45}  \\
		0 & 0 & 0 & 0 & 1 \\
	\end{matrix}\right),x^A=\left(\begin{matrix}
		1 & 1 & 1 & 0 & 0 \\
		0 & 1 & 0 & 0 & 0 \\
		0 & 0 & 1 & 1 & 0 \\
		0 & 0 & 0 & 1 & 1 \\
		0 & 0 & 0 & 0 & 1
	\end{matrix}\right)
\end{equation*} \newline
Where matrix $A$ has entries
\begin{align*}
	&d_{1} = 1, d_{2} = \frac{1}{x_{12}}, d_{3} = \frac{x_{34}}{x_{12} x_{24}}, d_{4} = \frac{1}{x_{12} x_{24}}, d_{5} = \frac{1}{x_{12} x_{24} x_{45}},\\ 
	&a_{12} = 1, a_{13} = 1, a_{14} = 1, a_{15} = 1,\\ 
	&a_{23} = \frac{1}{x_{12}}, a_{24} = \frac{x_{12} x_{24} - x_{14}}{x_{12}^{2} x_{24}}, a_{25} = -\frac{x_{12} x_{15} x_{24} - x_{12} x_{14} x_{25} - {\left(x_{12}^{2} x_{24}^{2} - x_{12} x_{14} x_{24} + x_{14}^{2}\right)} x_{45}}{x_{12}^{3} x_{24}^{2} x_{45}},\\ 
	&a_{34} = -\frac{x_{12} x_{25} x_{34} - x_{12} x_{24} x_{35} - {\left(x_{12} x_{24} - x_{14}\right)} x_{34} x_{45}}{x_{12}^{2} x_{24}^{2} x_{45}}, a_{35} = 1,\\ 
	&a_{45} = -\frac{x_{12} x_{25} - {\left(x_{12} x_{24} - x_{14}\right)} x_{45}}{x_{12}^{2} x_{24}^{2} x_{45}}
\end{align*}


First assume $x_{24}\neq \frac{-x_{13}x_{34}}{x_{12}}$.
\begin{equation*}x=\left(\begin{matrix}
		1 & x_{12} & x_{13} & 0 & 0 \\
		0 & 1 & 0 & x_{24} & x_{25} \\
		0 & 0 & 1 & x_{34} & x_{35}  \\
		0 & 0 & 0 & 1 & x_{45}  \\
		0 & 0 & 0 & 0 & 1 \\
	\end{matrix}\right),x^A=\left(\begin{matrix}
		1 & 1 & 1 & 0 & 0 \\
		0 & 1 & 0 & 0 & 0 \\
		0 & 0 & 1 & 1 & 0 \\
		0 & 0 & 0 & 1 & 1 \\
		0 & 0 & 0 & 0 & 1
	\end{matrix}\right)
\end{equation*} \newline
Where matrix $A$ has entries
\begin{align*}
	&d_{1} = 1, d_{2} = \frac{1}{x_{12}}, d_{3} = \frac{x_{34}}{x_{12} x_{24} + x_{13} x_{34}}, d_{4} = \frac{1}{x_{12} x_{24} + x_{13} x_{34}}, d_{5} = \frac{1}{{\left(x_{12} x_{24} + x_{13} x_{34}\right)} x_{45}},\\ 
	&a_{12} = 1, a_{13} = 1, a_{14} = 1, a_{15} = 1,\\ 
	&a_{23} = \frac{x_{24}}{x_{12} x_{24} + x_{13} x_{34}}, a_{24} = \frac{x_{13} x_{25} x_{34} - x_{13} x_{24} x_{35} + {\left(x_{12} x_{24}^{2} + x_{13} x_{24} x_{34}\right)} x_{45}}{{\left(x_{12}^{2} x_{24}^{2} + 2 \, x_{12} x_{13} x_{24} x_{34} + x_{13}^{2} x_{34}^{2}\right)} x_{45}}, a_{25} = 1,\\ 
	&a_{34} = -\frac{x_{12} x_{25} x_{34} - x_{12} x_{24} x_{35} - {\left(x_{12} x_{24} x_{34} + x_{13} x_{34}^{2}\right)} x_{45}}{{\left(x_{12}^{2} x_{24}^{2} + 2 \, x_{12} x_{13} x_{24} x_{34} + x_{13}^{2} x_{34}^{2}\right)} x_{45}}, a_{35} = -\frac{x_{12} - 1}{x_{13}},\\ 
	&a_{45} = -\frac{x_{12} x_{25} + x_{13} x_{35} - {\left(x_{12} x_{24} + x_{13} x_{34}\right)} x_{45}}{{\left(x_{12}^{2} x_{24}^{2} + 2 \, x_{12} x_{13} x_{24} x_{34} + x_{13}^{2} x_{34}^{2}\right)} x_{45}}
\end{align*}

Now assume $x_{24}= \frac{-x_{13}x_{34}}{x_{12}}$.
\begin{equation*}x=\left(\begin{matrix}
		1 & x_{12} & x_{13} & 0 & 0 \\
		0 & 1 & 0 & x_{24} & x_{25} \\
		0 & 0 & 1 & x_{34} & x_{35}  \\
		0 & 0 & 0 & 1 & x_{45}  \\
		0 & 0 & 0 & 0 & 1 \\
	\end{matrix}\right),x^A=\left(\begin{matrix}
		1 & 1 & 0 & 0 & 0 \\
		0 & 1 & 0 & 0 & 0 \\
		0 & 0 & 1 & 1 & 0 \\
		0 & 0 & 0 & 1 & 1 \\
		0 & 0 & 0 & 0 & 1
	\end{matrix}\right)
\end{equation*} \newline
Where matrix $A$ has entries
\begin{align*}
	&d_{1} = 1, d_{2} = \frac{1}{x_{12}}, d_{3} = 1, d_{4} = \frac{1}{x_{34}}, d_{5} = \frac{1}{x_{34} x_{45}},\\ 
	&a_{12} = 1, a_{13} = \frac{x_{12} x_{25} + x_{13} x_{35}}{x_{34} x_{45}}, a_{14} = 1, a_{15} = 1,\\ 
	&a_{23} = -\frac{x_{13}}{x_{12}}, a_{24} = 1, a_{25} = 1,\\ 
	&a_{34} = -\frac{x_{12} x_{34} x_{45} - x_{12} x_{25} - x_{13} x_{35}}{x_{13} x_{34} x_{45}}, a_{35} = -\frac{x_{12} - 1}{x_{13}},\\ 
	&a_{45} = -\frac{x_{12} x_{34} x_{45} - x_{12} x_{25}}{x_{13} x_{34}^{2} x_{45}}
\end{align*}

First assume $x_{24}\neq \frac{-x_{13}x_{34}}{x_{12}}$.
\begin{equation*}x=\left(\begin{matrix}
		1 & x_{12} & x_{13} & 0 & x_{15} \\
		0 & 1 & 0 & x_{24} & x_{25} \\
		0 & 0 & 1 & x_{34} & x_{35}  \\
		0 & 0 & 0 & 1 & x_{45}  \\
		0 & 0 & 0 & 0 & 1 \\
	\end{matrix}\right),x^A=\left(\begin{matrix}
		1 & 1 & 1 & 0 & 0 \\
		0 & 1 & 0 & 0 & 0 \\
		0 & 0 & 1 & 1 & 0 \\
		0 & 0 & 0 & 1 & 1 \\
		0 & 0 & 0 & 0 & 1
	\end{matrix}\right)
\end{equation*} \newline
Where matrix $A$ has entries
\begin{align*}
	&d_{1} = 1, d_{2} = \frac{1}{x_{12}}, d_{3} = \frac{x_{34}}{x_{12} x_{24} + x_{13} x_{34}}, d_{4} = \frac{1}{x_{12} x_{24} + x_{13} x_{34}}, d_{5} = \frac{1}{{\left(x_{12} x_{24} + x_{13} x_{34}\right)} x_{45}},\\ 
	&a_{12} = 1, a_{13} = 1, a_{14} = 1, a_{15} = 1,\\ 
	&a_{23} = \frac{x_{24}}{x_{12} x_{24} + x_{13} x_{34}}, a_{24} = \frac{x_{13} x_{25} x_{34} - x_{13} x_{24} x_{35} + {\left(x_{12} x_{24}^{2} + x_{13} x_{24} x_{34}\right)} x_{45}}{{\left(x_{12}^{2} x_{24}^{2} + 2 \, x_{12} x_{13} x_{24} x_{34} + x_{13}^{2} x_{34}^{2}\right)} x_{45}}, a_{25} = 1,\\ 
	&a_{34} = -\frac{x_{12} x_{25} x_{34} - x_{12} x_{24} x_{35} - {\left(x_{12} x_{24} x_{34} + x_{13} x_{34}^{2}\right)} x_{45}}{{\left(x_{12}^{2} x_{24}^{2} + 2 \, x_{12} x_{13} x_{24} x_{34} + x_{13}^{2} x_{34}^{2}\right)} x_{45}}, a_{35} = -\frac{x_{15} + {\left({\left(x_{12} - 1\right)} x_{13} x_{34} + {\left(x_{12}^{2} - x_{12}\right)} x_{24}\right)} x_{45}}{{\left(x_{12} x_{13} x_{24} + x_{13}^{2} x_{34}\right)} x_{45}},\\ 
	&a_{45} = -\frac{x_{12} x_{25} + x_{13} x_{35} - {\left(x_{12} x_{24} + x_{13} x_{34}\right)} x_{45}}{{\left(x_{12}^{2} x_{24}^{2} + 2 \, x_{12} x_{13} x_{24} x_{34} + x_{13}^{2} x_{34}^{2}\right)} x_{45}}
\end{align*}

Now assume $x_{24}= \frac{-x_{13}x_{34}}{x_{12}}$.
\begin{equation*}x=\left(\begin{matrix}
		1 & x_{12} & x_{13} & 0 & x_{15} \\
		0 & 1 & 0 & x_{24} & x_{25} \\
		0 & 0 & 1 & x_{34} & x_{35}  \\
		0 & 0 & 0 & 1 & x_{45}  \\
		0 & 0 & 0 & 0 & 1 \\
	\end{matrix}\right),x^A=\left(\begin{matrix}
		1 & 1 & 0 & 0 & 0 \\
		0 & 1 & 0 & 0 & 0 \\
		0 & 0 & 1 & 1 & 0 \\
		0 & 0 & 0 & 1 & 1 \\
		0 & 0 & 0 & 0 & 1
	\end{matrix}\right)
\end{equation*} \newline
Where matrix $A$ has entries
\begin{align*}
	&d_{1} = 1, d_{2} = \frac{1}{x_{12}}, d_{3} = 1, d_{4} = \frac{1}{x_{34}}, d_{5} = \frac{1}{x_{34} x_{45}},\\ 
	&a_{12} = 1, a_{13} = \frac{x_{12} x_{25} + x_{13} x_{35}}{x_{34} x_{45}}, a_{14} = 1, a_{15} = 1,\\ 
	&a_{23} = -\frac{x_{13}}{x_{12}}, a_{24} = 1, a_{25} = 1,\\ 
	&a_{34} = -\frac{x_{12} x_{34} x_{45} - x_{12} x_{25} - x_{13} x_{35}}{x_{13} x_{34} x_{45}}, a_{35} = -\frac{{\left(x_{12} - 1\right)} x_{34} x_{45} + x_{15}}{x_{13} x_{34} x_{45}},\\ 
	&a_{45} = -\frac{x_{12} x_{34} x_{45} - x_{12} x_{25}}{x_{13} x_{34}^{2} x_{45}}
\end{align*}

First assume $x_{24}\neq \frac{-x_{13}x_{34}}{x_{12}}$.
\begin{equation*}x=\left(\begin{matrix}
		1 & x_{12} & x_{13} & x_{14} & 0 \\
		0 & 1 & 0 & x_{24} & x_{25} \\
		0 & 0 & 1 & x_{34} & x_{35}  \\
		0 & 0 & 0 & 1 & x_{45}  \\
		0 & 0 & 0 & 0 & 1 \\
	\end{matrix}\right),x^A=\left(\begin{matrix}
		1 & 1 & 1 & 0 & 0 \\
		0 & 1 & 0 & 0 & 0 \\
		0 & 0 & 1 & 1 & 0 \\
		0 & 0 & 0 & 1 & 1 \\
		0 & 0 & 0 & 0 & 1
	\end{matrix}\right)
\end{equation*} \newline
Where matrix $A$ has entries
\begin{align*}
	&d_{1} = 1, d_{2} = \frac{1}{x_{12}}, d_{3} = \frac{x_{34}}{x_{12} x_{24} + x_{13} x_{34}}, d_{4} = \frac{1}{x_{12} x_{24} + x_{13} x_{34}}, d_{5} = \frac{1}{{\left(x_{12} x_{24} + x_{13} x_{34}\right)} x_{45}},\\ 
	&a_{12} = 1, a_{13} = 1, a_{14} = 1, a_{15} = 1,\\ 
	&a_{23} = \frac{x_{24}}{x_{12} x_{24} + x_{13} x_{34}}, a_{24} = \frac{x_{13} x_{25} x_{34} - x_{13} x_{24} x_{35} + {\left(x_{12} x_{24}^{2} + x_{13} x_{24} x_{34} - x_{14} x_{24}\right)} x_{45}}{{\left(x_{12}^{2} x_{24}^{2} + 2 \, x_{12} x_{13} x_{24} x_{34} + x_{13}^{2} x_{34}^{2}\right)} x_{45}}, a_{25} = 1,\\ 
	&a_{34} = -\frac{x_{12} x_{25} x_{34} - x_{12} x_{24} x_{35} - {\left(x_{13} x_{34}^{2} + {\left(x_{12} x_{24} - x_{14}\right)} x_{34}\right)} x_{45}}{{\left(x_{12}^{2} x_{24}^{2} + 2 \, x_{12} x_{13} x_{24} x_{34} + x_{13}^{2} x_{34}^{2}\right)} x_{45}}, \\
	&a_{35} = \frac{x_{12} x_{14} x_{25} + x_{13} x_{14} x_{35} - {\left({\left(x_{12} - 1\right)} x_{13}^{2} x_{34}^{2} + x_{12} x_{14} x_{24} - x_{14}^{2} + {\left(x_{12}^{3} - x_{12}^{2}\right)} x_{24}^{2} + {\left(x_{13} x_{14} + 2 \, {\left(x_{12}^{2} - x_{12}\right)} x_{13} x_{24}\right)} x_{34}\right)} x_{45}}{{\left(x_{12}^{2} x_{13} x_{24}^{2} + 2 \, x_{12} x_{13}^{2} x_{24} x_{34} + x_{13}^{3} x_{34}^{2}\right)} x_{45}},\\ 
	&a_{45} = -\frac{x_{12} x_{25} + x_{13} x_{35} - {\left(x_{12} x_{24} + x_{13} x_{34} - x_{14}\right)} x_{45}}{{\left(x_{12}^{2} x_{24}^{2} + 2 \, x_{12} x_{13} x_{24} x_{34} + x_{13}^{2} x_{34}^{2}\right)} x_{45}}
\end{align*}

Now assume $x_{24}= \frac{-x_{13}x_{34}}{x_{12}}$.
\begin{equation*}x=\left(\begin{matrix}
		1 & x_{12} & x_{13} & x_{14} & 0 \\
		0 & 1 & 0 & x_{24} & x_{25} \\
		0 & 0 & 1 & x_{34} & x_{35}  \\
		0 & 0 & 0 & 1 & x_{45}  \\
		0 & 0 & 0 & 0 & 1 \\
	\end{matrix}\right),x^A=\left(\begin{matrix}
		1 & 1 & 0 & 0 & 0 \\
		0 & 1 & 0 & 0 & 0 \\
		0 & 0 & 1 & 1 & 0 \\
		0 & 0 & 0 & 1 & 1 \\
		0 & 0 & 0 & 0 & 1
	\end{matrix}\right)
\end{equation*} \newline
Where matrix $A$ has entries
\begin{align*}
	&d_{1} = 1, d_{2} = \frac{1}{x_{12}}, d_{3} = 1, d_{4} = \frac{1}{x_{34}}, d_{5} = \frac{1}{x_{34} x_{45}},\\ 
	&a_{12} = 1, a_{13} = \frac{x_{12} x_{25} + x_{13} x_{35} + x_{14} x_{45}}{x_{34} x_{45}}, a_{14} = 1, a_{15} = 1,\\ 
	&a_{23} = -\frac{x_{13}}{x_{12}}, a_{24} = 1, a_{25} = 1,\\ 
	&a_{34} = -\frac{x_{12} x_{34} x_{45} - x_{12} x_{25} - x_{13} x_{35}}{x_{13} x_{34} x_{45}}, a_{35} = -\frac{x_{12} x_{14} x_{25} - {\left(x_{12} x_{14} x_{34} - {\left(x_{12} - 1\right)} x_{13} x_{34}^{2}\right)} x_{45}}{x_{13}^{2} x_{34}^{2} x_{45}},\\ 
	&a_{45} = -\frac{x_{12} x_{34} x_{45} - x_{12} x_{25}}{x_{13} x_{34}^{2} x_{45}}
\end{align*}


First assume $x_{24}\neq \frac{-x_{13}x_{34}}{x_{12}}$.
\begin{equation*}x=\left(\begin{matrix}
		1 & x_{12} & x_{13} & x_{14} & x_{15} \\
		0 & 1 & 0 & x_{24} & x_{25} \\
		0 & 0 & 1 & x_{34} & x_{35}  \\
		0 & 0 & 0 & 1 & x_{45}  \\
		0 & 0 & 0 & 0 & 1 \\
	\end{matrix}\right),x^A=\left(\begin{matrix}
		1 & 1 & 1 & 0 & 0 \\
		0 & 1 & 0 & 0 & 0 \\
		0 & 0 & 1 & 1 & 0 \\
		0 & 0 & 0 & 1 & 1 \\
		0 & 0 & 0 & 0 & 1
	\end{matrix}\right)
\end{equation*} \newline
Where matrix $A$ has entries
\begin{align*}
	&d_{1} = 1, d_{2} = \frac{1}{x_{12}}, d_{3} = \frac{x_{34}}{x_{12} x_{24} + x_{13} x_{34}}, d_{4} = \frac{1}{x_{12} x_{24} + x_{13} x_{34}}, d_{5} = \frac{1}{{\left(x_{12} x_{24} + x_{13} x_{34}\right)} x_{45}},\\ 
	&a_{12} = 1, a_{13} = 1, a_{14} = 1, a_{15} = 1,\\ 
	&a_{23} = \frac{x_{24}}{x_{12} x_{24} + x_{13} x_{34}}, a_{24} = \frac{x_{13} x_{25} x_{34} - x_{13} x_{24} x_{35} + {\left(x_{12} x_{24}^{2} + x_{13} x_{24} x_{34} - x_{14} x_{24}\right)} x_{45}}{{\left(x_{12}^{2} x_{24}^{2} + 2 \, x_{12} x_{13} x_{24} x_{34} + x_{13}^{2} x_{34}^{2}\right)} x_{45}}, a_{25} = 1,\\ 
	&a_{34} = -\frac{x_{12} x_{25} x_{34} - x_{12} x_{24} x_{35} - {\left(x_{13} x_{34}^{2} + {\left(x_{12} x_{24} - x_{14}\right)} x_{34}\right)} x_{45}}{{\left(x_{12}^{2} x_{24}^{2} + 2 \, x_{12} x_{13} x_{24} x_{34} + x_{13}^{2} x_{34}^{2}\right)} x_{45}}, \\
	&a_{35} = -\frac{\splitfrac{x_{12} x_{15} x_{24} - x_{12} x_{14} x_{25} + x_{13} x_{15} x_{34} - x_{13} x_{14} x_{35} +}
		{+ {\left({\left(x_{12} - 1\right)} x_{13}^{2} x_{34}^{2} + x_{12} x_{14} x_{24} - x_{14}^{2} + {\left(x_{12}^{3} - x_{12}^{2}\right)} x_{24}^{2} + {\left(x_{13} x_{14} + 2 \, {\left(x_{12}^{2} - x_{12}\right)} x_{13} x_{24}\right)} x_{34}\right)} x_{45}}}{{\left(x_{12}^{2} x_{13} x_{24}^{2} + 2 \, x_{12} x_{13}^{2} x_{24} x_{34} + x_{13}^{3} x_{34}^{2}\right)} x_{45}},\\ 
	&a_{45} = -\frac{x_{12} x_{25} + x_{13} x_{35} - {\left(x_{12} x_{24} + x_{13} x_{34} - x_{14}\right)} x_{45}}{{\left(x_{12}^{2} x_{24}^{2} + 2 \, x_{12} x_{13} x_{24} x_{34} + x_{13}^{2} x_{34}^{2}\right)} x_{45}}
\end{align*}

Now assume $x_{24}= \frac{-x_{13}x_{34}}{x_{12}}$.
\begin{equation*}x=\left(\begin{matrix}
		1 & x_{12} & x_{13} & x_{14} & x_{15} \\
		0 & 1 & 0 & x_{24} & x_{25} \\
		0 & 0 & 1 & x_{34} & x_{35}  \\
		0 & 0 & 0 & 1 & x_{45}  \\
		0 & 0 & 0 & 0 & 1 \\
	\end{matrix}\right),x^A=\left(\begin{matrix}
		1 & 1 & 0 & 0 & 0 \\
		0 & 1 & 0 & 0 & 0 \\
		0 & 0 & 1 & 1 & 0 \\
		0 & 0 & 0 & 1 & 1 \\
		0 & 0 & 0 & 0 & 1
	\end{matrix}\right)
\end{equation*} \newline
Where matrix $A$ has entries
\begin{align*}
	&d_{1} = 1, d_{2} = \frac{1}{x_{12}}, d_{3} = 1, d_{4} = \frac{1}{x_{34}}, d_{5} = \frac{1}{x_{34} x_{45}},\\ 
	&a_{12} = 1, a_{13} = \frac{x_{12} x_{25} + x_{13} x_{35} + x_{14} x_{45}}{x_{34} x_{45}}, a_{14} = 1, a_{15} = 1,\\ 
	&a_{23} = -\frac{x_{13}}{x_{12}}, a_{24} = 1, a_{25} = 1,\\ 
	&a_{34} = -\frac{x_{12} x_{34} x_{45} - x_{12} x_{25} - x_{13} x_{35}}{x_{13} x_{34} x_{45}}, a_{35} = -\frac{x_{12} x_{14} x_{25} + x_{13} x_{15} x_{34} - {\left(x_{12} x_{14} x_{34} - {\left(x_{12} - 1\right)} x_{13} x_{34}^{2}\right)} x_{45}}{x_{13}^{2} x_{34}^{2} x_{45}},\\ 
	&a_{45} = -\frac{x_{12} x_{34} x_{45} - x_{12} x_{25}}{x_{13} x_{34}^{2} x_{45}}
\end{align*}